\documentclass[a4paper,11pt,oneside,reqno]{article}
\usepackage[top=1in, left=1in, right=1in, bottom=1in]{geometry}

\usepackage{tabularx} 
\usepackage[table]{xcolor}
\usepackage{amsthm, amsmath, amssymb, amsfonts, mathrsfs, multicol, multirow, geometry, comment, tikz,enumitem}

\usepackage[table]{xcolor} 
\usepackage{booktabs}      
\usepackage{graphicx}      

\definecolor{grey1}{gray}{0.9} 
\definecolor{grey2}{gray}{0.8}
\definecolor{grey3}{gray}{0.7}
\definecolor{grey4}{gray}{0.6} 

\definecolor{yellow1}{RGB}{255, 255, 204} 
\definecolor{yellow2}{RGB}{255, 255, 153}
\definecolor{yellow3}{RGB}{255, 204, 102}
\definecolor{yellow4}{RGB}{255, 153, 51}  

\usepackage[colorlinks=true,    
            linkcolor=blue,     
            citecolor=cyan,    
            urlcolor=red]       
           {hyperref}
\usepackage{algorithm}
\usepackage{algpseudocode}
\usepackage{adjustbox} 

\allowdisplaybreaks
\newtheorem{theorem}{Theorem}[section]
   
\newtheorem{remark}[theorem]{Remark}
\newtheorem{proposition}[theorem]{Proposition}
\newtheorem{lemma}[theorem]{Lemma}

\newtheorem{problem}[theorem]{Problem}
\newtheorem{example}[theorem]{Example}
\newtheorem{assumption}[theorem]{Assumption}

\newcommand{\vertiii}[1]{{\left\vert\kern-0.25ex\left\vert\kern-0.25ex\left\vert #1
  \right\vert\kern-0.25ex\right\vert\kern-0.25ex\right\vert}}

\newcommand{\regu}{R}                  
\newcommand{\Jmin}{J^{\star}}          
\newcommand{\Je}{J_{\rho}}             
\newcommand{\Jeprime}{\Je^{\prime}}    
 
\newcommand{\Jedprime}{\Je^{\prime\prime}}
\newcommand{\Aad}{\mathcal{A}}         
\newcommand{\mina}{\underline{\alpha}}
\newcommand{\maxa}{\overline{\alpha}}
\newcommand{\alpharho}{\alpha_{\rho}} 
\newcommand{\VV}{H^{1}(\varOmega)}     
\newcommand{\bb}{b}                     
\newcommand{\nn}{n}                     
\newcommand{\ud}{u_{D}} 
\newcommand{\un}{u_{N}} 
\newcommand{\iu}{u_{i}}                 
\newcommand{\ru}{u_{r}}                 
\newcommand{\ip}{p_{i}}                 
\newcommand{\rp}{p_{r}}                 
\newcommand{\dela}{\beta}               
\newcommand{\SubsetK}{\mathcal{K}}      
\newcommand{\barAad}{\overline{\mathcal{A}}}
\newcommand{\dAad}{\partial{\overline{\mathcal{A}}}}
\newcommand{\dn}[1]{\partial_{\nn}{#1}} 
\newcommand{\norm}[1]{\|{#1}\|}         
\newcommand{\abs}[1]{\vert{#1}\vert}    
\newcommand{\intO}[1]{\int_{\varOmega} #1 \, dx}     
\newcommand{\intG}[1]{\int_{\varGamma} #1 \, ds}     
\newcommand{\inS}[1]{\langle{#1}\rangle}
\newcommand{\Fprime}{F^{\prime}}
\newcommand{\Fdprime}{F^{\prime\prime}}
\newcommand{\uprime}{u^{\prime}}
\newcommand{\udprime}{u^{\prime\prime}}
\newcommand{\iuprime}{u_{i}^{\prime}}
\newcommand{\iudprime}{u_{i}^{\prime\prime}}

\newcommand{\alert}[1]{{\color{black}{#1}}}

\begin{document}

\title{Numerical methods for diffusion coefficient recovery} 


\title{Numerical methods for diffusion coefficient recovery}

\author{
Sahat Pandapotan Nainggolan\thanks{Division of Mathematical and Physical Sciences, Institute of Science and Engineering, Kanazawa University, Japan. Email: sahatpandapotannainggolan@gmail.com}
\and
Julius Fergy Tiongson Rabago\thanks{Faculty of Mathematics and Physics, Institute of Science and Engineering, Kanazawa University, Japan. Email: jftrabago@gmail.com}
\and
Hirofumi Notsu\thanks{Faculty of Mathematics and Physics, Institute of Science and Engineering, Kanazawa University, Japan. Email: notsu@se.kanazawa-u.ac.jp}
}

\date{\today}

\maketitle

\let\thefootnote\relax
\footnotetext{MSC2020: Primary 65N21, 65F22, Secondary 47G40.} 

\begin{abstract}
We revisit the inverse problem of reconstructing a spatially varying diffusion coefficient in stationary elliptic equations from boundary Cauchy data. From a theoretical perspective, we introduce a gradient-weighted modification of the coupled complex-boundary method (CCBM) incorporating an \(H^1\)-type term, and formulate the reconstruction as a regularized optimization problem over bounded admissible coefficients. We establish continuity and differentiability of the forward map, Lipschitz continuity of the modified cost functional, existence of minimizers, stability with respect to noisy data, and convergence under vanishing noise.
From a numerical perspective, reconstructions are computed using a Sobolev-gradient descent scheme and evaluated through extensive numerical experiments across a range of noise levels, boundary inputs, and coefficient structures. In the reported tests, for sufficiently large but not excessive $H^1$-weights, the modified CCBM is observed to yield more stable reconstructions and to reduce certain high-frequency artifacts. Across the numerical scenarios considered in this study, the method often demonstrates favorable stability and robustness properties relative to several classical boundary-based formulations, although performance remains problem- and parameter-dependent. A projection-based extension further supports stable recovery of piecewise-constant diffusion coefficients in multi-subregion test cases.
Our results indicate that, as long as all subdomains share a portion of the boundary, the proposed CCBM-based Tikhonov regularization approach with a pick-a-point strategy enables stable and reliable reconstruction of diffusion parameters.
\end{abstract}


\bigskip
\tableofcontents 
\section{Introduction}\label{sec:Introduction}%
This work is concerned with the numerical recovery of a diffusion (or conductivity) coefficient in stationary elliptic partial differential equations. Let $\varOmega \subset \mathbb{R}^{d}$, $d \in \{2,3\}$, be a bounded, open, and simply connected domain with Lipschitz boundary $\varGamma := \partial \varOmega$. We consider the second-order elliptic boundary value problem
\begin{equation}\label{eq:general_elliptic_cost}
\mathcal{L}u \equiv
- \nabla \cdot (\alpha \nabla u) + \bb \cdot \nabla u + c\, u = Q 
\quad \text{in } \varOmega, 
\qquad 
u = f 
\quad \text{on } \varGamma,
\end{equation}
where $u:\varOmega \to \mathbb{R}$ denotes the state variable, and the coefficients $\alpha:\varOmega \to \mathbb{R}_{+}$, $\bb:\varOmega \to \mathbb{R}^{d}$, and $c:\varOmega \to \mathbb{R}$ represent diffusion, advection, and reaction effects, respectively.

The inverse problem consists of recovering the spatially varying diffusion coefficient $\alpha$ from boundary Cauchy data, namely the Dirichlet trace of $u$ and its normal flux on $\varGamma$. This is a coefficient identification problem for a steady-state elliptic equation.

Although \eqref{eq:general_elliptic_cost} allows for advection and reaction effects, we focus on the diffusion-dominated case
\begin{equation}\label{eq:diffusion_cost}
- \nabla \cdot (\alpha \nabla u) = Q
\quad \text{in } \varOmega, 
\qquad 
u = f 
\quad \text{on } \varGamma,
\end{equation}
which corresponds to the classical inverse conductivity problem (ICP), also known as electrical impedance tomography (EIT) \cite{KohnVogelius1984,KohnVogelius1985,Borcea2002,Uhlmann2009}. The diffusion coefficient $\alpha \in L^{\infty}(\varOmega)$ is assumed bounded and sufficiently regular. In this work, we consider both smooth and piecewise-defined diffusion coefficients, thereby complementing existing studies that typically emphasize one of these settings. Extensions to models incorporating advection and reaction effects are left for future work.

Boundary measurements are taken as Cauchy pairs $(f,g)$ of Dirichlet data and Neumann fluxes. In practice, such data are often collected via electrodes, and the Complete Electrode Model (CEM) accounts for electrode geometry and contact impedance \cite{ChengIsaacsonNewellGisser1989,SomersaloCheneyIsaacson1992,Hyvonen2004}. Electrode effects can significantly influence measurements and should be modeled or corrected \cite{Chen1990,SaulnierRossLiu2006,DarbasHeleineMendozaVelasco2021}. Although our analysis is formulated for continuum Cauchy data, it extends naturally to CEM settings \cite{WinklerRieder2014}.

Recovering spatially varying coefficients in elliptic equations is a nonlinear and severely ill-posed inverse problem \cite{LavrentievRomanovVasiliev1970,Lions1978}. Uniqueness from a single boundary measurement is generally not guaranteed, whereas global uniqueness holds when the full Dirichlet-to-Neumann map is available \cite{SylvesterUhlmann1987,SunUhlmann1991,Nachman1995,BrownUhlmann1997,AstalaPaivarinta2006}. For $d \geq 3$, uniqueness requires sufficient smoothness; in two dimensions, it extends to less regular coefficients, including $W^{2,p}(\varOmega)$ for $p>1$ and $L^\infty(\varOmega)$ \cite{Nachman1995,AstalaPaivarinta2006}. Stability estimates, typically logarithmic and conditional, have been established under additional assumptions \cite{Liu1997,Alessandrini1988,BarceloBarceloRuiz2001}.

In this work, we consider reconstruction from a single Cauchy pair $(f,g)$, for which identifiability is generally not guaranteed for arbitrary coefficients. A common approach is therefore to impose structural constraints on the unknown parameter, such as piecewise-constant or low-complexity models \cite{KohnVogelius1984}. A representative example is
\[
\alpha = 1 + \sigma \chi_{\omega},
\]
where $\chi_{\omega}$ denotes the characteristic function of an unknown subdomain $\omega \subset \varOmega$ and $\sigma$ is an unknown constant \cite{ChengIsaacsonNewellGisser1989,SomersaloCheneyIsaacson1992}. In view of these limitations, we focus on regularized numerical reconstruction from limited boundary data, with an emphasis on stability and robustness rather than strict uniqueness guarantees.

To address ill-posedness and measurement noise \cite{KohnVogelius1984}, we adopt a regularized optimization framework. Given $\mina, \maxa \in \mathbb{R}_{+}$, we define
\[
\Aad := \{\alpha \in L^\infty(\varOmega) \mid \mina < \alpha(x) < \maxa \ \text{ a.e. in } \varOmega\} \subset L^{\infty}_{+}(\varOmega)\footnote{Here, $L^{\infty}_{+}(\varOmega)$ denotes the set of essentially bounded functions on $\varOmega$ that are nonnegative almost everywhere.} 
\]
and introduce a regularization functional $\regu(\alpha)$ with parameter $\rho>0$. We denote by $\bar{\Aad}$ the closure of $\Aad$ in $L^\infty(\varOmega)$.

Let $u$ solve the forward problem \eqref{eq:diffusion_cost} for a given $\alpha \in \barAad$ and Dirichlet boundary data $f$. The Dirichlet-to-Neumann (DtN) map is defined as
\[
\Lambda_\alpha : H^{1/2}(\varGamma) \to H^{-1/2}(\varGamma), 
\quad \Lambda_\alpha f := \alpha{}\partial_\nn u \big|_\varGamma,
\]
which maps the Dirichlet trace $f$ to the corresponding Neumann flux. Its inverse, $\Lambda_\alpha^{-1}$, maps Neumann fluxes to the corresponding Dirichlet trace, with the additive constant fixed by, e.g., $\int_\varGamma v = 0$. 

Reconstruction is formulated as the minimization of suitable cost functionals:

\begin{enumerate}
    \item[(i)] \textit{Boundary-data tracking:}  
    Let $u \in H^1(\varOmega)$ solve \eqref{eq:diffusion_cost} with Dirichlet data $f$  (i.e., in the weak sense). Define the Neumann-tracking functional
    \begin{equation}\label{eq:tracking_Neumann_cost}
    J_{N,\rho}(\alpha) := \norm{\Lambda_\alpha f - g}_{L^2(\varGamma)}^2 + \rho{}\regu(\alpha),
    \end{equation}
    and consider the minimization problem $\min_{\alpha \in \barAad} J_{N,\rho}(\alpha)$.
    For pure Neumann data, one may equivalently track Dirichlet traces via the inverse DtN map $\Lambda_\alpha^{-1}$: 
    \begin{equation}\label{eq:tracking_Dirichlet_cost}
    J_{D,\rho}(\alpha) := \norm{\Lambda_\alpha^{-1} g - f}_{L^2(\varGamma)}^2 + \rho{}\regu(\alpha),
    \end{equation}
    where $v = \Lambda_\alpha^{-1} g$ solves the Neumann problem with $\intG{v} = 0$, and $\min_{\alpha \in \barAad} J_{D,\rho}(\alpha)$.

    Here, $\rho>0$ is a regularization parameter that balances the data misfit and the regularization term, introduced to mitigate ill-posedness and promote stable reconstructions.
    \item[(ii)] \textit{Kohn--Vogelius method:}  
    Let $\ud, \un \in H^1(\varOmega)$ solve the Dirichlet and Neumann problems, respectively:
    \begin{align}
    - \nabla \cdot (\alpha \nabla \ud) &= Q \quad \text{in } \varOmega, & \ud &= f \quad \text{on } \varGamma, \label{eq:state_ud} \\
    - \nabla \cdot (\alpha \nabla \un) &= Q\quad \text{in } \varOmega, & \alpha{}\partial_{\nn} \un &= g \quad \text{on } \varGamma, \quad \intG{\un} = 0, \label{eq:state_un}
    \end{align}
    where the integral constraint ensures uniqueness of $\un$. The Kohn--Vogelius functional is
    \begin{equation}\label{eq:Kohn_Vogelius_cost}
    J_{\mathrm{KV},\rho}(\alpha) := \intO{\alpha{}|\nabla(\ud - \un)|^2} + \intG{|\ud - \un|^2} + \rho{}\regu(\alpha),
    \end{equation}
    and the corresponding optimization problem reads $\min_{\alpha \in \barAad} J_{\mathrm{KV},\rho}(\alpha)$.

    \item[(iii)] \textit{Coupled complex-boundary method (CCBM):}  
    Let $w \in H^1(\varOmega;\mathbb{C})$ satisfy the complex Robin condition
    \[
    \alpha{}\partial_{\nn} w + i w = g + i f \quad \text{on } \varGamma.
    \]
    The CCBM functional is
    \begin{equation}\label{eq:CCBM_cost}
    J_{\mathrm{CCBM},\rho}(\alpha) := \intO{\Im\{w\}^2} + \rho{}\regu(\alpha),
    \end{equation}
    and the optimization problem is $\min_{\alpha \in \barAad} J_{\mathrm{CCBM},\rho}(\alpha)$.
\end{enumerate}

The CCBM approach, originally introduced in \cite{Chengetal2014} for an inverse source problem, 
has since been applied to a wide range of inverse problems. Early applications include 
conductivity reconstruction \cite{GongChenHan2017}, parameter identification in elliptic systems 
\cite{ZhengChengGong2020}, shape reconstruction \cite{Afraites2022}, and geometric inverse 
source problems \cite{AfraitesMasnaouiNachaoui2022}. It has also been employed in more complex 
settings, such as exterior Bernoulli problems \cite{Rabago2023b}, free surface flows 
\cite{RabagoNotsu2024}, inverse Cauchy problems for Stokes systems \cite{Ouiassaetal2022}, obstacle detection in Stokes flow \cite{RabagoAfraitesNotsu2025}, tumor localization \cite{Rabago2025}, approximation of unknown sources in a time fractional PDE \cite{HriziPrakashNovotny2025}, and recently in bioluminescence tomography \cite{WuGongGongZhangZhu2025}.

In this work, we build on these developments and revisit conductivity reconstruction 
\cite{GongChenHan2017} by introducing a gradient-weighted modification of the classical CCBM cost 
functional. We also propose corresponding numerical methods to improve reconstruction quality. 
In our numerical experiments, the modified approach with the additional gradient term showed indications of better control over high-frequency components and potentially improved stability.

Overall, in both free-boundary \cite{Rabago2023b} and free-surface \cite{RabagoNotsu2024} 
problems, CCBM has consistently shown improved stability and accuracy compared with standard 
least-squares methods, while maintaining computational costs comparable to Kohn--Vogelius formulations. 
In this work, we aim to highlight these advantages over classical methods by conducting extensive 
numerical experiments in the context of diffusion coefficient recovery.

In the general elliptic setting \eqref{eq:general_elliptic_cost}, the recovery of $\alpha$ extends naturally. Let the forward operator be
\[
\mathcal{F}: \barAad \to H^{-1/2}(\varGamma), \quad \mathcal{F}(\alpha) := \alpha{}\partial_{\nn} u \big|_{\varGamma},
\]
where $u$ solves \eqref{eq:general_elliptic_cost} with given Dirichlet data $f$. The inverse problem then reads:

\begin{problem}[Inverse Problem]\label{prob:inverse_problem} 
Given Cauchy data $(f,g)$ on $\varGamma$, find $\alpha \in \barAad$ such that $u$ solves \eqref{eq:general_elliptic_cost} and 
\[
\alpha{}\partial_{\nn} u \big|_{\varGamma} = g.
\]
\end{problem}

This problem can be addressed via regularized optimization in direct analogy with the methods presented for the diffusion-dominated case, \emph{i.e.}, boundary-data tracking, Kohn--Vogelius-type functionals, or the coupled complex-boundary method (CCBM) as described in \eqref{eq:tracking_Neumann_cost}--\eqref{eq:CCBM_cost}.

\alert{In this paper, the theoretical analysis and main results are presented for the \textit{modified} coupled complex-boundary method (CCBM), introduced in Subsection~\ref{subsec:modified_CCBM_based_optimization_formulation}. The boundary-data tracking and Kohn--Vogelius formulations are included mainly for numerical comparison and completeness. While not pursued here, similar theoretical results are expected to hold under suitable assumptions and minor technical adjustments.}

\alert{\textit{Position of the paper in the literature.}}

\begin{itemize}
    \item {To improve stability and mitigate noise amplification, we employ a Sobolev-gradient formulation, which smooths the descent direction and contributes to a more stable reconstruction process.}
    
    \item {When considering a piecewise-constant diffusion coefficient, we employ a first-order gradient-based descent method combined with a pick-a-point strategy, rather than the second-order sequential quadratic programming (SQP) methods used in previous work \cite{GongChenHan2017}. At each iteration, a representative point is selected within each subregion, and its corresponding coefficient value is assigned uniformly across that subregion. This yields a simple quasi-heuristic framework, which appears particularly suited to recovering piecewise diffusion coefficients when the number of subregions is relatively small. The approach could potentially be extended to related inverse problems with piecewise-defined coefficients, subject to appropriate assumptions and problem-specific modifications.}
    
    \item {Through extensive numerical experiments involving both smooth and piecewise diffusion coefficients, we observe that the CCBM yields stable and accurate reconstructions and, in the test cases considered, compares favorably with other established methods, including under moderate noise levels. These findings help to further substantiate the motivation for introducing the CCBM as a non-conventional approach to diffusion coefficient recovery. In this sense, the present work may be viewed as a complementary numerical study to \cite{GongChenHan2017, Chengetal2014}.}
    
    \item {We also propose a modest modification of the classical CCBM cost functional by incorporating a weighted tracking of the energy of the imaginary part of the solution, balancing its $L^2$- and $H^1$-type contributions (see \eqref{eq:modified_CCBM_cost_functional}). While the classical CCBM formulation already exhibits effective performance with an $L^2$-type tracking term, the numerical experiments presented in this work indicate that further improvements can be achieved by including an $H^1$ contribution with a sufficiently large weight.}
\end{itemize}

The numerical framework proposed in this work suggests potential applicability to related inverse problems involving piecewise-defined parameters, including absorption coefficient estimation and source identification, thereby further motivating the present study.

\alert{\textit{Paper organization.}}
The remainder of this paper is organized as follows.
Section~\ref{sec:Preliminaries} introduces the notation and formulates the inverse problem. 
The section also presents auxiliary analytical results, including continuity and differentiability of the forward map and key properties of the modified cost functional. 
Section~\ref{sec:convergence_analysis} is devoted to the theoretical analysis of existence, stability, and convergence under noisy measurements. Section~\ref{sec:grad_descent_algorithm} describes the numerical approximation strategy based on a Sobolev-gradient descent method. Section~\ref{sec:numerical_experiments} reports numerical experiments for both smooth and piecewise-constant diffusion coefficients. 
Finally, Section~\ref{sec:conclusion}  concludes the paper and discusses possible directions for future work.

\section{Preliminaries}\label{sec:Preliminaries}

\subsection{Notations and the problem setting}
Let $\varOmega \subset \mathbb{R}^{d}$, $d \in \{2,3\}$, be an open bounded domain with Lipschitz boundary $\varGamma := \partial \varOmega$. The Lebesgue spaces $L^{2}(\varOmega)$ and $L^{2}(\varGamma)$ are used throughout, and $H^{1/2}(\varGamma)$ denotes the trace space of $\VV$.
Unless otherwise stated, the same notation is used for real- and complex-valued functions.
Nevertheless, we ocassionally write a complex-valued functions space $X(\varOmega)$ as $X(\varOmega;\mathbb{C})$ only for emphasis.
For complex-valued functions, the $H^{1}(\varOmega)$ inner product is defined by $(u,v)_{1,\varOmega} := (\nabla u,\nabla v)_{0,\varOmega} + (u,v)_{0,\varOmega}$, where $(u,v)_{0,\varOmega} = \intO{u\,{v}}$. 
Throughout the paper, the complex conjugate is omitted from the notation whenever the meaning is clear from the context. The associated norm is given by $\norm{u}_{1,\varOmega}^{2} = (u,u)_{1,\varOmega}$.
Moreover, we denote by $\langle \cdot , \cdot \rangle_{\varGamma}$ the duality pairing between
$H^{-1/2}(\varGamma)$ and $H^{1/2}(\varGamma)$, which coincides with the
$L^{2}(\varGamma)$ inner product whenever both arguments are in $L^{2}(\varGamma)$.
Also, for simplicity, we write $\norm{\cdot}_{\infty,\varOmega}$ for the $L^{\infty}(\varOmega)$-norm.
Furthermore, we adopt a unified notation for norms of scalar- and vector-valued functions.
Lastly, throughout the paper, $C>0$ denotes a generic constant that may change from line to line.

We formulate the inverse problem of identifying the diffusion coefficient $\alpha$ within the framework of the coupled complex boundary method. In this setting, the inverse problem reduces to enforcing the vanishing of the imaginary part of the solution in $\varOmega$.

\begin{problem}\label{prob:main_problem} 
Determine $\alpha \in \barAad$ such that $\iu = 0$ in $\varOmega$, assuming $c \neq 0$, where 
$u = \ru + i \iu : \varOmega \to \mathbb{C}$ solves
\begin{equation}\label{eq:complex_BVP}
\mathcal{L} u = Q \quad \text{in } \varOmega, \qquad
\alpha \dn{u} + i u = g + i f \quad \text{on } \varGamma.
\end{equation}
\end{problem}

For analytical purposes, we impose the following conditions on the data and coefficients:
\begin{assumption}\label{ass:coefficients} 
\begin{enumerate}[label=(A\arabic*), ref=(A\arabic*)]
\item $Q \in L^{2}(\varOmega)$, $\bb \in W^{1,\infty}(\varOmega)^d$, and $c \in L^{\infty}(\varOmega)$.
\item $\bb \cdot \nn \ge 0$ a.e.\ on $\varGamma$.
\item $c - \tfrac12 \nabla \cdot \bb \ge \mina > 0$ a.e.\ in $\varOmega$.
\end{enumerate} 
\end{assumption}

Given Assumption~\ref{ass:coefficients}, we formulate the weak form of the complex boundary value problem~\eqref{eq:complex_BVP}.
Define $B : \VV \times \VV \to \mathbb{C}$ and $l : \VV \to \mathbb{C}$ by
\begin{align}
B(u,v;\alpha)
&=
(\alpha \nabla u, \nabla {v})_{0,\varOmega}
+
(\bb \cdot \nabla u, {v})_{0,\varOmega}
+
(c u, {v})_{0,\varOmega}
+
i{}\langle u, {v}\rangle_{\varGamma},\label{eq:sesquilinear_form} \\
l(v)
&=
(Q, {v})_{0,\varOmega}
+
\langle g, {v}\rangle_{\varGamma}
+
i{}\langle f, {v}\rangle_{\varGamma},\label{eq:linear_form} 
\end{align}
where $f,g, : \varGamma \to \mathbb{R}$, and $f,g \in H^{-1/2}(\varGamma)$.
\begin{problem}\label{prob:weak_form}
Find $u \in \VV$ such that $B(u,v;\alpha) = l(v)$, for all $v \in \VV$.
\end{problem}
Let $u = \ru + i \iu : \varOmega \to \mathbb{C}$.
The real and imaginary parts satisfy
\[
\begin{cases}
\begin{aligned}
-\nabla \cdot (\alpha \nabla \ru) + \bb \cdot \nabla \ru + c \ru &= Q && \text{in } \varOmega, \\
\alpha \dn{\ru} - \iu &= g && \text{on } \varGamma,
\end{aligned}
\end{cases}
\quad
\begin{cases}
\begin{aligned}
-\nabla \cdot (\alpha \nabla \iu) + \bb \cdot \nabla \iu + c \iu &= 0 && \text{in } \varOmega, \\
\alpha \dn{\iu} + \ru &= f && \text{on } \varGamma.
\end{aligned}
\end{cases}
\]

Under Assumption~\ref{ass:coefficients}, the sesquilinear form $B(\cdot,\cdot;\alpha)$ is coercive for all $\alpha \in \barAad$, and thus Problem~\ref{prob:weak_form} admits a unique weak solution (Lemma~\ref{lem:coercivity}). In this case, the imaginary part $\iu$ vanishes in $\varOmega$ if and only if $(u,\alpha)$ solves Problem~\ref{prob:inverse_problem}.
If $\iu$ is identically zero in $\varOmega$, then $u = \ru$ satisfies Problem~\ref{prob:inverse_problem} with the Cauchy data $f$ and $g$. Conversely, if $(u, \alpha)$ solves Problem~\ref{prob:inverse_problem}, the imaginary part of the corresponding complex solution is zero in $\varOmega$. 
As a result, the Dirichlet data of the original problem is recovered as $f = \ru$ on $\varGamma$, and Problem~\ref{prob:main_problem} is fully equivalent to Problem~\ref{prob:inverse_problem}.

The well-posedness of Problem~\ref{prob:weak_form}, which follows from the complex version of the Lax--Milgram lemma \cite[p.~376]{DautrayLionsv21998} (see also \cite[Lem.~2.1.51, p.~40]{SauterSchwab2011}), is stated below.

\begin{lemma}[Coercivity, well-posedness, and stability]\label{lem:coercivity}
Let Assumption~\ref{ass:coefficients} hold and let $\alpha \in \barAad$, so that
$\alpha \ge \mina > 0$ a.e.\ in $\varOmega$.
Then the sesquilinear form $B(\cdot,\cdot;\alpha)$ is coercive on $\VV$, i.e.,
there exists a constant $c_{B}>0$, independent of $\alpha$, such that
\[
\Re \{ B(v,v;\alpha) \}
\ge
c_{B}\,\norm{v}_{1,\varOmega}^{2},
\qquad
\forall v \in \VV.
\]
Consequently, for any $Q \in L^{2}(\varOmega)$ and $f,g \in H^{-1/2}(\varGamma)$, Problem~\ref{prob:weak_form} admits a unique solution
$u \in \VV$, which satisfies the stability estimate
\begin{equation}\label{eq:uniform_estimate}
\norm{u}_{1,\varOmega}
\le
\frac{C}{c_{B}}
\Big(
\norm{Q}_{0,\varOmega}
+
\norm{f}_{-1/2,\varGamma}
+
\norm{g}_{-1/2,\varGamma}
\Big),
\end{equation}
where $C>0$ depends only on $\varOmega$ and the coefficients $\bb$ and $c$,
but is independent of $\alpha$.
\end{lemma}

\begin{proof}
Let $v\in\VV$. Since $\Re\{i(v,v)_{0,\varGamma}\}=0$, we have
\[
\Re\{ B(v,v;\alpha) \}
=
(\alpha\nabla v,\nabla v)_{0,\varOmega}
+
\Re\{(\bb\cdot\nabla v,v)_{0,\varOmega}\}
+
(c v,v)_{0,\varOmega}.
\]
Using integration by parts and Assumption~\ref{ass:coefficients}(A1), we obtain
\[
\Re\{(\bb\cdot\nabla v,v)_{0,\varOmega}\}
=
-\frac12\big((\nabla\cdot\bb)v,v\big)_{0,\varOmega}
+
\frac12(\bb\cdot\nn v,v)_{0,\varGamma}.
\]
By Assumption~\ref{ass:coefficients}(A2), the boundary term is nonnegative, and
by Assumption~\ref{ass:coefficients}(A3),
\[
(c v,v)_{0,\varOmega}
-
\frac12\big((\nabla\cdot\bb)v,v\big)_{0,\varOmega}
\ge
\mina\,\norm{v}_{L^{2}(\varOmega)}^{2}.
\]
Moreover, since $\alpha \ge \mina$ a.e.\ in $\varOmega$,
$
(\alpha\nabla v,\nabla v)_{0,\varOmega}
\ge
\mina\,\norm{\nabla v}_{L^{2}(\varOmega)^d}^{2}$.
Collecting the estimates yields
\[
\Re\{ B(v,v;\alpha) \}
\ge
\mina\Big(
\norm{\nabla v}_{L^{2}(\varOmega)^d}^{2}
+
\norm{v}_{L^{2}(\varOmega)}^{2}
\Big)
=
\mina\,\norm{v}_{1,\varOmega}^{2}.
\]
Setting $c_{B} := \mina$ completes the coercivity proof.

The boundedness of the sesquilinear form $B(\cdot,\cdot;\alpha)$ and of the linear
functional, as well as the stability estimate, follow from standard arguments and
are therefore omitted. Uniqueness follows from coercivity by standard arguments.
\end{proof}
 

Under Assumption~\ref{ass:coefficients}, Lemma~\ref{lem:coercivity} guarantees that,
for every $\alpha \in \barAad$, Problem~\ref{prob:weak_form} admits a unique solution
$u \in \VV$ satisfying the uniform stability estimate~\eqref{eq:uniform_estimate}.
This well-posedness result allows us to introduce the associated forward
(parameter-to-state) map.

We therefore define the forward map $F$ by
\begin{equation}\label{eq:mapping_F}
\alpha \in \barAad
\longmapsto
F(\alpha) := u(\alpha) \in \VV,
\end{equation}
where $u(\alpha)$ denotes the unique solution of
Problem~\ref{prob:weak_form} corresponding to $\alpha$.
When no ambiguity arises, the dependence on $\alpha$ is omitted and we simply write $u$.

\subsection{Continuity and differentiability of the forward map}
The continuity of the forward map is stated next.
\begin{proposition}\label{prop:continuity_operator_F}
The forward map $F \colon \barAad \subset L^\infty(\varOmega) \to \VV$ is Lipschitz continuous.  
More precisely, there exists a constant $C>0$, independent of $\alpha_1,\alpha_2 \in \barAad$, such that
\begin{equation}\label{eq:lipschitzF}
\norm{F(\alpha_1)-F(\alpha_2)}_{1,\varOmega}
\le
\frac{C}{c_B^2}\, \norm{\alpha_1-\alpha_2}_{L^\infty(\varOmega)}, 
\quad \forall \alpha_1,\alpha_2 \in \barAad.
\end{equation}
Equivalently, the Fréchet derivative of $F$ satisfies the operator-norm bound
\[
\norm{\Fprime(\alpha)}_{\mathcal{B}(L^\infty(\varOmega),\VV)} 
\le \frac{C}{c_B^2}, \quad \forall \alpha \in \barAad.
\]
\end{proposition}

We next study the differentiability of the forward map
$F \colon \barAad \to \VV$ defined in \eqref{eq:mapping_F}.
Let $\alpha \in \Aad$ and let ${\dela} \in L^\infty(\varOmega)$ be such that
$\alpha + {\dela} \in \barAad$.
We define $\delta F := F(\alpha + {\dela}) - F(\alpha) \in \VV$.
By definition of $F$ and subtracting the corresponding variational problems,
$\delta F$ satisfies
\begin{equation}\label{eq:weakform_increment}
B(\delta F,v;\alpha+{\dela})
=
-({\dela}\nabla u,\nabla{v})_{0,\varOmega},
\qquad
\forall v \in \VV,
\end{equation}
where $u=F(\alpha)$.

This formulation allows us to identify the derivative of $F$ by passing to the
limit as ${\dela} \to 0$.

\begin{proposition}\label{prop:first_derivative_of_F}
Let $\alpha \in \Aad$.
Then the forward map $F$ is Fréchet differentiable at $\alpha$.
For any ${\dela} \in L^\infty(\varOmega)$, the derivative
$\uprime := \Fprime(\alpha){\dela} \in \VV$ is the unique solution of
\begin{equation}\label{eq:first_derivative}
B(\uprime,v;\alpha)
=
-({\dela}\nabla u,\nabla{v})_{0,\varOmega},
\qquad
\forall v \in \VV,
\end{equation}
where $u=F(\alpha)$.
Moreover, the following estimate holds:
\begin{equation}\label{eq:first_derivative_bound}
\norm{\Fprime(\alpha){\dela}}_{1,\varOmega}
\le
\frac{C}{c_B^{2}}\,
\norm{{\dela}}_{L^\infty(\varOmega)} .
\end{equation}
\end{proposition}

\begin{remark}
For $\alpha \in \dAad := \barAad \setminus \Aad$, denote by $\widetilde{\Aad}$ the largest subset of $L^{\infty}(\varOmega)$ such that 
for all ${\dela} \in \widetilde{\Aad}$ and sufficiently small $t>0$, we have $\alpha + t{}{\dela} \in \barAad$.  
Then the map $F$ is also differentiable at each $\alpha \in \dAad$, and $\norm{\Fprime(\alpha)}_{\mathcal{B}(\widetilde{\Aad}, \VV)} $is bounded, where $\mathcal{B}(\widetilde{\Aad}, \VV)$ denotes the space of bounded linear operators from 
$\widetilde{\Aad}$ to $\VV$.
\end{remark}

We now turn to the second-order differentiability of the forward map $F$.
We define $\delta^{2}F
:=
\Fprime(\alpha+{\dela}_1){\dela}_2
-
\Fprime(\alpha){\dela}_2
\in \VV .
$
Using Proposition~\ref{prop:first_derivative_of_F} and subtracting the corresponding
linearized variational problems, we obtain
\begin{equation}\label{eq:weakform_second_increment}
\begin{aligned}
B(\delta^{2}F,v;\alpha+{\dela}_1)
&=
-({\dela}_2\nabla[F(\alpha+{\dela}_1)-F(\alpha)],\nabla{v})_{0,\varOmega} \\
&\quad
-({\dela}_1\nabla \Fprime(\alpha){\dela}_2,\nabla{v})_{0,\varOmega},
\qquad
\forall v \in \VV .
\end{aligned}
\end{equation}
Passing to the limit as $\norm{{\dela}_1}_{L^\infty(\varOmega)} \to 0$
and using the stability and continuity properties of $F$ and $\Fprime$ (Propositions~\ref{prop:continuity_operator_F}--~\ref{prop:first_derivative_of_F}),
we obtain the second derivative of $F$ at $\alpha$.

\begin{proposition}\label{prop:second_derivative}
Let $\alpha \in \Aad$.
Then the forward map $F$ is twice Fr\'echet differentiable at $\alpha$
with respect to admissible directions.
For any admissible
${\dela}_1,{\dela}_2 \in L^\infty(\varOmega)$,
the second derivative $\udprime := \Fdprime(\alpha)[{\dela}_1,{\dela}_2] \in \VV$ is the unique solution of
\begin{equation}\label{eq:second_derivative}
B(\udprime,v;\alpha)
=
-({\dela}_2\nabla \Fprime(\alpha){\dela}_1,\nabla{v})_{0,\varOmega}
-({\dela}_1\nabla \Fprime(\alpha){\dela}_2,\nabla{v})_{0,\varOmega},
\quad
\forall v \in \VV ,
\end{equation}
where $u=F(\alpha)$.
Moreover, there exists a constant $C>0$, independent of $\alpha$, such that
\begin{equation}\label{eq:second_derivative_bound}
\norm{\Fdprime(\alpha)[{\dela}_1,{\dela}_2]}_{1,\varOmega}
\le
\frac{C}{c_B^{3}}\,
\norm{{\dela}_1}_{L^\infty(\varOmega)}
\norm{{\dela}_2}_{L^\infty(\varOmega)} .
\end{equation}
\end{proposition}

With the continuity and differentiability of the forward map established (see Propositions~\ref{prop:continuity_operator_F}--\ref{prop:second_derivative}, whose proofs follow standard arguments and are omitted), we can now reformulate Problem~\ref{prob:main_problem} as a regularized optimization problem suitable for numerical approximation via Tikhonov regularization.
\subsection{CCBM-based optimization reformulation of the inverse problem}\label{subsec:modified_CCBM_based_optimization_formulation}

For $\alpha \in \barAad$, let $u = F(\alpha) = \ru(\alpha) + i \iu(\alpha) \in \VV$ be the unique weak solution of~\eqref{eq:complex_BVP}.  
With a fixed $\rho > 0$, the inverse problem reduces to minimizing the regularized functional
\begin{equation}\label{eq:Tikhonov_regularization}
\Je(\alpha) := J(\alpha) + \regu_{\rho}(\alpha).
\end{equation}

Note that $\dela \in L^\infty(\varOmega) \subset L^2(\varOmega)$ (so that $\norm{\dela}_{0,\varOmega} \le \abs{\varOmega}^{1/2} \norm{\dela}_{L^\infty(\varOmega)}$; see \cite[Prop.~6.12]{Folland1999}).
Accordingly, unless stated otherwise, we choose the (Tikhonov) regularization term as
\begin{equation}\label{eq:Tikhonov_regularization_term}
\regu_{\rho}(\alpha) 
:= \frac{1}{2} \rho \regu(\alpha) 
:= \frac{1}{2} \rho \norm{\alpha}_{0,\varOmega}^{2}.
\end{equation}

We further introduce a modified CCBM-type misfit functional defined by
\begin{equation}\label{eq:modified_CCBM_cost_functional}
J(\alpha) 
:= \frac{1}{2}
\Big(
w_0 \norm{\iu(\alpha)}_{0,\varOmega}^{2}
+
w_1 \norm{\nabla \iu(\alpha)}_{0,\varOmega}^{2}
\Big),
\end{equation}
where $w_0,w_1 > 0$ are fixed weighting parameters.

The inclusion of the gradient term $\norm{\nabla \iu(\alpha)}_{0,\varOmega}^{2}$ in 
\eqref{eq:modified_CCBM_cost_functional} constitutes a novelty of this work.  
For $w_1>0$, the weighted misfit $J$ controls the $H^1(\varOmega)$-norm of 
$\iu(\alpha)$, i.e.,
\[
J(\alpha) \;\ge\; \frac{\min\{w_0, w_1\} }{2}\norm{\iu(\alpha)}_{1, \varOmega}^2, 
\qquad \forall \alpha\in\barAad,
\]
while $w_1=0$ provides control only in $L^2(\varOmega)$.  
In the numerical examples considered here, moderately large $w_1$ generally improves 
reconstruction quality by penalizing oscillatory components and promoting smoother 
states, though excessively large $w_1$ may reduce resolution.

For technical reasons, in particular to ensure strict convexity of the objective functional (see Proposition~\ref{prop:strict_convexity}), we impose the following assumption.

\begin{assumption}\label{ass:key_assumption}
Assume $\SubsetK \subset \Aad$ is finite-dimensional, closed, and convex.
\end{assumption}

A case of particular interest is when $\SubsetK$ consists of piecewise constant diffusion coefficients on a prescribed partition of $\varOmega$; other examples include bounded polynomial spaces $\{ \alpha \in P_m(\varOmega) \mid \underline{\alpha} \le \alpha \le \overline{\alpha} \}$ and bounded trigonometric expansions $\{ \alpha \in \mathcal{T}_N \mid \underline{\alpha} \le \alpha \le \overline{\alpha} \}$, where $\mathcal{T}_N := \operatorname{span} \{ \sin(k\cdot x), \cos(k\cdot x) \mid |k| \le N \}$.

Assumption~\ref{ass:key_assumption} is natural and may be viewed as a relaxation framework for convex functionals; consequently, standard Sobolev-gradient convergence theory, such as that in \cite[Ch.~V]{Glowinski1984}, applies.

We consider the following problem.
\begin{problem}\label{prob:minimization_problem}
Find $\alpharho \in \SubsetK$ such that $\alpharho = \arg\min_{\alpha \in \SubsetK} \Je(\alpha)$.
\end{problem}

\subsection{Lipschitz continuity of the modified CCBM cost  functional}
We show that the modified CCBM functional $\Je$ in \eqref{eq:Tikhonov_CCBM} is Lipschitz continuous on bounded subsets of admissible coefficients (Proposition~\ref{prop:Lipschitz_Tikhonov}), and that its first G\^ateaux and Fr\'echet derivative is also Lipschitz continuous with respect to $\alpha$ (Proposition~\ref{prop:Lipschitz_first_derivative}). These properties follow from the Lipschitz continuity and differentiability of the forward map (Propositions~\ref{prop:continuity_operator_F} and~\ref{prop:first_derivative_of_F}) and uniform bounds on the state and coefficients (Lemma~\ref{lem:coercivity}).
\begin{proposition}\label{prop:Lipschitz_Tikhonov}
Let $\Je$ be the functional defined in \eqref{eq:Tikhonov_CCBM}. 
Under the assumptions of Lemma~\ref{lem:coercivity} and Proposition~\ref{prop:continuity_operator_F}, $\Je$ is Lipschitz continuous on $\SubsetK \subset \barAad \subset L^\infty(\varOmega)$. 
That is, there exists a constant $L>0$,  independent of $\alpha_{1}, \alpha_{2} \in \SubsetK$, such that
\[
\abs{ \Je(\alpha_1) - \Je(\alpha_2) }
\le L \, \norm{\alpha_1 - \alpha_2}_{L^\infty(\varOmega)}, 
\quad \forall \alpha_1, \alpha_2 \in \SubsetK.
\]
\end{proposition}
\begin{proof}
Let $\alpha_1, \alpha_2 \in \SubsetK \subset \barAad$, and denote 
$u_j := \iu(\alpha_j) \in \VV$, $j=1,2$, 
where $u_j$ is the unique solution of Problem~\ref{prob:weak_form} corresponding to $\alpha_j$.

Since $u_j$ are complex-valued, we use the identity
\[
\norm{\varphi}_{0,\varOmega}^{2} - \norm{\psi}_{0,\varOmega}^{2} = \Re \{ \inS{\varphi + \psi, \varphi - \psi} \}, 
\quad \varphi, \psi \in L^2(\varOmega;\mathbb{C}),
\]
together with the Cauchy--Schwarz inequality, to obtain
\[
\abs{ \norm{u_1}_{0,\varOmega}^{2} - \norm{u_2}_{0,\varOmega}^{2} }
\le \norm{u_1 + u_2}_{0,\varOmega} \, \norm{u_1 - u_2}_{0,\varOmega}.
\]

Similarly, for the gradient term (componentwise),
\[
\abs{ \norm{\nabla u_1}_{0,\varOmega}^{2} - \norm{\nabla u_2}_{0,\varOmega}^{2} }
\le \norm{\nabla u_1 + \nabla u_2}_{0,\varOmega} \, \norm{\nabla u_1 - \nabla u_2}_{0,\varOmega}.
\]

By Proposition~\ref{prop:continuity_operator_F}, the forward map 
$F: \alpha \mapsto \iu(\alpha)$ is Lipschitz in $L^\infty$, and since $\SubsetK \subset \Aad$ is bounded in $L^\infty(\varOmega)$, there exists $C_F>0$ such that
\[
\norm{\iu(\alpha_1) - \iu(\alpha_2)}_{1, \varOmega} \le C_F \, \norm{\alpha_1 - \alpha_2}_{L^\infty(\varOmega)}.
\]

Moreover, Lemma~\ref{lem:coercivity} gives uniform $H^1$ stability of $u_j$, so there exists $C_\SubsetK>0$ with 
\(\norm{u_j}_{1, \varOmega} \le C_\SubsetK\) for all $\alpha_j \in \SubsetK$.

Combining these bounds, we obtain
\[
\abs{ \Je(\alpha_1) - \Je(\alpha_2) }
\le \big( (w_0 + w_1) C_\SubsetK C_F + \rho \, C_\SubsetK \big) \, \norm{\alpha_1 - \alpha_2}_{L^\infty(\varOmega)},
\]
showing that $\Je$ is Lipschitz continuous on $\SubsetK$.
\end{proof}

\begin{proposition}\label{prop:Lipschitz_first_derivative}
Let $\Je$ be the functional defined in \eqref{eq:Tikhonov_CCBM}, and let 
$\Je'(\alpha){\dela}$ be its first G\^ateaux derivative in the direction ${\dela} \in L^\infty(\varOmega)$, as in \eqref{eq:first_derivative_Tikhonov}--\eqref{eq:first_derivative_Tikhonov_forms}.  
Under the assumptions of Lemma~\ref{lem:coercivity} and Proposition~\ref{prop:first_derivative_of_F}, $\Je'$ is Lipschitz continuous with respect to $\alpha$ on $\SubsetK \subset \barAad$.  
That is, there exists a constant $L'>0$, independent of $\alpha_1, \alpha_2 \in \SubsetK$ and of $\dela \in L^\infty(\varOmega)$,  such that
\[
\abs{\Je'(\alpha_1){\dela} - \Je'(\alpha_2){\dela}}
\le L^{\prime}\, \norm{\alpha_1 - \alpha_2}_{L^\infty(\varOmega)} \, \norm{\dela}_{L^\infty(\varOmega)}, 
\quad \forall \alpha_1, \alpha_2 \in \SubsetK.
\]
\end{proposition}

\begin{proof}
Let $\alpha_1, \alpha_2 \in \SubsetK \subset \barAad$, and denote
$u_j := \iu(\alpha_j) \in \VV$ and $u_j^{\prime}:= \Fprime(\alpha_j){\dela} \in \VV$, $j=1,2$.  
Using the definition \eqref{eq:first_derivative_Tikhonov_forms} and the Cauchy--Schwarz inequality, we obtain
\[
\begin{aligned}
&\abs{J^{\prime}(\alpha_1){\dela} - J^{\prime}(\alpha_2){\dela}}\\
&\qquad = \left| w_0 (\iu(\alpha_1), u_{i, 1}^{\prime})_{0,\varOmega} - w_0 (\iu(\alpha_2), u_{i, 2}^{\prime})_{0,\varOmega} \right.\\
&\qquad \qquad \left. + w_1 (\nabla \iu(\alpha_1), \nabla u_{i, 1}^{\prime})_{0,\varOmega} - w_1 (\nabla \iu(\alpha_2), \nabla u_{i, 2}^{\prime})_{0,\varOmega} \right| \\
&\qquad \le w_0 \big( \norm{\iu(\alpha_1) - \iu(\alpha_2)}_{0,\varOmega} \, \norm{u_{i, 1}^{\prime}}_{0,\varOmega} + \norm{\iu(\alpha_2)}_{0,\varOmega} \, \norm{u_{i, 1}^{\prime} - u_{i, 2}^{\prime}}_{0,\varOmega} \big) \\
&\qquad \qquad + w_1 \big( \norm{\nabla \iu(\alpha_1) - \nabla \iu(\alpha_2)}_{0,\varOmega} \, \norm{\nabla u_{i, 1}^{\prime}}_{0,\varOmega} + \norm{\nabla \iu(\alpha_2)}_{0,\varOmega} \, \norm{\nabla u_{i, 1}^{\prime} - \nabla u_{i, 2}^{\prime}}_{0,\varOmega} \big).
\end{aligned}
\]

By Proposition~\ref{prop:continuity_operator_F} and Proposition~\ref{prop:first_derivative_of_F}, the forward map and its derivative satisfy
\begin{align*}
\norm{\iu(\alpha_1) - \iu(\alpha_2)}_{1, \varOmega} &\le C_F \, \norm{\alpha_1 - \alpha_2}_{L^\infty(\varOmega)}, \\
\norm{\Fprime(\alpha_1){\dela} - \Fprime(\alpha_2){\dela}}_{1, \varOmega} &\le C_F^{\prime}\, \norm{\alpha_1 - \alpha_2}_{L^\infty(\varOmega)} \, \norm{\dela}_{L^\infty(\varOmega)}.
\end{align*}
Moreover, Lemma~\ref{lem:coercivity} gives uniform bounds
$\norm{u_j}_{1, \varOmega} \le C_\SubsetK$ and $\norm{u_j^{\prime}}_{1, \varOmega} \le C_\SubsetK \norm{\dela}_{L^\infty(\varOmega)}$ for all $\alpha_j \in \SubsetK$.  
Combining these estimates, we obtain
\[
\abs{\Je'(\alpha_1){\dela} - \Je'(\alpha_2){\dela}}
\le L^{\prime}\, \norm{\alpha_1 - \alpha_2}_{L^\infty(\varOmega)} \, \norm{\dela}_{L^\infty(\varOmega)},
\]
with $L^{\prime}:= (w_0 + w_1) C_\SubsetK C_F'$, showing that $\Je'$ is Lipschitz continuous with respect to $\alpha$ on $\SubsetK$.
\end{proof}

\subsection{Convexity and characterization of the minimizer}
In what follows, we establish convexity of $\Je$ and uniqueness of the minimizer.  
Let $\alpha \in \SubsetK$ and let ${\dela} \in \SubsetK$ be such that $\alpha + {\dela} \in \SubsetK$.  
Denote by $\uprime = \Fprime(\alpha){\dela}$ and $\udprime = \Fdprime(\alpha)[{\dela},{\dela}]$ the first- and second-order variations of the state.

The first Gâteaux derivative of $\Je$ in direction ${\dela}$ is given by
\begin{equation}
\Je'(\alpha){\dela}
= J^{\prime}(\alpha){\dela}  + \regu_{\rho}'(\alpha){\dela}, \label{eq:first_derivative_Tikhonov}
\end{equation}
where
\begin{equation}
J^{\prime}(\alpha){\dela} 
:= w_0{}(\iu(\alpha), \iuprime)_{0,\varOmega}
+
w_1{}(\nabla \iu(\alpha), \nabla \iuprime)_{0,\varOmega},
\qquad
\regu_{\rho}'(\alpha){\dela} := \rho{}(\alpha,{\dela})_{0,\varOmega},
\label{eq:first_derivative_Tikhonov_forms}
\end{equation}
and $\iuprime = \Im\{\uprime\}$.

The second derivative satisfies
\begin{equation}\label{eq:second_derivative_Tikhonov}
\begin{aligned}
\Je''(\alpha)[{\dela},{\dela}]
&=
w_0 \Big(
\norm{\iuprime}_{0,\varOmega}^{2}
+
(\iu(\alpha), \iudprime)_{0,\varOmega}
\Big)
\\
&\qquad +
w_1 \Big(
\norm{\nabla \iuprime}_{0,\varOmega}^{2}
+
(\nabla \iu(\alpha), \nabla \iudprime)_{0,\varOmega}
\Big)
+
\rho \norm{{\dela}}_{0,\varOmega}^{2},
\end{aligned}
\end{equation}
where $\iudprime = \Im\{\udprime\}$ solves the second-order sensitivity problem~\eqref{eq:second_derivative}.

\begin{proposition}\label{prop:strict_convexity}
There exists $\rho_{0}>0$, independent of $\alpha \in \SubsetK$, such that for all $\rho \ge \rho_{0}$, the functional $\Je$ is strictly convex on $\SubsetK$.
\end{proposition}
\begin{proof}
Let $\alpha, {\dela} \in \SubsetK$ be arbitrary and ${\dela} \not\equiv 0$.  
We prove strict convexity by showing coercivity in the $L^{\infty}(\varOmega)$-sense of the second derivative $\Jedprime(\alpha)$ in \eqref{eq:second_derivative_Tikhonov}.

From \eqref{eq:uniform_estimate} and \eqref{eq:second_derivative_bound}, there exist constants $c_0, c_1 >0$, independent of $\alpha$, such that
\[
\norm{\iu(\alpha)}_{0,\varOmega}, \ \norm{\nabla \iu(\alpha)}_{0,\varOmega} \le c_0, 
\qquad
\norm{\iudprime}_{0,\varOmega}, \ \norm{\nabla \iudprime}_{0,\varOmega} \le c_1 \norm{{\dela}}_{\infty,\varOmega}^2.
\]
On the other hand, since $\SubsetK$ is finite-dimensional, there exists $K>0$ such that $\norm{{\dela}}_{0,\varOmega} \ge K \norm{{\dela}}_{\infty,\varOmega}$.

Define $J_0 := w_0 \norm{\iuprime}_{0,\varOmega}^{2} + w_1 \norm{\nabla \iuprime}_{0,\varOmega}^{2}$ and set $c_2 := c_0 c_1 >0$.  
Then, the second derivative satisfies
\begin{align*}
\Jedprime(\alpha)[{\dela},{\dela}]
&= J_0
+ w_0 (\iu(\alpha), \iudprime)_{0,\varOmega}
+ w_1 (\nabla \iu(\alpha), \nabla \iudprime)_{0,\varOmega}
+ \rho \norm{{\dela}}_{0,\varOmega}^{2} \\
&\ge J_0
- (w_0 + w_1) c_2 \norm{{\dela}}_{\infty,\varOmega}^{2}
+ \rho \norm{{\dela}}_{0,\varOmega}^{2} \\
&\ge (\rho K^2 - (w_0 + w_1)c_2) \norm{{\dela}}_{\infty,\varOmega}^{2}.
\end{align*}

Now, to ensure a uniform positive lower bound, we first fix any constant $\epsilon > 0$ and define
\[
\rho_{0} 
:= \frac{(w_0 + w_1)c_2}{K^2} + \epsilon > 0.
\]
Then, for all $\rho \ge \rho_{0}$, we have
\[
\Jedprime(\alpha)[{\dela},{\dela}] 
\ge (\rho K^2 - (w_0 + w_1)c_2) \norm{{\dela}}_{\infty,\varOmega}^{2} 
\ge \epsilon K^2 \norm{{\dela}}_{\infty,\varOmega}^{2} > 0,
\]
for all $0 \not\equiv{\dela} \in \SubsetK$.
Hence, $\Je$ is strictly convex on $\SubsetK$ uniformly in $\alpha$, with a uniform lower bound $\epsilon K^2$ on the coercivity.  
\end{proof}
Hereinafter, Proposition~\ref{prop:strict_convexity} is assumed to hold, unless explicitly stated otherwise.
\begin{remark}\label{rem:effect_of_rho}
The regularization parameter $\rho > 0$ cannot be chosen arbitrarily large or small.  
If $\rho$ is too large, the regularization term dominates, and the functional may become insensitive to $\alpha$, potentially losing strict convexity (or even convexity in general). This behavior occurs not only for the CCBM functional but also for the boundary-data tracking functionals $J_{N,\rho}$ and $J_{D,\rho}$, as well as the Kohn--Vogelius functional $J_{\mathrm{KV},\rho}$; see subsection~\ref{subsec:effect_of_rho} for a numerical illustration of this effect.
If $\rho$ is too small, coercivity is weakened and the estimate involving $\epsilon$ in the strict convexity proof (Proposition~\ref{prop:strict_convexity}) may fail.  
Thus, the admissible range of $\rho$ is restricted, which in turn constrains the choice of $\epsilon$ for which strict convexity holds.
\end{remark}

We now present a variational characterization of the solution to Problem~\ref{prob:minimization_problem}, which can be viewed as a first-order optimality condition.  
Thanks to the strict convexity of the functional, this condition uniquely identifies the minimizer; however, it should be interpreted with care in the context of approximation or numerical errors.\footnote{Since $\rho$ is the regularization parameter, it cannot be chosen arbitrarily large.  Consequently, convexity of the objective functional cannot be expected in general.}

\begin{theorem}\label{thm:optimality_condition}
Let $\SubsetK$ satisfy Assumption~\ref{ass:key_assumption}, and suppose Proposition~\ref{prop:strict_convexity} is valid.  
Then Problem~\ref{prob:minimization_problem} admits a unique solution $\alpharho \in \SubsetK$, which depends continuously on the data.  
Moreover, $\alpharho$ is characterized by the variational inequality
\begin{equation}\label{eq:variational_inequality}
(J^{\prime}(\alpharho) 
+ \rho R^{\prime}(\alpharho) ,\; \dela - \alpharho )_{0,\varOmega} =
\bigl( \nabla \ru \cdot \nabla \ip - \nabla \iu \cdot \nabla \rp + \rho \alpharho,\; \dela - \alpharho \bigr)_{0,\varOmega} \ge 0,
\quad
\forall \dela \in \SubsetK,
\end{equation}
where $u = \ru + i \iu = F(\alpharho)$ and $p = \rp + i \ip \in \VV$ uniquely solves the adjoint problem
\begin{equation}\label{eq:adjoint_problem}
\left\{
\begin{aligned}
- \nabla \cdot (\alpharho \nabla p) 
- \nabla \cdot (\bb{}p) 
+ c{}p
&= w_0{}\iu - w_1{}\Delta \iu & & \text{in } \varOmega, \\[1mm]
\alpharho{}\dn{p} 
+ \bb \cdot \nn{}p 
- i{}p 
&= w_1{} \dn{\iu} & & \text{on } \varGamma.
\end{aligned}
\right.
\end{equation}
\end{theorem}

Let us define $B^{\ast} : \VV \times \VV \to \mathbb{C}$ and $l^{\ast} : \VV \to \mathbb{C}$ by
\begin{align}
B^{\ast}(p,v;\alpharho)
&= (\alpharho \nabla p, \nabla v)_{0,\varOmega}
 + (p, \bb \cdot \nabla v)_{0,\varOmega} 
 + (c p, v)_{0,\varOmega}
- i{}\langle p, v\rangle_{\varGamma},
\label{eq:adjoint_sesquilinear_form}
\\
l^{\ast}(v)
&= w_0(\iu, v)_{0,\varOmega}
+ w_1(\nabla \iu, \nabla v)_{0,\varOmega}.
\label{eq:adjoint_linear_form}
\end{align}
Then, the weak formulation of the adjoint system~\eqref{eq:adjoint_problem} can be stated as follows:
\begin{problem}\label{prob:adjoint_weak_form}
Find $p \in \VV$ such that $B^{\ast}(p,v;\alpha) = l^{\ast}(v)$, for all $v \in \VV$.
\end{problem}
Existence and uniqueness of the solution of Problem~\ref{prob:adjoint_weak_form} again follow from the complex Lax--Milgram lemma.

\begin{proof}[Proof of Proposition~\ref{thm:optimality_condition}]
Since $J_\rho$ is strictly convex on $\SubsetK$, standard results on convex minimization (Theorem~5.3.19 in~\cite[p.~233]{AtkinsonHan2009}) guarantee a unique minimizer $\alpharho \in \SubsetK$, characterized by $J_\rho'(\alpharho)(\dela - \alpharho) \ge 0$, for all $\dela \in \SubsetK$.
Setting $\widehat{\beta} := \dela - \alpharho \in \SubsetK$, and using~\eqref{eq:first_derivative_Tikhonov} and~\eqref{eq:first_derivative_Tikhonov_forms}, we have
\begin{equation}
\label{eq:first_inequality}
J_\rho'(\alpharho)\widehat{\beta}
= J^{\prime}(\alpharho){\widehat{\beta}}  + \regu_{\rho}'(\alpharho){\widehat{\beta}}  \ge 0,
\end{equation}
where $\iuprime = \Im\{\uprime\} = \Im\{\Fprime(\alpharho)\widehat{\beta}\}$ solves the linearized state problem (Proposition~\ref{prop:first_derivative_of_F}).

Let $p \in \VV$ solve Problem~\ref{prob:adjoint_weak_form}.
Using the forms defined in \eqref{eq:sesquilinear_form}, \eqref{eq:adjoint_sesquilinear_form}, and \eqref{eq:adjoint_linear_form}, the linearized state problem and its adjoint can be written concisely (after selecting suitable test functions) as
\[
B(\uprime, p; \alpharho) = -(\widehat{\beta} \nabla u, \nabla p)_{0,\varOmega}, \qquad
B(\uprime, p; \alpharho) = \overline{B^{\ast}(p, \uprime; \alpharho)} =\overline{ l^{\ast}(\uprime) }.
\]

Using integration by parts, the terms involving $\bb$ cancel exactly. Taking the imaginary part yields
\[
J^{\prime}(\alpharho){\widehat{\beta}}  
= \Im\{ w_0(\iu, \uprime)_{0,\varOmega} + w_1(\nabla \iu, \nabla \uprime)_{0,\varOmega} \}
= \intO{ \widehat{\beta} ( \nabla \ru \cdot \nabla \ip - \nabla \iu \cdot \nabla \rp ) }.
\]
Combining this with inequality \eqref{eq:first_inequality} yields the variational characterization \eqref{eq:variational_inequality}, thereby proving the theorem.
\end{proof}

\section{Convergence Analysis under Noisy Measurements}\label{sec:convergence_analysis}
In the presence of noisy boundary data, we consider the Tikhonov-regularized inverse problem in the finite-dimensional subspace $\SubsetK$ (Assumption~\ref{ass:key_assumption}) and summarize its key properties, including existence of minimizers, continuous dependence on the data, and convergence to the exact solution under Assumption~\ref{ass:exact_solution_in_K}.

\begin{assumption}\label{ass:exact_solution_in_K}
There exists a unique $\alpha^\dagger \in \SubsetK$ such that 
$u^\dagger = F(\alpha^\dagger)$ satisfies $\iu^\dagger = 0$ in $\varOmega$, i.e., $\alpha^\dagger$ is the exact coefficient producing the measured Neumann data $g$ on $\varGamma$.
\end{assumption}

\subsection{Convergence results}\label{subsec:convergence_results}
In practice, boundary measurements are affected by noise. We model this by  
\begin{equation}\label{eq:noisy_boundary}
g^\delta = g + \eta, \qquad \norm{\eta}_{0,\varGamma} \le \delta,
\end{equation}
where $\delta>0$ denotes the noise level. Substituting $g^\delta$ for $g$ in Problem~\ref{prob:weak_form} yields the perturbed forward problem: find $u^\delta \in \VV$ such that
\begin{equation}\label{eq:noisy_forward}
B(u^\delta,v;\alpha) = l^\delta(v), \qquad \forall v \in \VV,
\end{equation}
where $l^\delta(v) := (Q,v)_{0,\varOmega} + \langle g^\delta, v \rangle_\varGamma + i \langle f, v \rangle_\varGamma$.

By standard stability arguments and Lemma~\ref{lem:coercivity}, there exists a constant $C>0$, independent of $\alpha$ and $\delta$, such that
\begin{equation}\label{eq:stability_noisy}
\norm{u^\delta - u}_{1,\varOmega} \le C\,\delta.
\end{equation}

The corresponding Tikhonov functional under noisy data is defined as
\begin{equation}\label{eq:Tikhonov_CCBM}
\Je^\delta(\alpha)
:= \frac{1}{2}
\Big(
w_0 \norm{\iu^\delta(\alpha)}_{0,\varOmega}^{2}
+
w_1 \norm{\nabla \iu^\delta(\alpha)}_{0,\varOmega}^{2}
\Big)
+ \frac{\rho}{2} \norm{\alpha}_{0,\varOmega}^2, 
\qquad \alpha \in \SubsetK,
\end{equation}
where $u^\delta(\alpha) = \ru^\delta(\alpha) + i \iu^\delta(\alpha)$ solves \eqref{eq:noisy_forward}.

We now consider the finite-dimensional regularized inverse problem with noisy data, aiming to recover a stable approximation of the diffusion coefficient in $\SubsetK$.

\begin{problem}\label{prob:regularized_inverse}
Find $\alpharho^\delta \in \SubsetK$ such that
\begin{equation}\label{eq:minimization_problem}
\Je^\delta(\alpharho^\delta) = \min_{\alpha \in \SubsetK} \Je^\delta(\alpha).
\end{equation}
\end{problem}

The first question we address is whether a minimizer actually exists.

\begin{proposition}\label{prop:existence_minimizer}
Problem~\ref{prob:regularized_inverse} admits at least one solution $\alpharho^\delta \in \SubsetK$.
\end{proposition}

Once existence is established, it is natural to ask how sensitive the minimizer is to variations in the data.  

\begin{proposition}\label{prop:continuous_dependence}
Let Assumption~\ref{ass:key_assumption} hold, and let $\rho \ge \rho_0$ as in Proposition~\ref{prop:strict_convexity}.  
Let $\{g^n\} \subset H^{-1/2}(\varGamma)$ satisfy
\[
\|g^n - g^\infty\|_{H^{-1/2}(\varGamma)} \to 0 \quad \text{as } n \to \infty.
\]
For each $n$, let $\alpharho^n \in \SubsetK$ denote the unique minimizer of $\Je^n$ corresponding to $g^n$, and let $u^n := F(\alpharho^n; g^n) \in \VV$ be the corresponding state.

Then, as $n\to\infty$,
\[
\alpharho^n \to \alpharho^\infty \quad \text{strongly in } L^2(\varOmega),
\qquad
u^n \to u^\infty \quad \text{strongly in } \VV,
\]
where $\alpharho^\infty$ and $u^\infty$ correspond to $g^\infty$.
\end{proposition}
\begin{remark}
In Proposition~\ref{prop:continuous_dependence}, the Neumann data are assumed to lie in the natural trace space $H^{-1/2}(\varGamma)$.  
If instead $\{g^n\} \subset L^2(\varGamma)$ with $g^n \to g^\infty$ in $L^2(\varGamma)$, the same convergence result holds.  
Indeed, the trace embedding $H^1(\varOmega) \hookrightarrow L^2(\varGamma)$ ensures that
\[
	\abs{\langle g^n - g^\infty, v \rangle_\varGamma} \le \norm{g^n - g^\infty}_{0,\varGamma} \, \norm{v}_{1,\varOmega}, \quad \forall v \in \VV,
\]
so that the forward map $F$ remains Lipschitz continuous with respect to the $L^2(\varGamma)$-norm of the boundary data.  
Therefore, Proposition~\ref{prop:continuous_dependence} holds with $H^{-1/2}(\varGamma)$ replaced by $L^2(\varGamma)$.
\end{remark}

Finally, we examine the behavior of the regularized solutions as both noise and regularization vanish. This establishes the convergence of the numerical scheme towards the exact solution.

\begin{proposition}\label{prop:convergence}
Let $\{\delta_n\} \to 0$ and let $\rho_n = \rho(\delta_n) \to 0$.  
For each $n$, let $\alpha_{\rho_n}^{\delta_n} \in \SubsetK$ be a minimizer of $J^{\delta_n}_{\rho_n}$.  
Then, under Assumption~\ref{ass:exact_solution_in_K}, there exists a subsequence such that
\[
\alpha_{\rho_n}^{\delta_n} \to \alpha^\dagger \quad \text{in } L^2(\varOmega), \qquad
F^{\delta_n}(\alpha_{\rho_n}^{\delta_n}) \rightharpoonup u^\dagger \text{ in } \VV.
\]
\end{proposition}

The convergence above guarantees that, in the vanishing-noise and vanishing-regularization limit, the finite-dimensional Tikhonov minimizers recover the exact diffusion coefficient $\alpha^\dagger$.  
This establishes the well-posedness of the regularized inverse problem in $\SubsetK$.

\subsection{Convergence proofs}\label{subsec:convergence_proofs}

\subsubsection{Proof of existence of a minimizer}
\begin{proof}[Proof of Proposition~\ref{prop:existence_minimizer}]
Let $f$ be the imposed Dirichlet input and $g$ the associated Neumann measurement. These data are assumed fixed throughout the following arguments.

Since $\Je(\alpha) \ge 0$ for all $\alpha \in \SubsetK$, we can define $j_{\rho} := \inf_{\alpha \in \SubsetK} \Je(\alpha) \ge 0$.
Let $\{\alpha_n\} := \{\alpha_n\}_{n\in\mathbb{N}}  \subset \SubsetK$ be a minimizing sequence such that $\Je(\alpha_n) \to j_{\rho}$ as $n \to \infty$.
By definition of $J_\rho$ and the coercivity of the regularization term, we have
\[
J_\rho(\alpha) = J(\alpha) + \regu_\rho(\alpha)
= J(\alpha) + \frac{\rho}{2}\norm{\alpha}_{0,\varOmega}^2 \ge \frac{\rho}{2}\norm{\alpha}_{0,\varOmega}^2,
\]
for any $\alpha \in \SubsetK$, since $J(\alpha) \ge 0$.  
Furthermore, there exists a constant $C>0$ such that $J_\rho(\alpha_n) \le C$ for all $n$.  
Combining the two inequalities gives 
$\norm{\alpha_n}_{0,\varOmega} \le \sqrt{2C/\rho}$, for all $n$.
This shows that $\{\alpha_n\}$ is bounded in $L^2(\varOmega)$.

Since $\SubsetK \subset \Aad \subset L^\infty(\varOmega)$ is finite-dimensional (Assumption~\ref{ass:key_assumption}), all norms on $\SubsetK$ are equivalent and every bounded sequence in $\SubsetK$ is relatively compact.  
Hence, there exists a subsequence (still denoted $\{\alpha_n\}$) and some $\alpharho \in \SubsetK$ such that $\alpha_n \to \alpharho$ strongly in $L^2(\varOmega)$ and in $L^\infty(\varOmega)$.

Let $u_n = F(\alpha_n) \in \VV$ and $u_\rho = F(\alpharho) \in \VV$ denote the corresponding state solutions. By the Lipschitz continuity of $F$ (Proposition~\ref{prop:continuity_operator_F}), we have
\[
\norm{u_n - u_\rho}_{1,\varOmega} \le \frac{C}{c_B^2} \norm{\alpha_n - \alpharho}_{L^\infty(\varOmega)} \to 0.
\]

The functional $\Je$ is continuous with respect to $\alpha \in \SubsetK$ because $J(\alpha)$ depends continuously on $\iu(\alpha)$ and $\nabla \iu(\alpha)$ (via \eqref{eq:modified_CCBM_cost_functional}), and $\regu_\rho(\alpha)$ is continuous. 
Hence, $\Je(\alpharho) = \lim_{n\to\infty} \Je(\alpha_n) = j_{\rho}$.

Thus, $\alpharho \in \SubsetK$ attains the minimum of $\Je$, proving existence.
\end{proof}

\subsubsection{Proof of continuous dependence to the measured data}
\begin{proof}[Proof of Proposition~\ref{prop:continuous_dependence}]
Let $\{g^n\} \subset H^{-1/2}(\varGamma)$ satisfy $\norm{g^n - g^\infty}_{H^{-1/2}(\varGamma)} \to 0$ as $n \to \infty$, and let $u^n(\alpha) := F(\alpha; g^n)$ denote the unique solution of the weak problem
\[
B(u^n(\alpha), v; \alpha) = (Q,v)_{0,\varOmega} + \langle g^n, v \rangle_\varGamma + i \langle f, v \rangle_\varGamma, \quad \forall v \in \VV,
\]
for each $\alpha \in \SubsetK$.  
The corresponding Tikhonov-regularized cost functional is
\[
\Je^n(\alpha) := \frac{1}{2} \Big( w_0 \norm{\iu^n(\alpha)}_{0,\varOmega}^2 + w_1 \norm{\nabla \iu^n(\alpha)}_{0,\varOmega}^2 \Big) + \frac{\rho}{2} \norm{\alpha}_{0,\varOmega}^2.
\]

By Lemma~\ref{lem:coercivity}, the forward map depends continuously on the Neumann data, so for each fixed $\alpha \in \SubsetK$,
\[
\norm{u^n(\alpha) - u^\infty(\alpha)}_{1,\varOmega} \le \frac{C}{c_B} \, \norm{g^n - g^\infty}_{H^{-1/2}(\varGamma)} \longrightarrow 0,
\]
which implies pointwise convergence of the functionals: $\Je^n(\alpha) \to \Je^\infty(\alpha)$.  
Since $\SubsetK \subset L^\infty(\varOmega)$ is finite-dimensional, the sequence $\{\alpharho^n\}$ is uniformly bounded and admits a strongly convergent subsequence in $L^2(\varOmega)$, say $\alpharho^n \to \bar{\alpha} \in \SubsetK$.  
By weak lower semicontinuity and since $\alpharho^n$ minimizes $\Je^n$, i.e., $\Je^n(\alpharho^n) \le \Je^n(\alpharho^\infty)$, we obtain
\[
\Je^\infty(\bar{\alpha}) \le \Je^\infty(\alpharho^\infty).
\]
Strict convexity of $\Je^\infty$ then yields $\bar{\alpha} = \alpharho^\infty$. Hence, the full sequence converges strongly: $\alpharho^n \to \alpharho^\infty$ in $L^2(\varOmega)$.

Finally, since $\{u^n(\alpharho^n)\}$ is bounded in $\VV$, up to a subsequence $u^n(\alpharho^n) \rightharpoonup \bar{u}$ weakly in $\VV$. Passing to the limit in the weak formulation and using $\alpharho^n \to \alpharho^\infty$ yields
\[
	B(\bar{u}, v; \alpharho^\infty) = (Q,v)_{0,\varOmega} + \langle g^\infty, v \rangle_\varGamma + i \langle f, v \rangle_\varGamma,
\]
for all $v \in \VV$, so $\bar{u} = F(\alpharho^\infty; g^\infty)$. Uniqueness of the weak limit then gives
\[
u^n(\alpharho^n) \rightharpoonup u^\infty := F(\alpharho^\infty; g^\infty) \quad \text{in } \VV.
\]
Moreover, by the Lipschitz continuity of $F$ in both $\alpha$ (Proposition~\ref{prop:continuity_operator_F}) and $g$ (Lemma~\ref{lem:coercivity}), one can strengthen the convergence to strong convergence:
\[
\norm{u^n(\alpharho^n) - u^\infty}_{1,\varOmega} \longrightarrow 0 \quad \text{as } n \to \infty.
\]
This completes the proof.
\end{proof}

\subsubsection{Proof of convergence under vanishing noise}
\begin{proof}[Proof of Proposition~\ref{prop:convergence}]
In contrast to Proposition~\ref{prop:continuous_dependence}, here both the measurement data $g^{\delta_n}$ and the regularization parameter $\rho_n$ vary with $n$.  

Let $\{\delta_n\}_{n \in \mathbb{N}}$ be a sequence with $\delta_n \to 0$ and let
$\{\rho_n\}_{n \in \mathbb{N}}$ satisfy $\rho_n = \rho(\delta_n) \to 0$.
For each $n$, denote by 
\[
\alpha^n := \alpha_{\rho_n}^{\delta_n} \in \SubsetK
\]
the minimizer of $\Je^n$ corresponding to the noisy data $g^{\delta_n}$.  
Set 
\[
u^n := F^{\delta_n}(\alpha^n) \in \VV, \qquad \iu^n := \Im\{u^n\}.
\]

Since $\alpha^n \in \SubsetK$, which is finite-dimensional and compact in $L^2(\varOmega)$, there exists a subsequence (not relabeled) and $\bar{\alpha} \in \SubsetK$ such that
\begin{equation}\label{eq:conv_alpha_n}
\alpha^n \to \bar{\alpha} \quad \text{strongly in } L^2(\varOmega).
\end{equation}

By Lemma~\ref{lem:coercivity}, the sequence $\{u^n\}$ is uniformly bounded in $\VV$.  
Hence, there exists $u^\infty \in \VV$ and a subsequence such that
\begin{equation}\label{eq:conv_un}
u^n \rightharpoonup u^\infty \quad \text{weakly in } \VV, 
\qquad 
u^n \to u^\infty \quad \text{strongly in } L^2(\varOmega).
\end{equation}

For each fixed $\alpha \in \SubsetK$, the continuity of the forward map with respect to the data (Lemma~\ref{lem:coercivity} and Proposition~\ref{prop:continuity_operator_F}) ensures that
\[
F^{\delta_n}(\alpha) \to F(\alpha) \quad \text{in } \VV \quad \text{as } \delta_n \to 0.
\]
Combining this with the strong convergence \eqref{eq:conv_alpha_n} and the weak convergence \eqref{eq:conv_un}, we can pass to the limit in the variational formulation of the state equation to conclude
\[
u^\infty = F(\bar{\alpha}).
\]

Next, since $\alpha^n$ minimizes $\Je^n$ over $\SubsetK$, we have
\begin{equation}\label{eq:minimizers_for_each_n}
\Je^n(\alpha^n) \le \Je^n(\alpha^\dagger), \quad \forall n \in \mathbb{N},
\end{equation}
where $\alpha^\dagger \in \SubsetK$ is the exact solution (Assumption~\ref{ass:exact_solution_in_K}).  

Now, note that $\alpha^\dagger$ is exact, so the exact state $F(\alpha^\dagger)$ is real-valued: $\Im\{ F(\alpha^\dagger)\} = 0$.  
By continuity of the forward map with respect to the data and since $\rho_n \to 0$, we have
\begin{equation}\label{eq:limit_of_minimizing_cost_at_exact_alpha}
\Je^n(\alpha^\dagger) 
= \frac12 \Big( w_0 \norm{\Im\{ F^{\delta_n}(\alpha^\dagger) \}}_{0,\varOmega}^2 + w_1 \norm{ \nabla \Im\{ F^{\delta_n}(\alpha^\dagger) \} }_{0,\varOmega}^2 \Big)
+ \frac{\rho_n}{2} \norm{\alpha^\dagger}_{0,\varOmega}^2 \longrightarrow 0.
\end{equation}
By weak lower semicontinuity of the $L^2$- and $H^1$-norms, together with \eqref{eq:minimizers_for_each_n} and \eqref{eq:limit_of_minimizing_cost_at_exact_alpha}, it follows that
\[
\frac12 \Big( w_0 \norm{\Im\{F(\bar{\alpha})\}}_{0,\varOmega}^2 + w_1 \norm{\nabla \Im\{F(\bar{\alpha})\}}_{0,\varOmega}^2 \Big) 
\le \liminf_{n\to\infty} \Je^n(\alpha^n) \le \lim_{n\to\infty} \Je^n(\alpha^\dagger) = 0.
\]
Hence,  $\Im\{F(\bar{\alpha})\} = 0$ in $\varOmega$, so that the pair $(\bar{\alpha}, F(\bar{\alpha}))$ solves the exact inverse problem.  

By uniqueness (Assumption~\ref{ass:exact_solution_in_K}), we conclude $\bar{\alpha} = \alpha^\dagger$, and therefore the full sequence converges strongly in $L^2(\varOmega)$: $\alpha^n = \alpha_{\rho_n}^{\delta_n} \to \alpha^\dagger \quad \text{in } L^2(\varOmega)$.

Finally, since the sequence $\{u^n\}$ is bounded in $\VV$, passing to the limit in the variational formulation of the state equation as above shows
\[
F^{\delta_n}(\alpha_{\rho_n}^{\delta_n}) \rightharpoonup u^\dagger \quad \text{in } \VV.
\]
This completes the proof.
\end{proof}

For regularization parameters $\rho_n \ge \rho_0$, Proposition~\ref{prop:strict_convexity} ensures that the functional $J_{\rho_n}^{\delta_n}$ is strictly convex, so the minimizer $\alpha_{\rho_n}^{\delta_n}$ is unique.  
When $\rho_n$ becomes smaller than the threshold $\rho_0$ from Proposition~\ref{prop:strict_convexity}, strict convexity may no longer hold, and multiple minimizers may exist for a given $n$.
Nonetheless, since $\alpha_{\rho_n}^{\delta_n} \in \SubsetK$, which is finite-dimensional and compact, the sequence is bounded and thus admits a convergent subsequence.  
Crucially, Assumption~\ref{ass:exact_solution_in_K} guarantees that any such limit of convergent subsequences is uniquely determined as the exact solution $\alpha^\dagger$, despite the potential non-uniqueness of finite-$n$ minimizers in \eqref{eq:minimizers_for_each_n}.



\section{Numerical Approximation via Sobolev Gradient Method}\label{sec:grad_descent_algorithm}

We consider the minimization of a general cost functional 
\[
J : \barAad \to \mathbb{R}_{+},
\] 
and solve Problem~\ref{prob:inverse_problem} via a gradient descent method in the Sobolev space $H^1(\varOmega)$.  

To improve stability and control the smoothness of the descent direction, we employ a Sobolev-gradient formulation.  
Let $\mu > 0$ and let $J^{\prime}(\alpha)$ denote the Fréchet derivative of $J$ with respect to $\alpha$.  
Instead of using the standard $L^2$ gradient (i.e., the Riesz representative of $J^{\prime}(\alpha)$ in $L^2(\varOmega)$), we compute the Sobolev gradient as the Riesz representative in the weighted $H^1$ inner product \cite{Neuberger1997}:
\begin{equation}\label{eq:Sobolev_gradient_unreg}
(\nabla_{H^1} J(\alpha), \varphi)_{\mu, \varOmega} = J^{\prime}(\alpha)[\varphi], 
\quad \forall \varphi \in H^1(\varOmega),
\end{equation}
where the inner product is defined by
\begin{equation}\label{eq:mu_weighted_inner_product}
(\varphi, \psi)_{\mu, \varOmega} := \mu \, (\nabla \varphi, \nabla \psi)_{0, \varOmega} + (\varphi, \psi)_{0, \varOmega}.
\end{equation}
In particular, for $\mu = 1$, $(\cdot, \cdot)_{\mu,\varOmega}$ coincides with the standard $H^1$ inner product.  
Choosing a sufficiently large $\mu$ enhances the smoothing effect, which improves stability and mitigates noise amplification in the reconstruction, whereas for $\mu \ll 1$, the influence of the $H^1$-smoothing term is negligible (see Figure~\ref{fig:effect_of_mu} for numerical illustration).

Note that the solution of \eqref{eq:Sobolev_gradient_unreg} depends continuously on $J^{\prime}(\alpha)$ and satisfies the stability estimate
\[
\norm{\nabla_{H^1} J(\alpha)}_{1, \varOmega} 
\le \frac{1}{\min(1,\mu)} \, \norm{J^{\prime}(\alpha)}_{(H^1(\varOmega))^{\ast}},
\]
where $\norm{\cdot}_{(H^1(\varOmega))^{\ast}}$ denotes the dual norm on $H^1(\varOmega)^{\ast}$.  

For $k \in \mathbb{N}$, the Sobolev-gradient descent update reads
\begin{equation}\label{eq:H1_gradient_update_unreg}
\alpha^{[k+1]} = \alpha^{[k]} - t^{[k]} \, \nabla_{H^1} J(\alpha^{[k]}), 
\end{equation}
where $t^{[k]} > 0$ is a step size, either fixed and sufficiently small, or determined via a line search satisfying, for example, the Armijo or Wolfe conditions \cite{NocedalWright2006}.
The iteration may be terminated when a prescribed maximum number of iterations is reached.

\begin{algorithm}[H]
\caption{Gradient Descent in $H^1$ (Unregularized)}
\label{alg:H1_grad_descent_unreg}
\begin{algorithmic}[1]
\Require objective function $J$, initial guess $\alpha^{[0]} \in \barAad$, step size rule, maximum iterations $k_{\max}$
\Ensure approximate stationary point $\alpha^\star$
\State Set $k = 0$
\While{$k \le k_{\max}$}
    \State Solve the state equation for $\alpha^{[k]}$ to obtain $u(\alpha^{[k]})$
    \State Solve the adjoint equation (if applicable) to obtain $p(\alpha^{[k]})$
    \State Compute the $H^1$-Sobolev gradient $\nabla_{H^1} J(\alpha^{[k]})$ using $u(\alpha^{[k]})$ and $p(\alpha^{[k]})$
    \State Determine step size $t^{[k]}$ (fixed or via line search \cite{NocedalWright2006})
    \State Do the update \eqref{eq:H1_gradient_update_unreg}
    \State $k \gets k+1$
\EndWhile
\State \Return $\alpha^{[k]}$ as $\alpha^\star$
\end{algorithmic}
\end{algorithm}

Before proceeding to numerical experiments, we note that the Sobolev-gradient iteration (Algorithm~\ref{alg:H1_grad_descent_unreg}) generates iterates $\{\alpha^{[k]}\}$ such that there exists $k^\star \le k_{\max}$ with
\[
\norm{\nabla_{H^1} J(\alpha^{[k^\star]})}_{1, \varOmega} \le \varepsilon^{\star}
\]
for some $\varepsilon^{\star} > 0$, i.e., the iteration produces an \emph{approximate stationary point} of $J$.

\begin{proposition}\label{prop:grad_descent_convergence_unreg}
Let Assumption~\ref{ass:key_assumption} hold, and let $\alpha^{[0]} \in \SubsetK$.  
Assume the step sizes $t^{[k]} > 0$ satisfy standard descent conditions (e.g., Armijo or Wolfe).  
Then, the iterates $\{\alpha^{[k]}\}$ generated by Algorithm~\ref{alg:H1_grad_descent_unreg} satisfy
\[
J(\alpha^{[k+1]}) \le J(\alpha^{[k]}), \quad \forall k \le k_{\max}.
\]
Moreover, there exists an index $k^\star \le k_{\max}$ such that
\[
\norm{\nabla_{H^1} J(\alpha^{[k^\star]})}_{1, \varOmega} \le \frac{J(\alpha^{[0]}) - \Jmin}{t_{\min} k_{\max}},
\]
where $\Jmin := \min_{\alpha \in \SubsetK} J(\alpha)$ and $t_{\min} := \min_{0 \le k \le k_{\max}} t^{[k]}$.  
\end{proposition}

\begin{proof}
Since $\mathcal{K}$ is finite-dimensional and convex (Assumption~\ref{ass:key_assumption}), the Sobolev gradient $\nabla_{H^1} J(\alpha)$ exists and is uniquely defined for all $\alpha \in \mathcal{K}$ (definition given in equation~\eqref{eq:Sobolev_gradient_unreg}).  

Let $\alpha_1, \alpha_2 \in \mathcal{K}$ and $v \in H^1(\varOmega)$. By \eqref{eq:Sobolev_gradient_unreg}, we have
\[
(\nabla_{H^1} J(\alpha_1) - \nabla_{H^1} J(\alpha_2), v)_{\mu,\varOmega} = J^{\prime}(\alpha_1)[v] - J^{\prime}(\alpha_2)[v].
\]
Using the Lipschitz continuity of the forward map and its derivative (Proposition~\ref{prop:continuity_operator_F} and Lemma~\ref{lem:coercivity}) and uniform boundedness of $u(\alpha)$ in $H^1(\varOmega)$, we obtain 
\[
\abs{J^{\prime}(\alpha_1)[v] - J^{\prime}(\alpha_2)[v]} \le C^{\prime}\, \norm{\alpha_1 - \alpha_2}_{L^\infty(\varOmega)} \, \norm{v}_{1, \varOmega},
\]
for some constant $C^{\prime}>0$. Taking the supremum over $\norm{v}_{1, \varOmega} = 1$ and using norm equivalence on $\mathcal{K}$ yields the Lipschitz continuity of $\nabla_{H^1} J$.  

For the iterates $\{\alpha^{[k]}\}$ of Algorithm~\ref{alg:H1_grad_descent_unreg} with a standard line search (e.g., Armijo), we have the descent inequality
\[
J(\alpha^{[k+1]}) \le J(\alpha^{[k]}) - t^{[k]} \, \norm{\nabla_{H^1} J(\alpha^{[k]})}_{1, \varOmega}^2.
\]
Summation over $k = 0,\dots,k_{\max}-1$ then guarantees the existence of $k^\star$ such that
\[
\norm{\nabla_{H^1} J(\alpha^{[k^\star]})}_{1, \varOmega}^2 \le \frac{J(\alpha^{[0]}) - \Jmin}{t_{\min} k_{\max}},
\]
i.e., $\alpha^{[k^\star]}$ is an approximate stationary point of $J$.
\end{proof}

Incorporating a Tikhonov regularization term \eqref{eq:Tikhonov_regularization_term} further enhances the stability of the reconstruction, particularly in the presence of noisy measurement data. 

The Sobolev-gradient formulation can naturally include this term by replacing $J^{\prime}(\alpha)$ with $\Jeprime(\alpha)$ in equation~\eqref{eq:Sobolev_gradient_unreg}. Alternatively, the Tikhonov term can be added after computing the $H^1$-Riesz representative of $J^{\prime}(\alpha)$, effectively smoothing the descent direction. As a result, the iterates tend to avoid spurious oscillations and converge more robustly toward a stable approximate minimizer. 

By Proposition~\ref{prop:strict_convexity}, for sufficiently large $\rho$, $J_\rho$ is strictly convex, ensuring a unique minimizer and the error estimate
\[
\norm{\alpha^{[k^\star]} - \alpharho}_{1, \varOmega} \le \frac{1}{C} \, \norm{\nabla_{H^1} J_\rho(\alpha^{[k^\star]})}_{1, \varOmega}.
\]

In practice, $\rho$ can be chosen as a fixed parameter (as we do here) or updated heuristically during the iterations to balance data fidelity and regularization. For example, using the \emph{balancing principle} \cite{ClasonJinKunisch2010b}:
\begin{equation}\label{eq:balancing_principle_H1}
({\gamma}-1) J(\alpha) - \rho \regu(\alpha) = 0, \quad {\gamma}>1,
\end{equation}
which balances the data-fidelity term against the regularization term without requiring explicit knowledge of the noise level \cite{ClasonJinKunisch2010,Clason2012}. 
This approach provides robust recovery of the dominant features of $\alpha$ while suppressing noise amplification in the gradient iterations \cite{ItoJinTakeuchi2011}.

Using the unique solution of \eqref{eq:Sobolev_gradient_unreg}, the $H^1$-gradient update in equation~\eqref{eq:H1_gradient_update_unreg} can be modified in two ways.  
The first option incorporates the Tikhonov term directly into the Sobolev gradient:  
\begin{equation}\label{eq:H1_gradient_update_with_Tikhonov}
    \alpha^{[k+1]} = \alpha^{[k]} - t^{[k]} \nabla_{H^1} J_{\rho^{[k]}}(\alpha^{[k]}),
\end{equation}
where the gradient is computed from the fully regularized functional $J_\rho$, so the smoothing effect of the Sobolev inner product applies to both the data-fidelity and regularization terms.  

Alternatively, the Tikhonov contribution can be added separately after computing the Sobolev gradient of the unregularized functional:
\begin{equation}\label{eq:H1_gradient_update_plus_Tikhonov}
    \alpha^{[k+1]} = \alpha^{[k]} - t^{[k]} \nabla_{H^1} J(\alpha^{[k]}) + t^{[k]} \rho^{[k]} \regu'(\alpha^{[k]}),
\end{equation}
where only the data term is smoothed in $H^1$, and the regularization term is added in its “raw” form.  

Note that when $\mu$ is set to zero in \eqref{eq:Sobolev_gradient_unreg}, \eqref{eq:H1_gradient_update_plus_Tikhonov} reduces to the conventional $L^2$-gradient update with Tikhonov regularization:
\begin{equation}\label{eq:conventional_update_plus_Tikhonov}
    \alpha^{[k+1]} = \alpha^{[k]} - t^{[k]} J^{\prime}(\alpha^{[k]}) + t^{[k]} \rho^{[k]} \regu'(\alpha^{[k]}).
\end{equation}

In our numerical experiments, we focus on the application of \eqref{eq:H1_gradient_update_plus_Tikhonov} and \eqref{eq:conventional_update_plus_Tikhonov}.

%
\begin{remark}
Even if the regularization parameter $\rho$ is small or vanishes, so that $\Je$ may lose strict convexity (Proposition~\ref{prop:strict_convexity}), the Sobolev-gradient iterates $\alpha^{[k]} \in \SubsetK$ remain bounded.  
By compactness of $\SubsetK$ (Assumption~\ref{ass:key_assumption}), convergent subsequences exist, and any limit of such subsequences can be uniquely identified as the exact solution $\alpha^\dagger$ under Assumption~\ref{ass:exact_solution_in_K}, in analogy with Proposition~\ref{prop:convergence}.
\end{remark}

\begin{remark}\label{rem:coercivity_condition_for_the_Sobolev_gradient}
Assume that the Sobolev gradient $\nabla_{H^1} J : \mathcal{K} \to H^1(\varOmega)$ is 
\emph{strongly monotone}, i.e.,
\[
(\nabla_{H^1} J(\alpha) - \nabla_{H^1} J(\beta), \alpha - \beta)_{1, \varOmega}
\ge \kappa \, \norm{\alpha - \beta}_{1, \varOmega}^2,
\quad \forall \alpha,\beta \in \mathcal{K},
\]
for some $\kappa > 0$, and Lipschitz continuous with constant $L_S > 0$, i.e.,
\[
\norm{\nabla_{H^1} J(\alpha) - \nabla_{H^1} J(\beta)}_{1, \varOmega}
\le L_S \, \norm{\alpha - \beta}_{1, \varOmega}.
\]

Under these conditions, $J$ admits a unique minimizer $\alpha^\star \in \mathcal{K}$ (cf. Proposition~\ref{prop:existence_minimizer}; see also, e.g., \cite[Ch.~V, Sec.~2]{Glowinski1984}).  
Moreover, Algorithm~\ref{alg:H1_grad_descent_unreg} with a sufficiently small constant step size $t>0$ (e.g., $0 < t \le 1/L_S$) generates a sequence $\{\alpha^{[k]}\}$ satisfying
\[
\norm{\alpha^{[k]} - \alpha^\star}_{1, \varOmega}
\le (1 - t \, \kappa)^k \, \norm{\alpha^{[0]} - \alpha^\star}_{1, \varOmega},
\quad \forall k \in \mathbb{N}.
\]
In particular, the Sobolev-gradient iteration converges \emph{linearly} in $H^1(\varOmega)$ to the unique minimizer $\alpha^\star$.
\end{remark}

The following result summarizes the effect of the Sobolev gradient and the weighting parameter $\mu$ on the stability and smoothness of the descent direction, providing a theoretical justification for its use in the numerical reconstruction.

\begin{proposition}\label{prop:strong_monotonicity_of_the_Sobolev_gradient}
Let Assumption~\ref{ass:key_assumption} hold, so that $\SubsetK \subset \Aad \subset L^\infty(\varOmega)$ is finite-dimensional, closed, and convex.  
Let $J: H^1(\varOmega) \to \mathbb{R}$ be Fr\'echet differentiable, and assume $J^{\prime}(\alpha)$ is \emph{strongly monotone in $L^2(\varOmega)$} over $\SubsetK$, i.e., there exists $C > 0$ such that
\[
(J^{\prime}(\alpha_1) - J^{\prime}(\alpha_2), \alpha_1 - \alpha_2)_{0, \varOmega} 
\ge C \, \norm{\alpha_1 - \alpha_2}_{0, \varOmega}^2, 
\quad \forall \alpha_1, \alpha_2 \in \SubsetK.
\]
Let $\nabla_{H^1} J(\alpha)$ denote the Sobolev gradient defined via the $\mu$-weighted inner product \eqref{eq:mu_weighted_inner_product}.
Then, there exists a constant $C_{\SubsetK} > 0$ (depending only on the finite-dimensional set $\SubsetK$) such that
\begin{equation}\label{eq:strong_monotonicity}
(\nabla_{H^1} J(\alpha_1) - \nabla_{H^1} J(\alpha_2), \alpha_1 - \alpha_2)_{\mu,\varOmega} 
\ge \frac{\min\{C, \mu\}}{C_{\SubsetK}} \, \norm{\alpha_1 - \alpha_2}_{1, \varOmega}^2, 
\quad \forall \alpha_1, \alpha_2 \in \SubsetK.
\end{equation}
In particular, the Sobolev gradient is strongly monotone over $\SubsetK$, and the monotonicity constant scales with $\mu$, providing enhanced smoothing and stability of the descent direction.
\end{proposition}

\begin{proof}
Let $\alpha_1, \alpha_2 \in \SubsetK$. By definition of the Sobolev gradient (equation~\eqref{eq:Sobolev_gradient_unreg}), for any $v \in H^1(\varOmega)$ we have
\[
(\nabla_{H^1} J(\alpha_1) - \nabla_{H^1} J(\alpha_2), v)_{\mu,\varOmega} 
= J^{\prime}(\alpha_1)[v] - J^{\prime}(\alpha_2)[v].
\]

Choosing $v = \alpha_1 - \alpha_2 \in \SubsetK \subset H^1(\varOmega)$ gives
\begin{align*}
(\nabla_{H^1} J(\alpha_1) - \nabla_{H^1} J(\alpha_2), \alpha_1 - \alpha_2)_{\mu,\varOmega} 
	&= J^{\prime}(\alpha_1)[\alpha_1 - \alpha_2] - J^{\prime}(\alpha_2)[\alpha_1 - \alpha_2]\\
	&= (J^{\prime}(\alpha_1) - J^{\prime}(\alpha_2), \alpha_1 - \alpha_2)_{0, \varOmega}.
\end{align*}
By the assumption of $L^2$ strong monotonicity over $\SubsetK$, we have
\[
(J^{\prime}(\alpha_1) - J^{\prime}(\alpha_2), \alpha_1 - \alpha_2)_{0, \varOmega} \ge C \, \norm{\alpha_1 - \alpha_2}_{0, \varOmega}^2.
\]

Next, since $\SubsetK$ is finite-dimensional, all norms on $\SubsetK$ are equivalent. In particular, there exists a constant $C_{\SubsetK} \ge 1$ such that
\[
\norm{\alpha}_{1, \varOmega}^2 = \norm{\alpha}_{0, \varOmega}^2 + \norm{\nabla \alpha}_{0, \varOmega}^2
\le C_{\SubsetK} \left( \norm{\alpha}_{0, \varOmega}^2 + \mu \norm{\nabla \alpha}_{0, \varOmega}^2 \right)
= C_{\SubsetK} \, \norm{\alpha}_{\mu, \varOmega}^2,
\]
where $\norm{\alpha}_{\mu, \varOmega}^2 := (\alpha, \alpha)_{\mu,\varOmega} = \norm{\alpha}_{0, \varOmega}^2 + \mu \norm{\nabla \alpha}_{0, \varOmega}^2$.
Similarly, we have
\[
\norm{\alpha}_{\mu, \varOmega}^2 \ge \min\{1, \mu\} \, \norm{\alpha}_{1, \varOmega}^2.
\]
Combining these estimates leads to our desired condition \eqref{eq:strong_monotonicity}.
%
\end{proof}

Although obvious, it must be observed from Propositions~\ref{prop:strong_monotonicity_of_the_Sobolev_gradient} that over the finite-dimensional admissible set $\SubsetK$, increasing $\mu$ strengthens the $H^1$-smoothing term, which improves stability and suppresses high-frequency oscillations in the gradient.  
Very large $\mu$ enhances smoothness but may slow convergence, while very small $\mu$ reduces smoothing and can amplify noise.  
The strong monotonicity over $\SubsetK$ guarantees that Sobolev-gradient updates converge stably toward a unique minimizer of the functional restricted to $\SubsetK$.


\section{Numerical Experiments} \label{sec:numerical_experiments}
\subsection{Reconstruction of smooth diffusion} \label{subsec:experiments_smooth_diffusion}
%
%
For the initial set of tests, we consider the domain $\varOmega = C(0,1)$, the unit circle centered at the origin, with $\alpha^{\star} = 1.0 + 0.5\,xy$, $\alpha^{[0]} = 1$, $b = 0$, $c = 1$, and $Q = 1$.
The Cauchy pair is generated synthetically. 
The forward problem is solved on a fine mesh using $P_{2}$ finite elements to generate high-resolution synthetic data with Neumann boundary condition $g = 1$. The corresponding Dirichlet data $f$ is then computed from \eqref{eq:state_un} using $\alpha = \alpha^{\star}$.
To avoid inverse crimes (see \cite[p.~154]{ColtonKress2013}), the inverse problem is solved on a coarser mesh, and all variational equations in the inversion process are discretized using $P_{1}$ finite elements. Computations are carried out in {\sc FreeFem++} \cite{Hecht2012} on a MacBook Pro equipped with an Apple M1 chip and 16\,GB RAM.

To assess robustness against noise, we define $u^\delta = (1 + \delta{}\mathrm{g.n.}){}u^{\star}$, where ``g.n.'' denotes Gaussian noise with zero mean and standard deviation $\norm{u^{\star}}_{L^{\infty}(\varOmega)}$, 
and take the noisy boundary measurement as $f|_{\varSigma} = u^\delta|_{\varSigma}$.

We introduce the abbreviations used in the numerical experiments. 
CCBM denotes the \emph{coupled boundary measurement} formulation, KV the \emph{Kohn--Vogelius} functional, TD the \emph{tracking of Dirichlet data}, and TN the \emph{tracking of Neumann data}. 
These formulations differ in the boundary data entering the objective functional and, consequently, in their sensitivity to noise and regularization.
\subsubsection{Effect of $H^{1}$ gradient smoothing on the reconstruction under exact data}
Figure~\ref{fig:effect_of_mu} compares the exact diffusion coefficient with reconstructions obtained using the CCBM, KV, TD, and TN formulations for different values of the $H^1$ smoothing parameter $\mu$. Without smoothing, all methods recover the overall shape of the coefficient, but oscillatory artifacts are evident, particularly for the TN approach. Introducing $H^1$ smoothing significantly improves stability, and for moderate values of $\mu$, the CCBM, KV, and TD methods yield accurate and smooth reconstructions, whereas the TN formulation remains more sensitive to small $\mu$ and exhibits persistent oscillations.

Figure~\ref{fig:effect_of_mu_cost_and_gradient} shows the corresponding histories of the cost functional and gradient norm. In all cases, the cost decreases monotonically, indicating stable optimization. Smaller values of $\mu$ lead to faster initial decay but slower and less regular gradient norm reduction, while larger $\mu$ produces smoother and more consistent convergence. These trends are consistent with the reconstruction quality observed in Figure~\ref{fig:effect_of_mu} and highlight the stabilizing role of $H^1$ smoothing.

\begin{figure}[htp!]
\centering
\resizebox{0.3\textwidth}{!}{\includegraphics{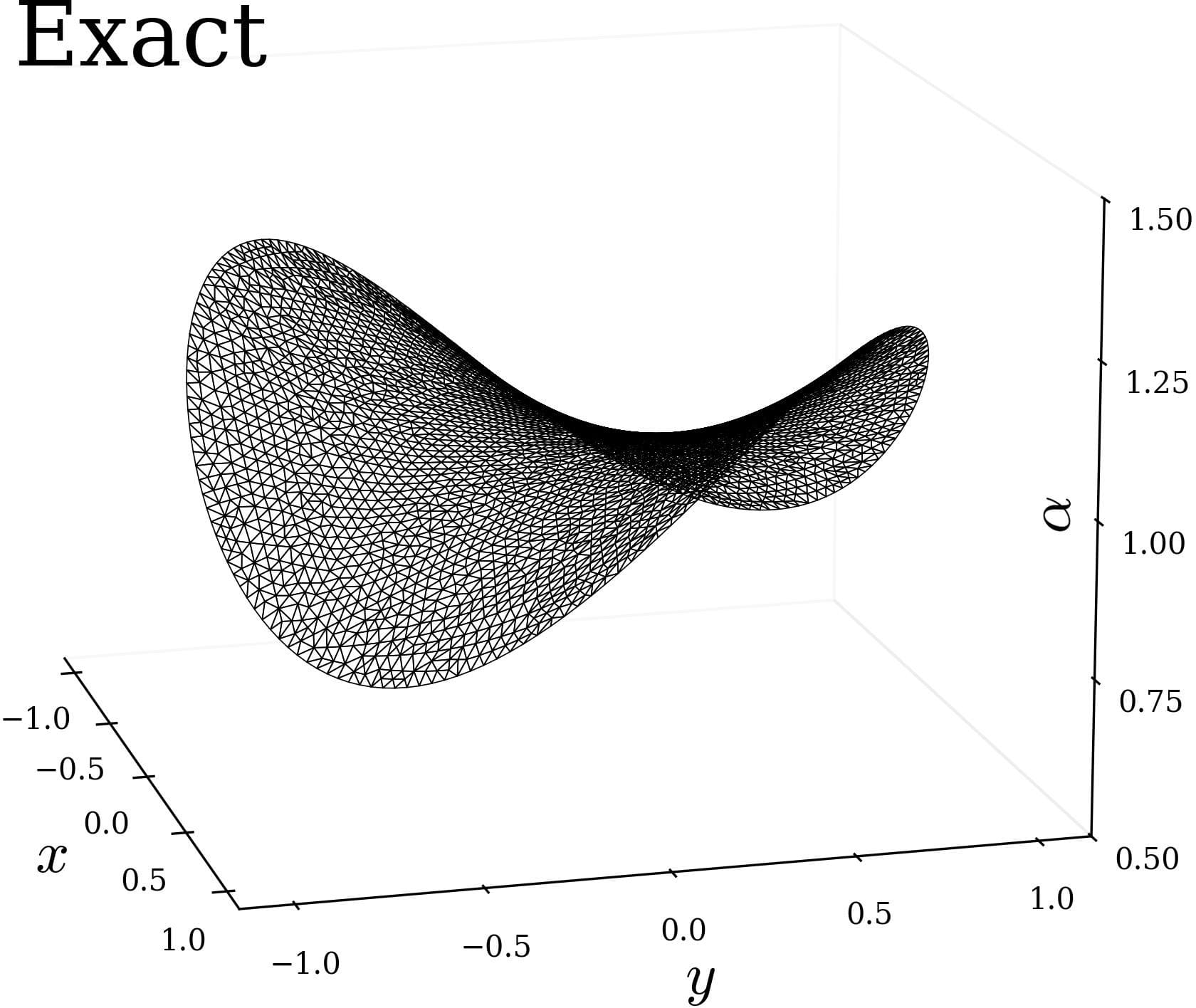}} \\[1em]
\resizebox{0.225\textwidth}{!}{\includegraphics{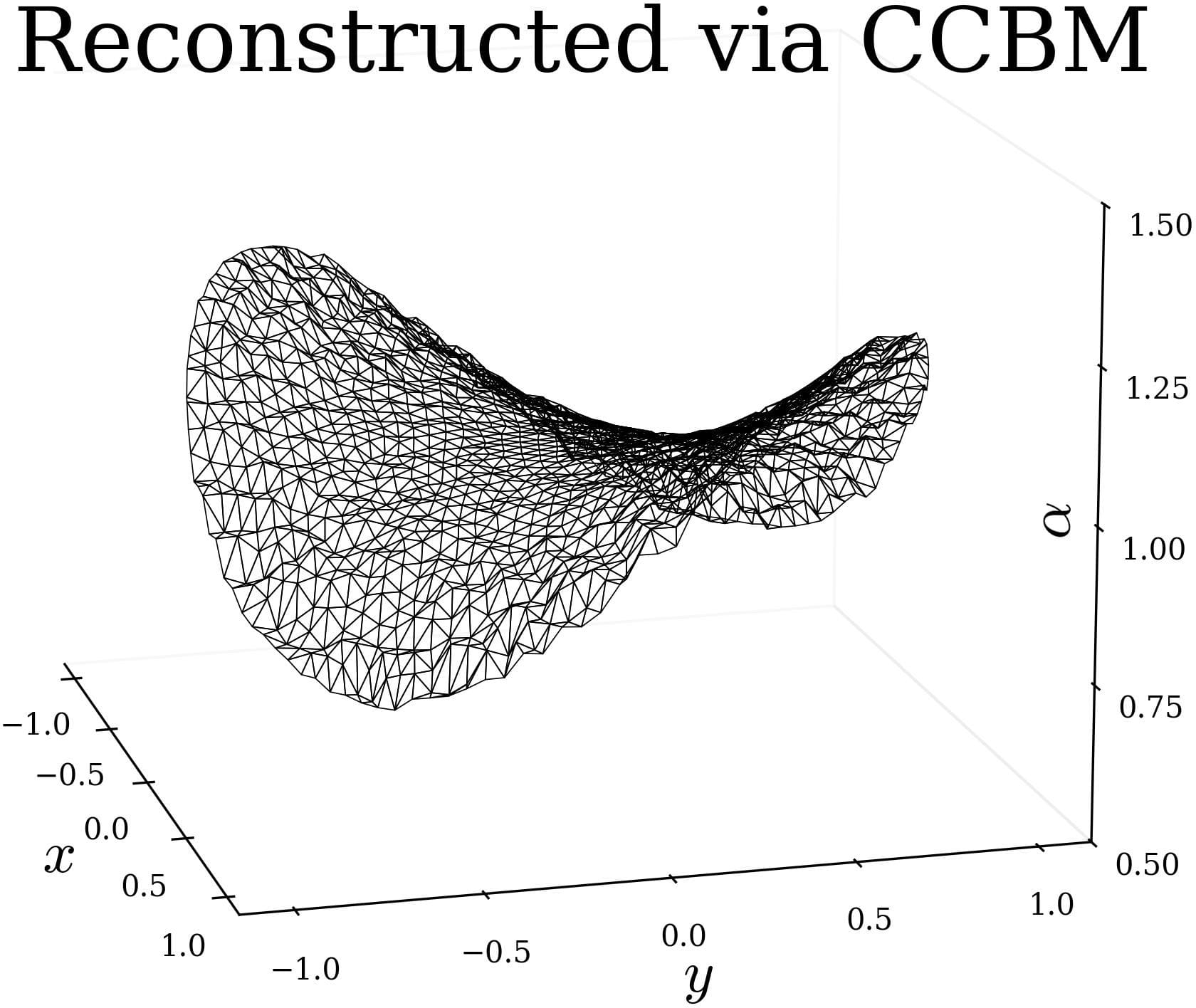}} \ 
\resizebox{0.225\textwidth}{!}{\includegraphics{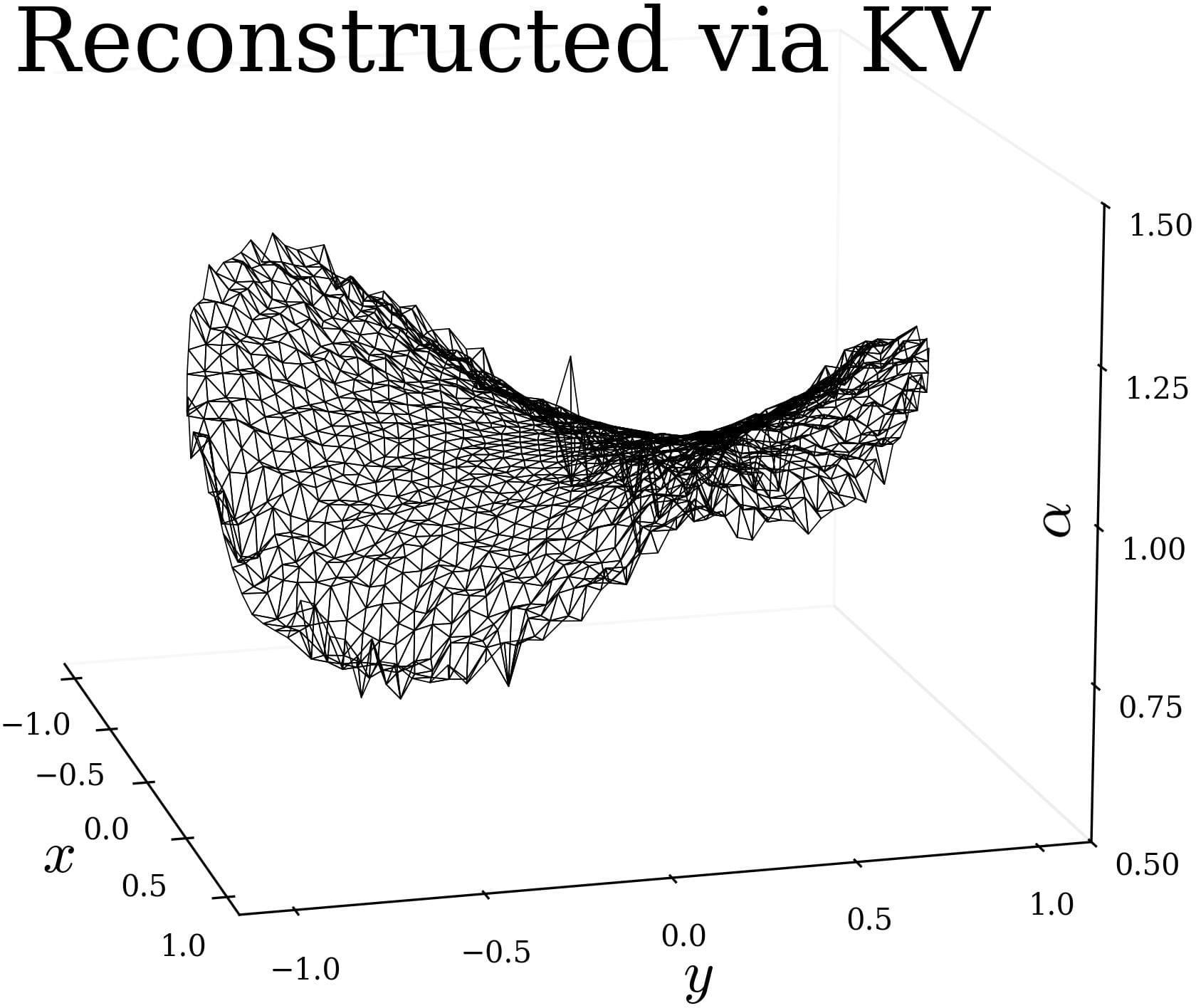}} \ 
\resizebox{0.225\textwidth}{!}{\includegraphics{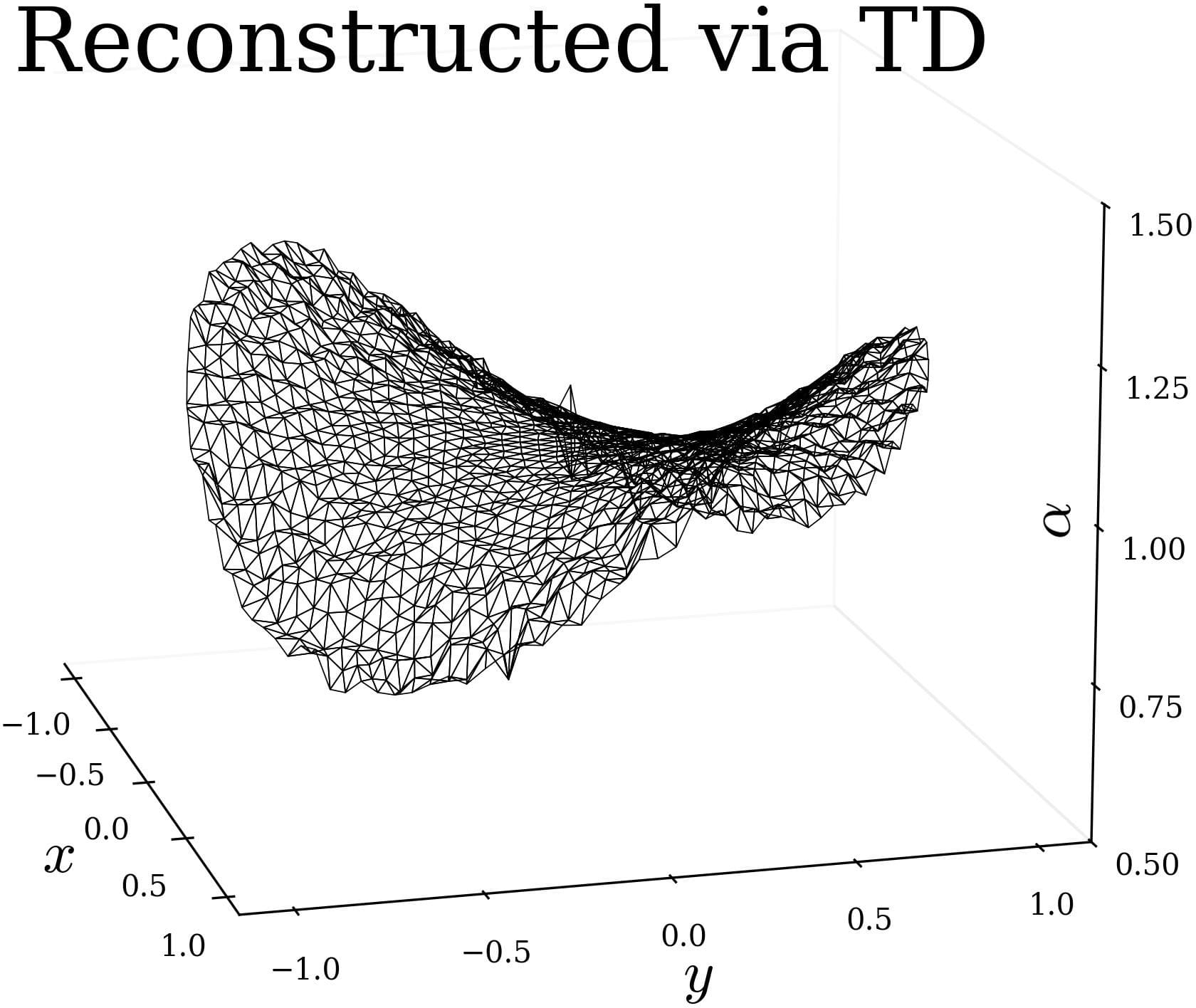}} \ 
\resizebox{0.225\textwidth}{!}{\includegraphics{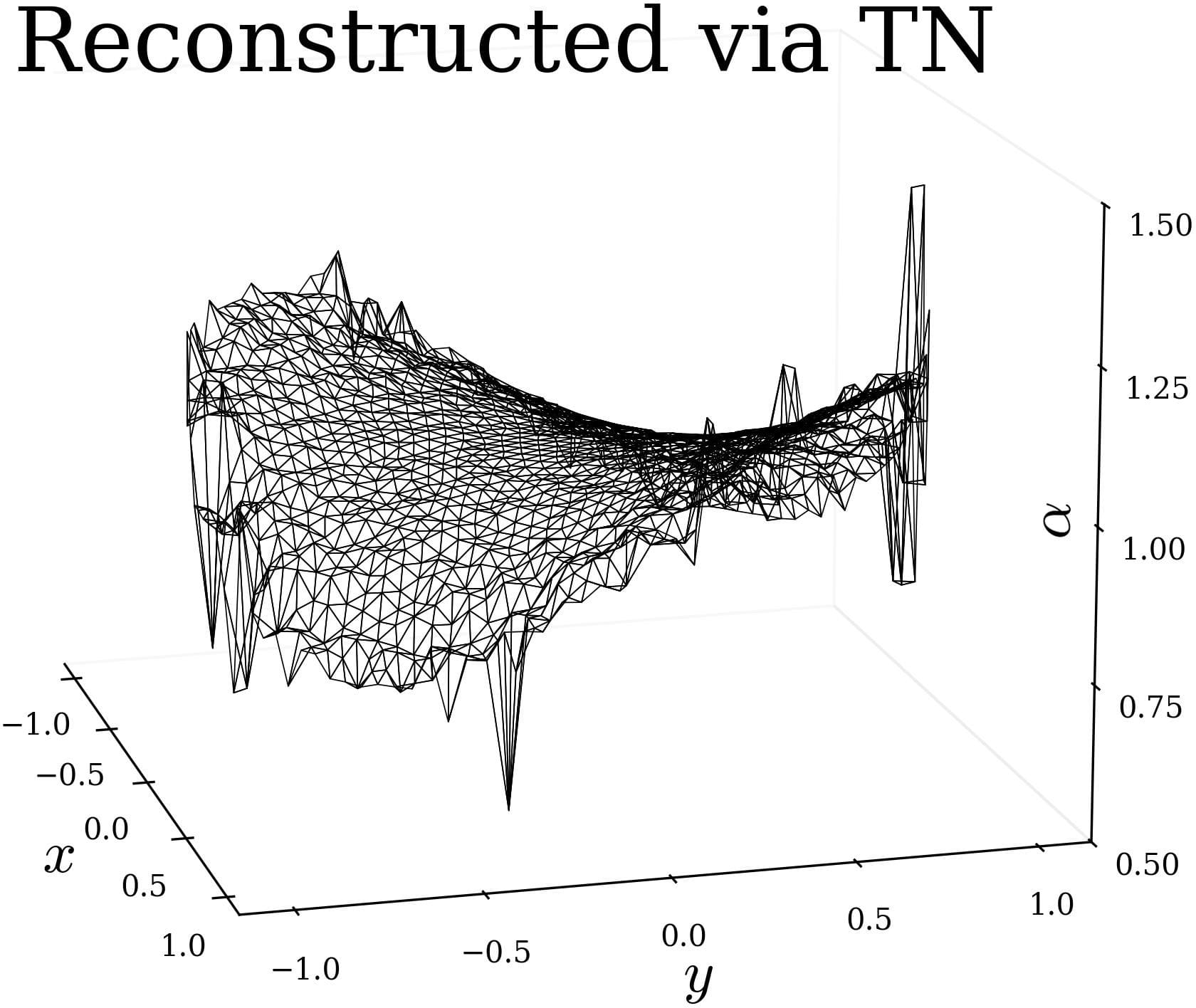}}\\ [1em]
\resizebox{0.225\textwidth}{!}{\includegraphics{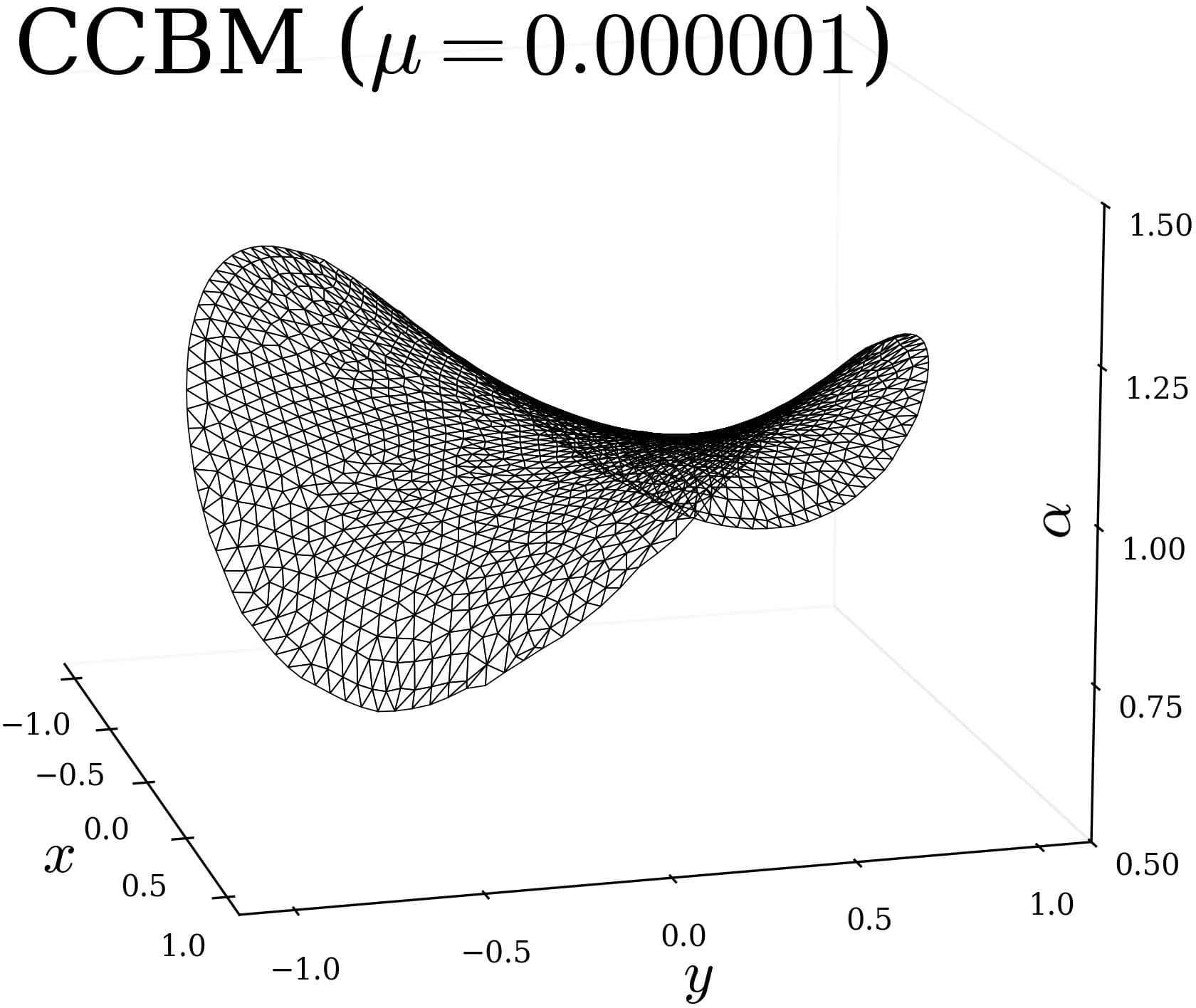}} \ 
\resizebox{0.225\textwidth}{!}{\includegraphics{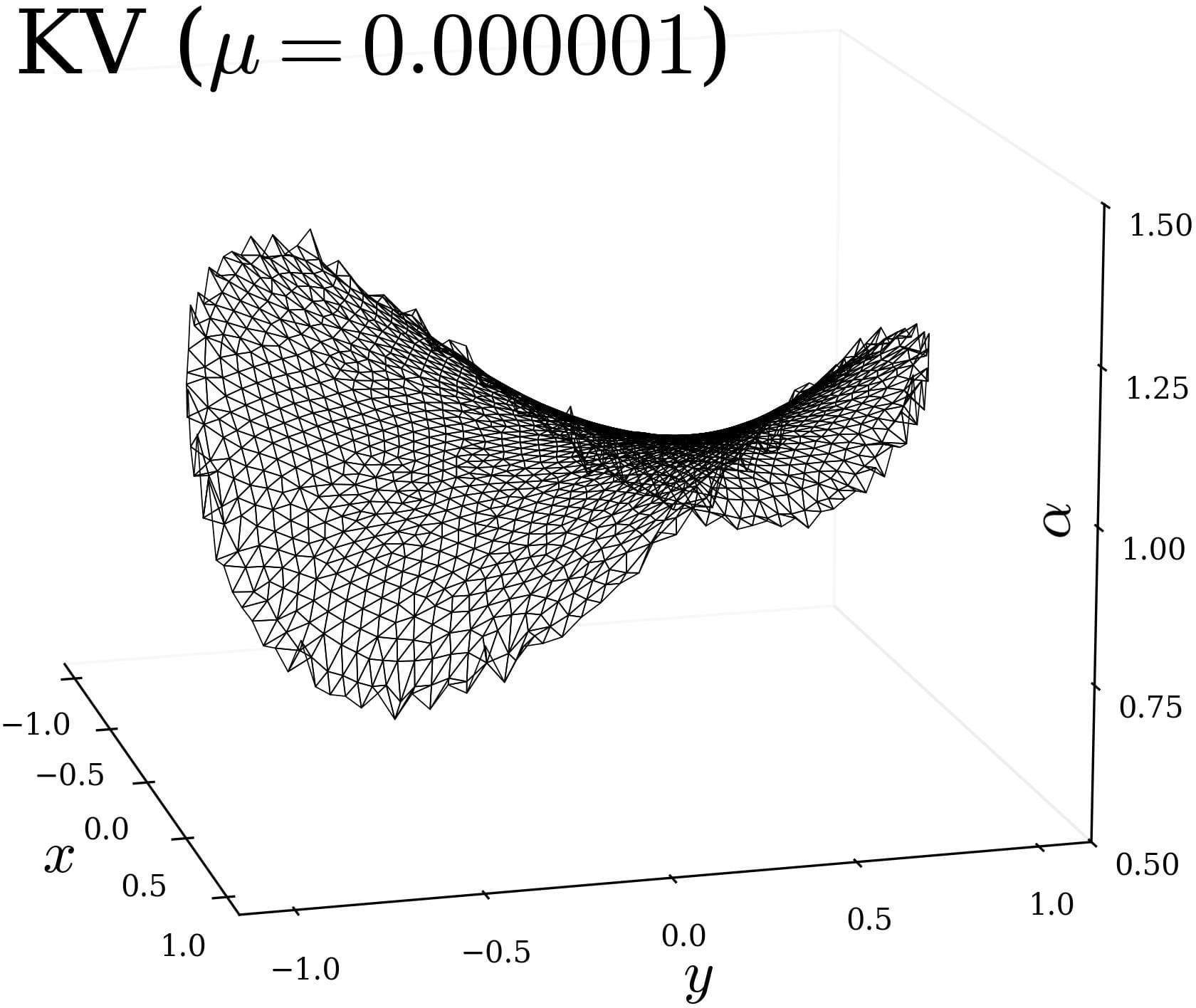}} \ 
\resizebox{0.225\textwidth}{!}{\includegraphics{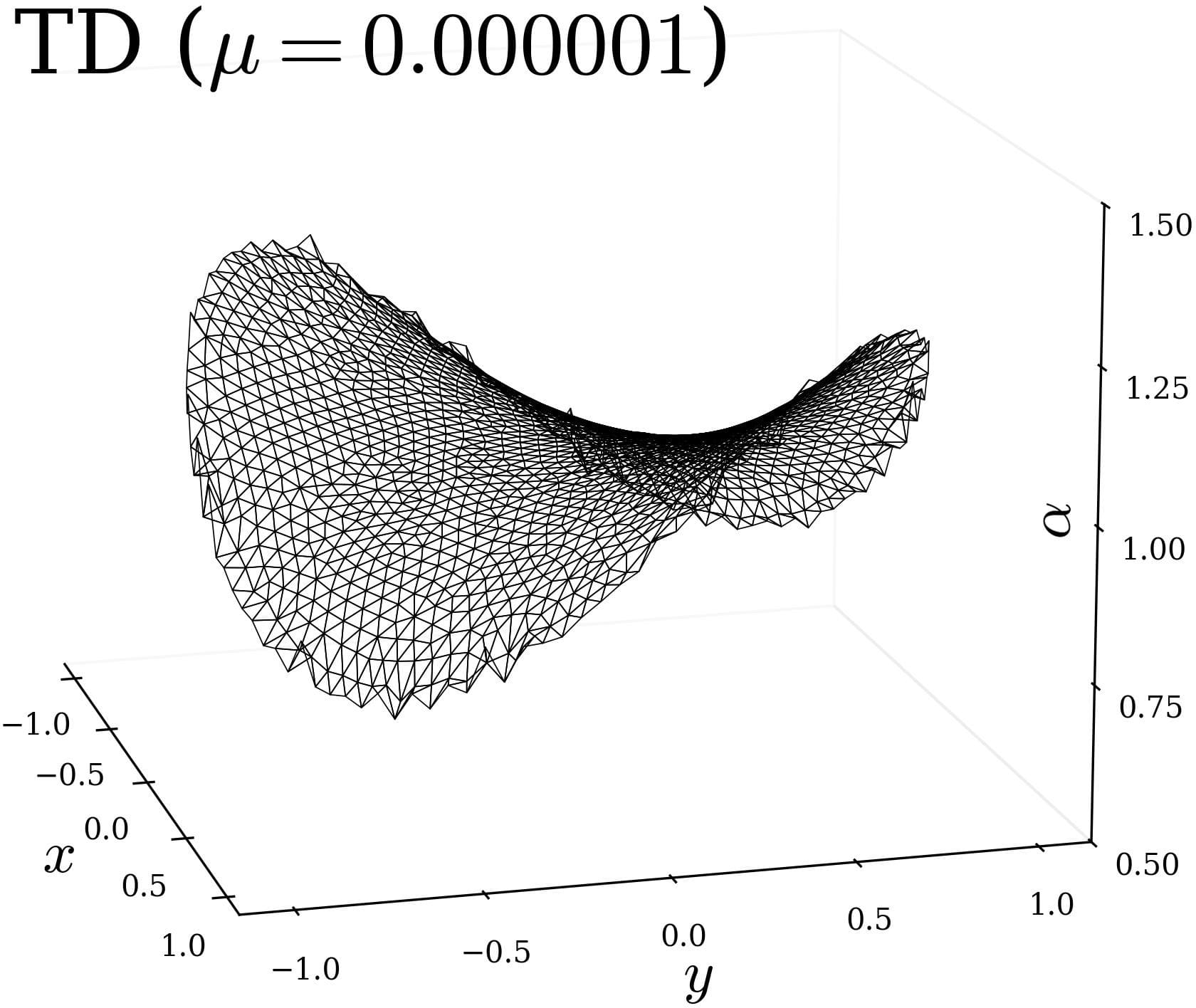}} \ 
\resizebox{0.225\textwidth}{!}{\includegraphics{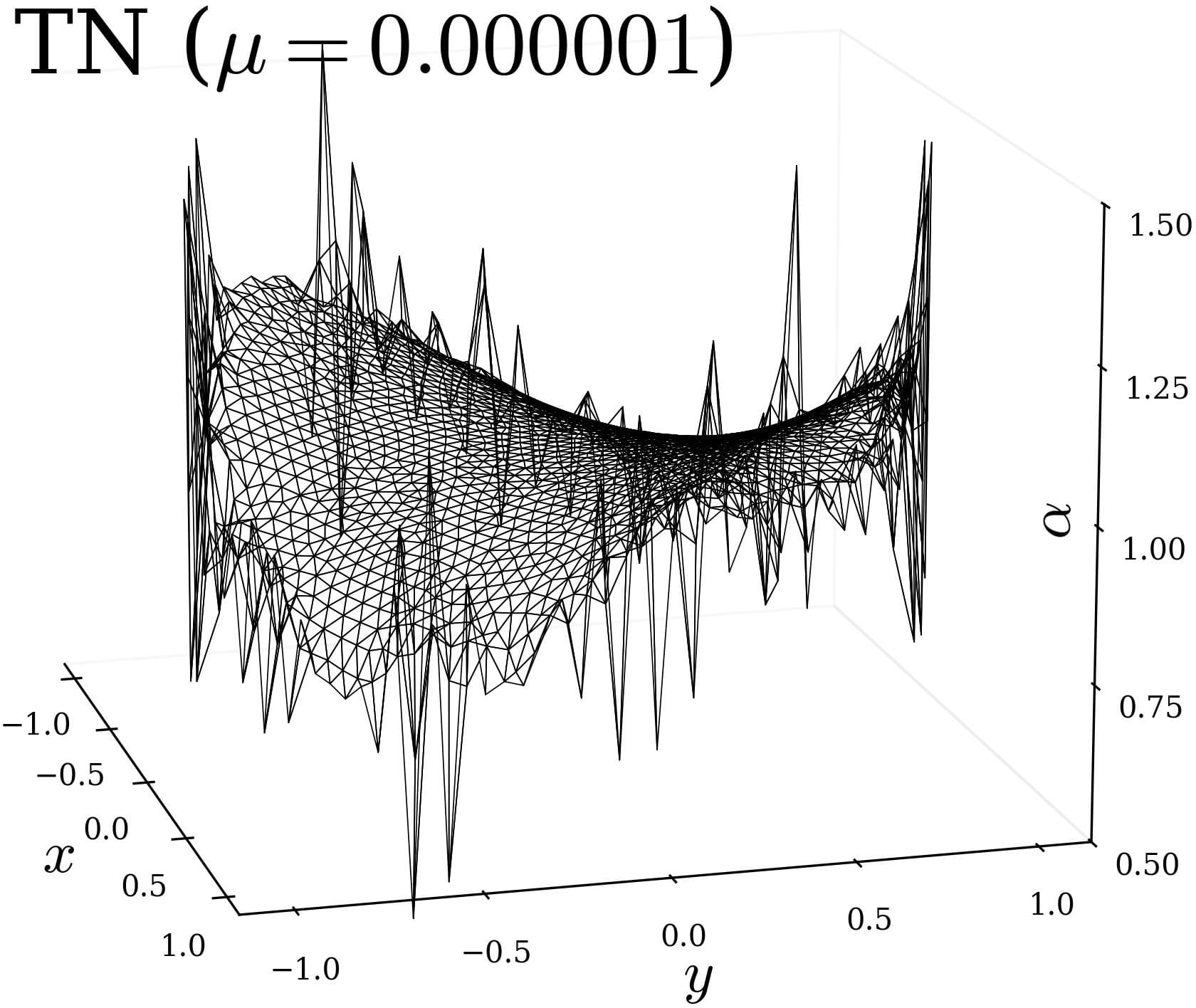}}\\[1em]
\resizebox{0.225\textwidth}{!}{\includegraphics{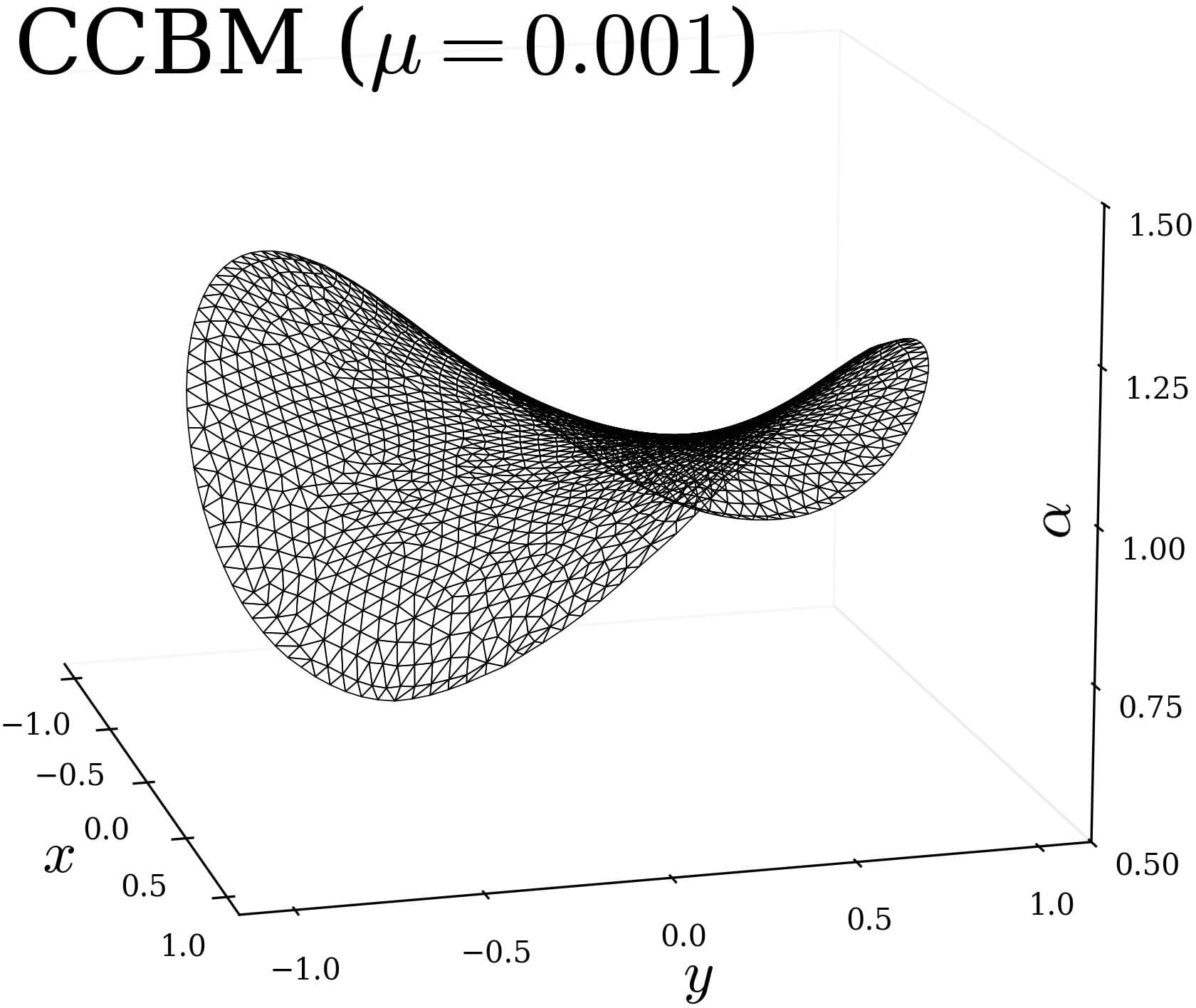}} \ 
\resizebox{0.225\textwidth}{!}{\includegraphics{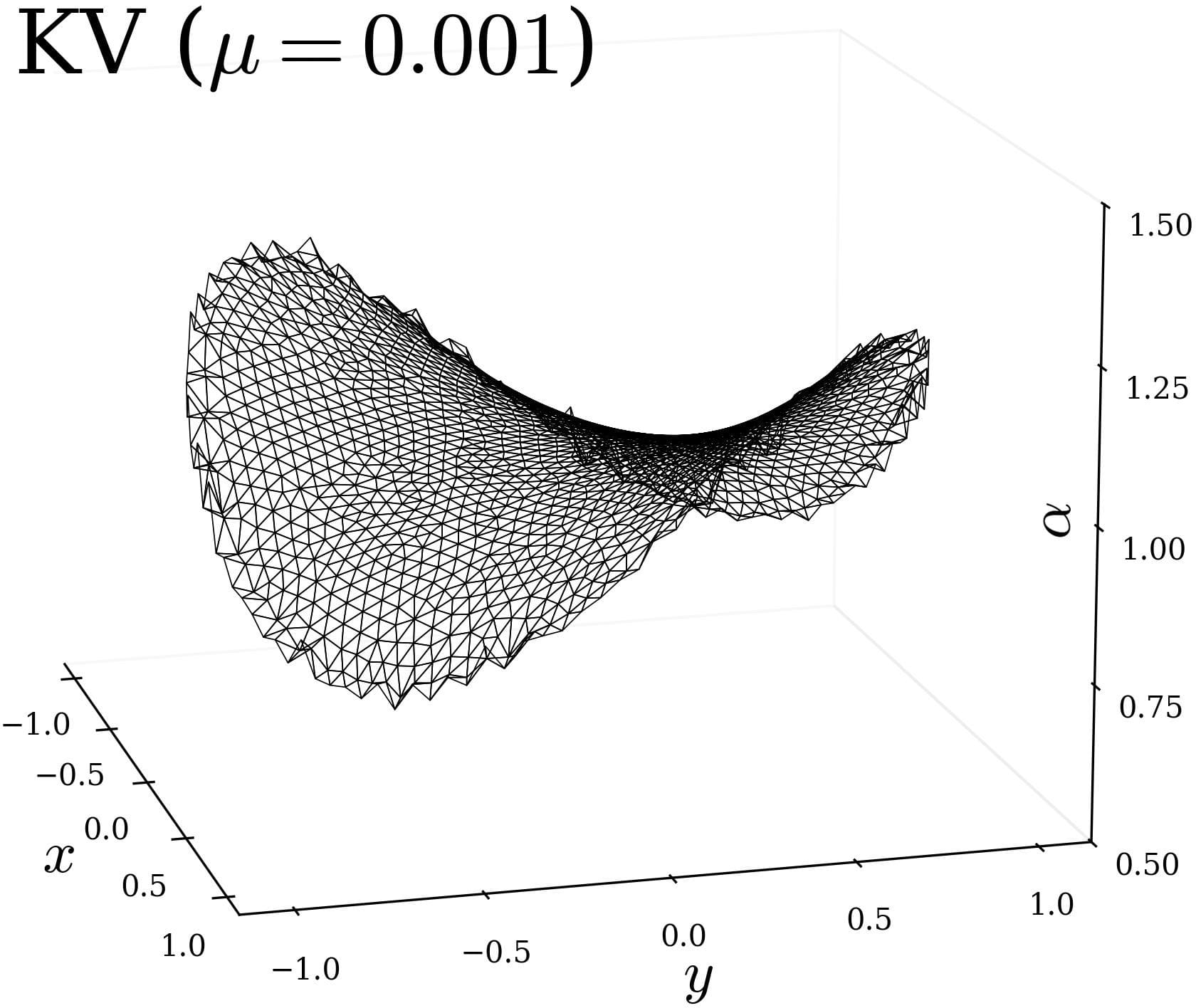}} \ 
\resizebox{0.225\textwidth}{!}{\includegraphics{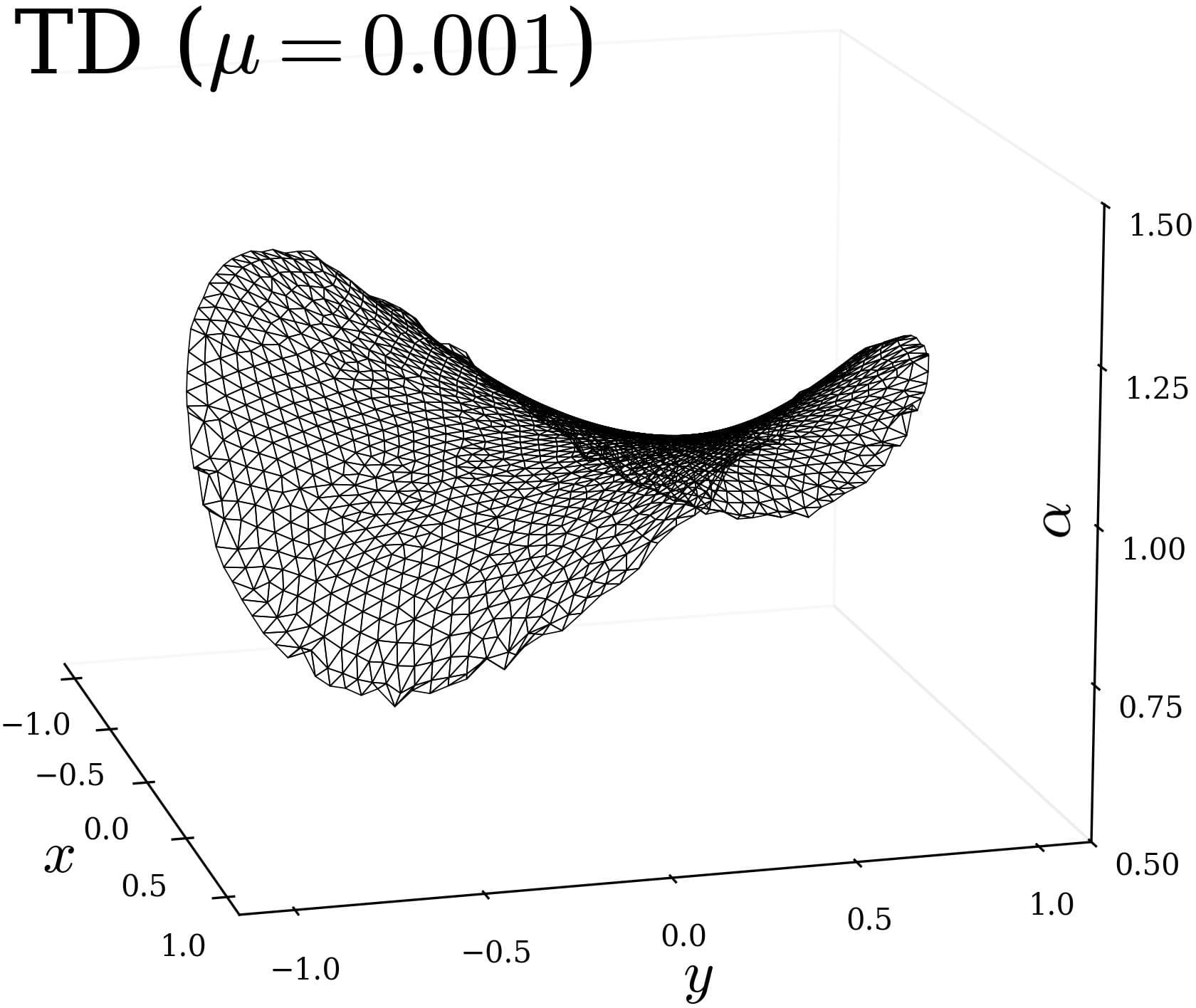}} \ 
\resizebox{0.225\textwidth}{!}{\includegraphics{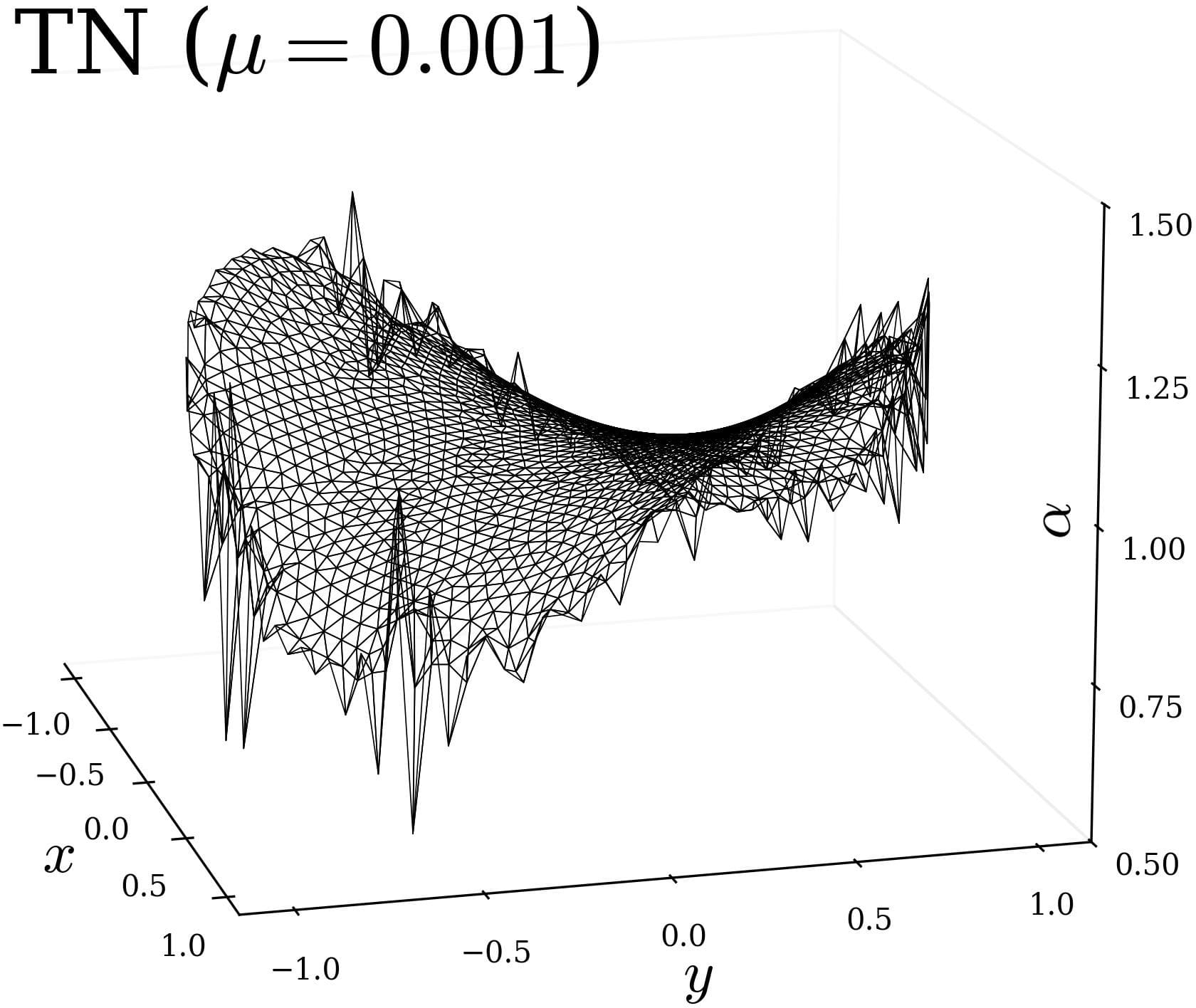}}  \\[1em]
\resizebox{0.225\textwidth}{!}{\includegraphics{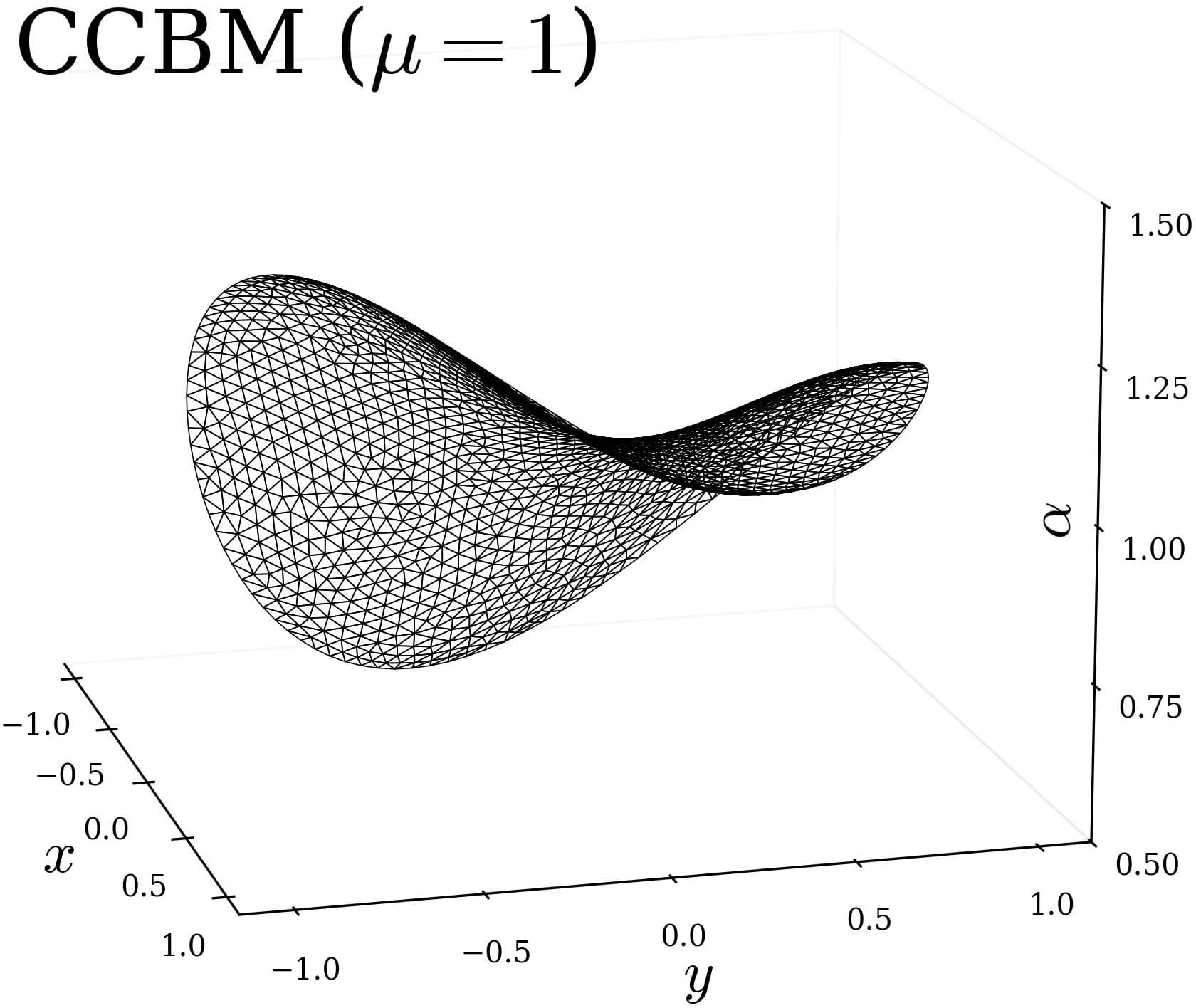}} \ 
\resizebox{0.225\textwidth}{!}{\includegraphics{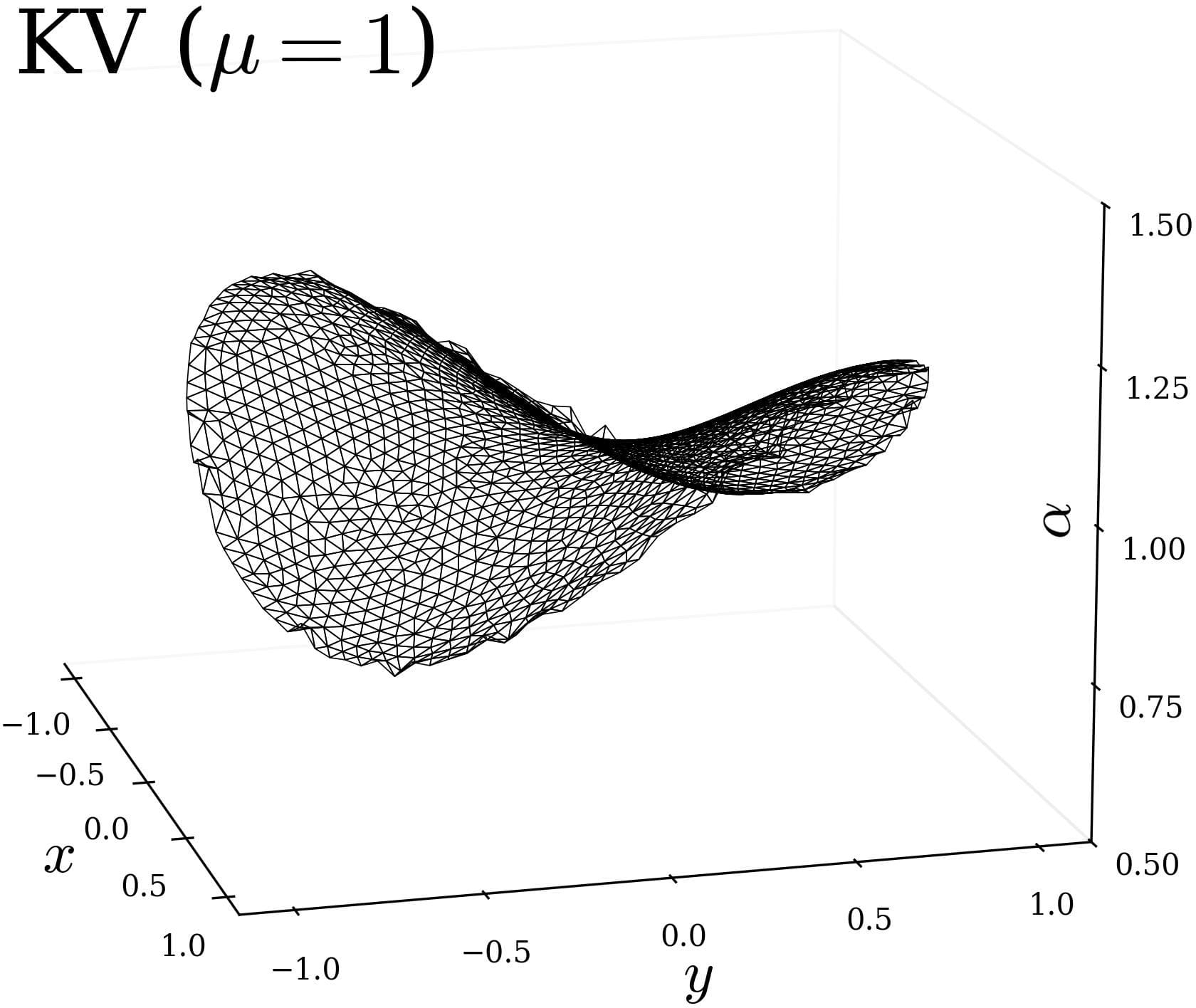}} \ 
\resizebox{0.225\textwidth}{!}{\includegraphics{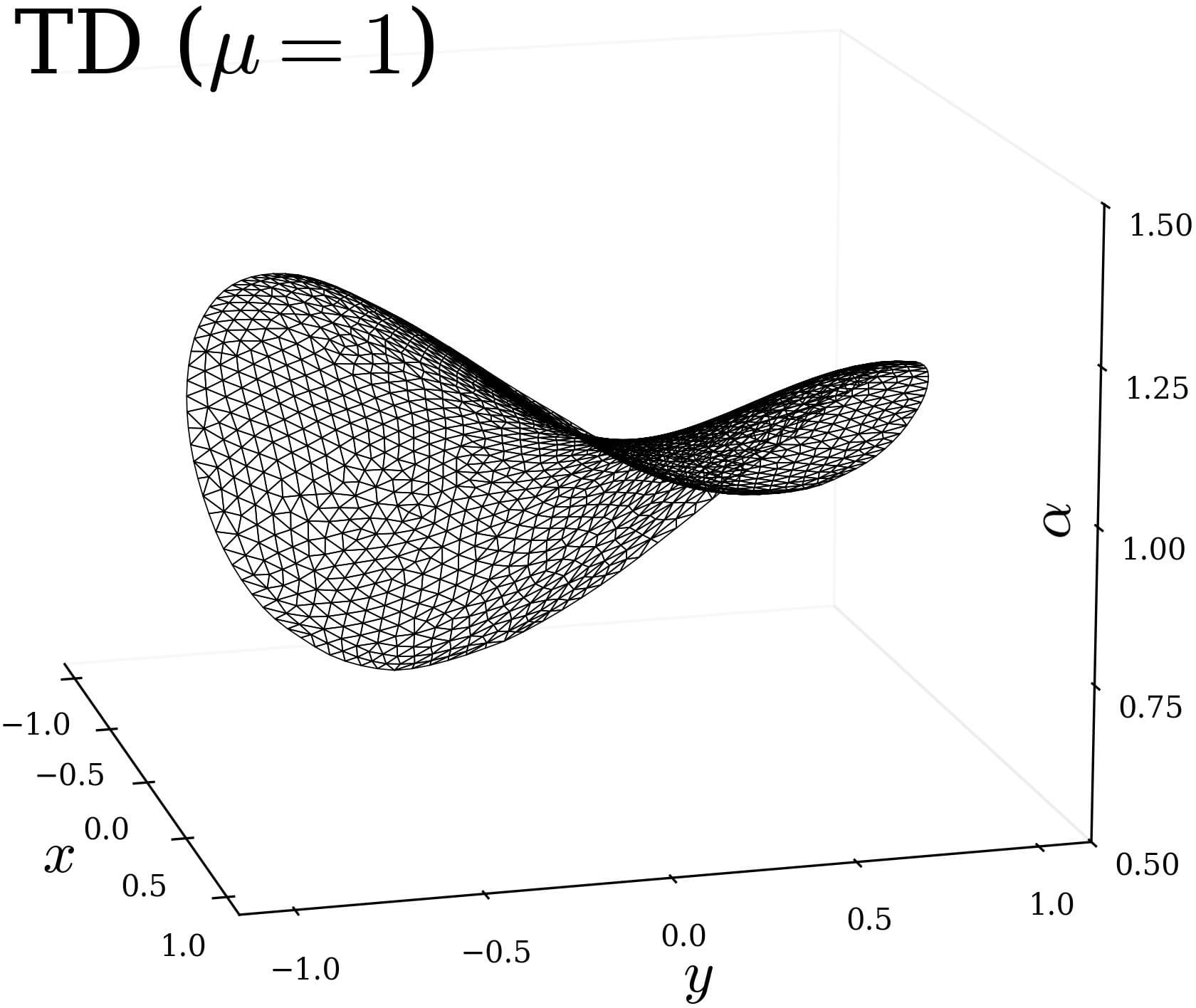}} \ 
\resizebox{0.225\textwidth}{!}{\includegraphics{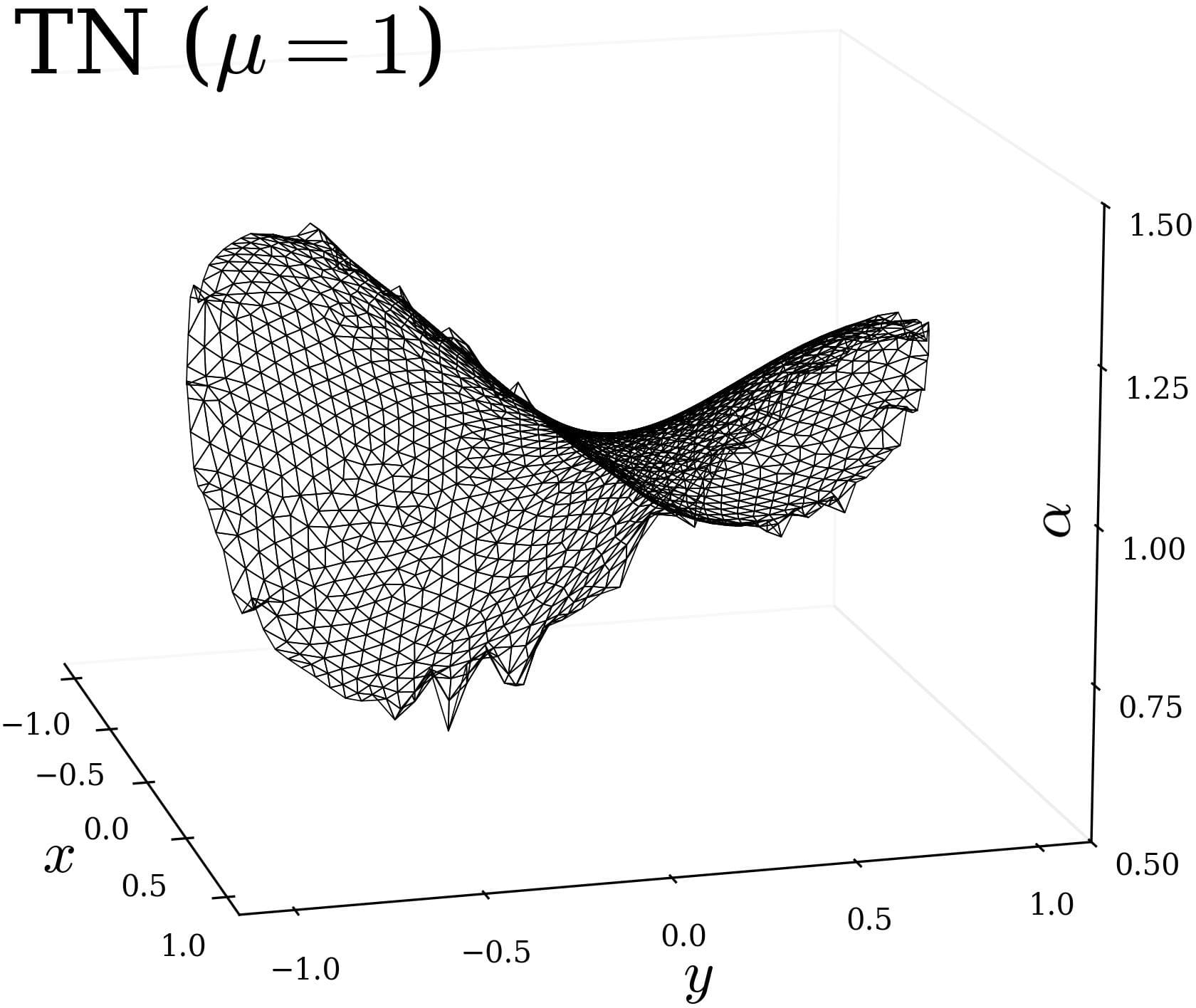}}  
\caption{Exact diffusion coefficient profile (top row) and reconstruction results using exact measurements, without $H^{1}$ smoothing (second row) and with $H^{1}$ smoothing for different smoothing parameters $\mu$ (third to fifth rows).}
\label{fig:effect_of_mu}
\end{figure}
\begin{figure}[htp!]
\centering
\resizebox{0.225\textwidth}{!}{\includegraphics{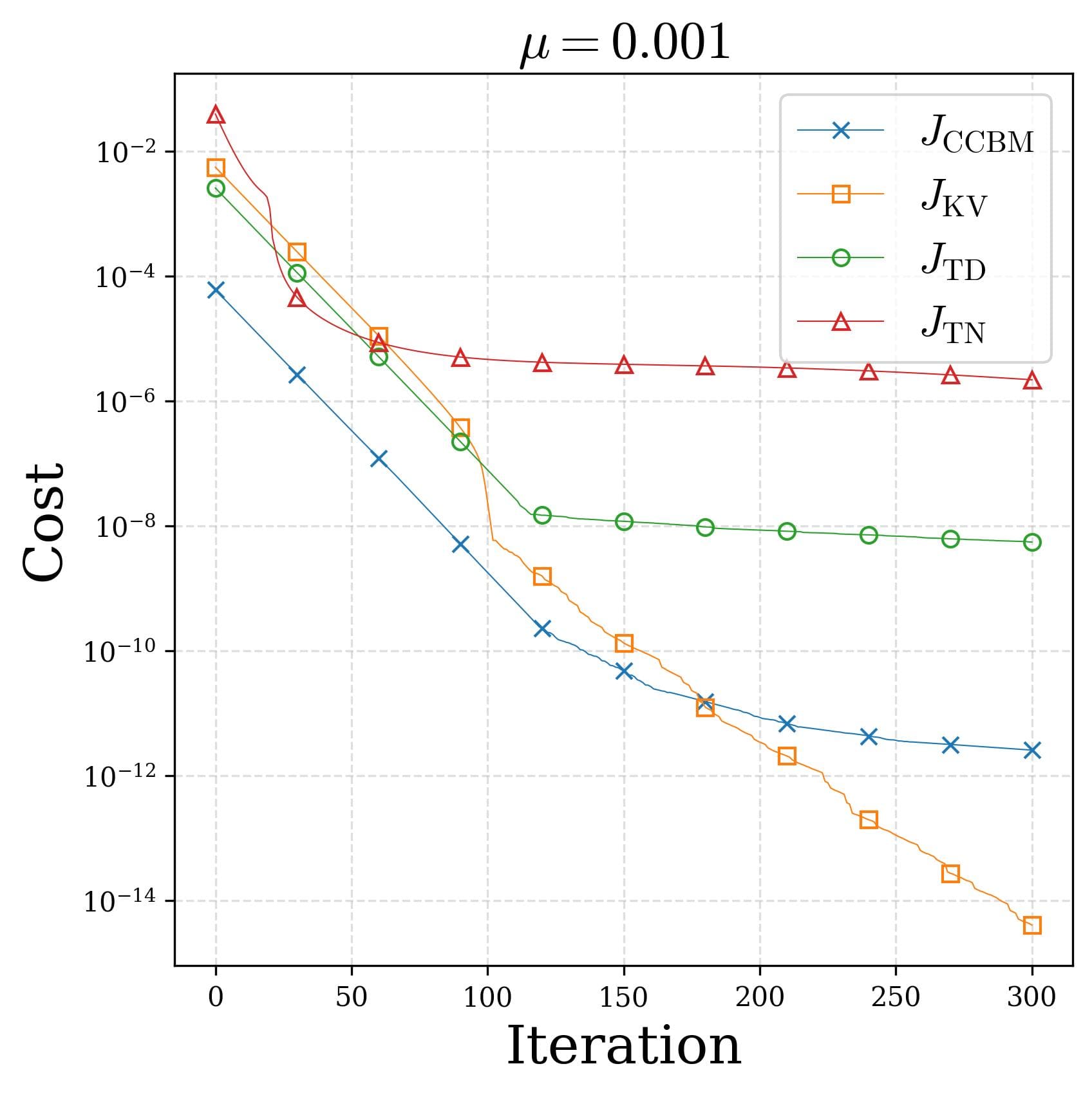}} 
\resizebox{0.225\textwidth}{!}{\includegraphics{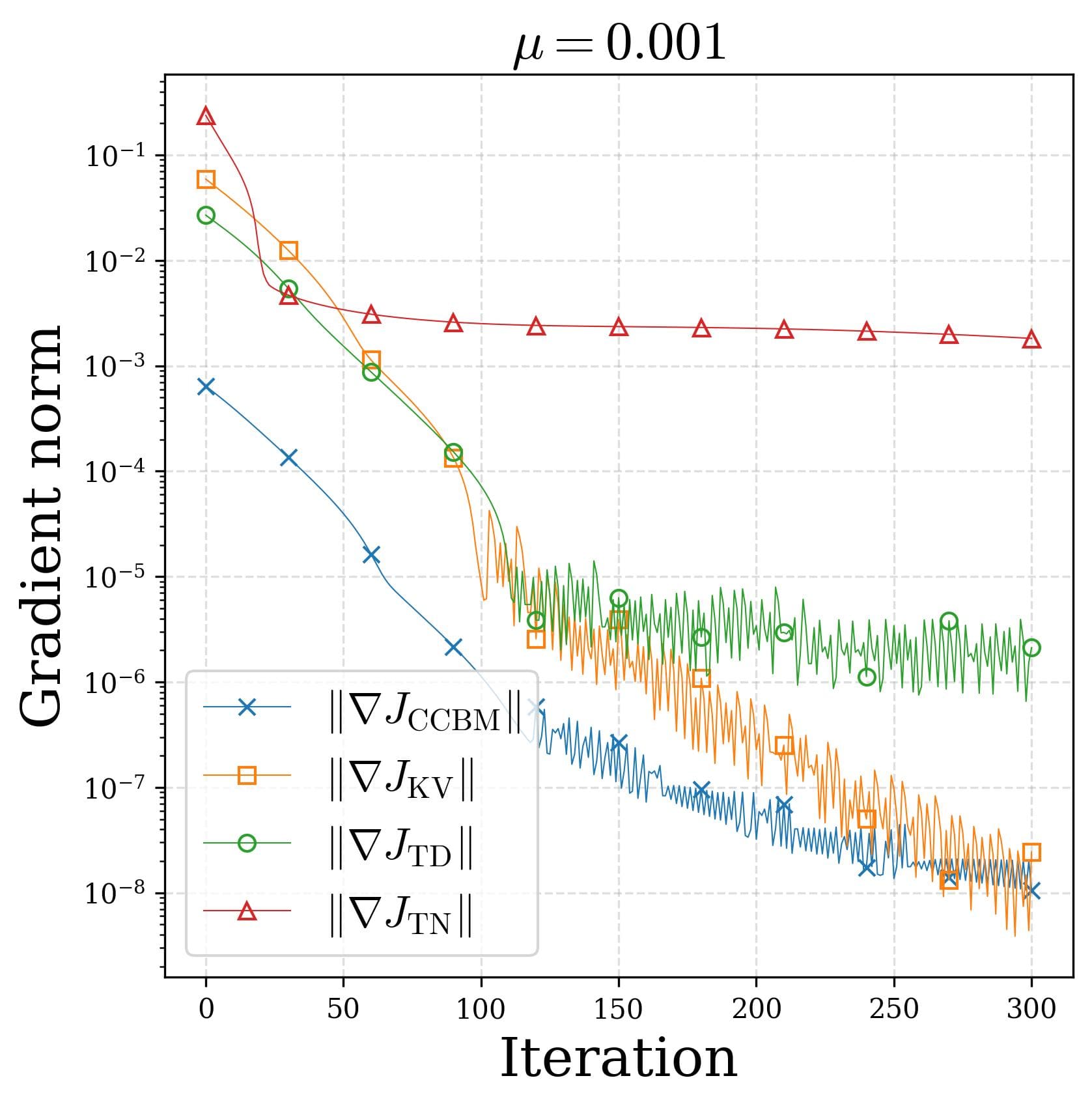}} \
\resizebox{0.225\textwidth}{!}{\includegraphics{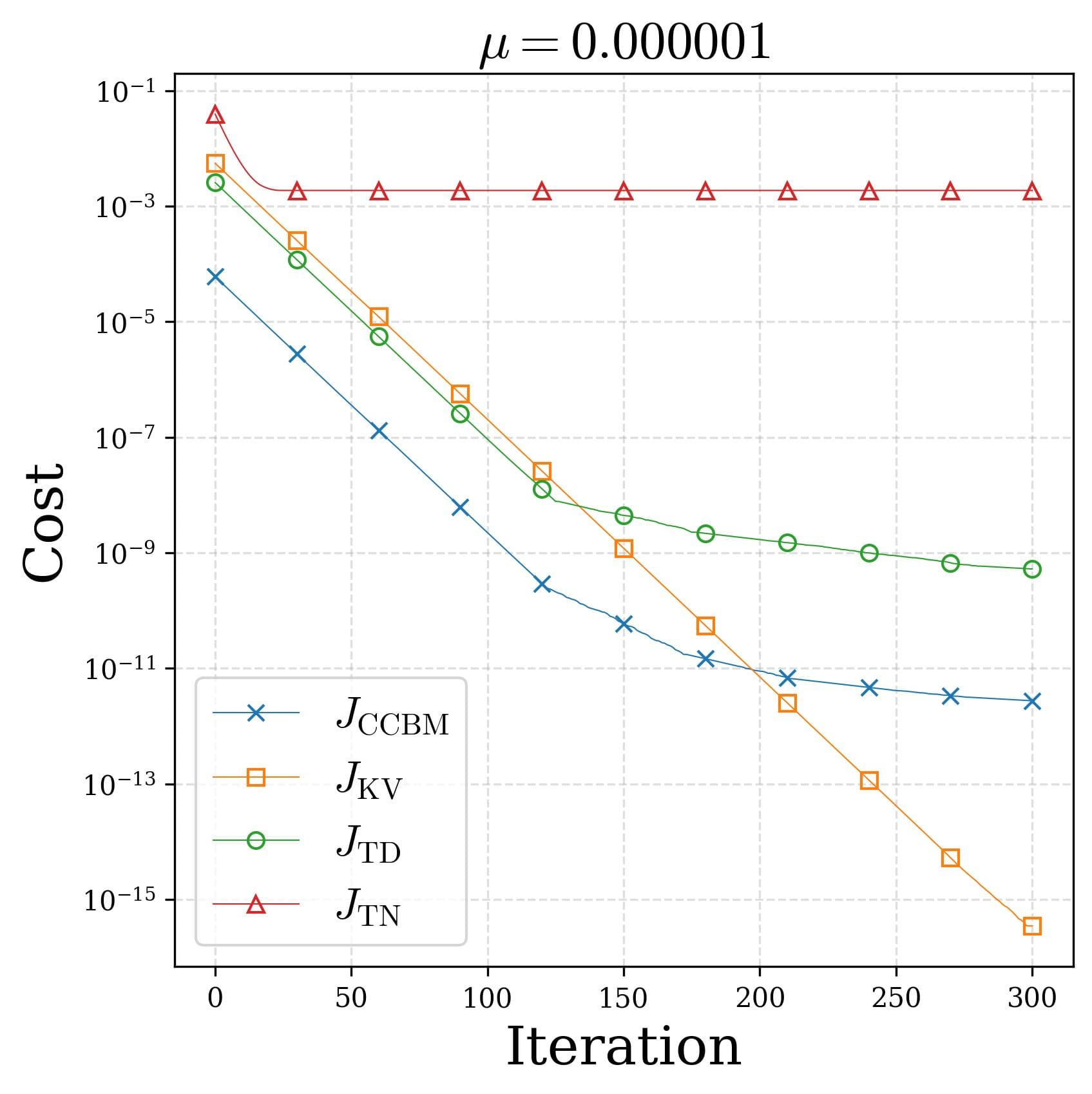}} 
\resizebox{0.225\textwidth}{!}{\includegraphics{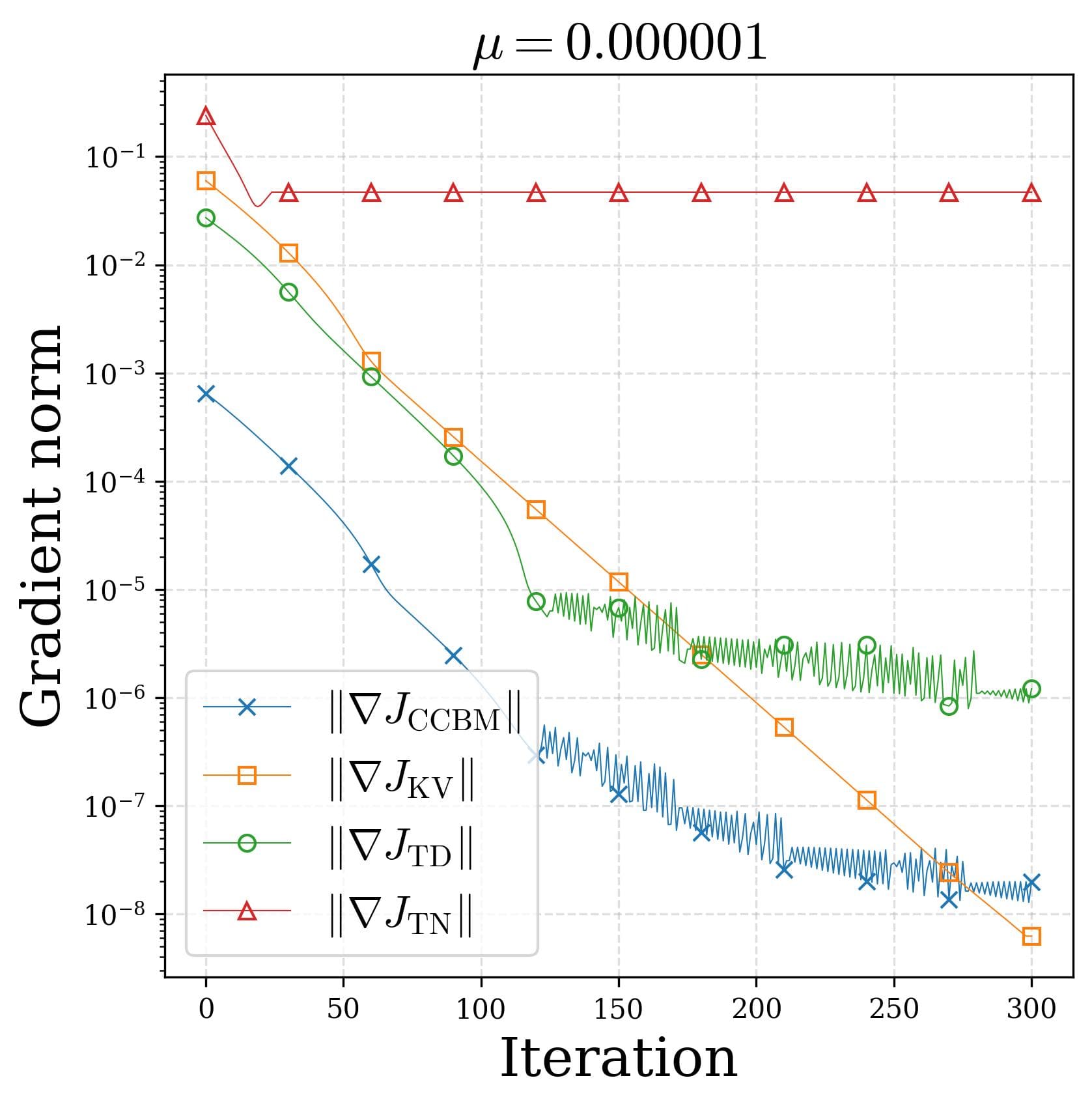}}
\caption{Histories of the cost functional and gradient norm corresponding to Figure~\ref{fig:effect_of_mu}. Left: $\mu=10^{-3}$; right: $\mu=10^{-6}$.}
\label{fig:effect_of_mu_cost_and_gradient}
\end{figure}
\subsubsection{Effect of the regularization parameter $\rho$ under noisy data}\label{subsec:effect_of_rho}
We next examine the influence of noise and regularization on the different formulations to assess their relative robustness. 
Figure~\ref{fig:measurements} shows that increasing the noise level $\delta$ induces strong oscillations in the boundary measurements, amplifying the ill-posedness of the inverse problem.

Figures~\ref{fig:effect_of_rho} and~\ref{fig:effect_of_rho_with_regularization} illustrate the effect of the Tikhonov parameter $\rho$ for fixed $\delta=0.0003$, without and with $H^1$ gradient smoothing, respectively. 
Without gradient smoothing (Figure~\ref{fig:effect_of_rho}), all methods exhibit increasing oscillations as $\rho$ decreases, but the deterioration is markedly less pronounced for the CCBM formulation, which maintains a stable reconstruction over a wider range of $\rho$. 
Note that, in the CCBM formulation, a large $\rho$ reduces the sensitivity of the cost function, consistent with Remark~\ref{rem:effect_of_rho}. Of course, what counts as ``large" depends on the cost function; typically, CCBM convergence occurs around $\rho \sim 10^{-7}$ when Tikhonov regularization is used, and even smaller, $\rho \sim 10^{-11}$, without regularization (see Figure~\ref{fig:effect_of_mu_cost_and_gradient}).
With gradient smoothing (Figure~\ref{fig:effect_of_rho_with_regularization}), stability is substantially improved for all approaches; however, CCBM consistently delivers the most robust and accurate reconstructions, while KV, TD, and especially TN remain more sensitive to small $\rho$.

The corresponding cost and gradient-norm histories in Figure~\ref{fig:effect_of_rho_cost_and_gradient} further confirm this behavior: CCBM exhibits smoother and more reliable convergence, particularly in the low-$\rho$ regime, highlighting its superior robustness with respect to noise and regularization.
Observe that the gradient-norm histories corroborate Proposition~\ref{prop:grad_descent_convergence_unreg}, confirming that the iteration produces an approximate minimizer of $\Je$.

\begin{figure}[htp!]
\centering
\resizebox{0.225\textwidth}{!}{\includegraphics{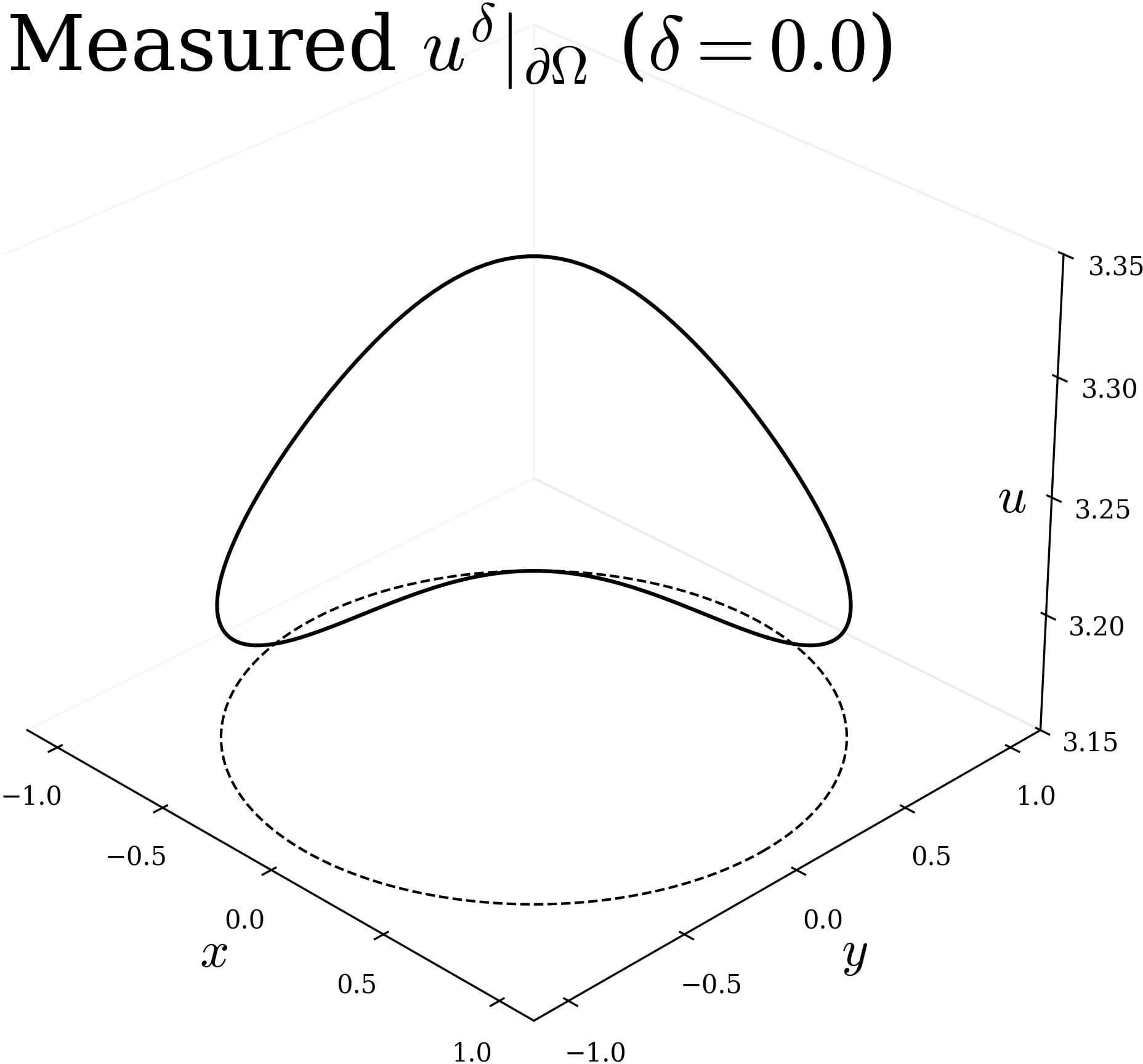}} \ 
\resizebox{0.225\textwidth}{!}{\includegraphics{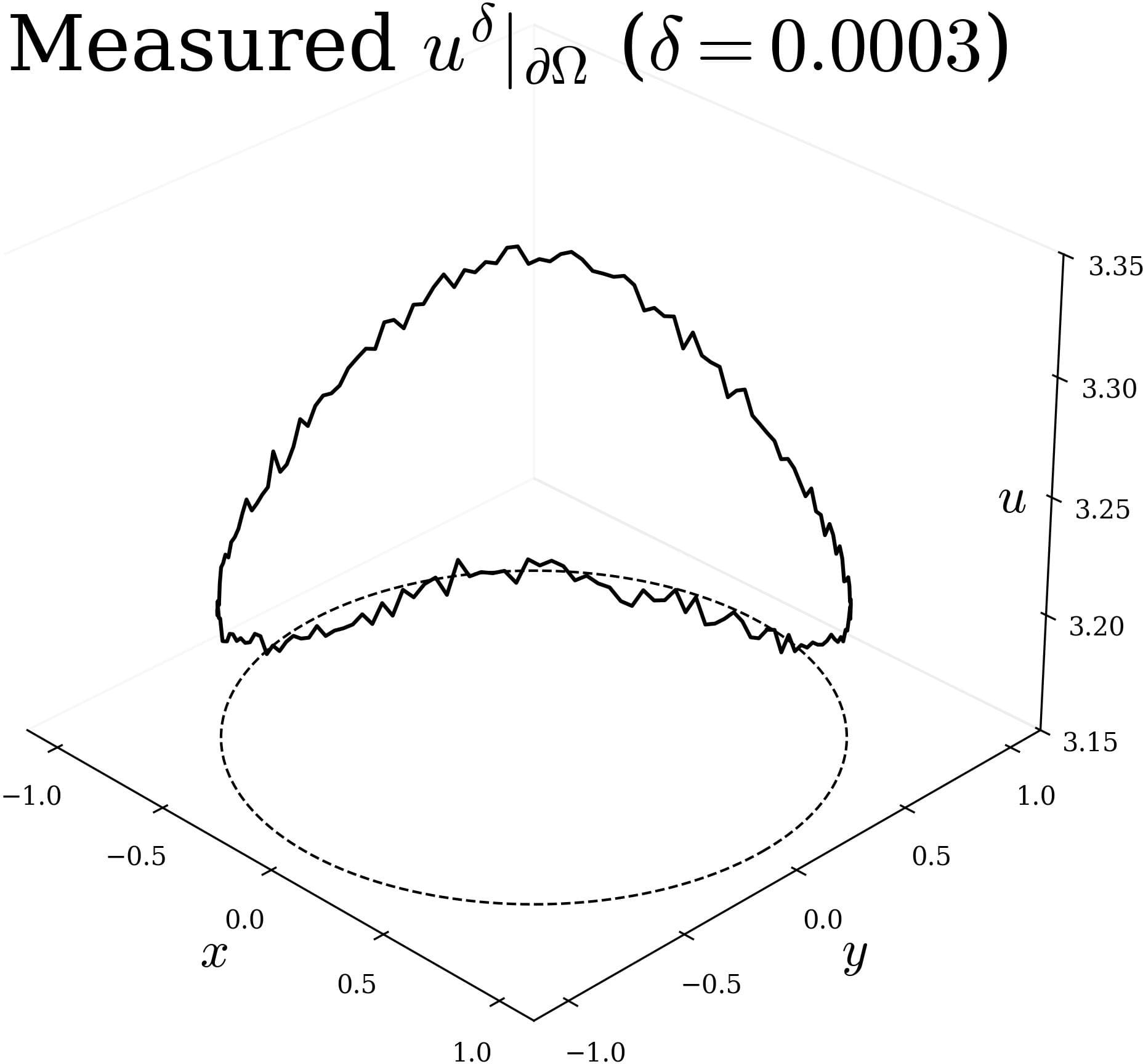}} \
\resizebox{0.225\textwidth}{!}{\includegraphics{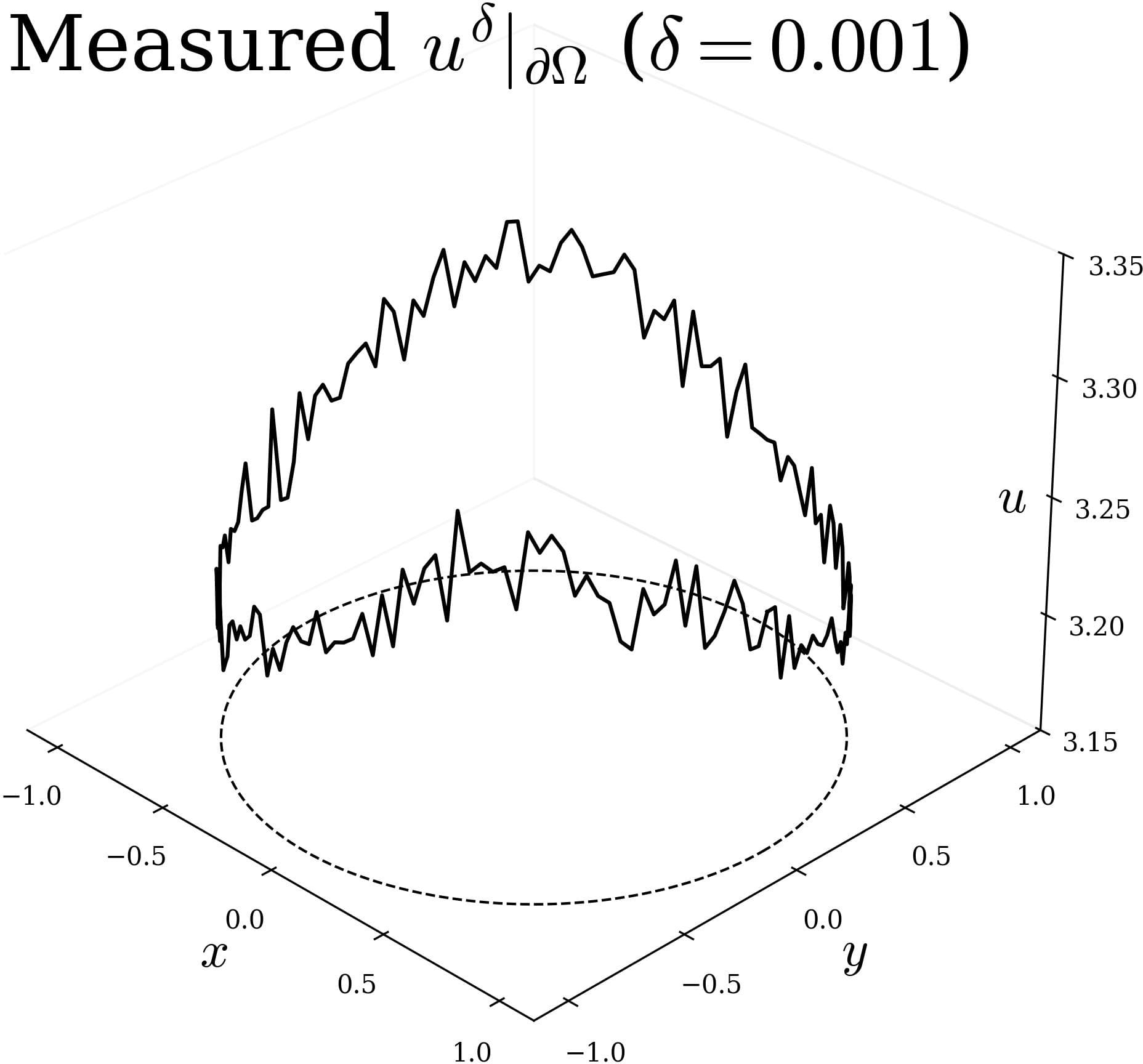}} \
\resizebox{0.225\textwidth}{!}{\includegraphics{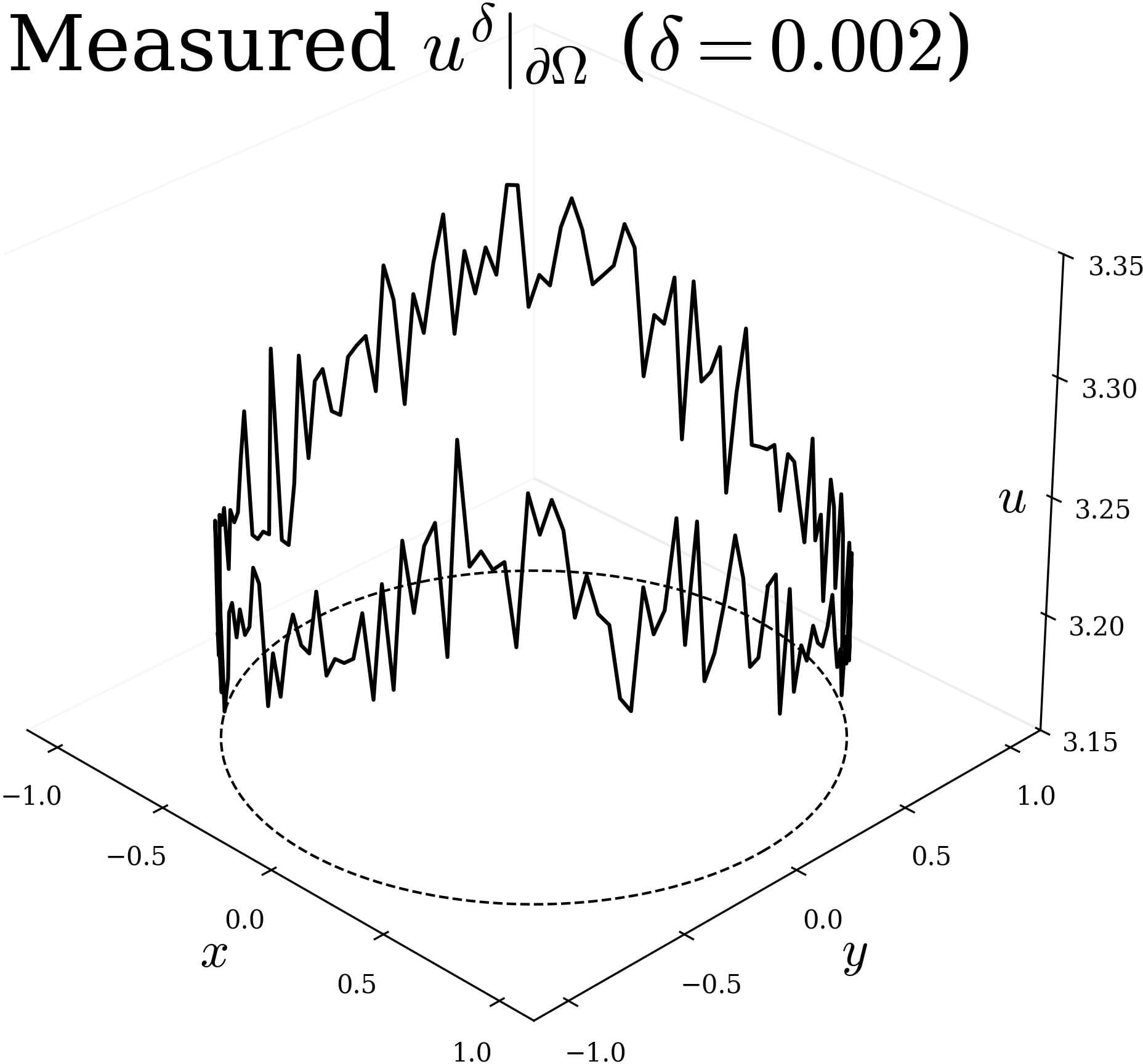}}
\caption{Boundary measurements at different noise levels $\delta$ with input data $g=1.0$.}
\label{fig:measurements}
\end{figure}
%
%
%
\begin{figure}[htp!]
\centering
\resizebox{0.225\textwidth}{!}{\includegraphics{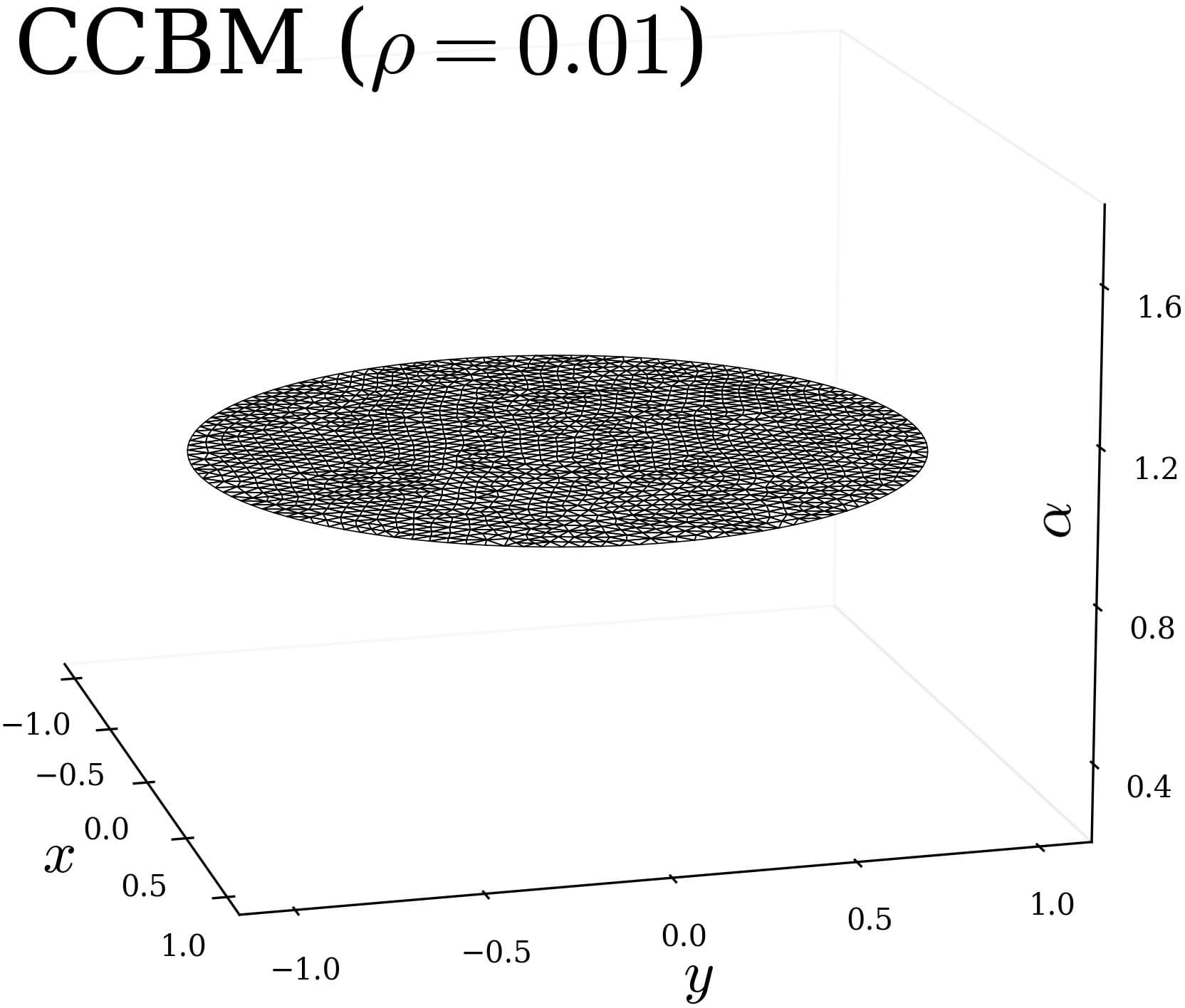}} \ 
\resizebox{0.225\textwidth}{!}{\includegraphics{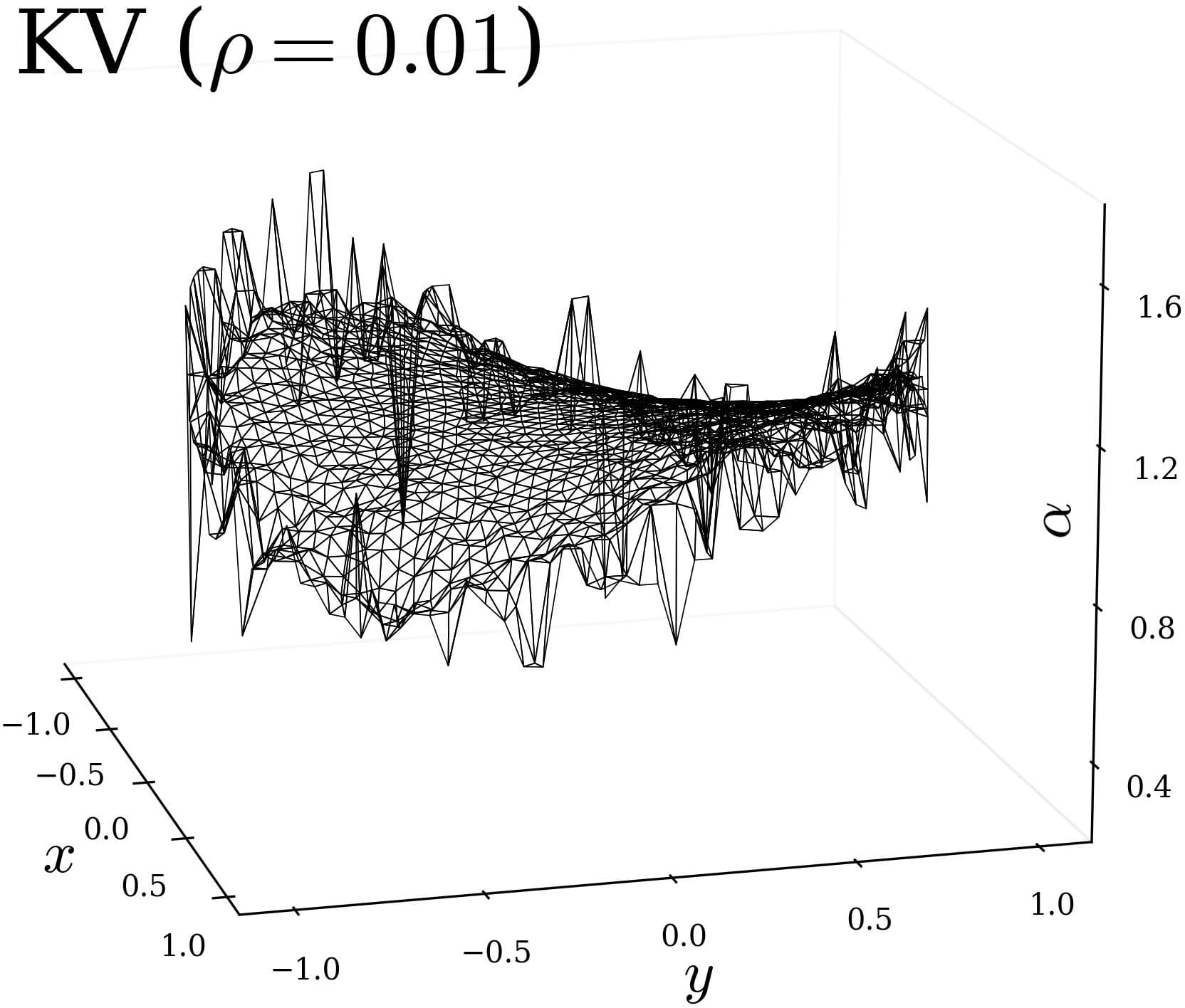}} \ 
\resizebox{0.225\textwidth}{!}{\includegraphics{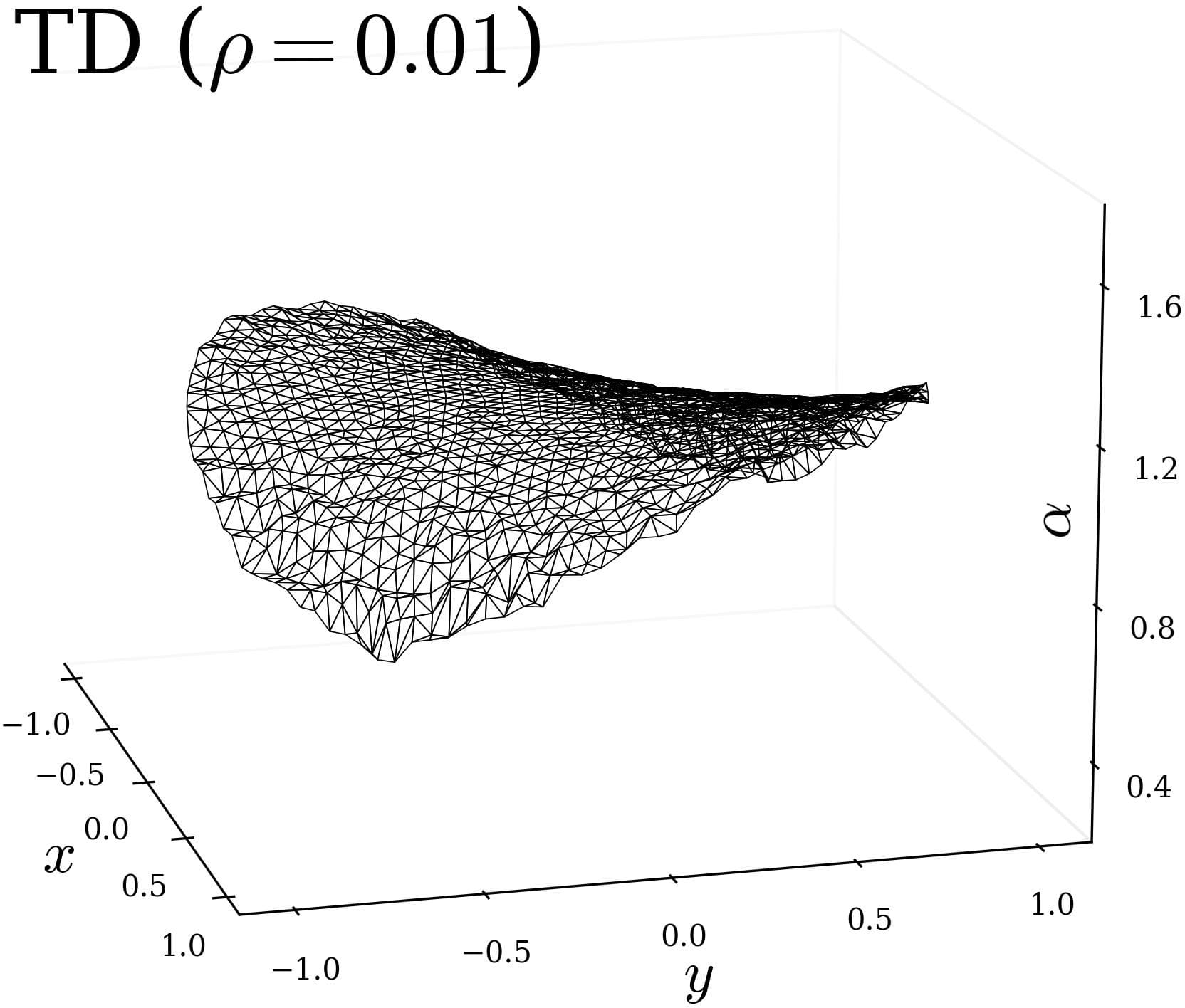}} \ 
\resizebox{0.225\textwidth}{!}{\includegraphics{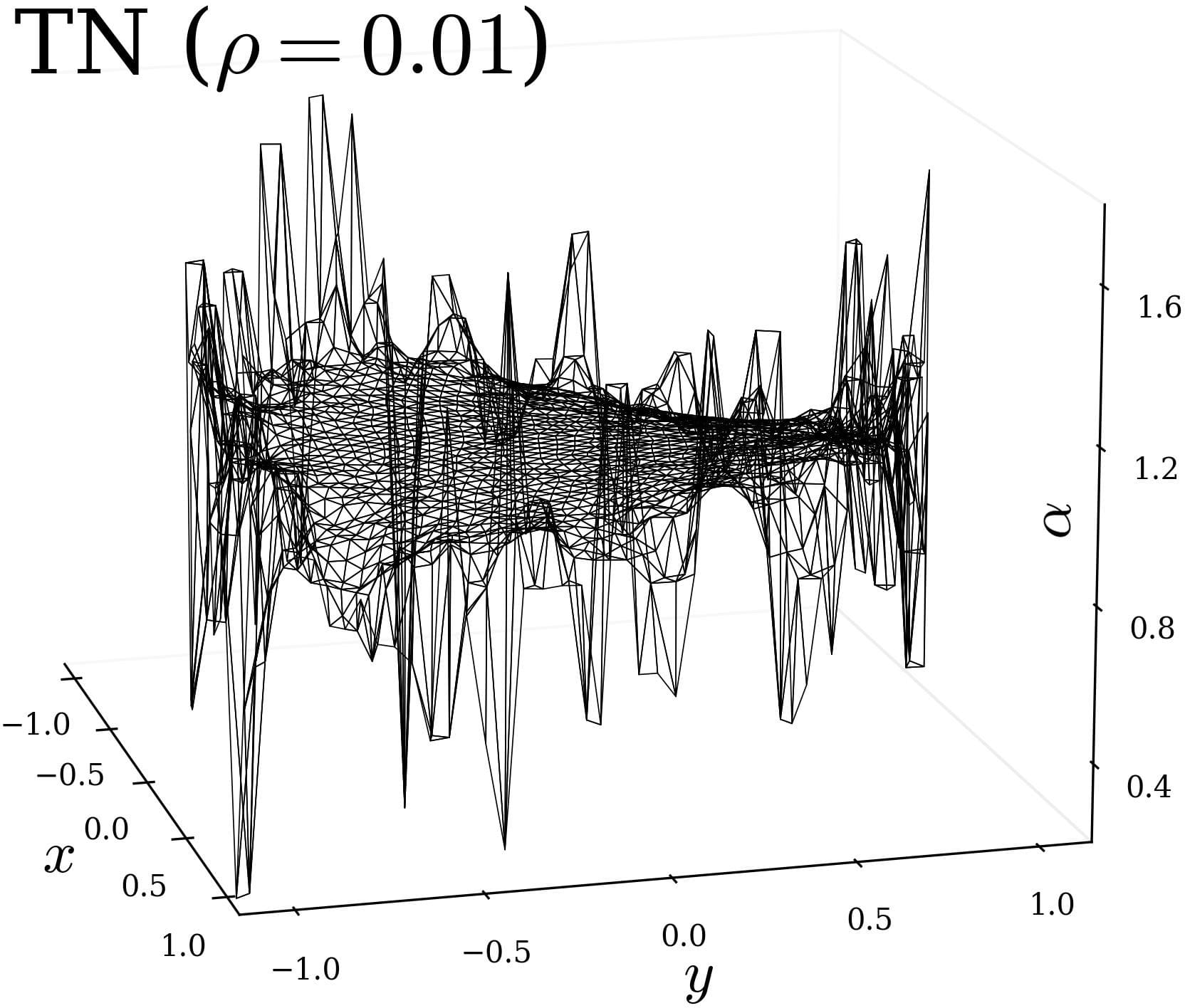}} \\[1em]
\resizebox{0.225\textwidth}{!}{\includegraphics{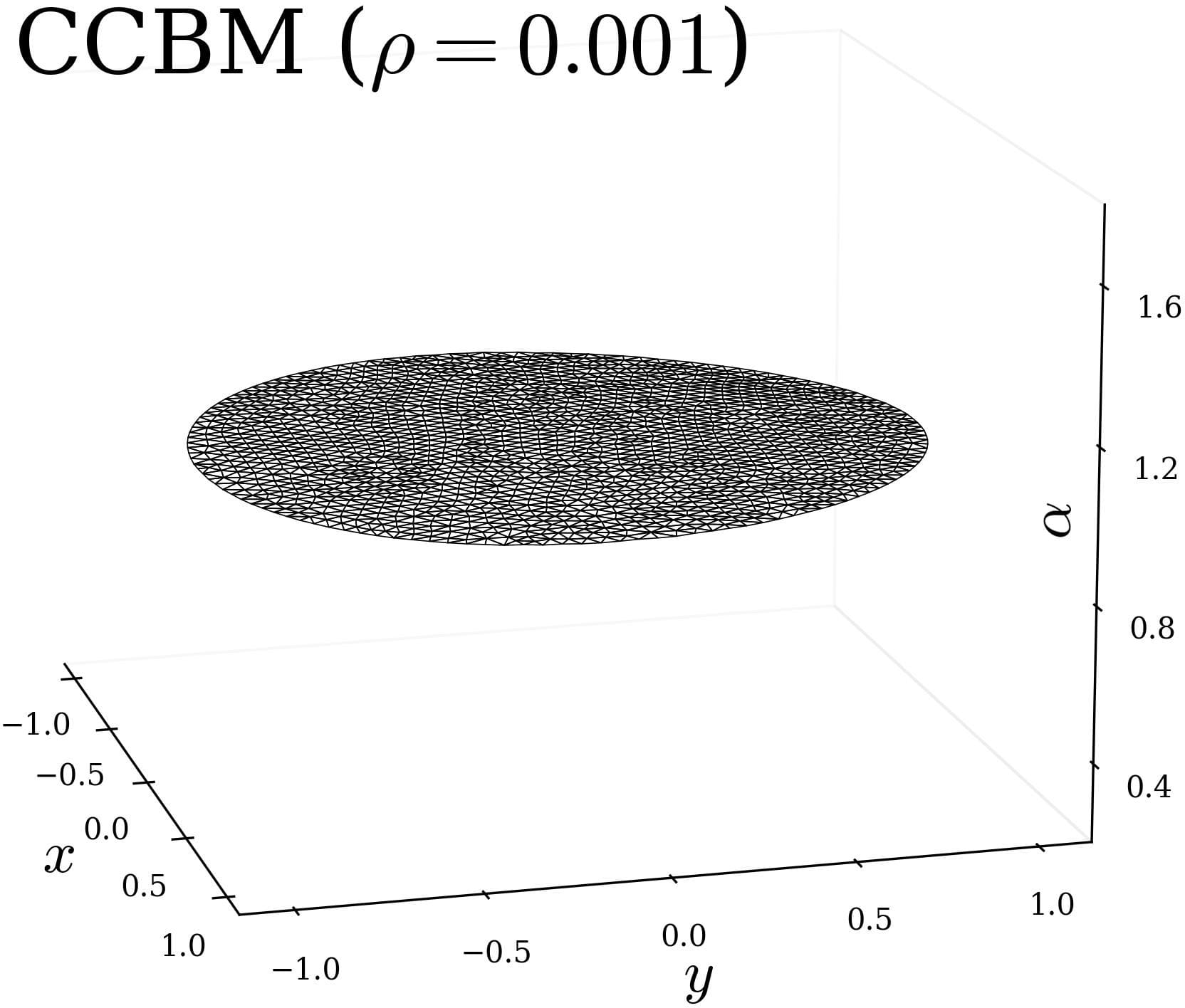}} \ 
\resizebox{0.225\textwidth}{!}{\includegraphics{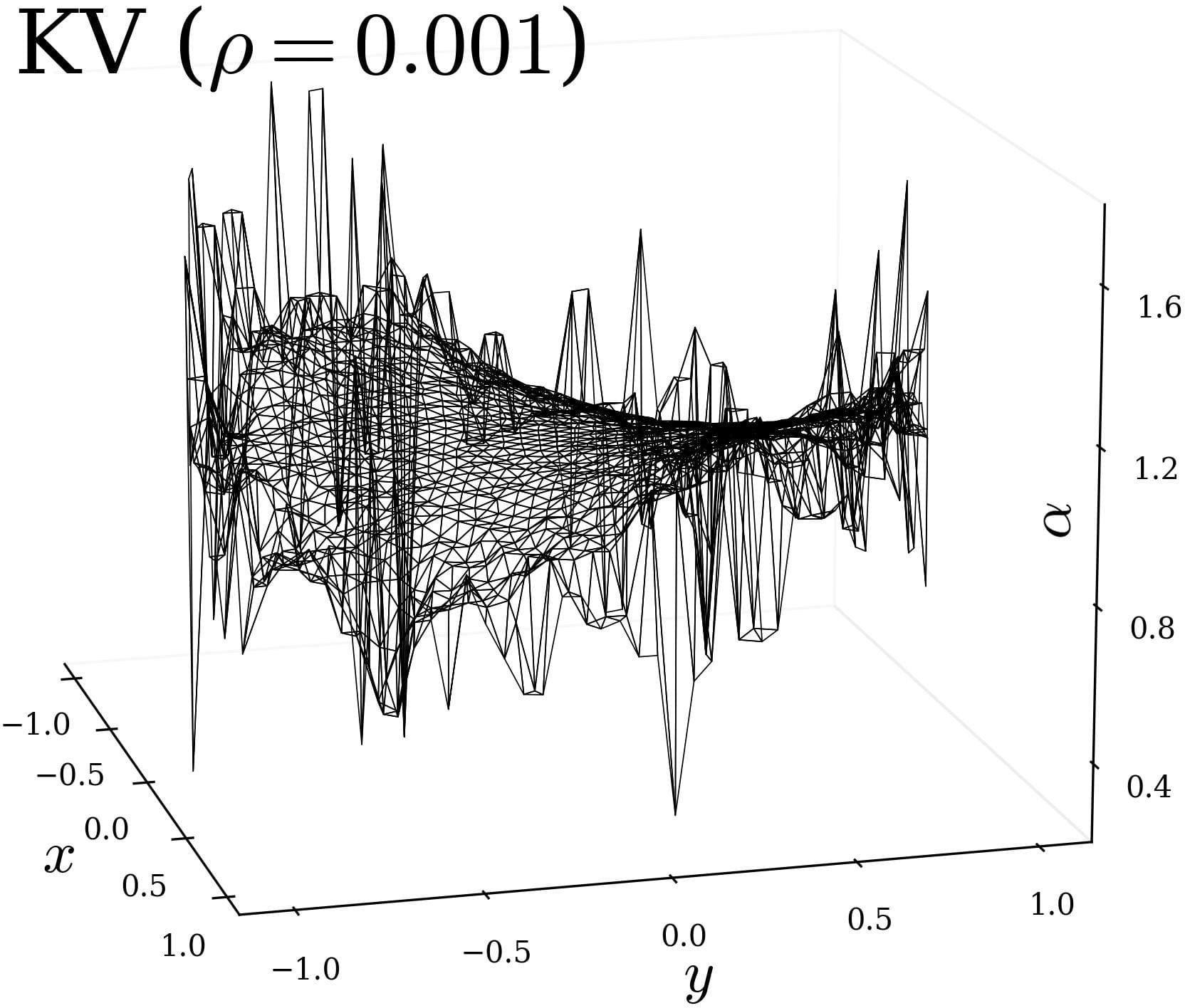}} \ 
\resizebox{0.225\textwidth}{!}{\includegraphics{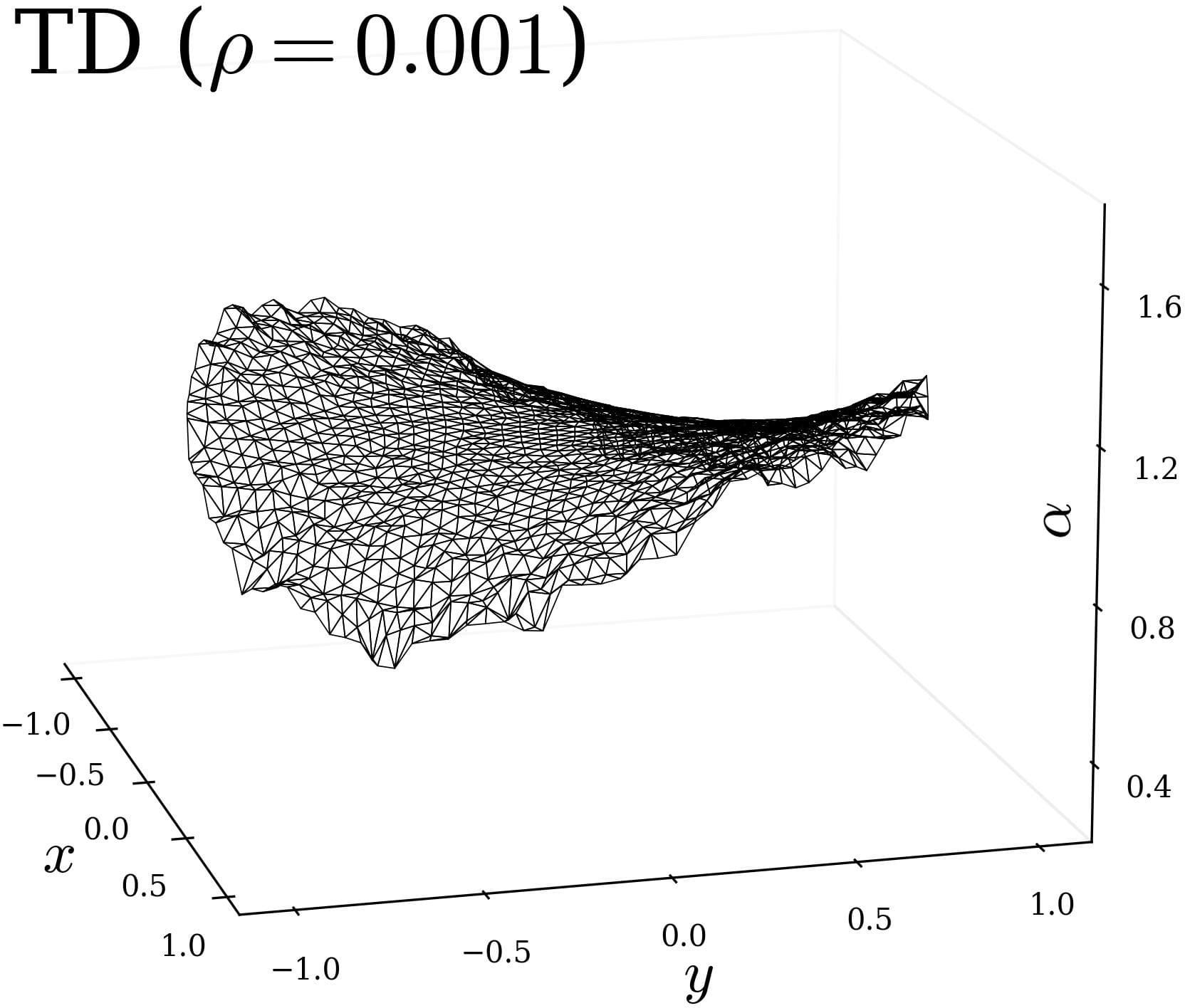}} \ 
\resizebox{0.225\textwidth}{!}{\includegraphics{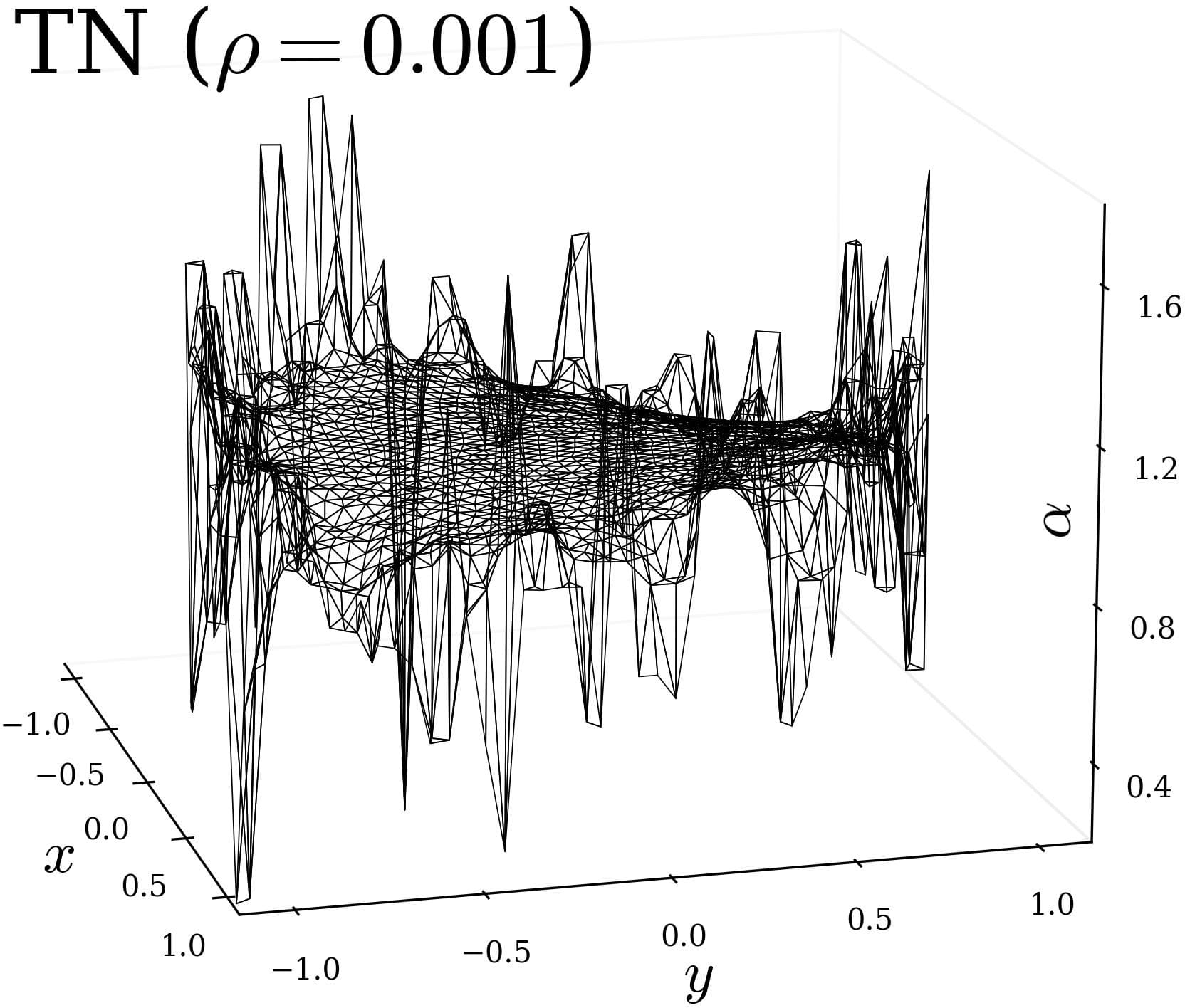}} \\[1em]
\resizebox{0.225\textwidth}{!}{\includegraphics{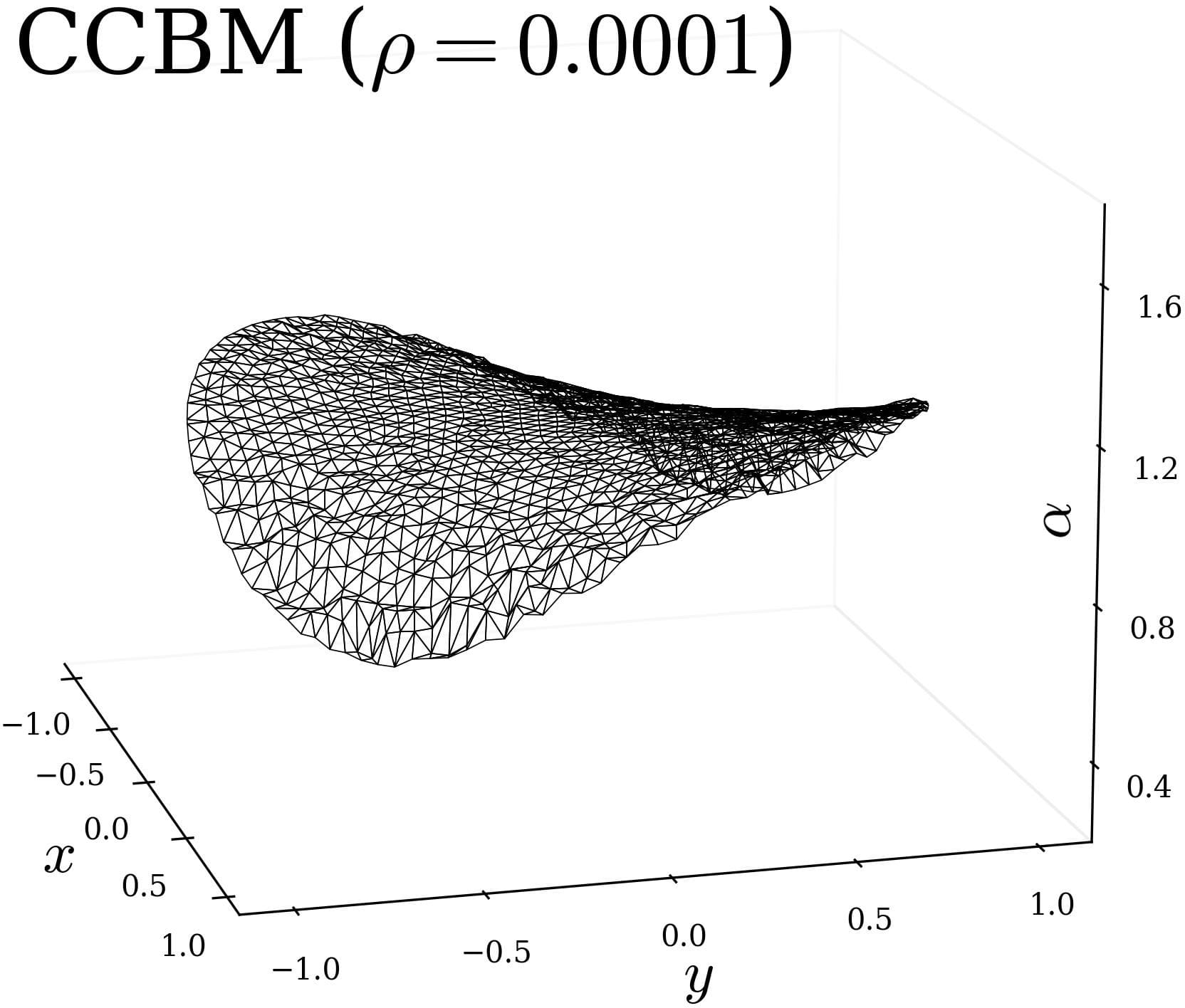}} \ 
\resizebox{0.225\textwidth}{!}{\includegraphics{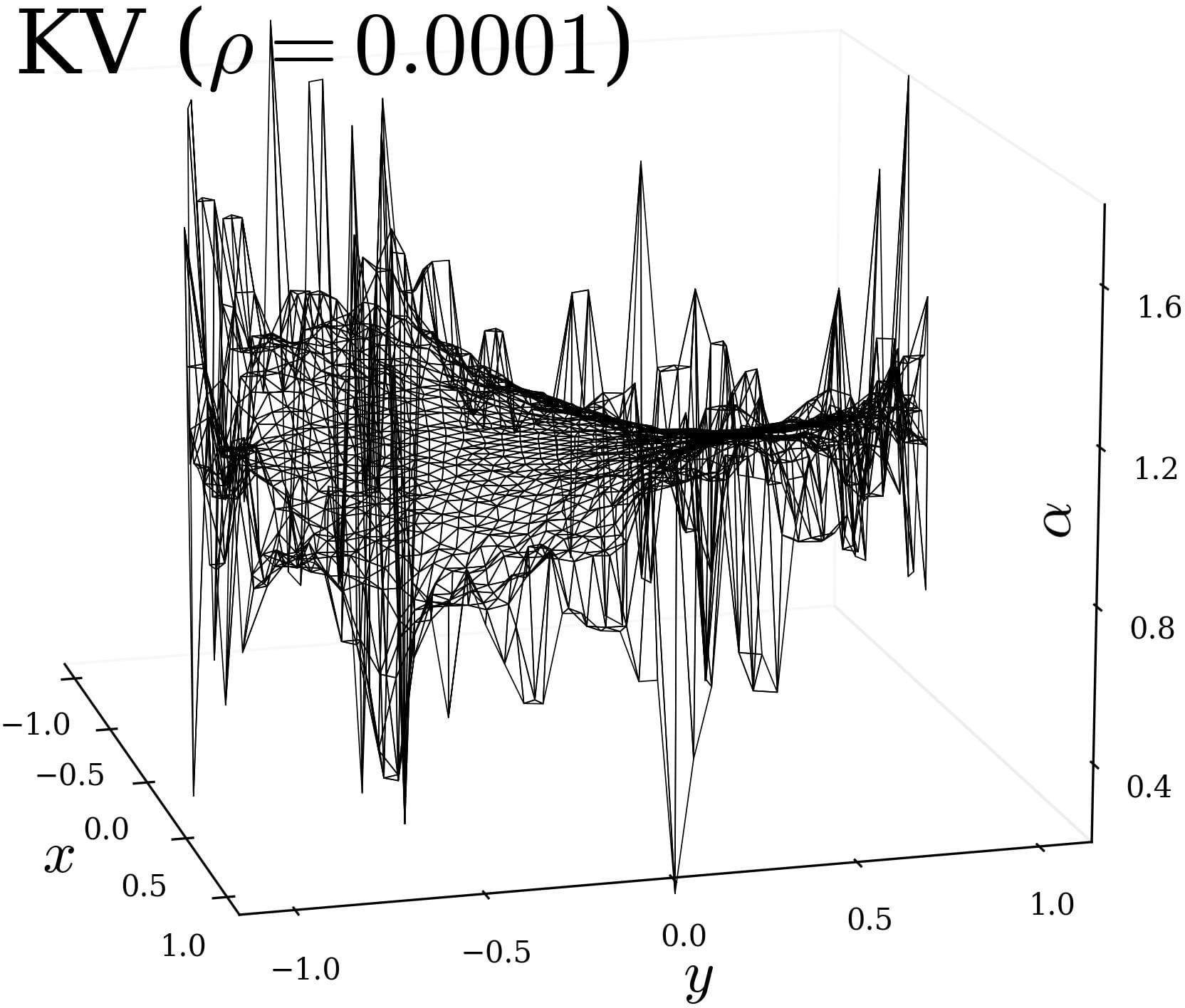}} \ 
\resizebox{0.225\textwidth}{!}{\includegraphics{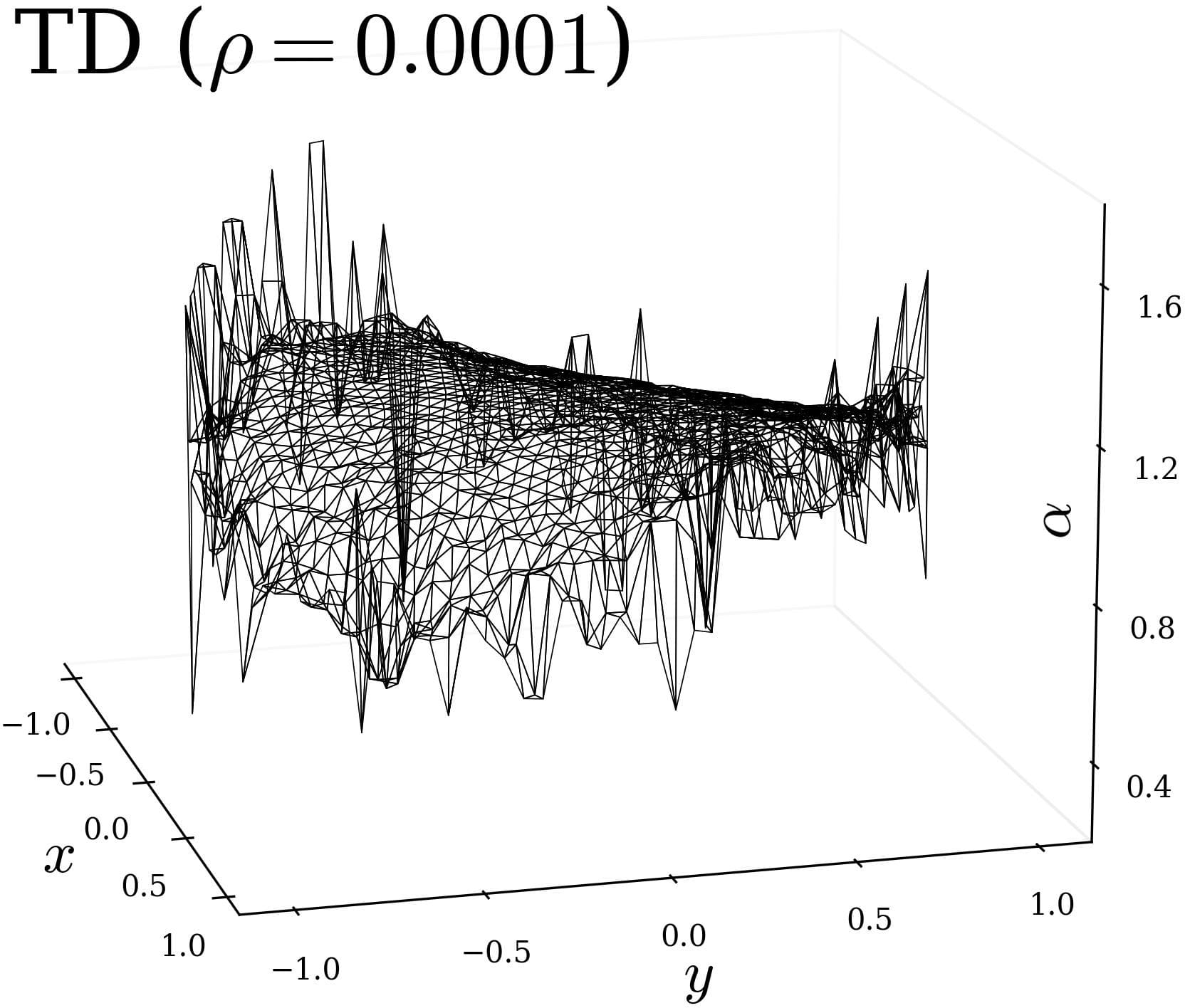}} \ 
\resizebox{0.225\textwidth}{!}{\includegraphics{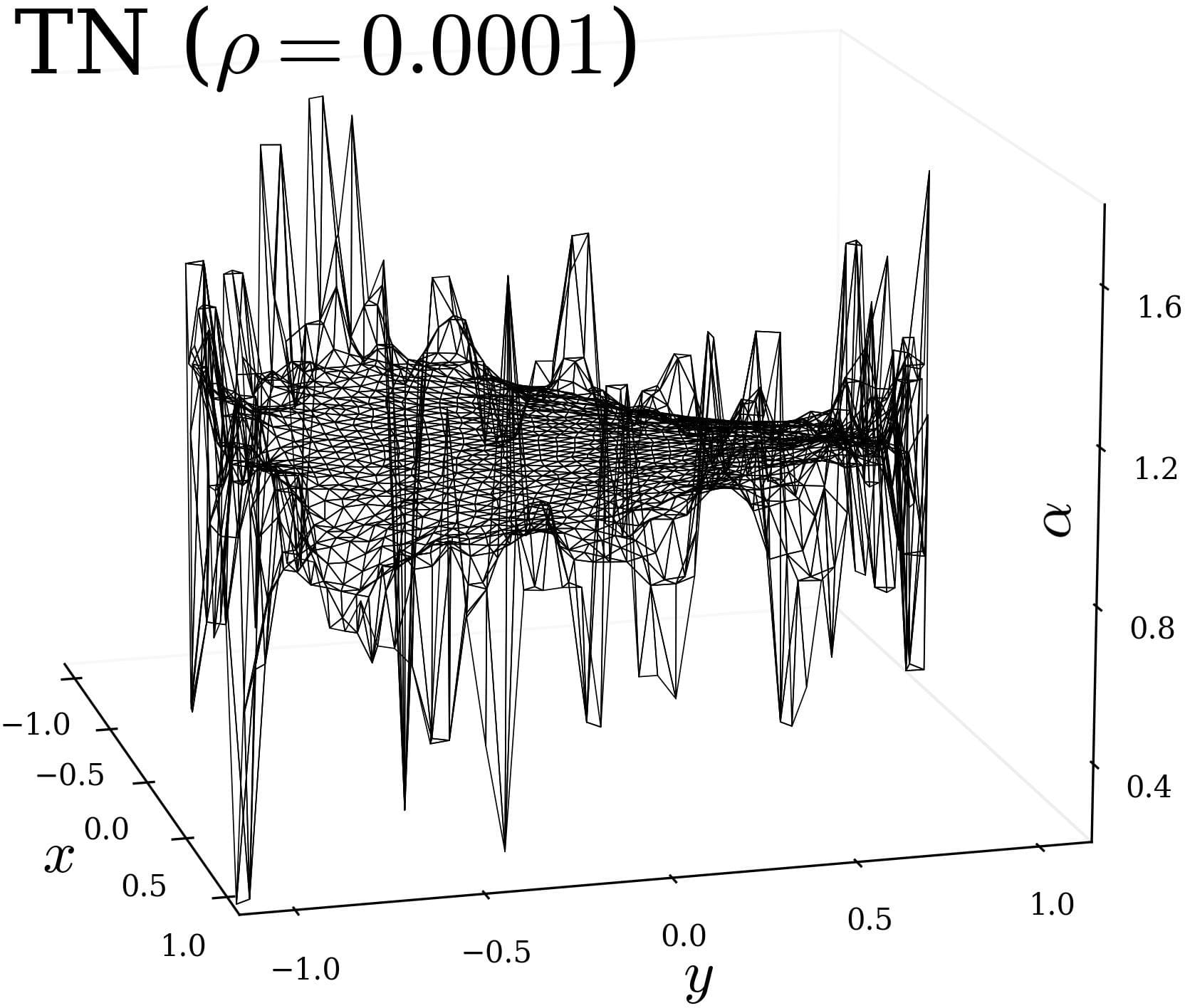}} \\[1em]
\resizebox{0.225\textwidth}{!}{\includegraphics{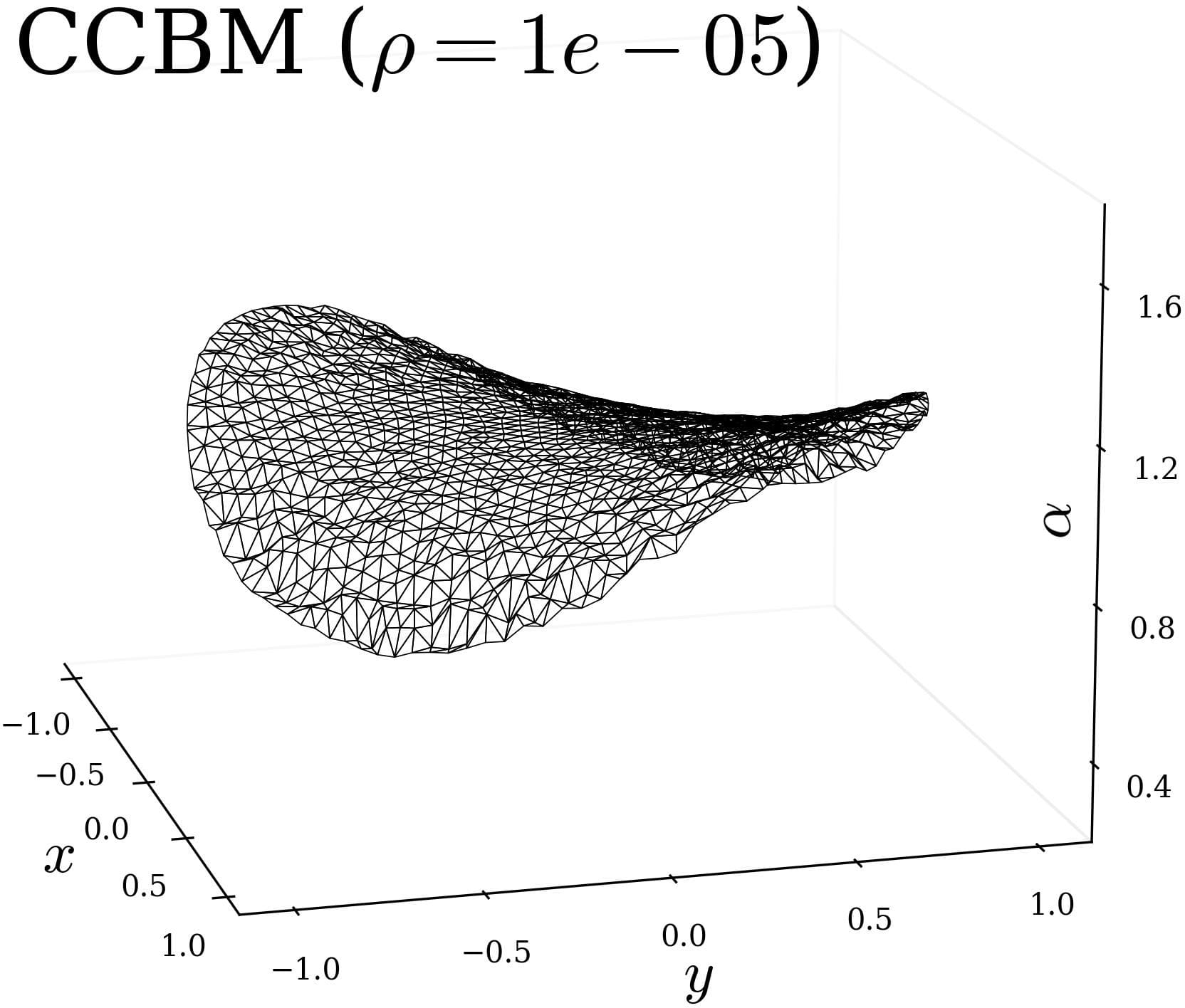}} \ 
\resizebox{0.225\textwidth}{!}{\includegraphics{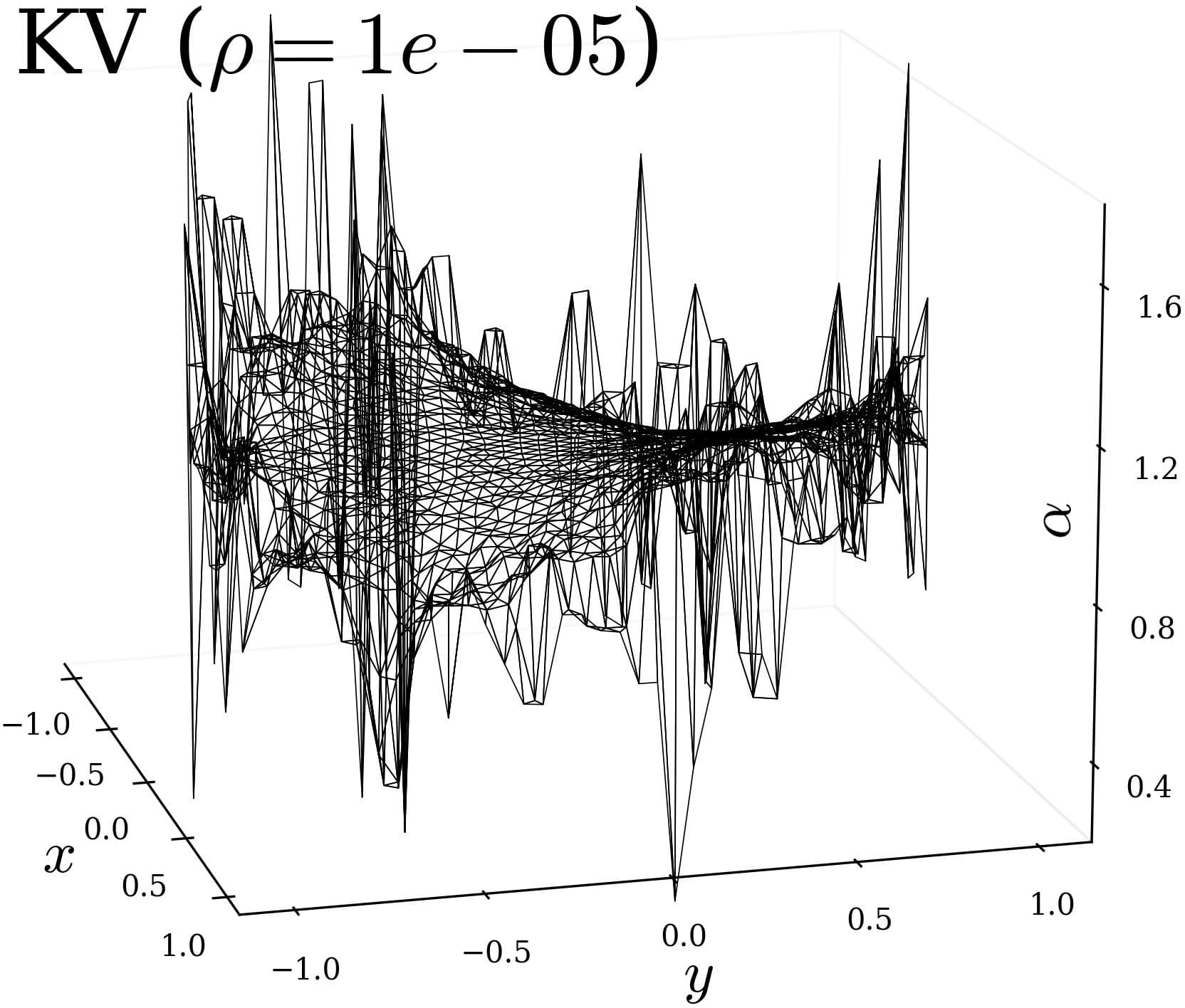}} \
\resizebox{0.225\textwidth}{!}{\includegraphics{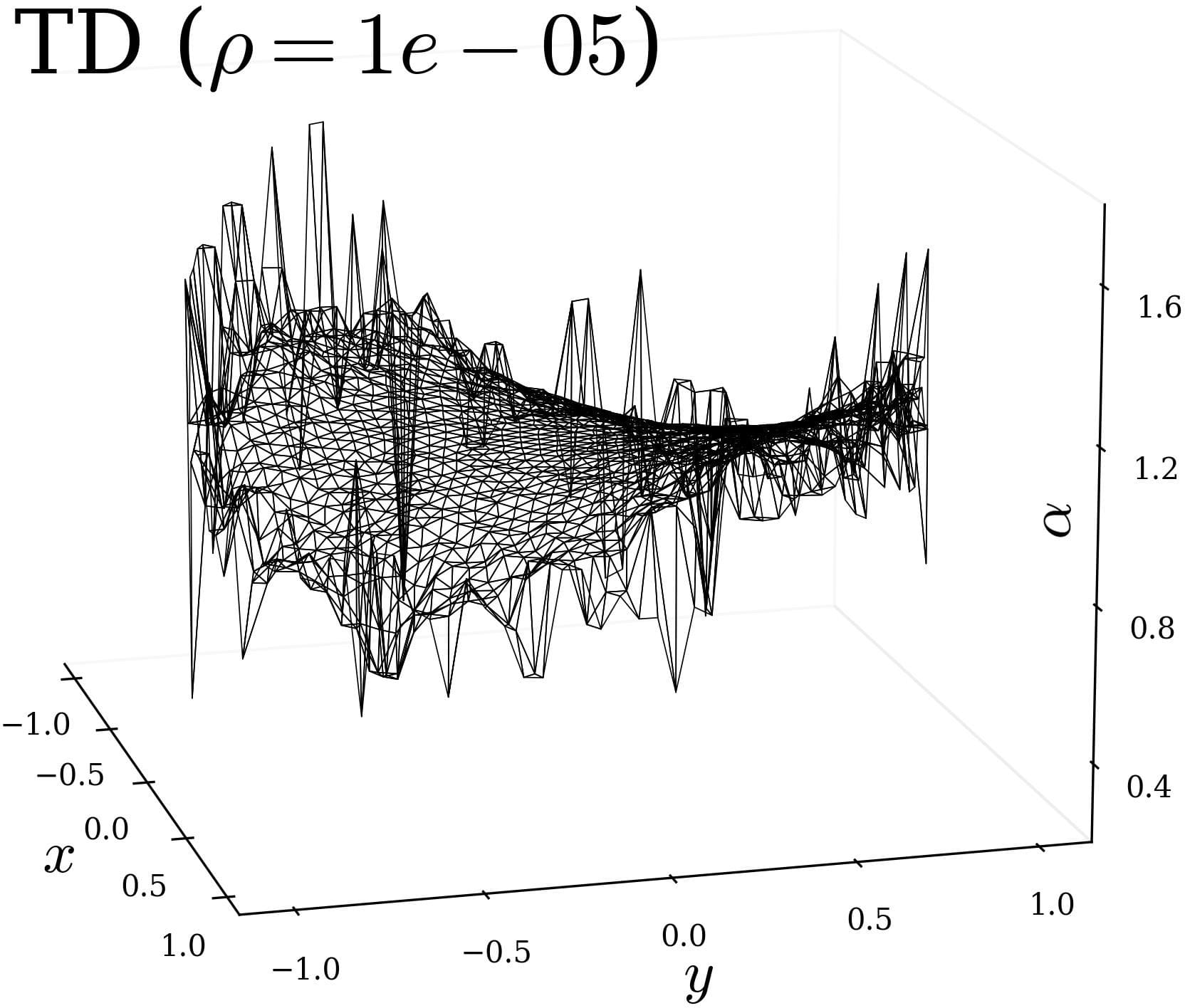}} \
\resizebox{0.225\textwidth}{!}{\includegraphics{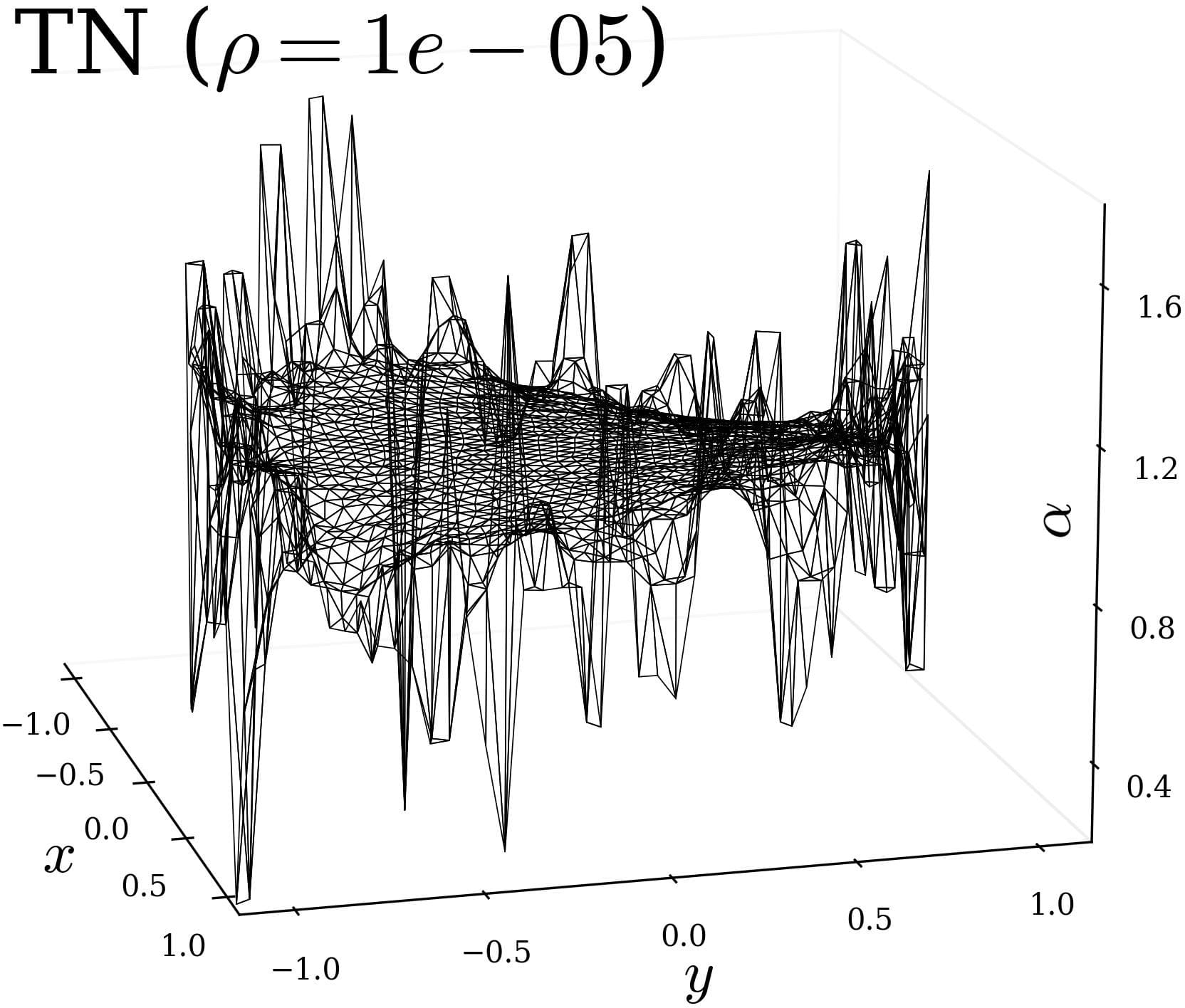}}
\caption{Influence of the Tikhonov parameter $\rho$ on the reconstruction when $\delta = 0.0003$, without gradient smoothing.}
\label{fig:effect_of_rho}
\end{figure}
%
%
%
\begin{figure}[htp!]
\centering
\resizebox{0.225\textwidth}{!}{\includegraphics{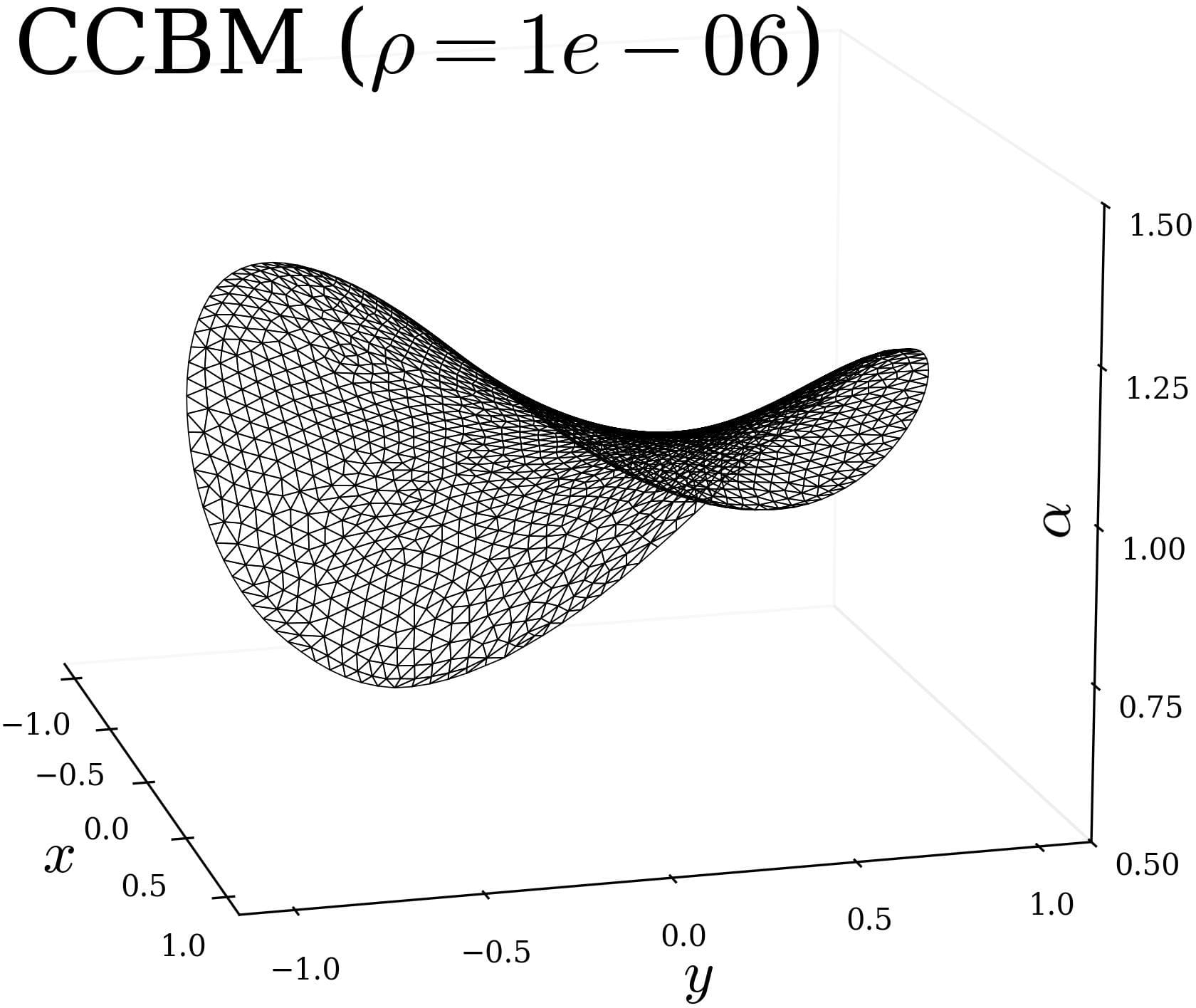}} \ 
\resizebox{0.225\textwidth}{!}{\includegraphics{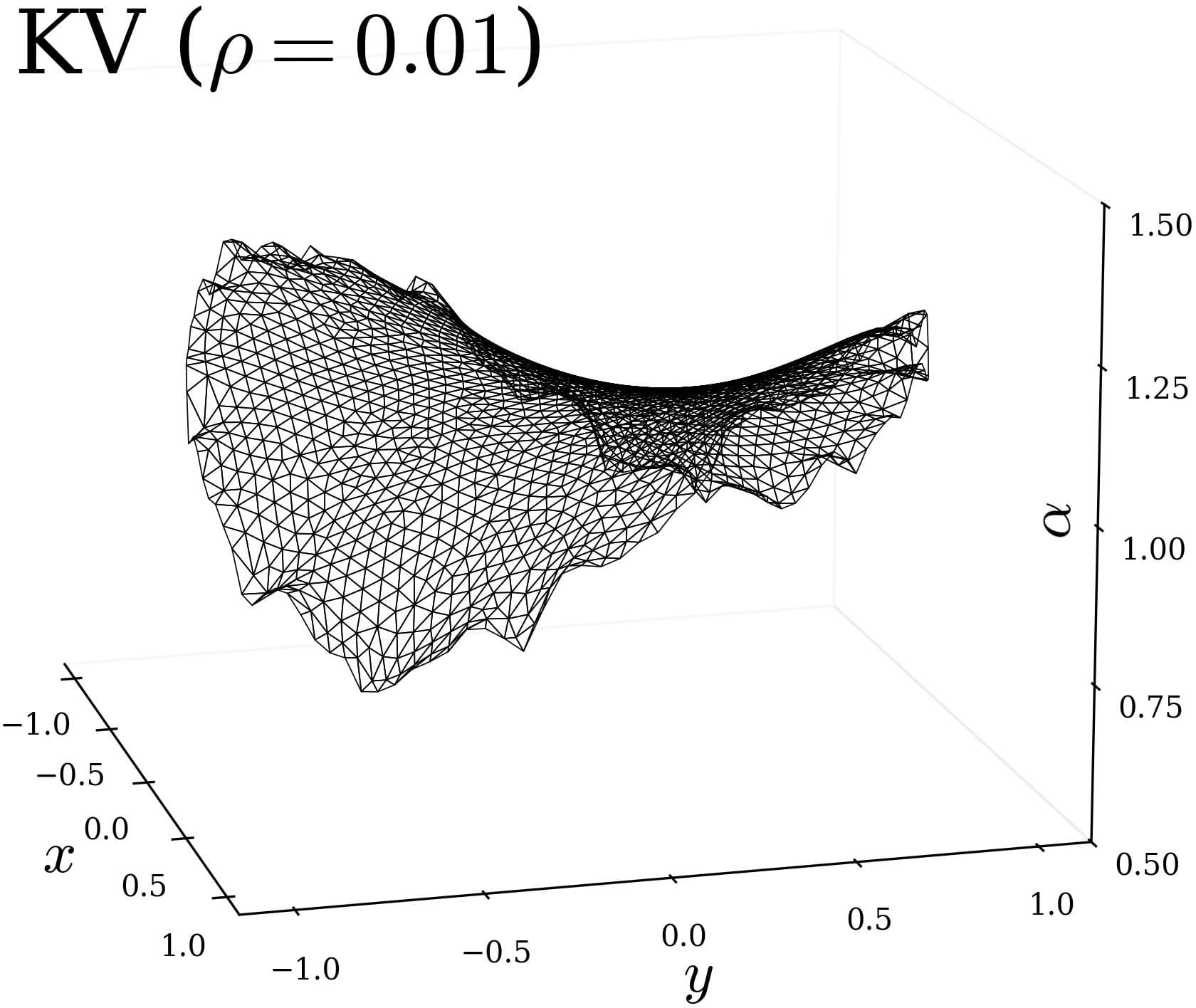}} \ 
\resizebox{0.225\textwidth}{!}{\includegraphics{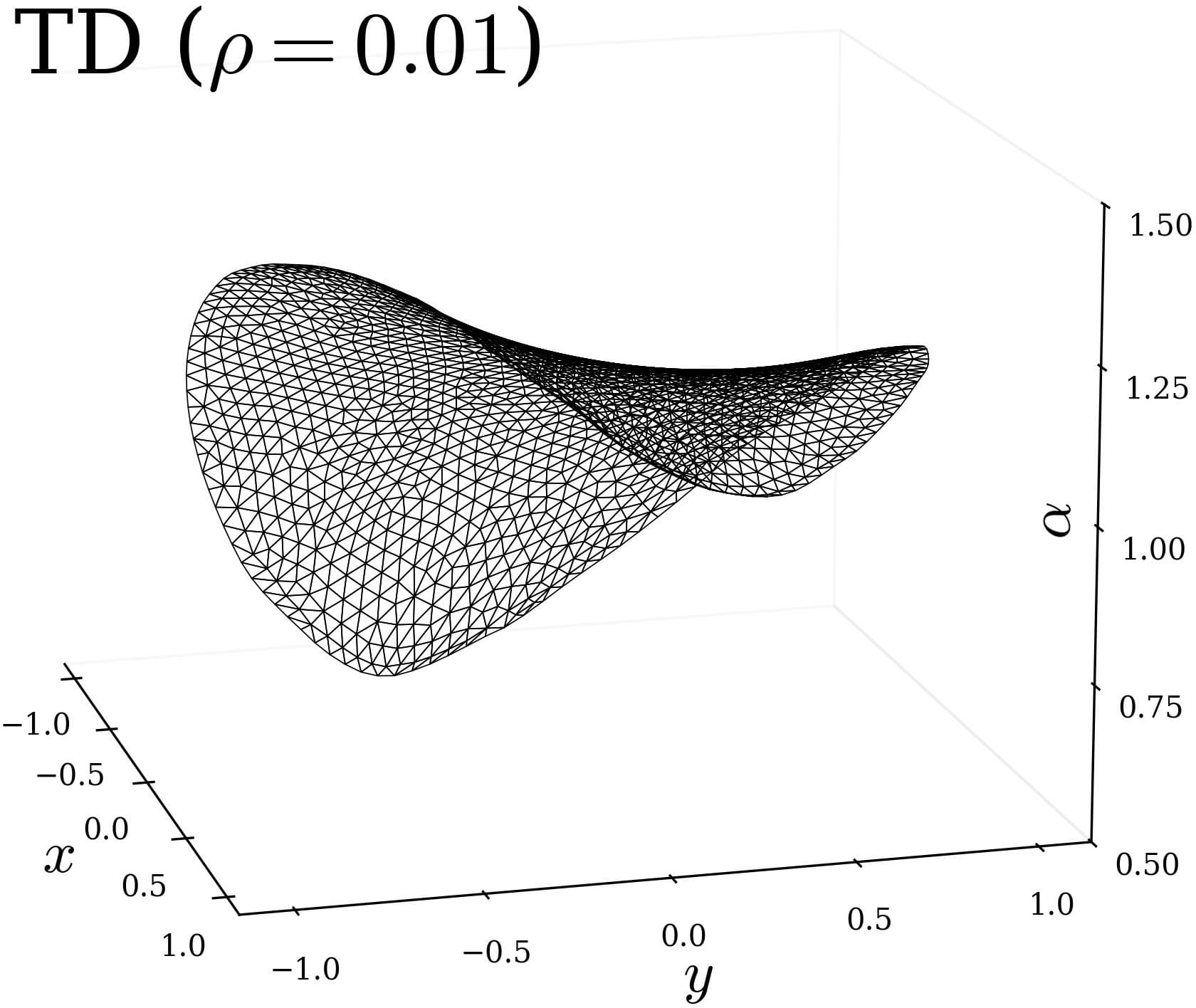}} \ 
\resizebox{0.225\textwidth}{!}{\includegraphics{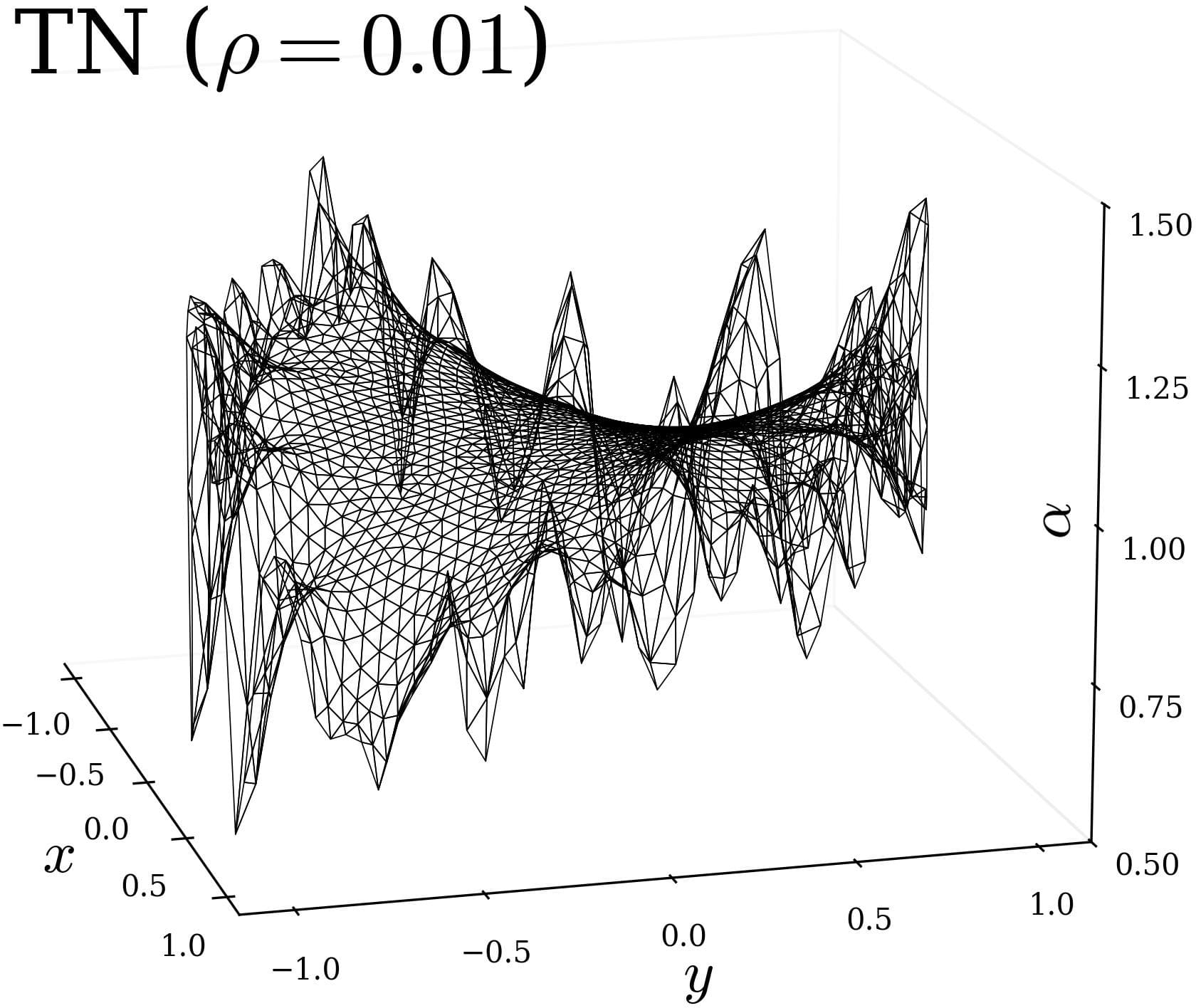}} \\[1em]
\resizebox{0.225\textwidth}{!}{\includegraphics{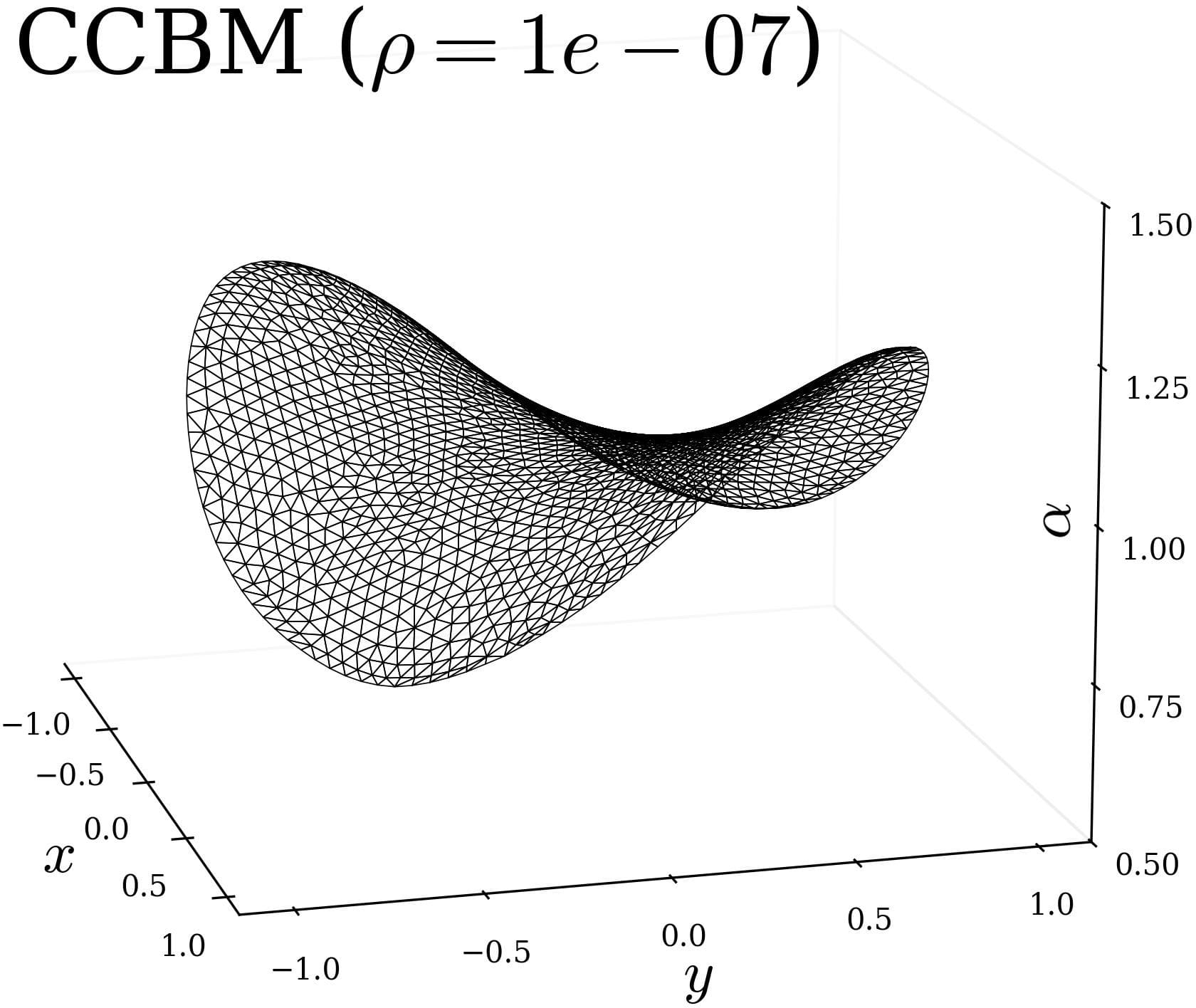}} \ 
\resizebox{0.225\textwidth}{!}{\includegraphics{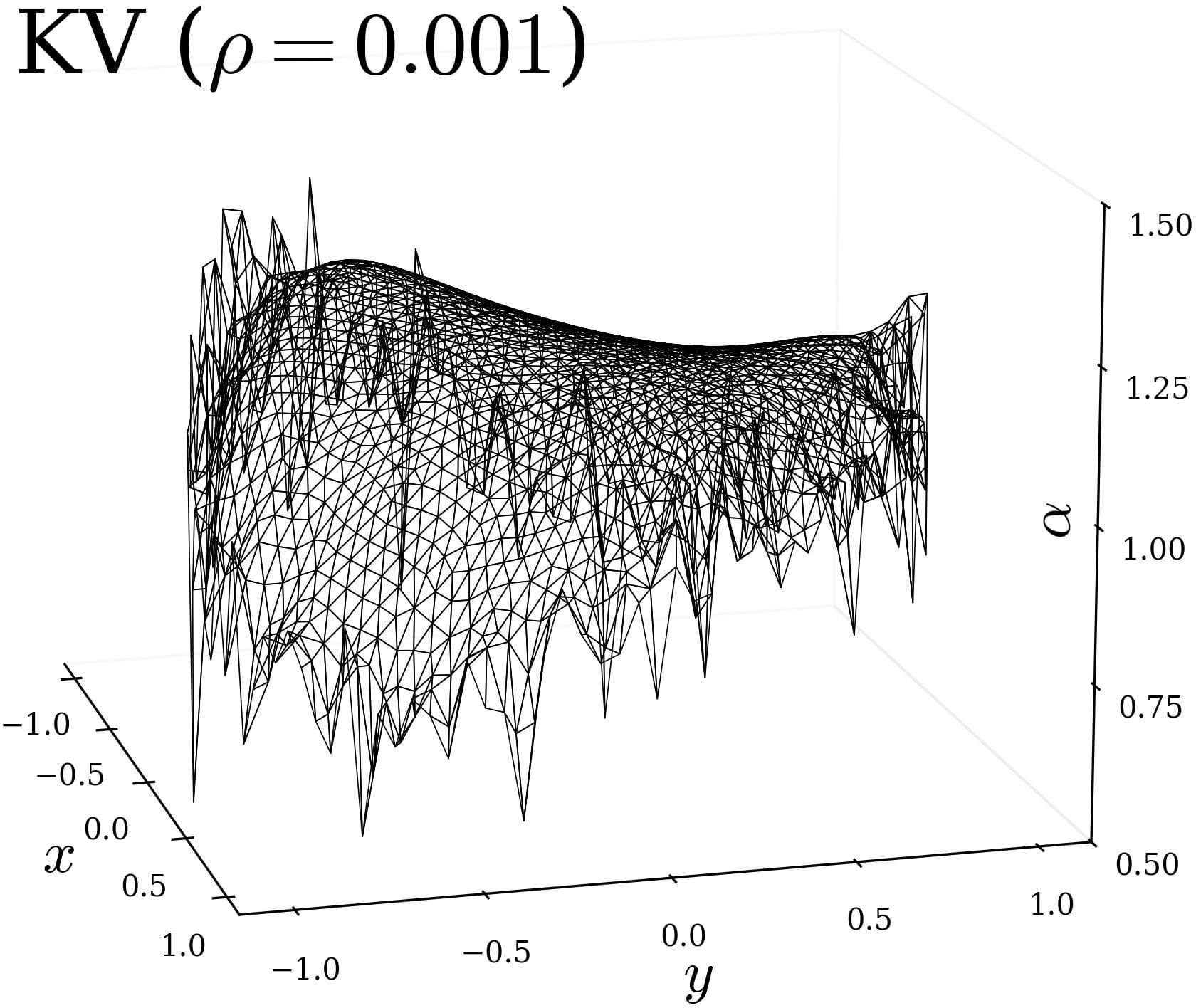}} \ 
\resizebox{0.225\textwidth}{!}{\includegraphics{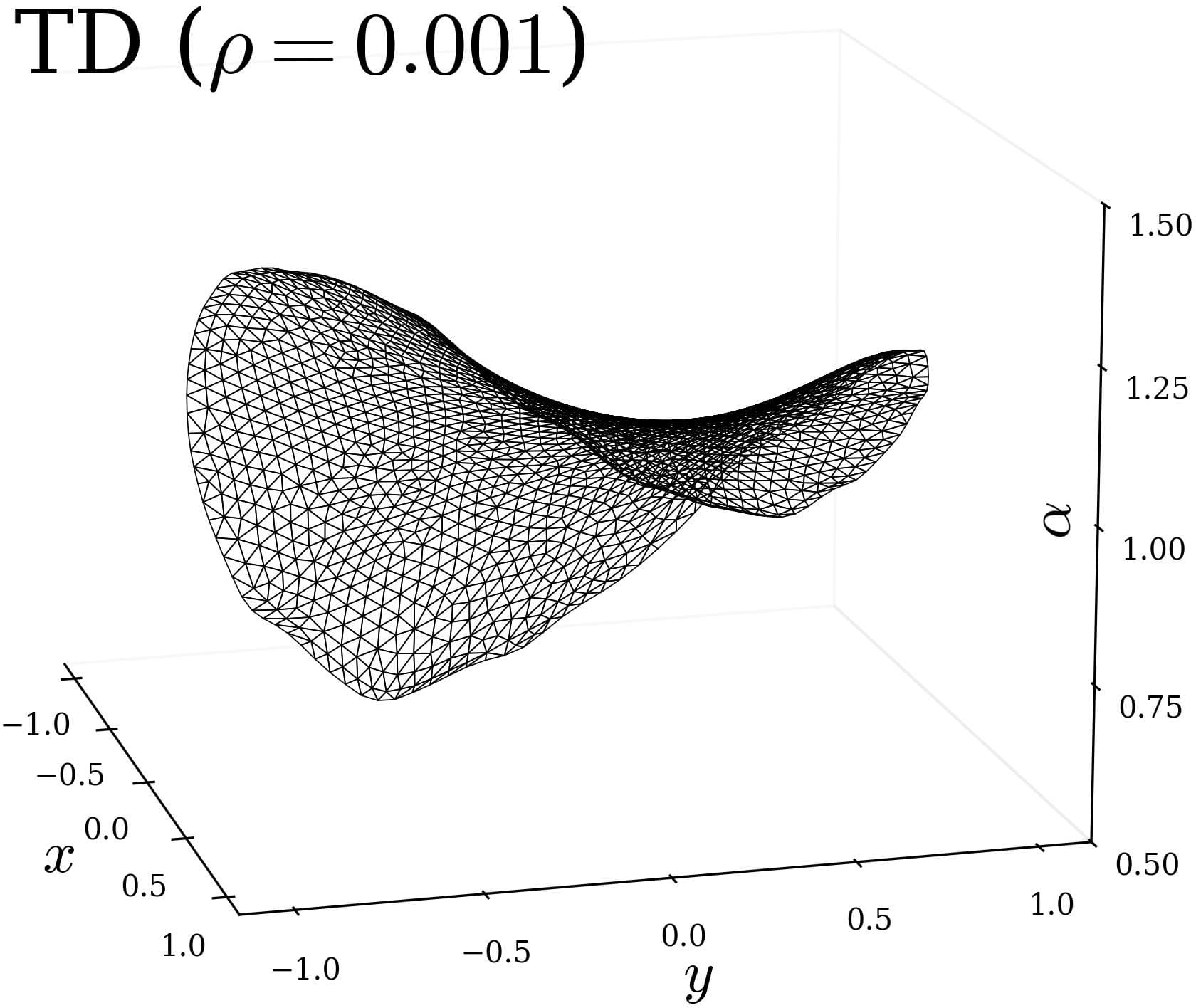}} \ 
\resizebox{0.225\textwidth}{!}{\includegraphics{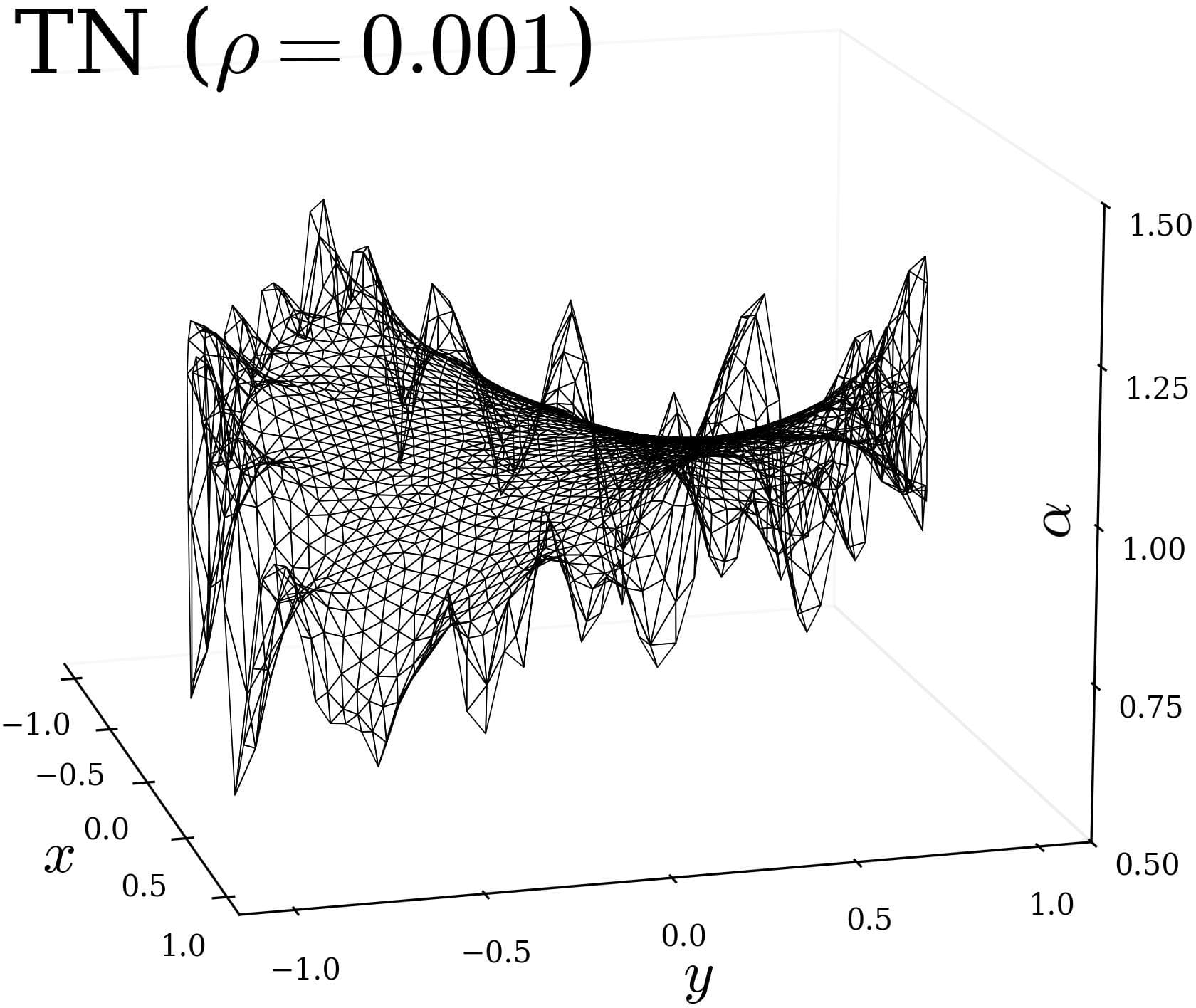}} \\[1em]
\resizebox{0.225\textwidth}{!}{\includegraphics{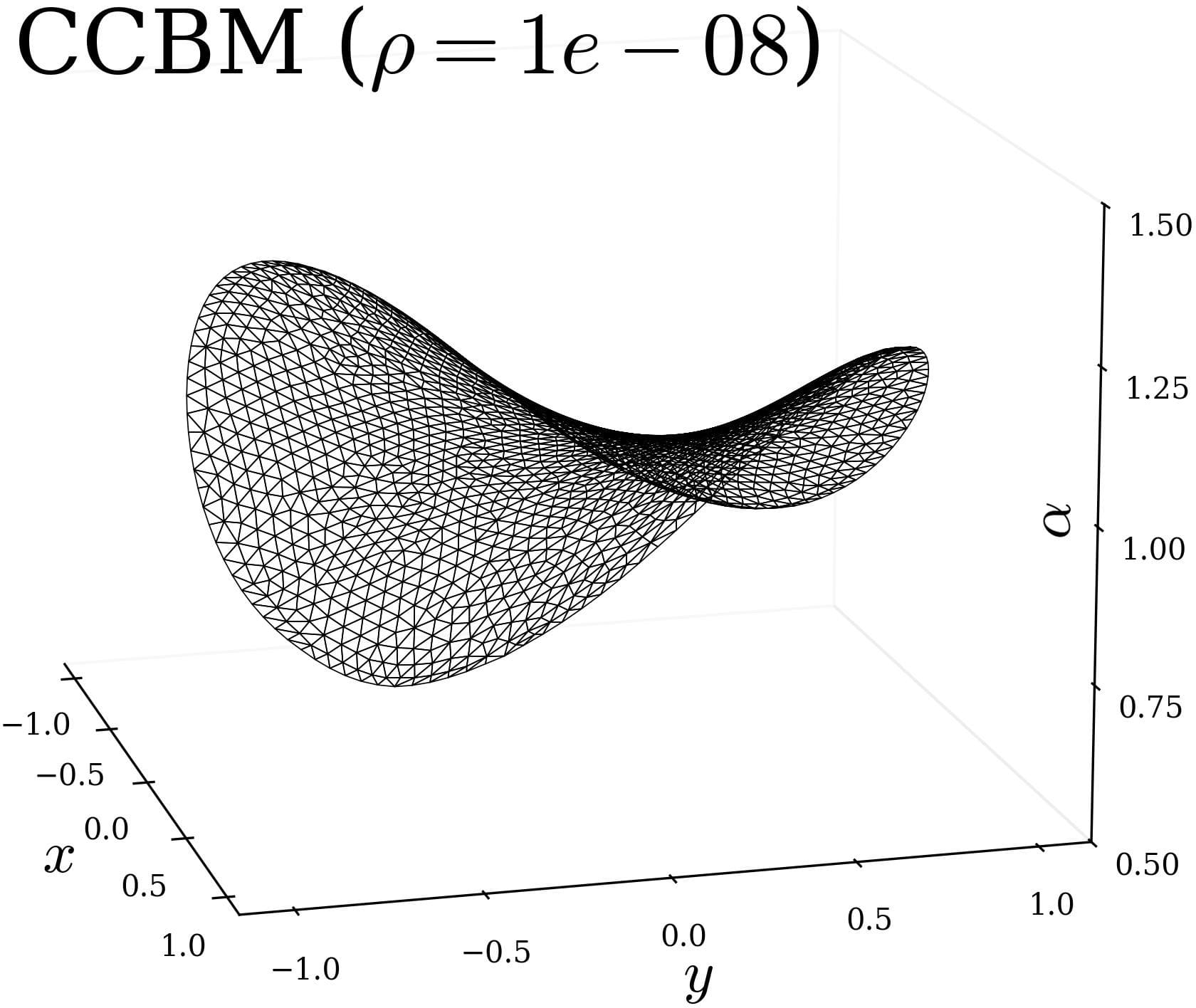}} \ 
\resizebox{0.225\textwidth}{!}{\includegraphics{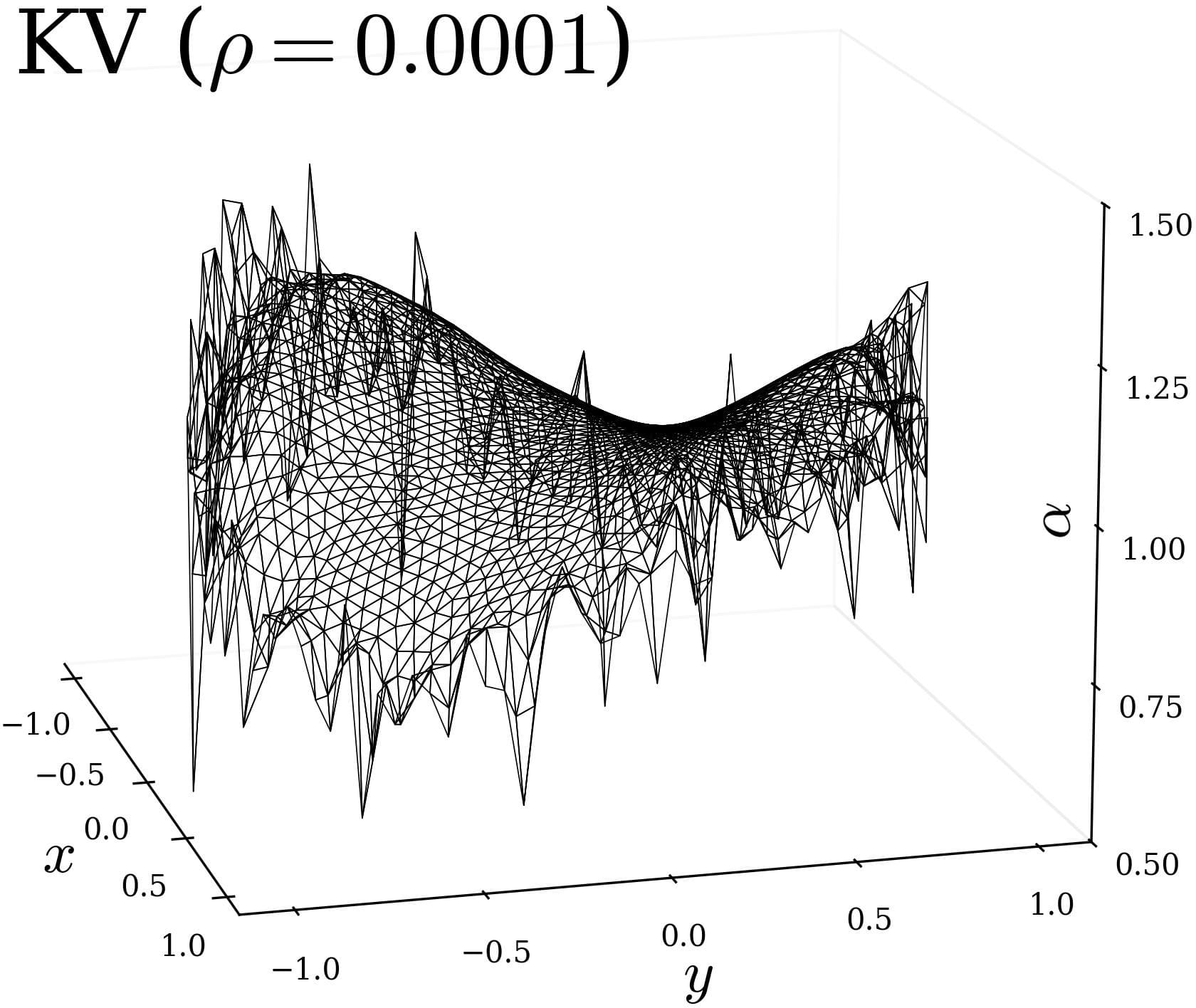}} \ 
\resizebox{0.225\textwidth}{!}{\includegraphics{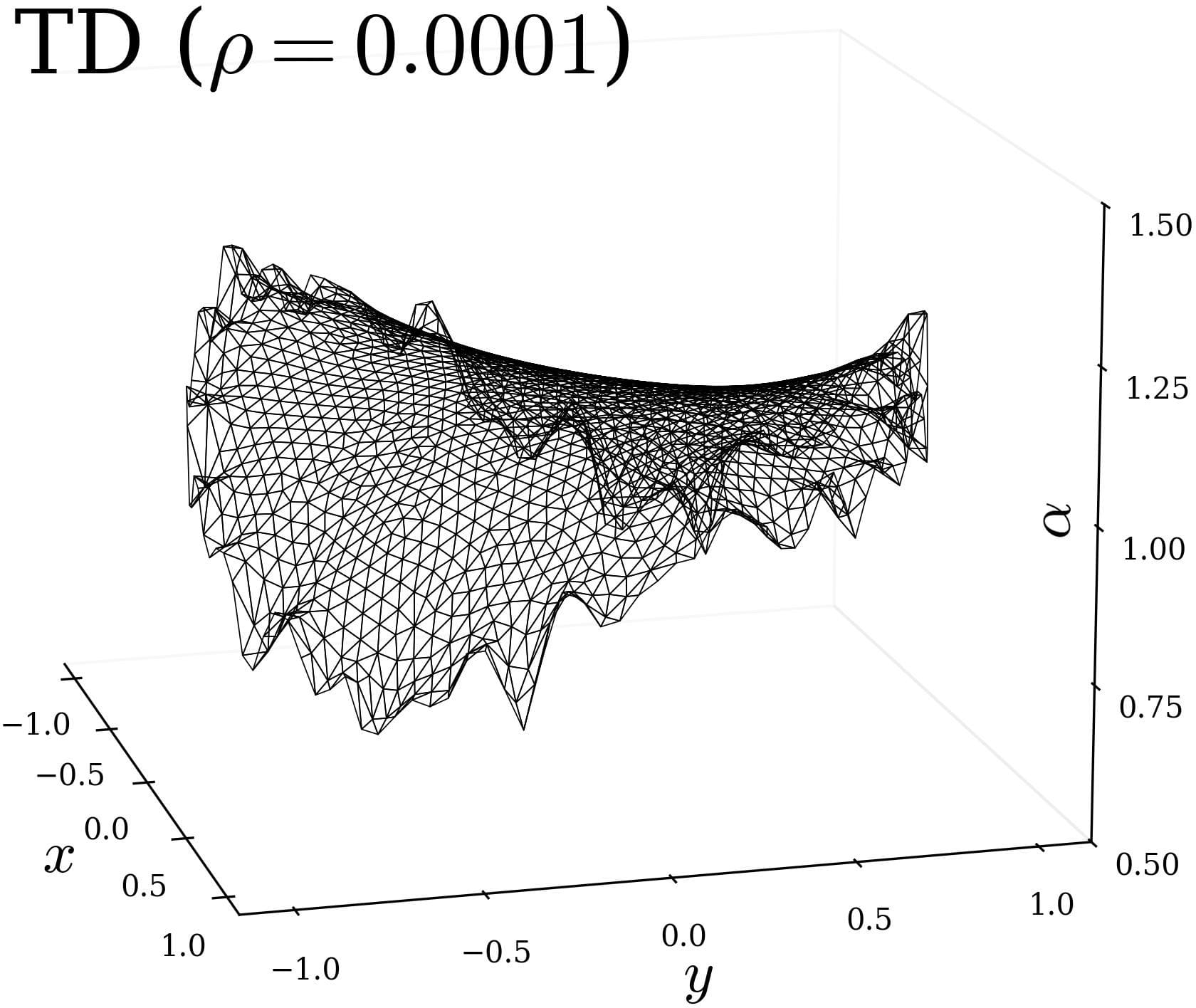}} \ 
\resizebox{0.225\textwidth}{!}{\includegraphics{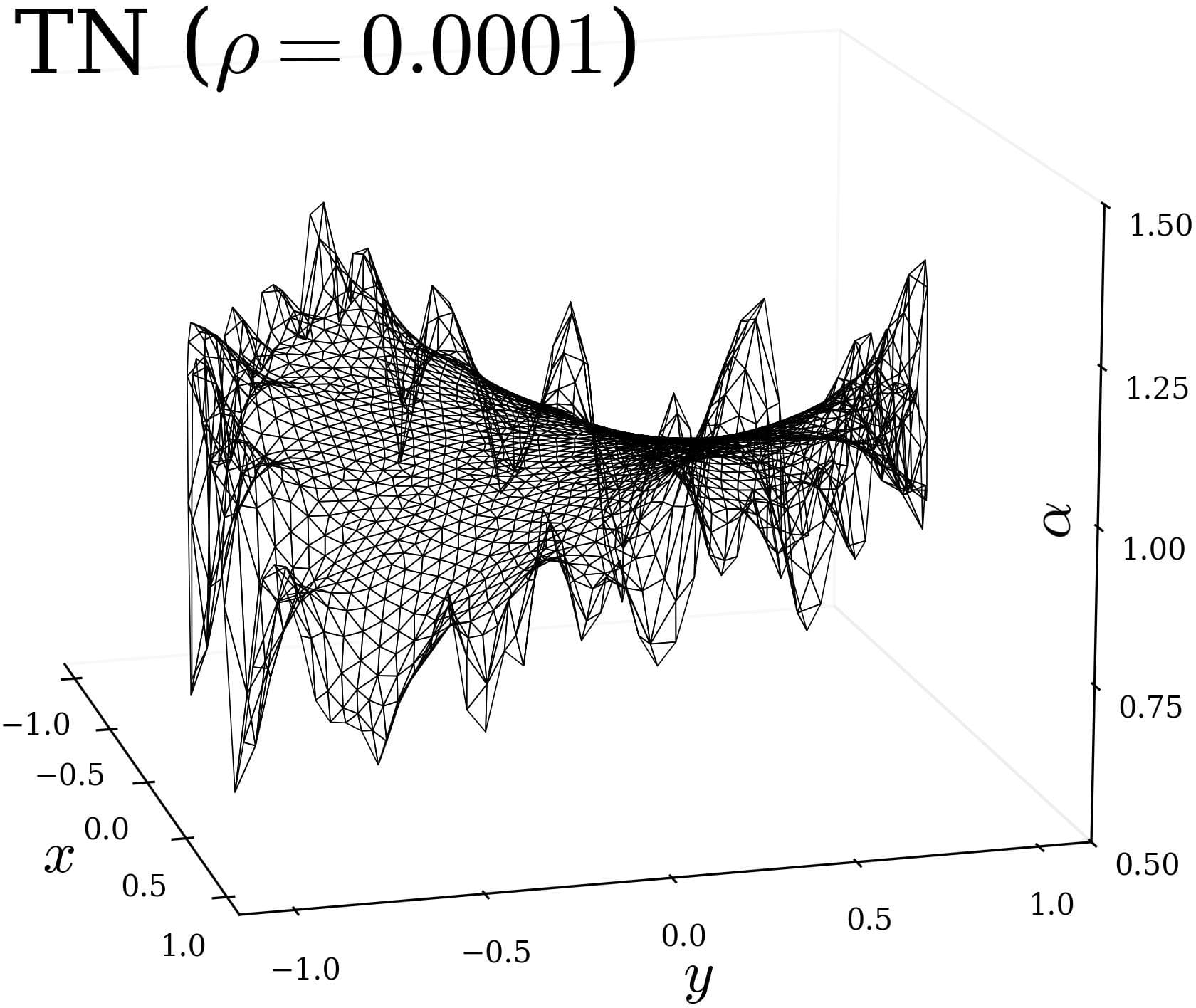}} \\[1em]
\resizebox{0.225\textwidth}{!}{\includegraphics{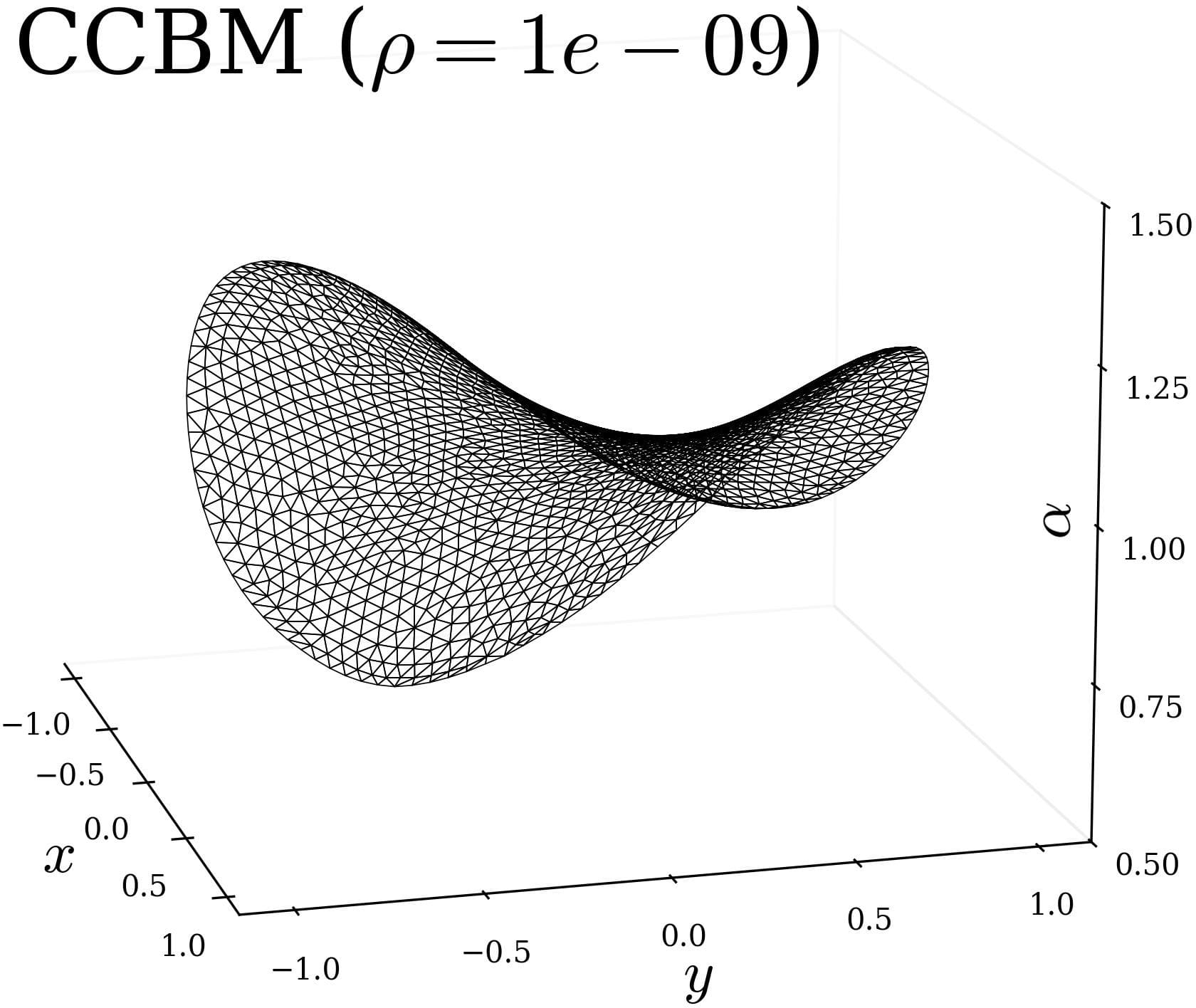}} \ 
\resizebox{0.225\textwidth}{!}{\includegraphics{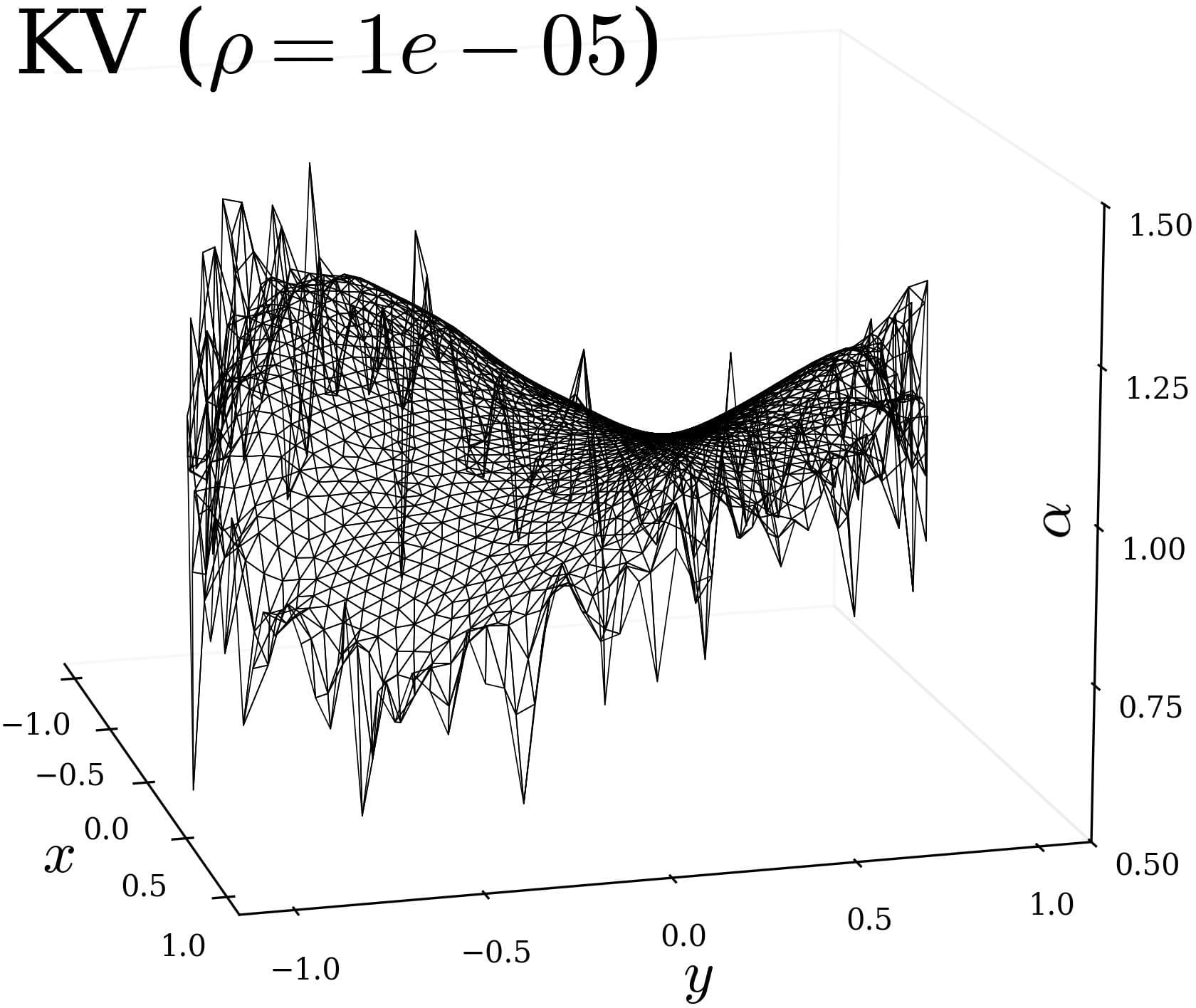}} \
\resizebox{0.225\textwidth}{!}{\includegraphics{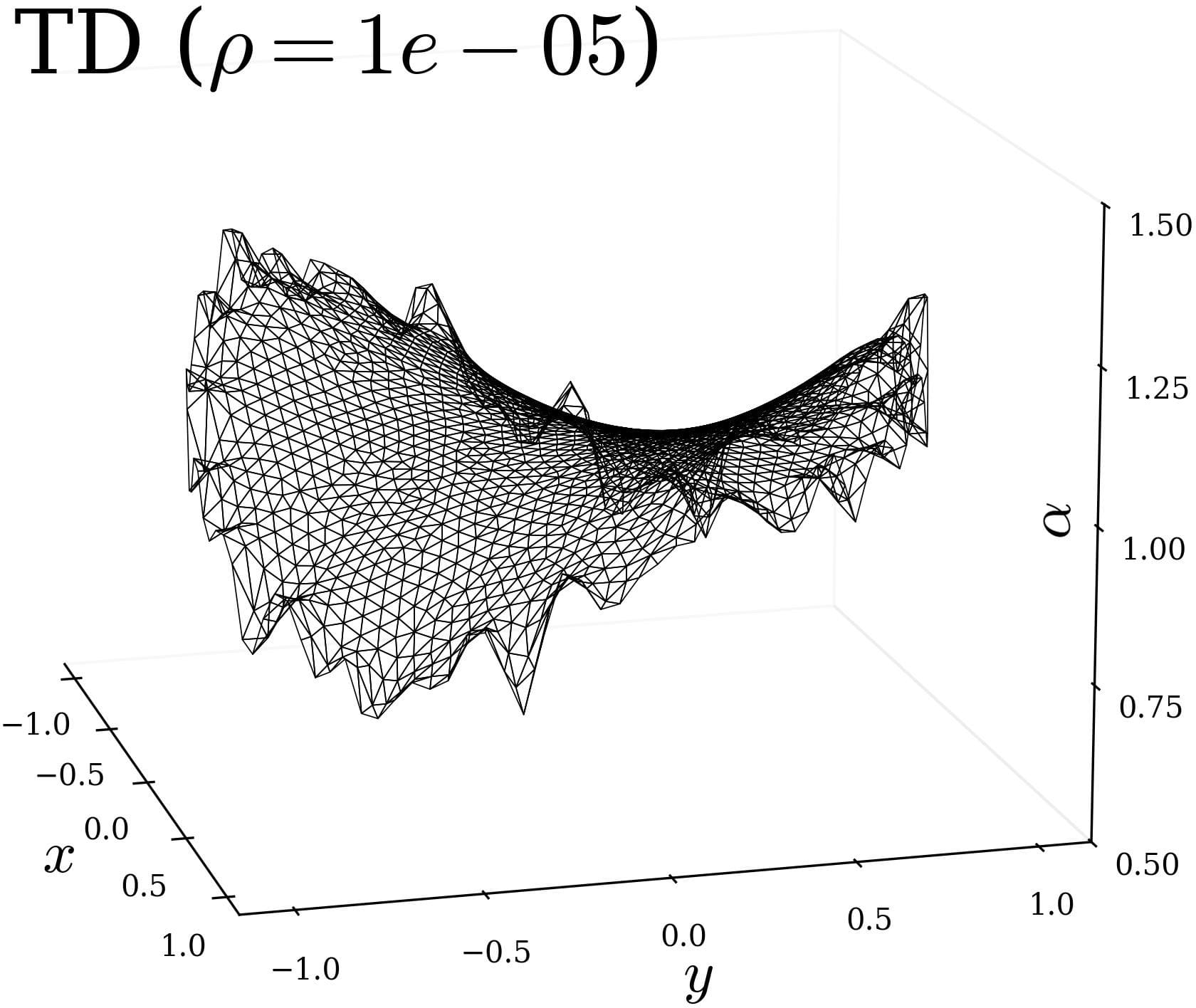}} \
\resizebox{0.225\textwidth}{!}{\includegraphics{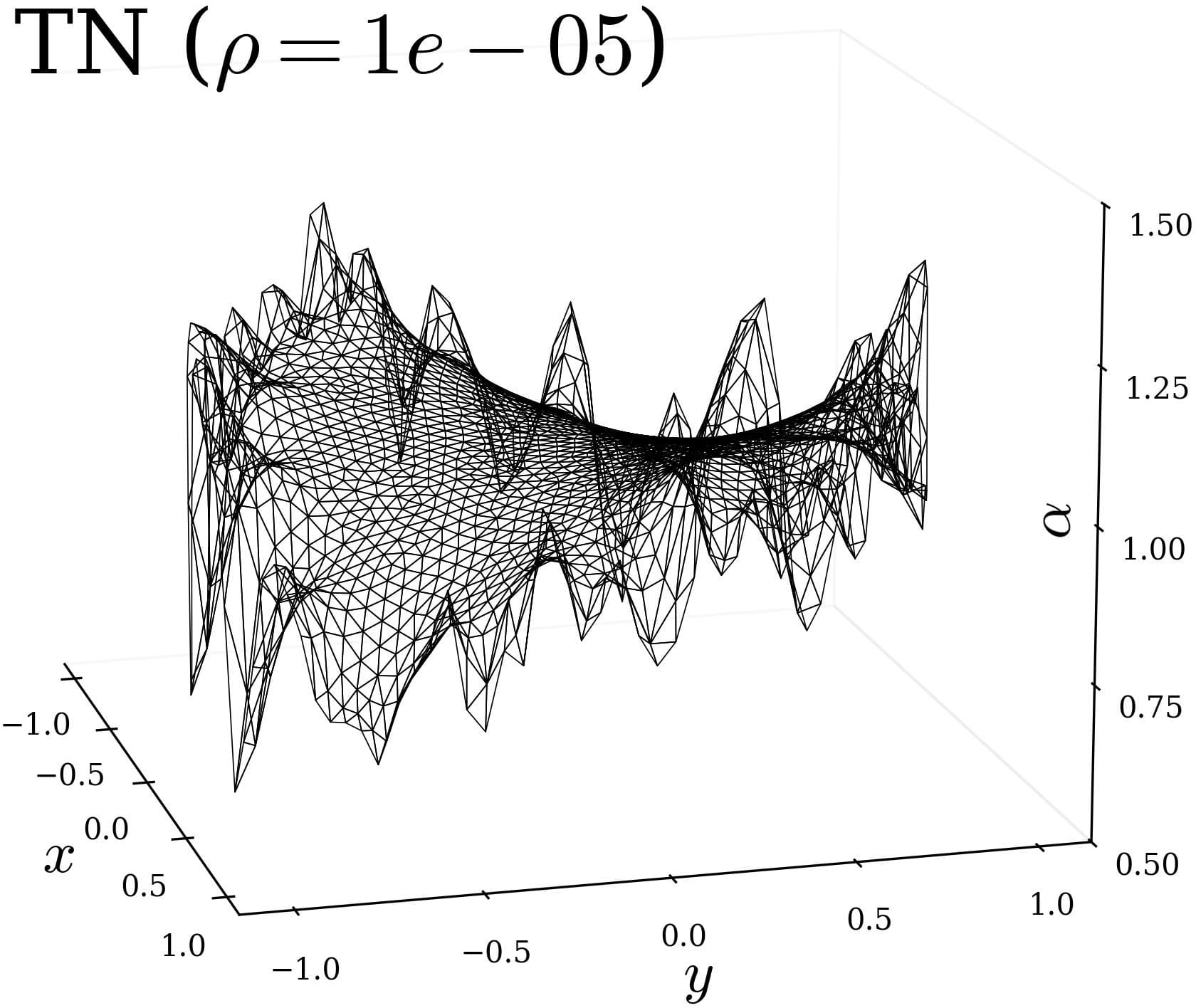}}
\caption{Influence of the Tikhonov parameter $\rho$ on the reconstruction when $\delta = 0.0003$, with gradient smoothing ($\mu = 0.01$).}
\label{fig:effect_of_rho_with_regularization}
\end{figure}
%
%
\begin{figure}[htp!]
\centering
\resizebox{0.225\textwidth}{!}{\includegraphics{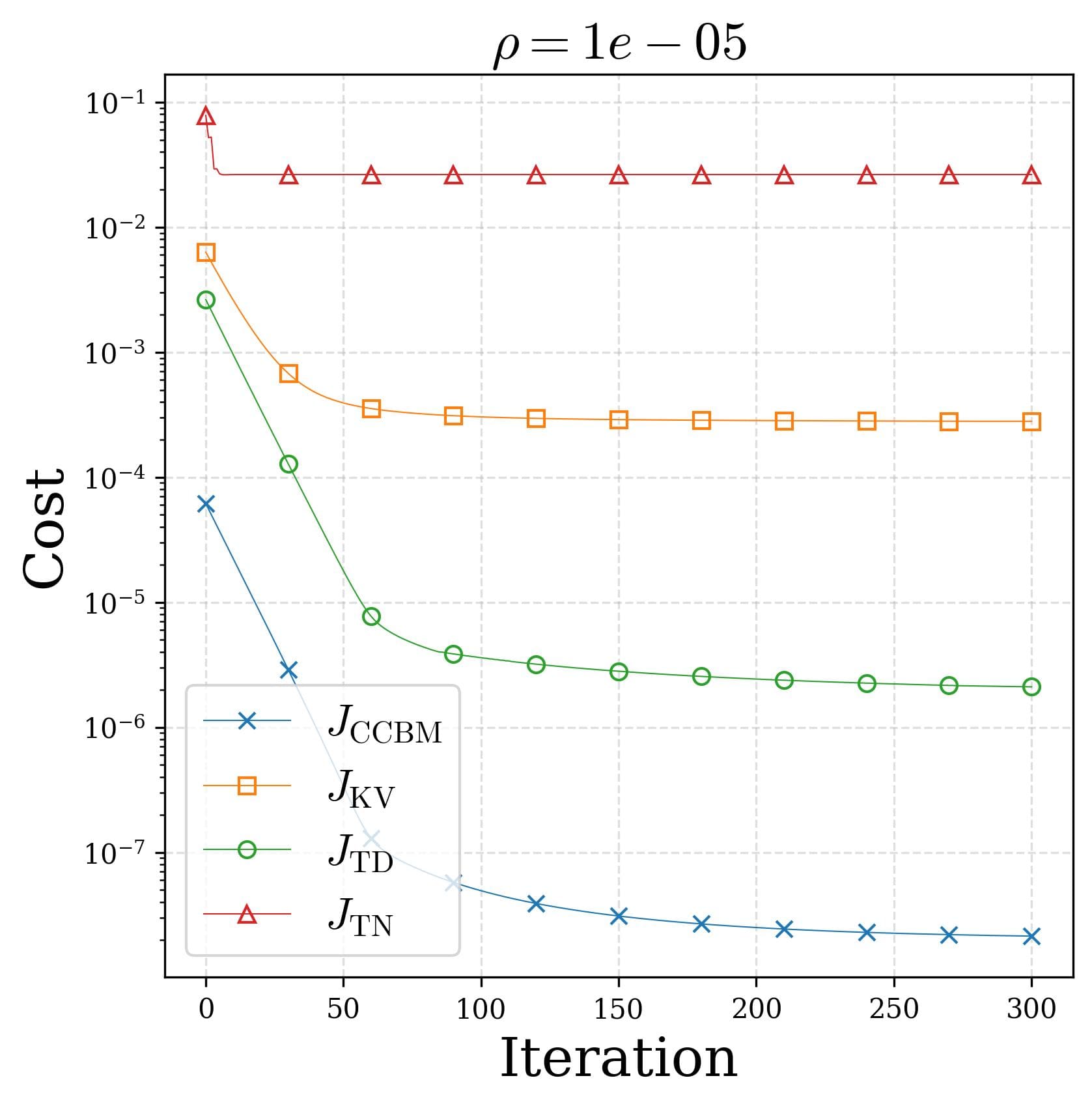}} 
\resizebox{0.225\textwidth}{!}{\includegraphics{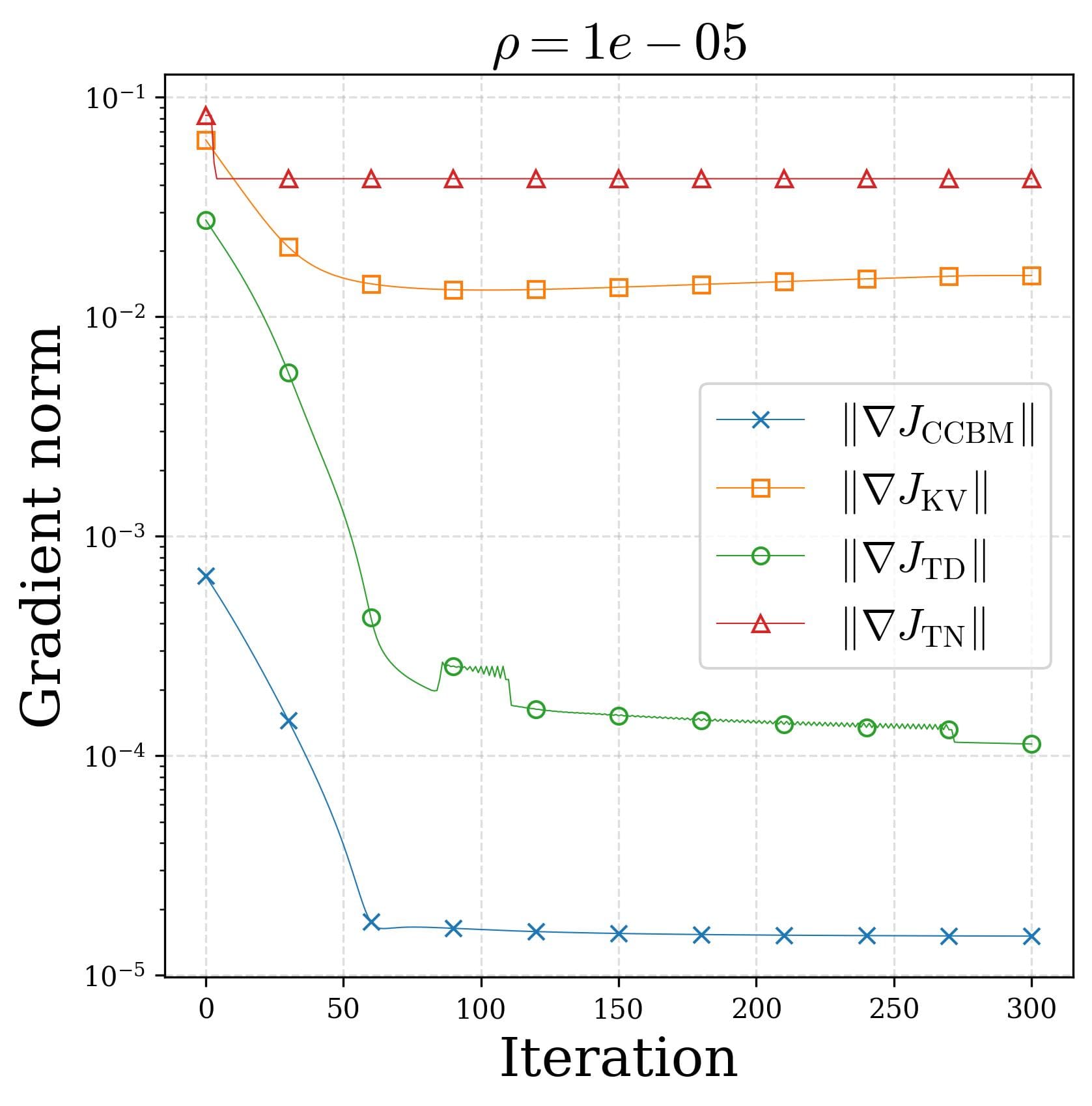}} \
\resizebox{0.225\textwidth}{!}{\includegraphics{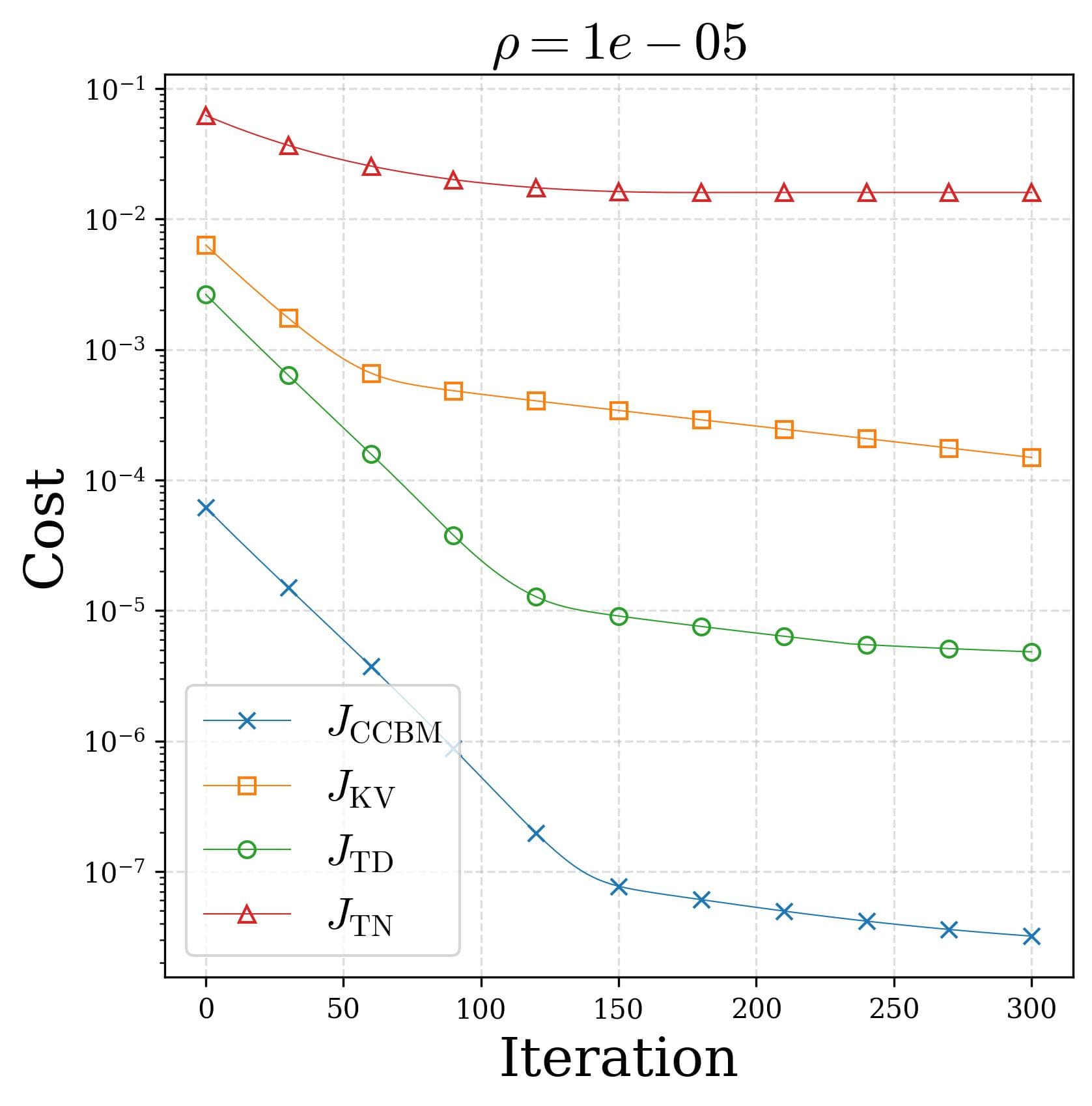}} 
\resizebox{0.225\textwidth}{!}{\includegraphics{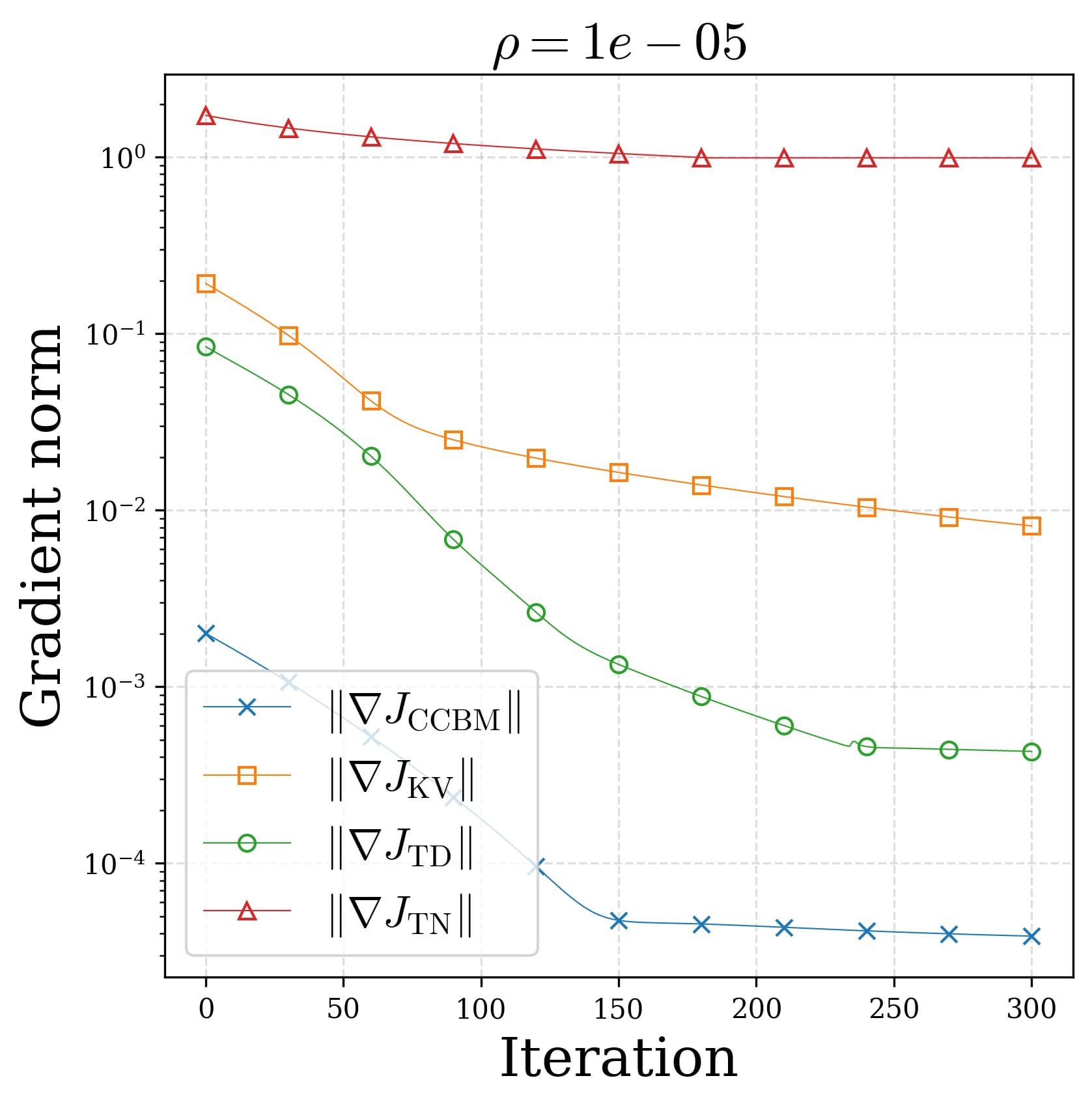}}
\caption{Histories of the cost functional and gradient norm corresponding to Figure~\ref{fig:effect_of_rho_with_regularization}. Left: without gradient smoothing; right: with gradient smoothing ($\mu = 0.01$).}
\label{fig:effect_of_rho_cost_and_gradient}
\end{figure}
\subsubsection{Effect of boundary input data under varying noise levels}
We next examine the effect of the boundary input data on the reconstruction under noisy measurements. Figure~\ref{fig:measurements_with_trigo_input} illustrates the boundary data generated by the oscillatory input $g=\sin(\pi x)\sin(\pi y)$ for increasing noise levels $\delta$, which introduce significant high-frequency perturbations compared to the constant input case.
%
\begin{figure}[htp!]
\centering
\resizebox{0.225\textwidth}{!}{\includegraphics{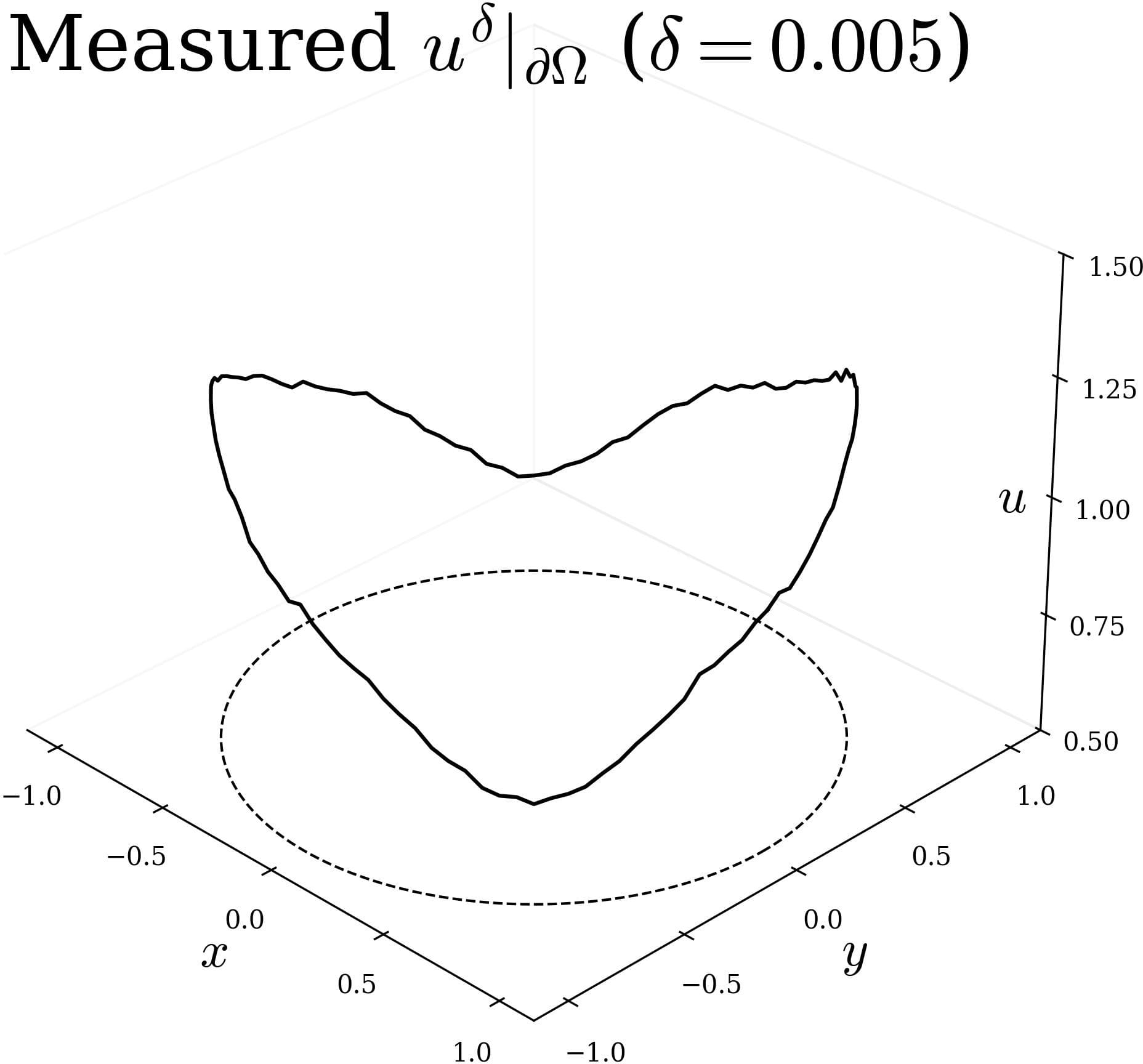}} \ 
\resizebox{0.225\textwidth}{!}{\includegraphics{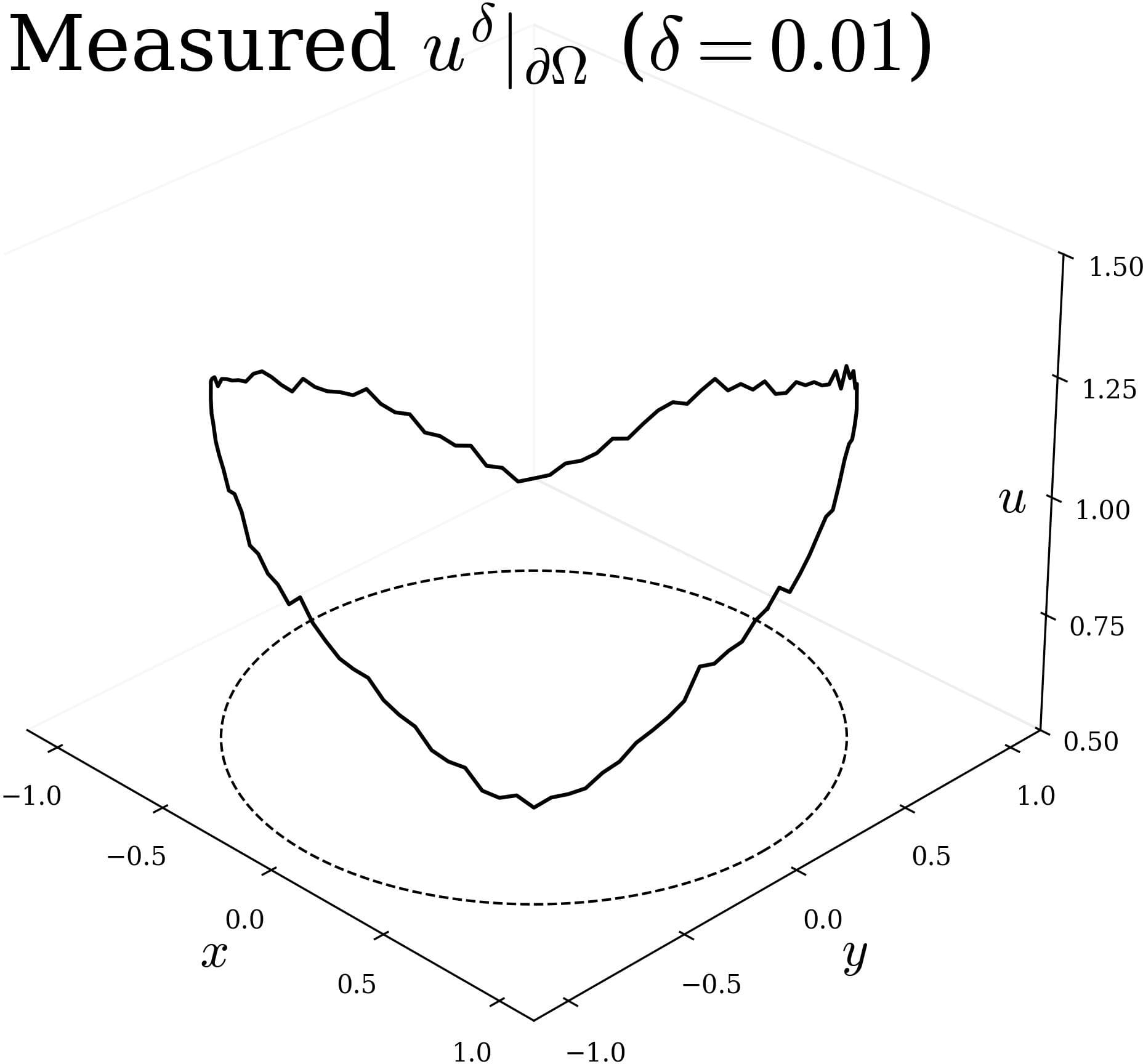}} \
\resizebox{0.225\textwidth}{!}{\includegraphics{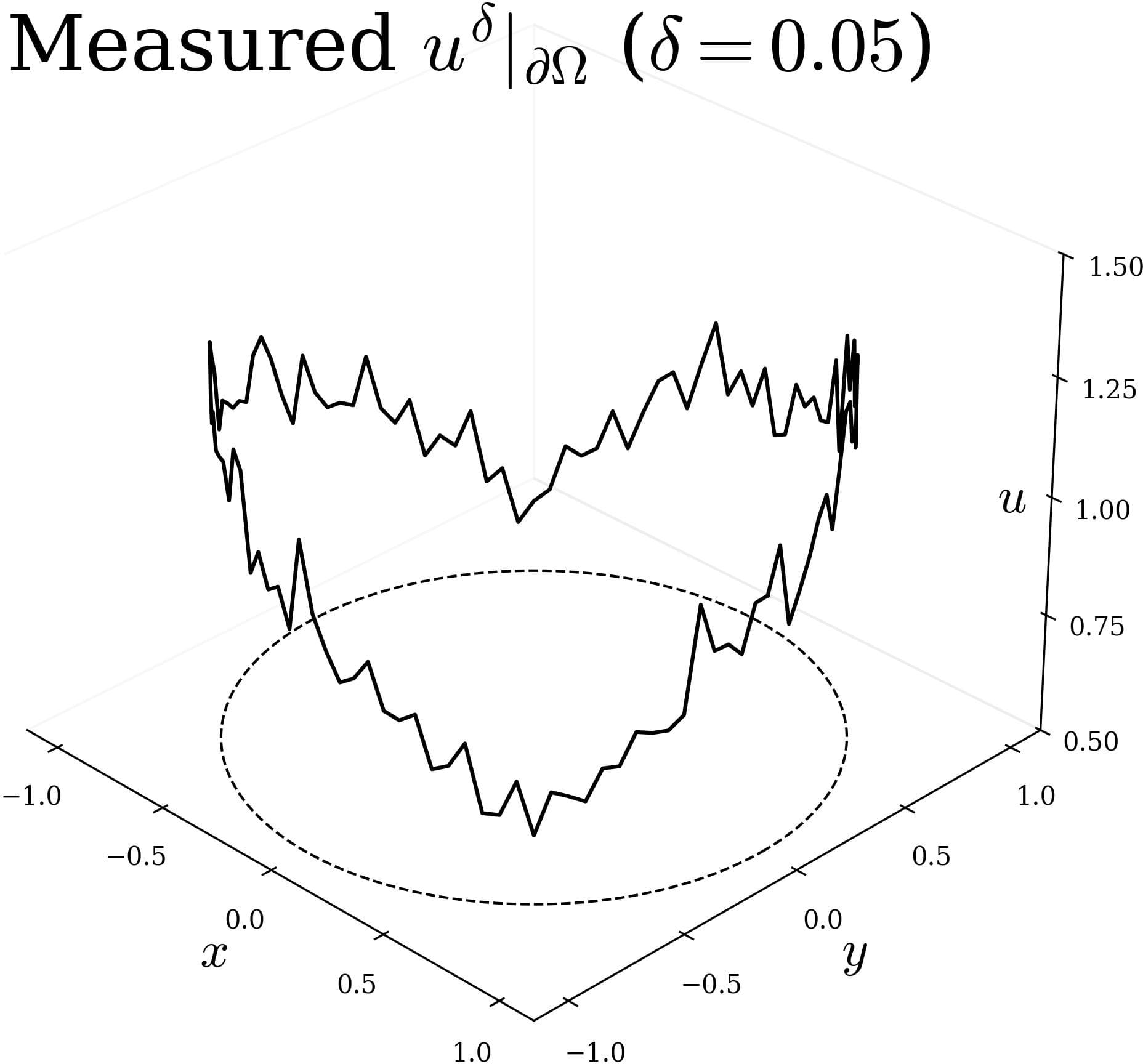}} \
\resizebox{0.225\textwidth}{!}{\includegraphics{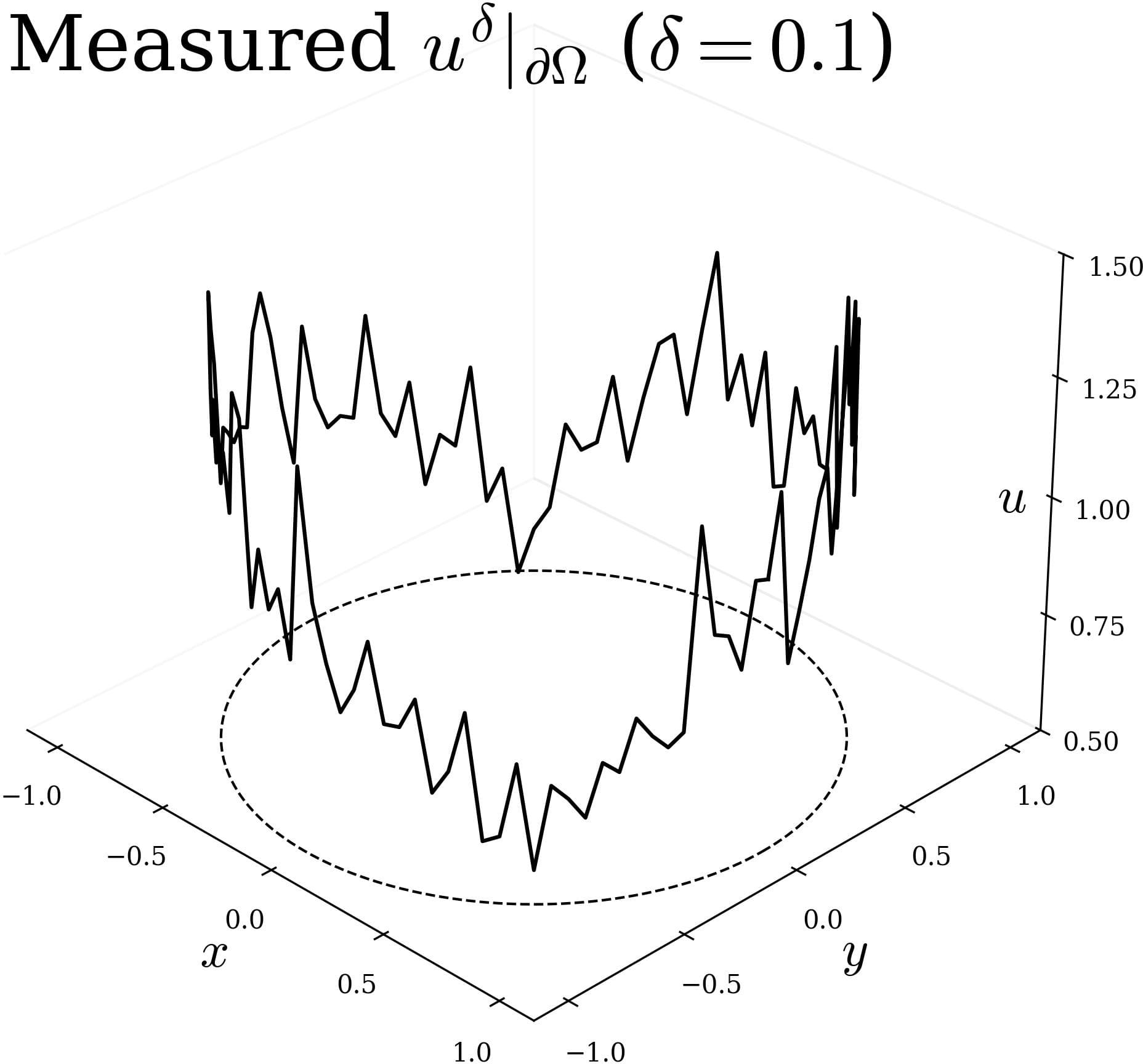}}
\caption{Boundary measurements at different noise levels $\delta$ with input data $g=\sin(\pi x)\sin(\pi y)$.}
\label{fig:measurements_with_trigo_input}
\end{figure}
The corresponding reconstructions are shown in Figures~\ref{fig:effect_of_input_data_g_constant} and~\ref{fig:effect_of_input_data_g_trigo}. Despite the increased oscillatory content and noise contamination, the CCBM formulation remains remarkably robust with respect to the choice of the Tikhonov parameter $\rho$. In particular, CCBM consistently produces smooth and accurate reconstructions over a wide range of $\rho$, with only minor sensitivity to parameter tuning.

By contrast, the KV, TD, and TN formulations exhibit a much stronger dependence on $\rho$. As $\rho$ decreases, these methods develop pronounced oscillations and localized instabilities that are further amplified by the noisy oscillatory boundary input. Although gradient smoothing alleviates some of these effects, the resulting reconstructions remain noticeably less stable than those obtained with CCBM.

Figure~\ref{fig:effect_of_input_data_g_trigo_with_noise_0.005} shows the influence of the Tikhonov parameter $\rho$ for the oscillatory input $g=\sin(\pi x)\sin(\pi y)$ at $\delta=0.005$. The CCBM reconstructions remain smooth and geometrically accurate over several orders of magnitude in $\rho$, with no visible oscillations or loss of shape integrity. This indicates a weak sensitivity to regularization tuning and an intrinsic stabilization effect of the CCBM formulation.

By contrast, the KV and TD methods exhibit a strong dependence on $\rho$. As $\rho$ decreases, both approaches develop increasingly pronounced oscillations and geometric distortions, which are amplified by the oscillatory boundary data. The TD formulation is particularly sensitive, showing severe high-frequency artifacts for small $\rho$. The TN approach suppresses oscillations effectively but produces overly diffusive reconstructions that deviate substantially from the true shape.

Figure~\ref{fig:reconstructions_under_higher_noise_levels} illustrates the effect of increasing noise levels $\delta=0.01,0.03,0.05,0.10$. The CCBM reconstructions remain stable as $\delta$ increases, preserving the main geometric features and exhibiting a gradual, well-controlled degradation. In contrast, the TD-based reconstructions deteriorate rapidly with increasing noise, leading to strong distortions and loss of physical relevance despite gradient smoothing.

Overall, Figures~\ref{fig:effect_of_input_data_g_constant} and~\ref{fig:reconstructions_under_higher_noise_levels} demonstrate that oscillatory and noisy boundary data severely affect classical formulations, whereas CCBM consistently delivers robust reconstructions across wide ranges of $\rho$ and $\delta$.

\begin{figure}[htp!]
\centering
\resizebox{0.225\textwidth}{!}{\includegraphics{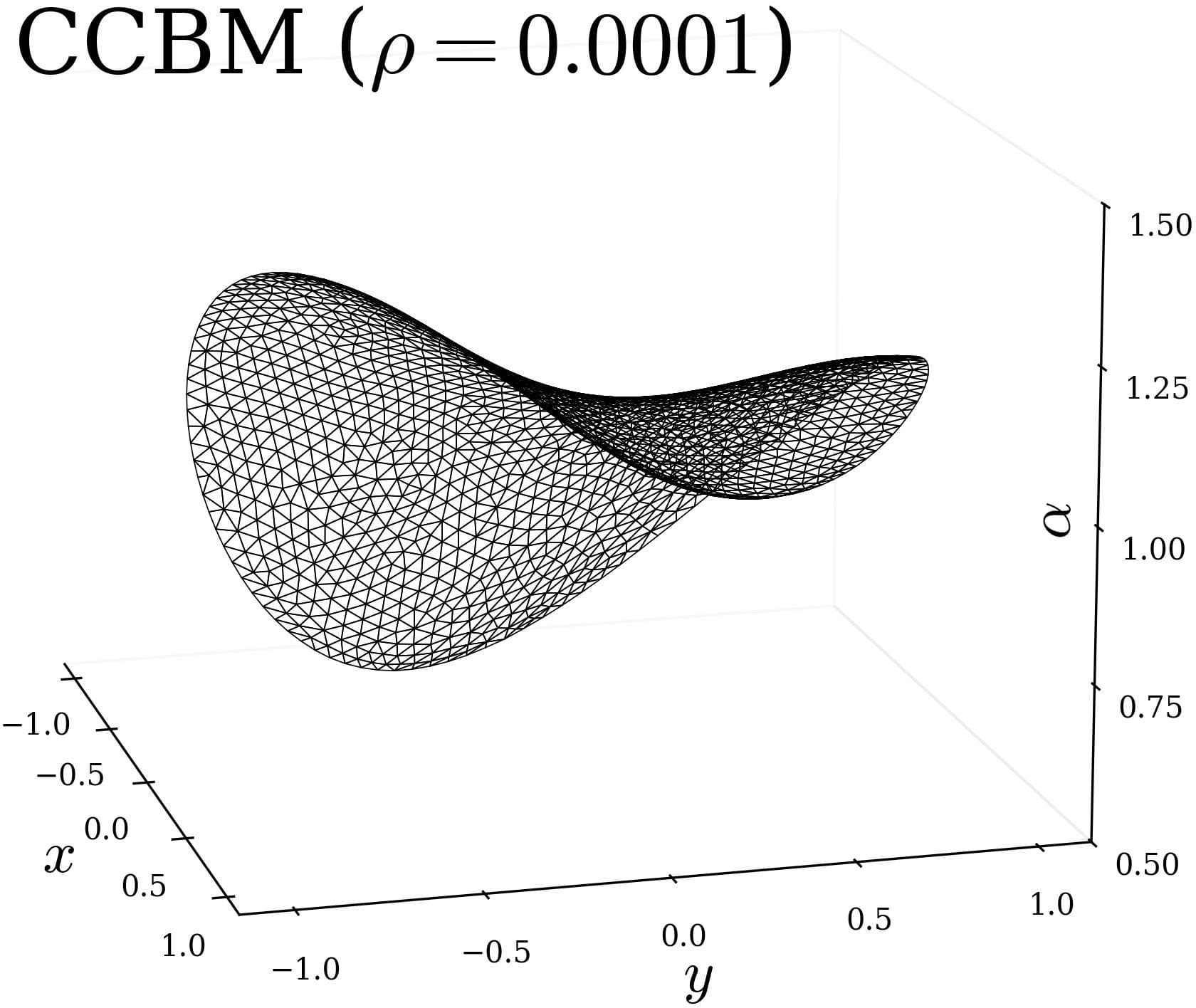}} \ 
\resizebox{0.225\textwidth}{!}{\includegraphics{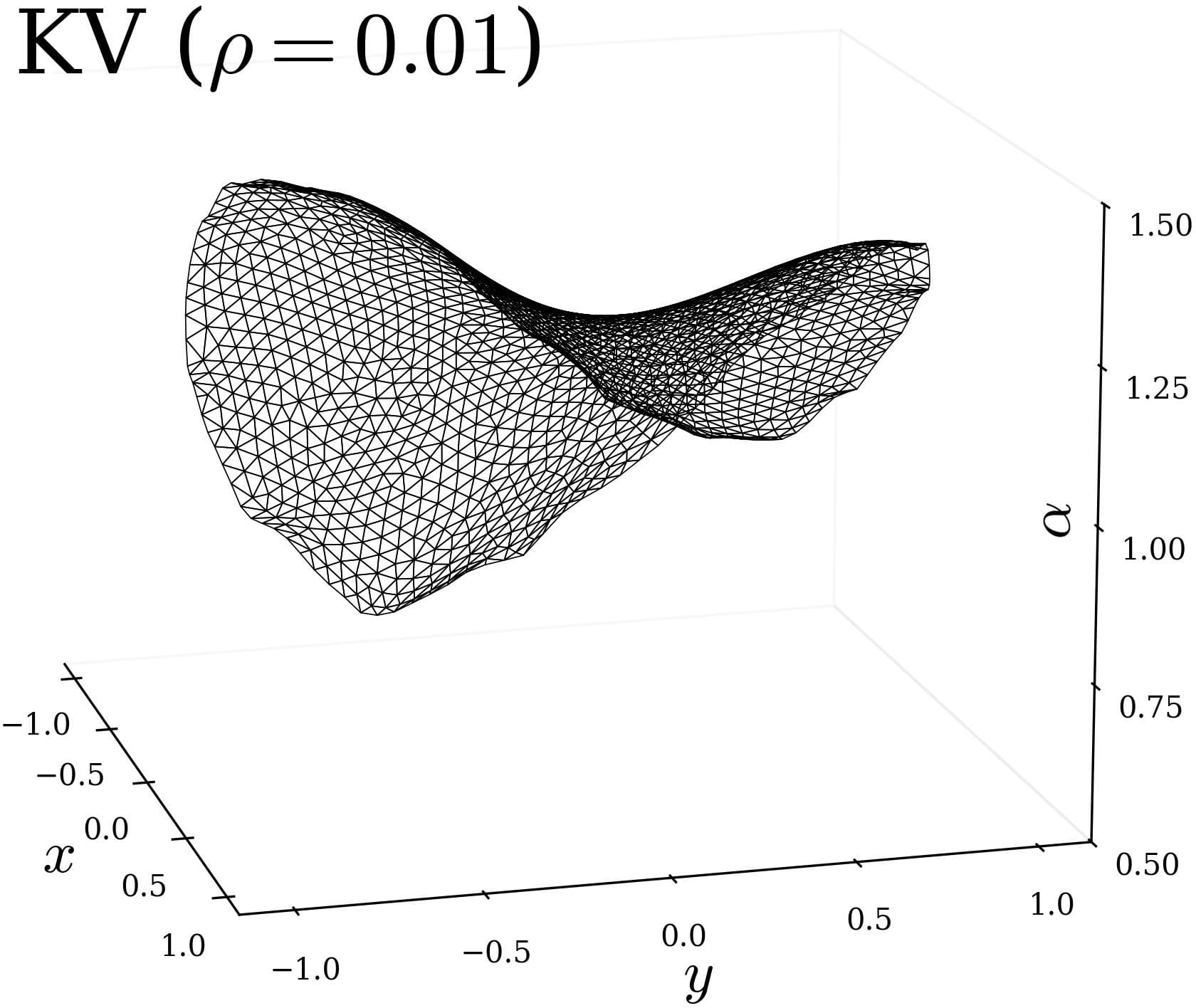}} \ 
\resizebox{0.225\textwidth}{!}{\includegraphics{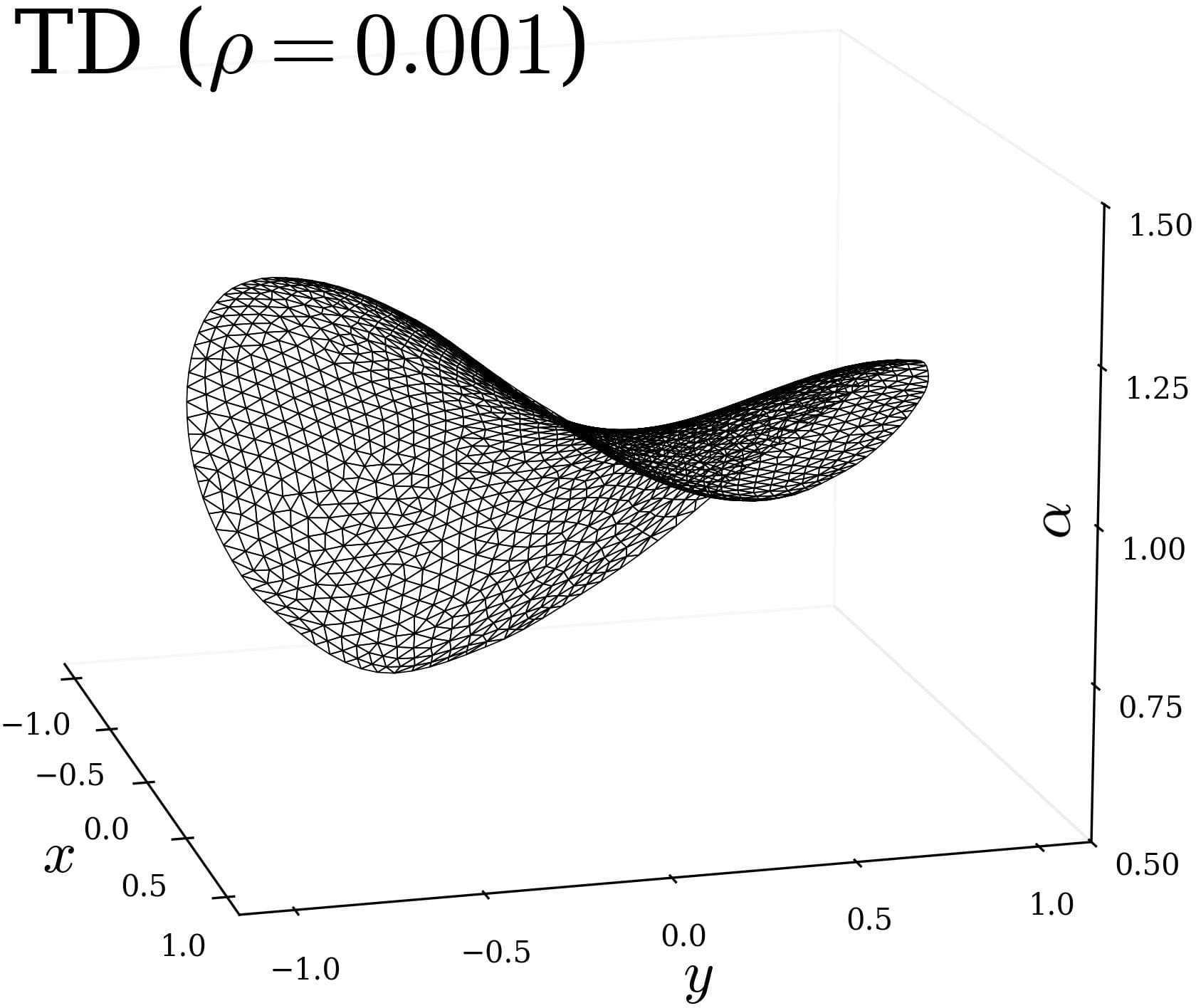}} \ 
\resizebox{0.225\textwidth}{!}{\includegraphics{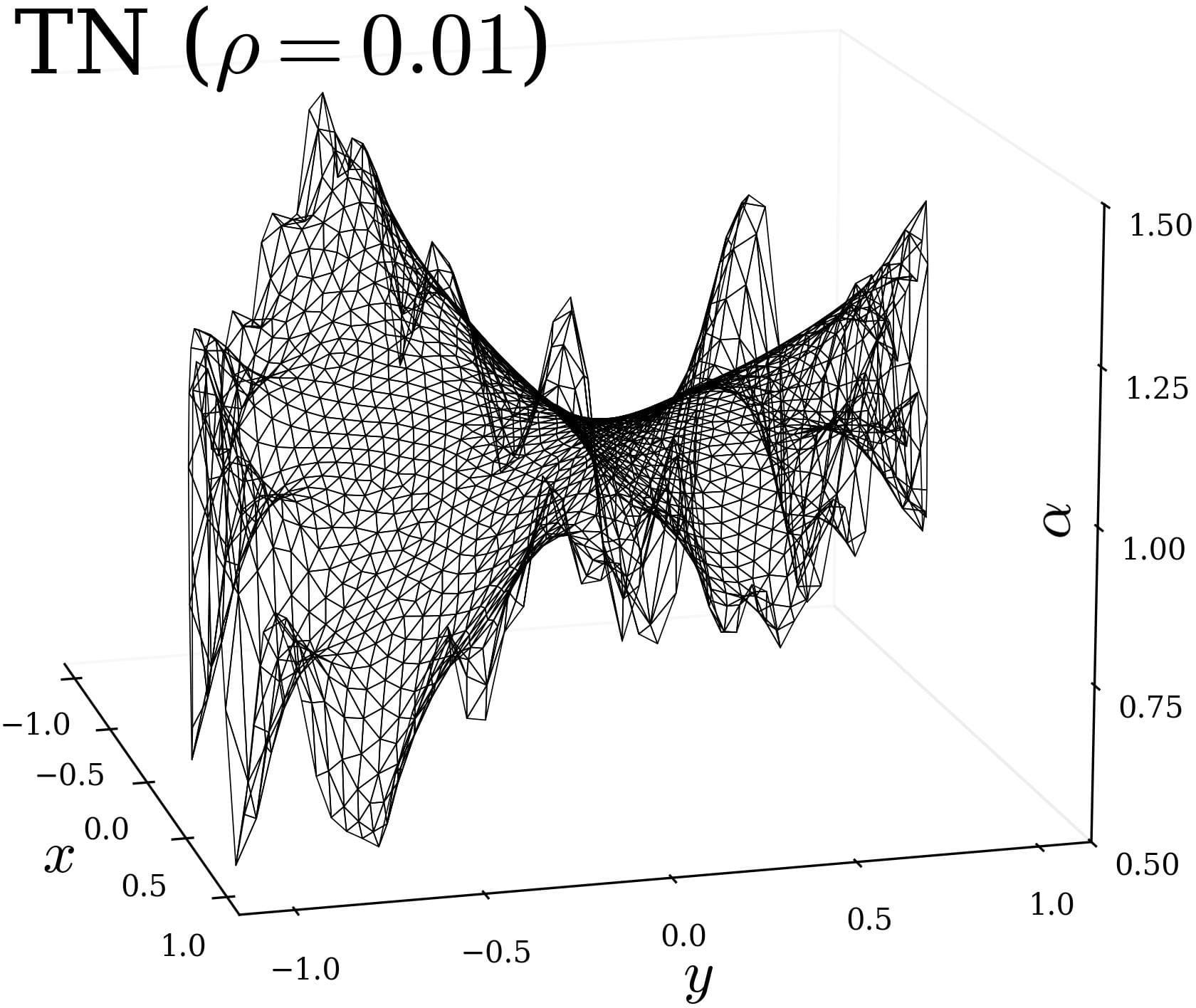}} \\[1em]
\resizebox{0.225\textwidth}{!}{\includegraphics{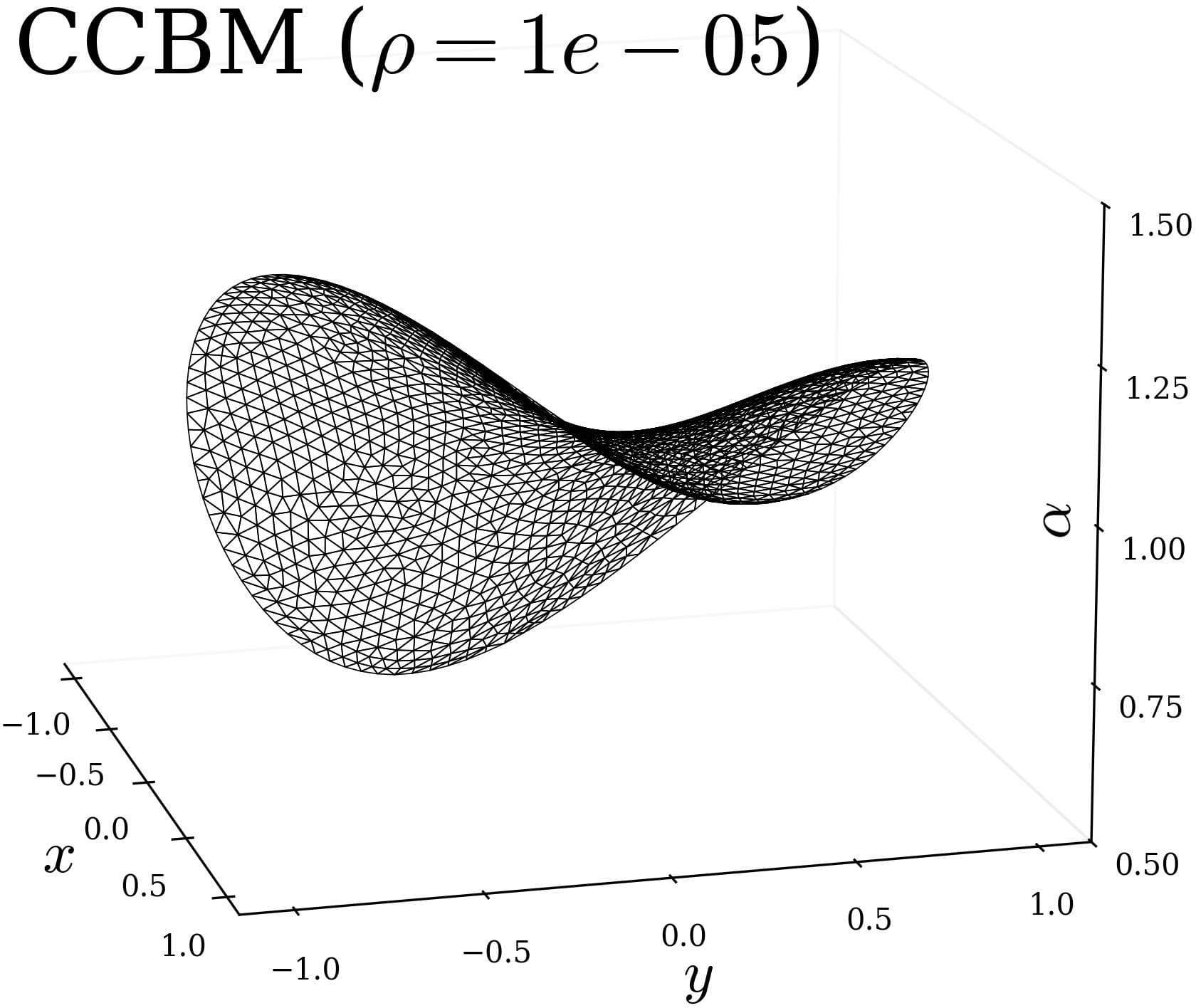}} \ 
\resizebox{0.225\textwidth}{!}{\includegraphics{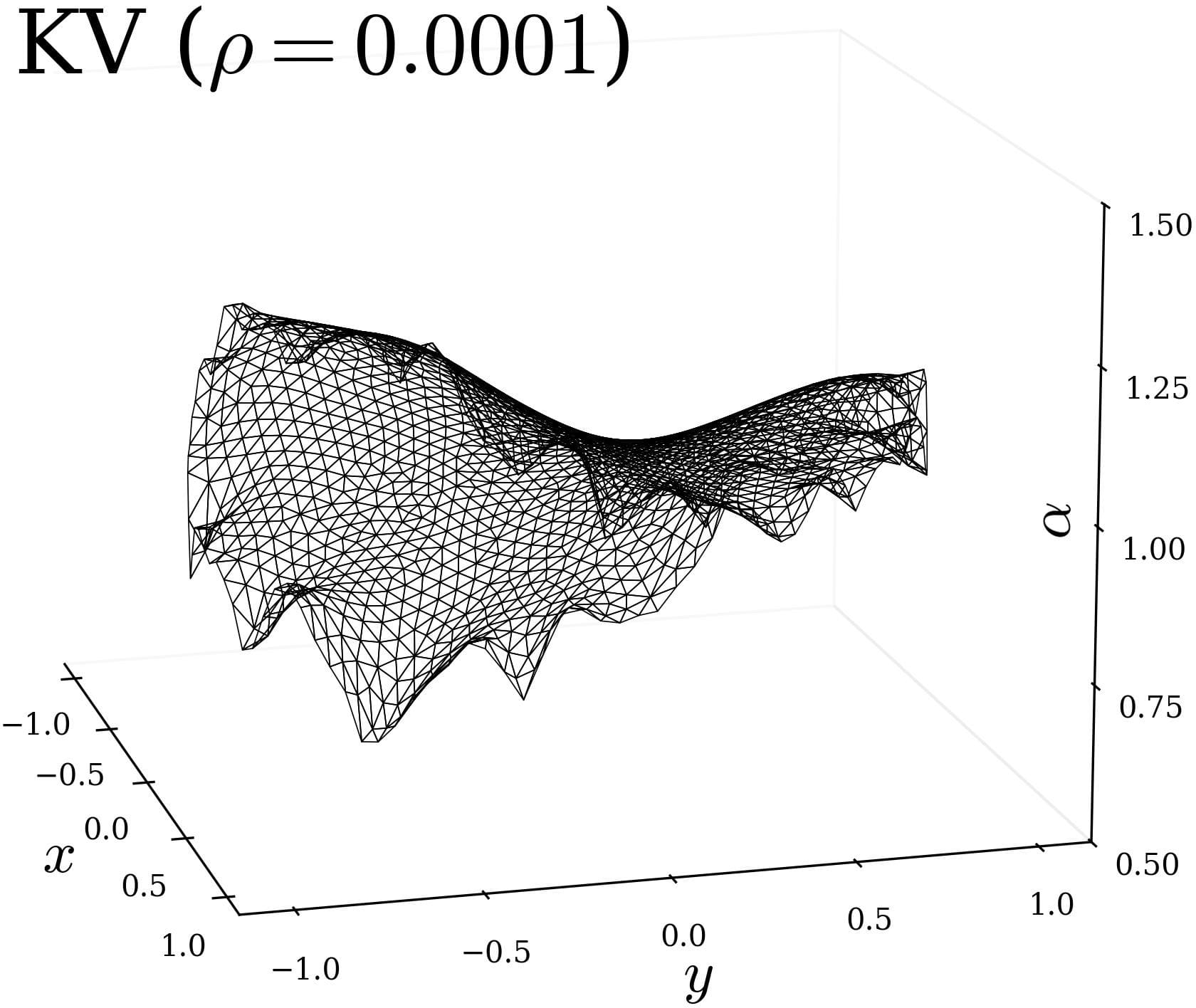}} \ 
\resizebox{0.225\textwidth}{!}{\includegraphics{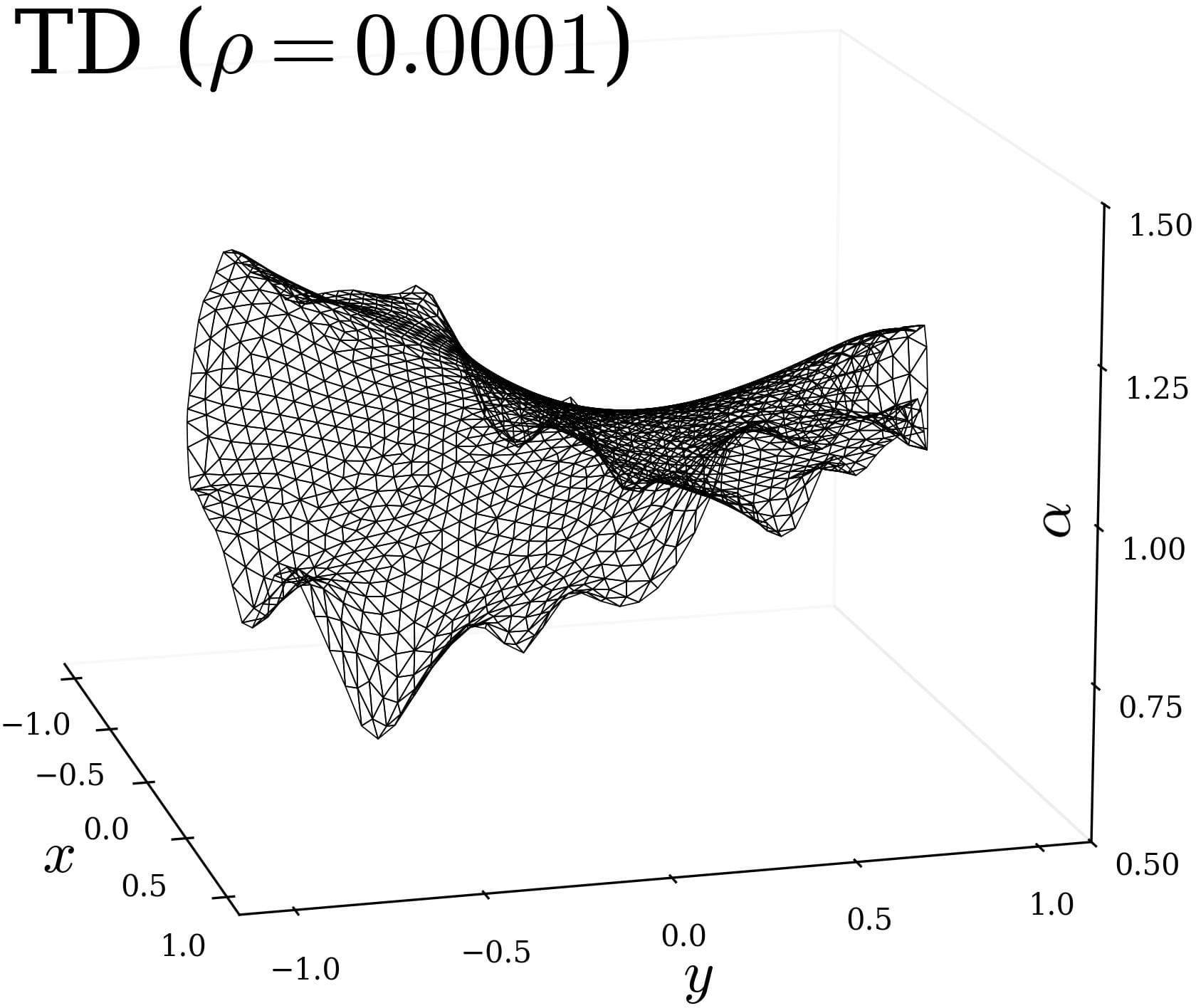}} \ 
\resizebox{0.225\textwidth}{!}{\includegraphics{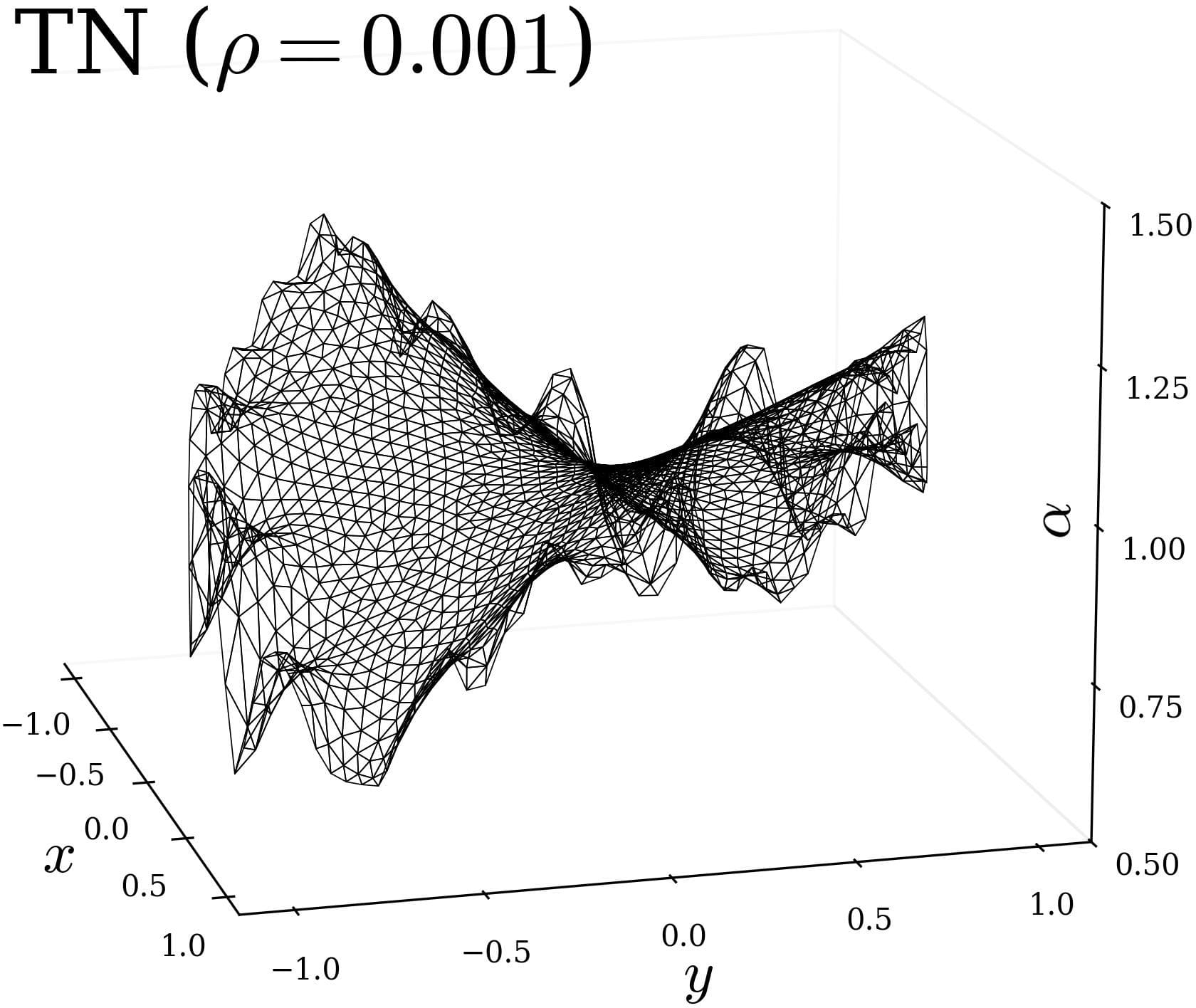}} \\[1em]
\resizebox{0.225\textwidth}{!}{\includegraphics{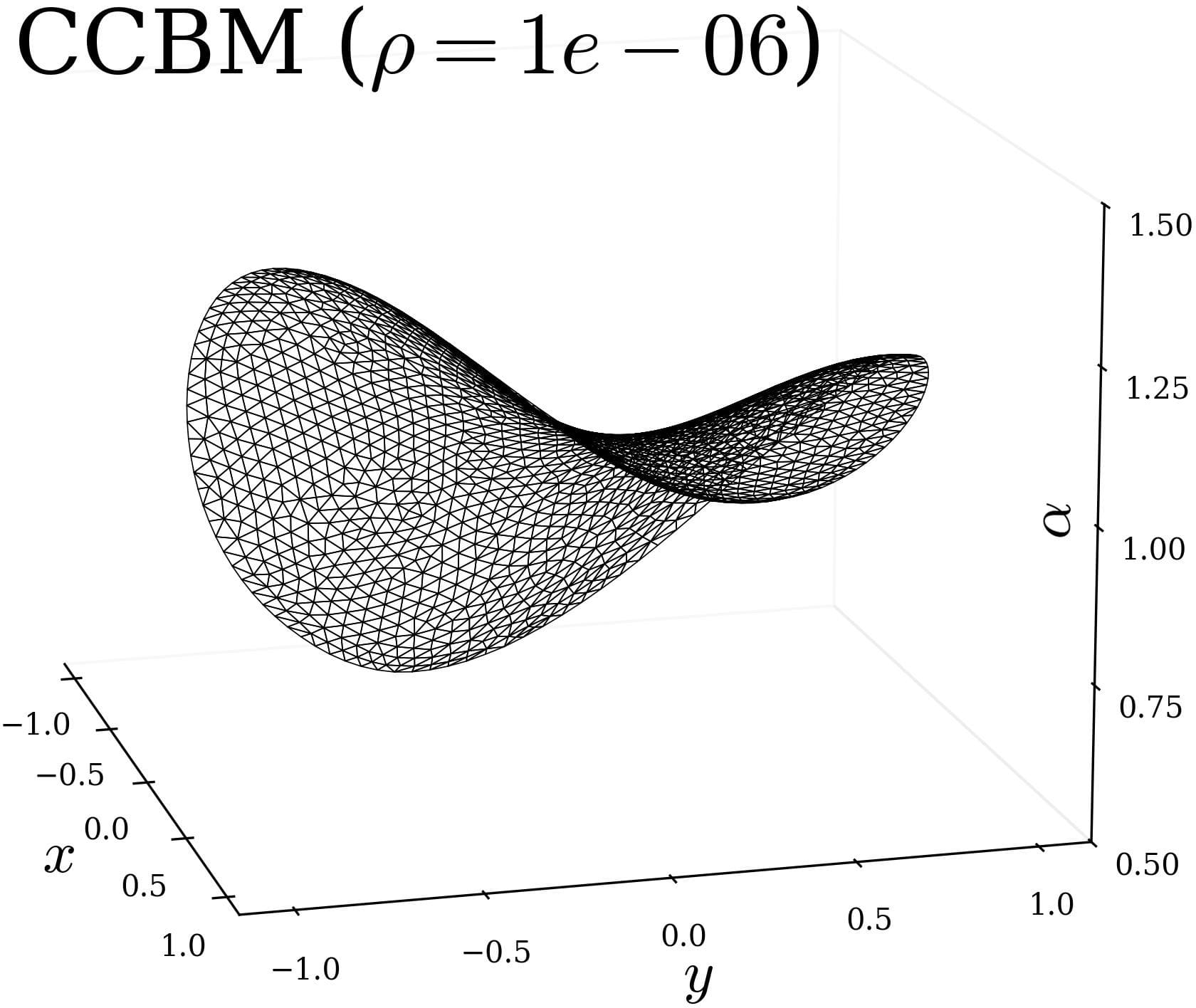}} \ 
\resizebox{0.225\textwidth}{!}{\includegraphics{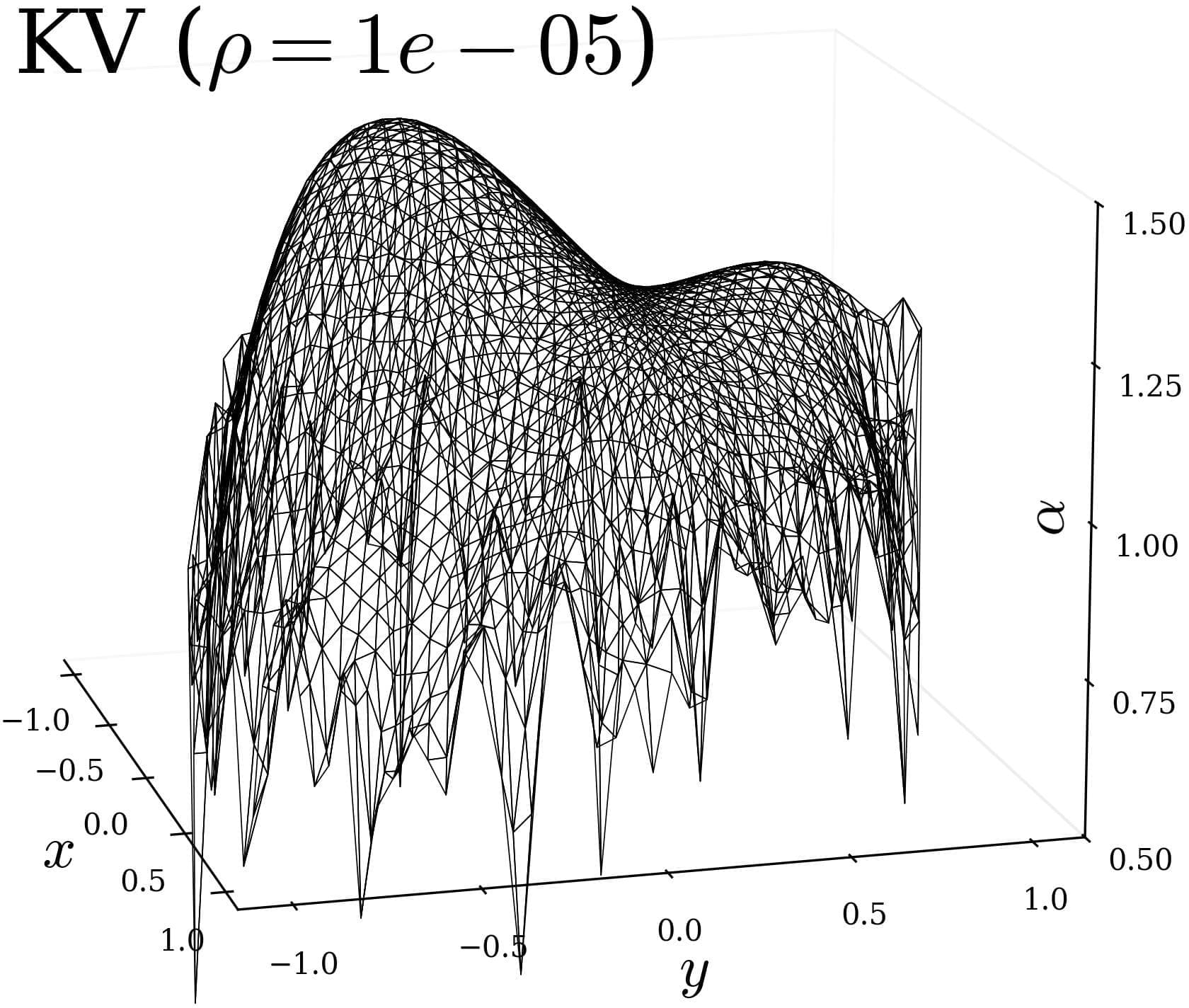}} \ 
\resizebox{0.225\textwidth}{!}{\includegraphics{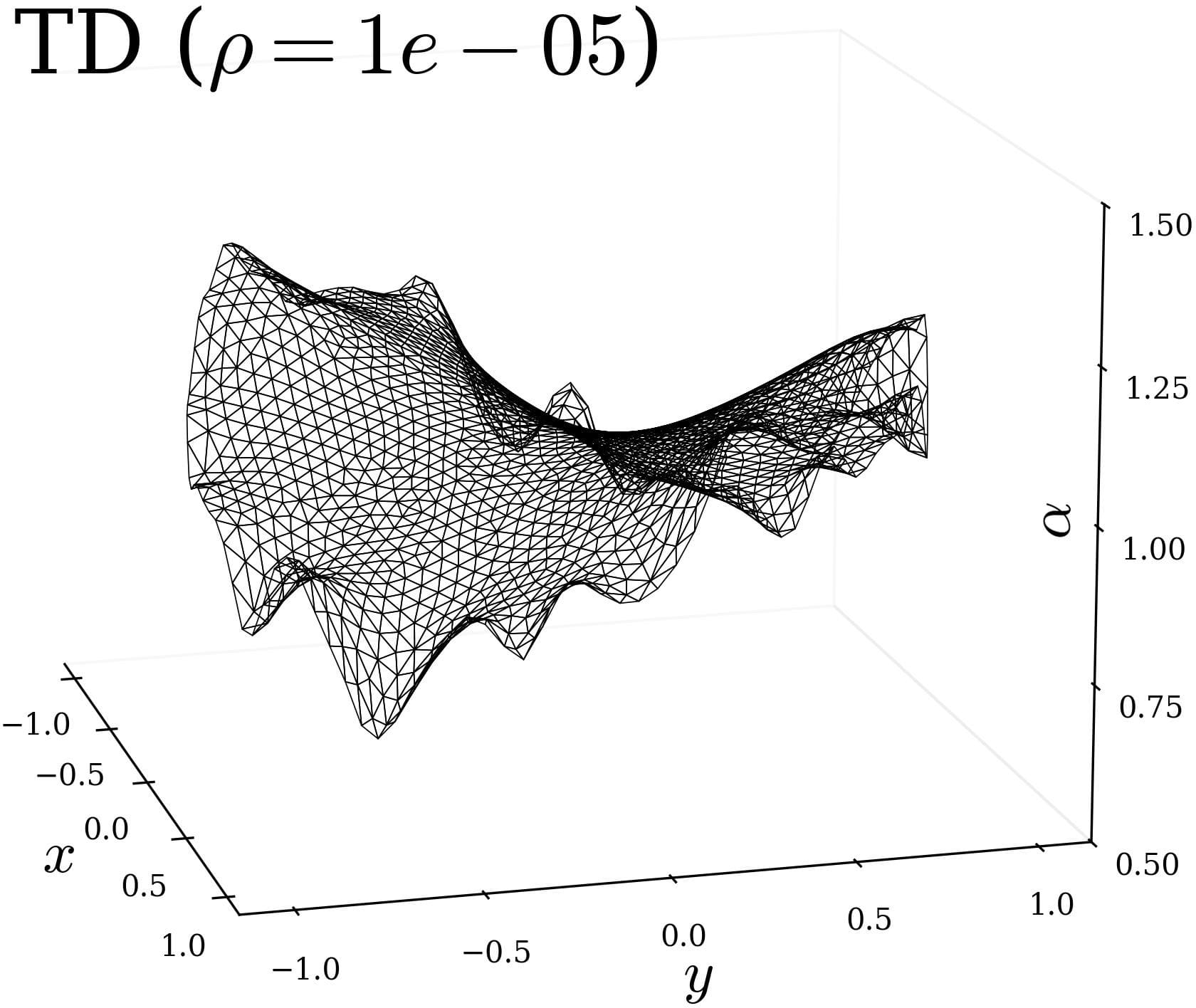}} \ 
\resizebox{0.225\textwidth}{!}{\includegraphics{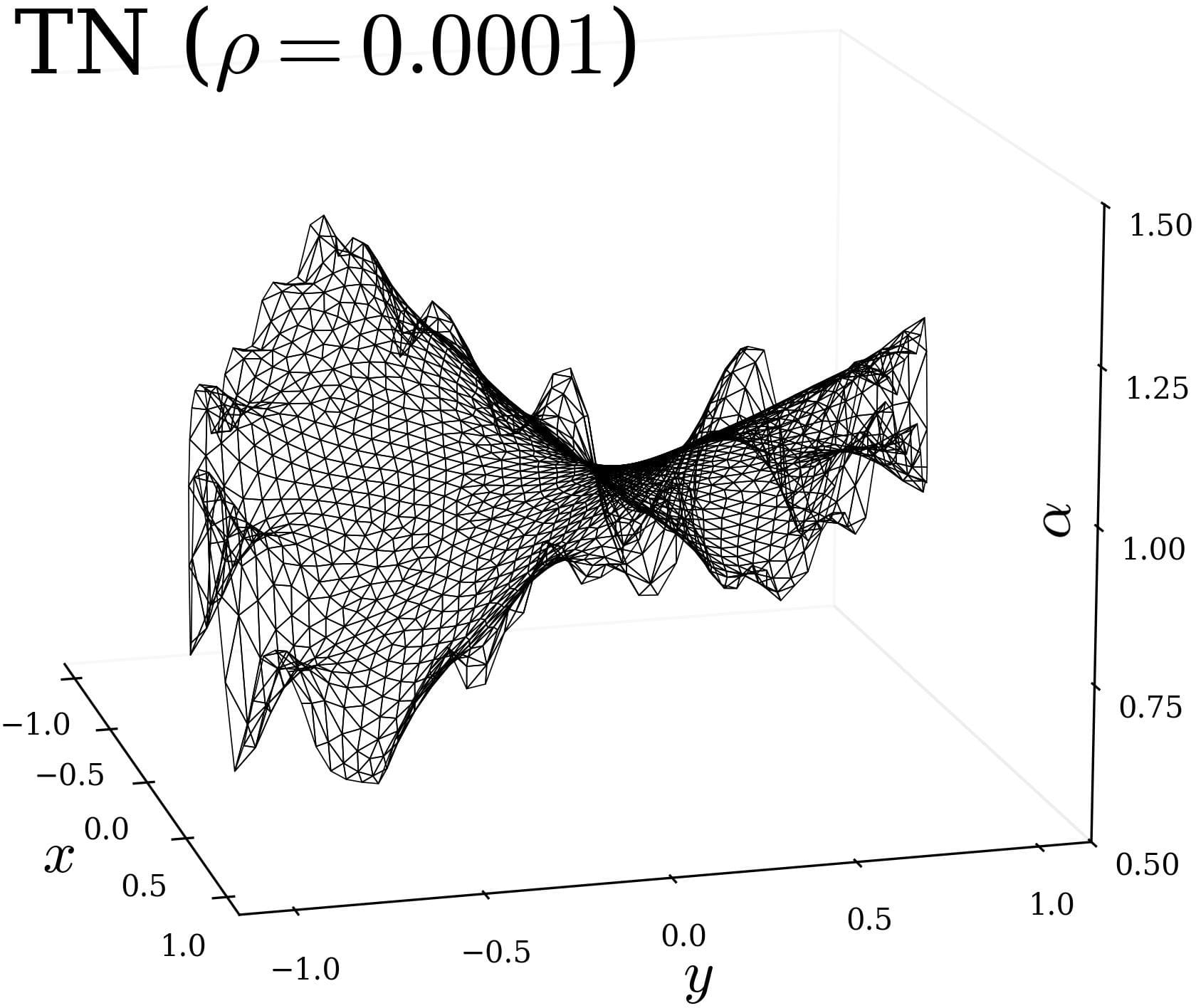}} \\[1em]
\resizebox{0.225\textwidth}{!}{\includegraphics{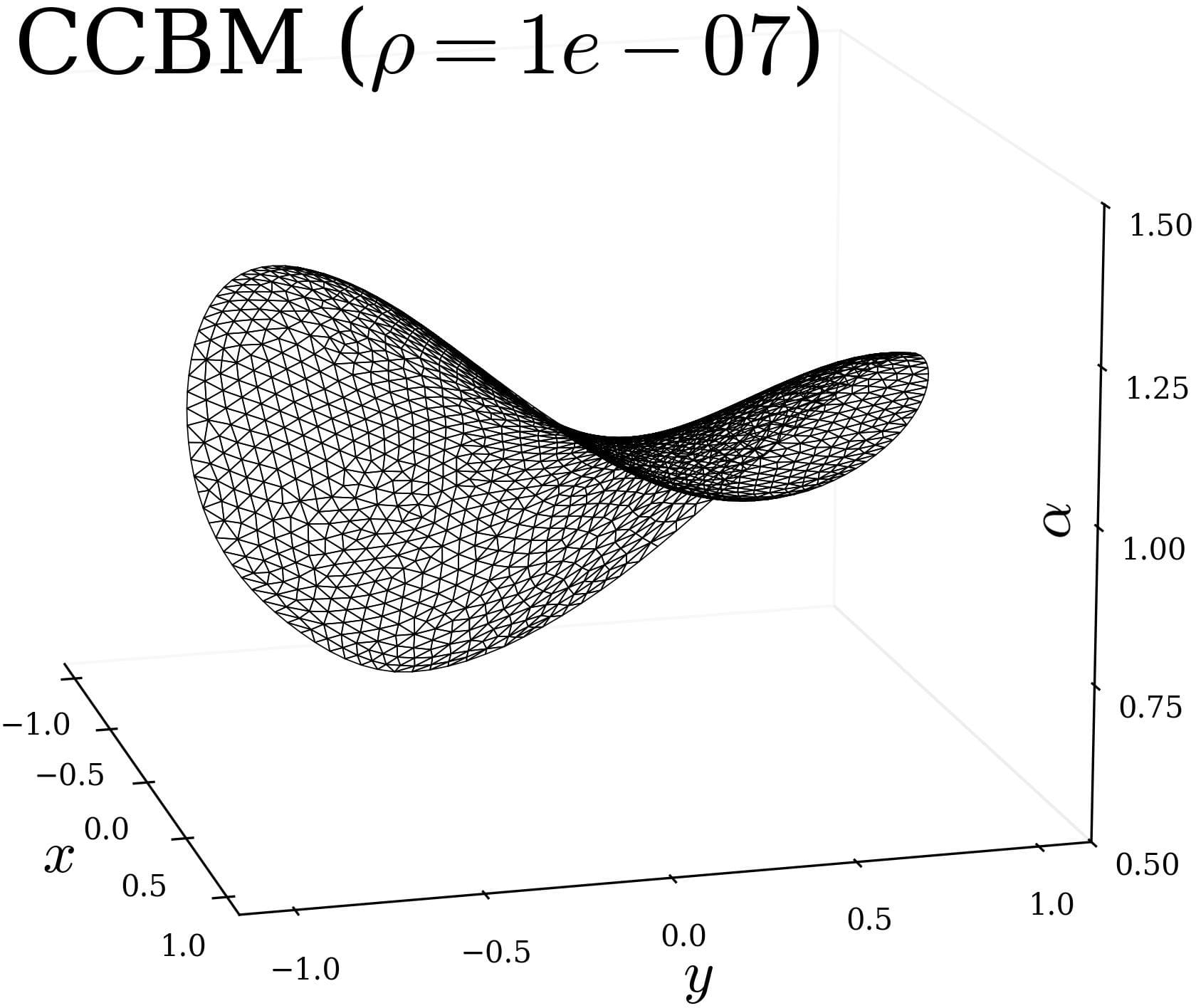}} \ 
\resizebox{0.225\textwidth}{!}{\includegraphics{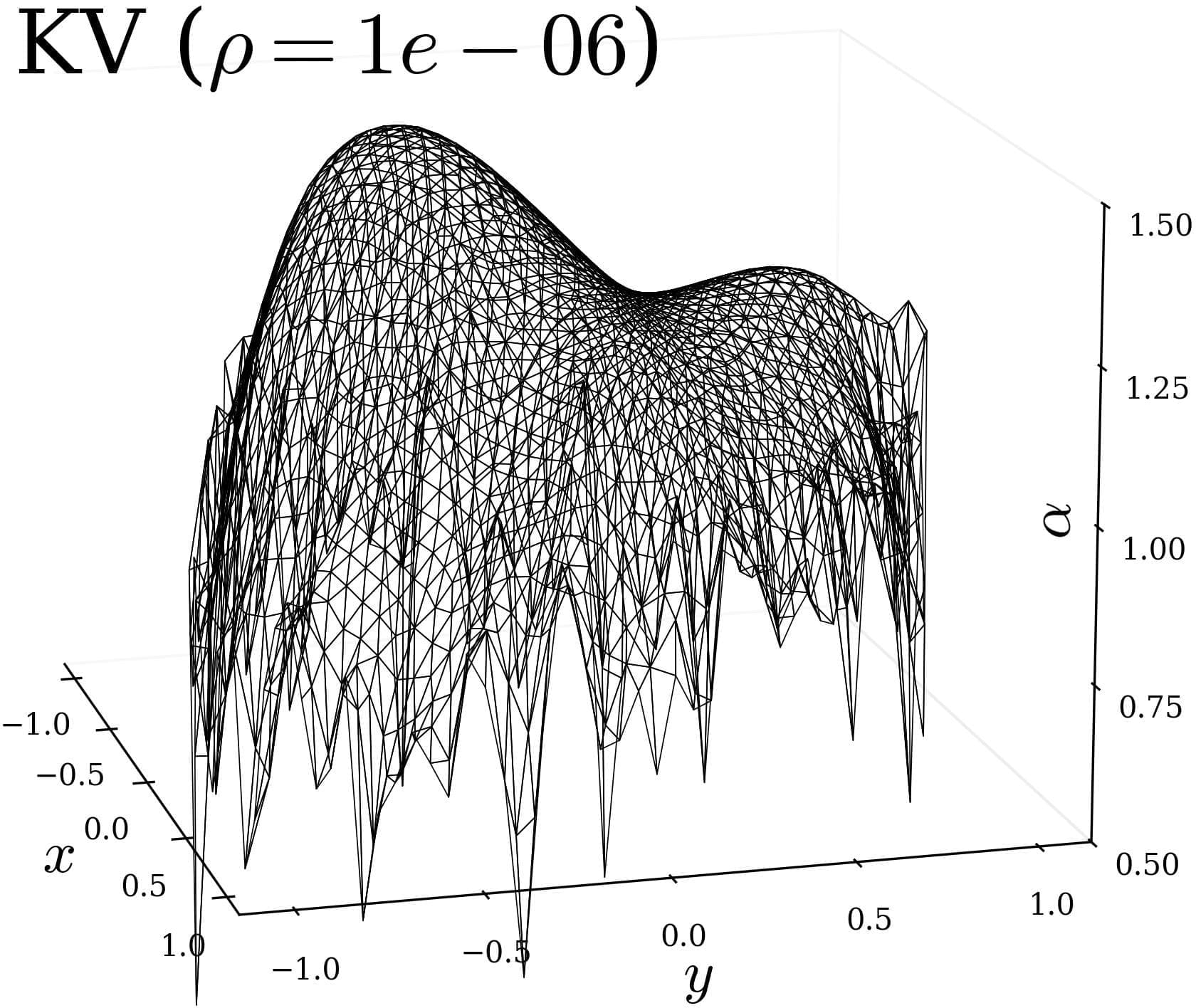}} \
\resizebox{0.225\textwidth}{!}{\includegraphics{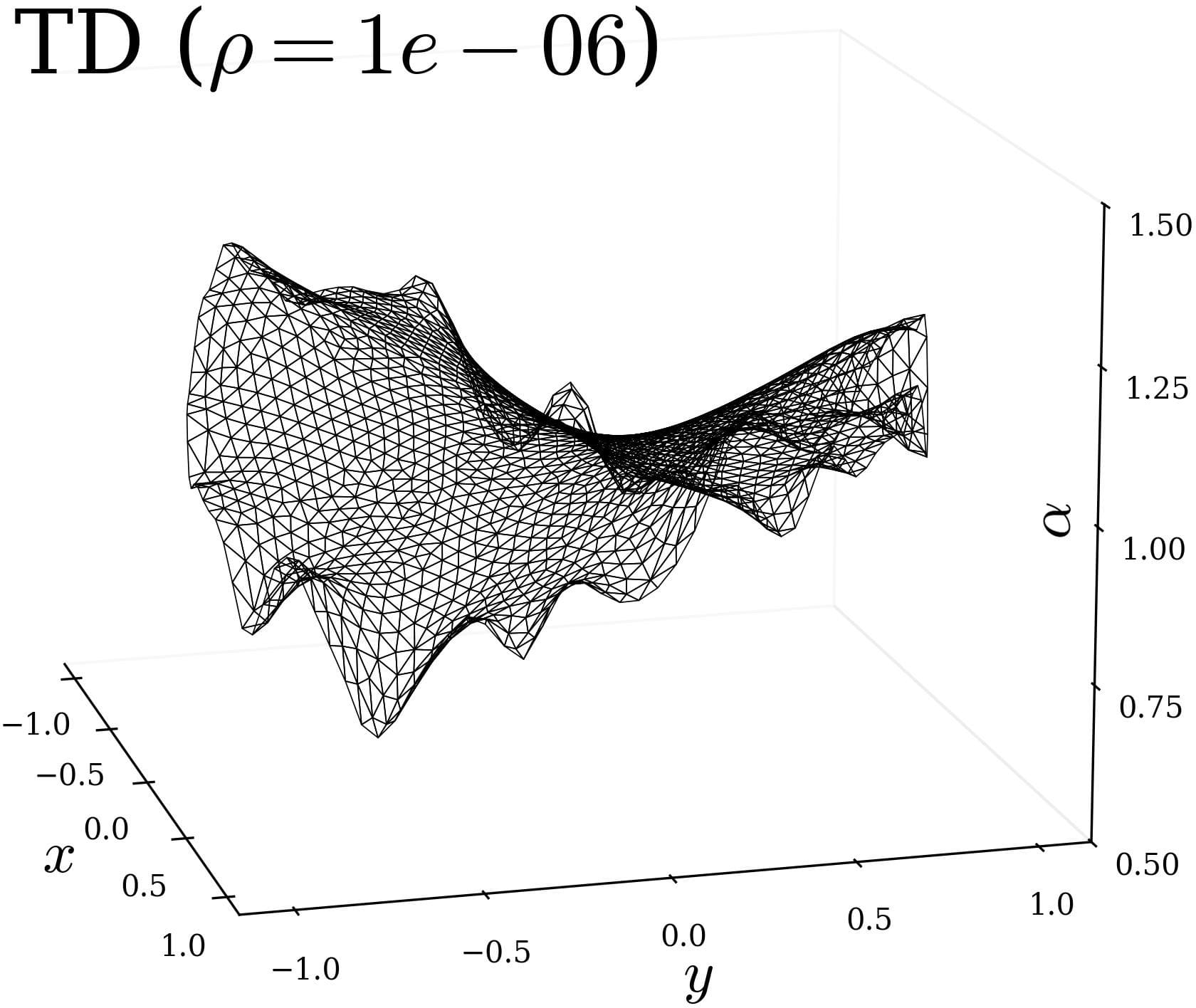}} \
\resizebox{0.225\textwidth}{!}{\includegraphics{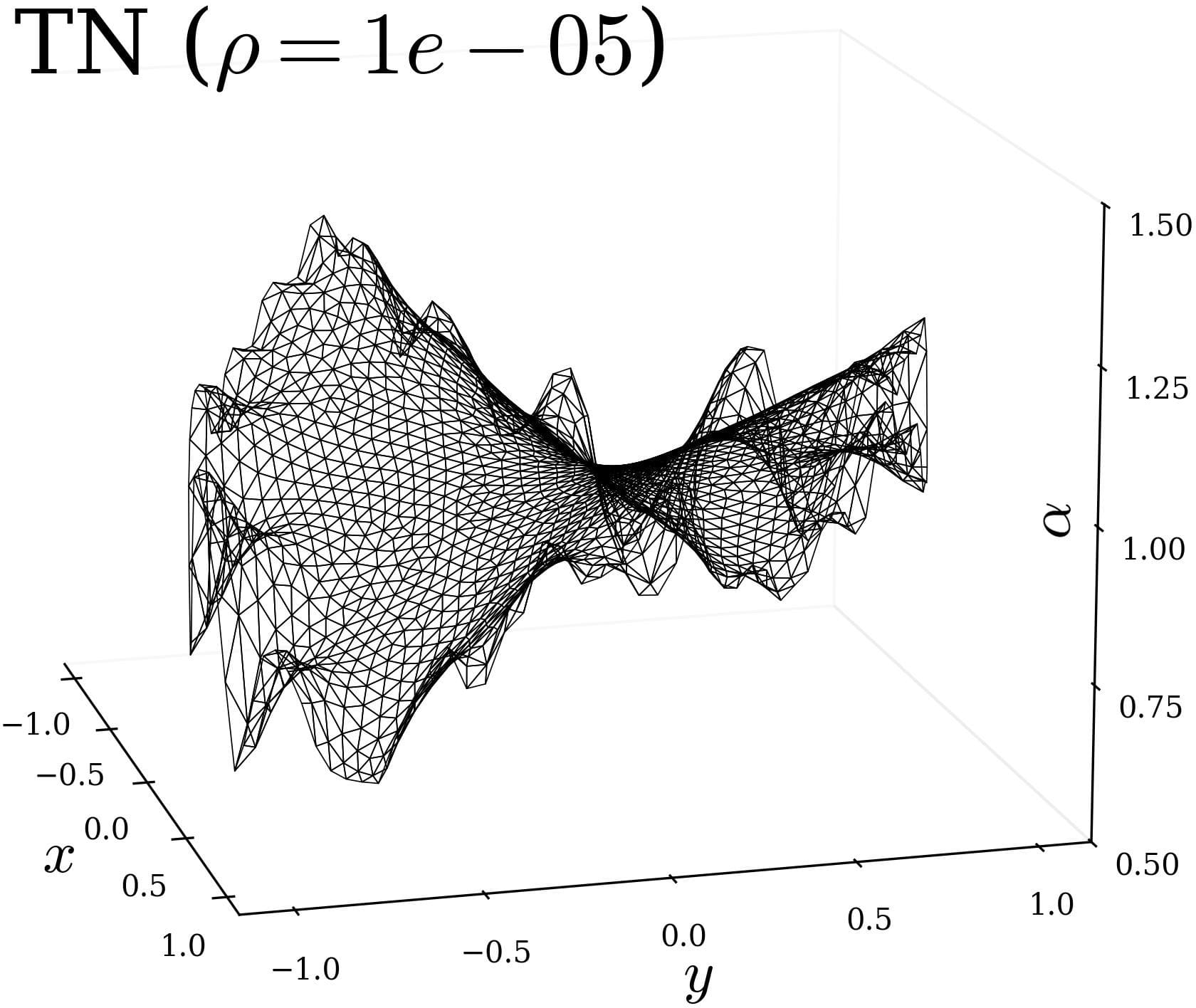}}
\caption{Influence of the Tikhonov parameter $\rho$ on the reconstruction when $\delta = 0.001$, with gradient smoothing ($\mu = 1.0$) and input data $g=1.0$.}
\label{fig:effect_of_input_data_g_constant}
\end{figure}

\begin{figure}[htp!]
\centering
\resizebox{0.225\textwidth}{!}{\includegraphics{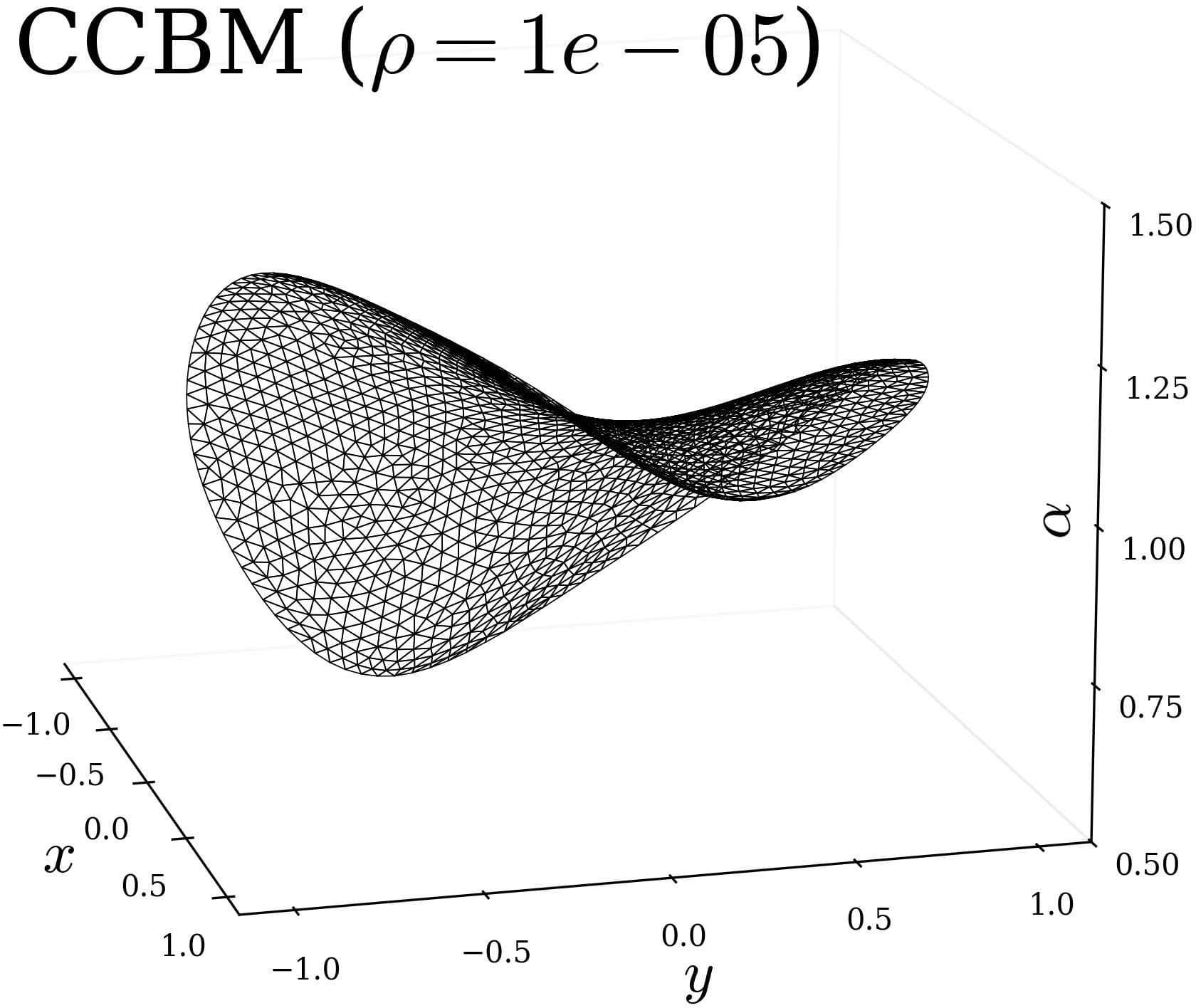}} \ 
\resizebox{0.225\textwidth}{!}{\includegraphics{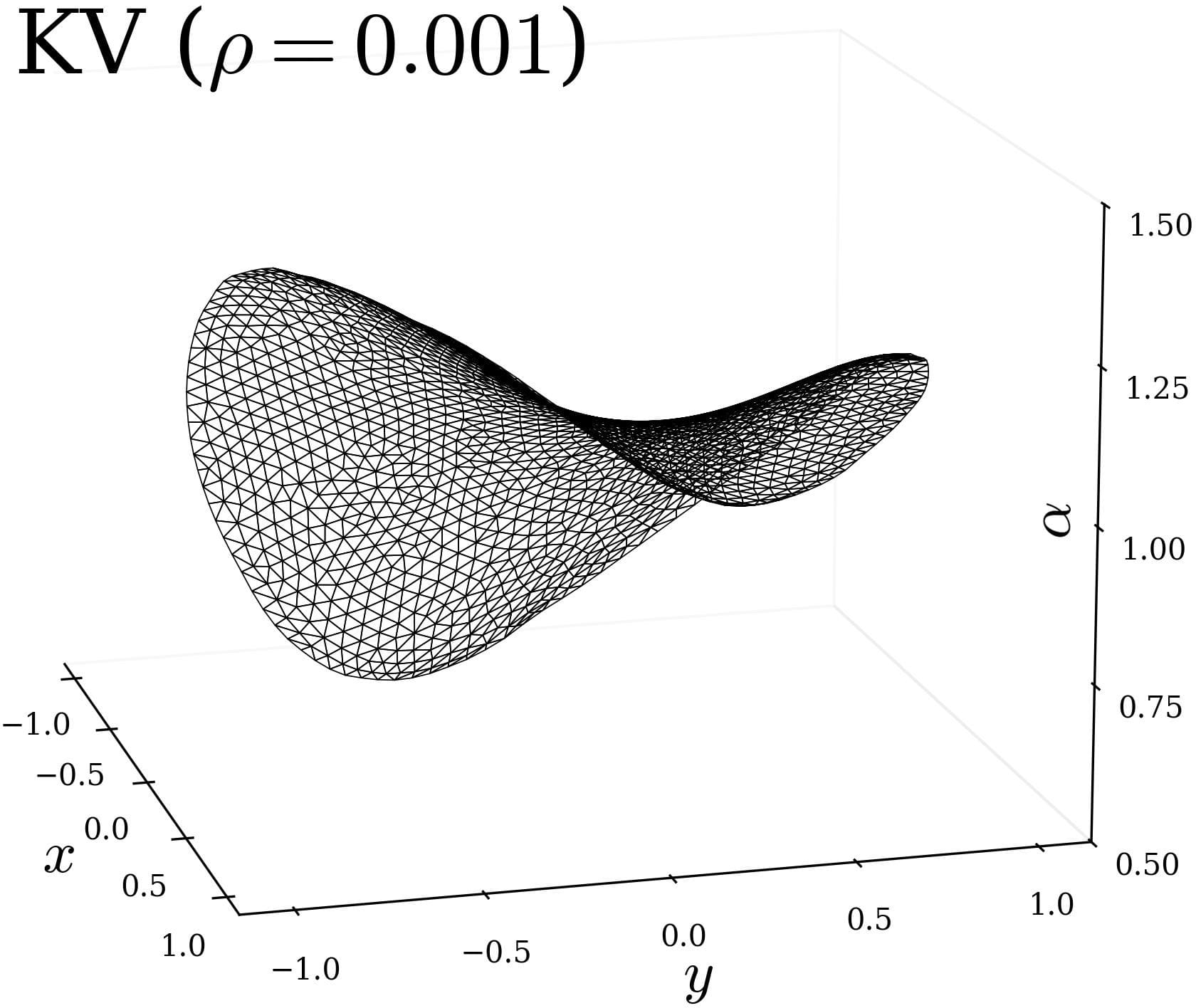}} \ 
\resizebox{0.225\textwidth}{!}{\includegraphics{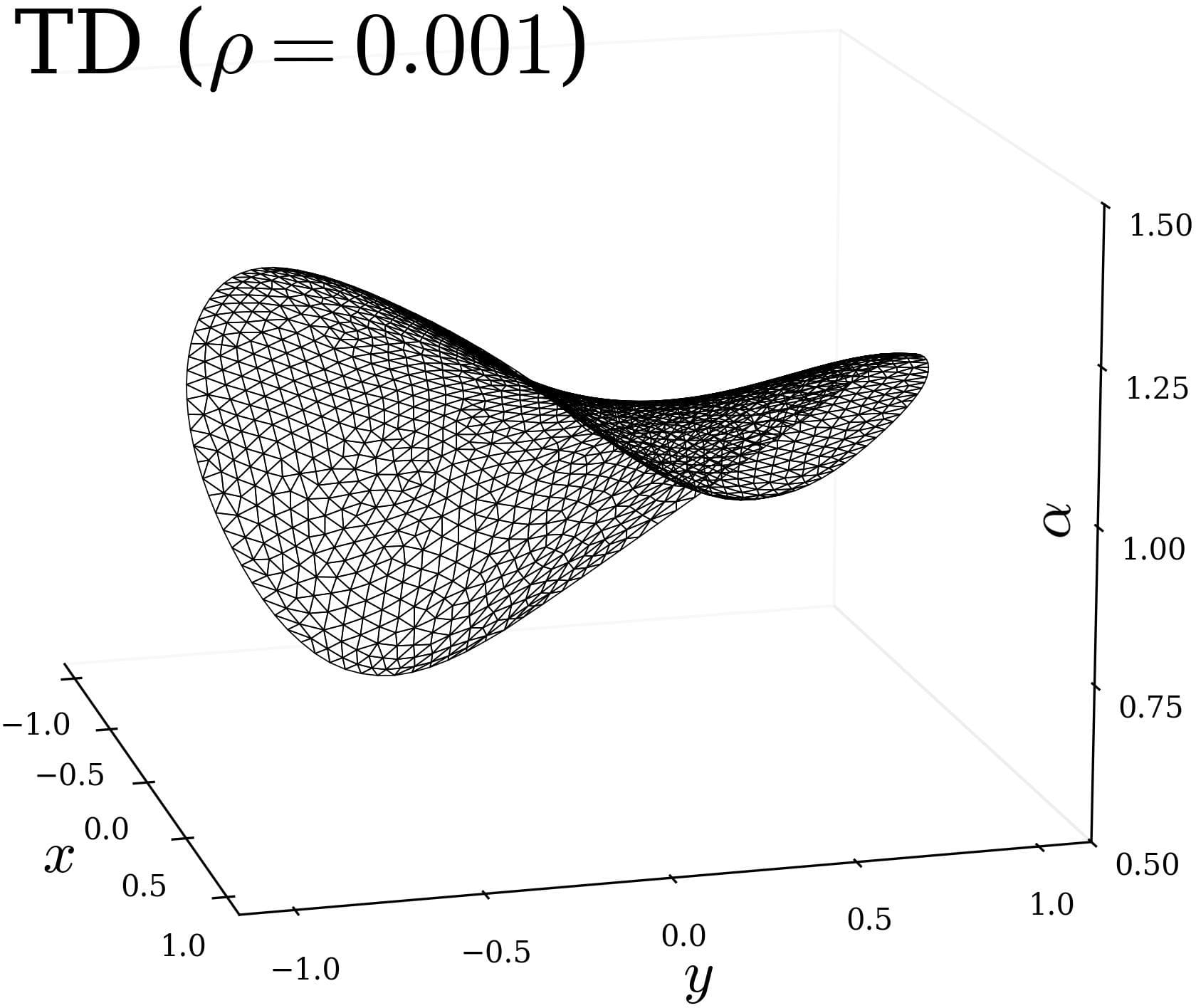}} \ 
\resizebox{0.225\textwidth}{!}{\includegraphics{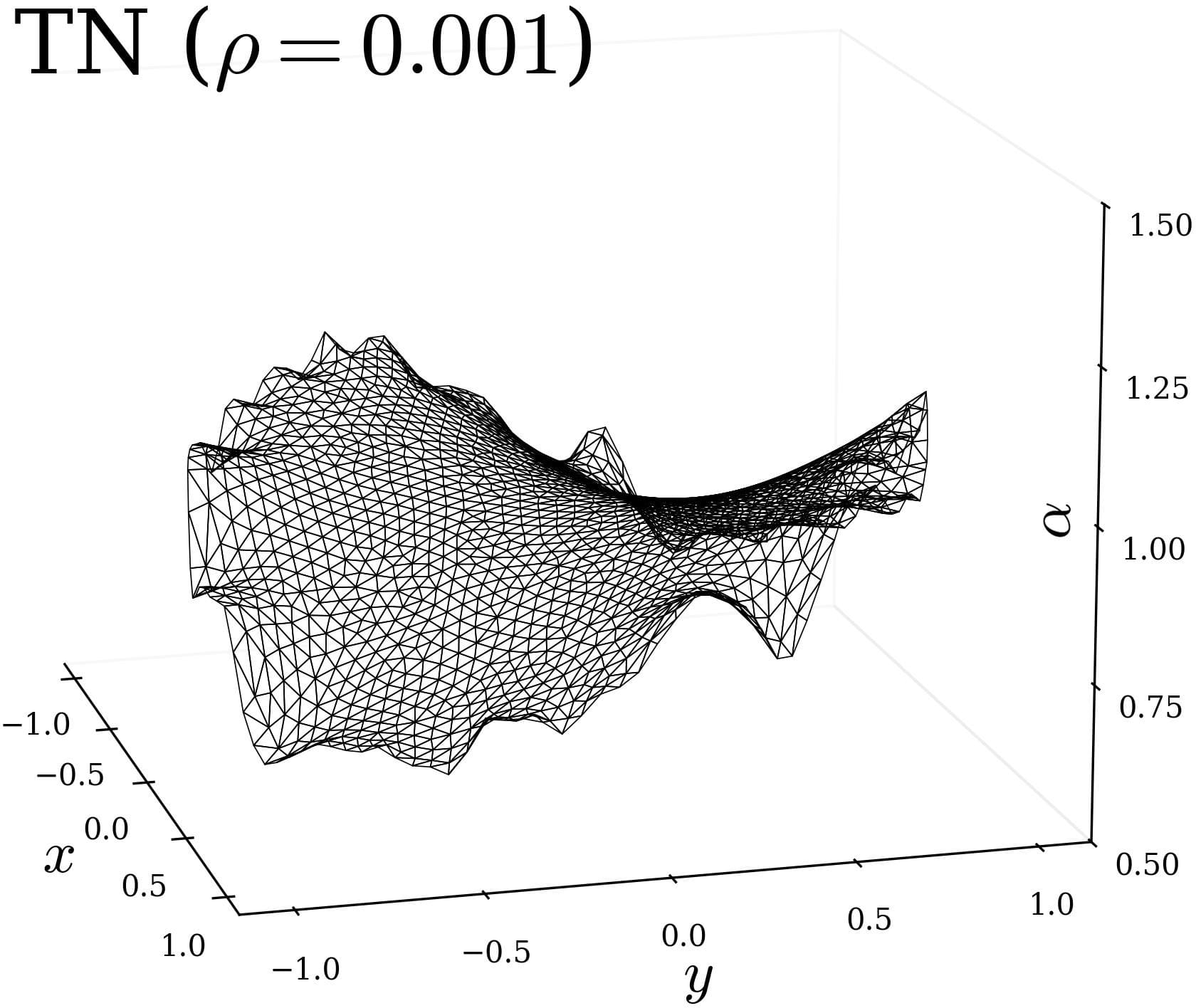}} \\[1em]
\resizebox{0.225\textwidth}{!}{\includegraphics{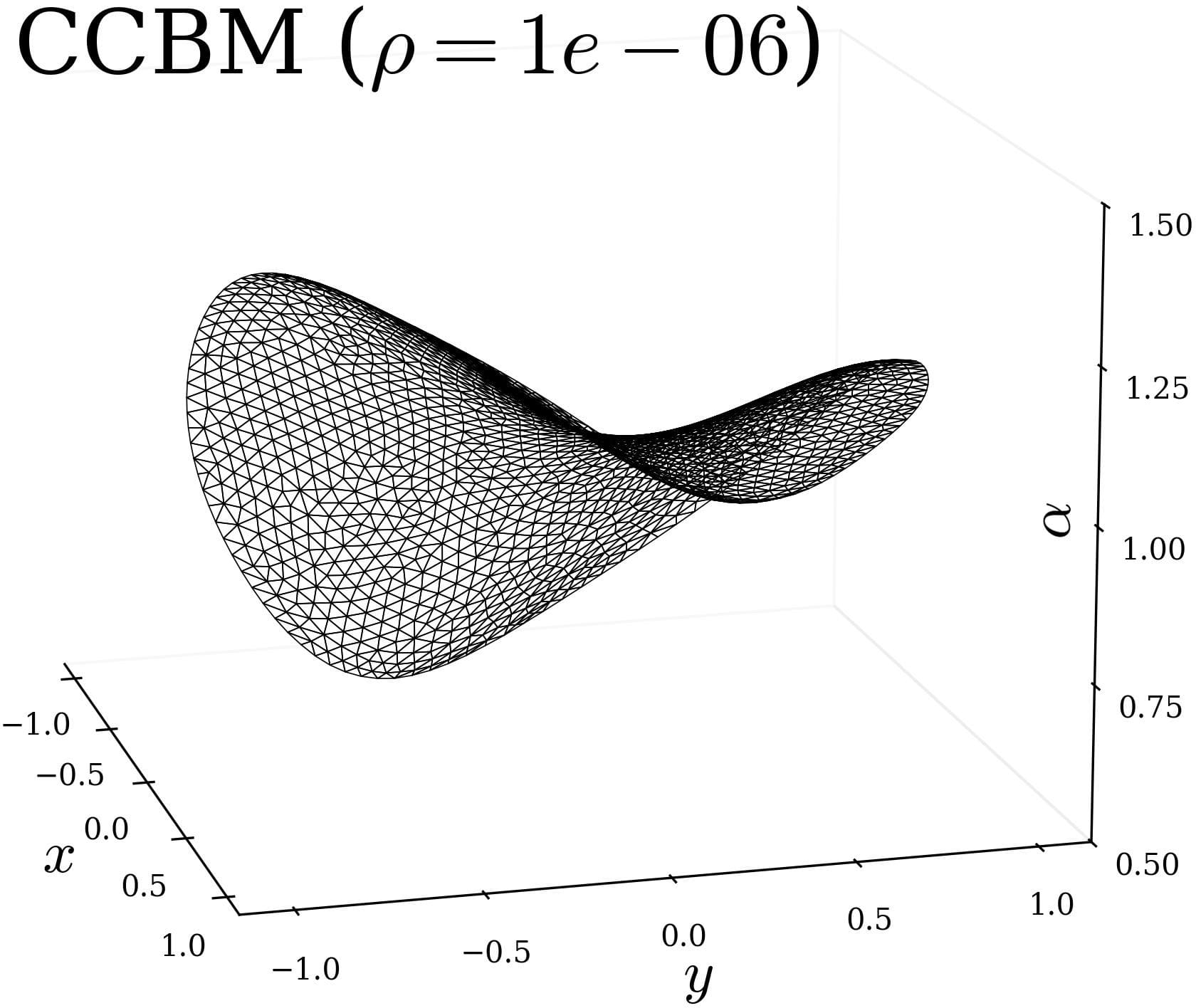}} \ 
\resizebox{0.225\textwidth}{!}{\includegraphics{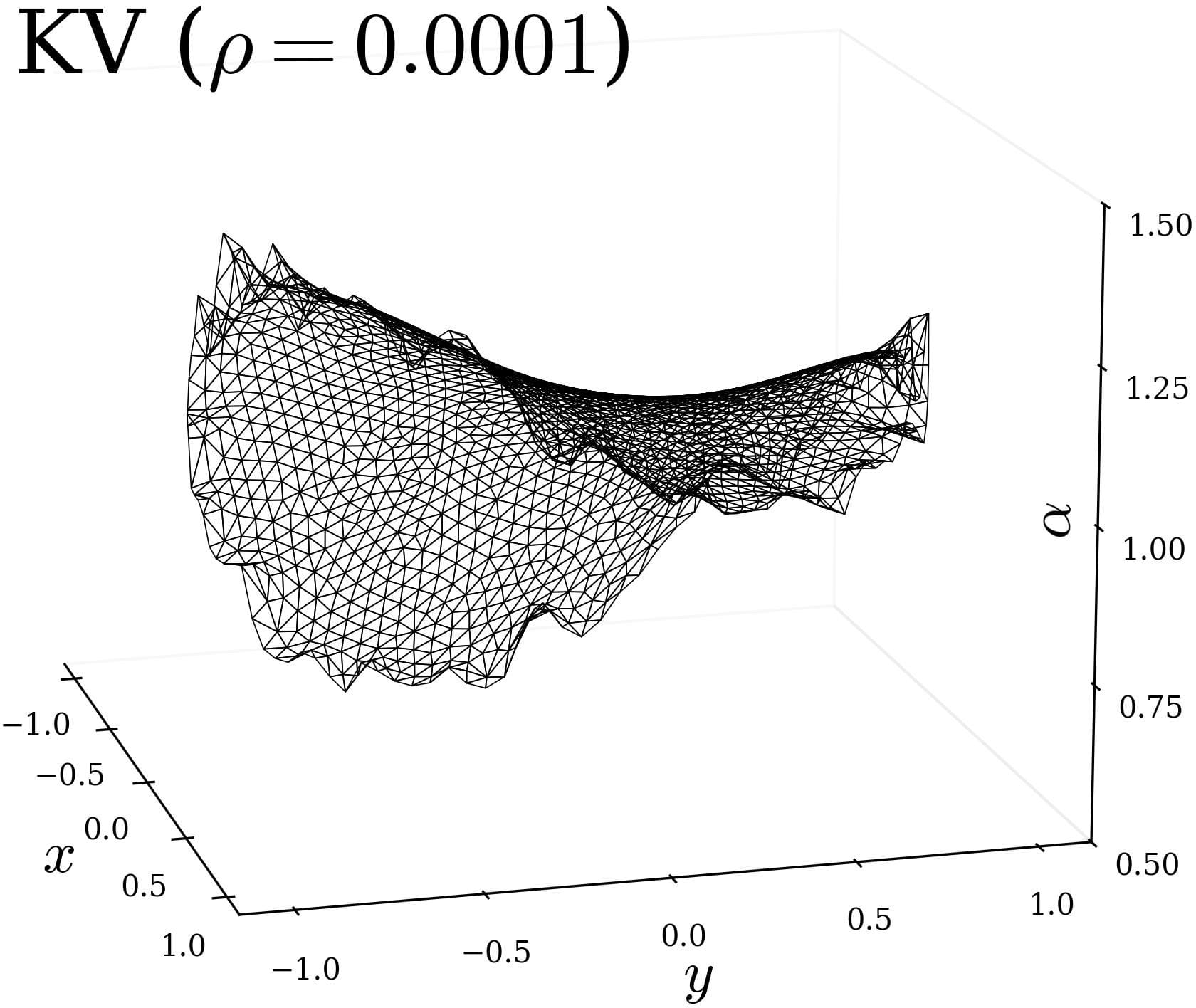}} \ 
\resizebox{0.225\textwidth}{!}{\includegraphics{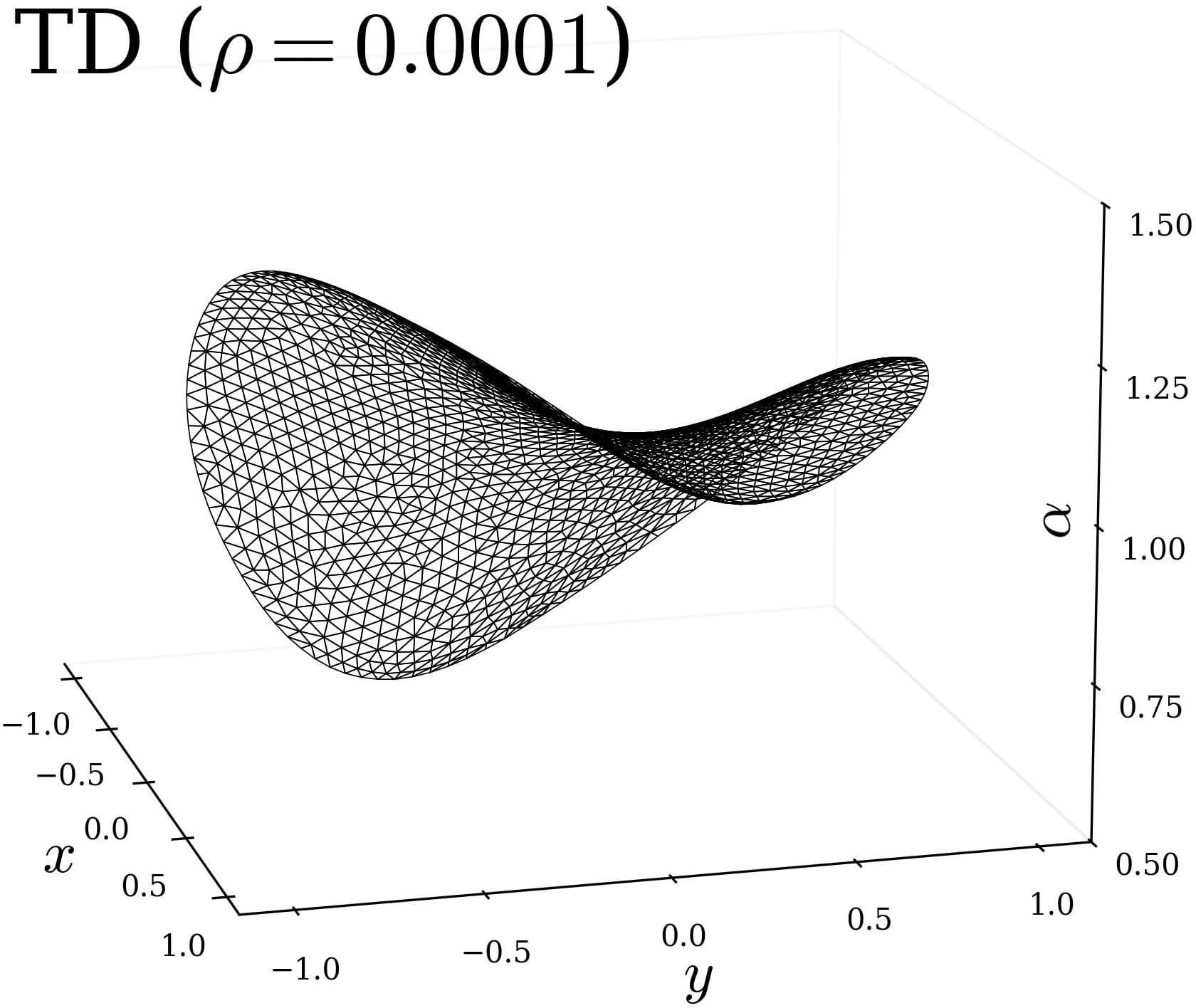}} \ 
\resizebox{0.225\textwidth}{!}{\includegraphics{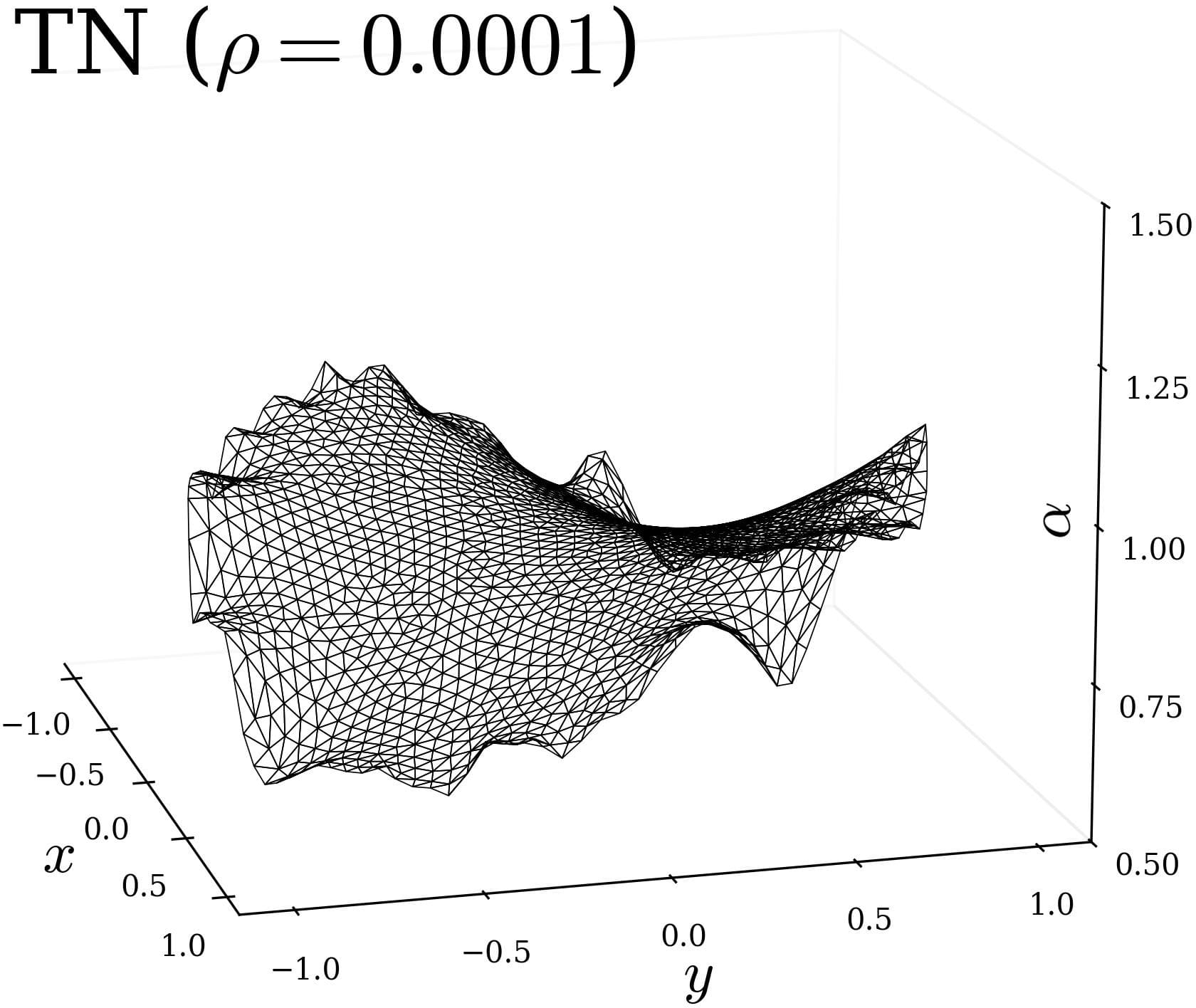}} \\[1em]
\resizebox{0.225\textwidth}{!}{\includegraphics{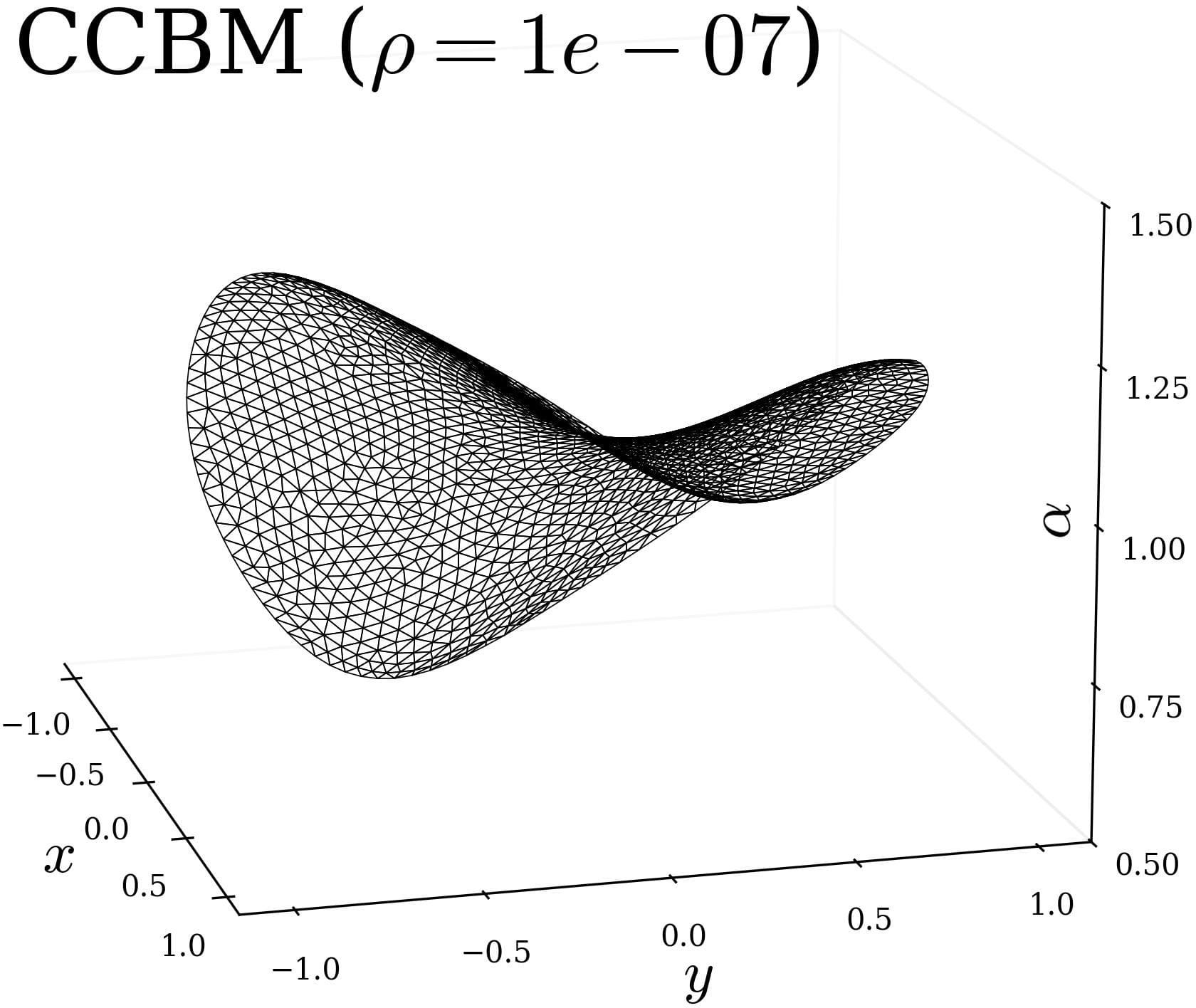}} \ 
\resizebox{0.225\textwidth}{!}{\includegraphics{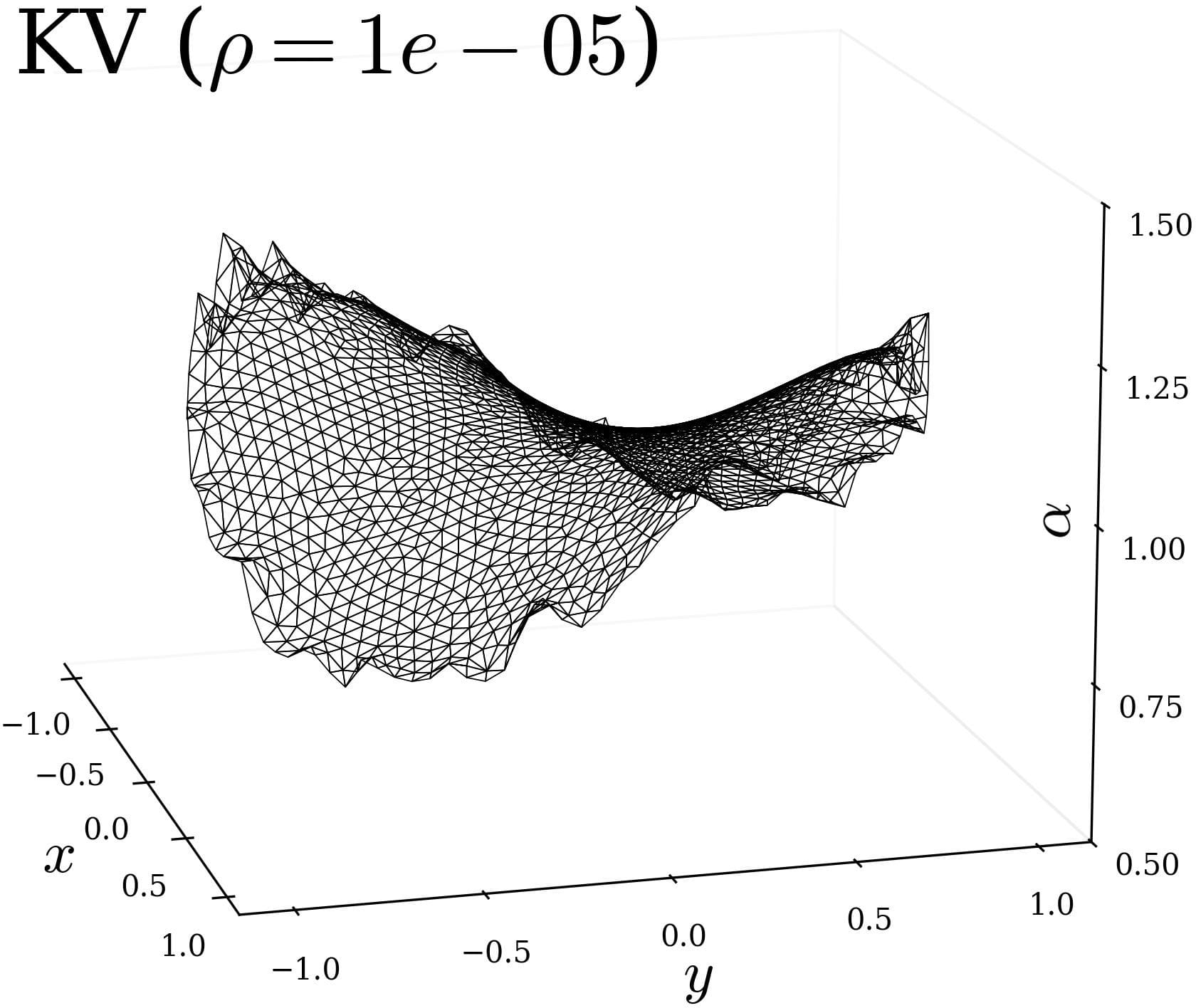}} \ 
\resizebox{0.225\textwidth}{!}{\includegraphics{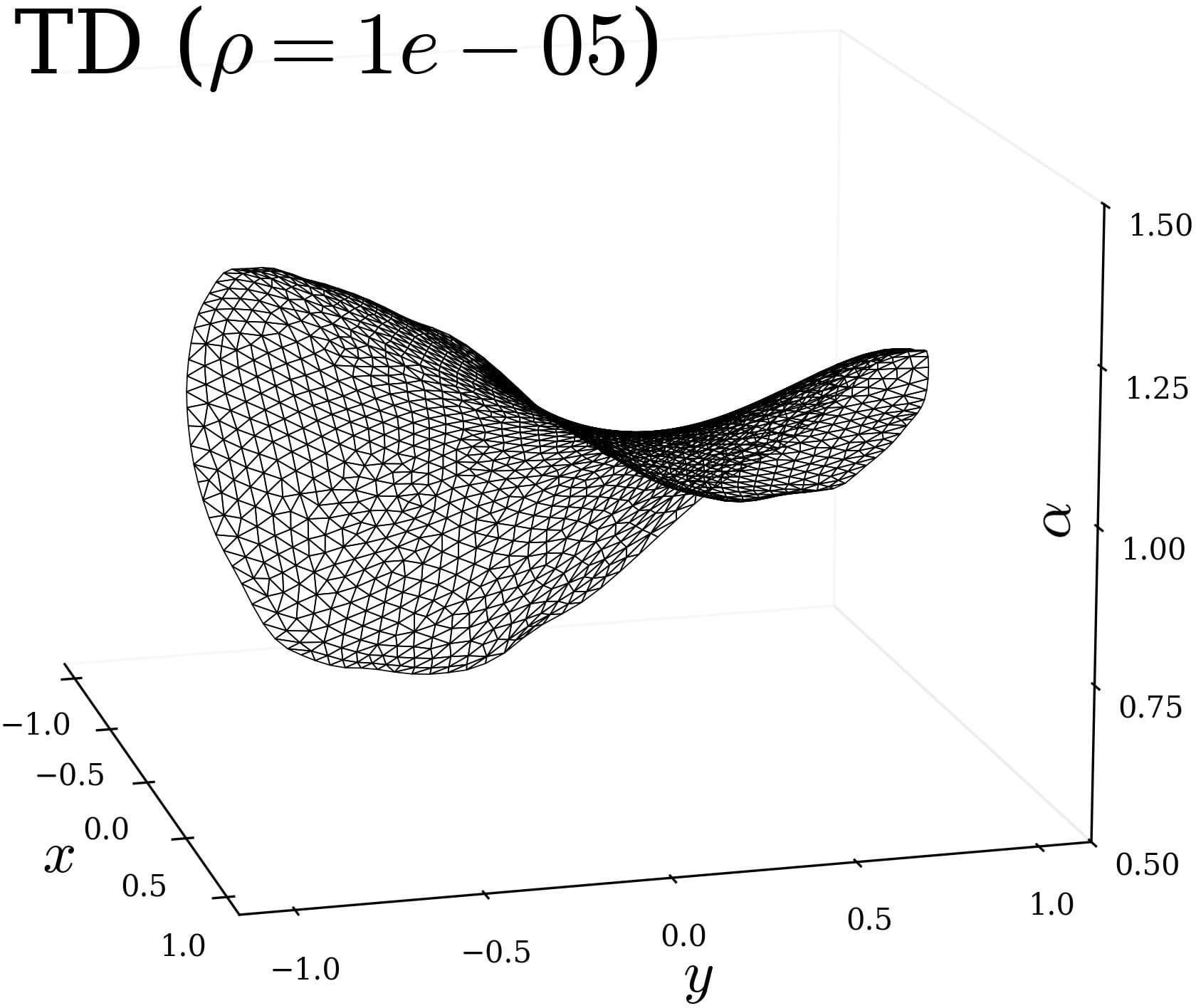}} \ 
\resizebox{0.225\textwidth}{!}{\includegraphics{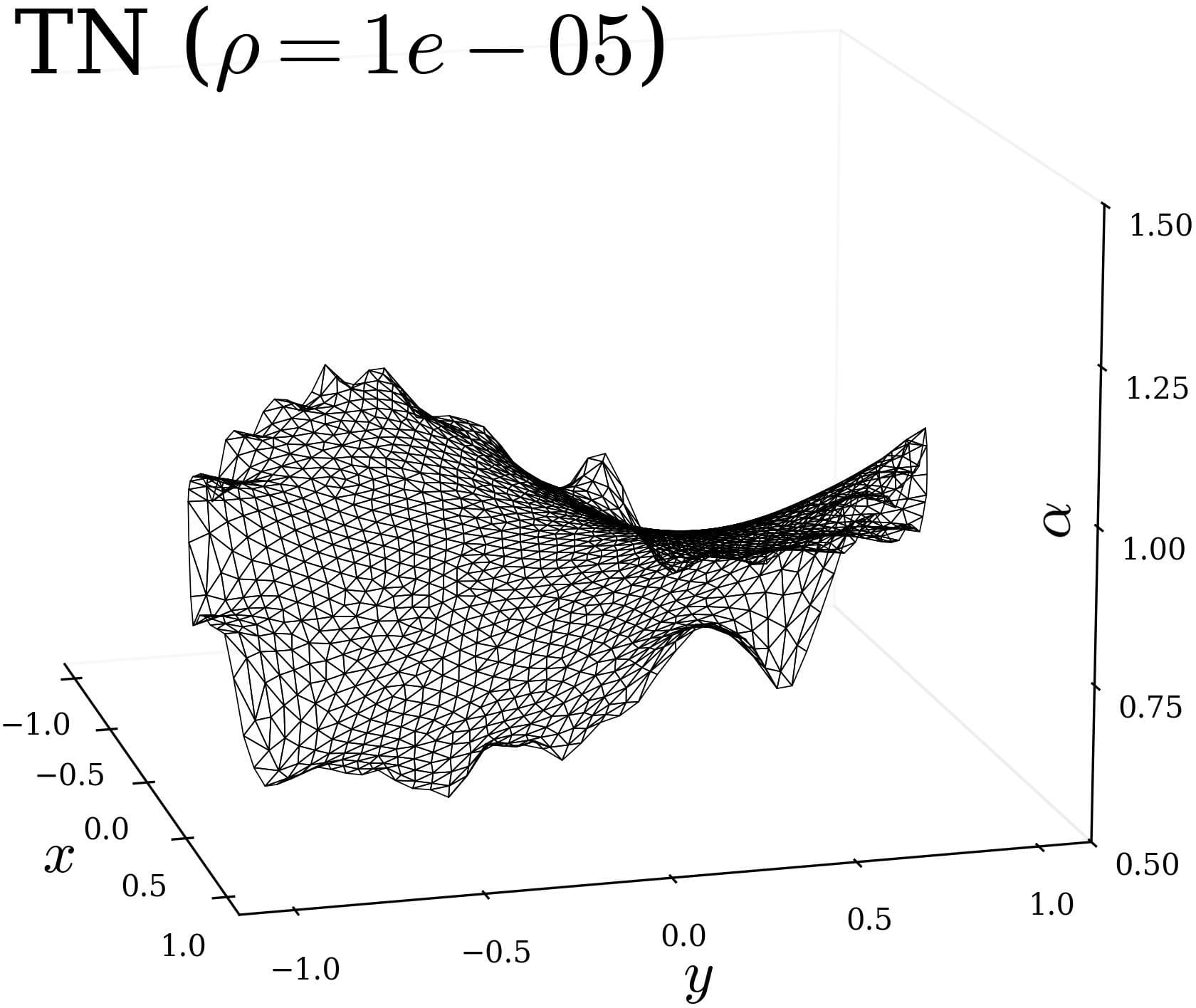}} \\[1em]
\resizebox{0.225\textwidth}{!}{\includegraphics{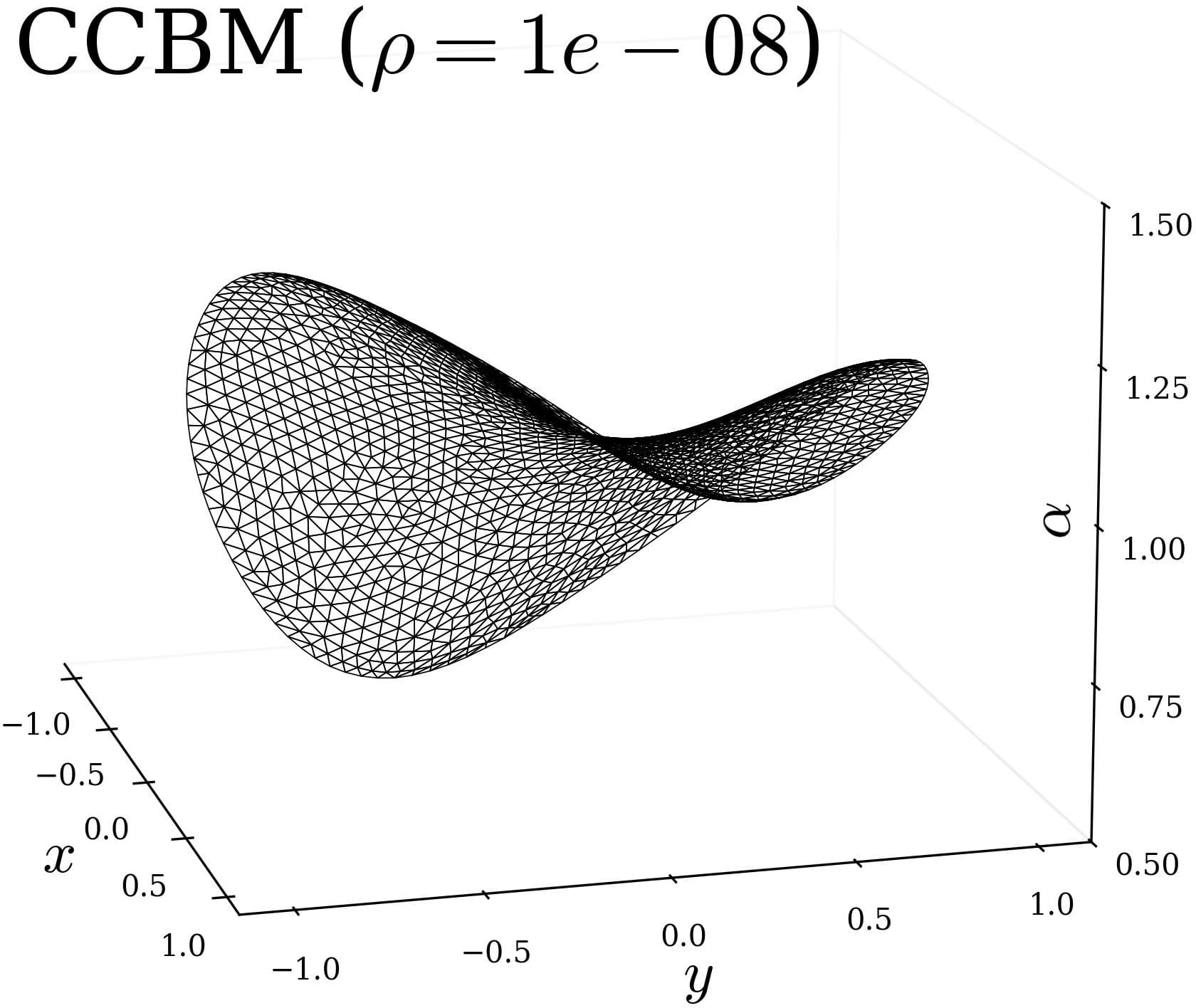}} \ 
\resizebox{0.225\textwidth}{!}{\includegraphics{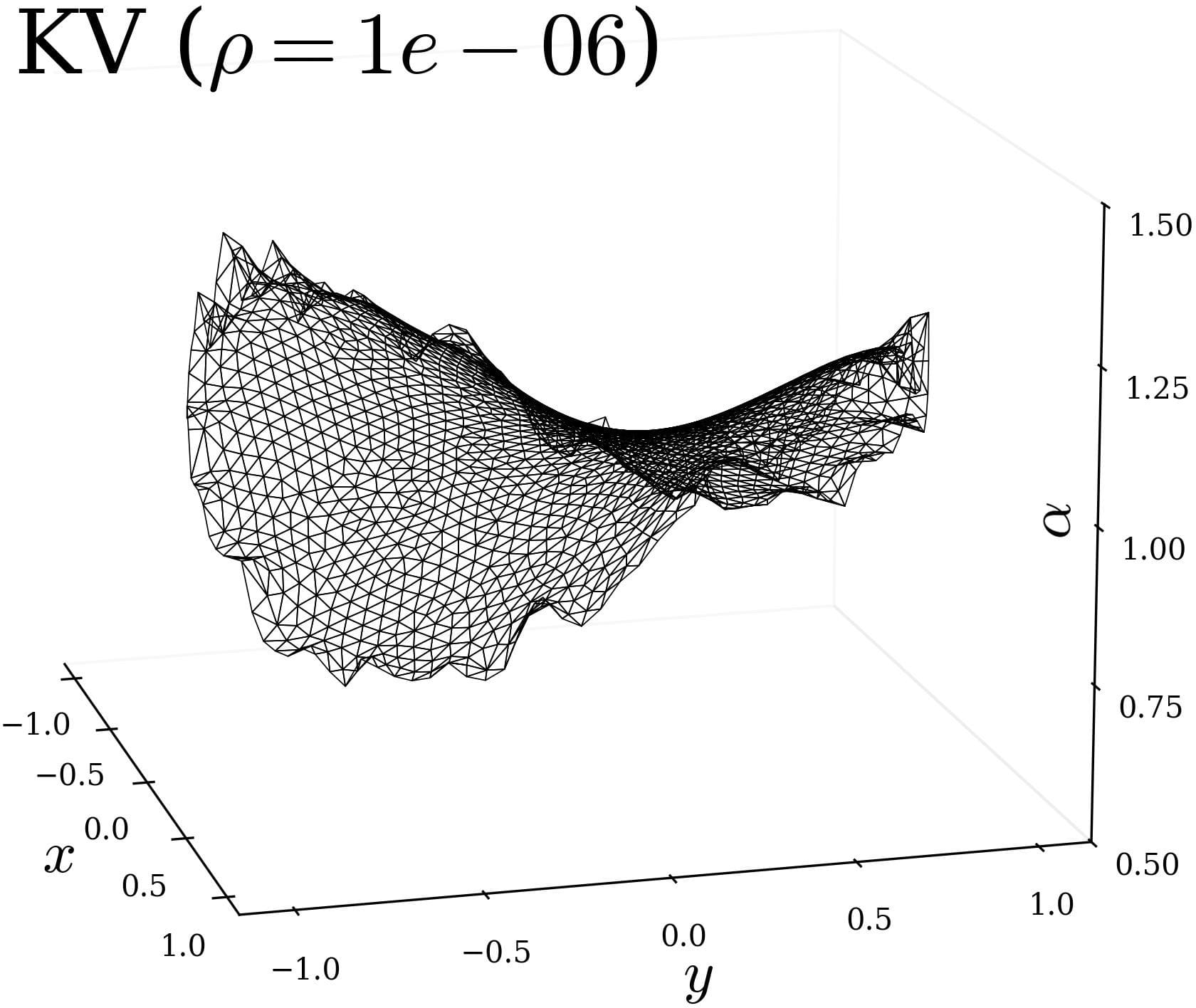}} \
\resizebox{0.225\textwidth}{!}{\includegraphics{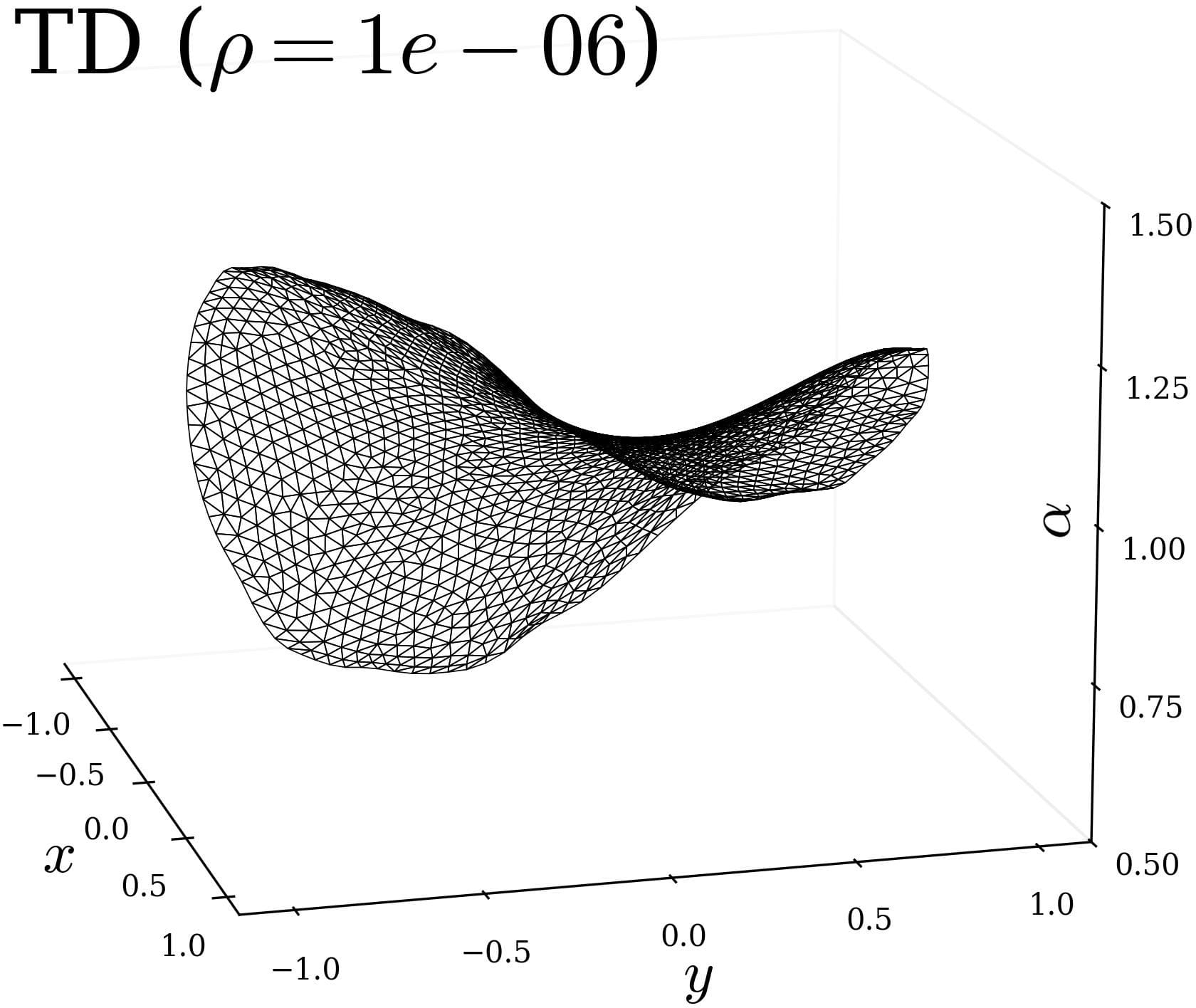}} \
\resizebox{0.225\textwidth}{!}{\includegraphics{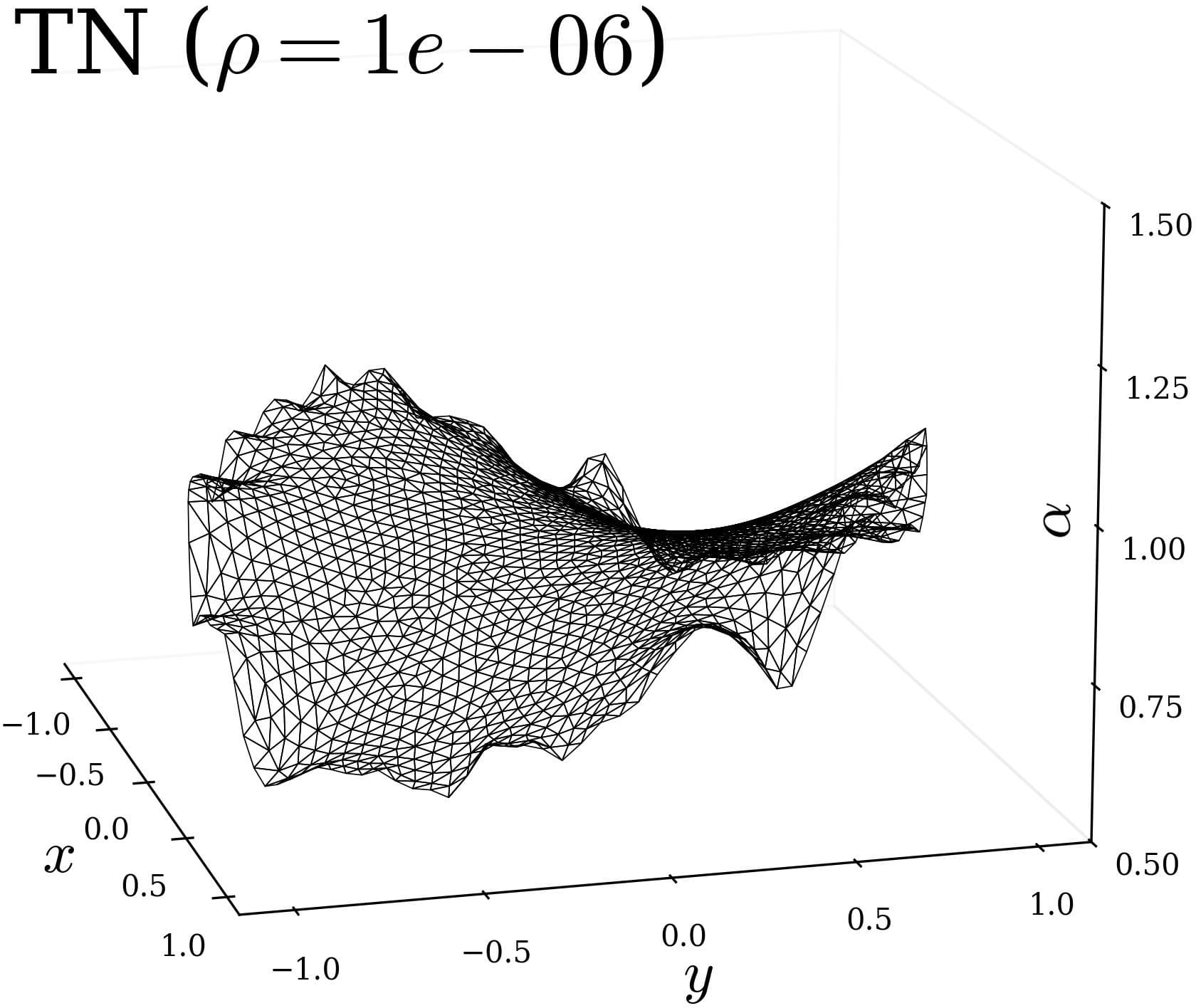}}
\caption{Influence of the Tikhonov parameter $\rho$ on the reconstruction when $\delta = 0.001$, with gradient smoothing ($\mu = 1.0$) and input data $g=\sin(\pi x)\sin(\pi y)$.}
\label{fig:effect_of_input_data_g_trigo}
\end{figure}

\begin{figure}[htp!]
\centering
\resizebox{0.225\textwidth}{!}{\includegraphics{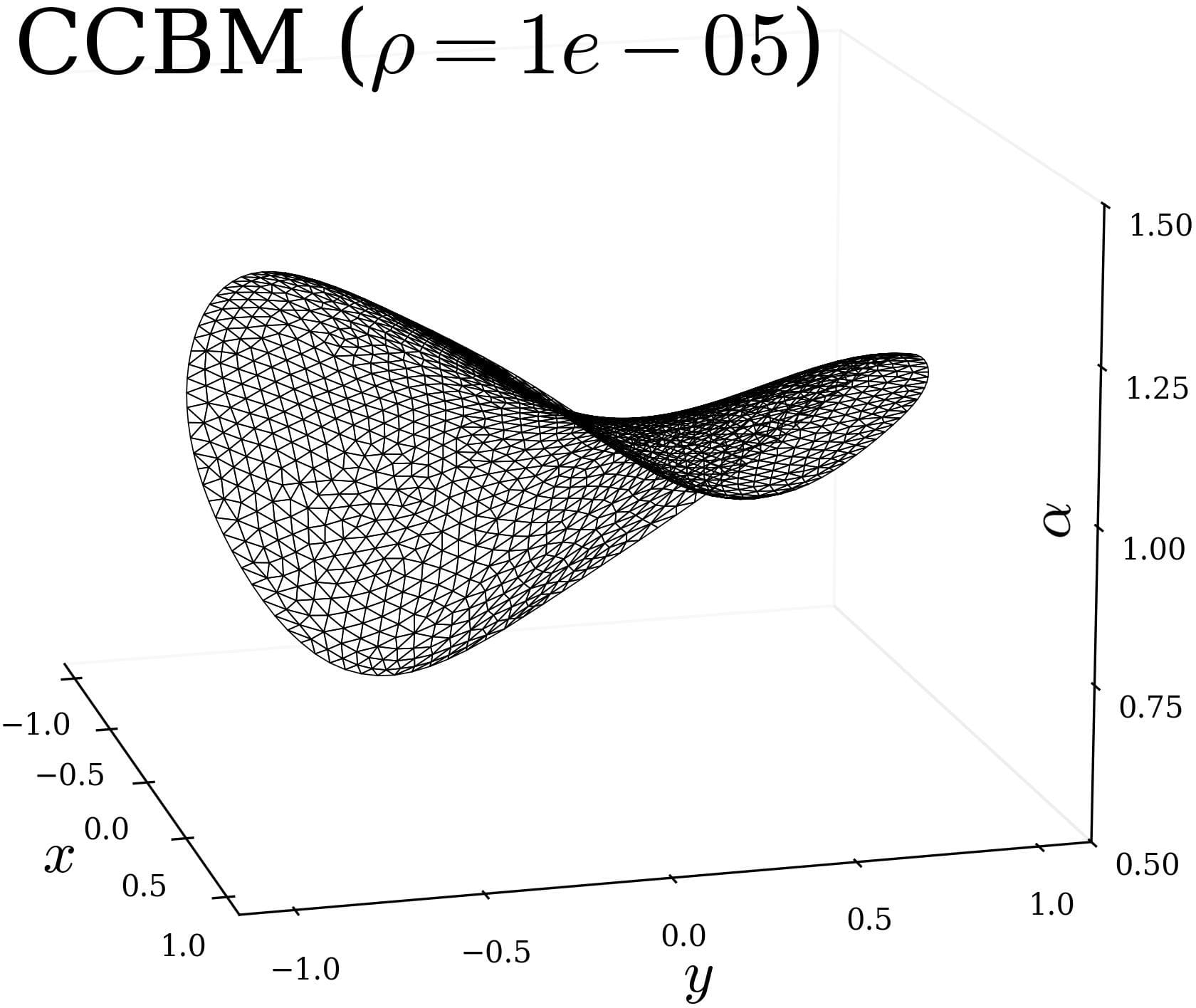}} \ 
\resizebox{0.225\textwidth}{!}{\includegraphics{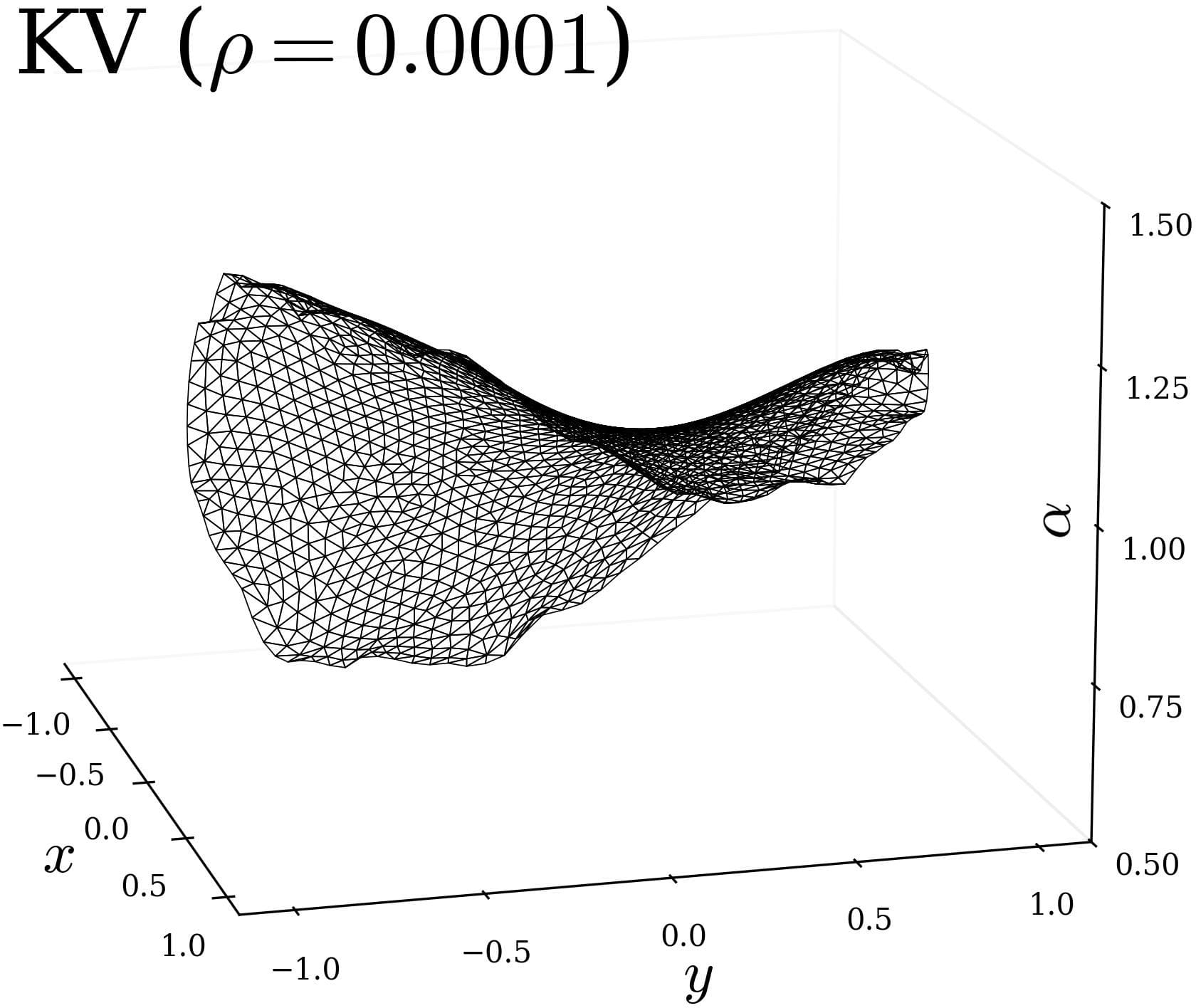}} \ 
\resizebox{0.225\textwidth}{!}{\includegraphics{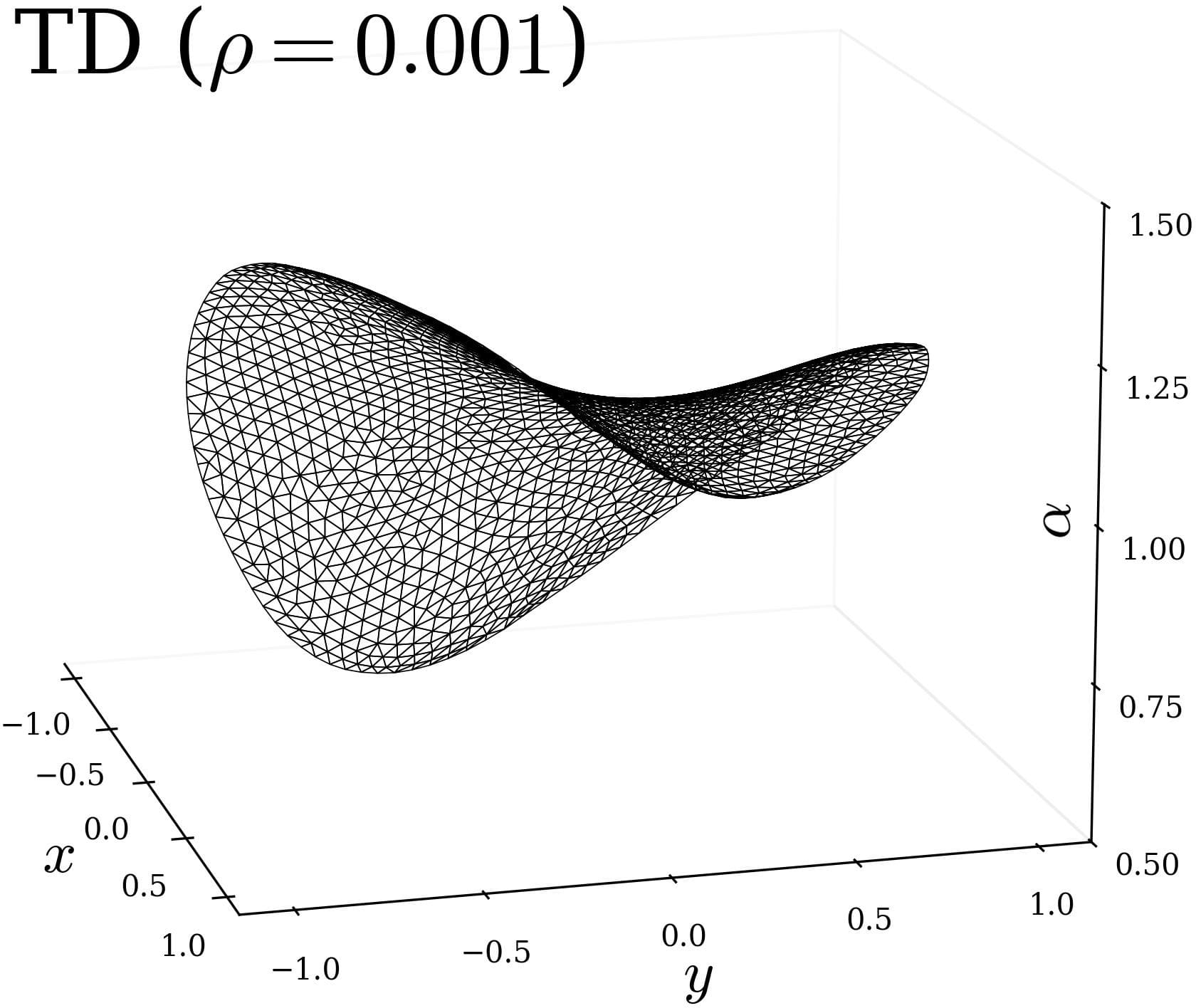}} \ 
\resizebox{0.225\textwidth}{!}{\includegraphics{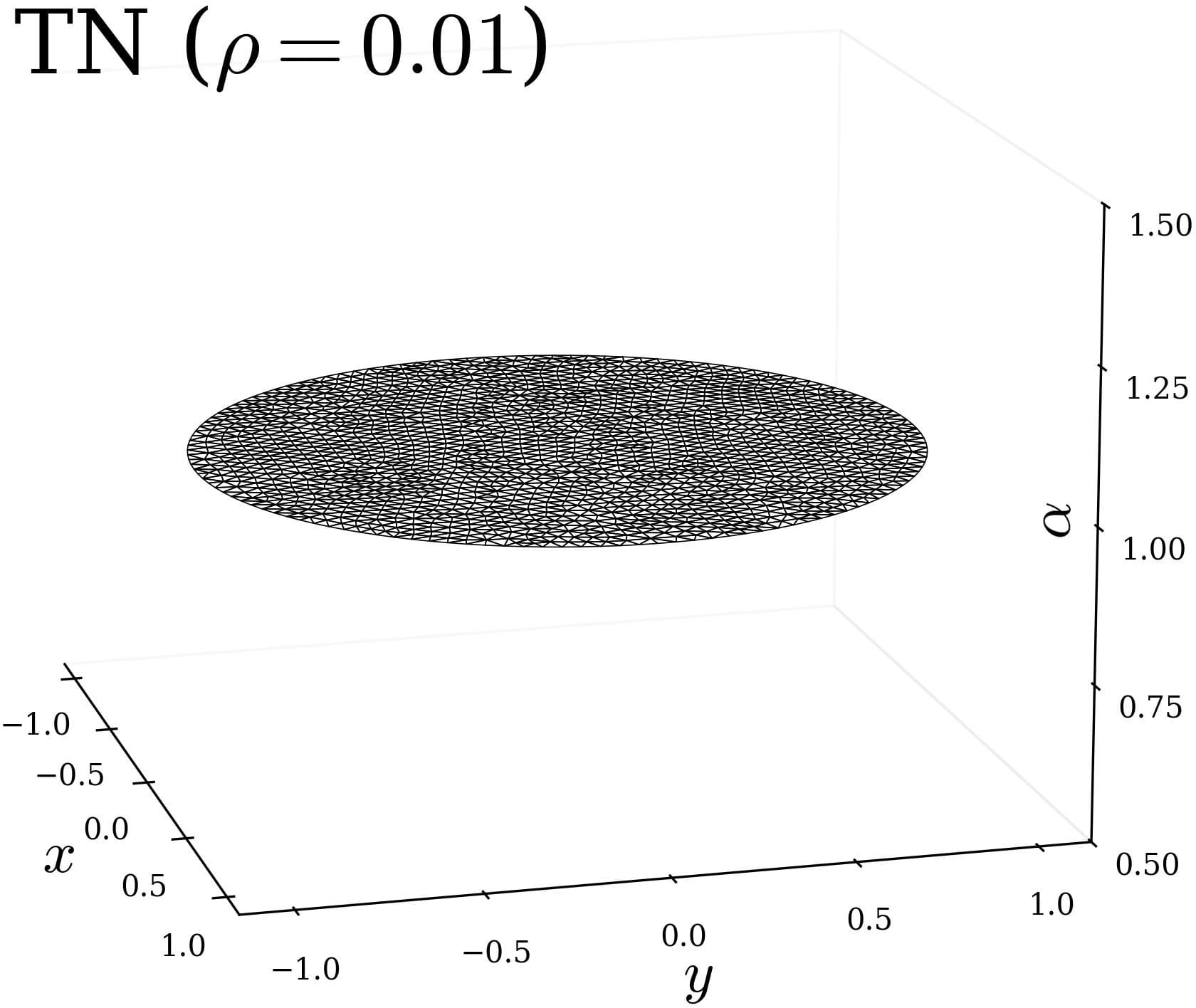}} \\[1em]
\resizebox{0.225\textwidth}{!}{\includegraphics{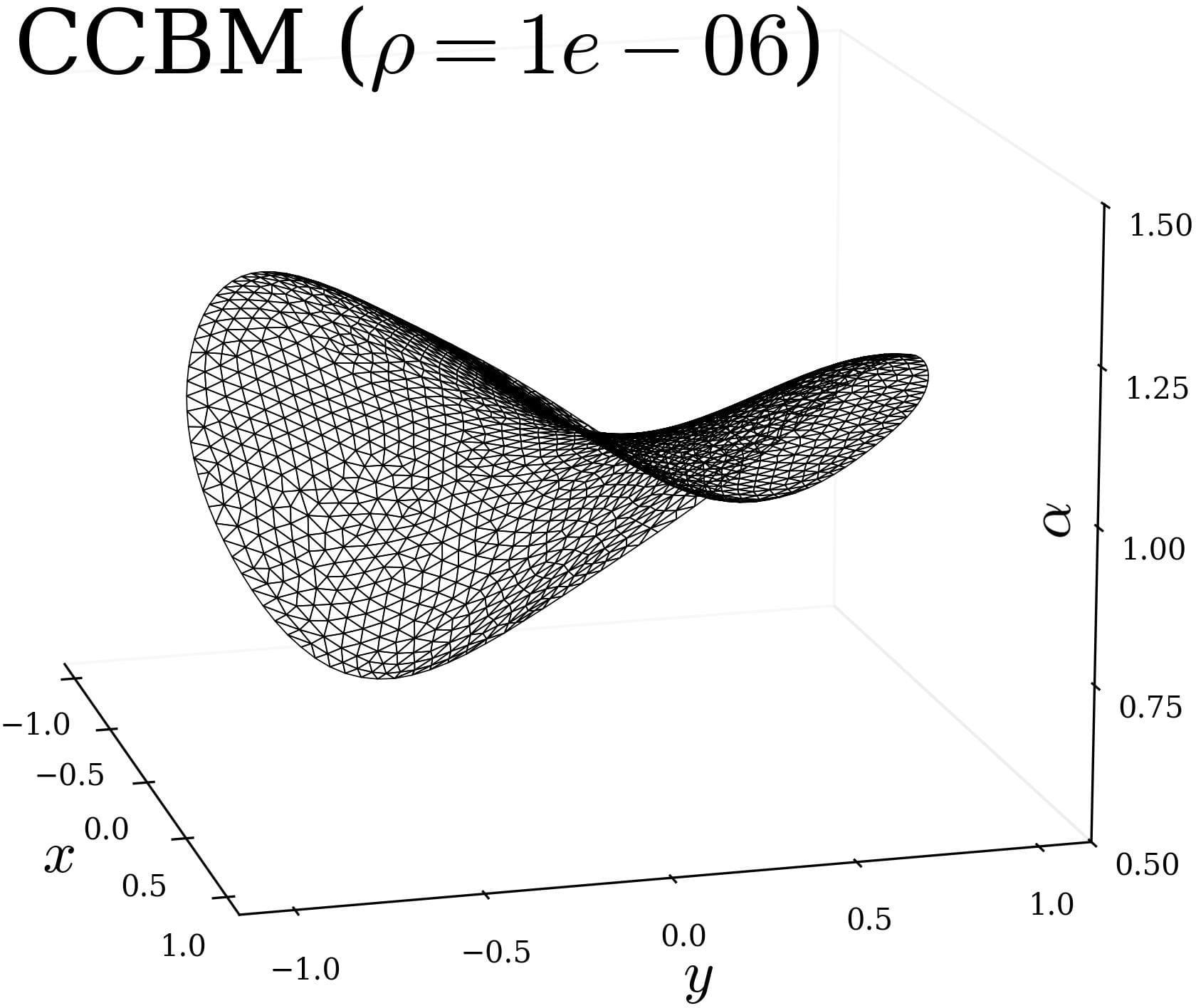}} \ 
\resizebox{0.225\textwidth}{!}{\includegraphics{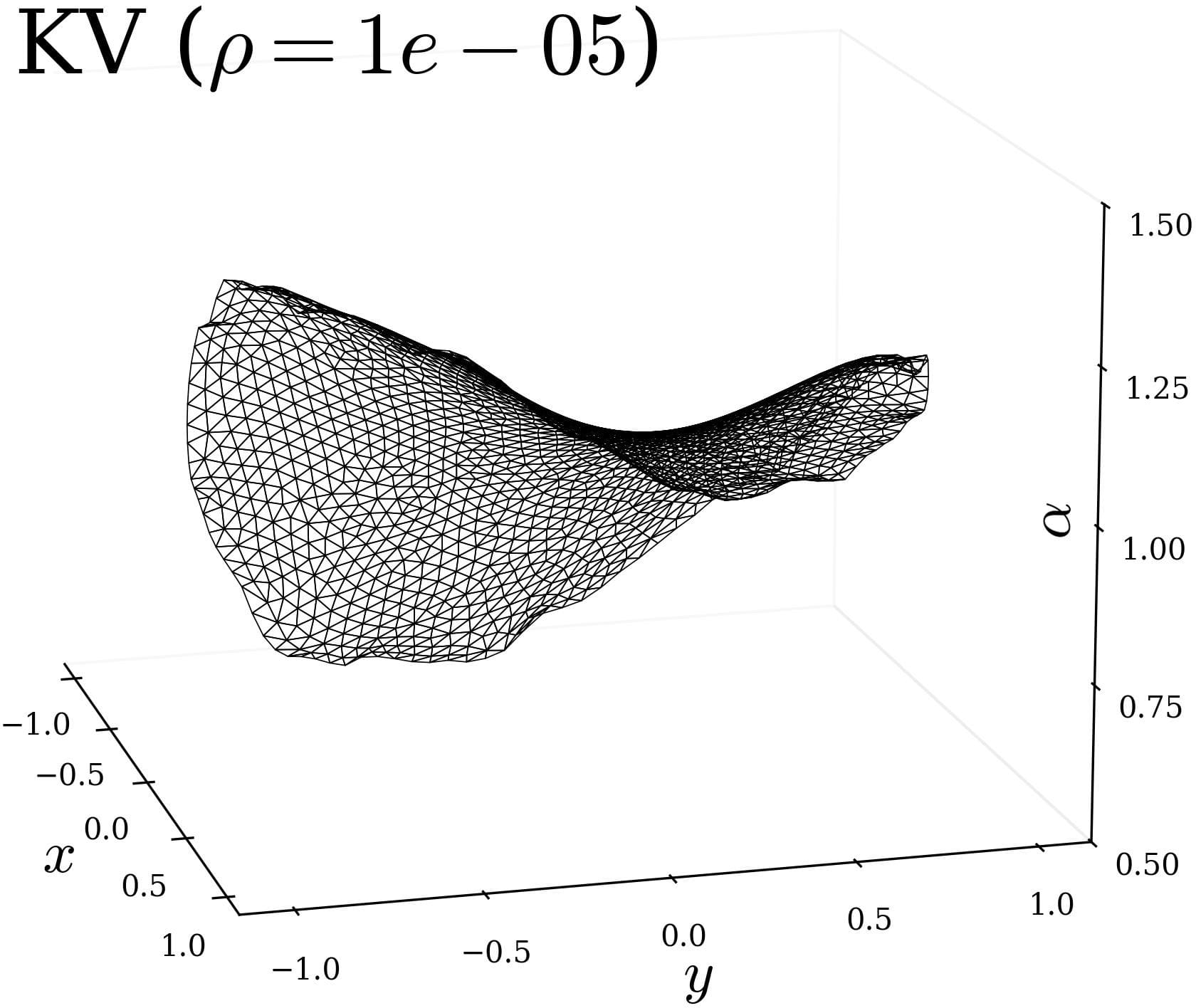}} \ 
\resizebox{0.225\textwidth}{!}{\includegraphics{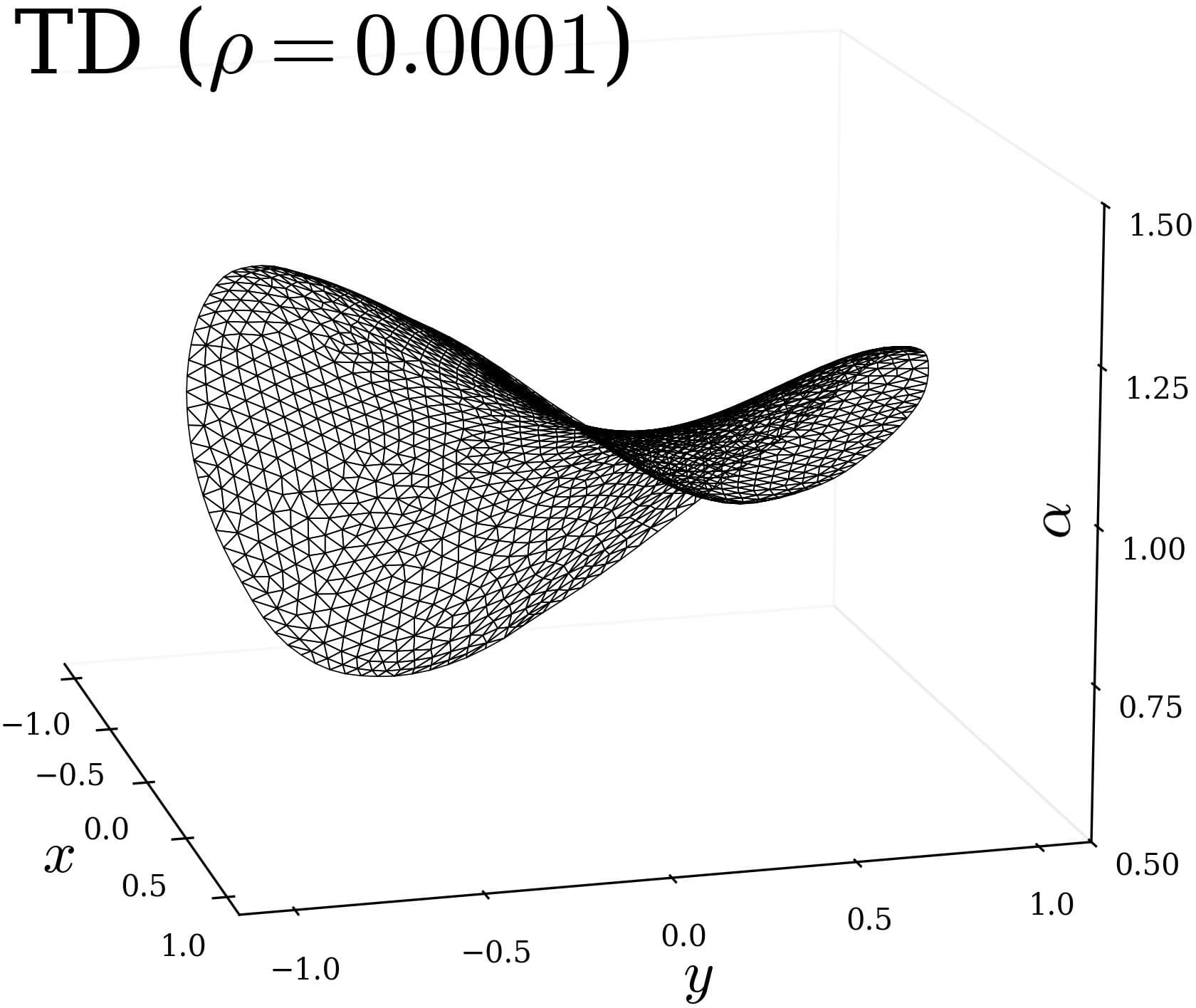}} \ 
\resizebox{0.225\textwidth}{!}{\includegraphics{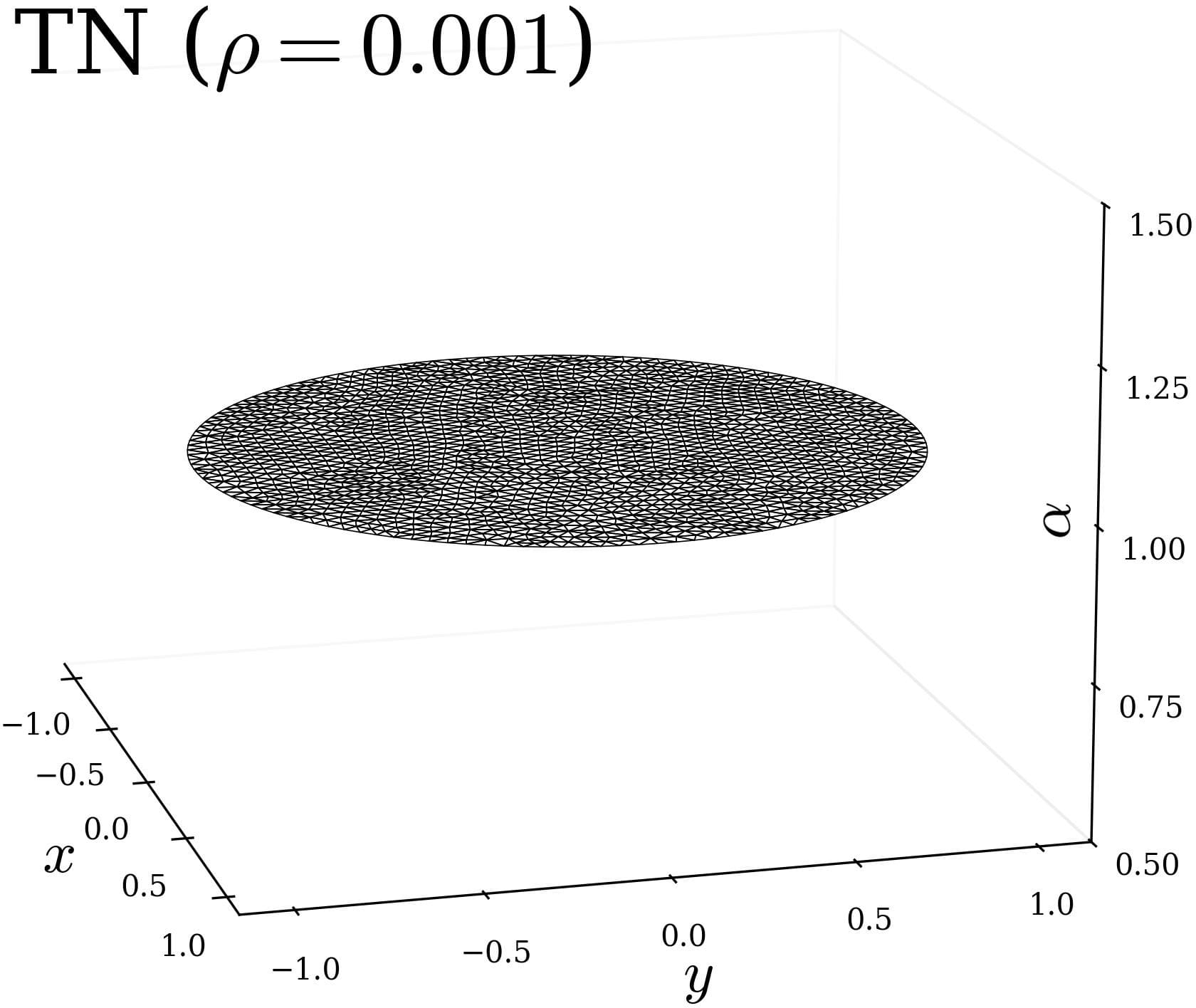}} \\[1em]
\resizebox{0.225\textwidth}{!}{\includegraphics{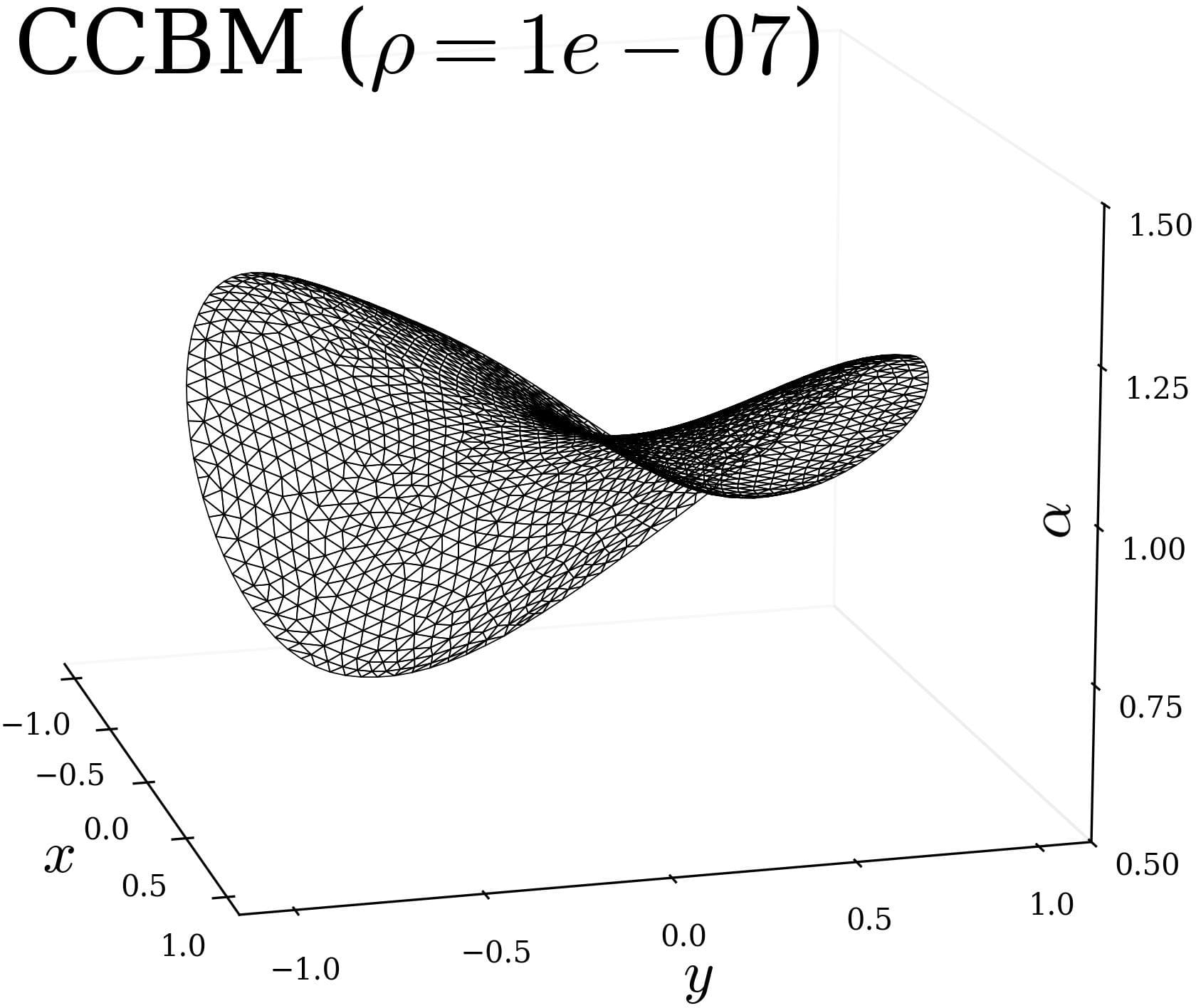}} \ 
\resizebox{0.225\textwidth}{!}{\includegraphics{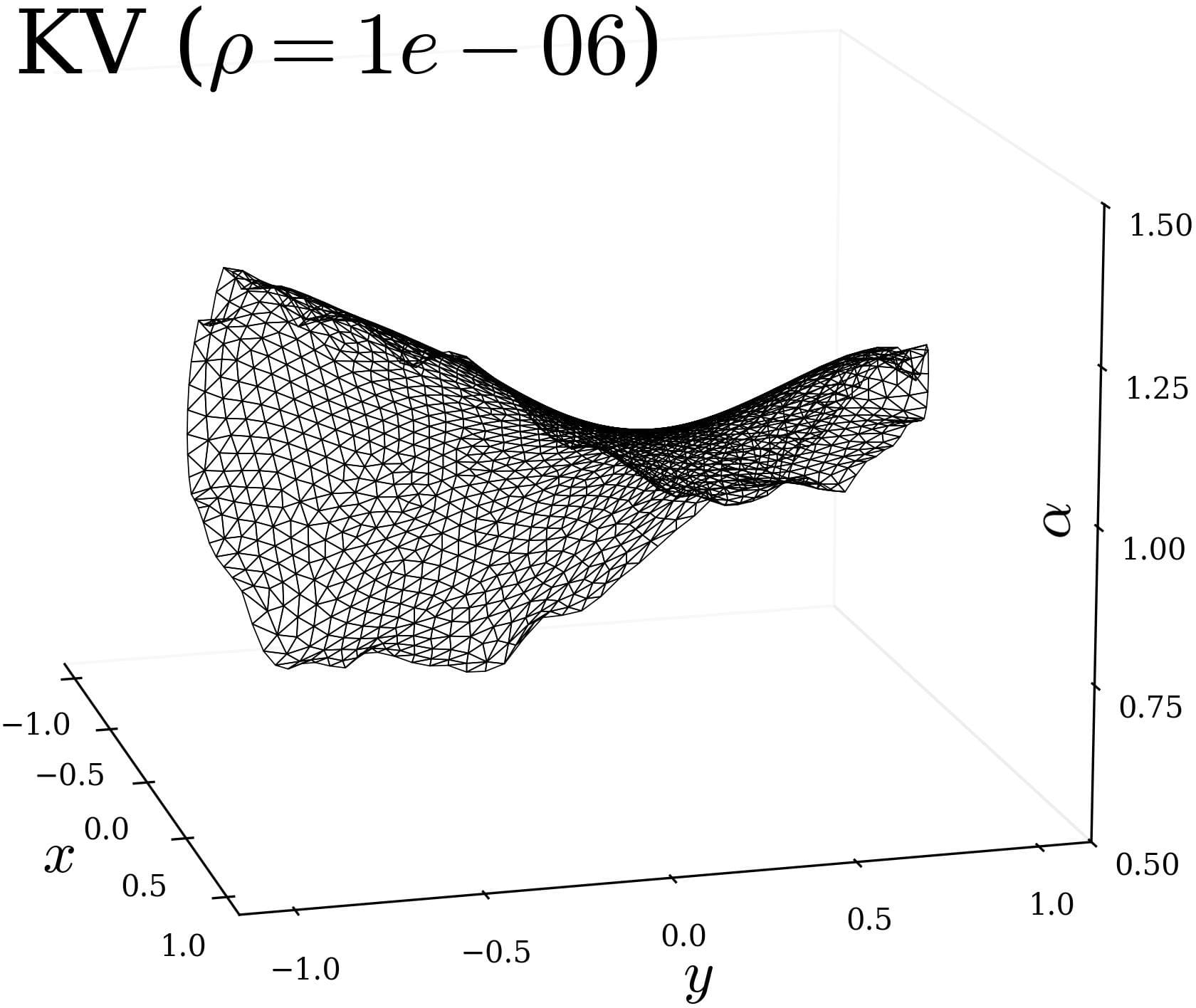}} \ 
\resizebox{0.225\textwidth}{!}{\includegraphics{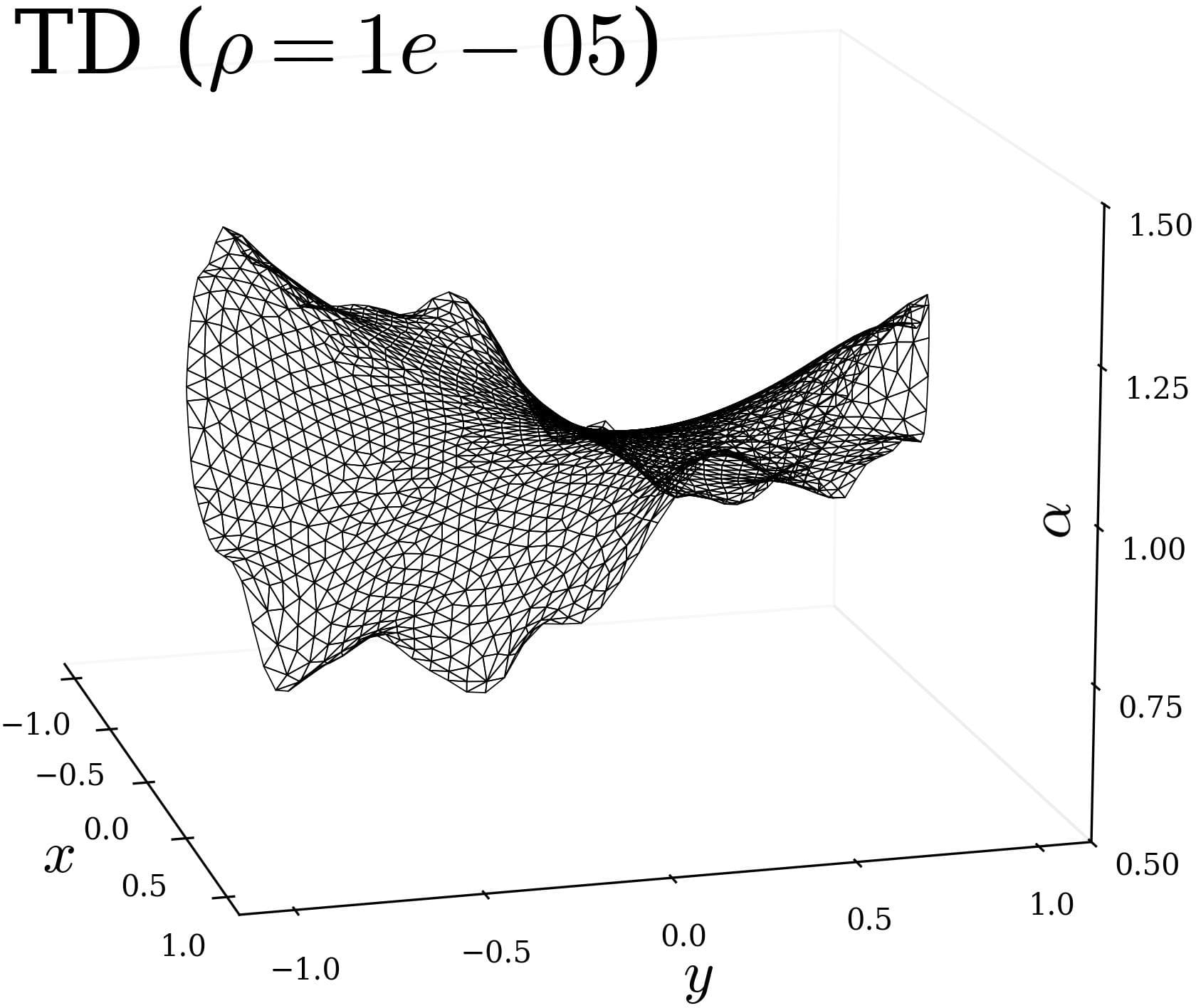}} \ 
\resizebox{0.225\textwidth}{!}{\includegraphics{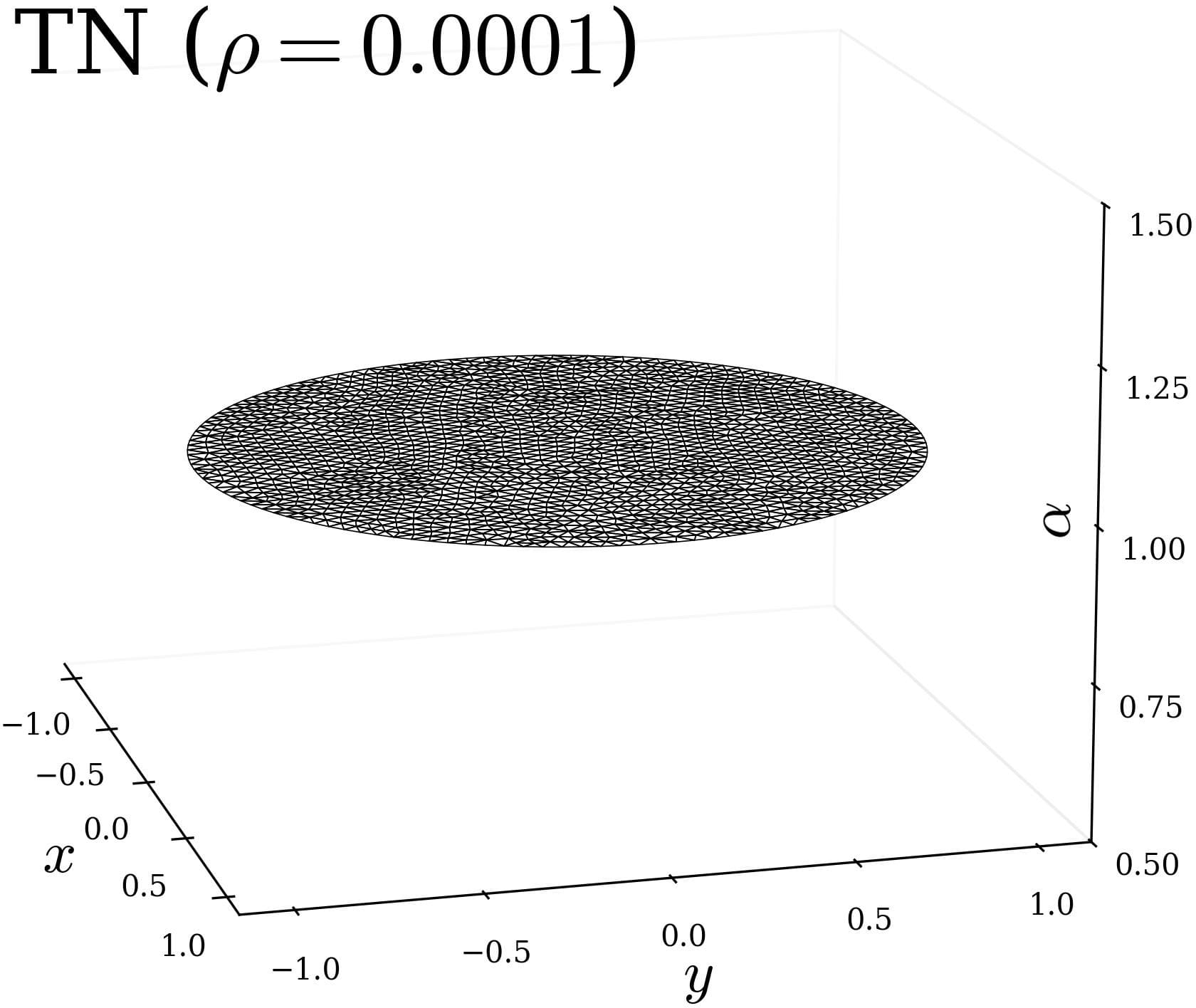}} \\[1em]
\resizebox{0.225\textwidth}{!}{\includegraphics{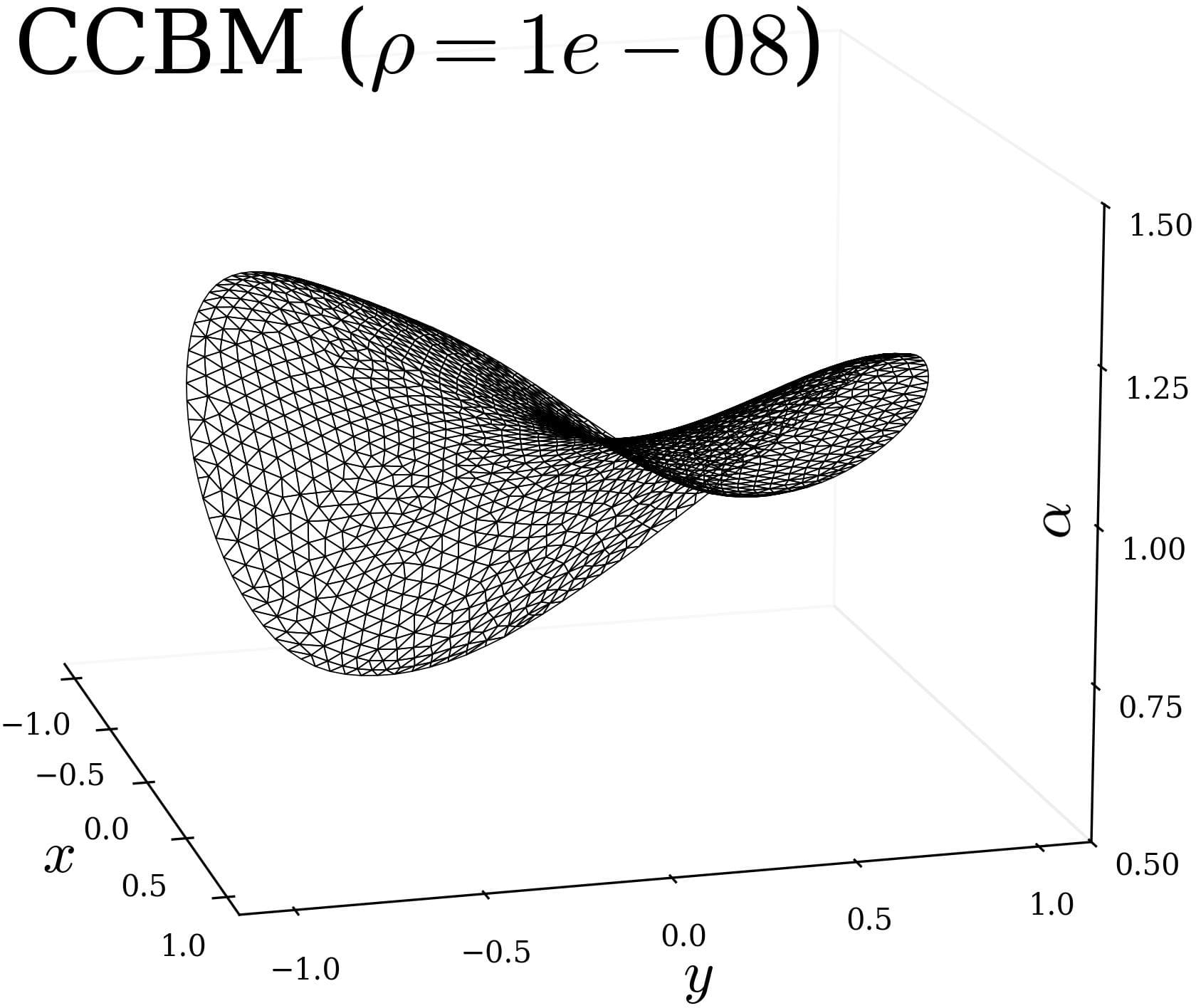}} \ 
\resizebox{0.225\textwidth}{!}{\includegraphics{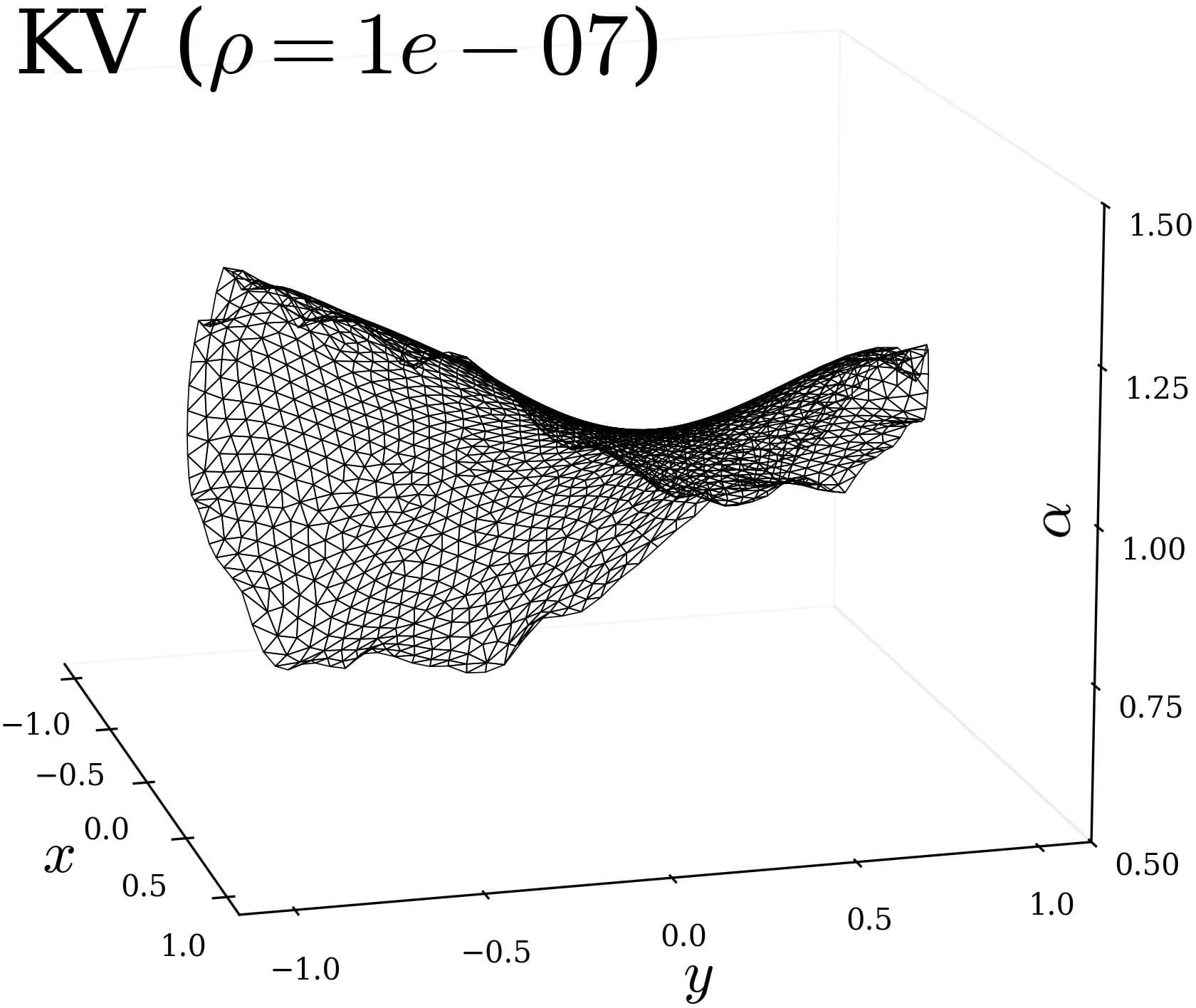}} \
\resizebox{0.225\textwidth}{!}{\includegraphics{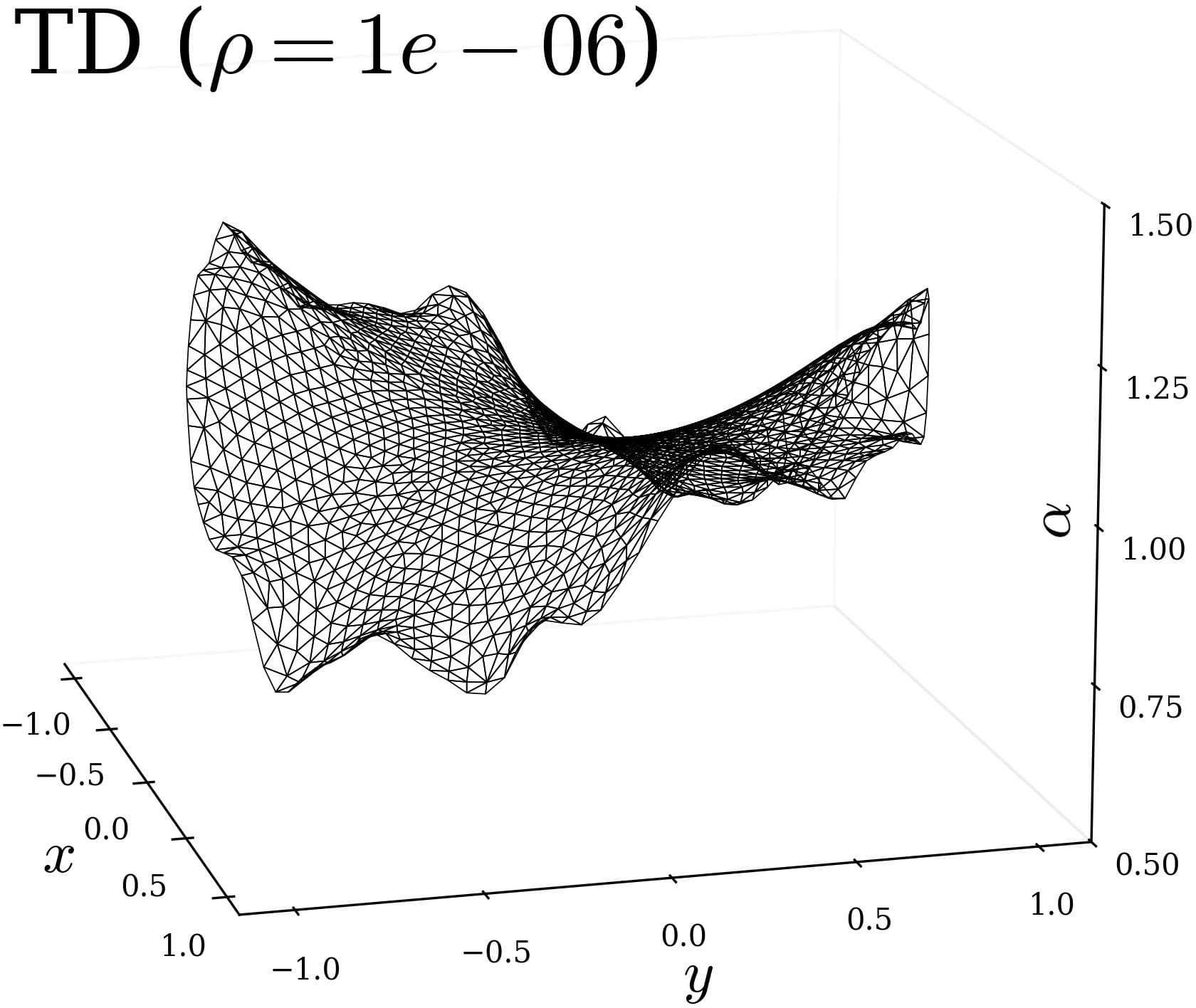}} \
\resizebox{0.225\textwidth}{!}{\includegraphics{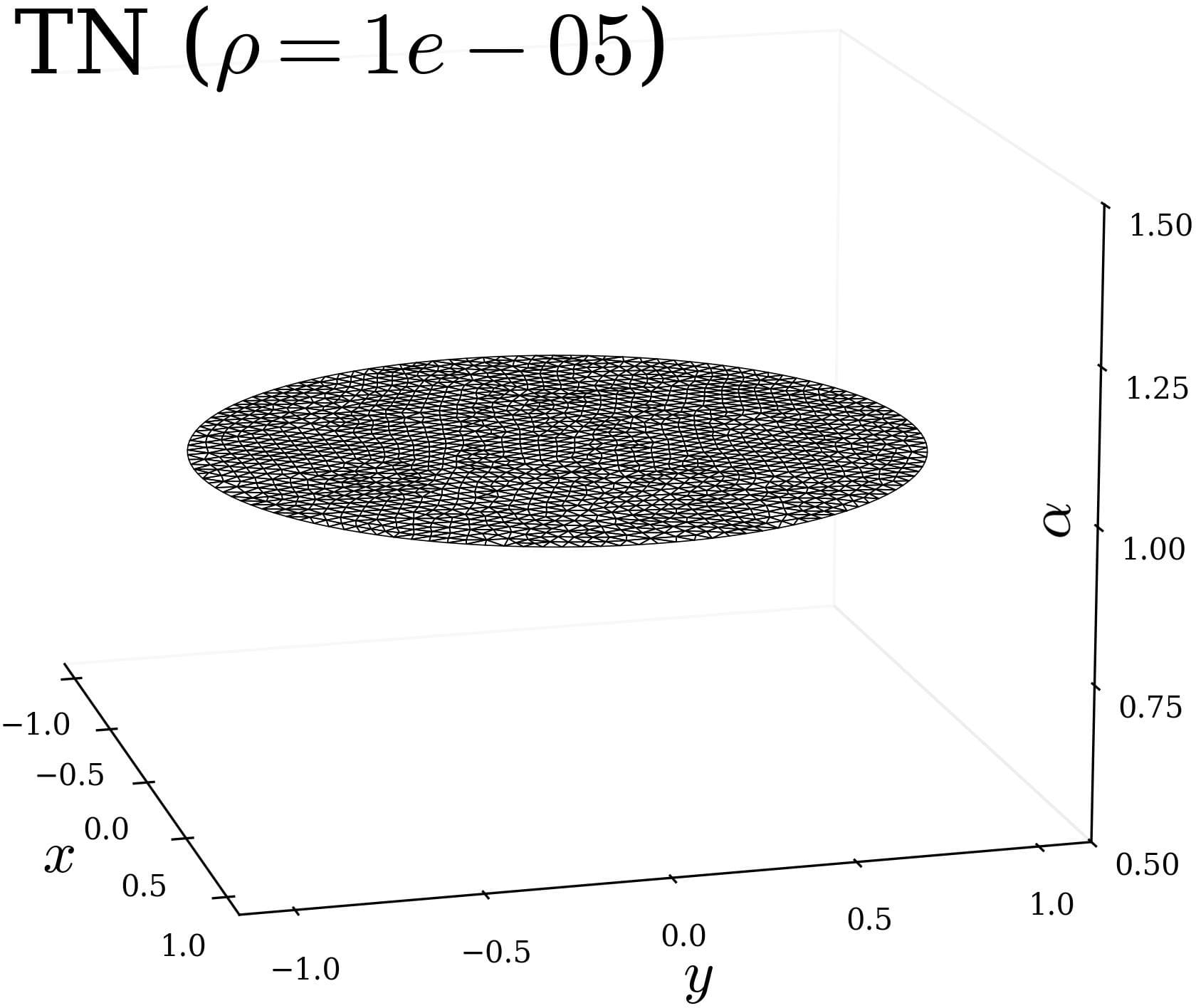}}
\caption{Influence of the Tikhonov parameter $\rho$ on the reconstruction when $\delta = 0.005$, with gradient smoothing ($\mu = 1.0$) and input data $g=\sin(\pi x)\sin(\pi y)$.}
\label{fig:effect_of_input_data_g_trigo_with_noise_0.005}
\end{figure}

\begin{figure}[htp!]
\resizebox{0.28\textwidth}{!}{\includegraphics{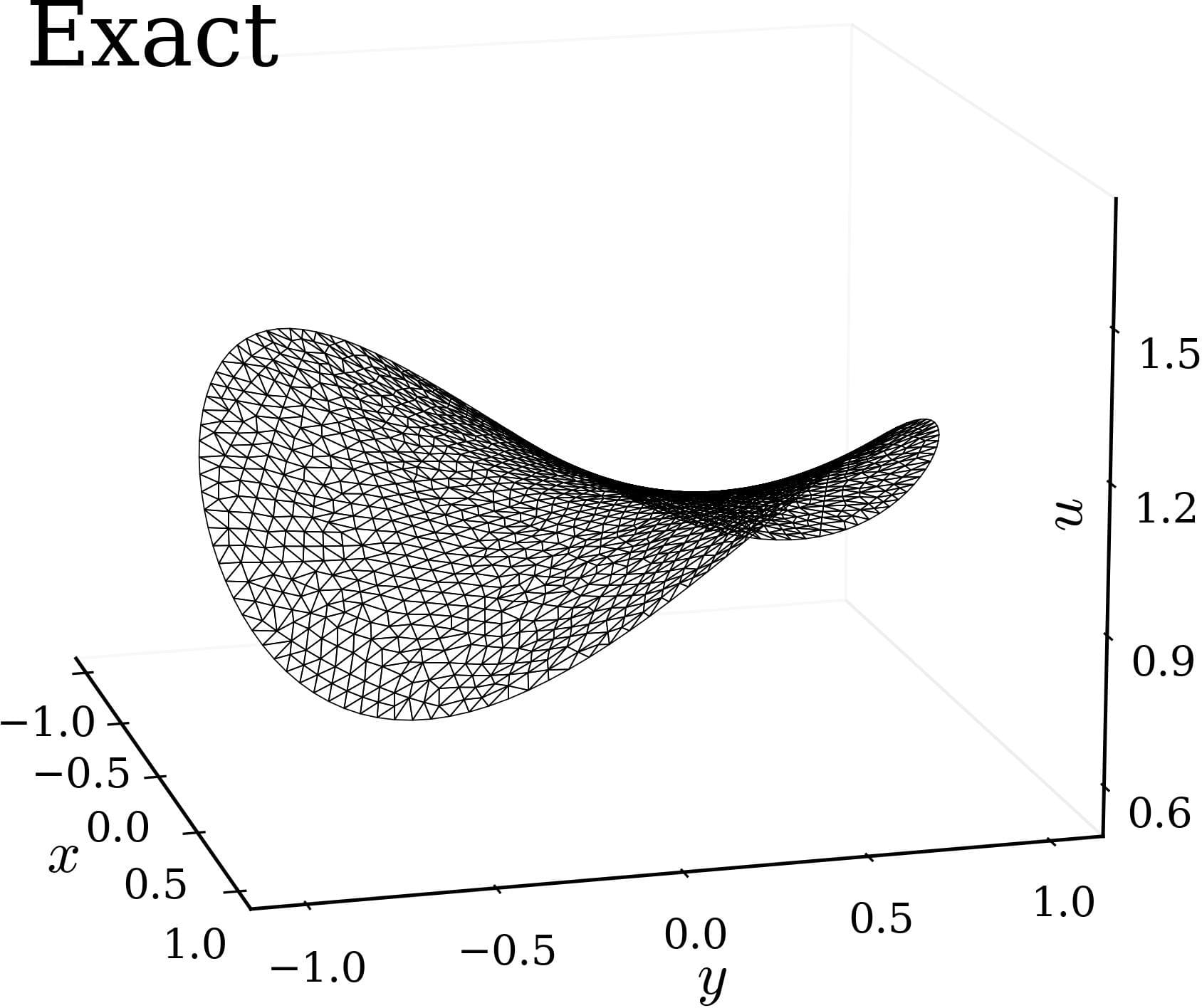}} \quad
\resizebox{0.225\textwidth}{!}{\includegraphics{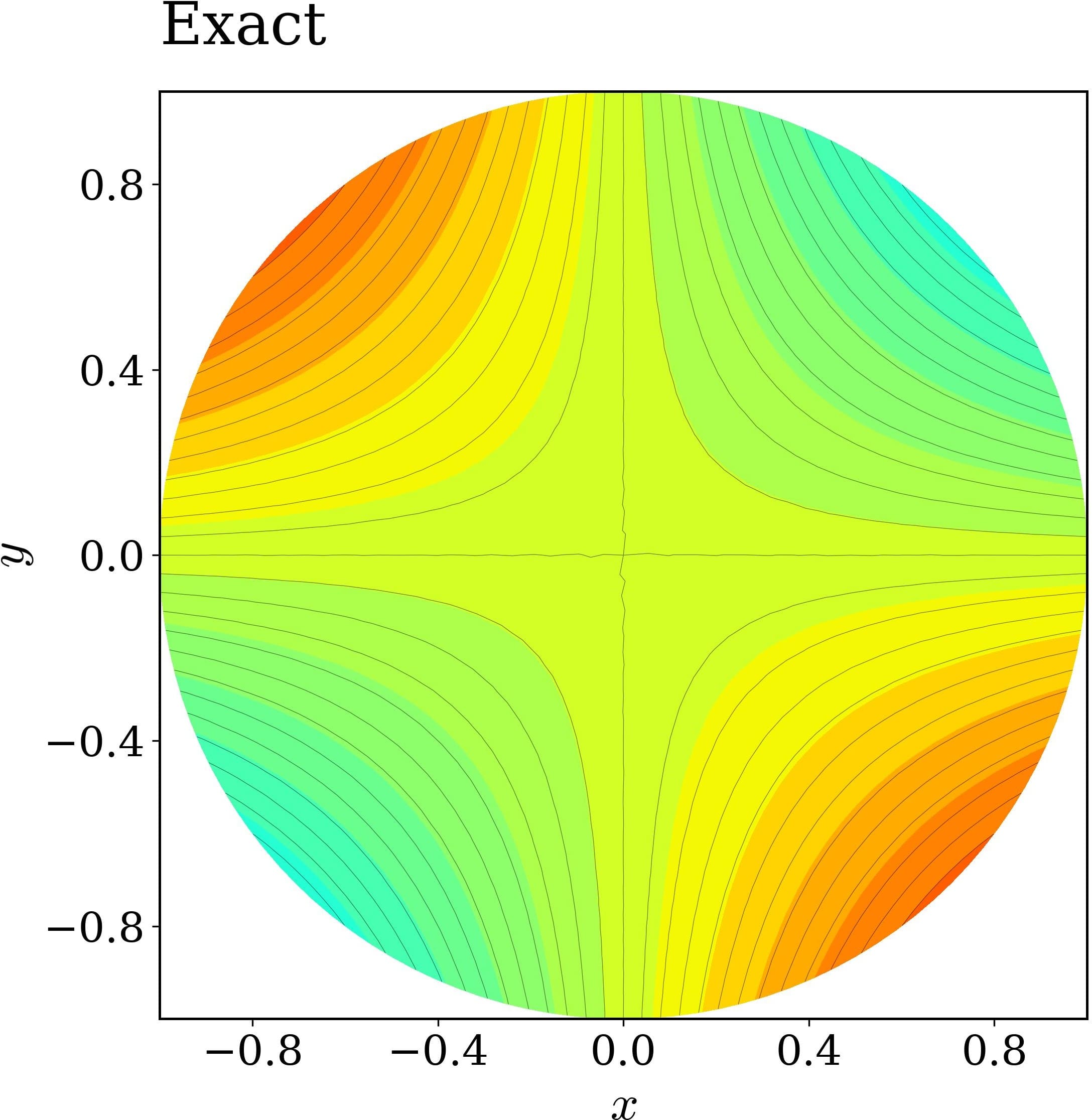}} \\[1em]
\centering
\resizebox{0.225\textwidth}{!}{\includegraphics{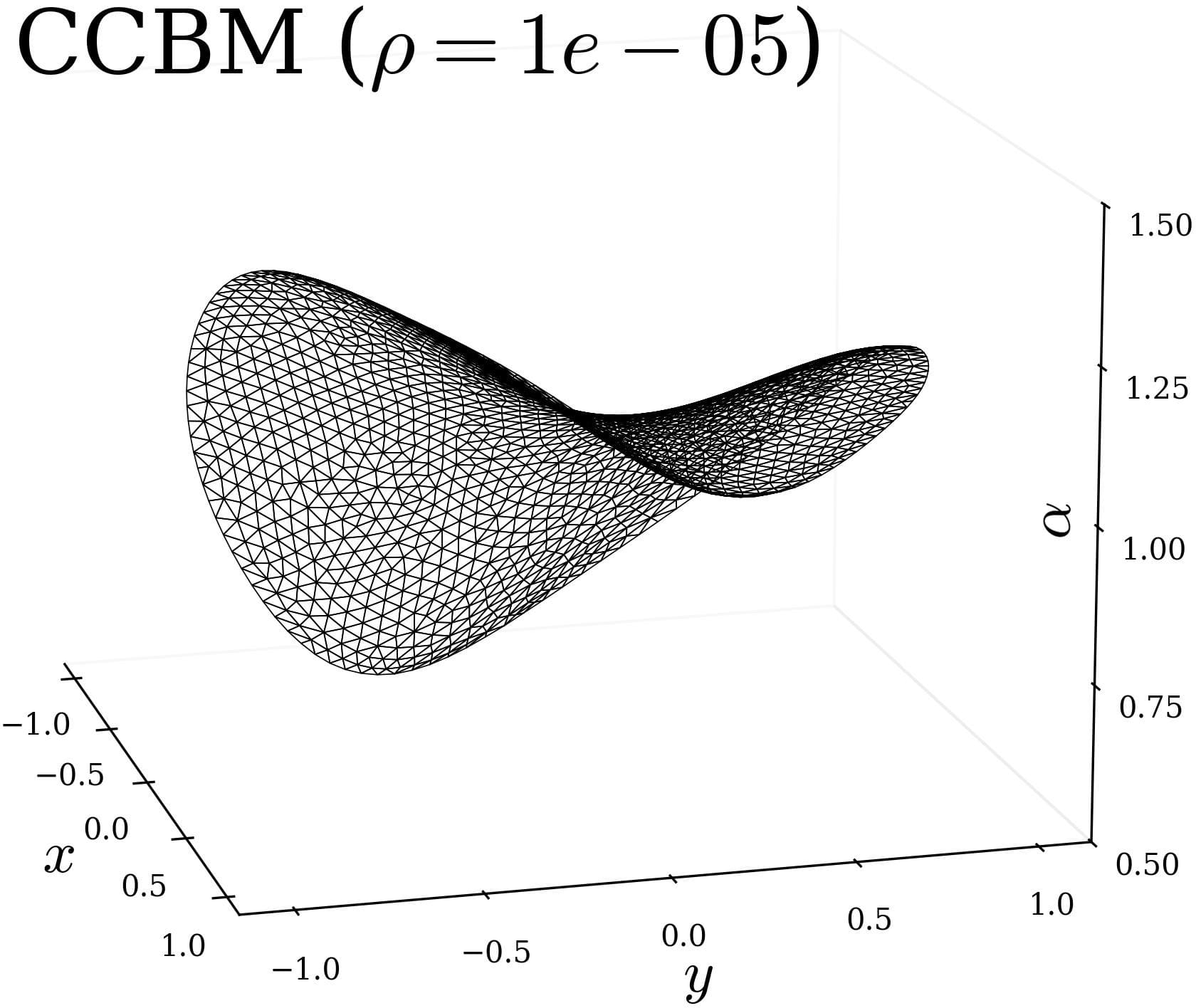}} \ 
\resizebox{0.225\textwidth}{!}{\includegraphics{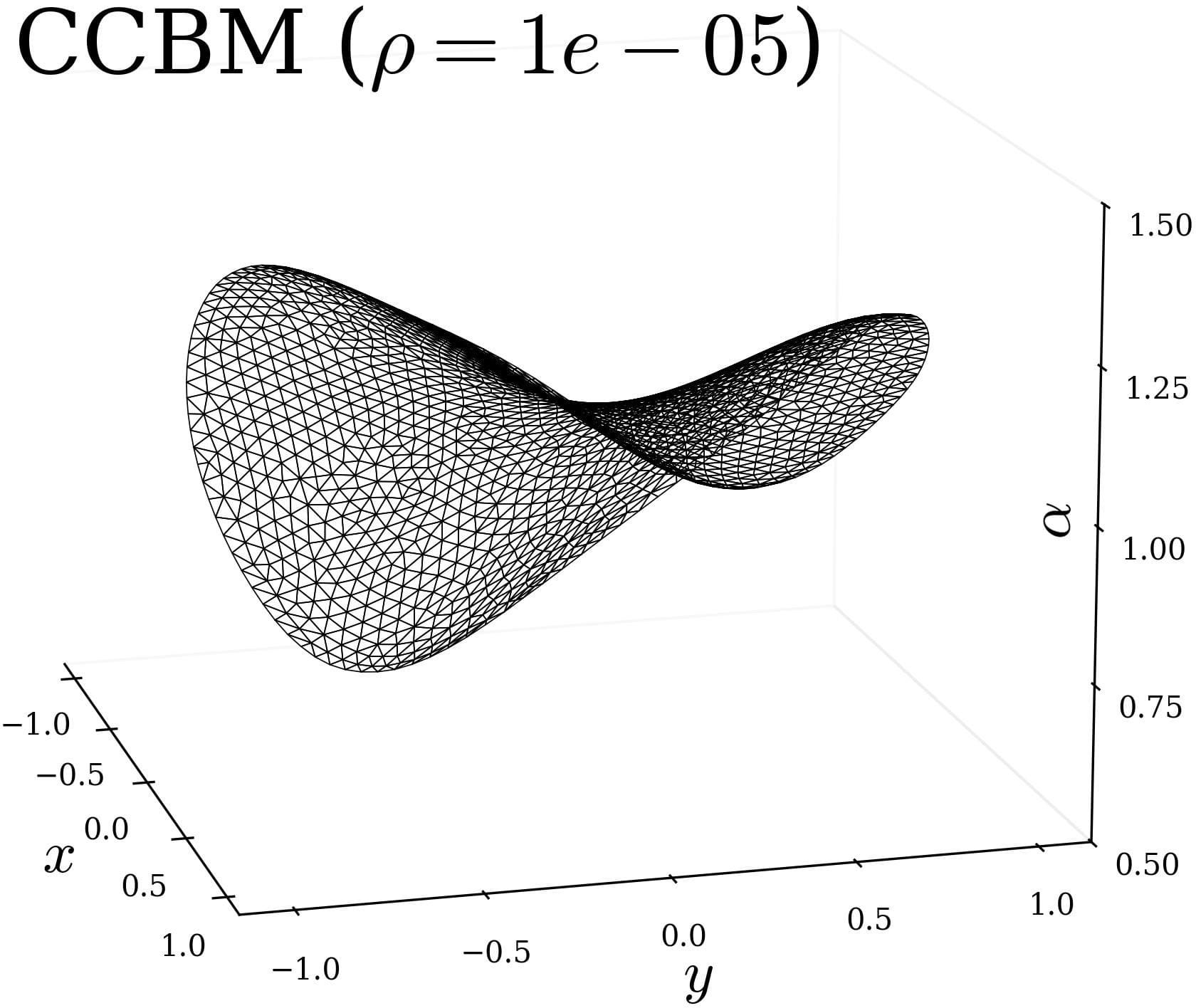}} \ 
\resizebox{0.225\textwidth}{!}{\includegraphics{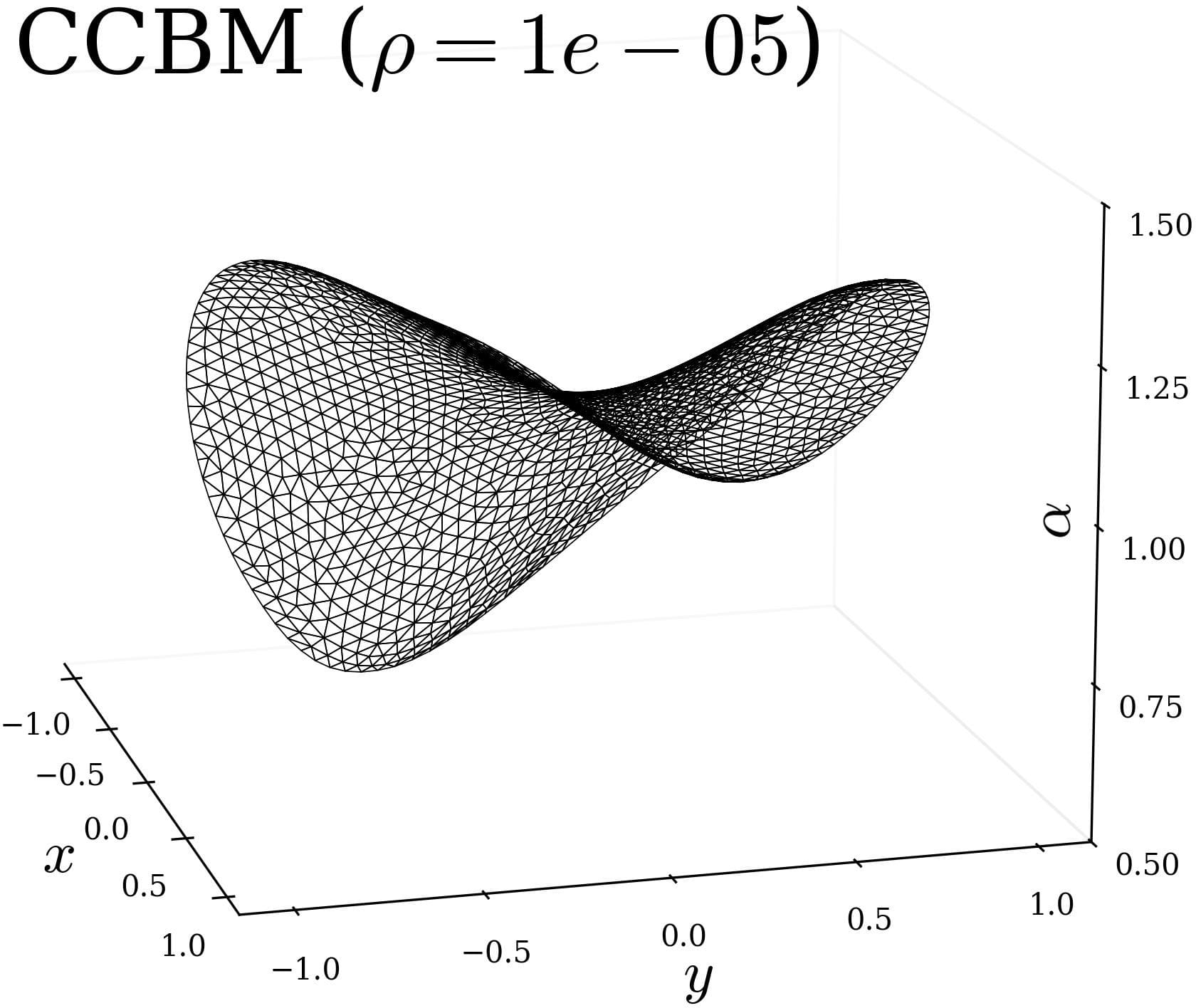}} \ 
\resizebox{0.225\textwidth}{!}{\includegraphics{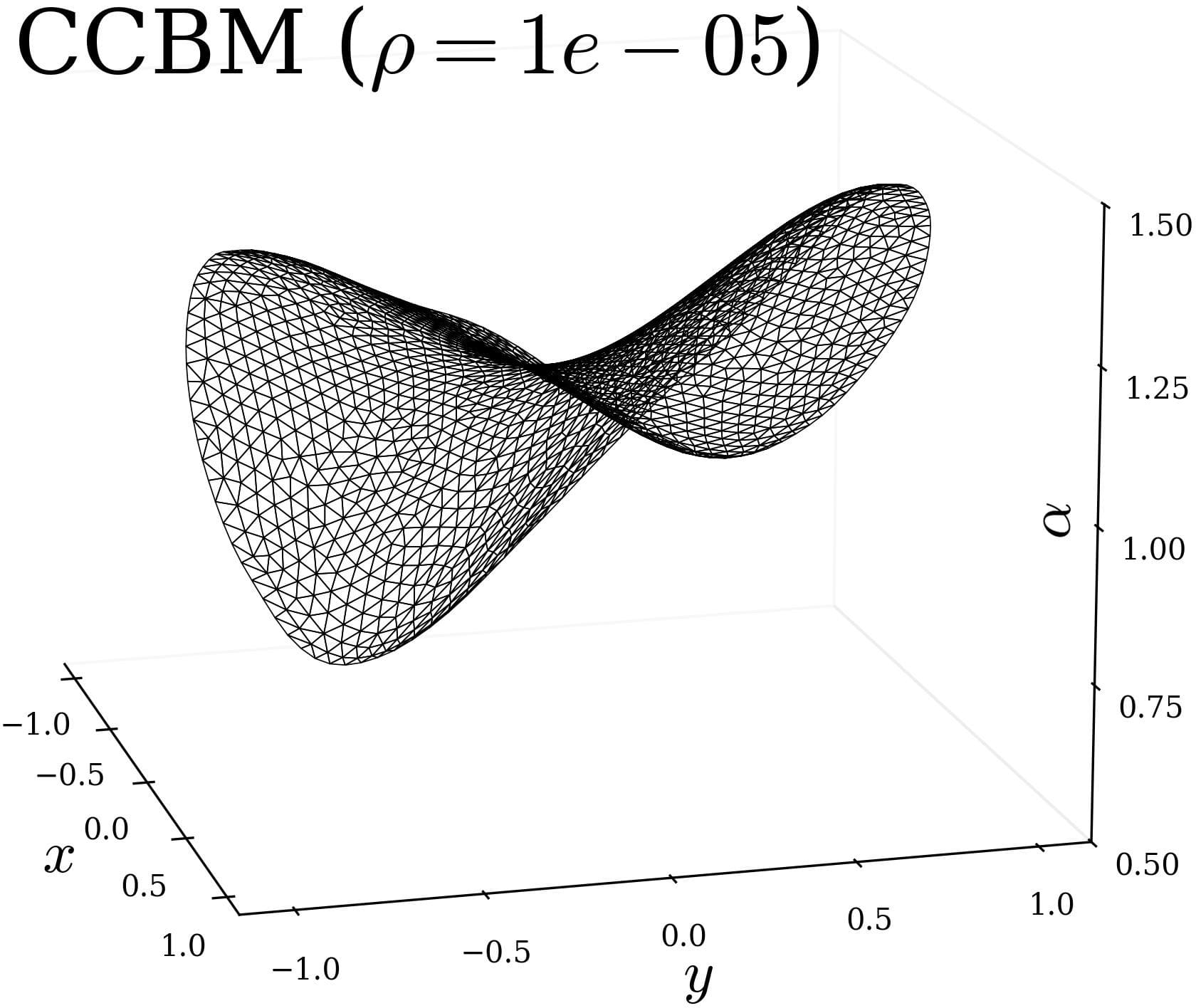}} \\[1em]
\resizebox{0.225\textwidth}{!}{\includegraphics{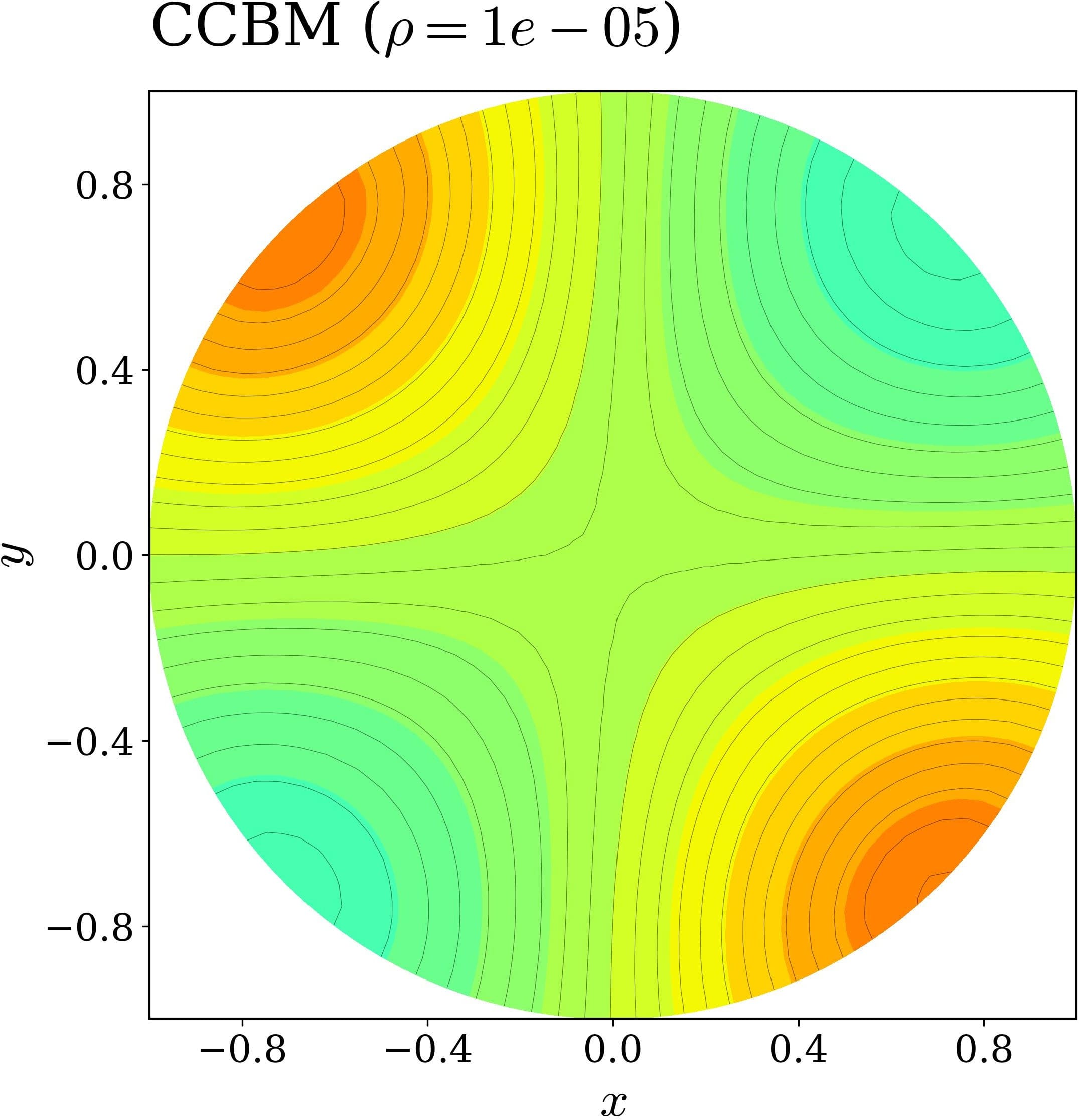}} \ 
\resizebox{0.225\textwidth}{!}{\includegraphics{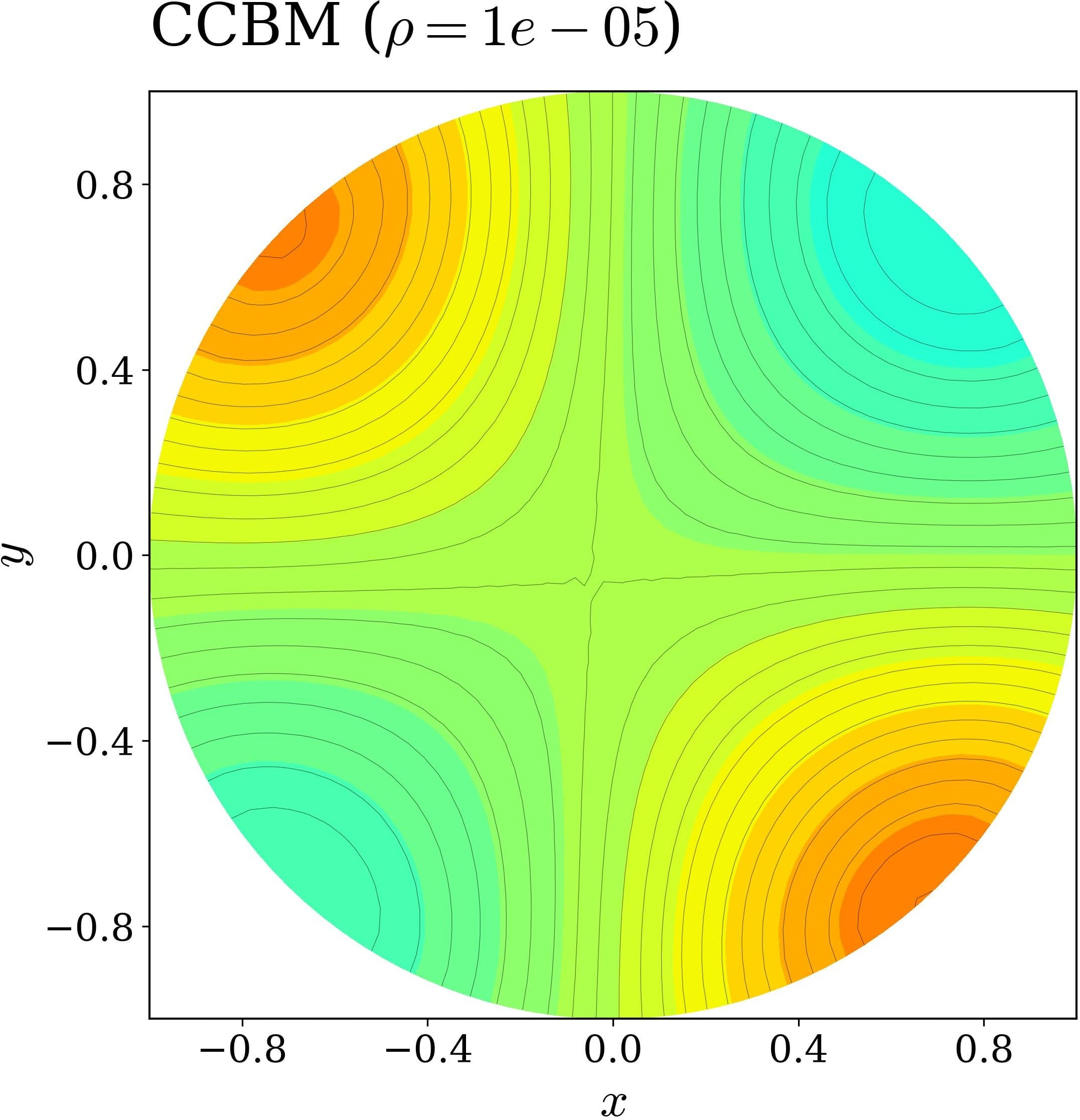}} \ 
\resizebox{0.225\textwidth}{!}{\includegraphics{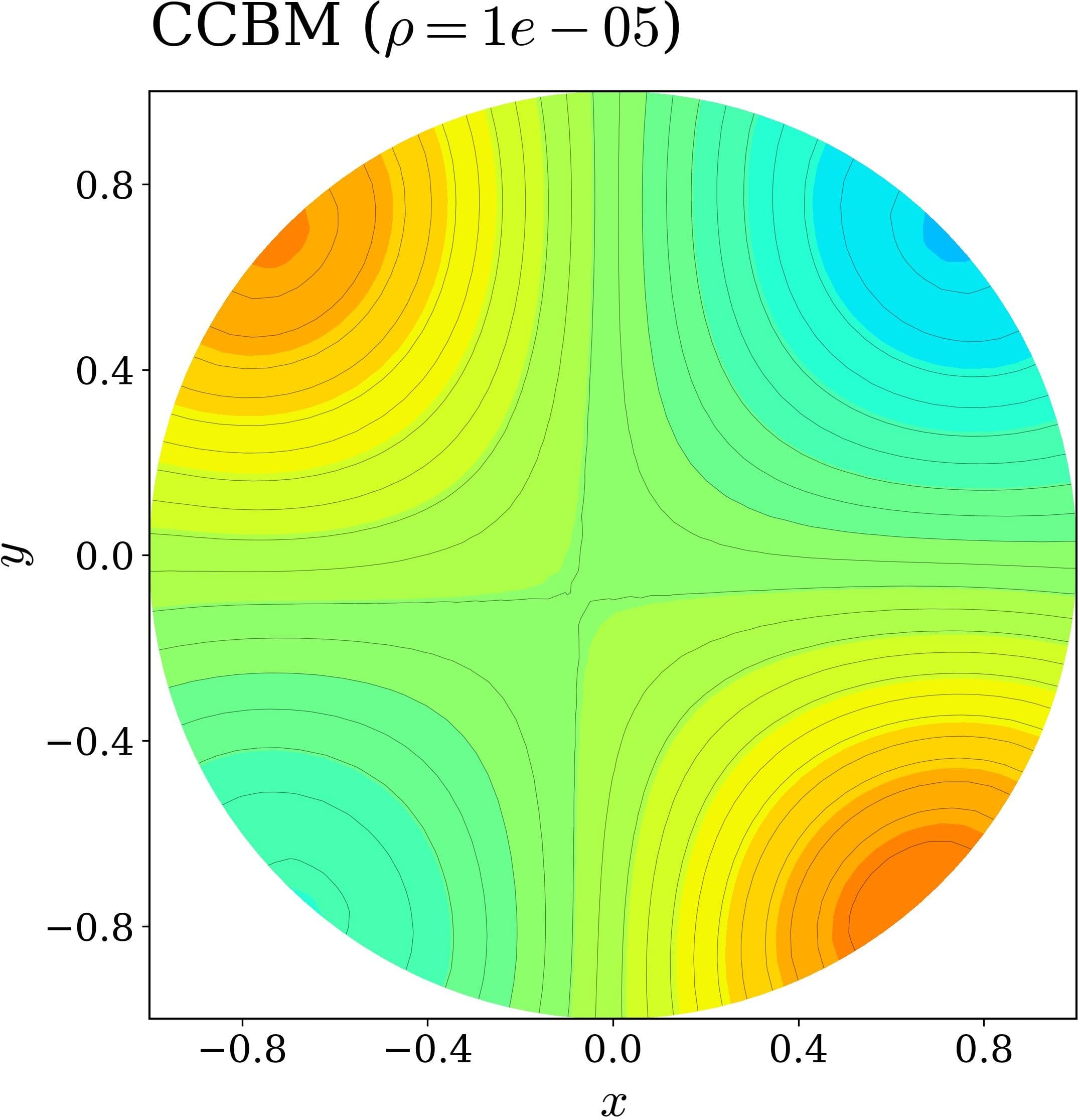}} \ 
\resizebox{0.225\textwidth}{!}{\includegraphics{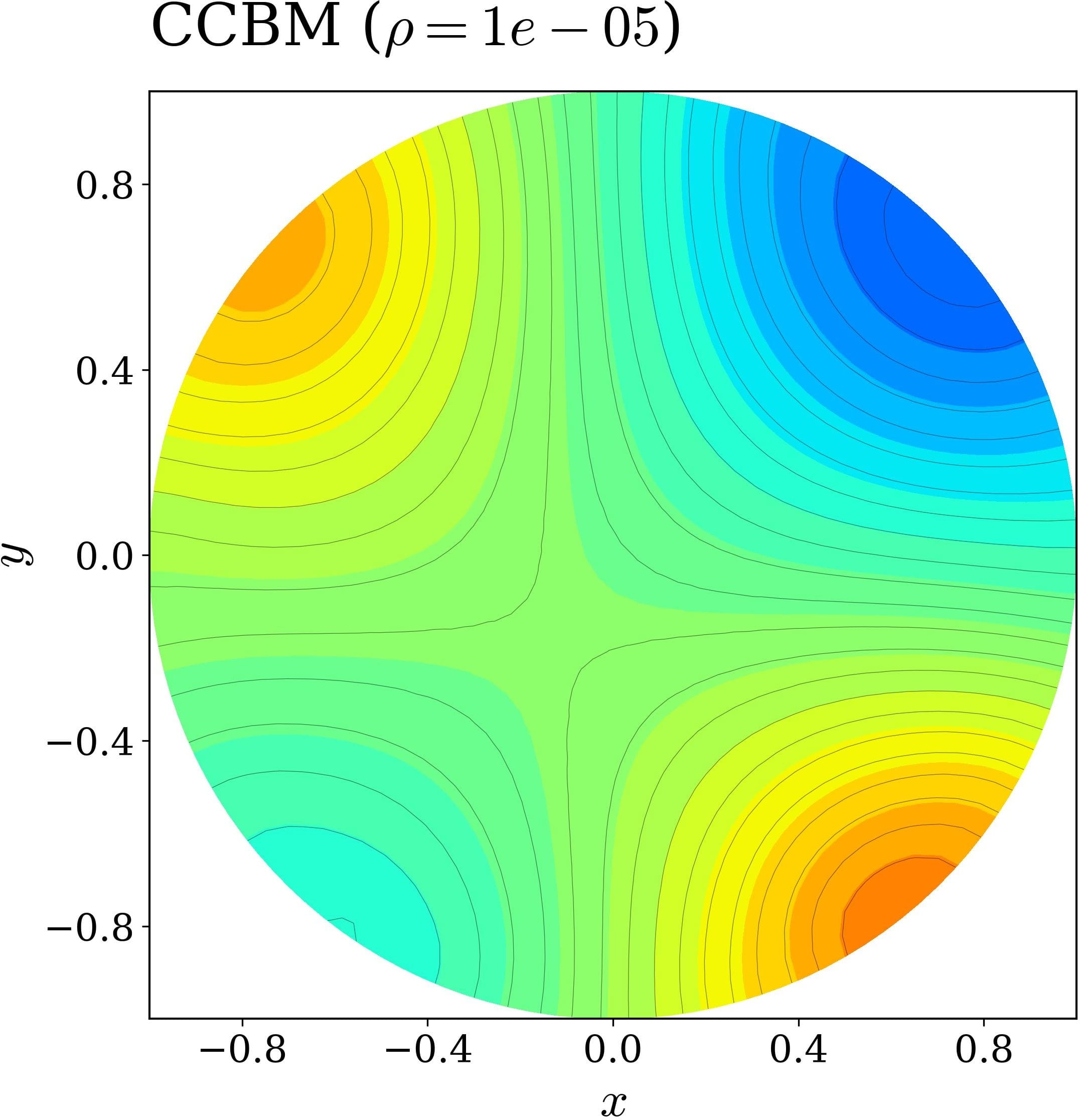}} \\[1em]
\resizebox{0.225\textwidth}{!}{\includegraphics{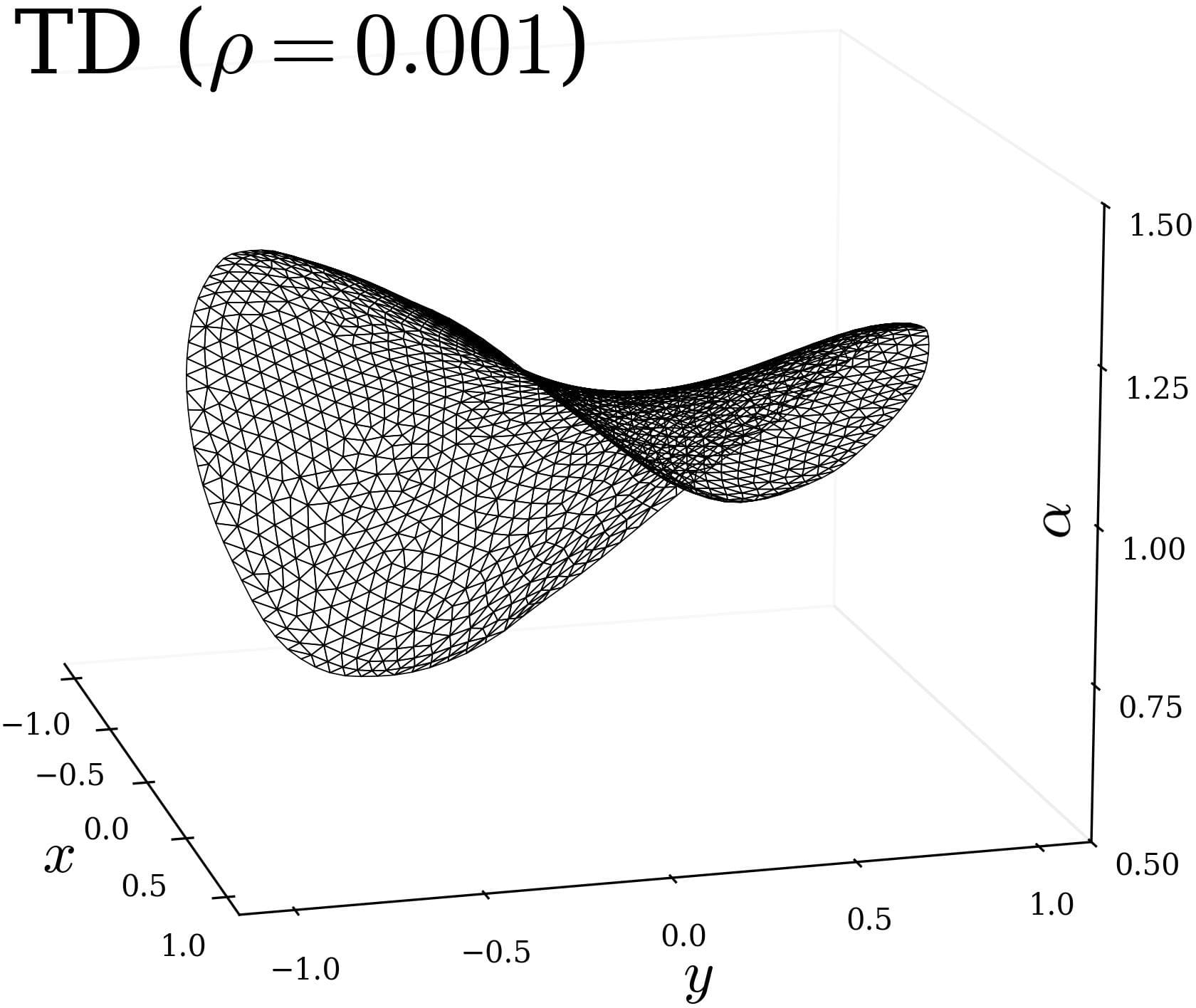}} \ 
\resizebox{0.225\textwidth}{!}{\includegraphics{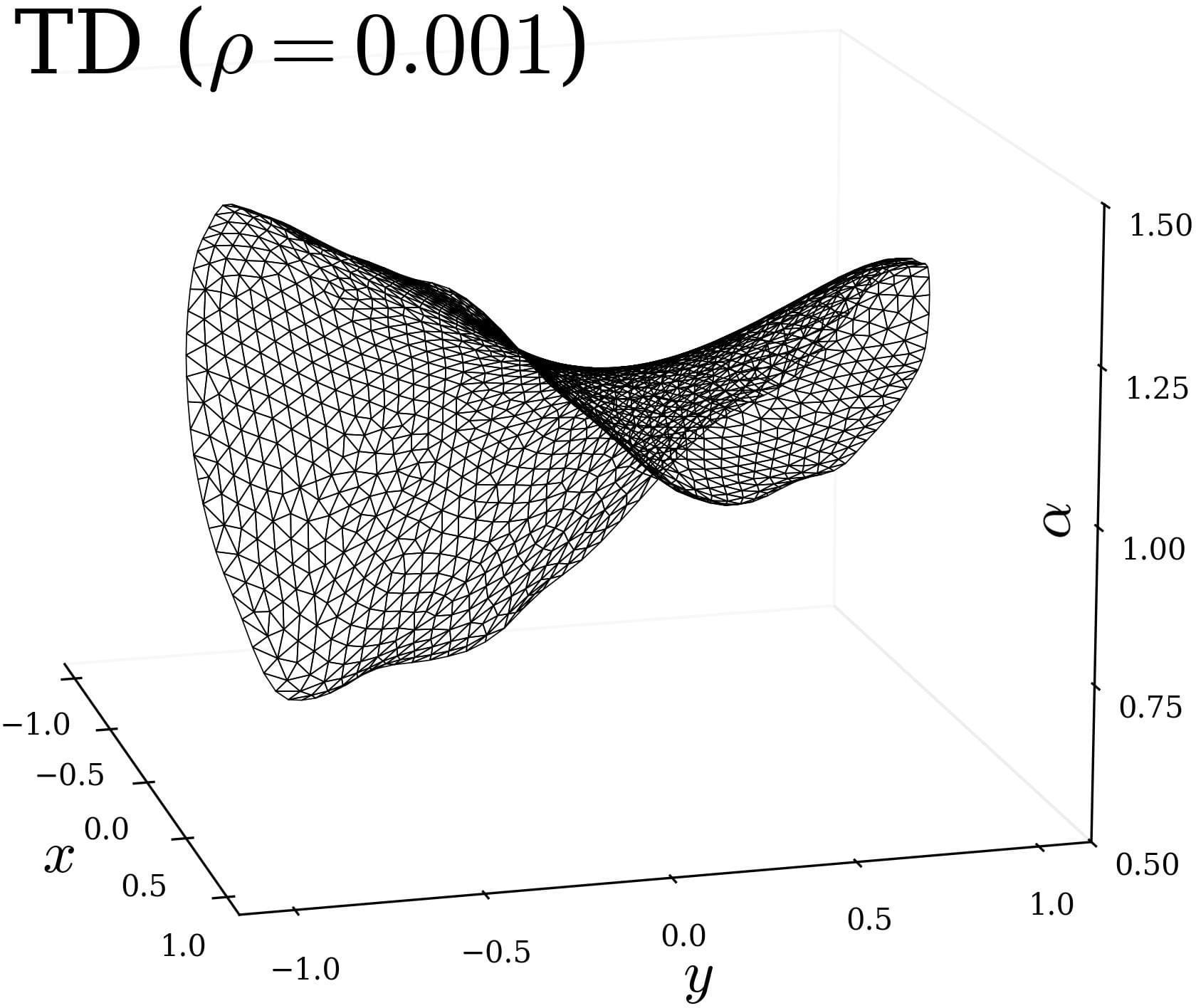}} \ 
\resizebox{0.225\textwidth}{!}{\includegraphics{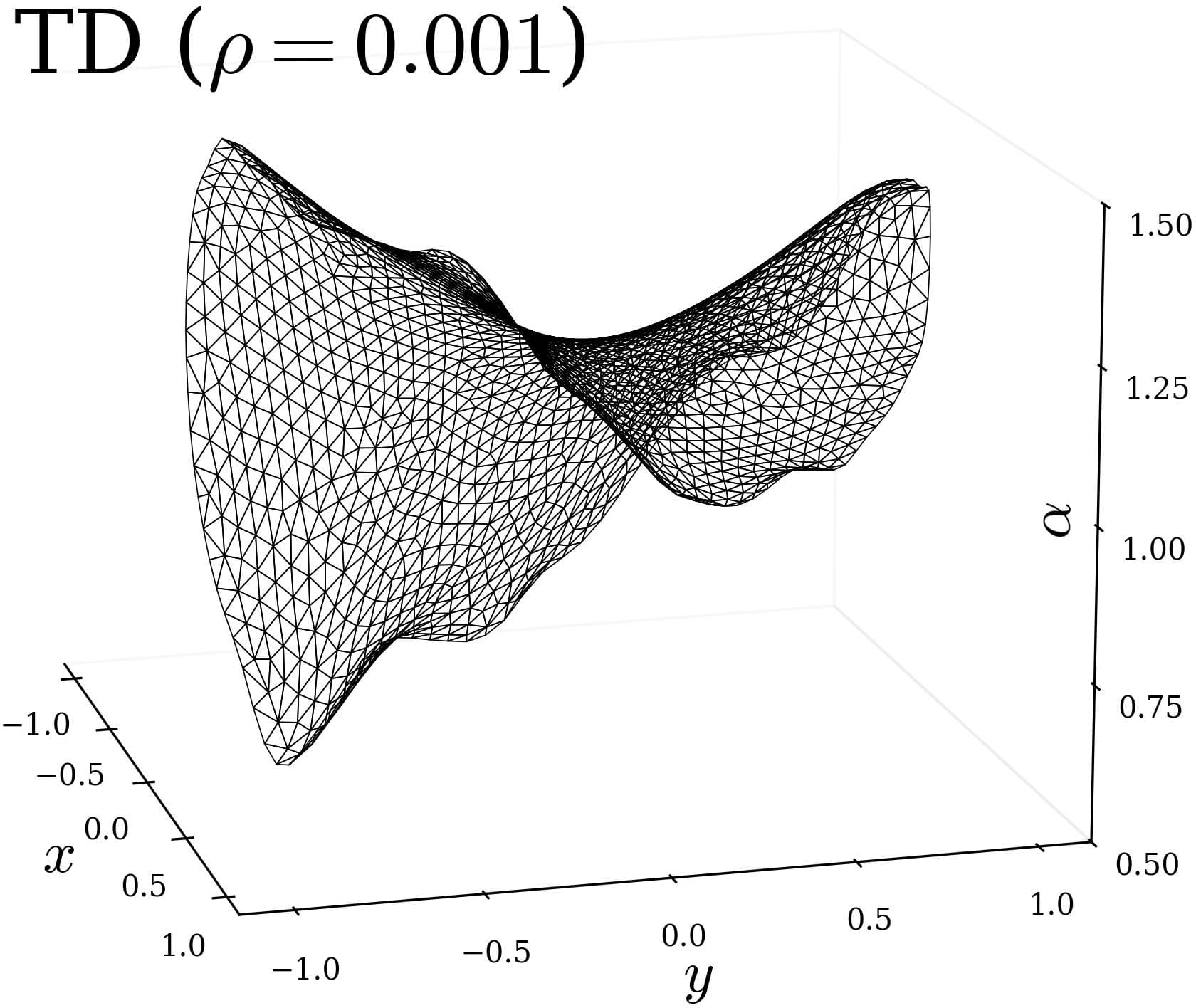}} \ 
\resizebox{0.225\textwidth}{!}{\includegraphics{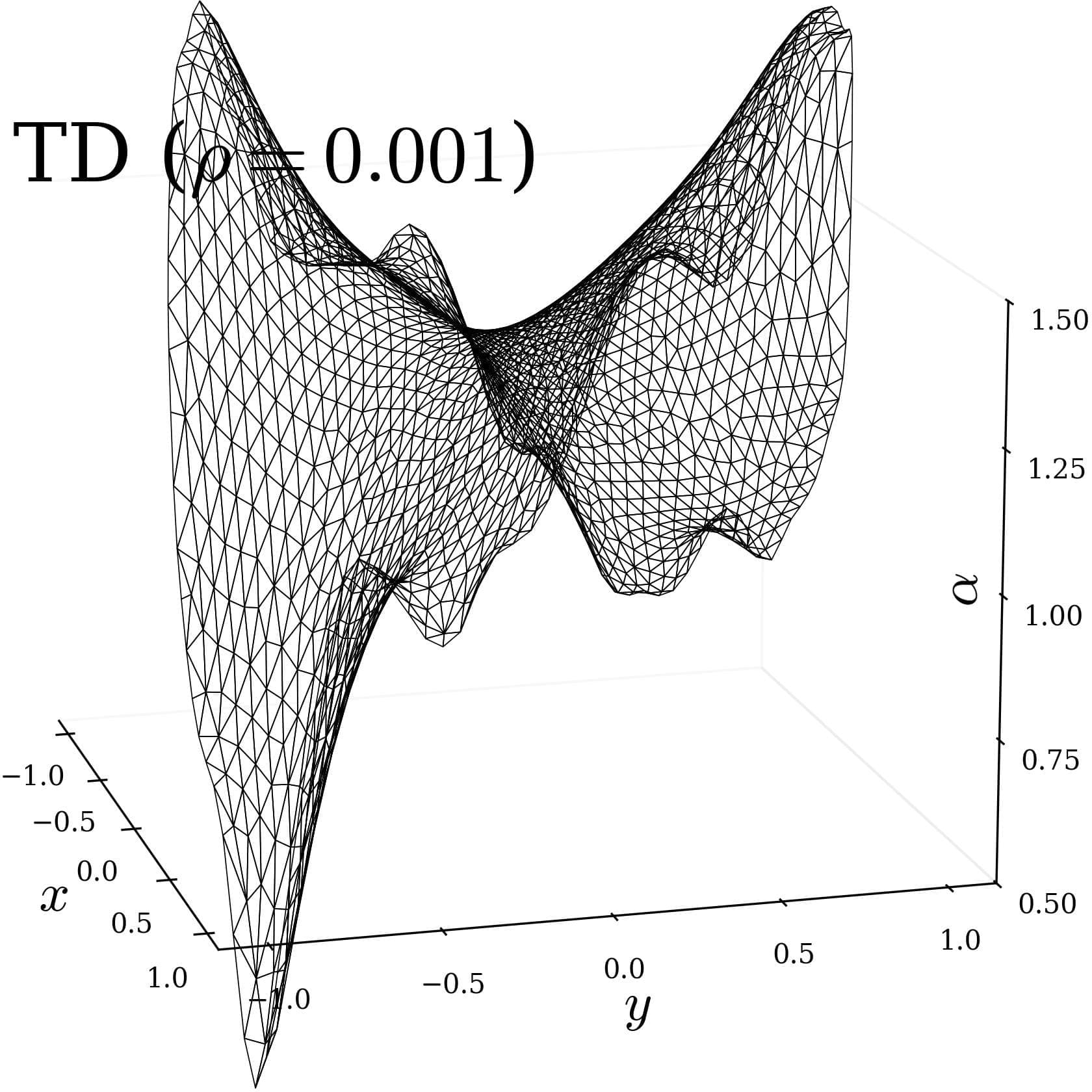}} \\[1em]
\resizebox{0.225\textwidth}{!}{\includegraphics{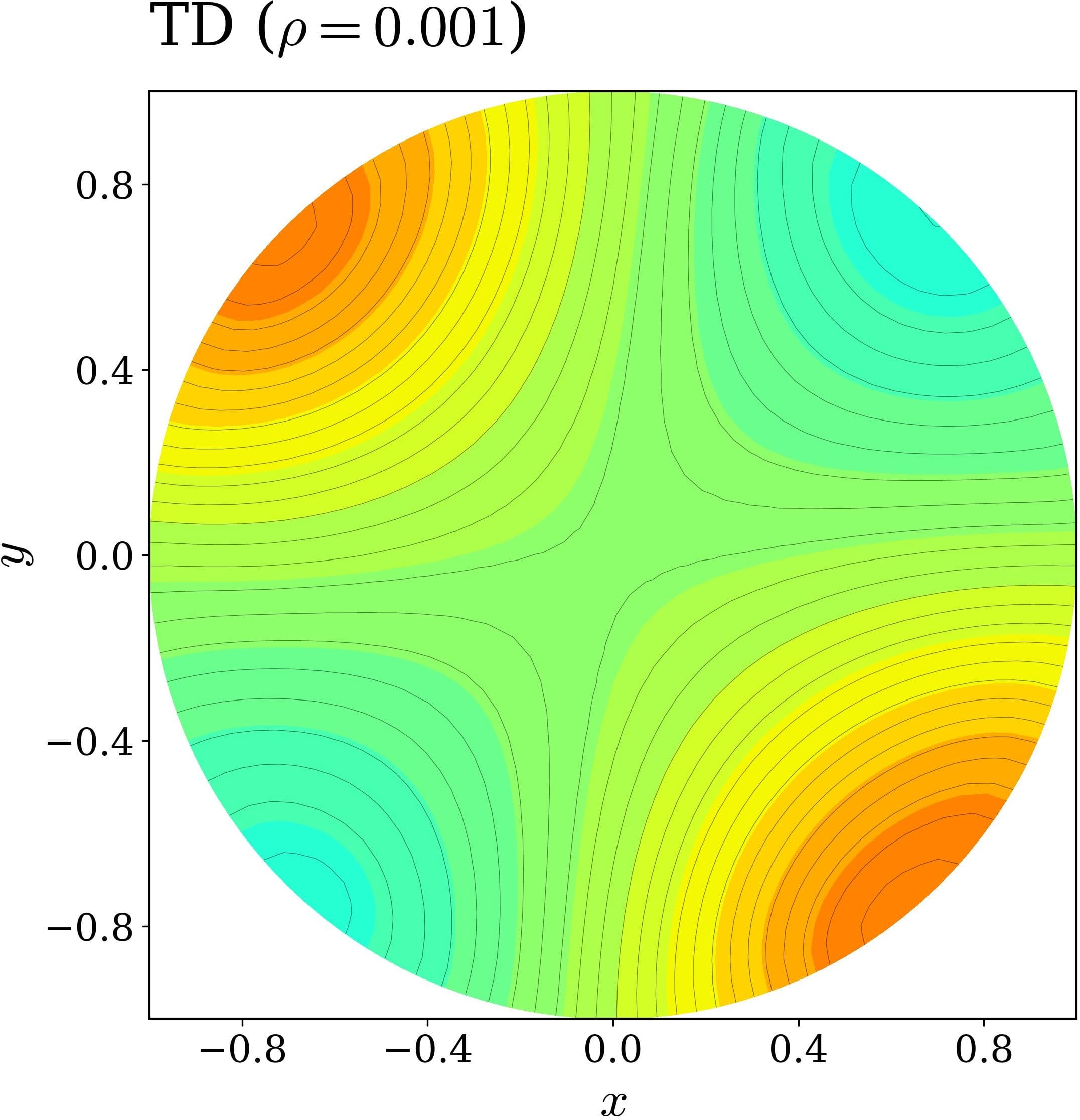}} \ 
\resizebox{0.225\textwidth}{!}{\includegraphics{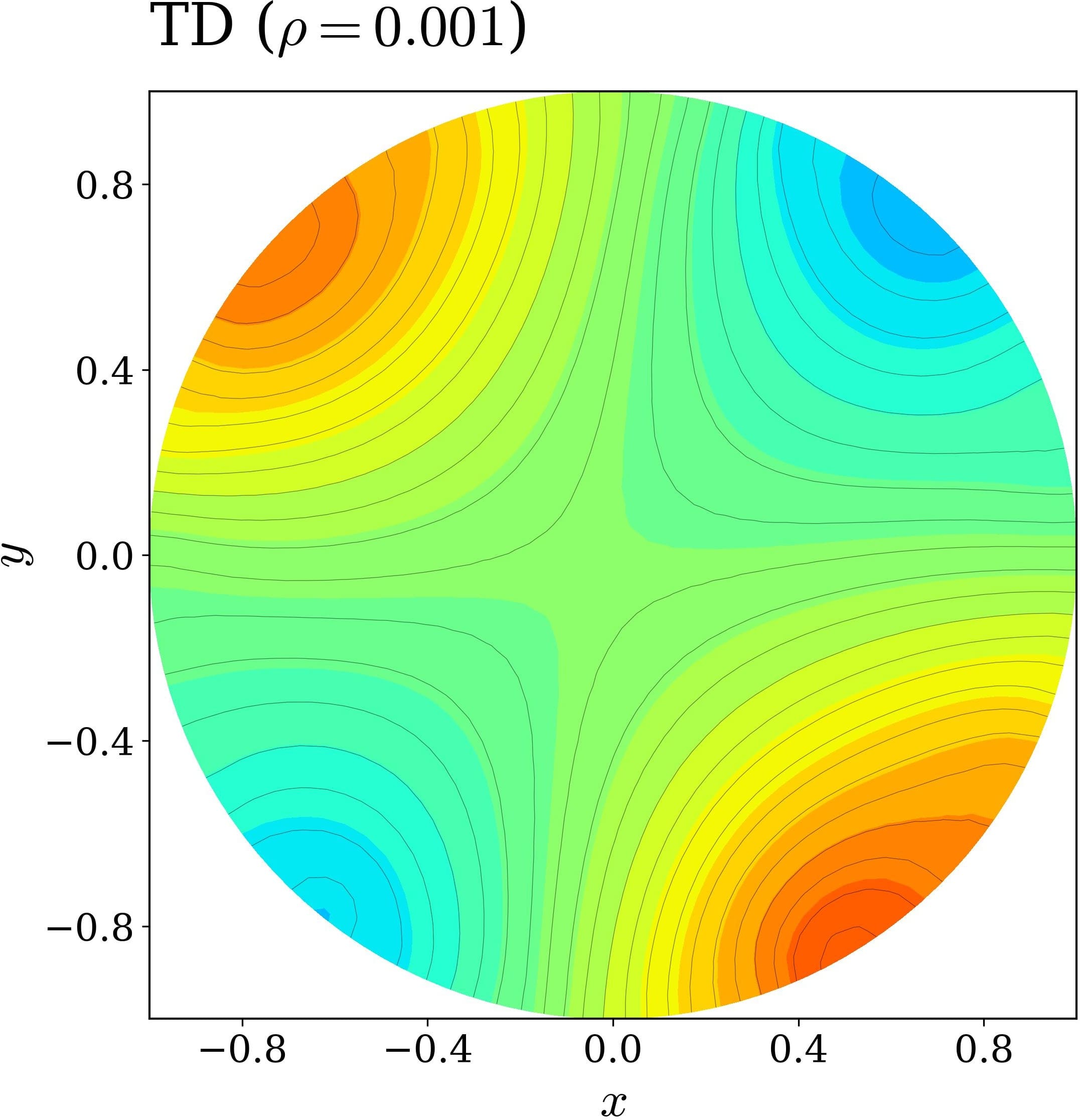}} \ 
\resizebox{0.225\textwidth}{!}{\includegraphics{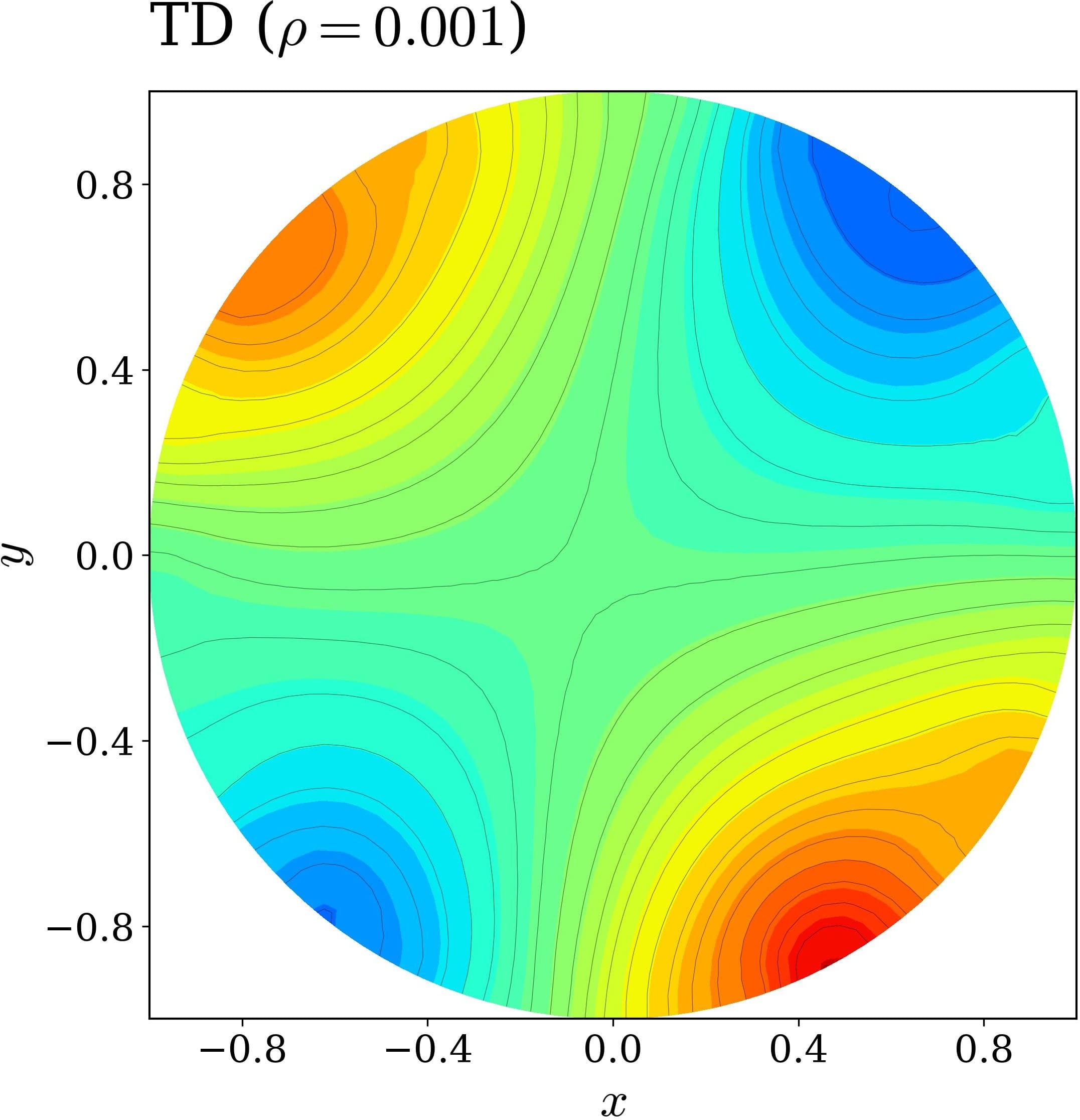}} \ 
\resizebox{0.225\textwidth}{!}{\includegraphics{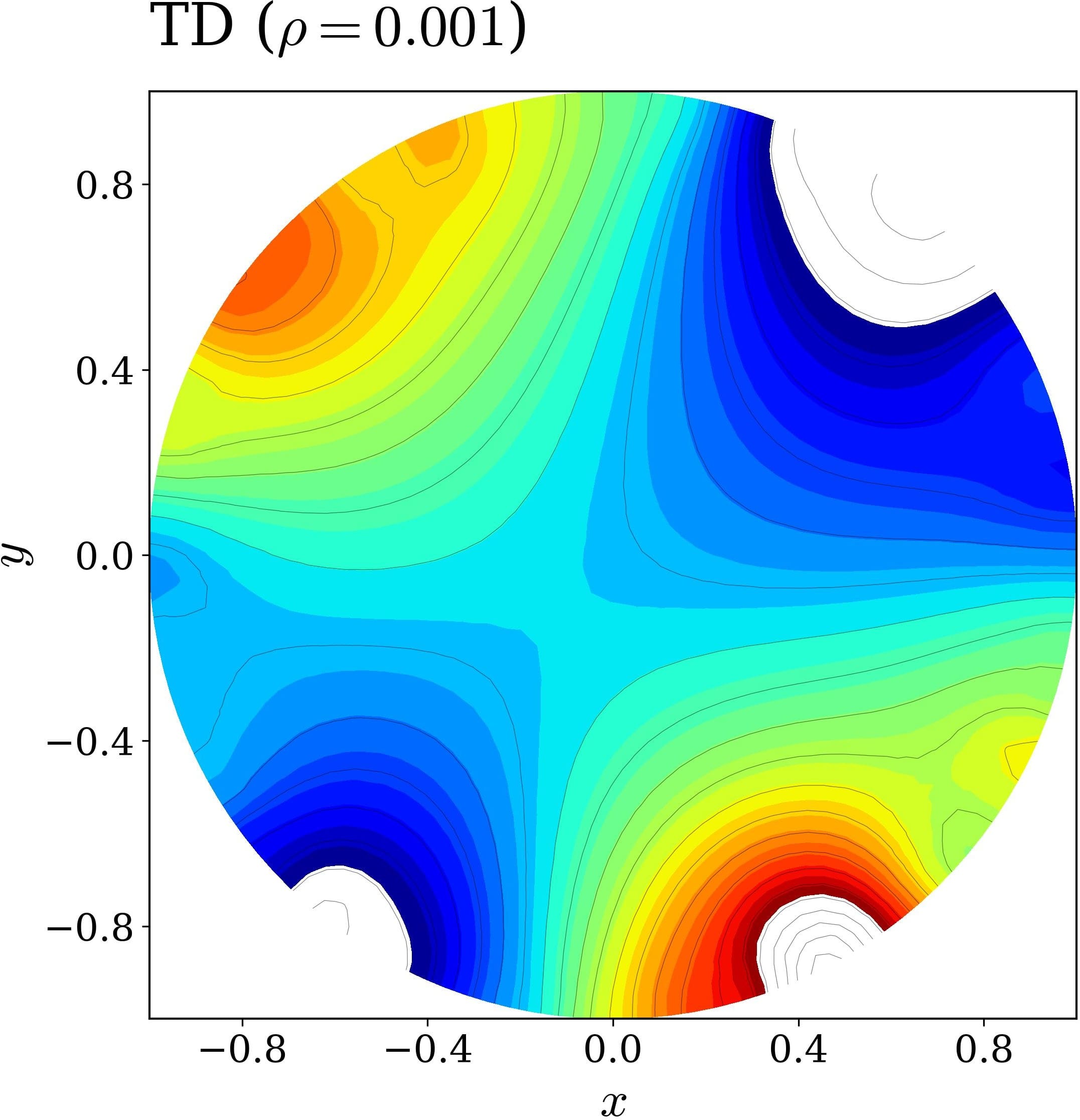}}\\[1em]
\resizebox{0.48\textwidth}{!}{\includegraphics{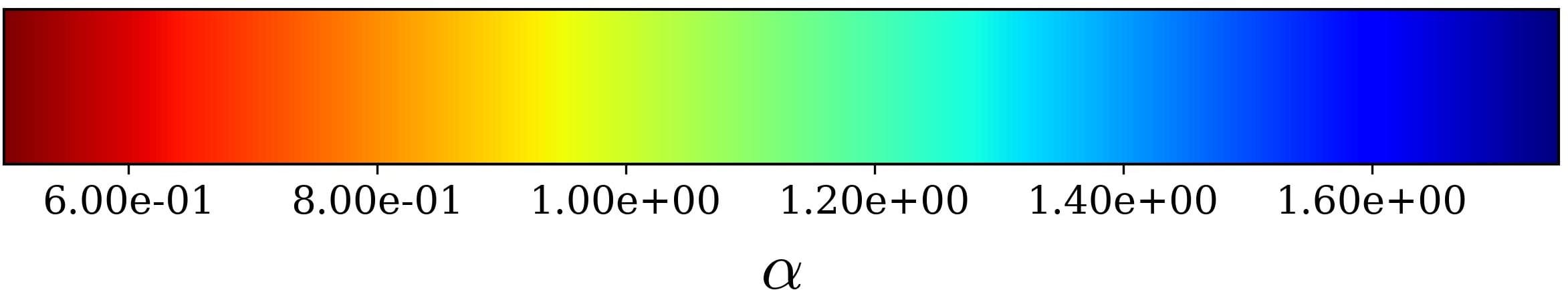}}
\caption{Reconstruction results for Dirichlet-data tracking and the CCBM under different noise levels (from left to right: $\delta = 0.01, 0.03, 0.05, 0.10$), with gradient smoothing ($\mu = 1.0$) and input data $g=\sin(\pi x)\sin(\pi y)$.}
\label{fig:reconstructions_under_higher_noise_levels}
\end{figure}
%
\subsubsection{Reconstructing a mildly oscillatory diffusion}\label{subsec:mildly_oscillatory_diffusion}
Building on the observations in the previous subsection regarding the role of boundary input data under noise, we now examine the reconstruction of a mildly oscillatory diffusion, $\alpha^{\star} = 1 + 0.25 \sin(\pi x)\sin(\pi y)$, from noisy measurements with $\delta=0.001$ and gradient smoothing $\mu=1.0$. To this end, we consider $\varOmega = C(0,2)$ with $\alpha^{[0]} = 1$, $b = 0$, $c = 1$, and $Q = 1$, and now take the target diffusion as $\alpha^{\star} = 1 + 0.25\sin(\pi x)\sin(\pi y)$.

Figure~\ref{fig:reconstructions_mildly_oscillatory_diffusion} compares reconstructions obtained with constant boundary input $g=1$ (second row), oscillatory input $g=\sin(\pi x)\sin(\pi y)$ (third row), and mixed input $g=2+\sin(\pi x)\sin(\pi y)$, where the last row shows the reconstructed diffusion $\alpha$. Consistent with the observations of the previous subsection, constant boundary data provide limited excitation of the internal oscillatory modes, resulting in reconstructions that remain close to the mean diffusion for all methods. For the present problem setup, the oscillatory boundary input yields a visibly improved recovery of the internal oscillations in $\alpha^{\star}$, as illustrated in Figure~\ref{fig:reconstructions_mildly_oscillatory_diffusion}. In this case, CCBM reproduces the amplitude and spatial structure of the oscillations with good accuracy, while the TD formulation captures the main pattern but exhibits mild boundary distortions. By comparison, the KV and TN reconstructions show stronger boundary-induced oscillations that extend into the interior. For the mixed boundary input $g=2+\sin(\pi x)\sin(\pi y)$, the diffusion reconstructed by CCBM remains closest to the reference solution among the tested methods, whereas the other formulations are more strongly affected by noise and boundary artifacts.


It is also worth noting that, across all boundary inputs and formulations, the reconstruction of $\alpha$ is generally less accurate in the interior of the domain than near the boundary, as visible in Figure~\ref{fig:reconstructions_mildly_oscillatory_diffusion}. Recovering fine spatial variations of the diffusion coefficient in the inner region appears to be more challenging, particularly for the oscillatory components. A plausible explanation is that interior features are only indirectly informed by boundary measurements, whose sensitivity to internal variations diminishes with distance from the boundary. In addition, regions where $\alpha$ attains larger values may partially shield or attenuate the influence of deeper structures on the boundary data, thereby further obscuring the true values of $\alpha$ in the interior. These effects are consistent with the ill-posed nature of inverse boundary value problems and have been commonly observed in inverse coefficient reconstruction settings.

For comparison, the histories of the cost functional and the gradient norm corresponding to Figure~\ref{fig:reconstructions_mildly_oscillatory_diffusion} are shown in Figure~\ref{fig:reconstructions_mildly_oscillatory_diffusion_cost_and_gradient}.
\begin{figure}[htp!]
\resizebox{0.28\textwidth}{!}{\includegraphics{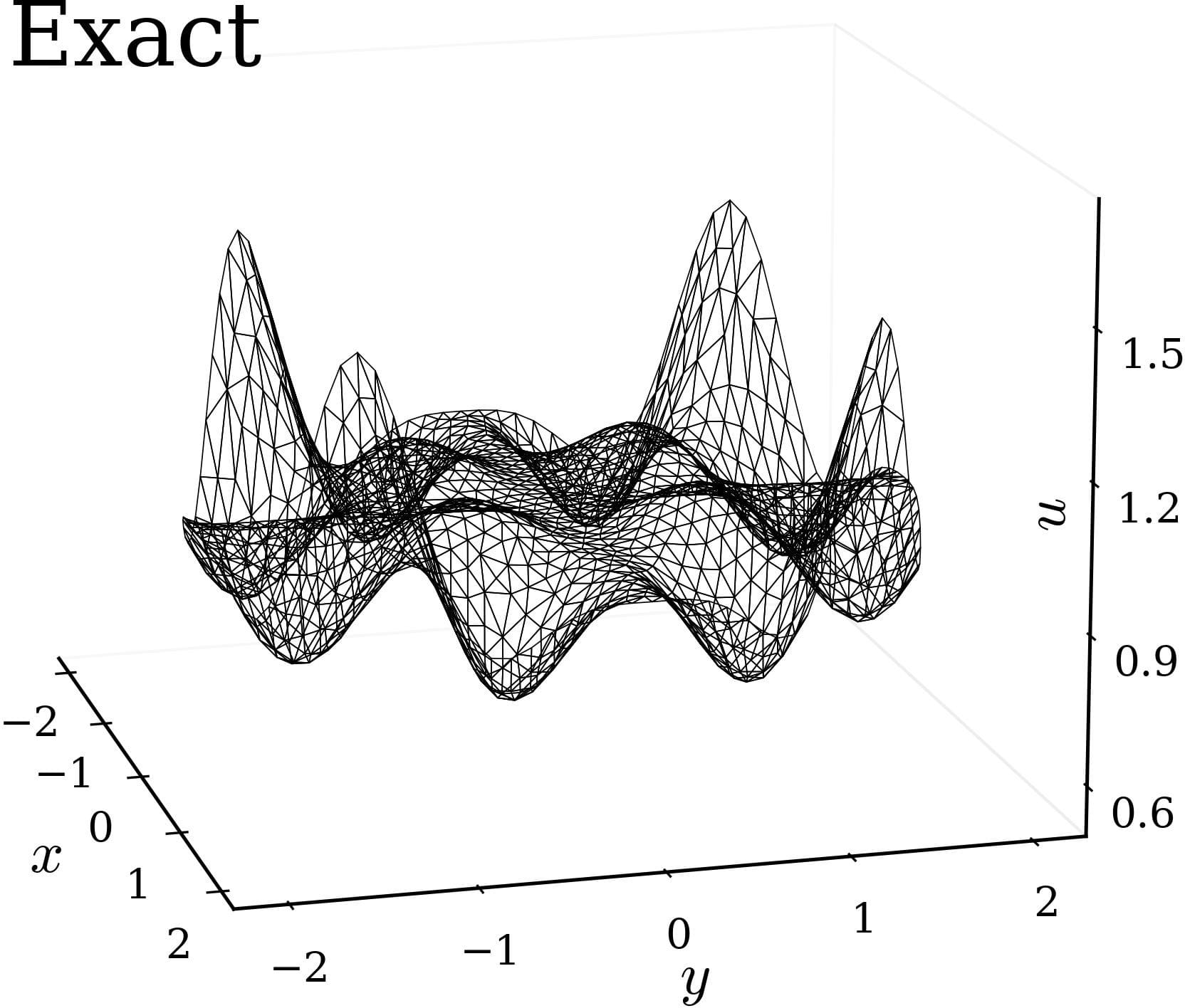}} \quad
\resizebox{0.225\textwidth}{!}{\includegraphics{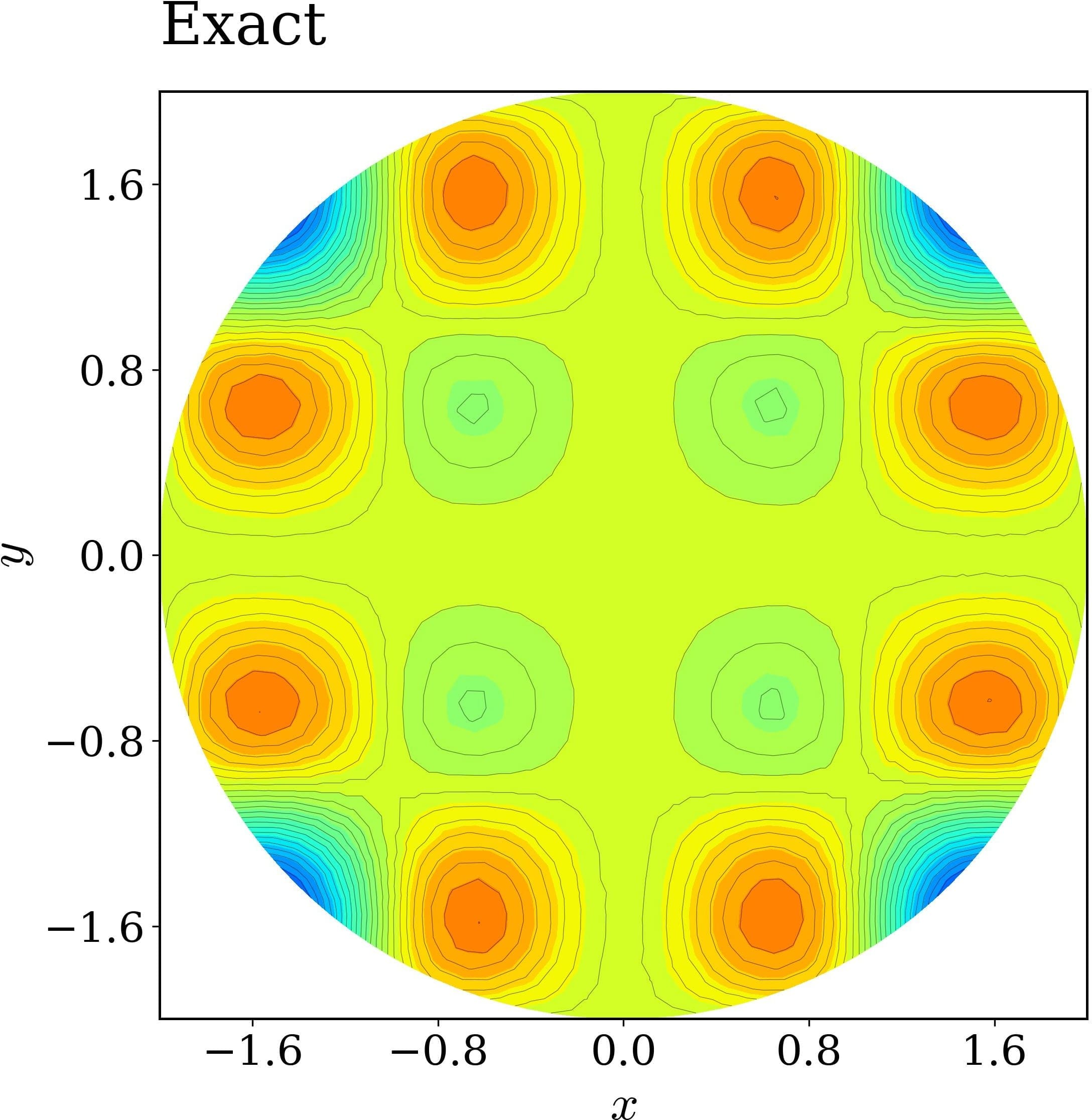}} \\[1em]
\centering
\resizebox{0.225\textwidth}{!}{\includegraphics{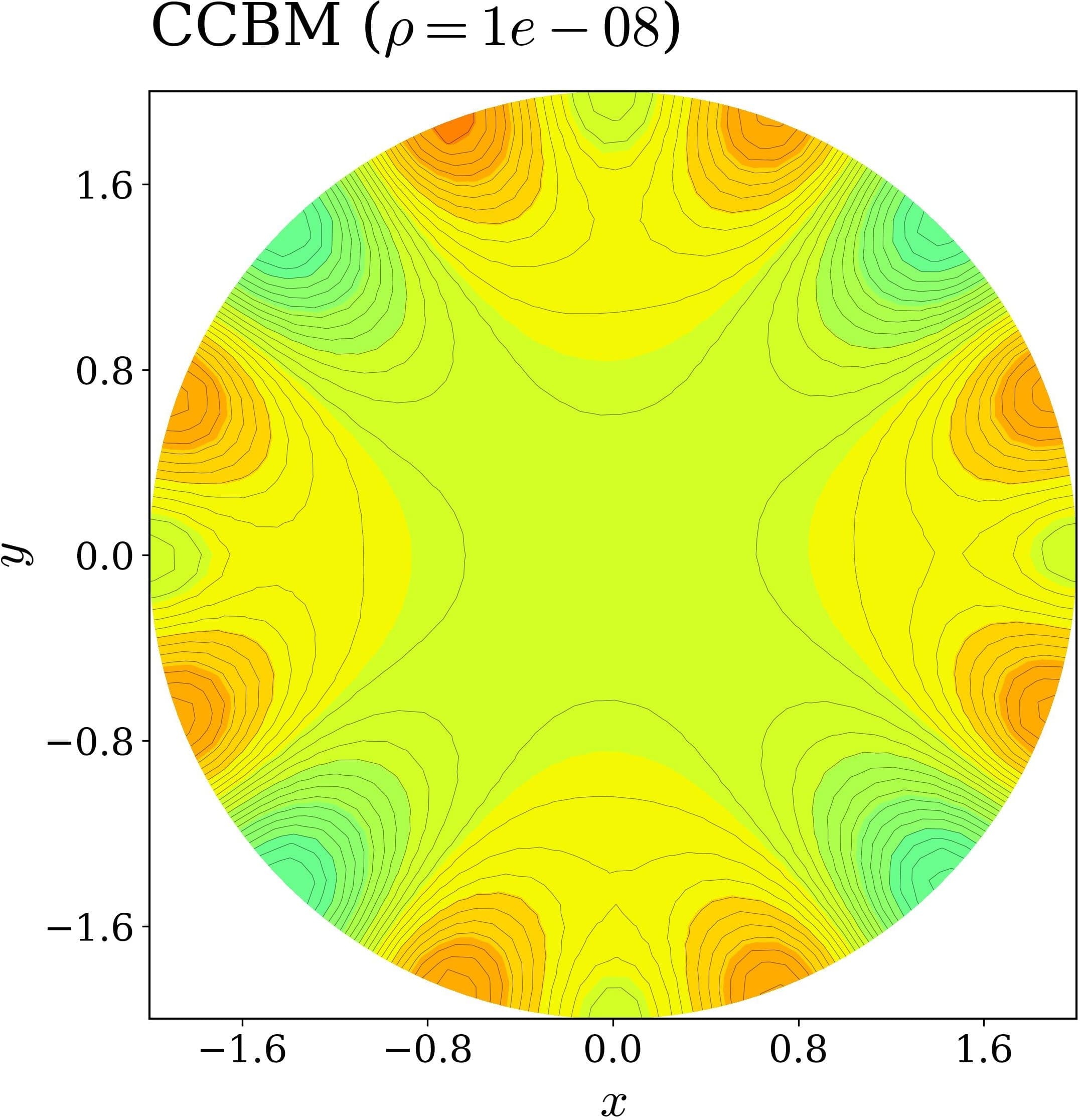}} \ 
\resizebox{0.225\textwidth}{!}{\includegraphics{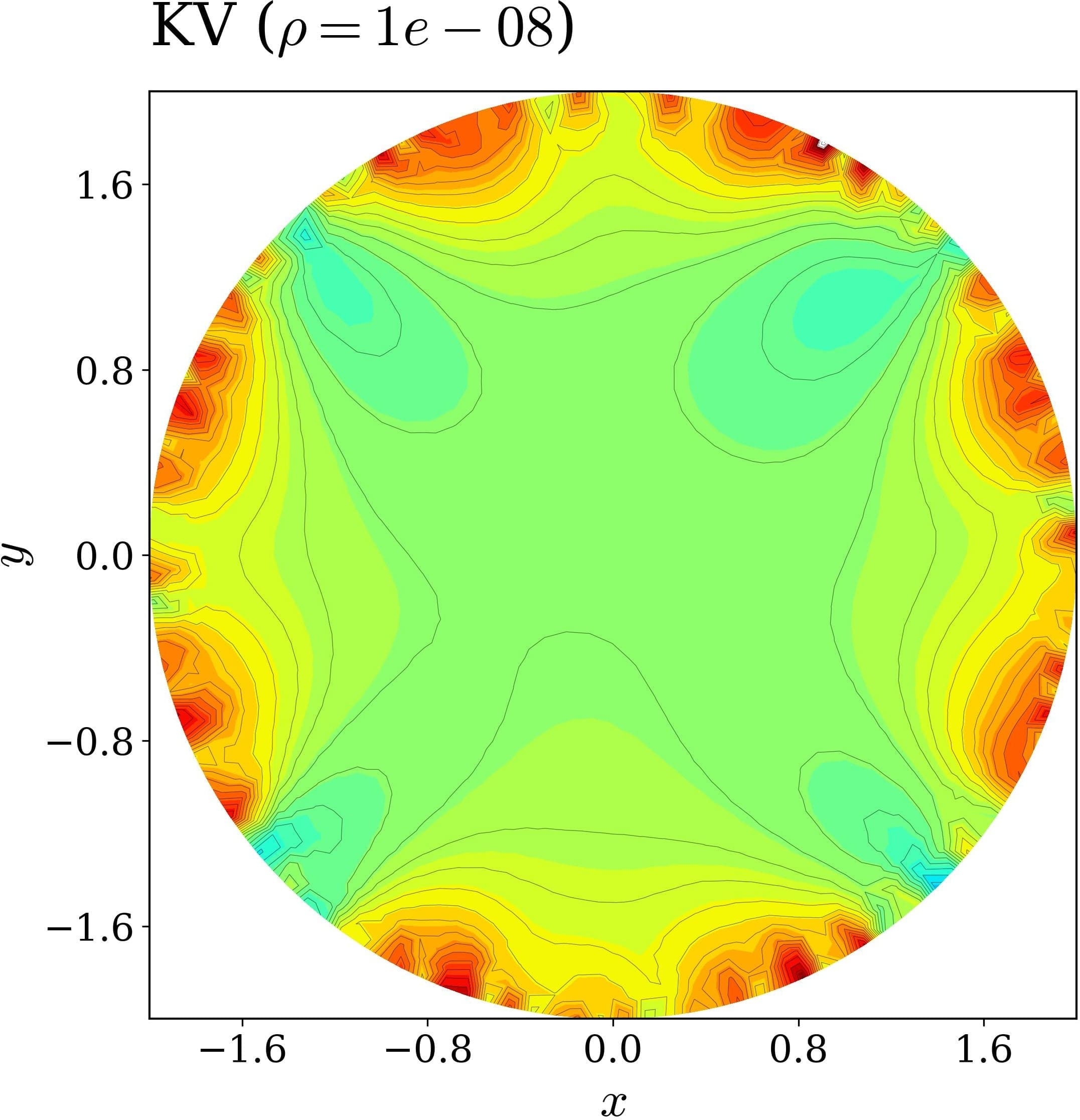}} \ 
\resizebox{0.225\textwidth}{!}{\includegraphics{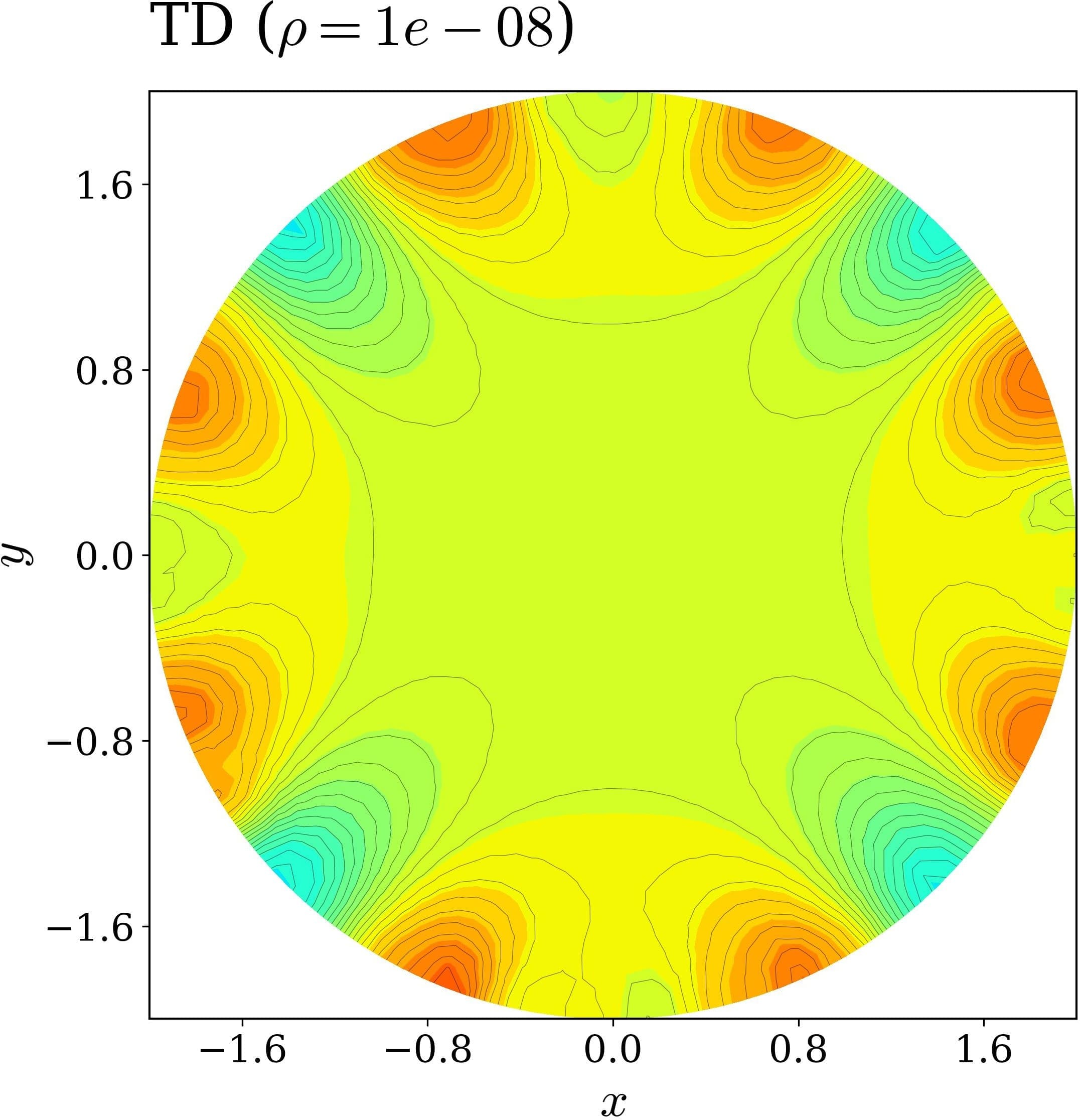}} \ 
\resizebox{0.225\textwidth}{!}{\includegraphics{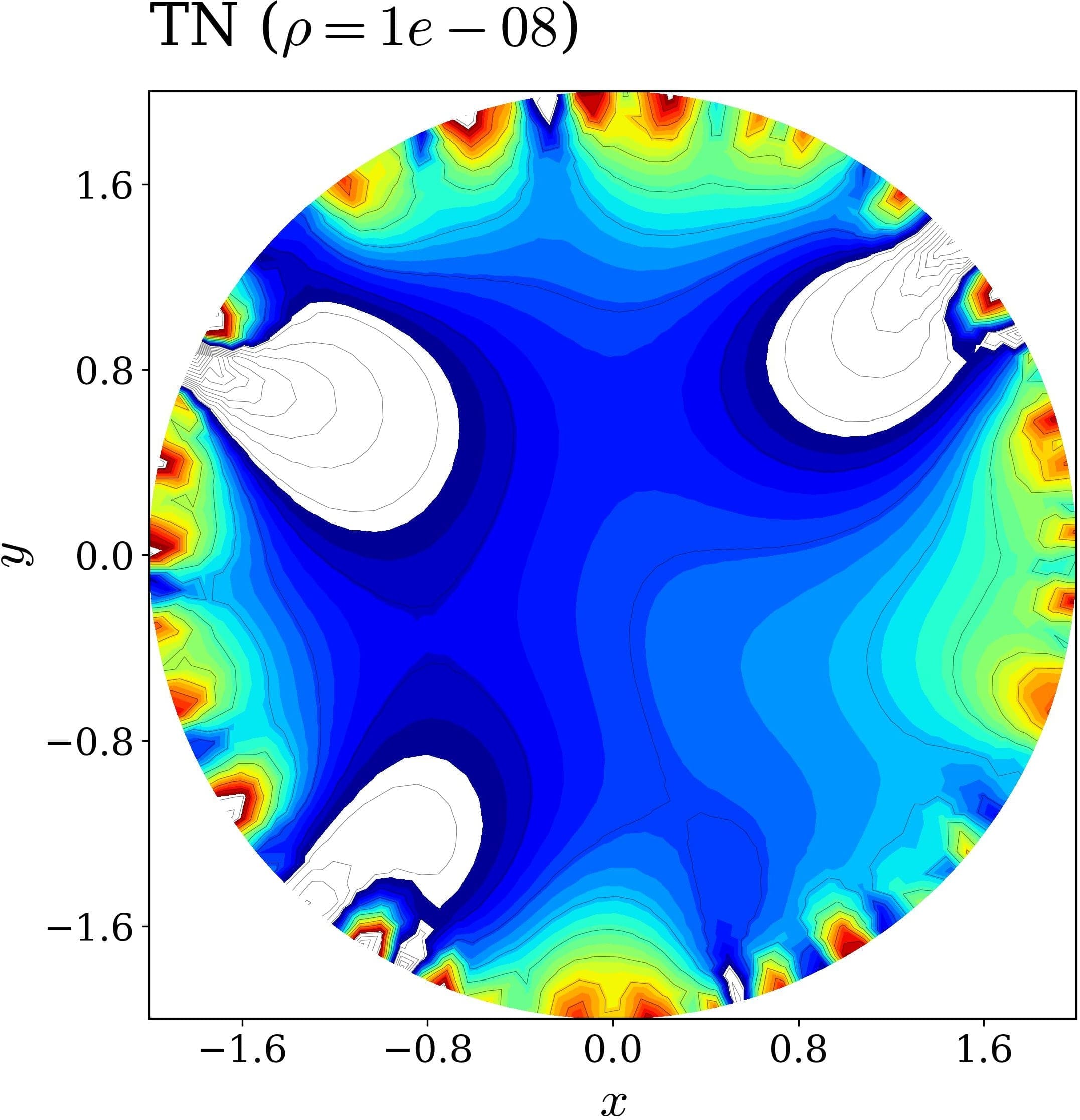}} \\[1em]
\resizebox{0.225\textwidth}{!}{\includegraphics{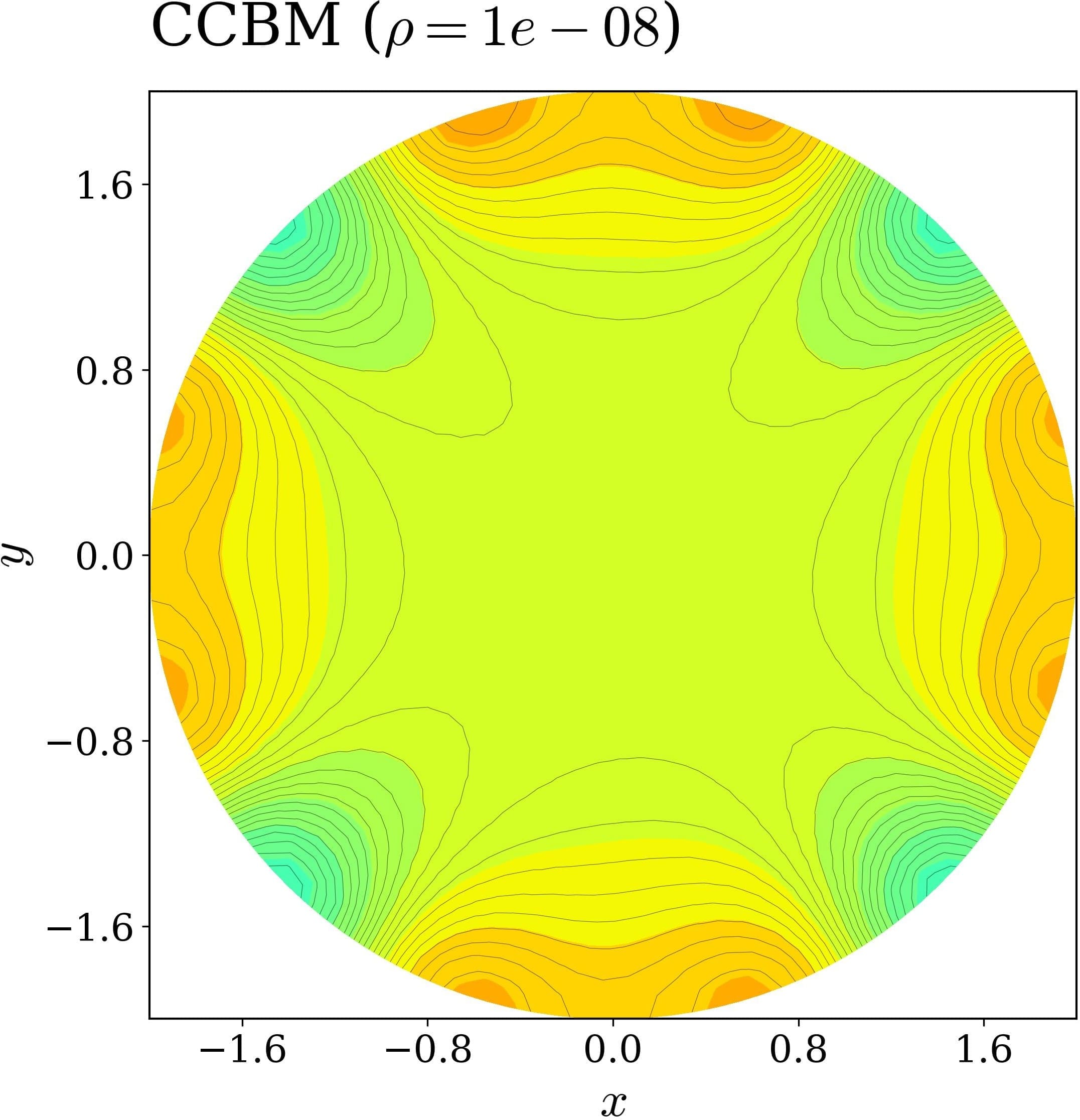}} \ 
\resizebox{0.225\textwidth}{!}{\includegraphics{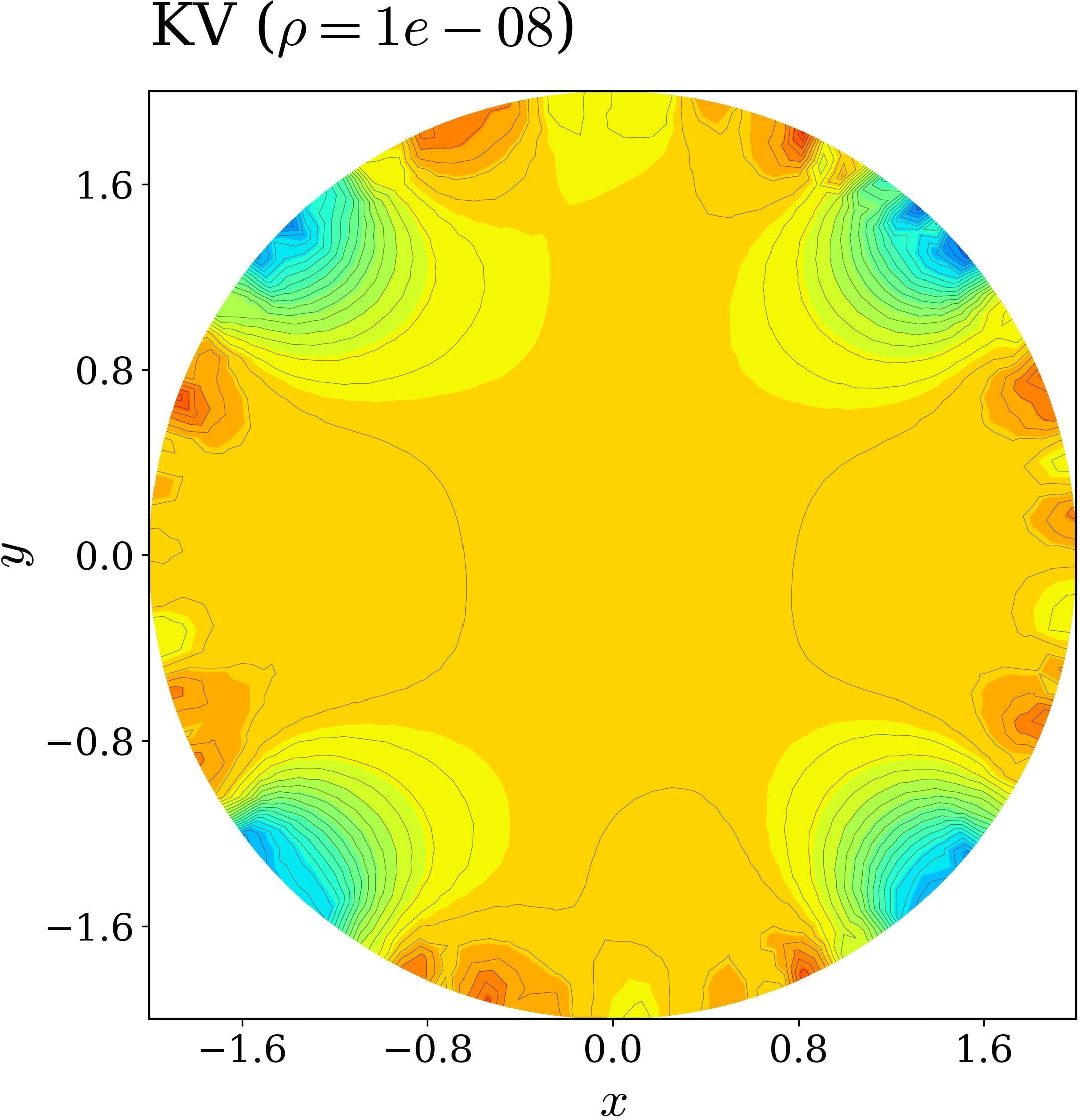}} \ 
\resizebox{0.225\textwidth}{!}{\includegraphics{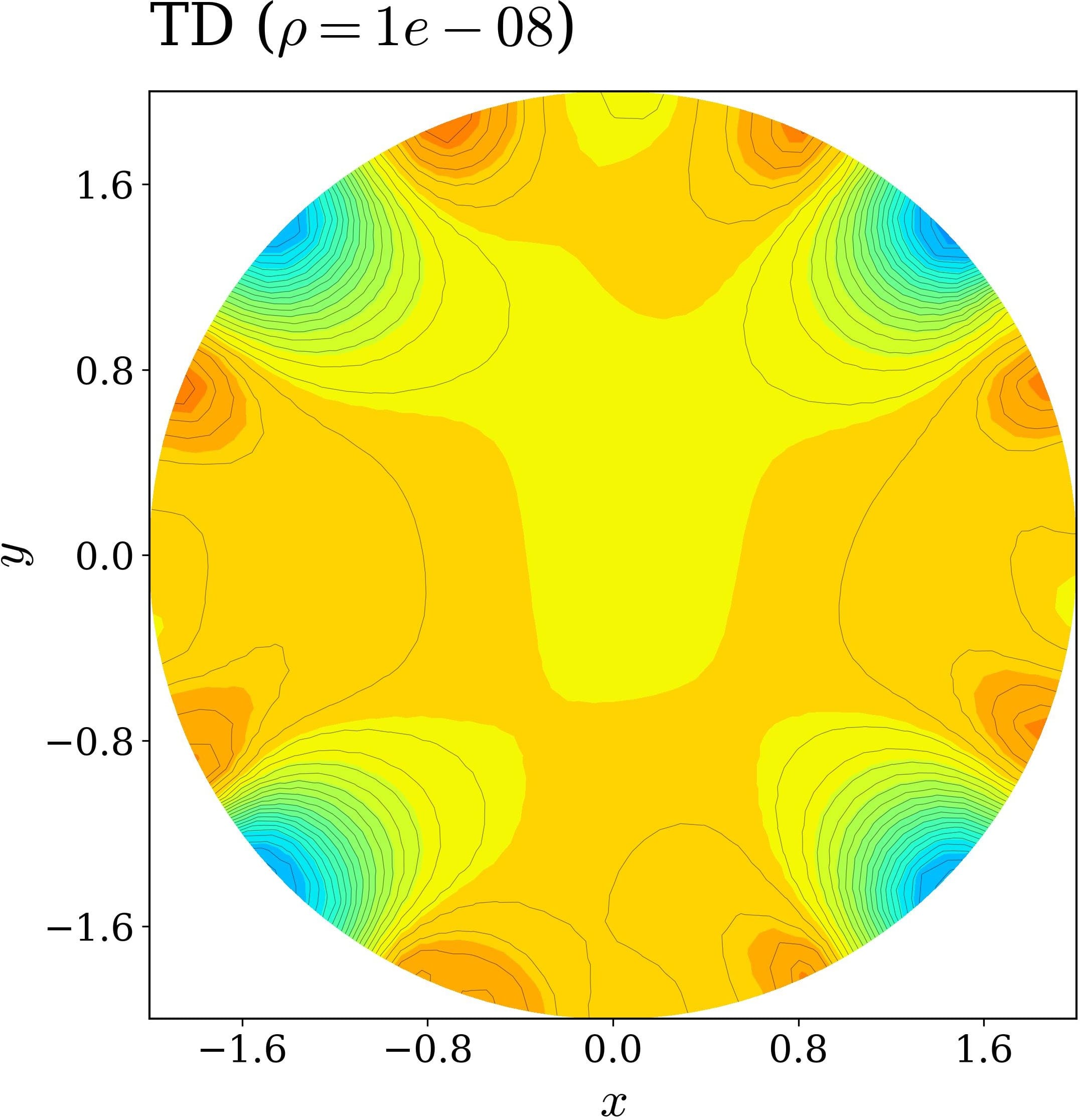}} \ 
\resizebox{0.225\textwidth}{!}{\includegraphics{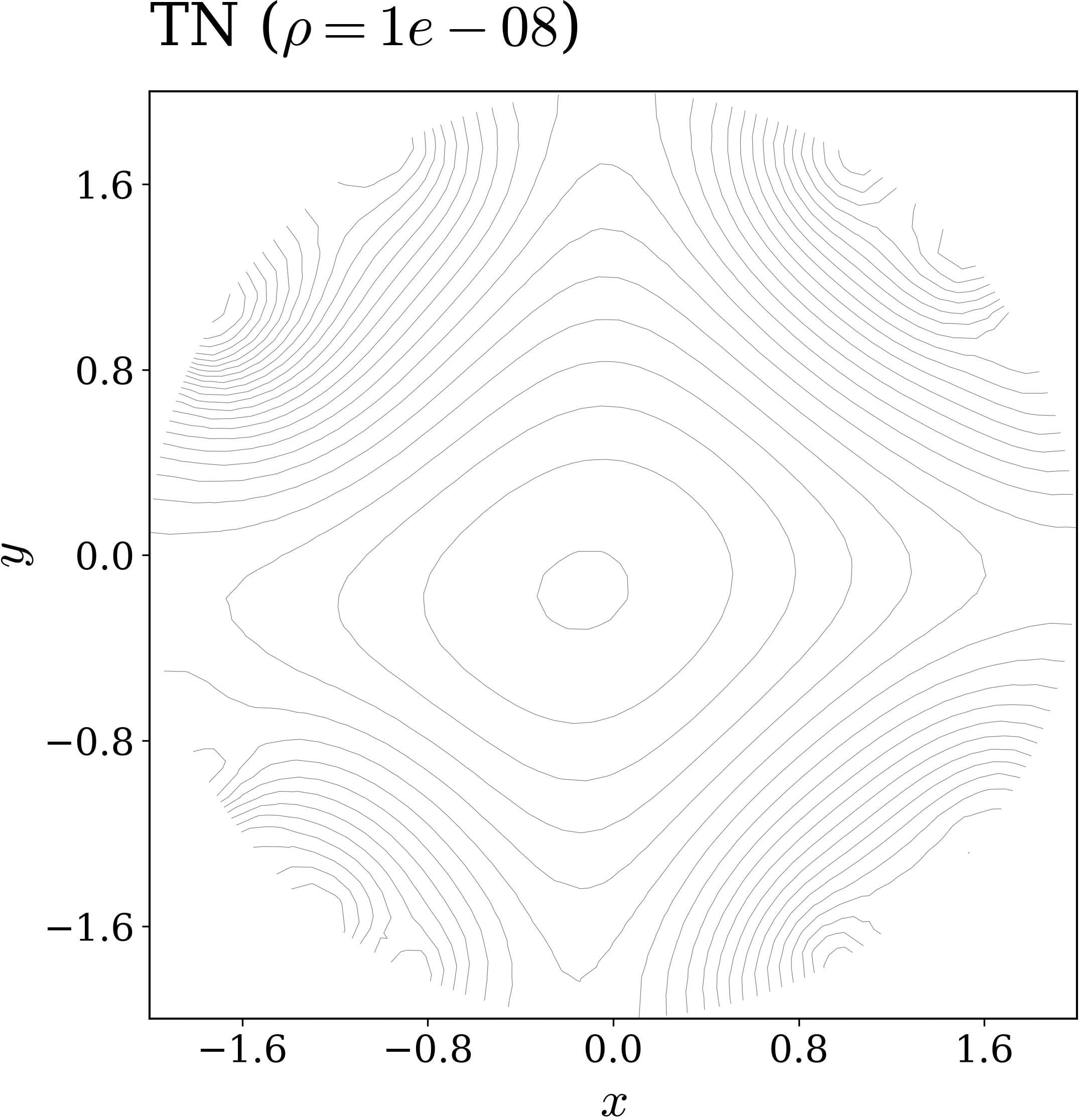}} \\[1em]
\resizebox{0.225\textwidth}{!}{\includegraphics{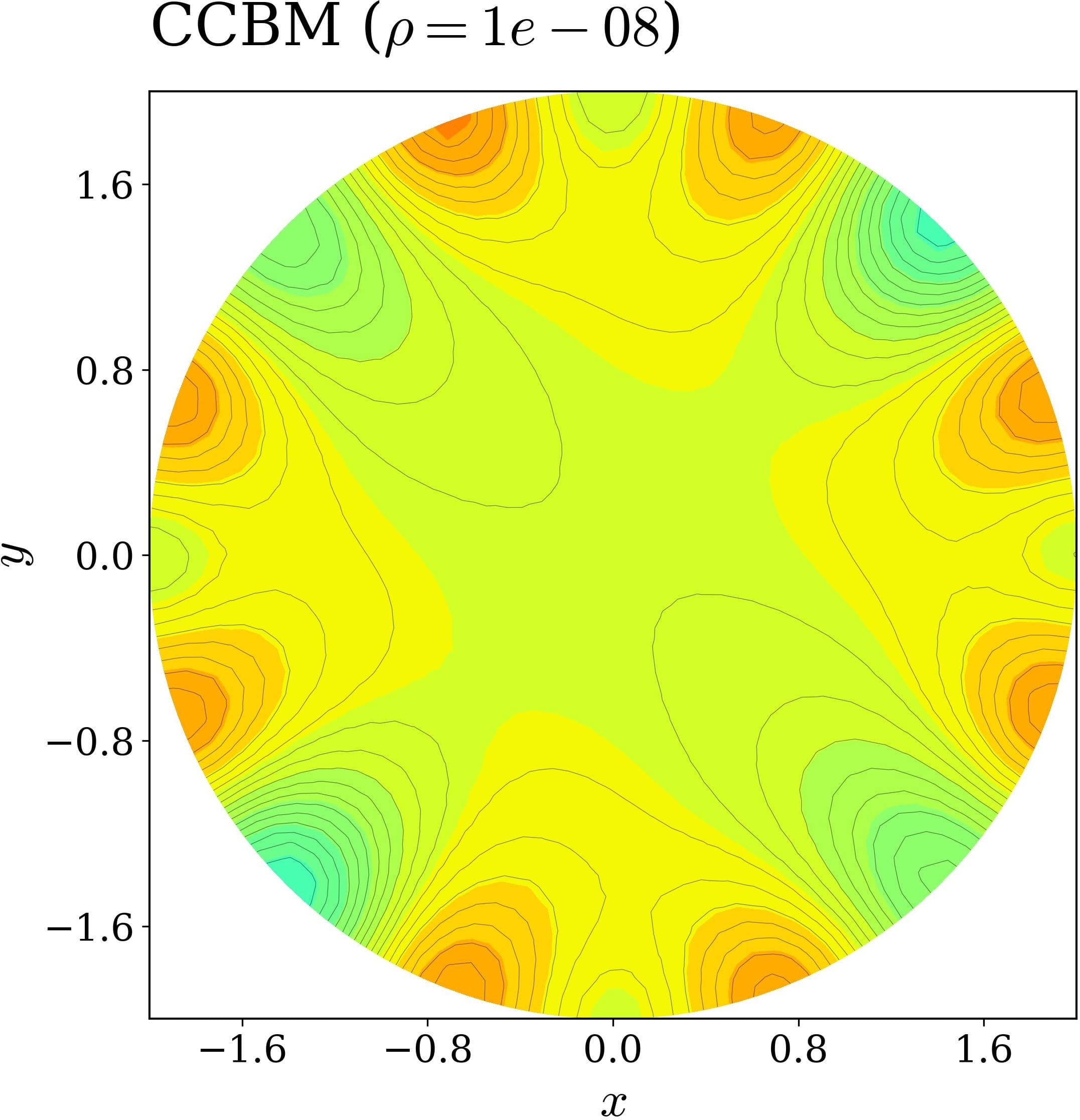}} \ 
\resizebox{0.225\textwidth}{!}{\includegraphics{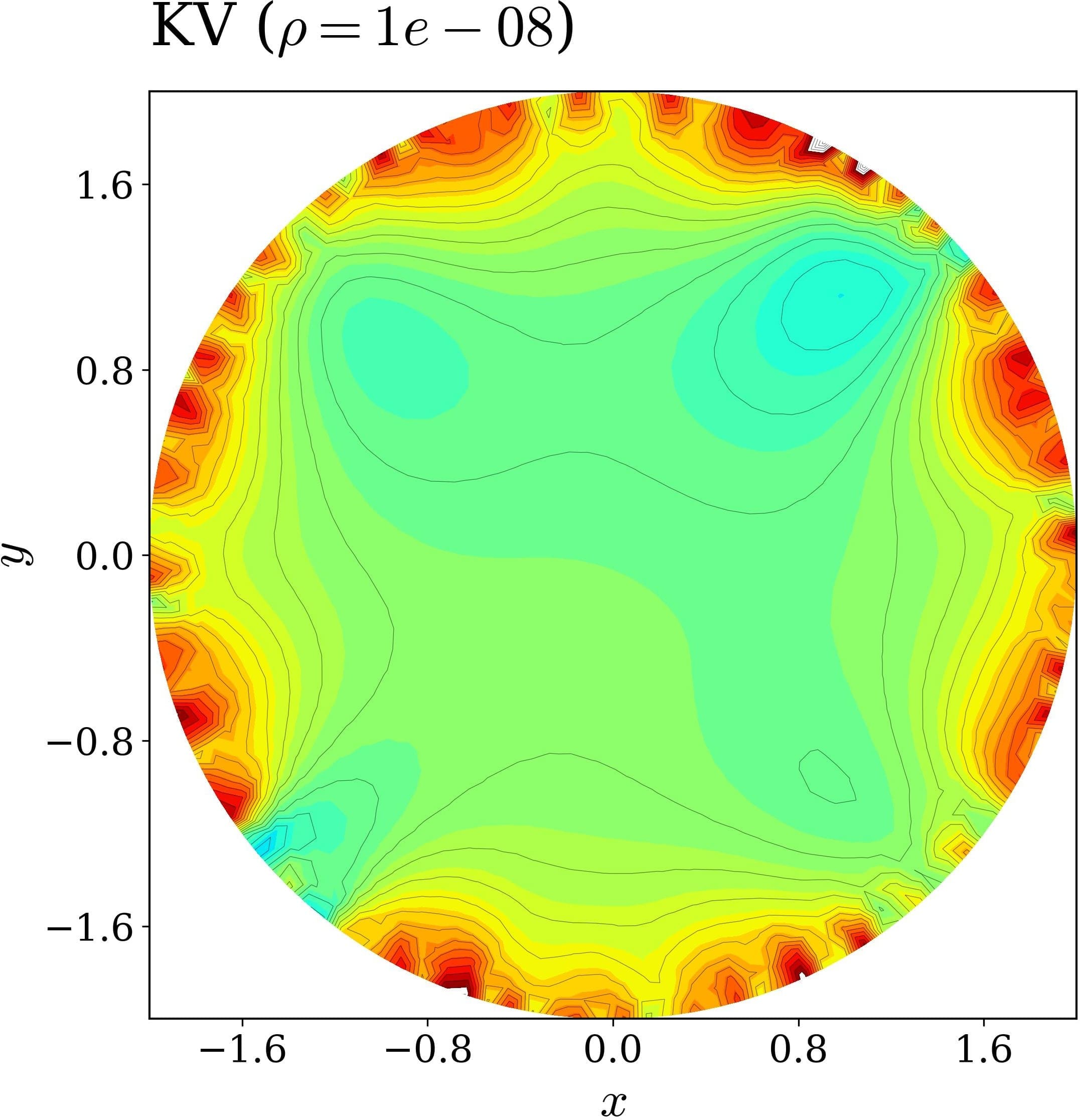}} \ 
\resizebox{0.225\textwidth}{!}{\includegraphics{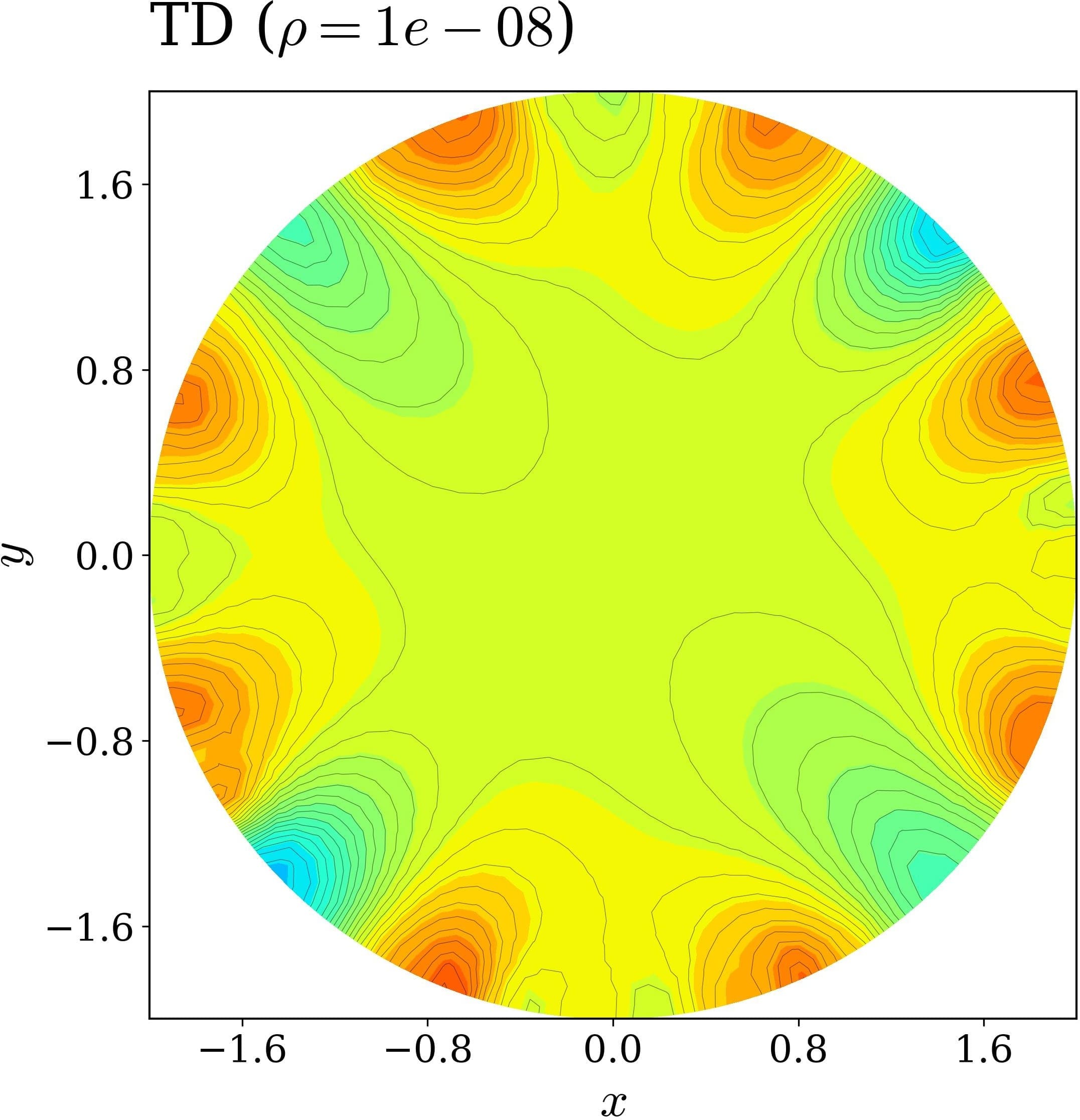}} \ 
\resizebox{0.225\textwidth}{!}{\includegraphics{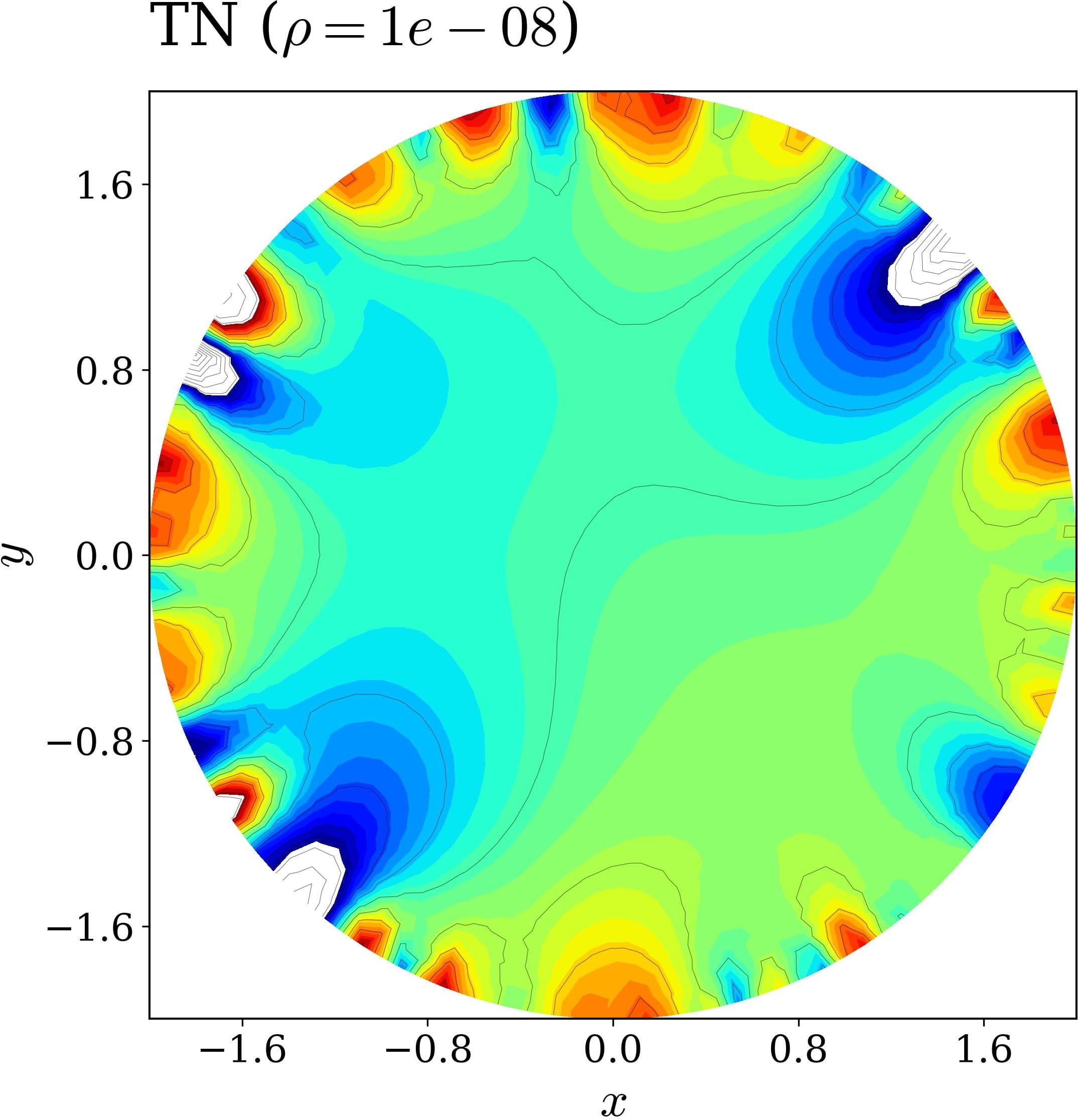}} \\[1em]
\resizebox{0.225\textwidth}{!}{\includegraphics{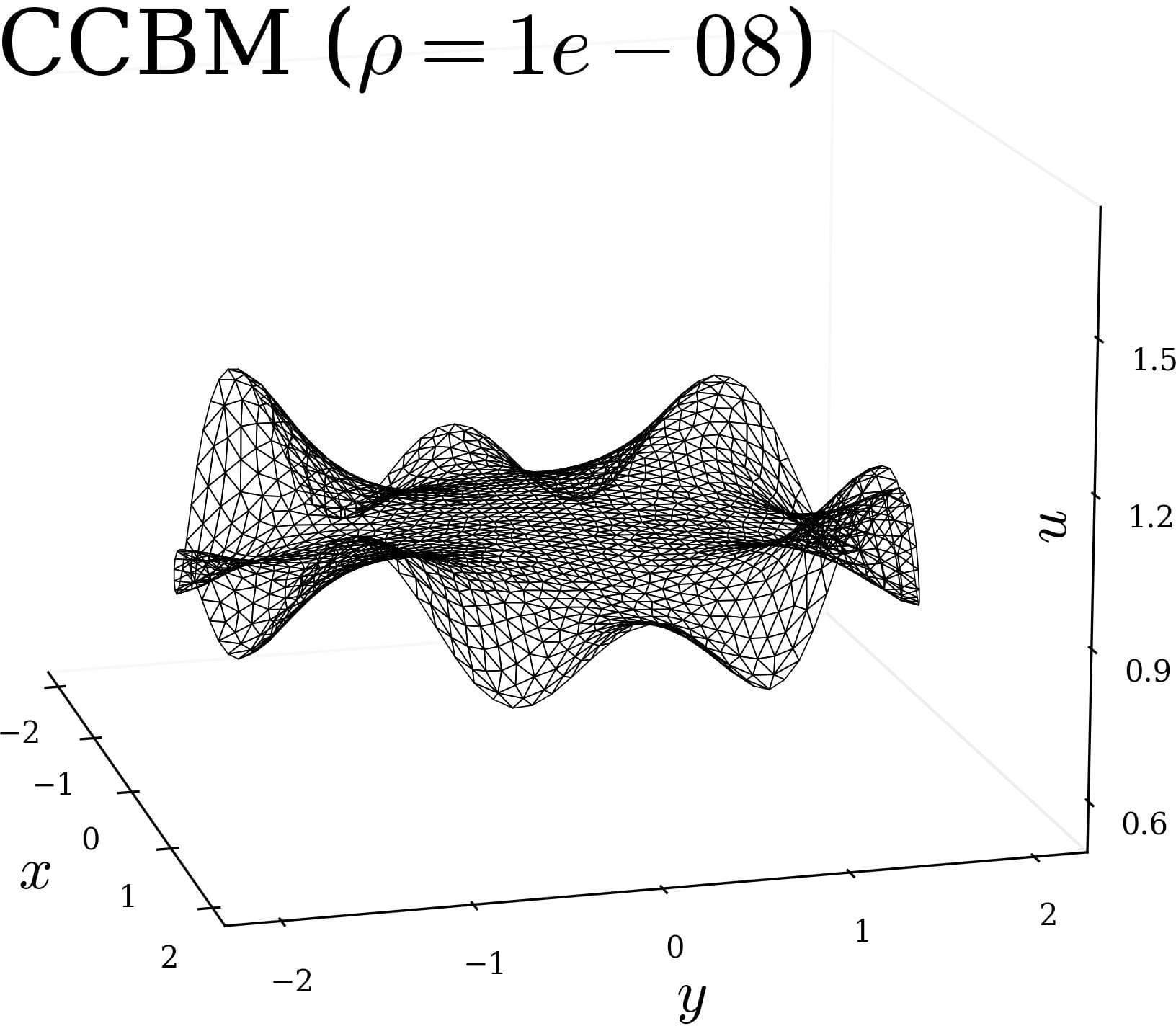}} \ 
\resizebox{0.225\textwidth}{!}{\includegraphics{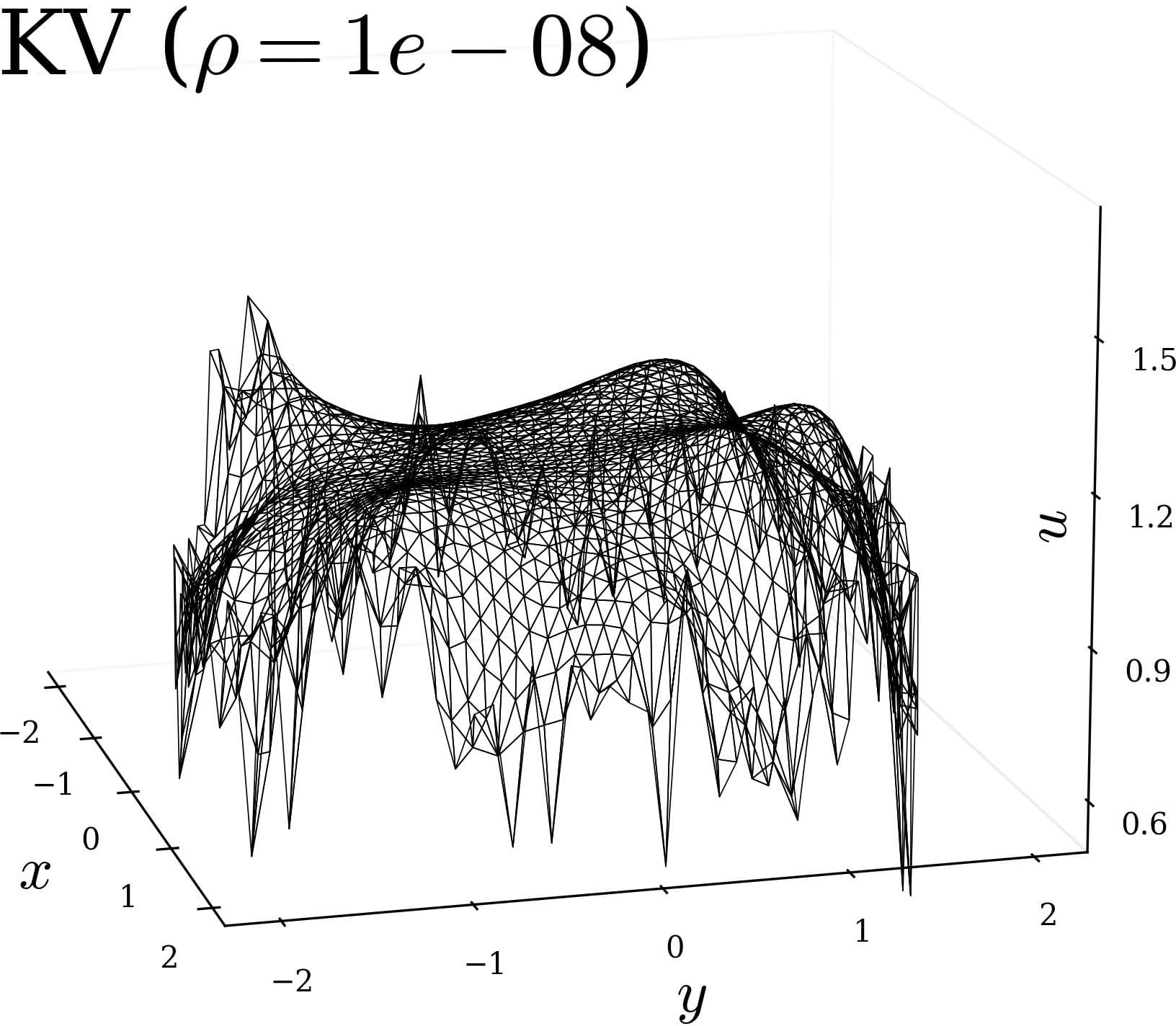}} \ 
\resizebox{0.225\textwidth}{!}{\includegraphics{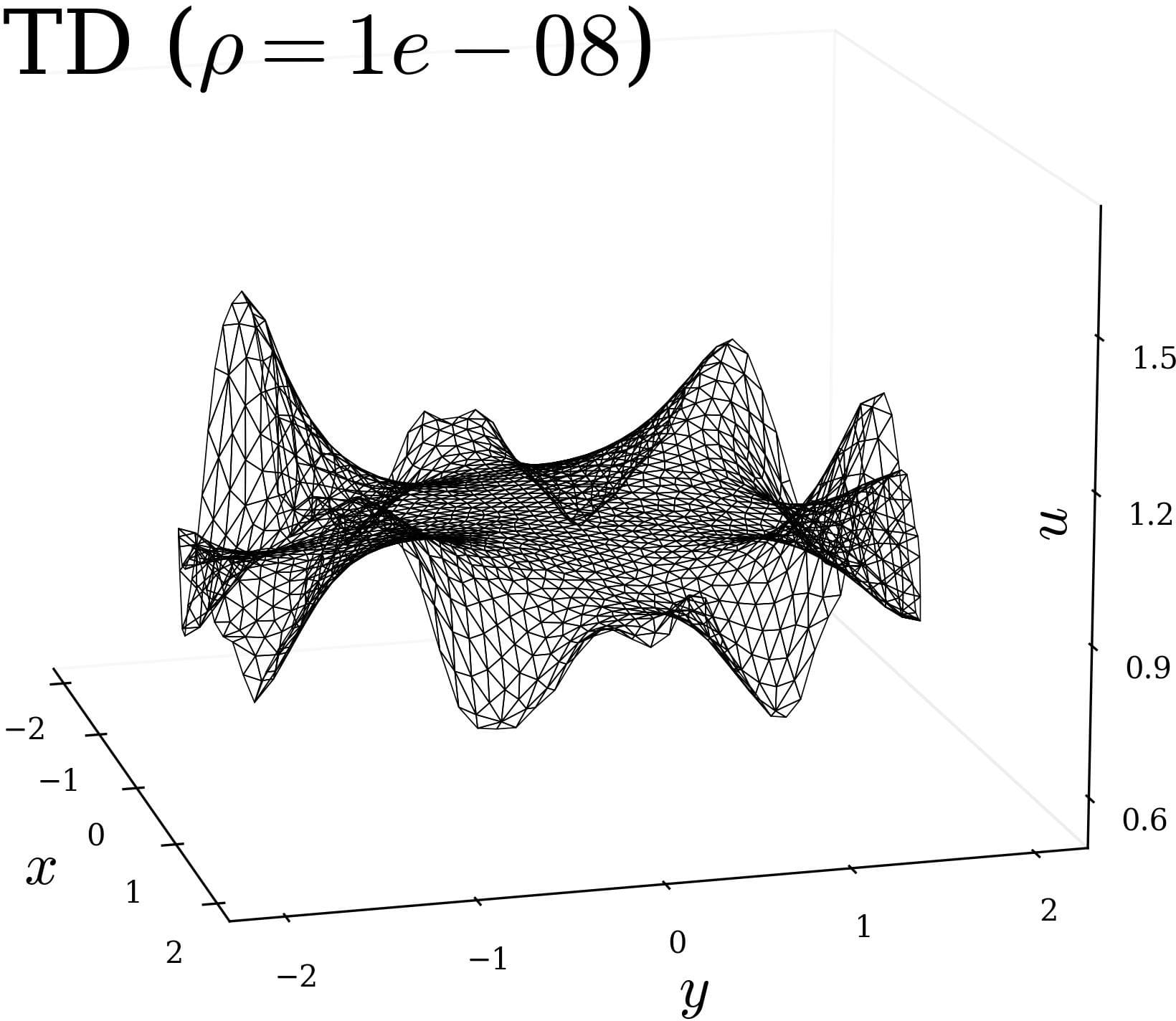}} \ 
\resizebox{0.225\textwidth}{!}{\includegraphics{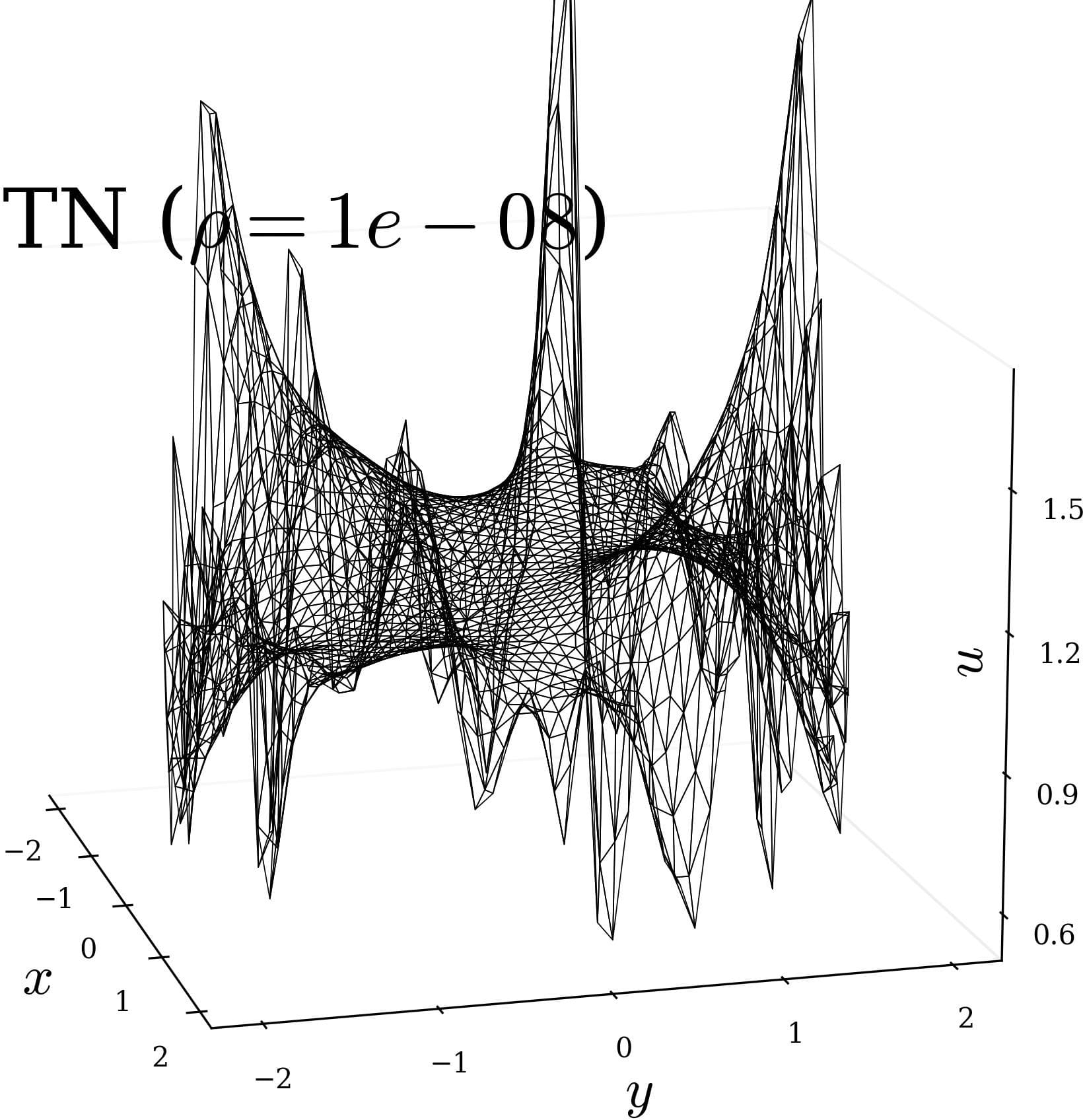}} \\[1em]
\resizebox{0.48\textwidth}{!}{\includegraphics{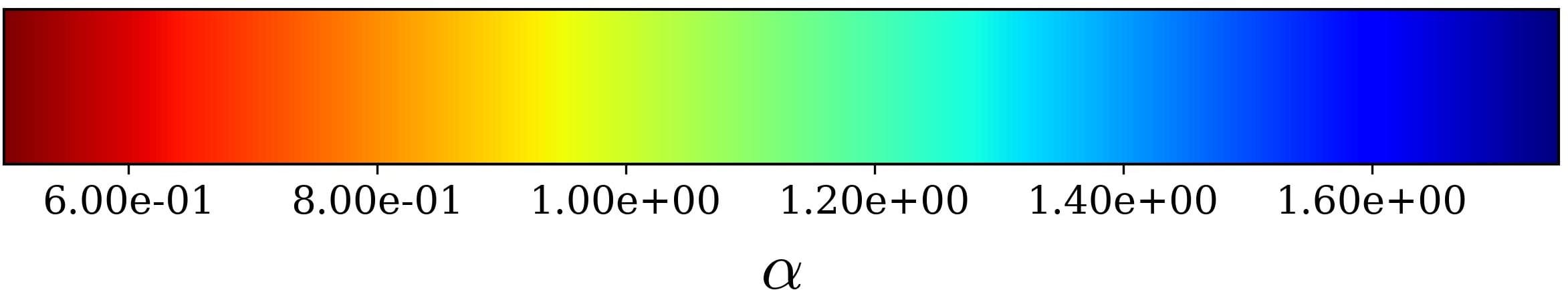}}
\caption{Iso-values of the reconstructed mildly oscillatory diffusion
$\alpha^{\star} = 1 + 0.25 x y \sin(\pi x)\sin(\pi y)$,obtained with gradient smoothing ($\mu=1.0$) from noisy measurements with noise level $\delta=0.001$ under different boundary inputs: $g=1$ (second row), $g=\sin(\pi x)\sin(\pi y)$ (third row), and $g=2+\sin(\pi x)\sin(\pi y)$ (fourth row). The first row shows the reference diffusion.}
\label{fig:reconstructions_mildly_oscillatory_diffusion}
\end{figure}
%
\begin{figure}[htp!]
\centering
\resizebox{0.225\textwidth}{!}{\includegraphics{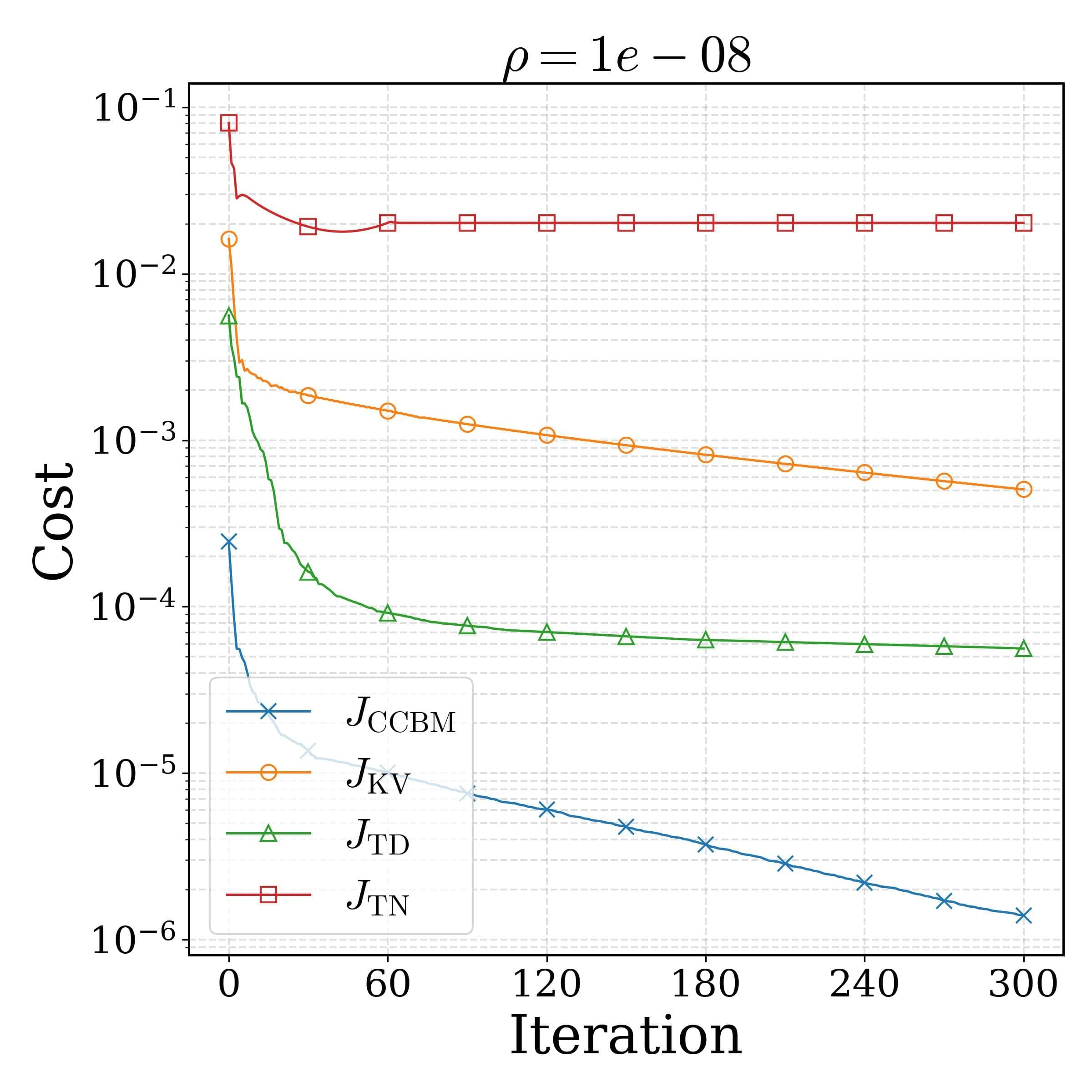}} \ 
\resizebox{0.225\textwidth}{!}{\includegraphics{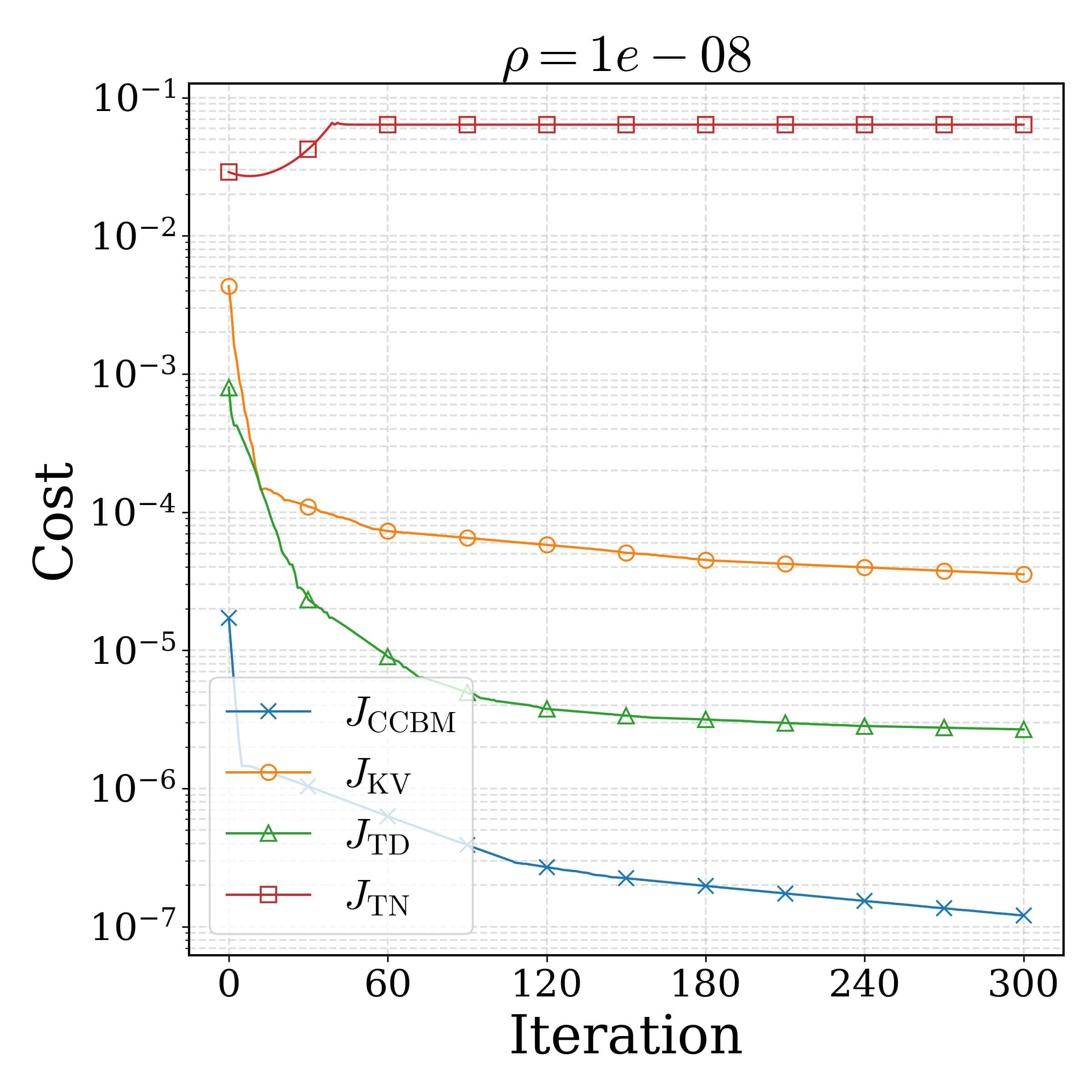}} \ 
\resizebox{0.225\textwidth}{!}{\includegraphics{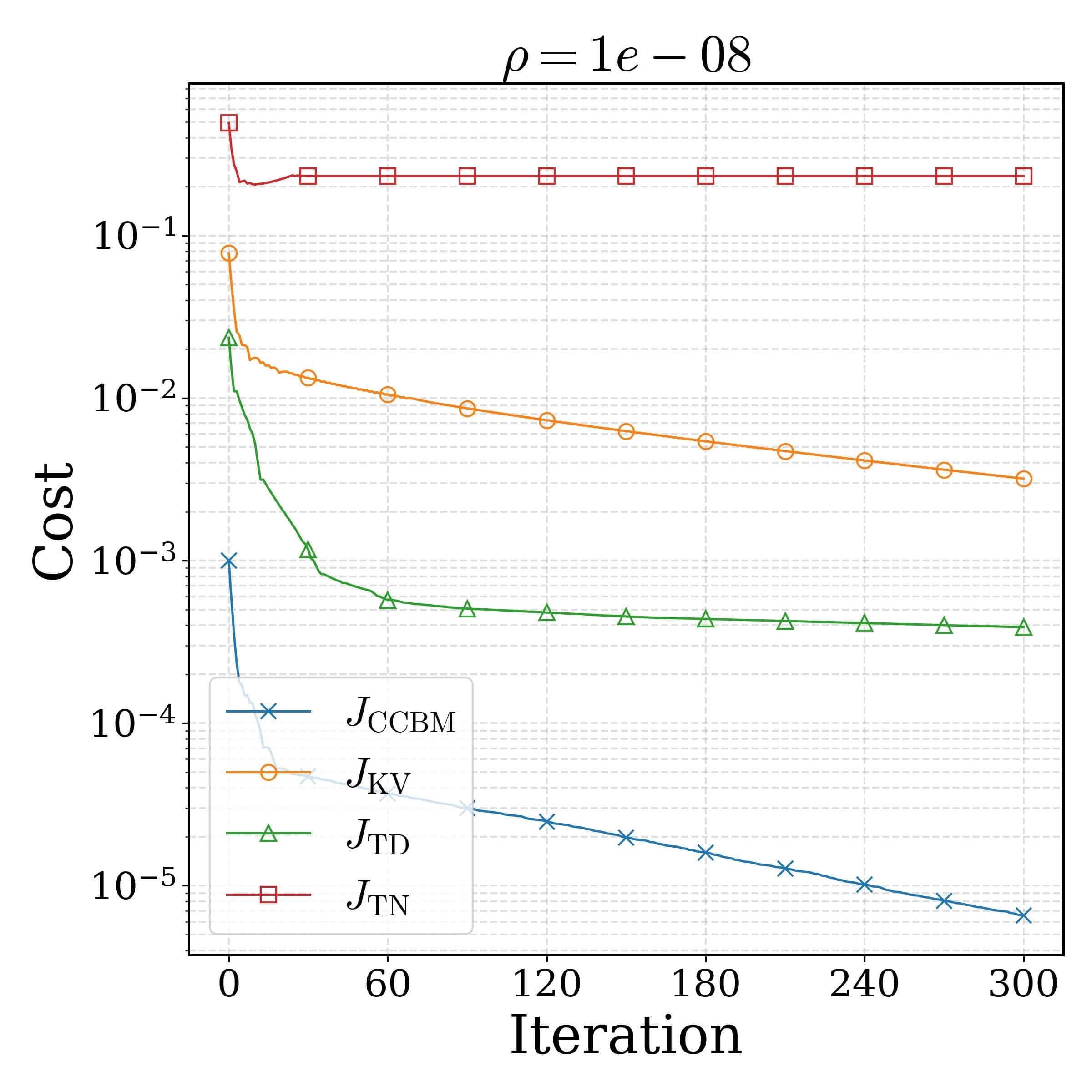}} \\[1em]
\resizebox{0.225\textwidth}{!}{\includegraphics{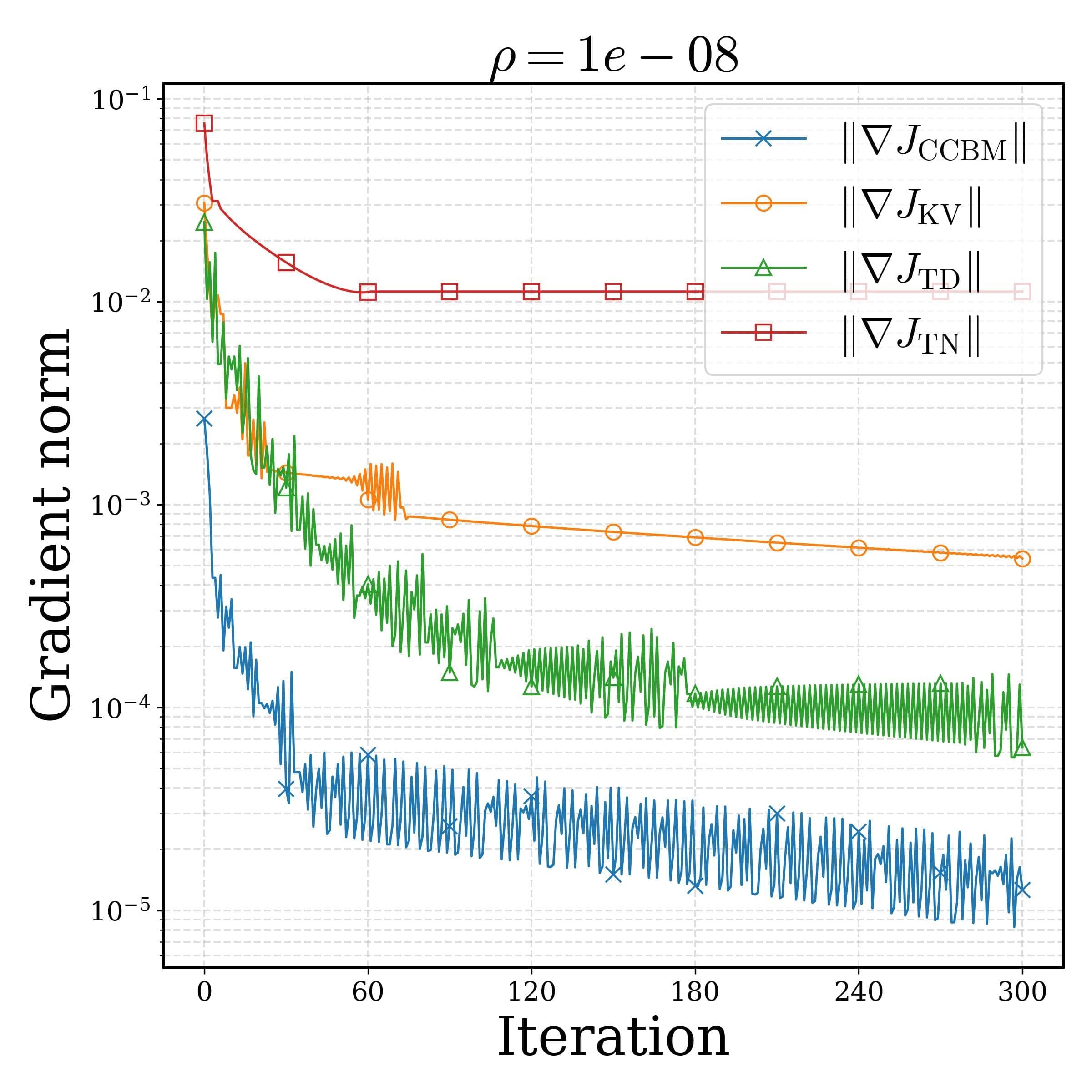}} \
\resizebox{0.225\textwidth}{!}{\includegraphics{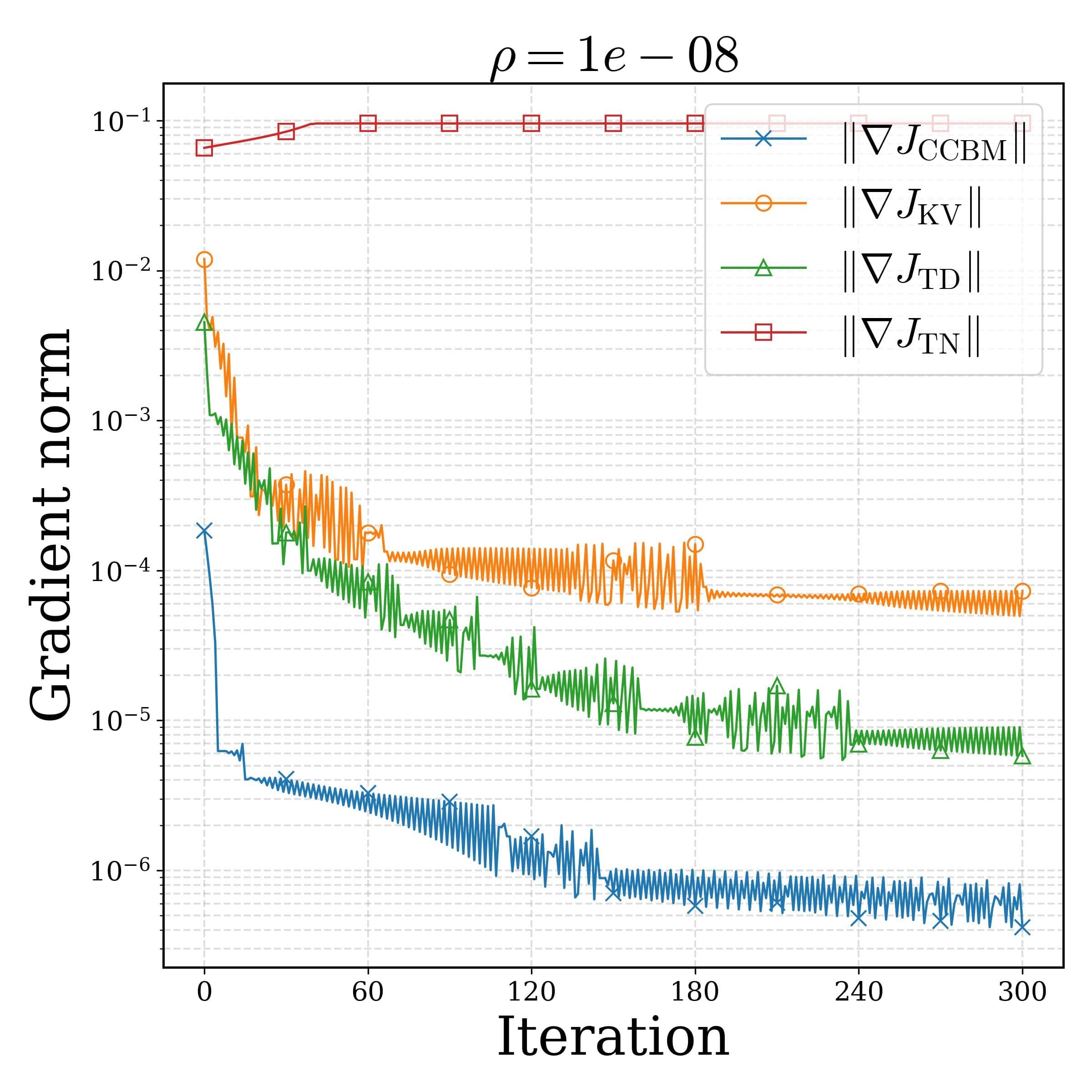}} \
\resizebox{0.225\textwidth}{!}{\includegraphics{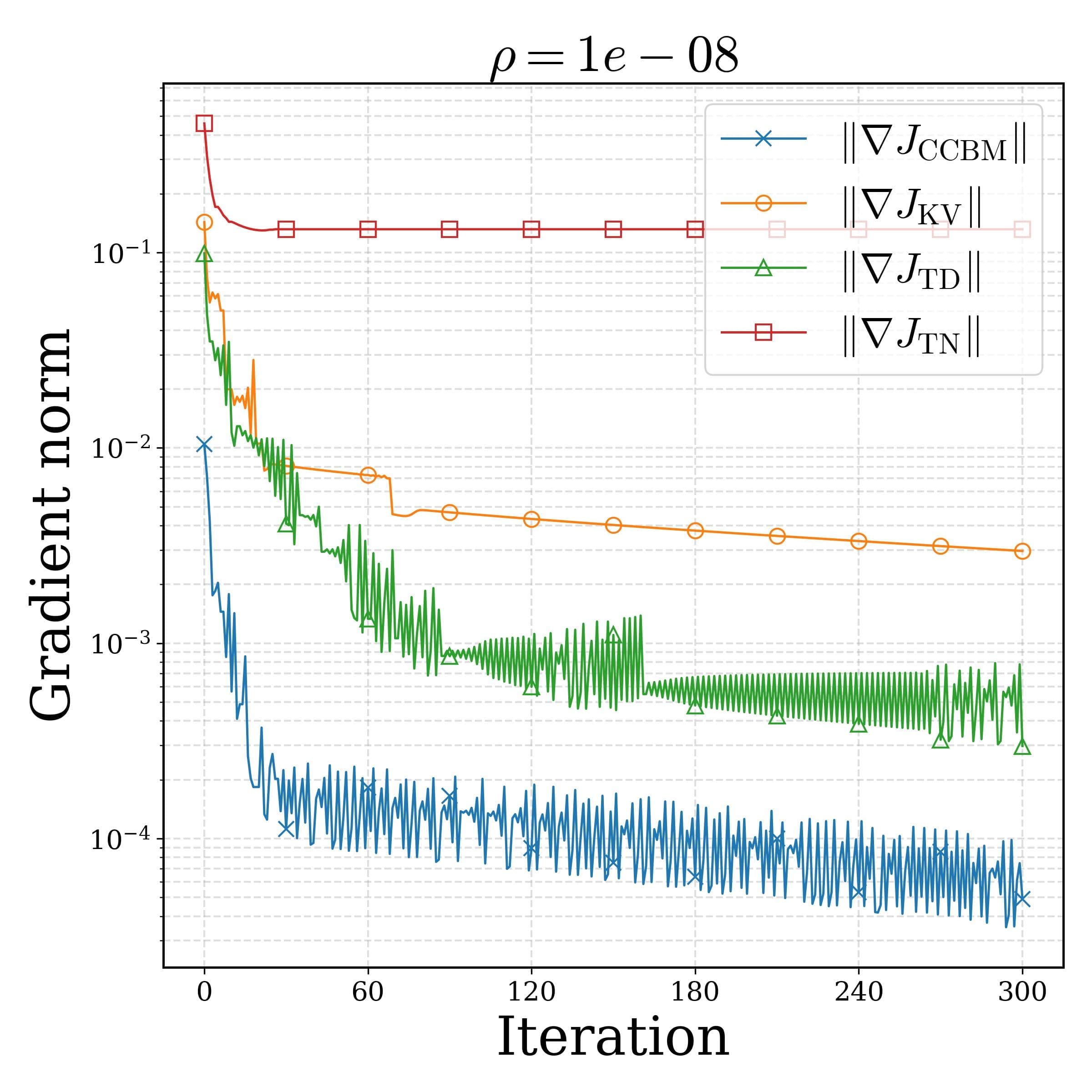}}
\caption{Histories of the cost functional and the gradient norm corresponding to Figure~\ref{fig:reconstructions_mildly_oscillatory_diffusion}. From left to right: $g=1$, $g=\sin(\pi x)\sin(\pi y)$, and $g=2+\sin(\pi x)\sin(\pi y)$.}
\label{fig:reconstructions_mildly_oscillatory_diffusion_cost_and_gradient}
\end{figure}
%
%
%
%
%
%
\subsubsection{Effect of incorporating the $H^{1}$ term in the cost functional} \label{subsec:effect_of_energy_minimization}
We now investigate the influence of incorporating a gradient term into the cost functional.
The experiments are performed on the domain $\varOmega = C(0,2)$ with exact diffusion
\[
\alpha_{1}^{\star}(x,y) = 1 + 0.5 x y
\quad \text{or} \quad
\alpha_{2}^{\star}(x,y) = 1 + \sin\!\left(\dfrac{\pi x}{8}\right)\sin\!\left(\dfrac{\pi y}{8}\right),
\]
starting from the constant initial guess $\alpha^{[0]} \equiv 1$. 
We set $b = 0$, $c = 1$, and $Q = 1$, and generate synthetic data using $f = 0$ and boundary input $g_1 = 1$ or $g_2 = \sin{\pi x} \sin{\pi y}$, corrupted with noise level $\delta = 0.001, 0.003, 0.005, 0.01$. 
The gradient regularization parameter is fixed at $\mu = 0.1$, while the Tikhonov parameter is set to $\rho = 0.001$.

Figures~\ref{fig:effect_of_weight_pringle}--\ref{fig:effect_of_weight_highly_oscillating_g2} show the effect of the weight parameter $w_{1}$ associated with the $H^{1}$-type term in the cost functional.
In the reported experiments, intermediate values of $w_{1}$ are often associated with lower reconstruction errors in both $L^{2}$ and $H^{1}$ norms, whereas very small weights provide limited regularization and very large weights may lead to over-smoothing. This trend is observed across different diffusion profiles, boundary inputs, and noise levels, although the optimal range of $w_{1}$ remains problem-dependent.

As noise increases, the inclusion of the $H^{1}$-term appears to reduce high-frequency artifacts in several test cases and to contribute to more stable reconstructions. Overall, these results suggest that gradient weighting can be beneficial in the considered settings when the regularization strength is chosen appropriately.

%
\begin{figure}[htp!]
\centering
\resizebox{0.235\textwidth}{!}{\includegraphics{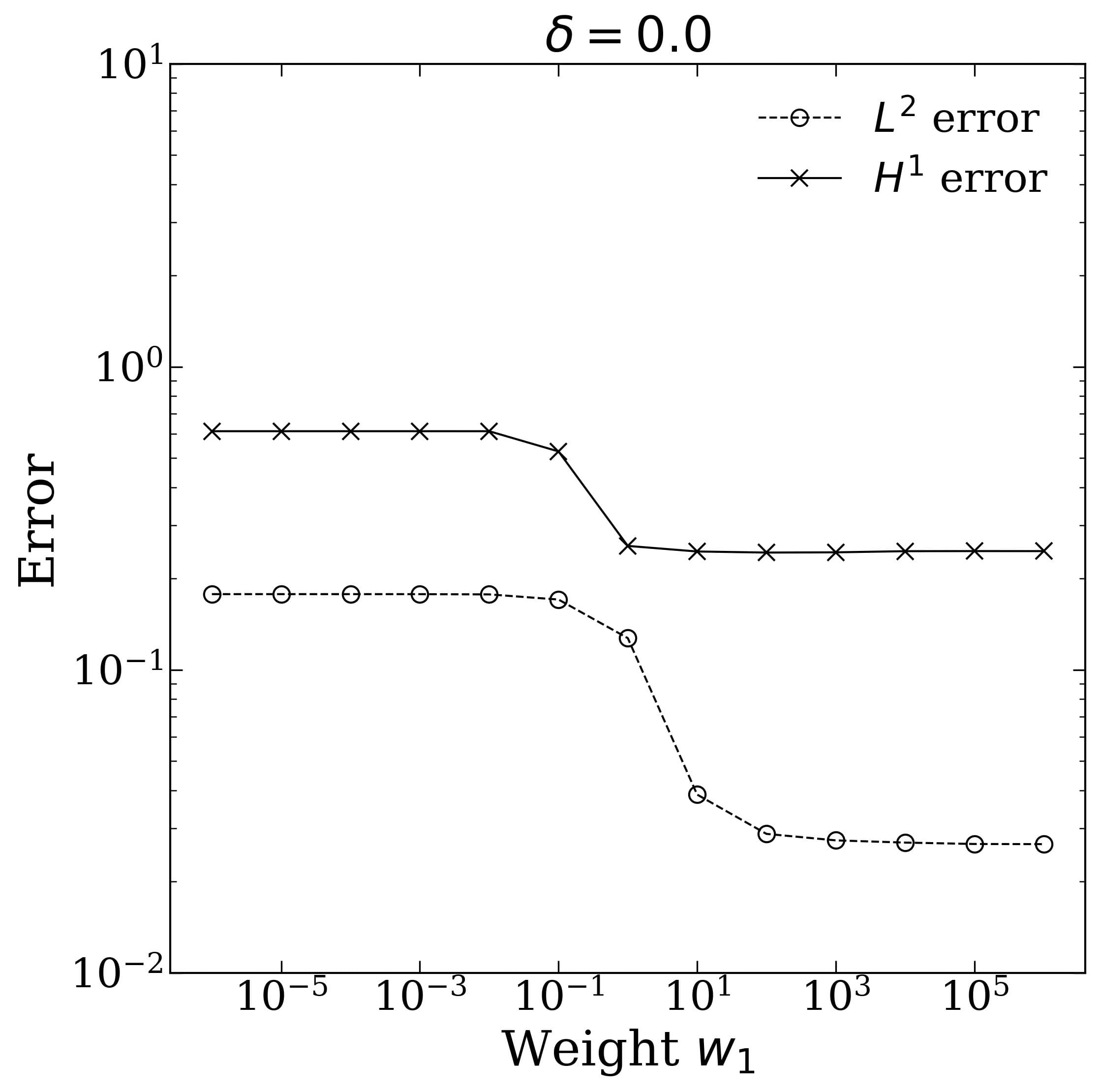}} \  
\resizebox{0.235\textwidth}{!}{\includegraphics{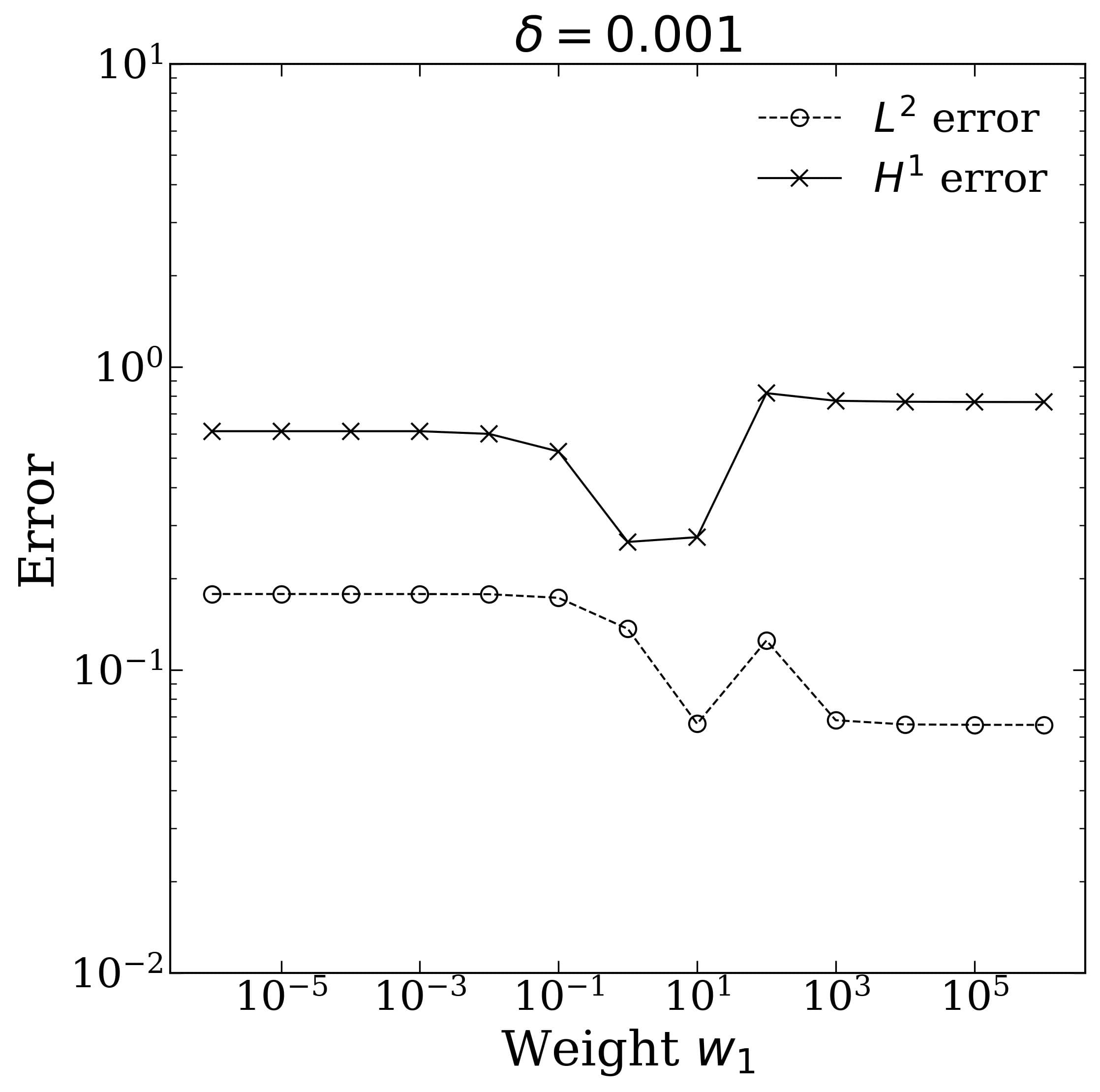}} \  
\resizebox{0.235\textwidth}{!}{\includegraphics{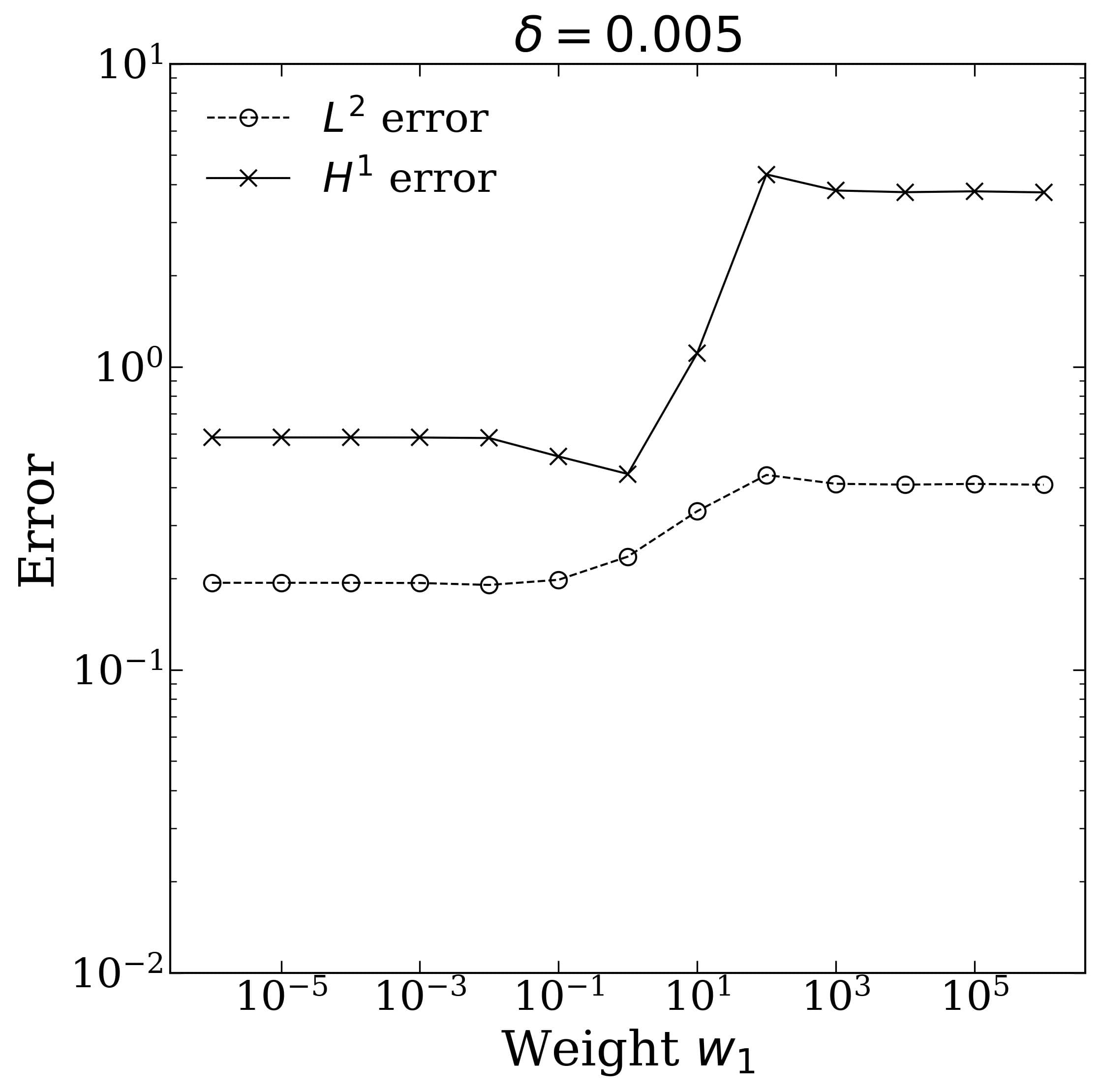}} \  
\resizebox{0.235\textwidth}{!}{\includegraphics{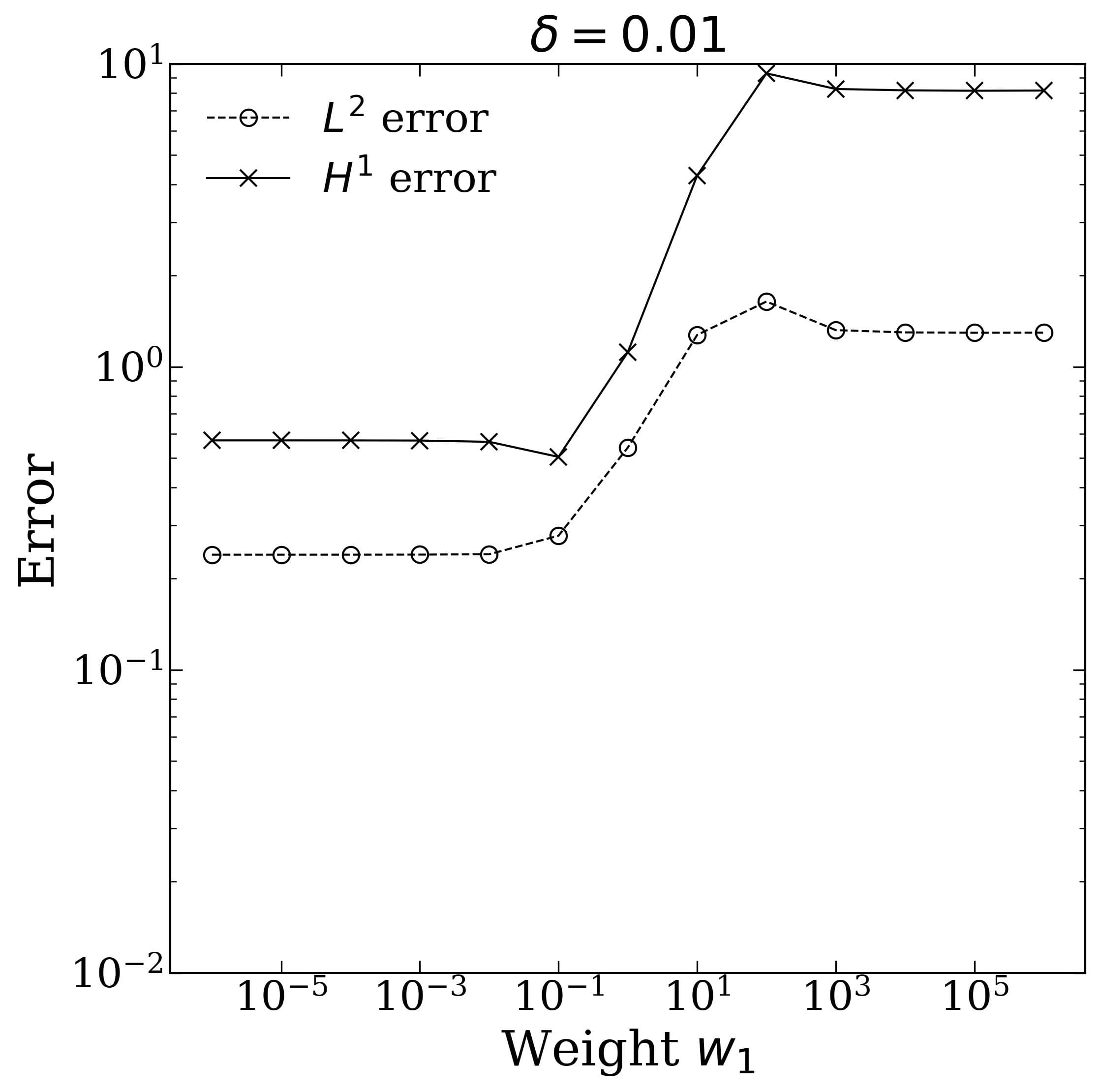}} \  
\caption{Effect of weight parameter $w_{1}$ when $\alpha^{\star} = \alpha_{1}^{\star}$ with $g = g_{1}$}
\label{fig:effect_of_weight_pringle}
\end{figure}
%
%

%
\begin{figure}[htp!]
\centering
\resizebox{0.235\textwidth}{!}{\includegraphics{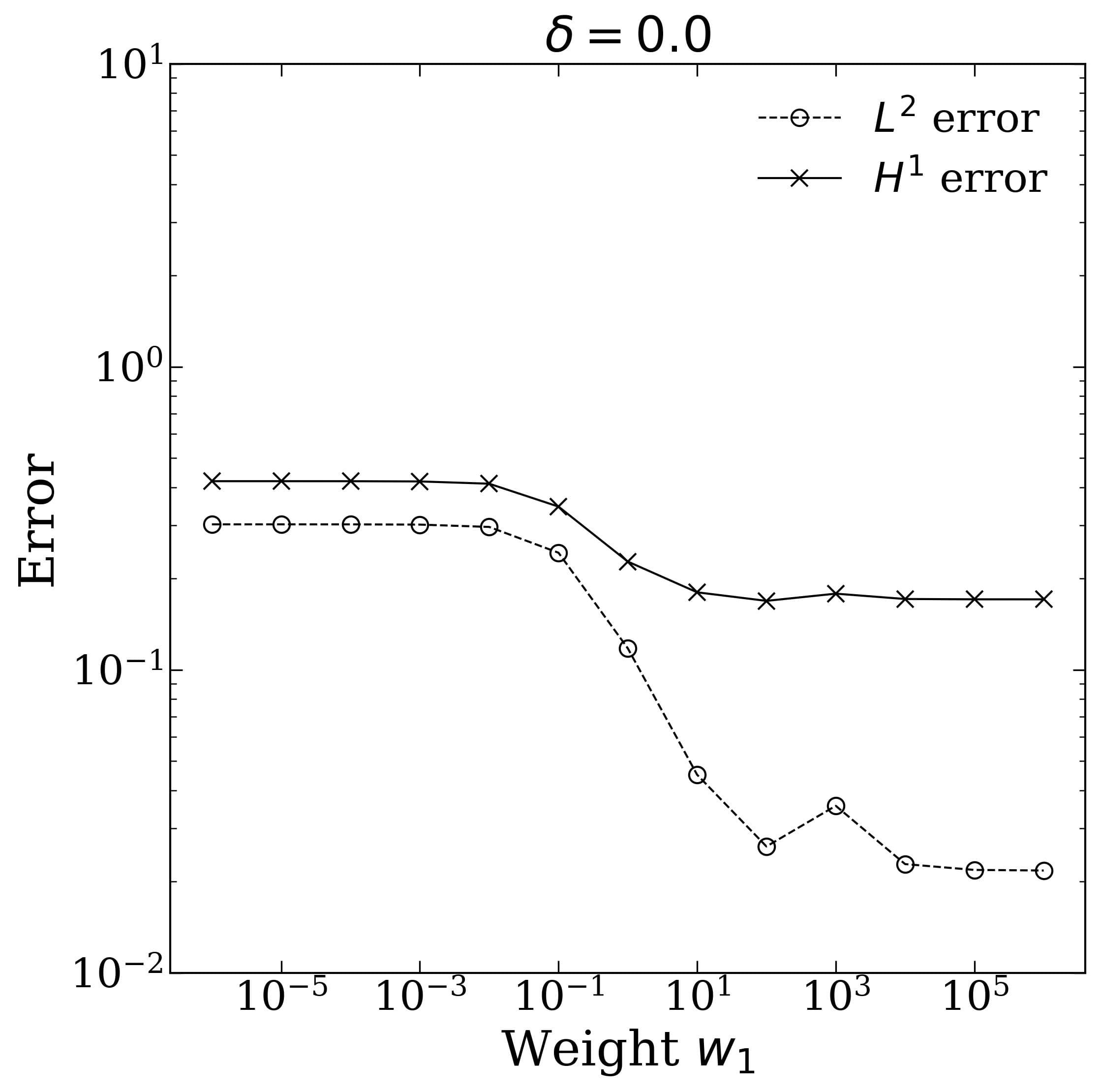}} \  
\resizebox{0.235\textwidth}{!}{\includegraphics{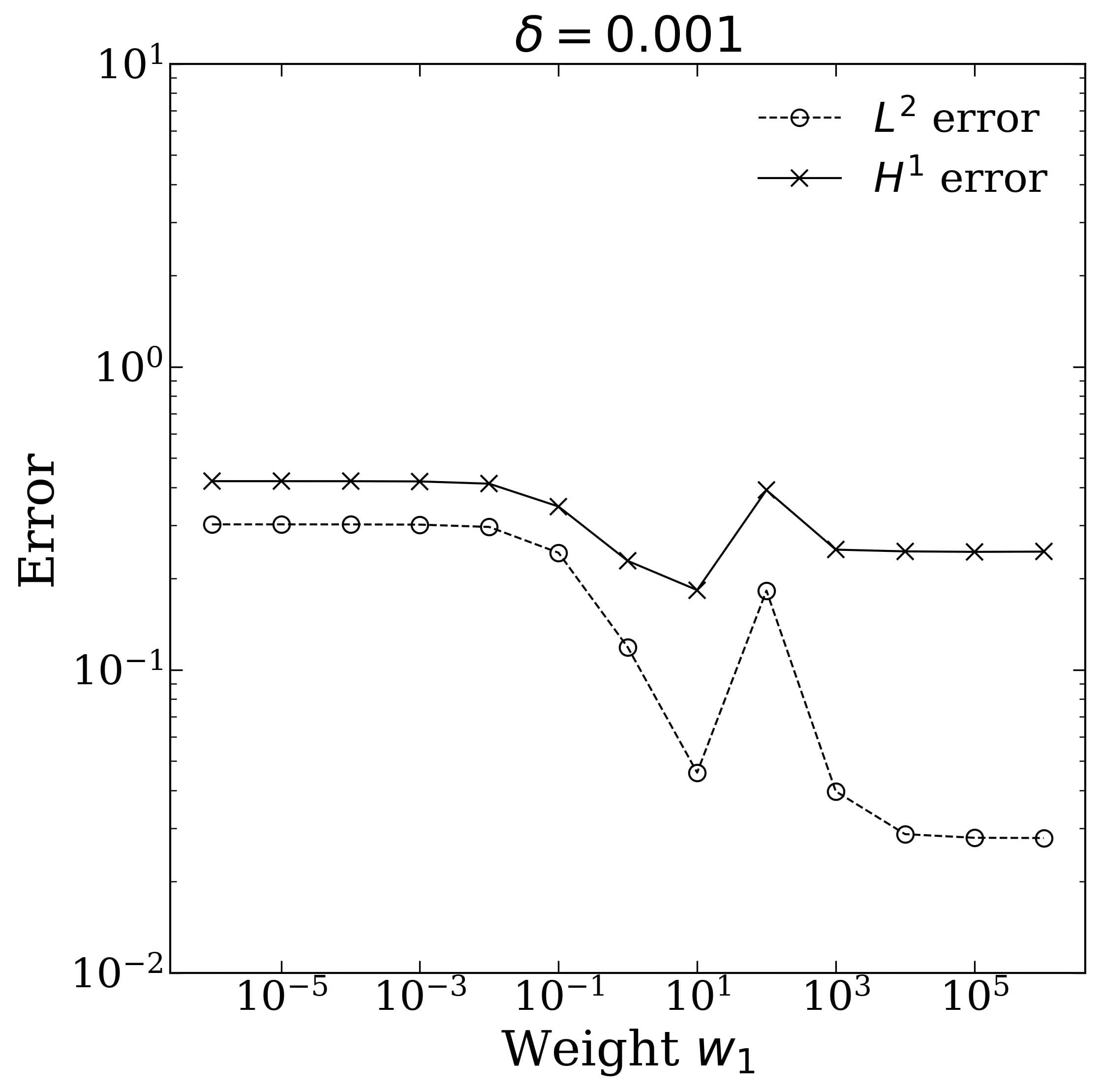}} \  
\resizebox{0.235\textwidth}{!}{\includegraphics{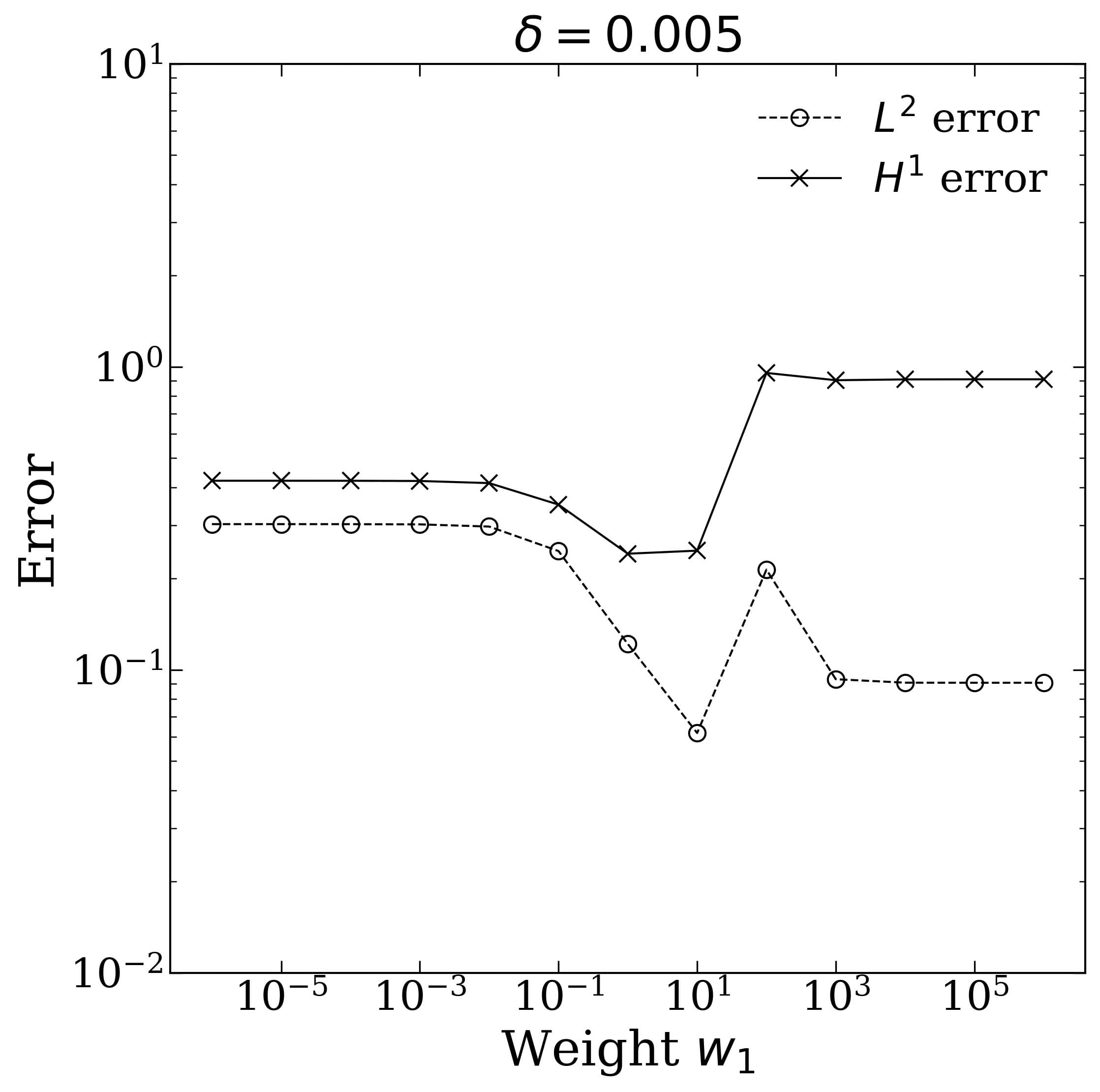}} \  
\resizebox{0.235\textwidth}{!}{\includegraphics{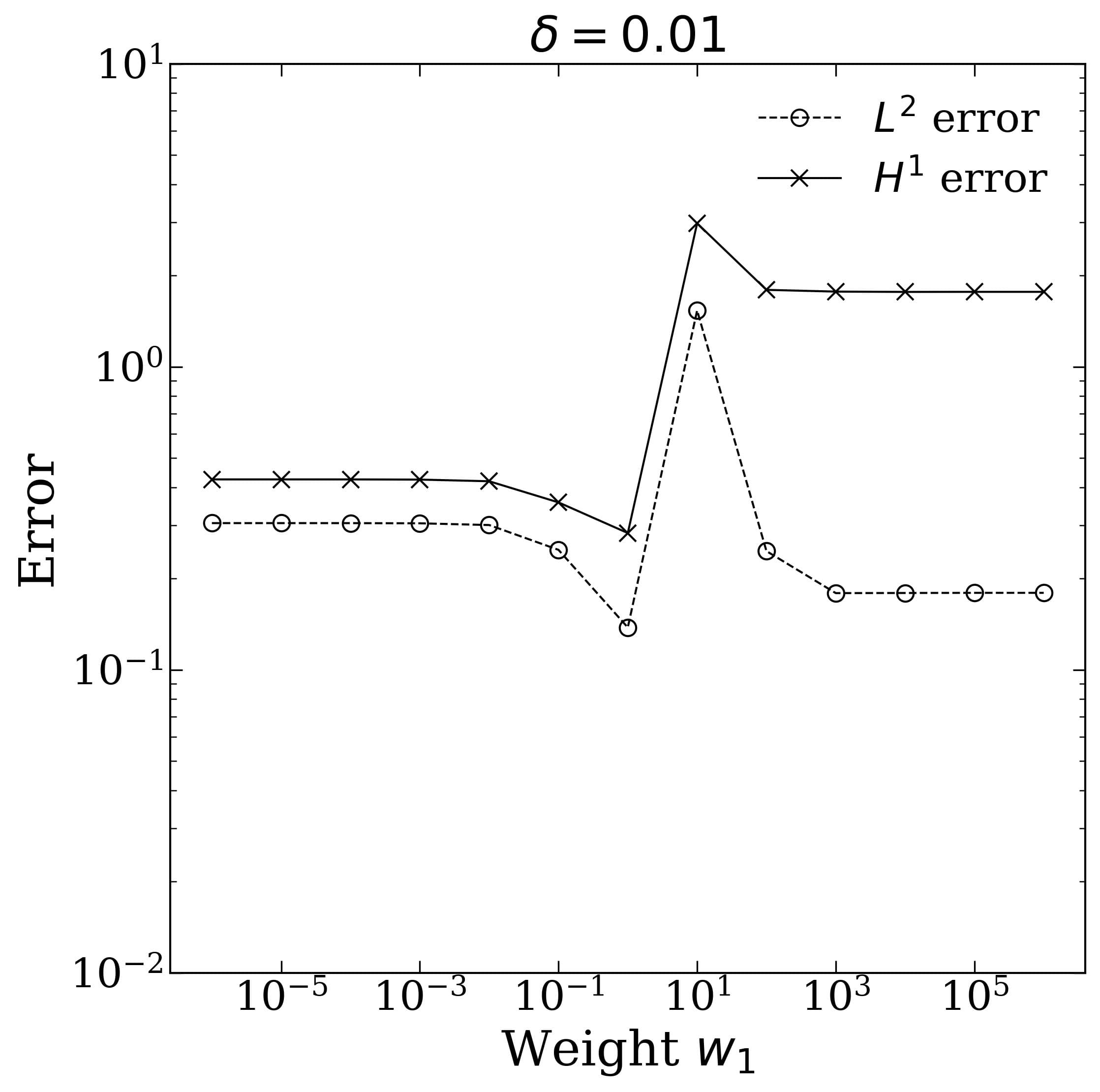}} \  
\caption{Effect of weight parameter $w_{1}$ when $\alpha^{\star} = \alpha_{2}^{\star}$ with $g = g_{1}$}
\label{fig:effect_of_weight_highly_oscillating}
\end{figure}
%
%

%
\begin{figure}[htp!]
\centering
\resizebox{0.235\textwidth}{!}{\includegraphics{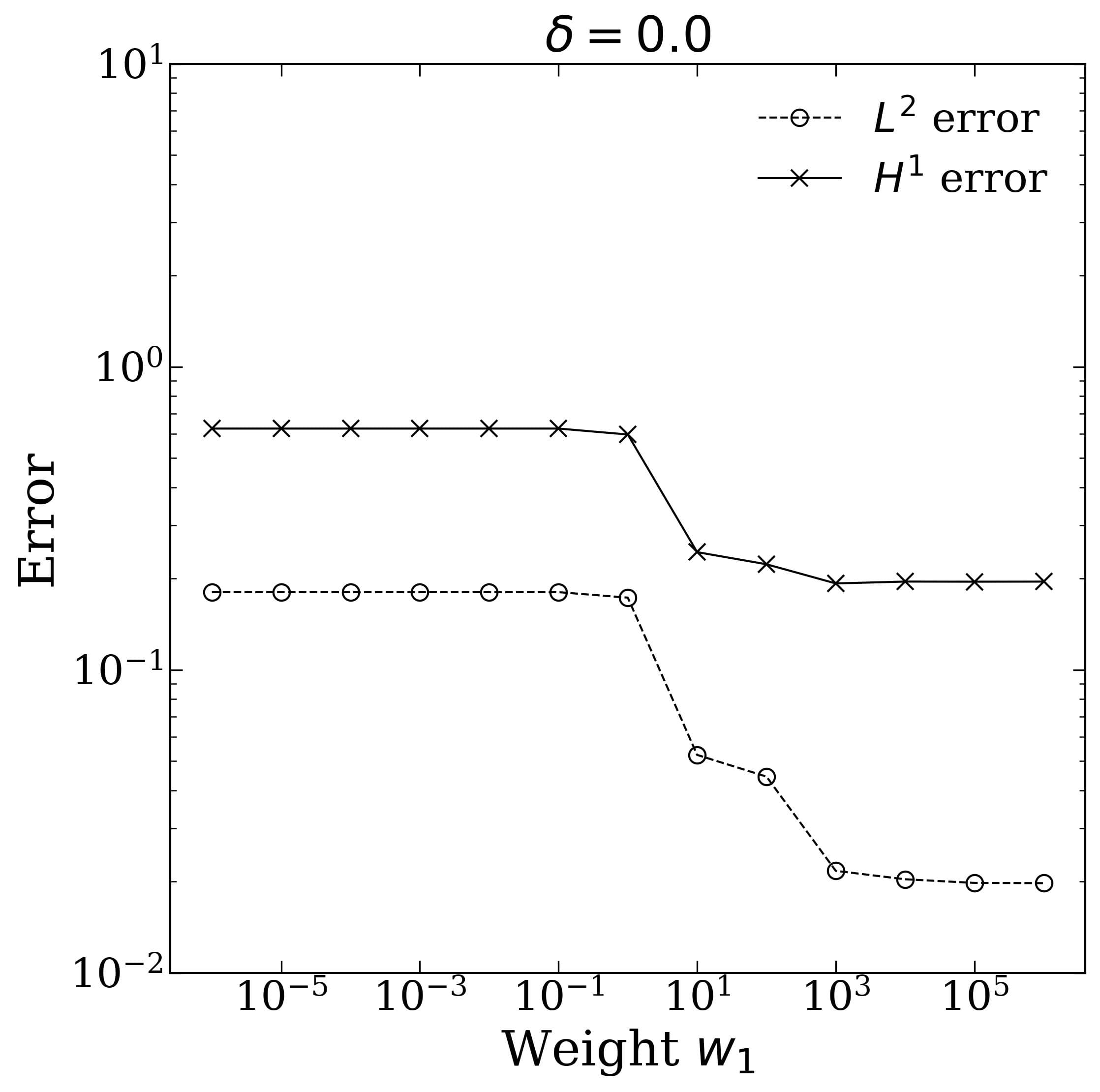}} \  
\resizebox{0.235\textwidth}{!}{\includegraphics{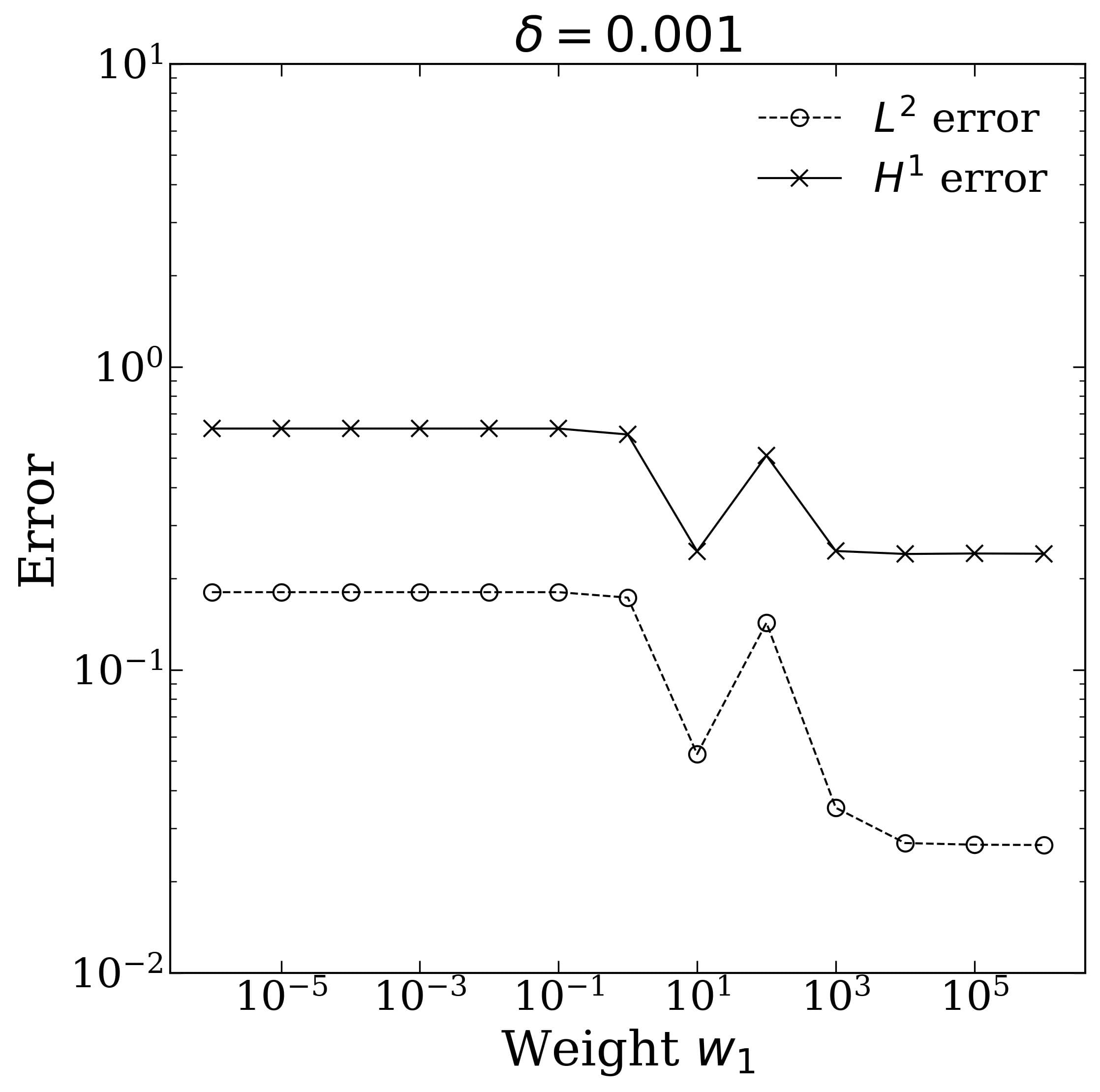}} \  
\resizebox{0.235\textwidth}{!}{\includegraphics{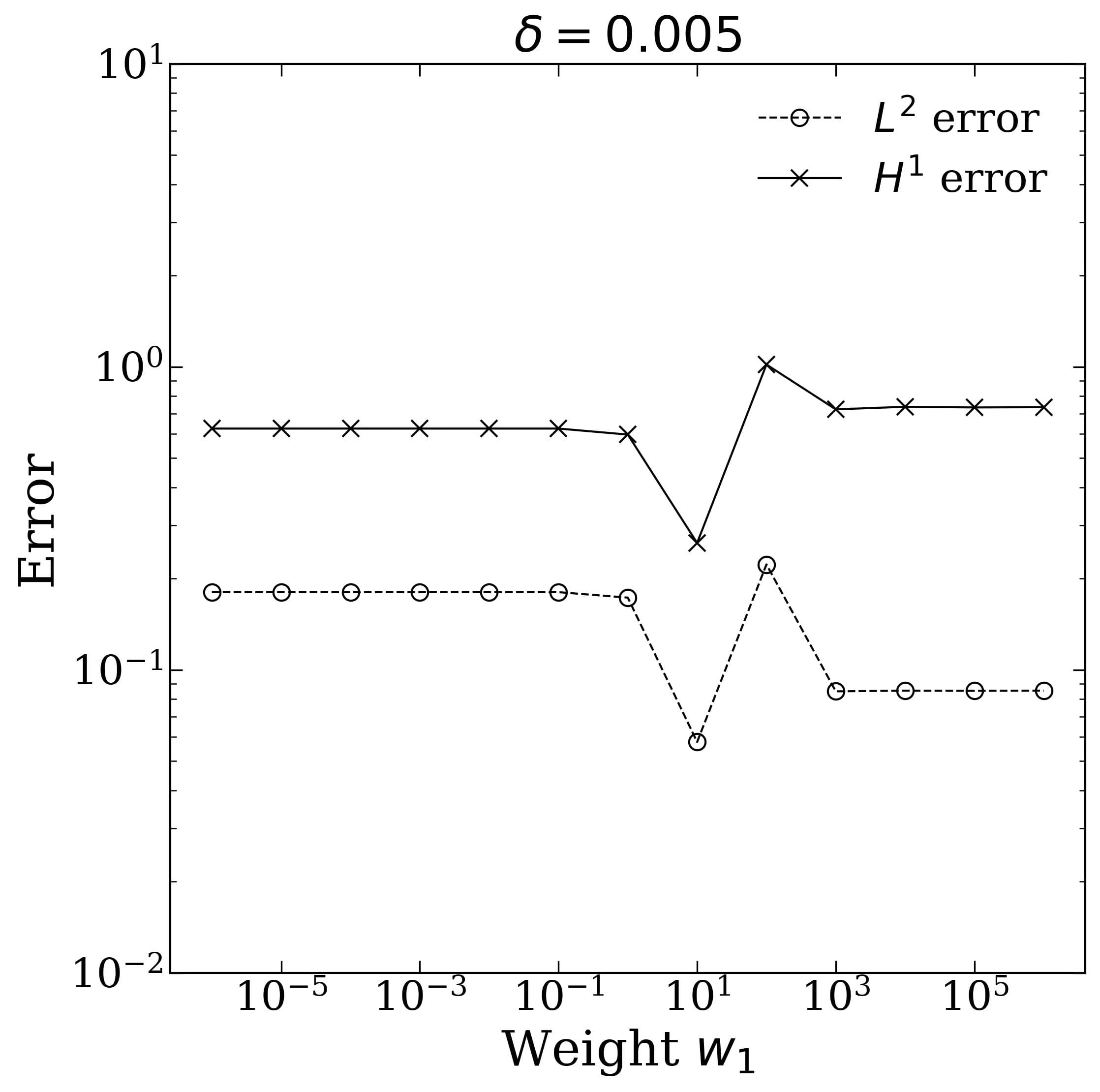}} \  
\resizebox{0.235\textwidth}{!}{\includegraphics{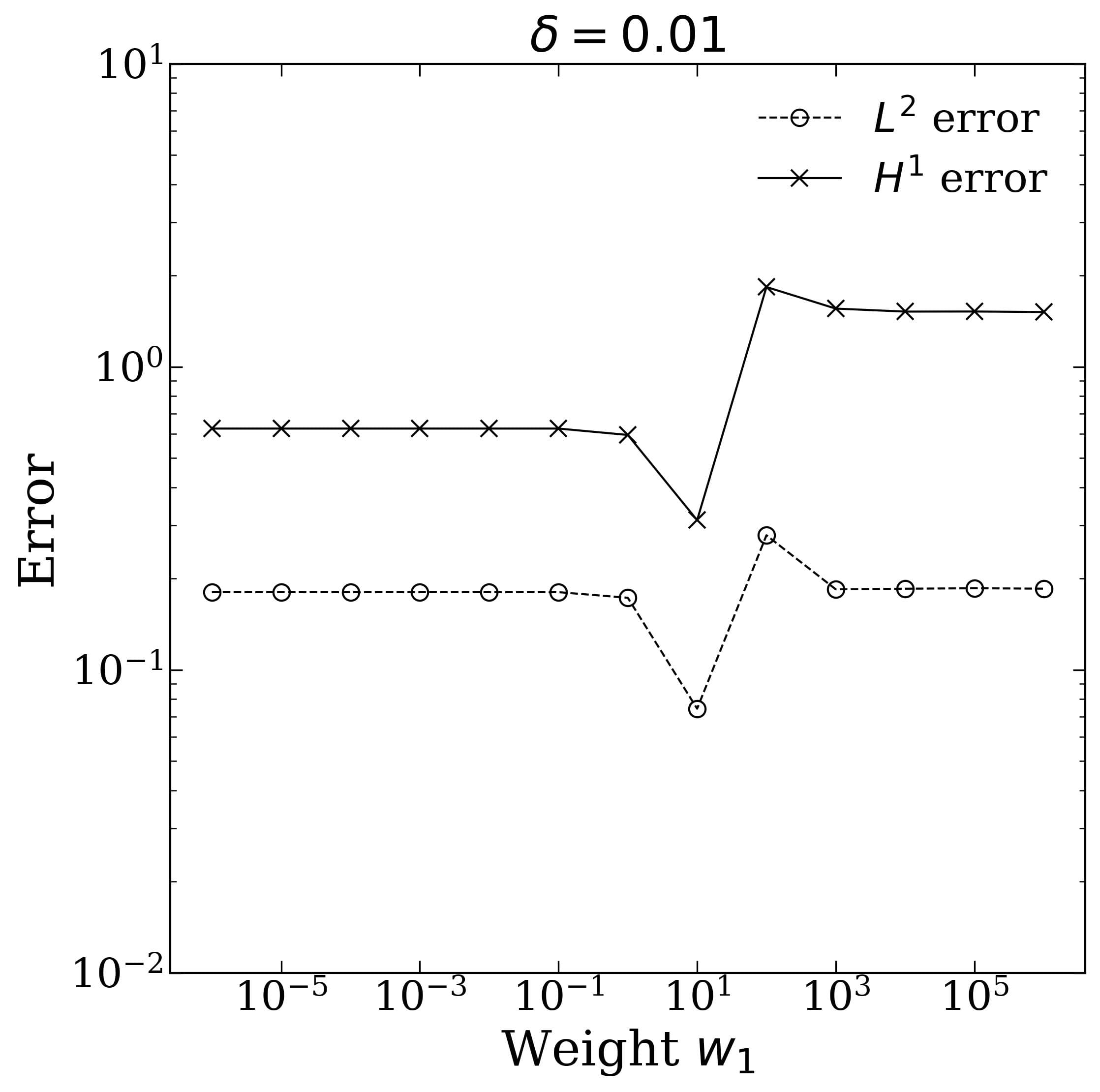}} \  
\caption{Effect of weight parameter $w_{1}$ when $\alpha^{\star} = \alpha_{1}^{\star}$ with $g = g_{2}$}
\label{fig:effect_of_weight_pringle_g2}
\end{figure}
%
%

%
\begin{figure}[htp!]
\centering
\resizebox{0.235\textwidth}{!}{\includegraphics{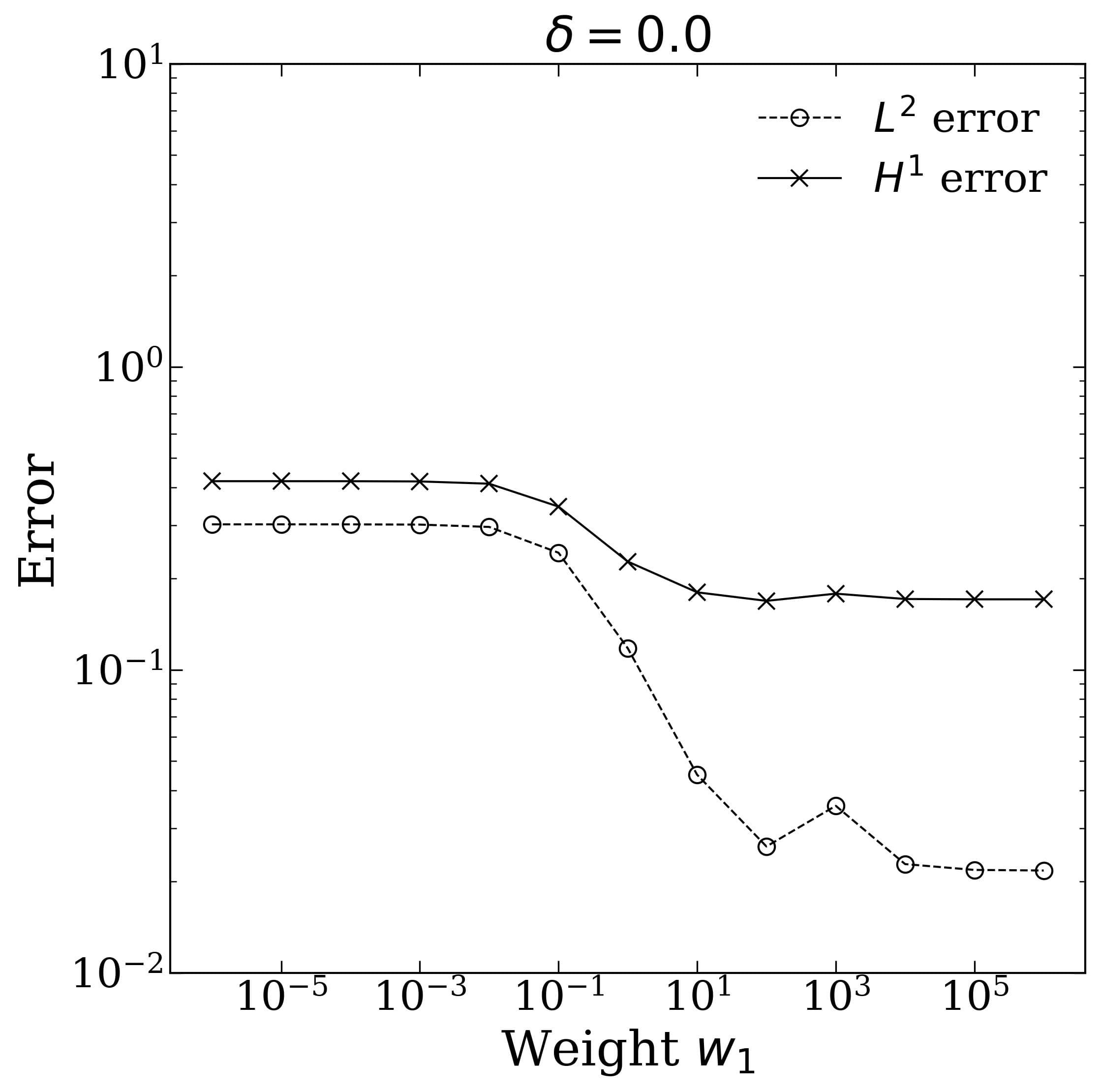}} \  
\resizebox{0.235\textwidth}{!}{\includegraphics{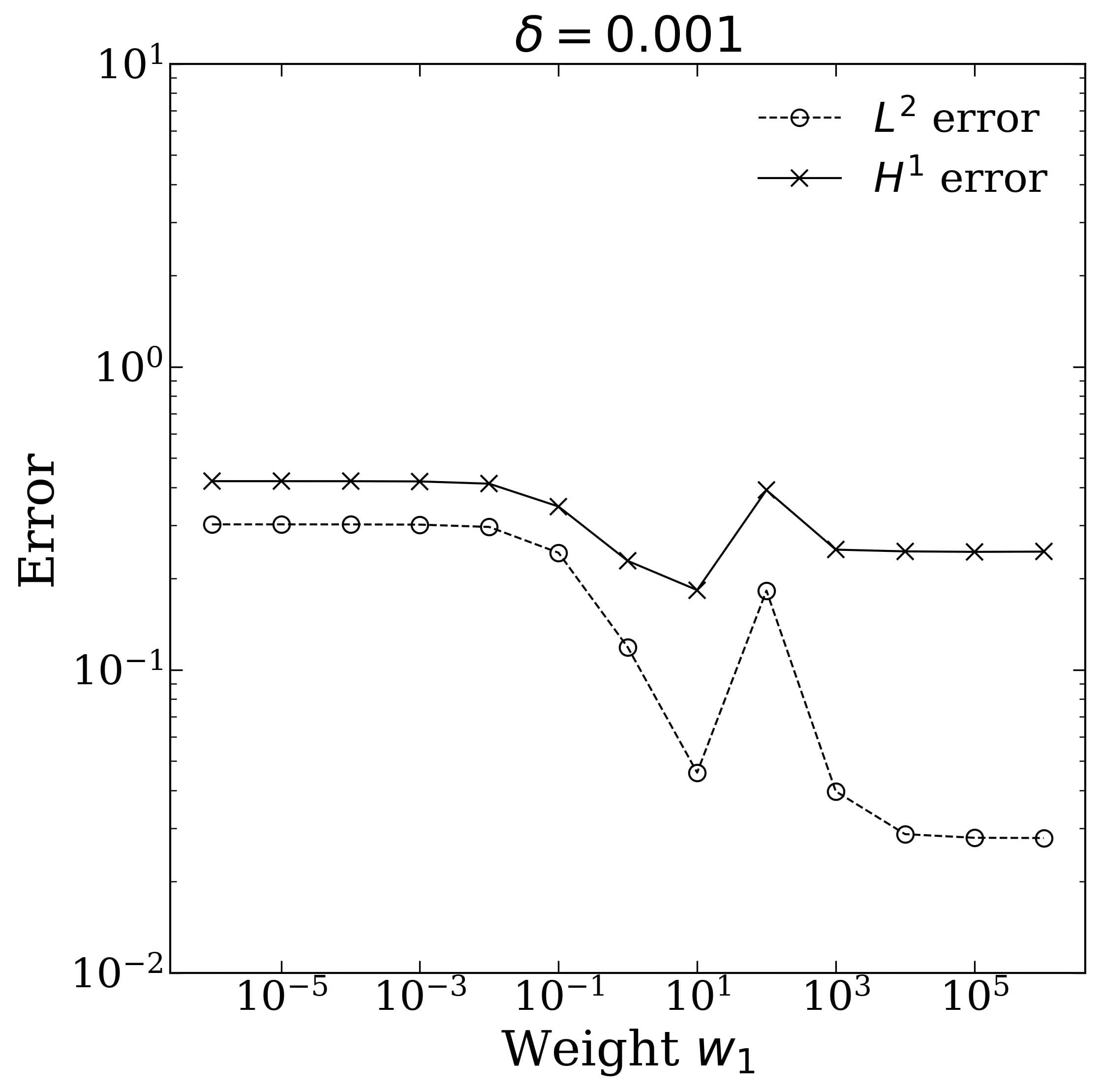}} \  
\resizebox{0.235\textwidth}{!}{\includegraphics{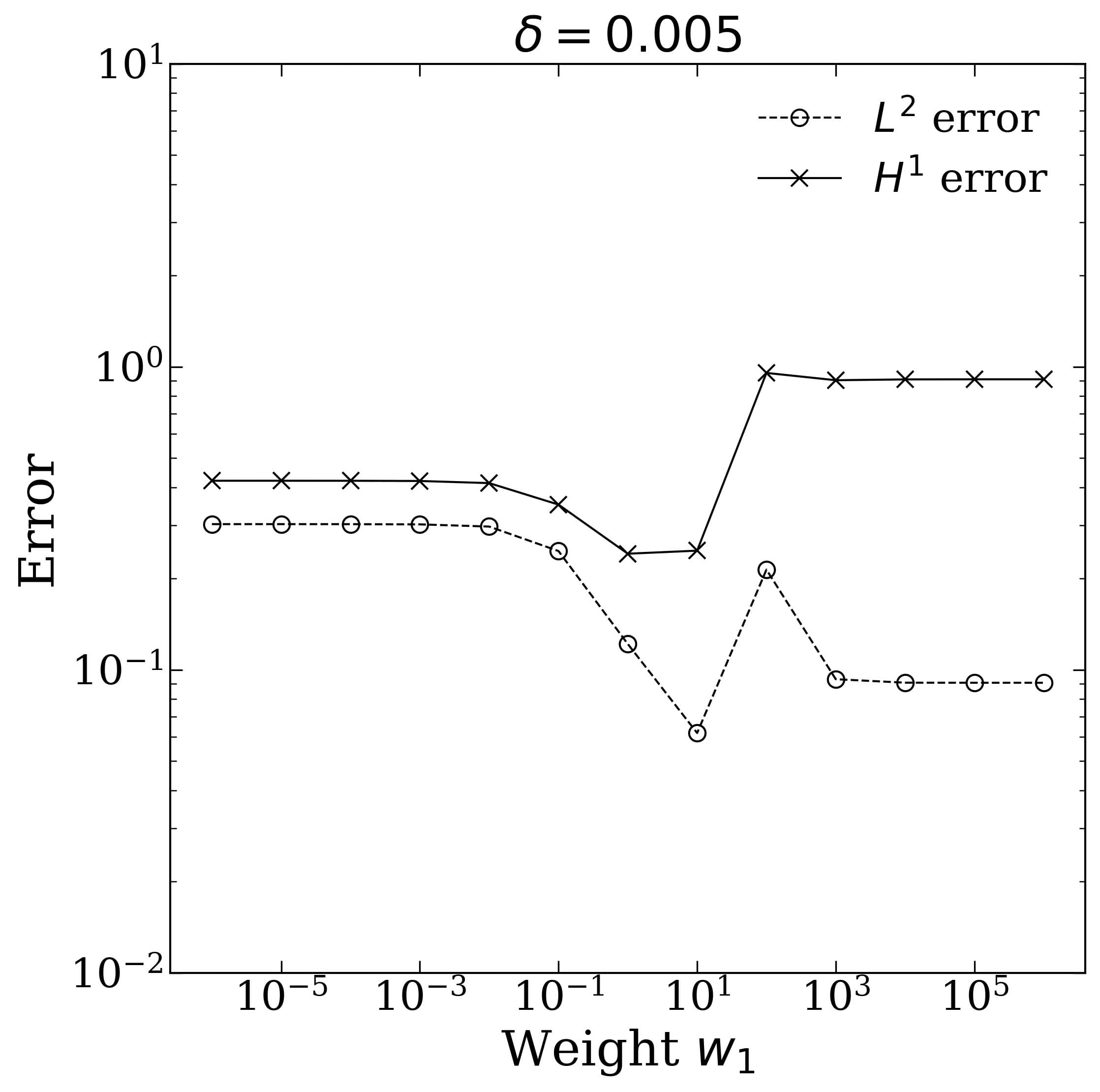}} \  
\resizebox{0.235\textwidth}{!}{\includegraphics{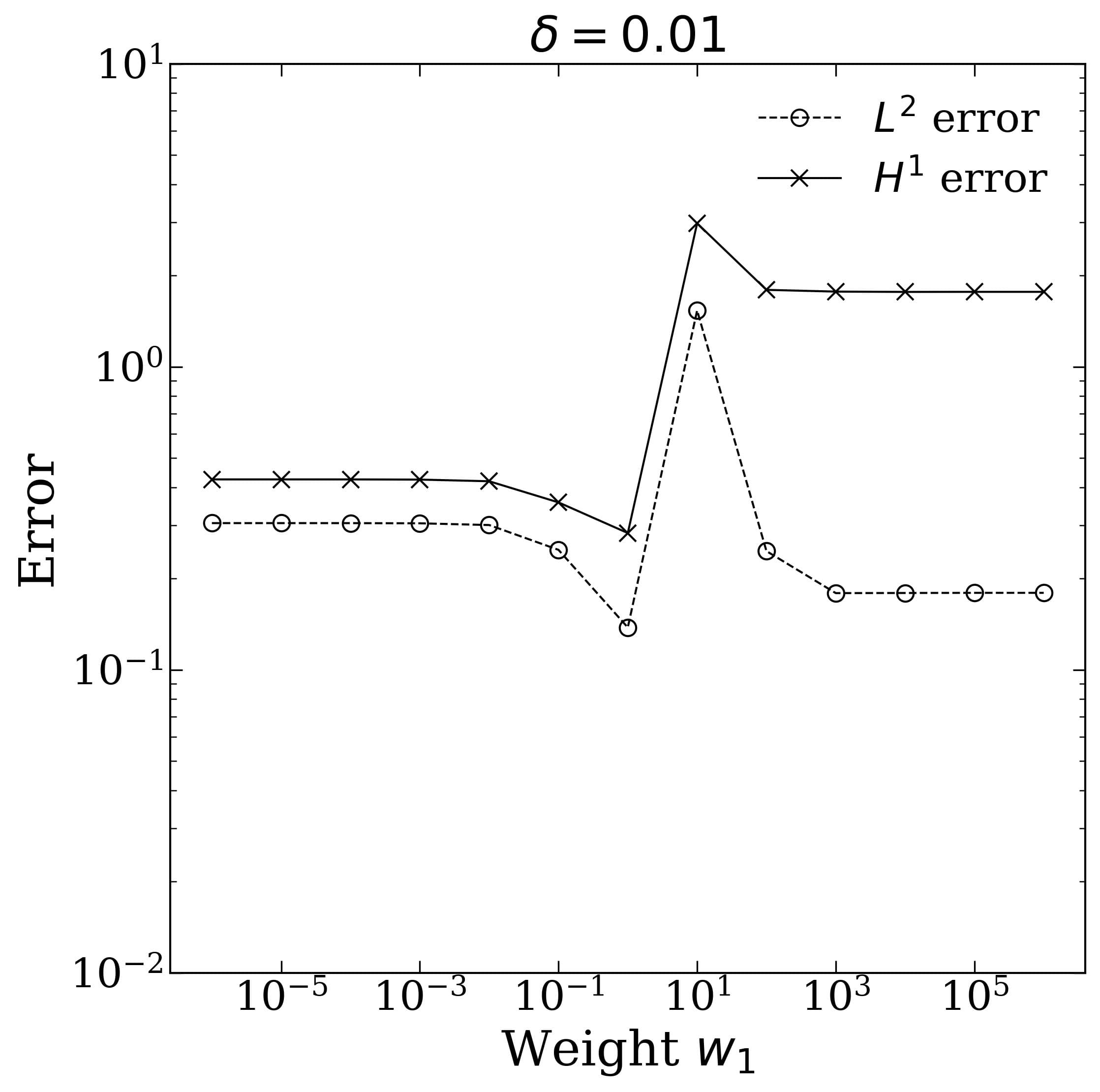}} \  
\caption{Effect of weight parameter $w_{1}$ when $\alpha^{\star} = \alpha_{2}^{\star}$ with $g = g_{2}$}
\label{fig:effect_of_weight_highly_oscillating_g2}
\end{figure}
%
%

\begin{remark}\label{rem:role_of_w1}
We summarize the role of the gradient-dependent term in the modified CCBM misfit~\eqref{eq:modified_CCBM_cost_functional} and its observed numerical effects.

\begin{itemize}
\item The inclusion of the term $\norm{\nabla \iu(\alpha)}_{0,\varOmega}^{2}$ introduces an $H^1$-type contribution to the state discrepancy.
In the reported experiments, this appears to increase sensitivity to spatial variations in the diffusion coefficient, partially counteracting the smoothing effect of the forward map.

\item Although the explicit Tikhonov regularization $\regu_\rho$ controls only $\norm{\alpha}_{0,\varOmega}$, the gradient-weighted misfit may introduce an additional implicit smoothing effect through the forward and adjoint problems.
Empirically, this is associated with Sobolev gradients that better reflect $H^1$-type reconstruction errors.

\item Larger values of $w_1$ are observed to improve contraction of the Sobolev-gradient iteration in several test cases.
While strong monotonicity is assumed in Remark~\ref{rem:coercivity_condition_for_the_Sobolev_gradient}, increased $w_1$ appears to enhance convergence behavior in practice within a suitable parameter range.

\item In the presented numerical tests, sufficiently large but not excessive values of $w_1$ are often associated with lower reconstruction errors in both $L^2(\varOmega)$ and $H^1(\varOmega)$.
This is consistent with improved control of certain high-frequency components, though only $L^2$-regularization is imposed explicitly.

\item These observations are empirical and depend on the problem setting and parameter choice.
Excessively large values of $w_1$ may increase noise sensitivity or numerical stiffness, requiring careful tuning of step sizes and regularization parameters.
\end{itemize}
\end{remark}

\subsection{Reconstruction of piecewise-constant diffusion via projection} \label{subsec:experiments_piecewise_diffusion}
While the numerical experiments above demonstrate the effectiveness of the proposed Sobolev gradient descent scheme for smoothly varying diffusion coefficients, many applications involve heterogeneous media with piecewise-constant diffusion and sharp interfaces. In such cases, reconstructions obtained solely from the continuous $H^1$-based update may exhibit residual smooth transitions across subdomain boundaries, even in the presence of total variation regularization.

To address this limitation, we augment the gradient-based iteration with a constraint-enforcement step that explicitly incorporates the known piecewise-constant structure of the diffusion coefficient. After each regularized Sobolev gradient update, the intermediate iterate is mapped onto the admissible set of piecewise-constant functions by assigning a single representative value to each prescribed subregion. In the present implementation, this value is obtained by evaluating the updated coefficient at a fixed sampling point within each subdomain, thereby enforcing constancy on that region.

This approach can be interpreted as a projection-type strategy, in which the descent step is performed in the continuous space $H^1(\varOmega)$ and subsequently followed by a non-orthogonal projection onto a lower-dimensional subspace encoding the piecewise-constant parametrization. Although this projection is not orthogonal in the $L^2$ or $H^1$ sense, it effectively suppresses intra-region oscillations and stabilizes the identification of subdomain-wise constant diffusion values. As demonstrated in the subsequent numerical experiments, this combined strategy enables accurate recovery of discontinuous diffusion coefficients while retaining the stability observed in the smooth case.

Motivated by the above discussion, we slightly modify Algorithm~\ref{alg:H1_grad_descent_unreg} to account for piecewise-constant diffusion coefficients. The Sobolev gradient descent step is retained, followed by an additional restriction step that enforces the prescribed piecewise-constant structure. This restriction is applied \emph{after} each regularized gradient update, ensuring admissibility while preserving the smoothing and stability properties of the $H^1$-based descent. The resulting algorithm is summarized below.

\begin{algorithm}[H]
\caption{Regularized Gradient Descent in $H^1$ with Piecewise-Constant Restriction}
\label{alg:H1_grad_descent_pc}
\begin{algorithmic}[1]
\Require objective function $J$, initial guess $\alpha^{[0]} \in \barAad$, step size rule, maximum iterations $k_{\max}$
\Ensure approximate minimizer $\alpha^\star$
\State Set $k = 0$
    \While{$k \leq k_{\max}$}
    \State Solve the state equation for $\alpha^{[k]}$ to obtain $u(\alpha^{[k]})$
    \State Solve the adjoint equation (if applicable) to obtain $p(\alpha^{[k]})$
    \State Compute the $H^1$-Sobolev gradient $\nabla_{H^1} J(\alpha^{[k]})$ using $u(\alpha^{[k]})$ and $p(\alpha^{[k]})$
    \State Determine step size $t^{[k]}$ (fixed or via line search \cite{NocedalWright2006})
    \State Perform the regularized update:
    \[
    \tilde{\alpha}^{[k+1]}
    =
    \alpha^{[k]}
    - t^{[k]} \nabla_{H^1} J(\alpha^{[k]})
    + t^{[k]} \rho^{[k]} \regu'(\alpha^{[k]})
    \]
    \State Restrict $\tilde{\alpha}^{[k+1]}$ to the admissible piecewise-constant space:
    \[
    \alpha^{[k+1]} = \Pi_{\mathrm{pc}}(\tilde{\alpha}^{[k+1]})
    \]
    \State $k \gets k+1$
\EndWhile
\State \Return $\alpha^{[k]}$ as $\alpha^\star$
\end{algorithmic}
\end{algorithm}
In this work, the projection-type operator $\Pi_{\mathrm{pc}}$ is implemented using a simple \emph{pick-a-point} strategy. For each subregion, a fixed sampling point located strictly inside the domain is selected, and the value of $\tilde{\alpha}^{[k+1]}$ at this point is assigned uniformly over the entire subregion. In our implementation, the Sobolev gradient is computed in $P_1$, while the coefficient $\alpha$ is represented in $P_0$. While this procedure does not correspond to an orthogonal projection in an $L^2$ or $H^1$ sense and does not provide robustness for complex geometries, it is effective in simple test cases involving two or three subregions, as considered in the present numerical experiments. In these settings, the restriction successfully enforces the desired piecewise-constant structure when combined with Sobolev smoothing and total variation regularization.

In the following examples, we consider piecewise-constant diffusion coefficients with the setups described as follows: $b = 0$ and $c = 1$, the source term as $f(x_1,x_2) = x_1 + x_2 + 2$, and the initial guess for the coefficient as $\alpha^{[0]} = 2$.
\subsubsection{Two-subregions case}
For the following example, we set the boundary data to be $g(x_1,x_2) = 1 + 0.5\,\sin(\pi x_1)\sin(\pi x_2)$.
\begin{example}[Two-subregions case]\label{ex:two_subregions}
Consider the domain $\varOmega = (-1,1)^2$ divided into two subregions by $x_1=0$:
\[
\varOmega_L = \{ x_1 < 0 \}, \qquad \varOmega_R = \{ x_1 \geq 0 \}.
\]
The diffusion coefficient is piecewise constant:
\[
\alpha(x_1) =
\begin{cases}
0.75, & x_1 < 0,\\
0.50, & x_1 \geq 0.
\end{cases}
\]

The piecewise-constant structure is enforced using a pick-a-point strategy: after each gradient update, the value at a fixed point inside each subregion ($x_1=-0.95$ for $\varOmega_L$, $x_1=0.95$ for $\varOmega_R$) is assigned uniformly over the subregion.
\end{example}
Figure~\ref{fig:two_subregions} shows the reconstruction histories for the two-subregions case under different noise levels $\delta$. 
For $\delta=0$, all methods converge to the exact coefficients, with CCBM exhibiting faster convergence. 
As $\delta$ increases, CCBM converges to noise-dependent values close to the exact coefficients, despite non-monotone iteration histories, while KV and TD show increasing bias. 
The TN method is highly sensitive to noise, as expected, and fails to stabilize, particularly in the right subregion.

%
%
%
\begin{figure}[htp!]
\resizebox{0.225\textwidth}{!}{\includegraphics{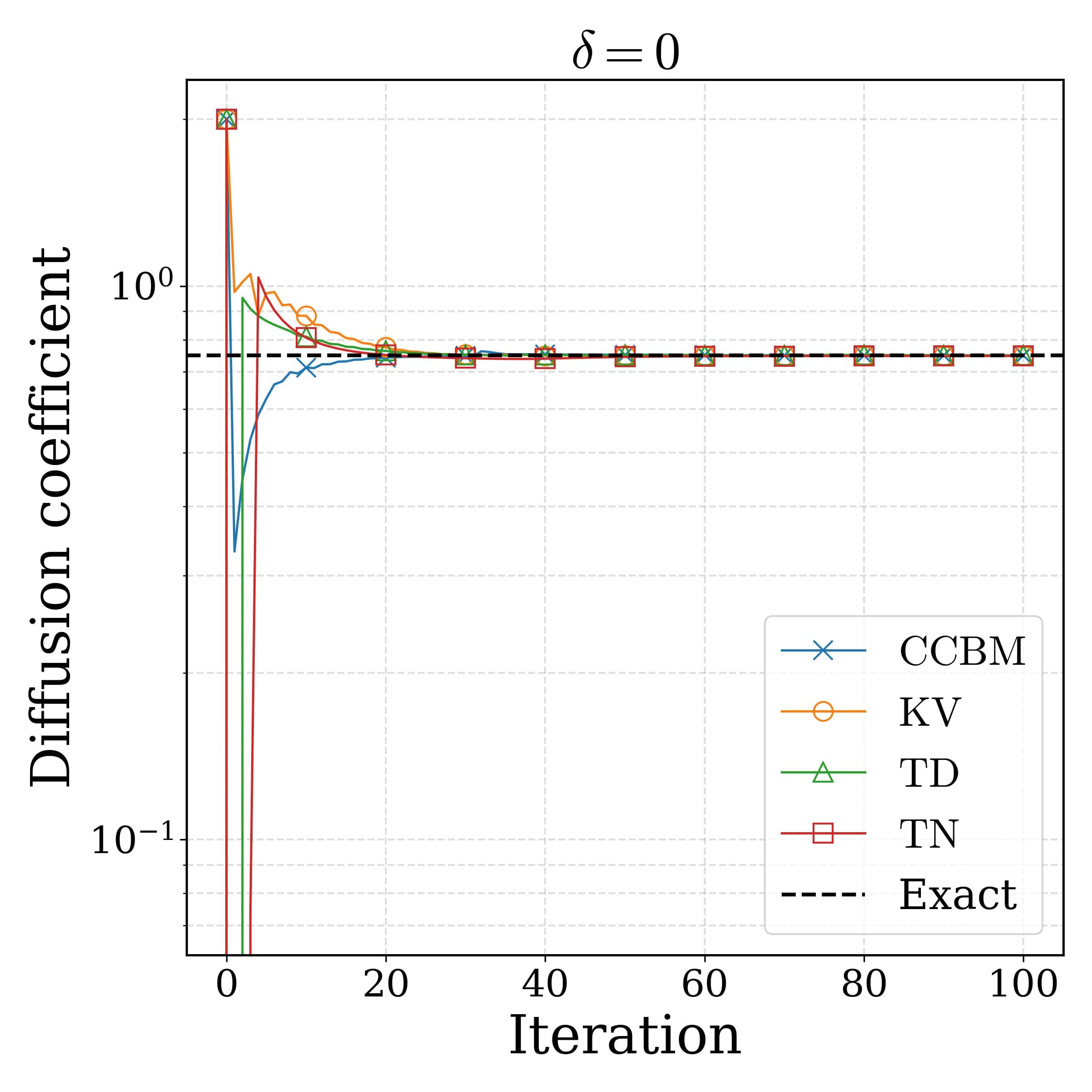}} \  
\resizebox{0.225\textwidth}{!}{\includegraphics{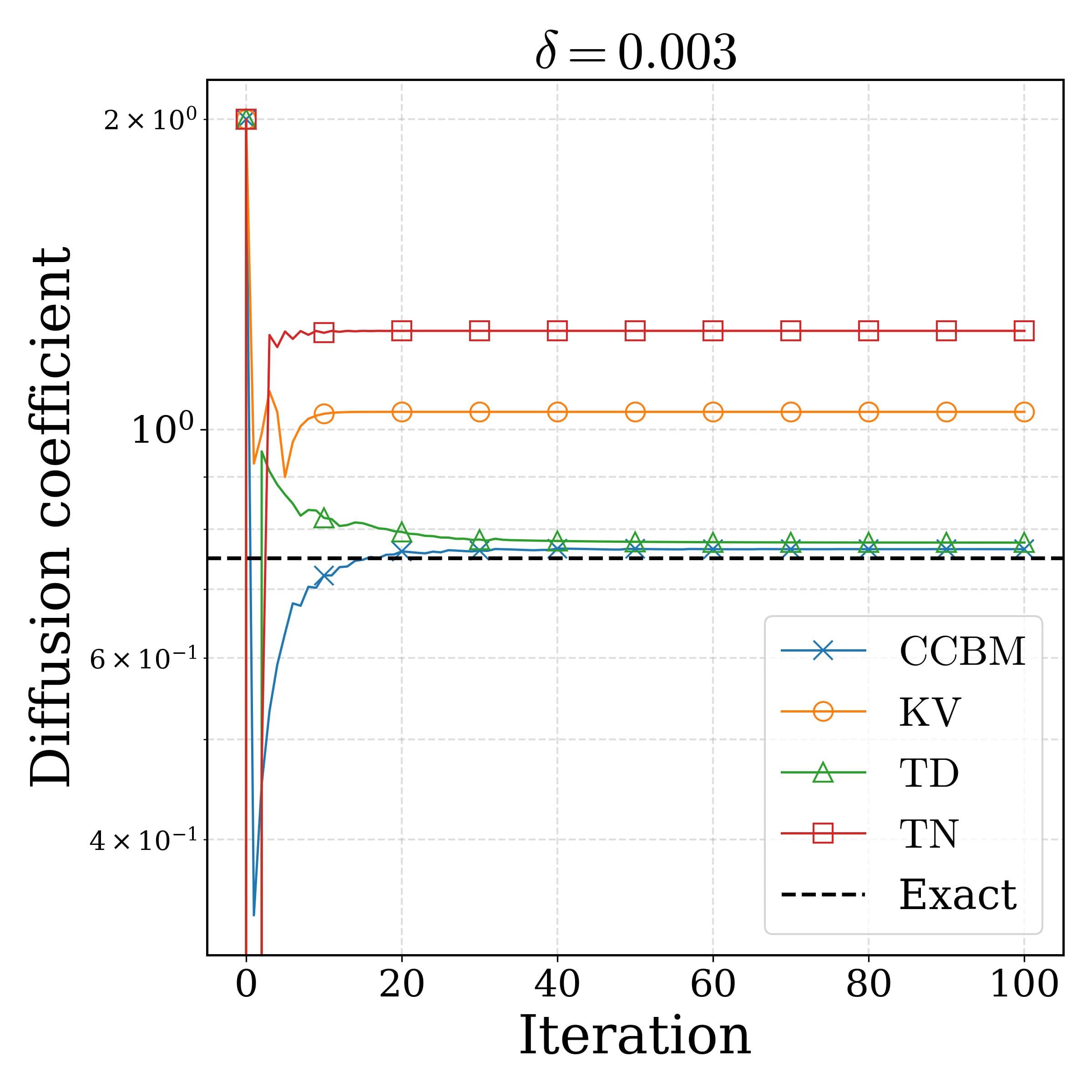}} \  
\resizebox{0.225\textwidth}{!}{\includegraphics{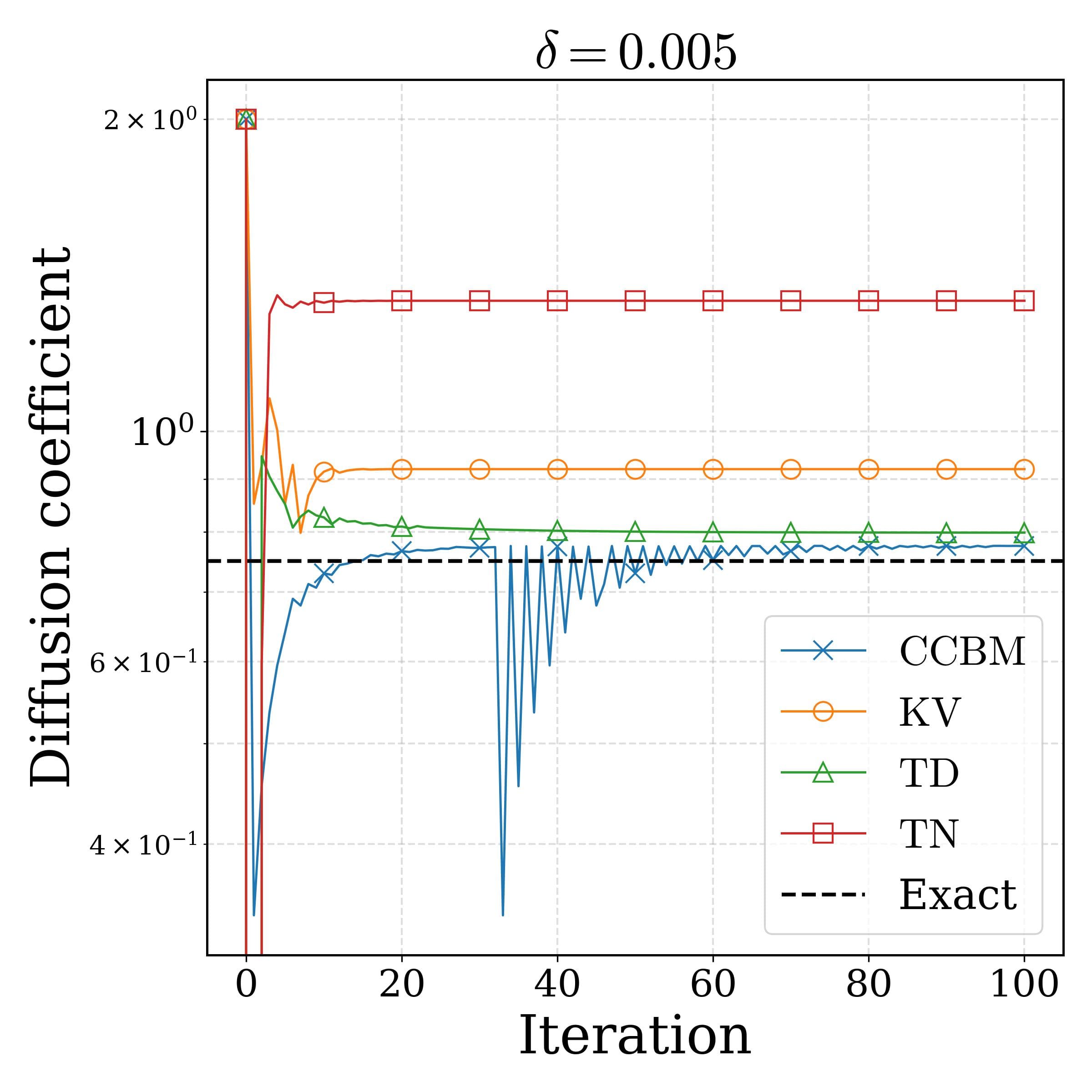}} \  
\resizebox{0.225\textwidth}{!}{\includegraphics{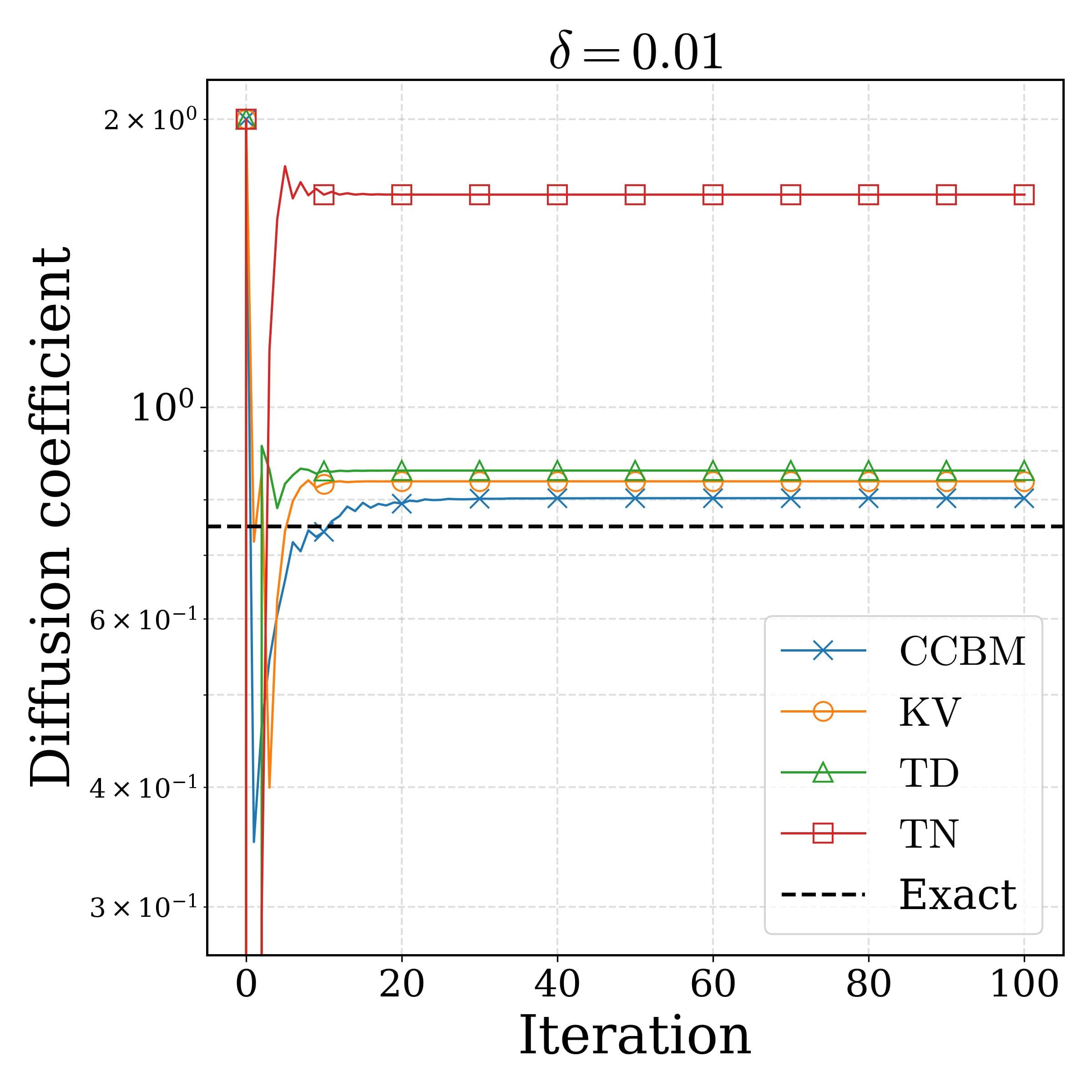}} \\[1em]
\resizebox{0.225\textwidth}{!}{\includegraphics{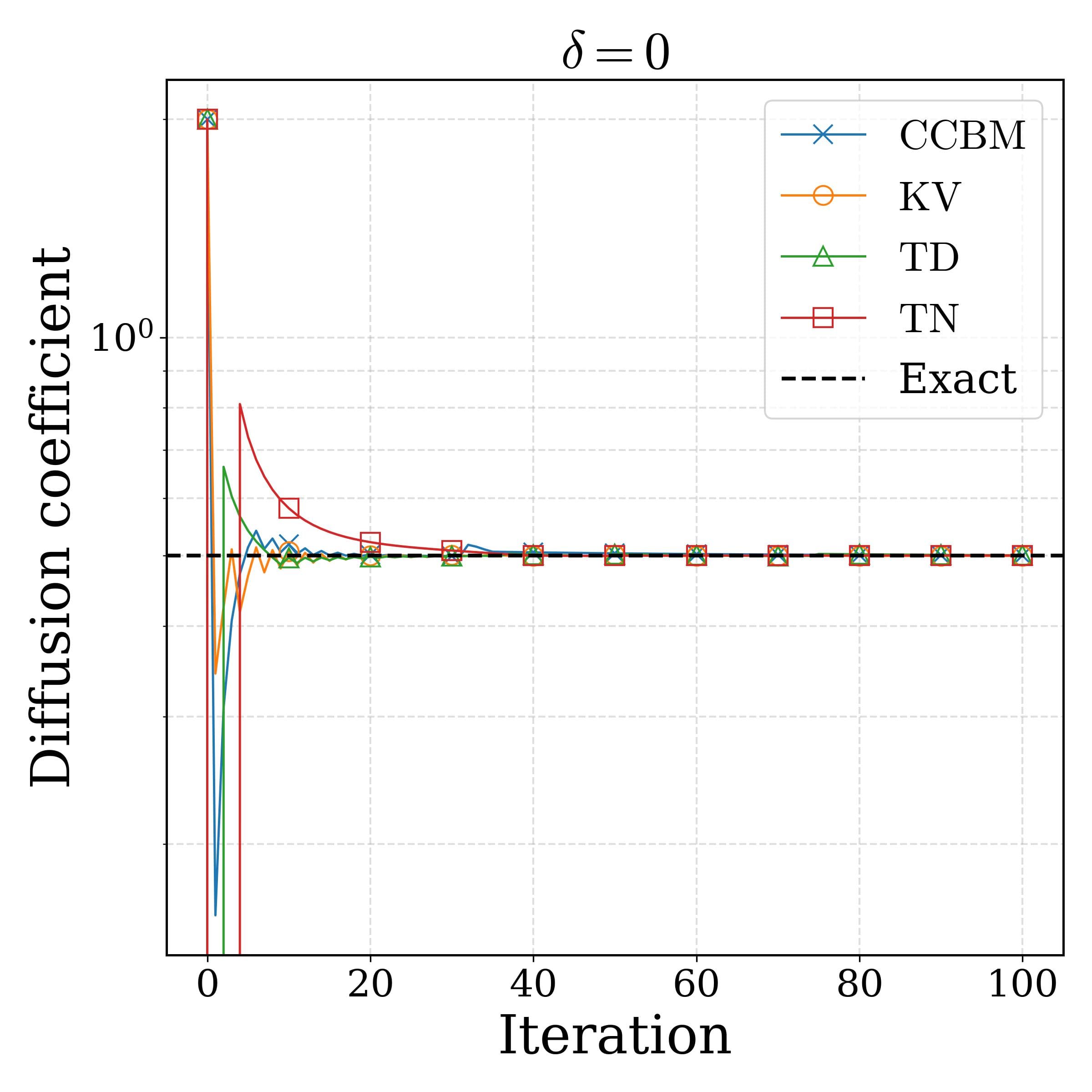}} \  
\resizebox{0.225\textwidth}{!}{\includegraphics{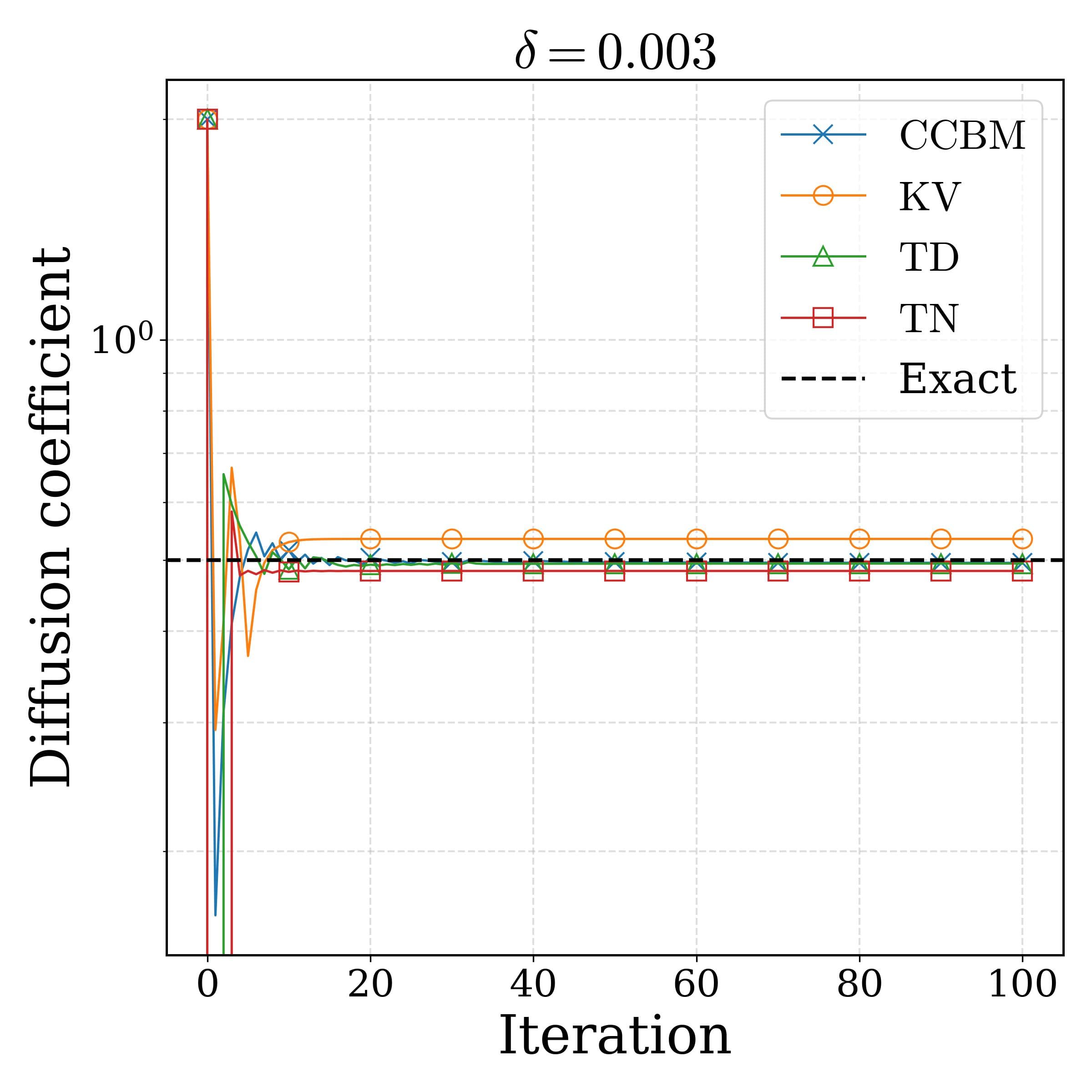}} \  
\resizebox{0.225\textwidth}{!}{\includegraphics{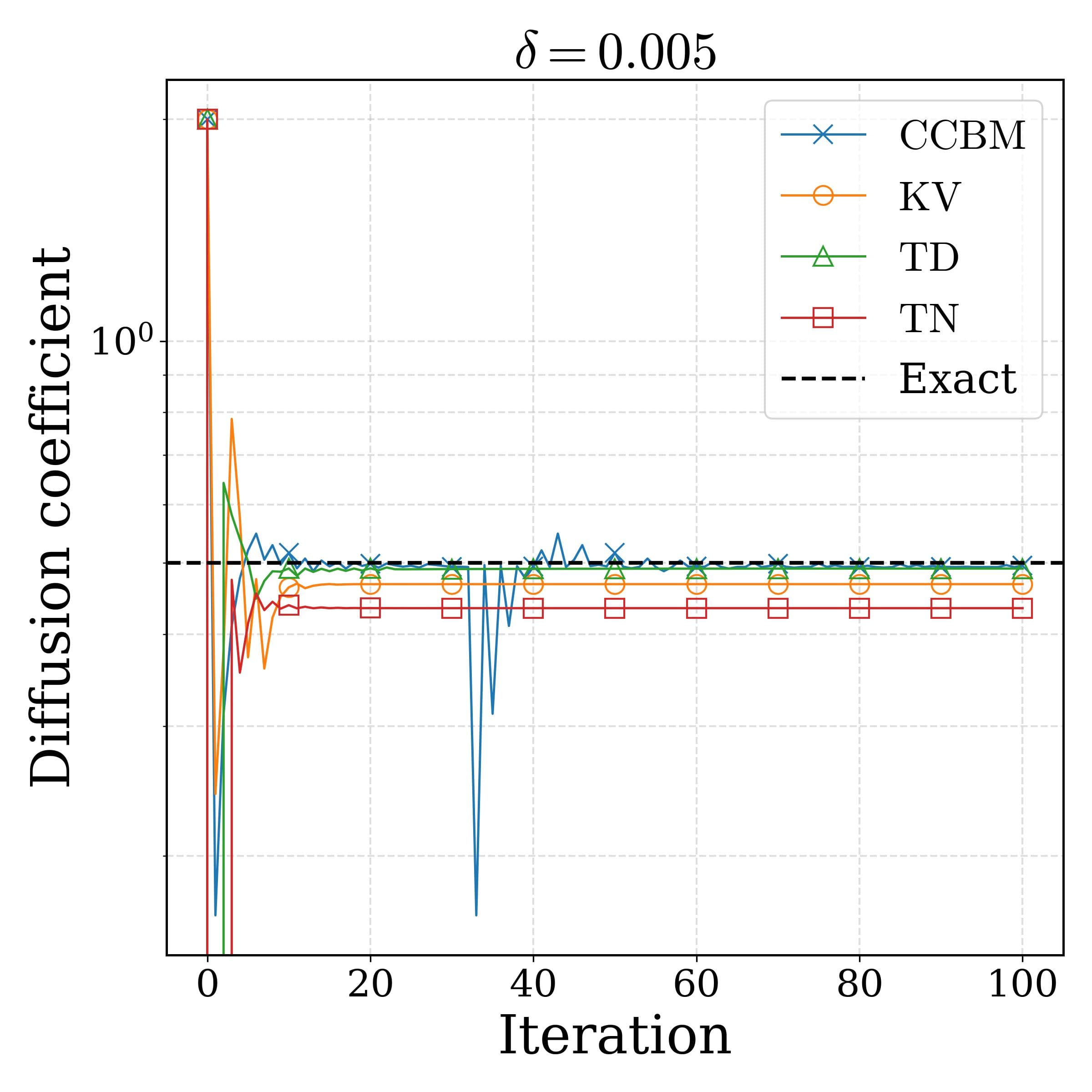}} \  
\resizebox{0.225\textwidth}{!}{\includegraphics{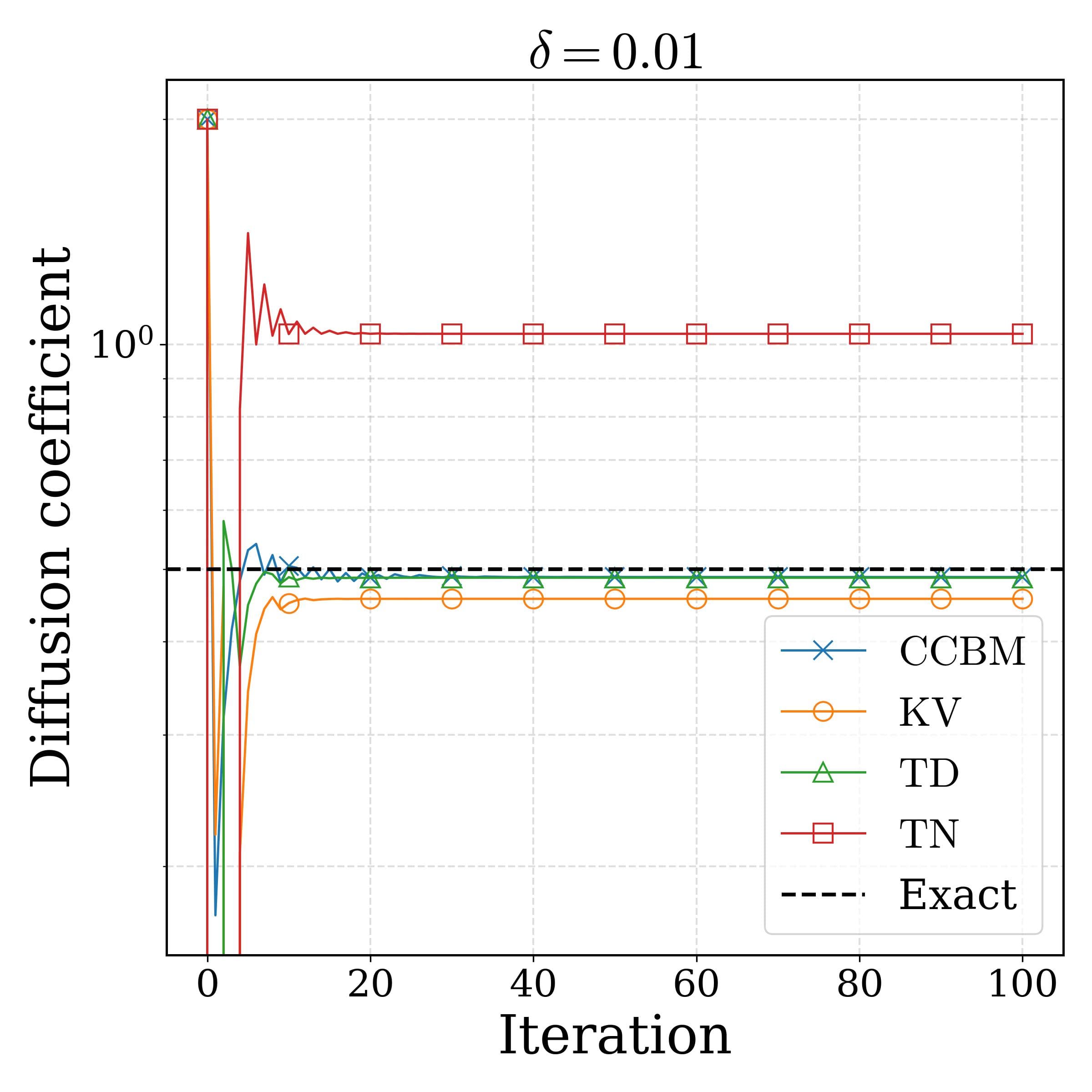}}
\caption{Iteration histories of the reconstructed diffusion coefficients in the two-subregions case for noise levels $\delta = 0$, $0.003$, $0.005$, and $0.01$.  Dashed lines denote the exact values.}
\label{fig:two_subregions}
\end{figure}
%

Figure~\ref{fig:two_subregions_histories_and_CCBM} shows the histories of the cost functionals and gradient norms. 
For exact data (left), all methods rapidly reduce the cost, while CCBM continues to decrease both the cost and gradient norm more steadily. 
For noisy data (right), CCBM stabilizes at noise-dependent levels, reflecting the effect of measurement noise on the achievable accuracy. 
These results suggest that, in this example, the pick-a-point CCBM approach exhibits improved robustness for piecewise-constant coefficient identification.

%
\begin{figure}[htp!]
\resizebox{0.225\textwidth}{!}{\includegraphics{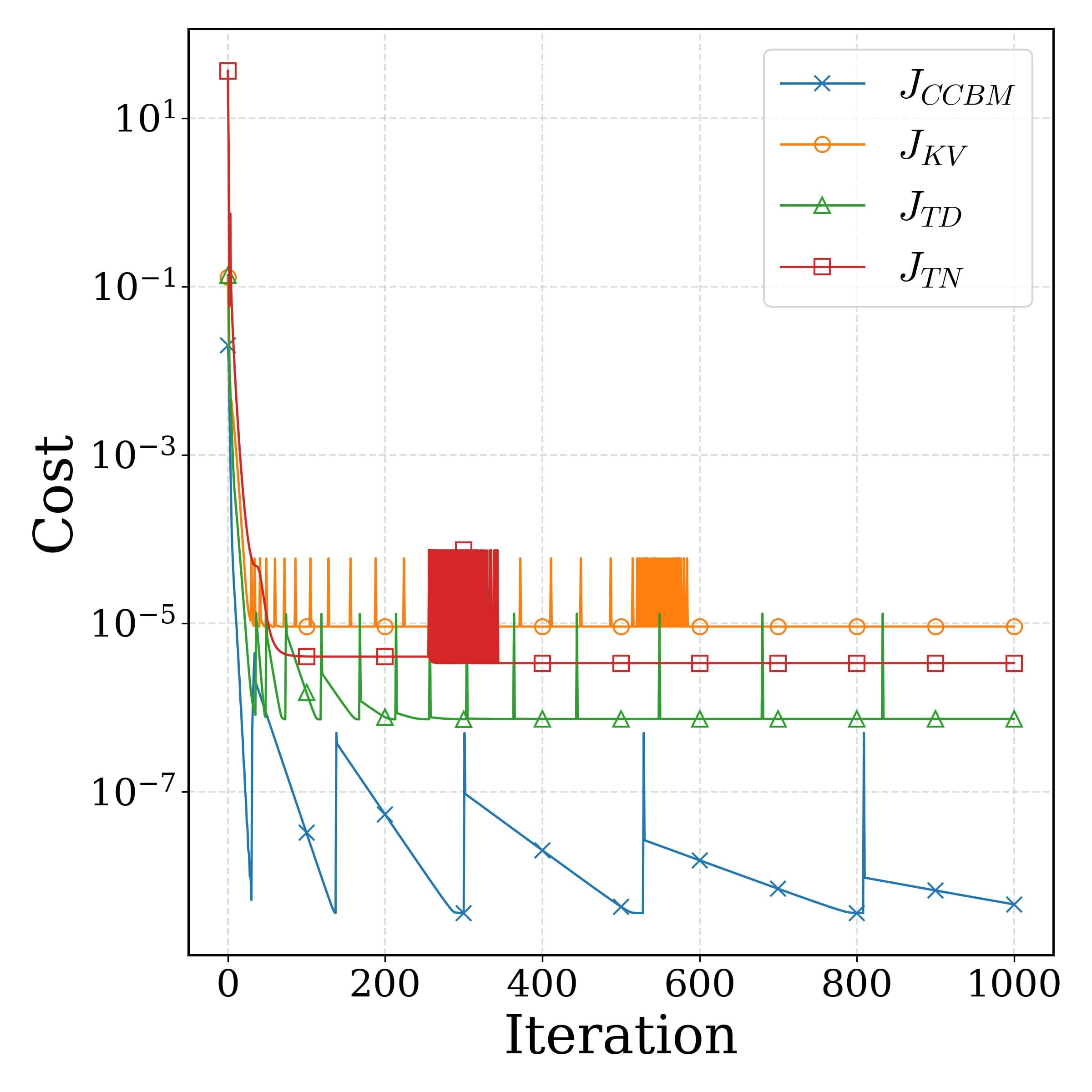}} \  
\resizebox{0.225\textwidth}{!}{\includegraphics{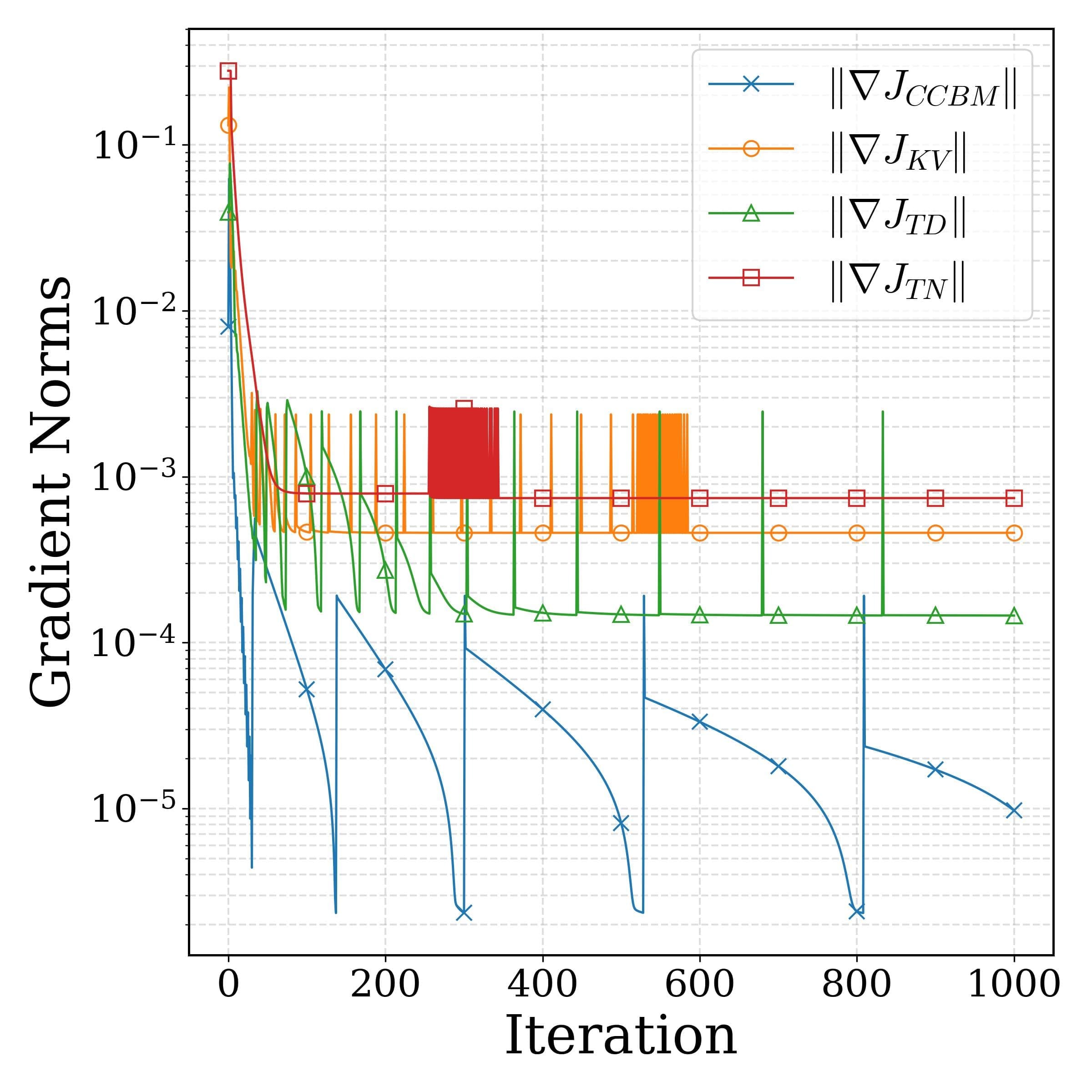}} \hfill 
\resizebox{0.225\textwidth}{!}{\includegraphics{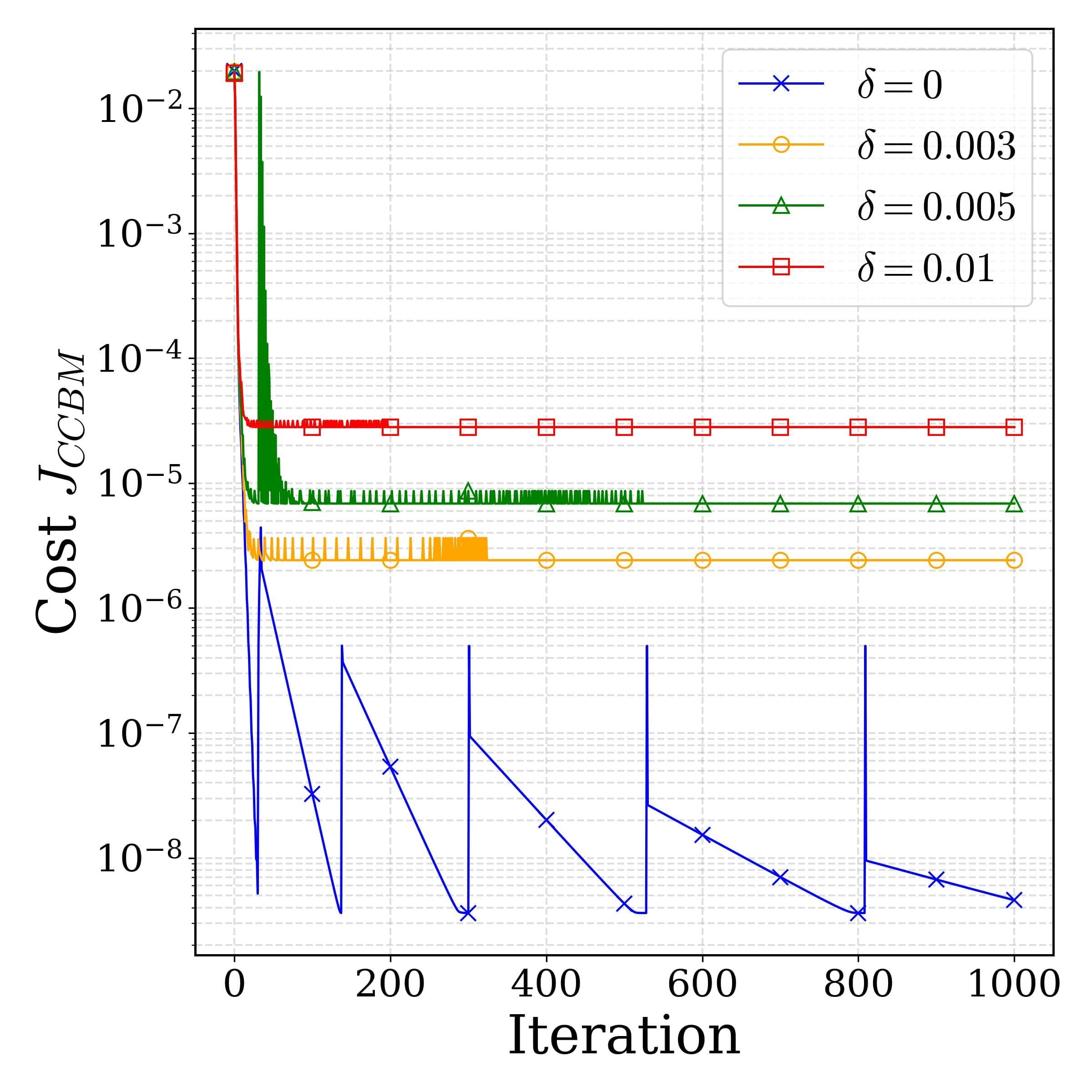}} \  
\resizebox{0.225\textwidth}{!}{\includegraphics{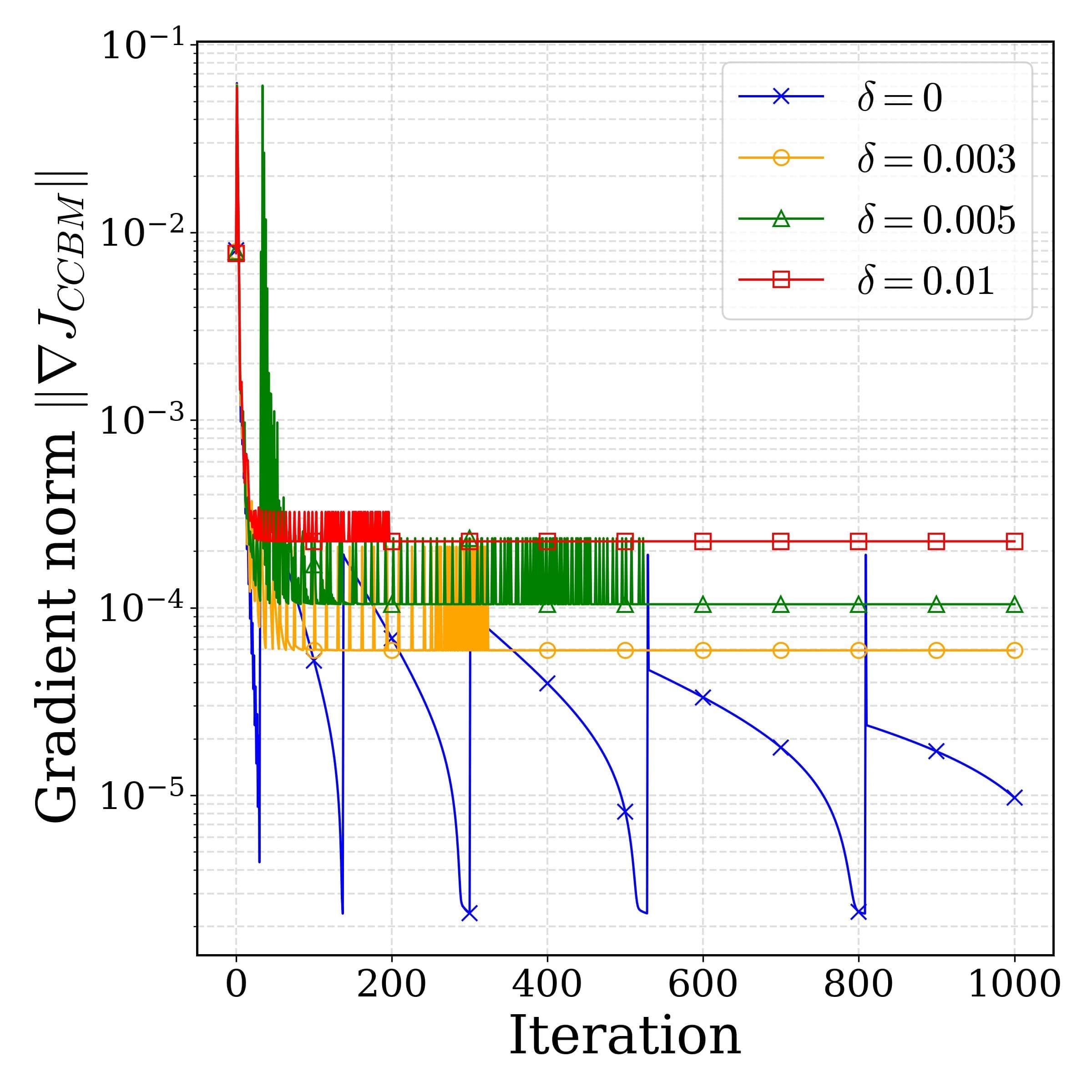}} 
\caption{Left: Histories of the cost functionals and their corresponding gradient norms for each method using exact measurements; Right: Histories of the cost functional and gradient norms for CCBM under different noise levels.}
\label{fig:two_subregions_histories_and_CCBM}
\end{figure}
\subsubsection{Three-subregions case}
For the following example, the boundary data are given by $g(x_1,x_2) = \exp\!\bigl(\sin(\pi x_1)\sin(\pi x_2)\bigr)$.
\begin{example}[Three-subregions case]
Consider the domain $\varOmega = (-1,1)^2$ partitioned into three subregions: a central circular region of radius $R$, a left exterior region, and a right exterior region:
\[
\begin{aligned}
\varOmega_C &= \{ (x_1,x_2) : x_1^2 + x_2^2 \leq R^2 \},\\
\varOmega_L &= \{ (x_1,x_2) : x_1 < 0,\ x_1^2 + x_2^2 > R^2 \},\\
\varOmega_R &= \{ (x_1,x_2) : x_1 \geq 0,\ x_1^2 + x_2^2 > R^2 \}.
\end{aligned}
\]
The diffusion coefficient is piecewise constant:
\[
\alpha(x_1,x_2) =
\begin{cases}
1.5, & (x_1,x_2) \in \varOmega_C,\\
0.75, & (x_1,x_2) \in \varOmega_L,\\
0.50, & (x_1,x_2) \in \varOmega_R.
\end{cases}
\]
The coefficient is restricted using a pick-a-point strategy: after each update, the value at a fixed point in each subregion $(0,0)$ for $\varOmega_C$, $(-0.95,0)$ for $\varOmega_L$, and $(0.95,0)$ for $\varOmega_R$ is assigned uniformly over that subregion.

\end{example}
Figures~\ref{fig:measurements_with_exponential_input_square_setup} and~\ref{fig:three_subregions_histories_of_values}, together with Table~\ref{tab:three_subregions}, summarize the numerical results for the three-subregions case. 
As the noise level $\delta$ increases, the boundary measurements become increasingly oscillatory (Figure~\ref{fig:measurements_with_exponential_input_square_setup}), as expected. 
The reconstruction histories in Figure~\ref{fig:three_subregions_histories_of_values} show that, for exact and low-noise data, the methods converge to stable, subregion-dependent coefficients, whereas higher noise levels lead to earlier stabilization at noise-dependent values and larger reconstruction errors. 
The sensitivity to noise varies across subregions, with larger errors observed in regions farther from the measurement boundary. 
These observations are consistent with the smoother diffusion case discussed in Subsection~\ref{subsec:mildly_oscillatory_diffusion}; see in particular Figure~\ref{fig:reconstructions_mildly_oscillatory_diffusion}.

Table~\ref{tab:three_subregions} shows that CCBM consistently yields the smallest or near-smallest average errors across a wide range of noise levels, providing reliable reconstructions in all subregions. 
TD performs competitively for low to moderate noise levels but exhibits increased bias as the noise level grows. 
KV shows good performance only for very low noise and deteriorates rapidly at moderate noise levels, particularly in the central subregion. 
The TN method is the least reliable, displaying strong sensitivity to noise and producing unstable or nonphysical estimates at moderate to high noise levels. 
Overall, among the considered methods, CCBM offers the most robust and consistent performance for three-subregion coefficient identification.

%
\begin{figure}[htp!]
\centering
\resizebox{0.225\textwidth}{!}{\includegraphics{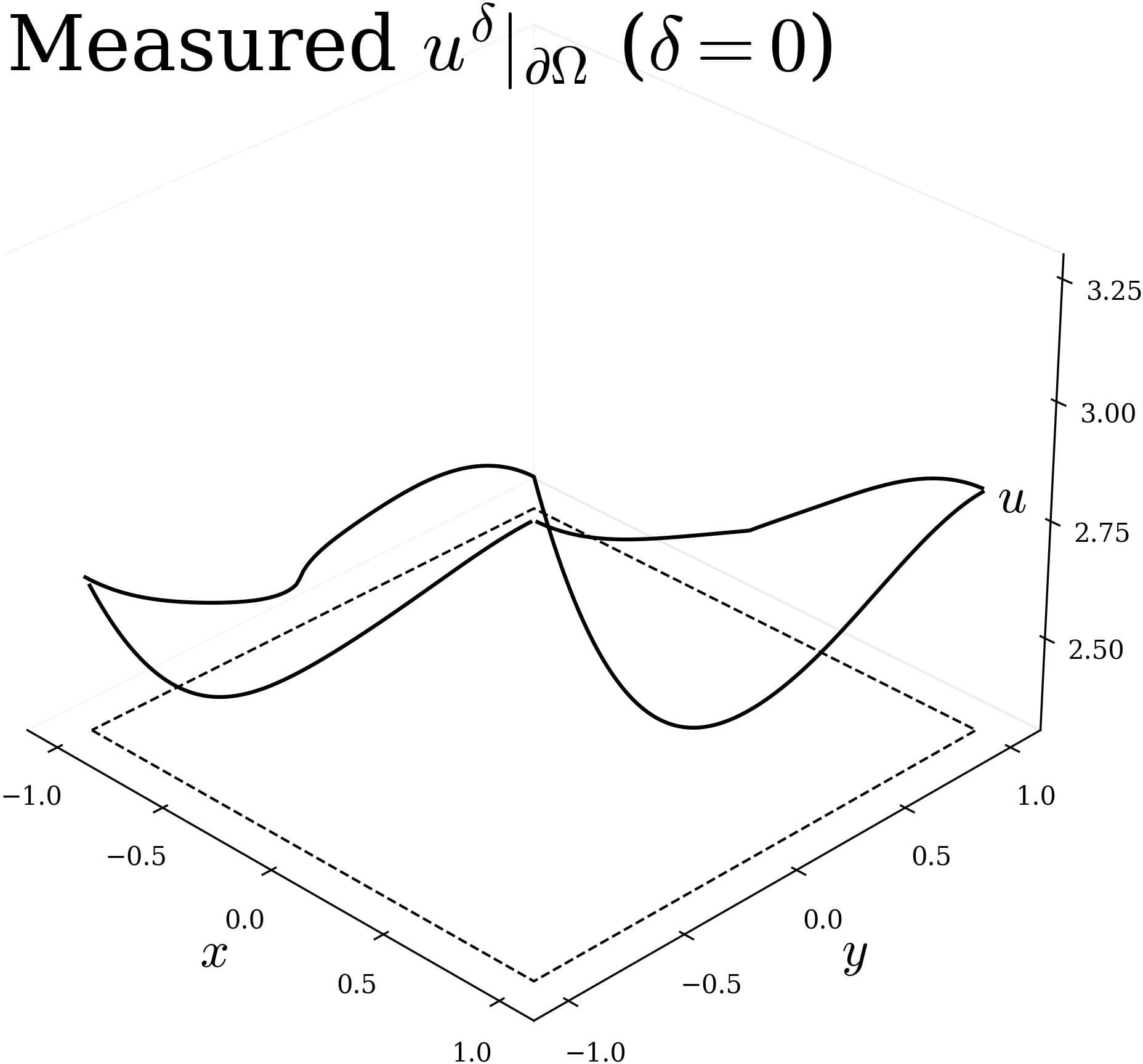}} \ 
\resizebox{0.225\textwidth}{!}{\includegraphics{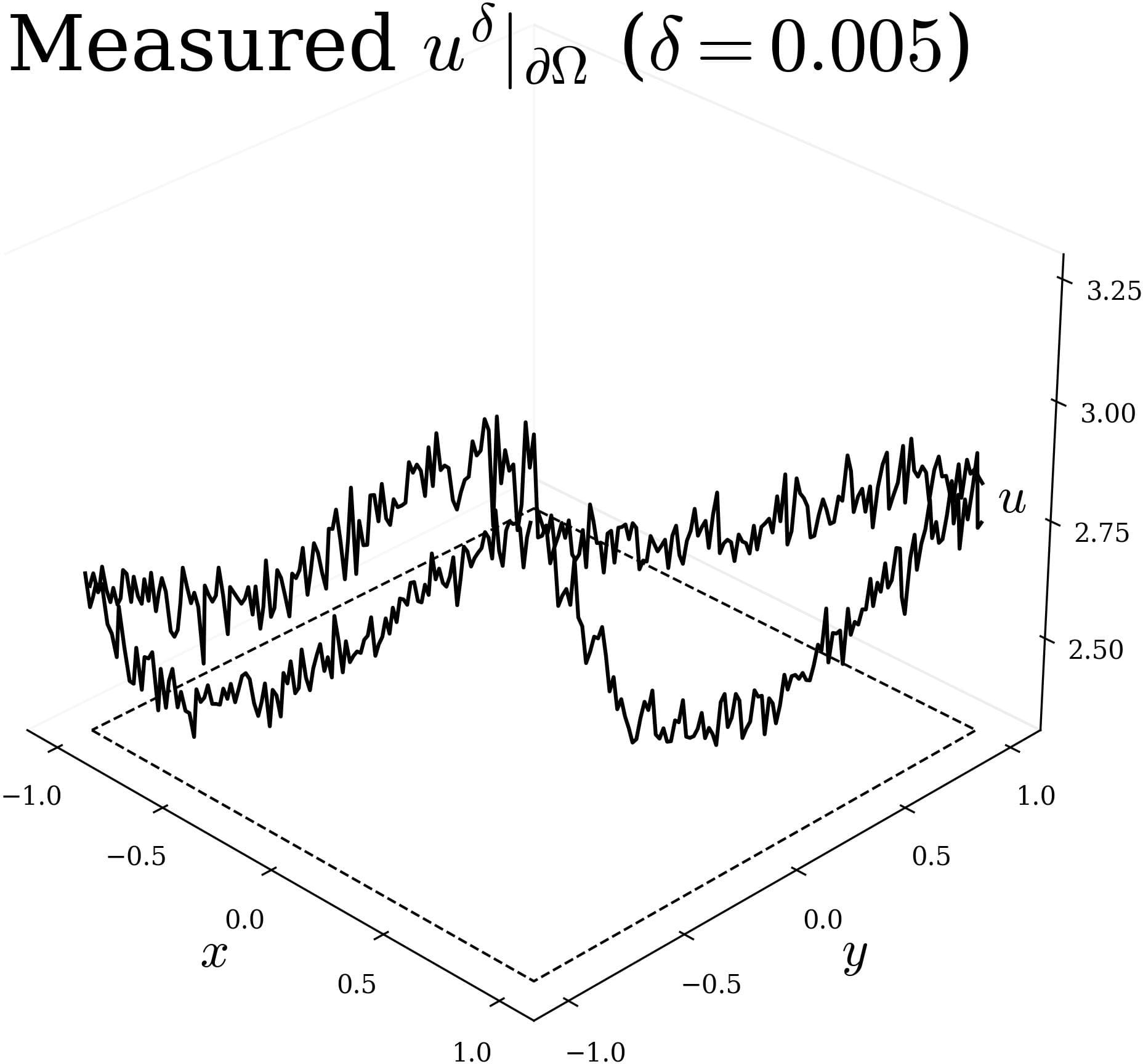}} \ 
\resizebox{0.225\textwidth}{!}{\includegraphics{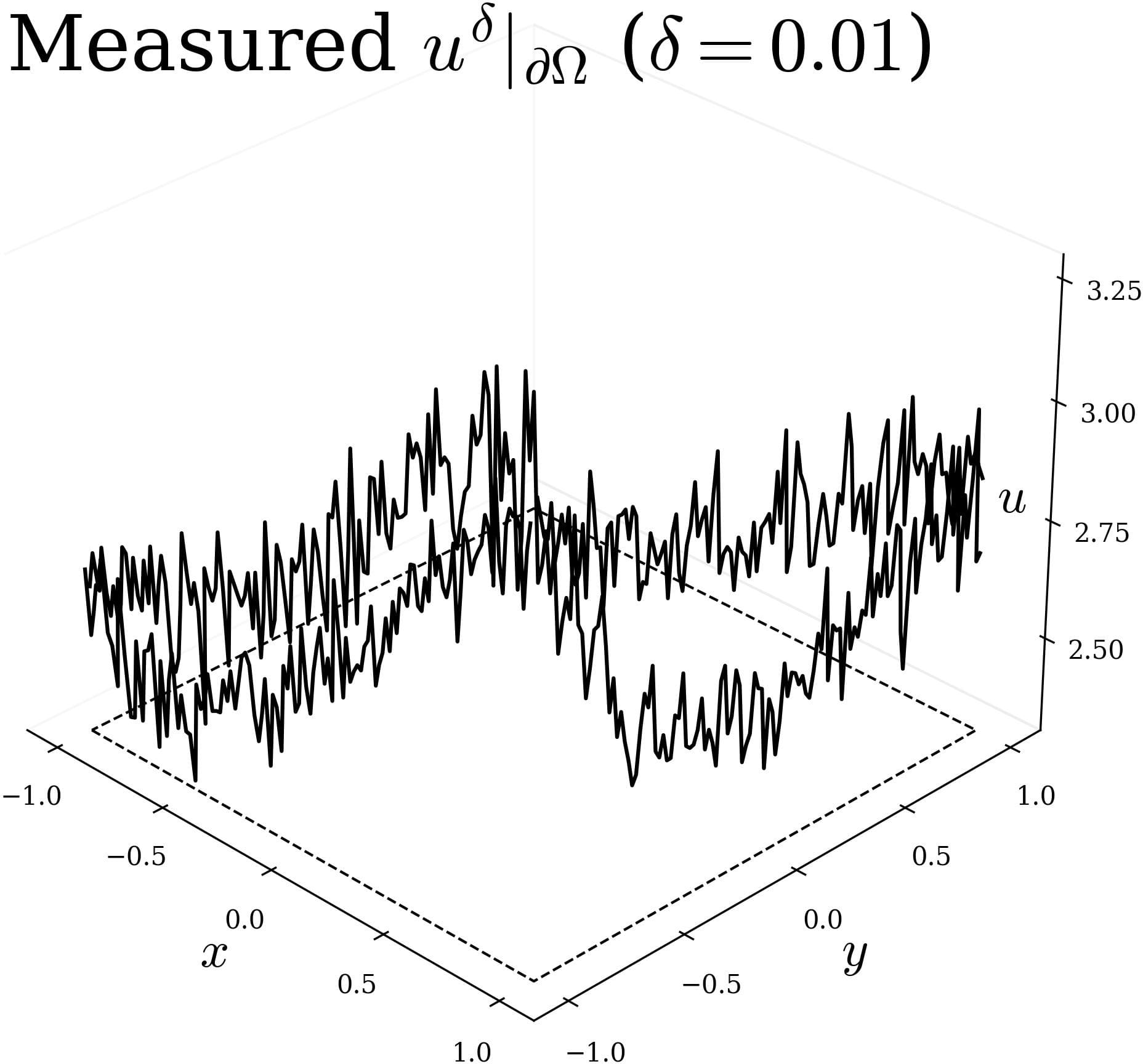}} \ 
\resizebox{0.225\textwidth}{!}{\includegraphics{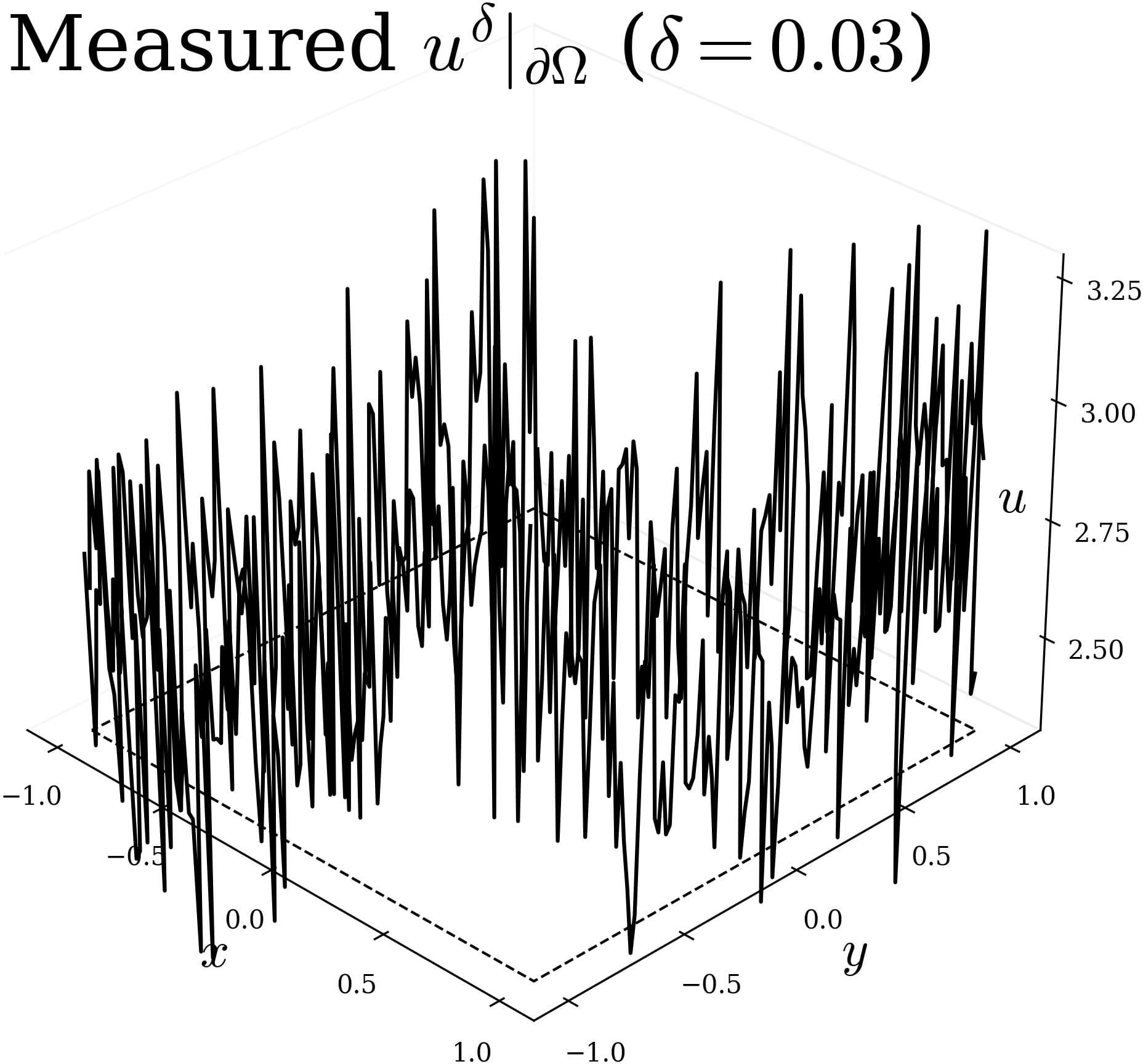}} 
\caption{Boundary measurements at different noise levels $\delta$ with input data $g=\operatorname{exp}(\sin(\pi x)\sin(\pi y))$.}
\label{fig:measurements_with_exponential_input_square_setup}
\end{figure}

\begin{figure}[htp!]
\resizebox{0.225\textwidth}{!}{\includegraphics{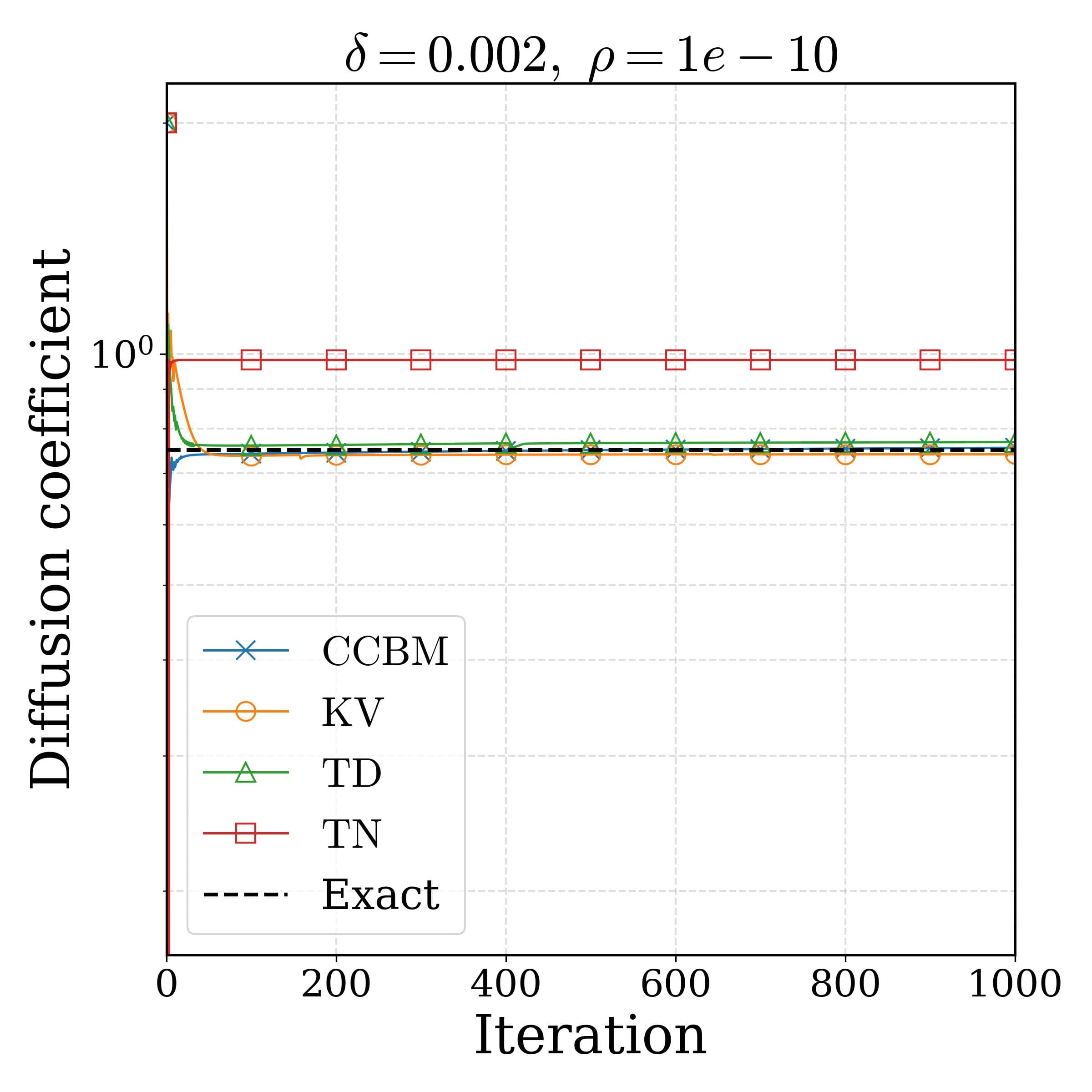}} \  
\resizebox{0.225\textwidth}{!}{\includegraphics{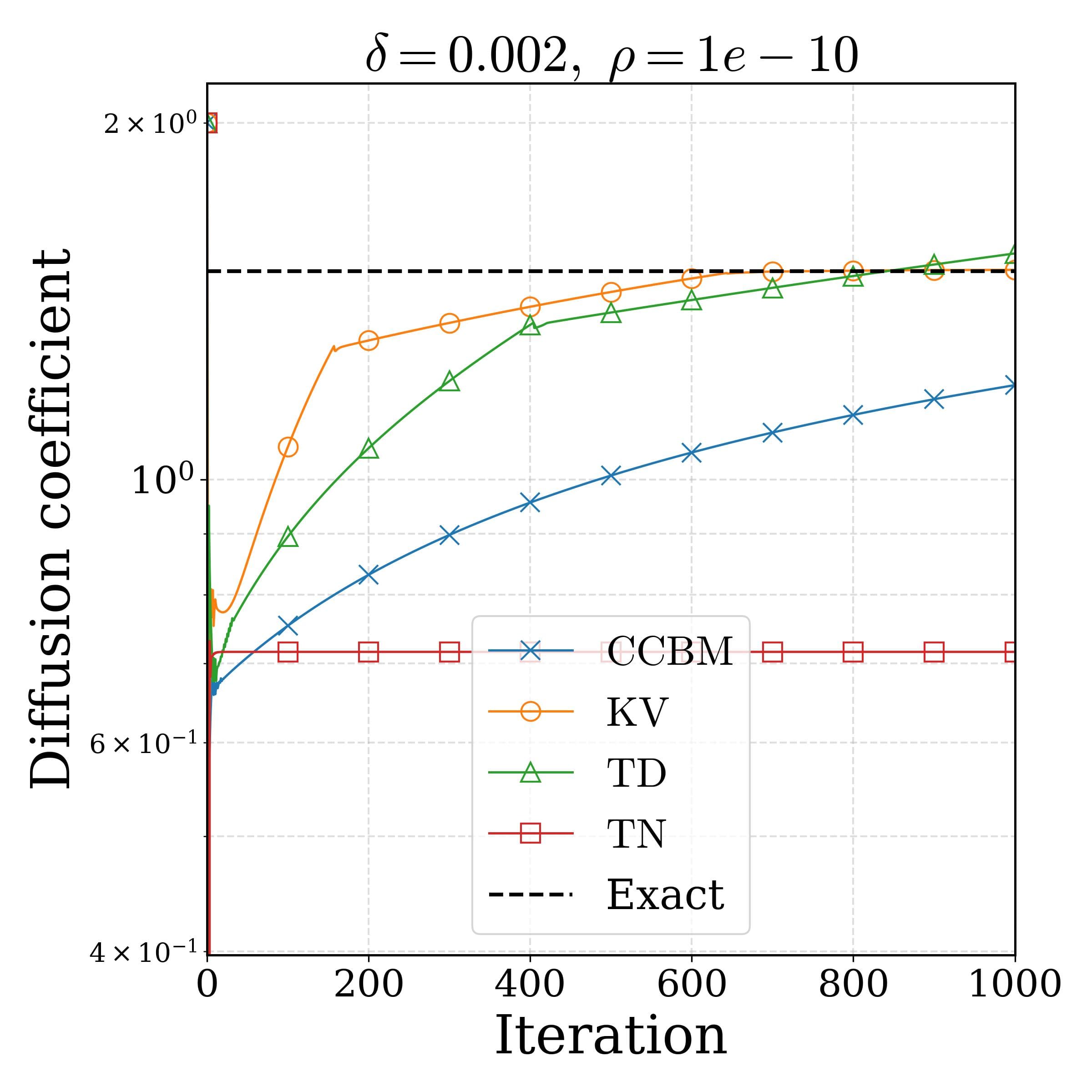}} \  
\resizebox{0.225\textwidth}{!}{\includegraphics{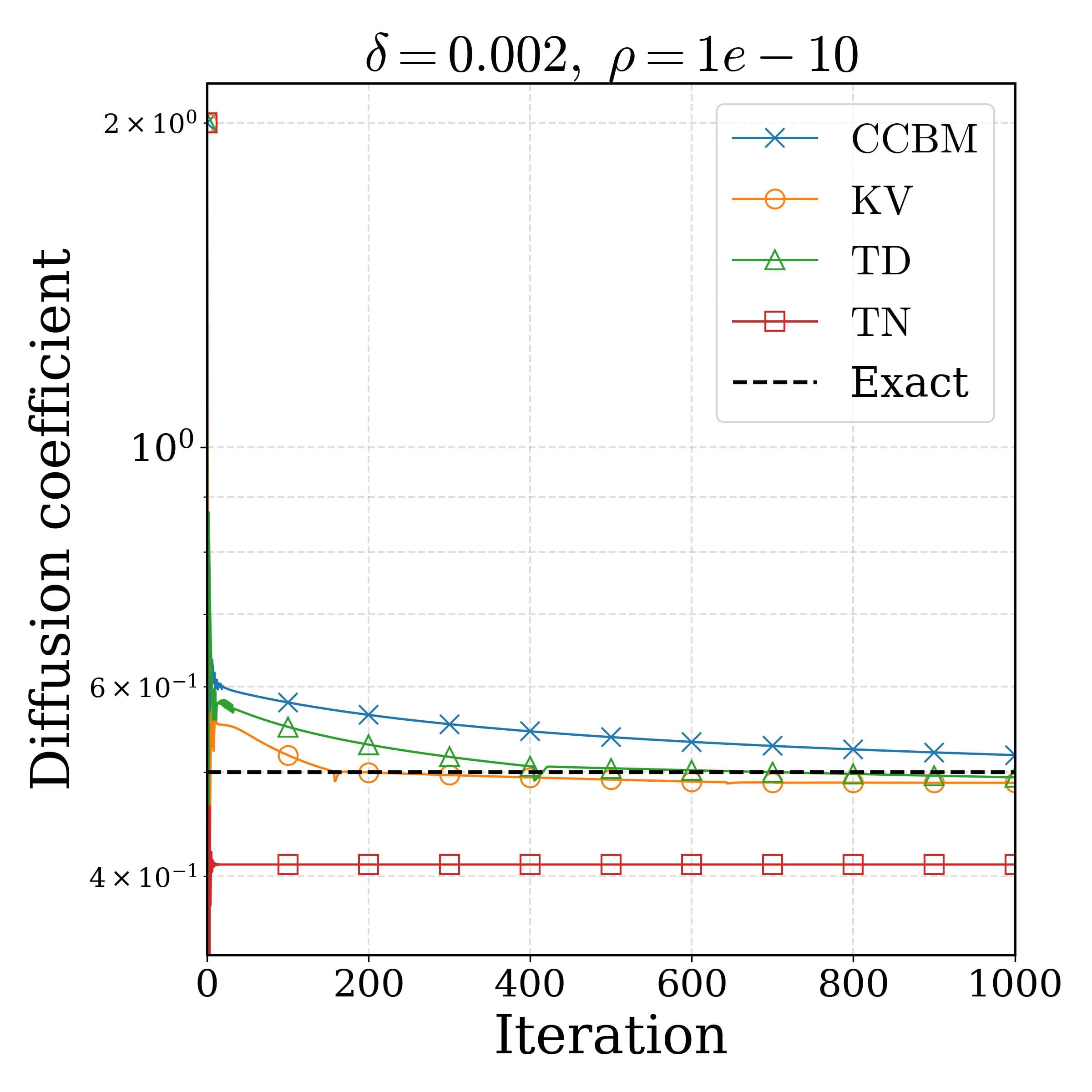}} \  
\resizebox{0.225\textwidth}{!}{\includegraphics{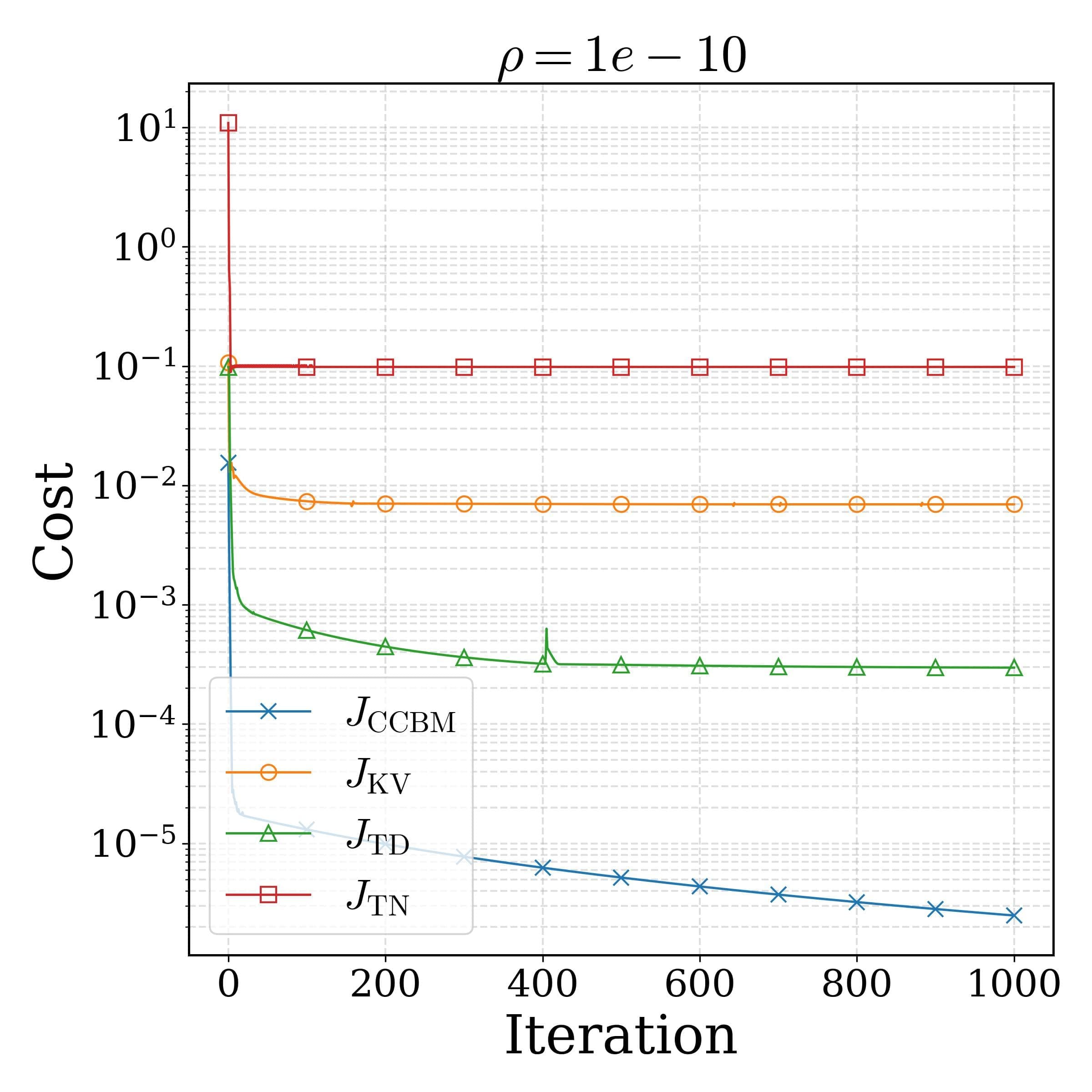}} \\[1em]
\resizebox{0.225\textwidth}{!}{\includegraphics{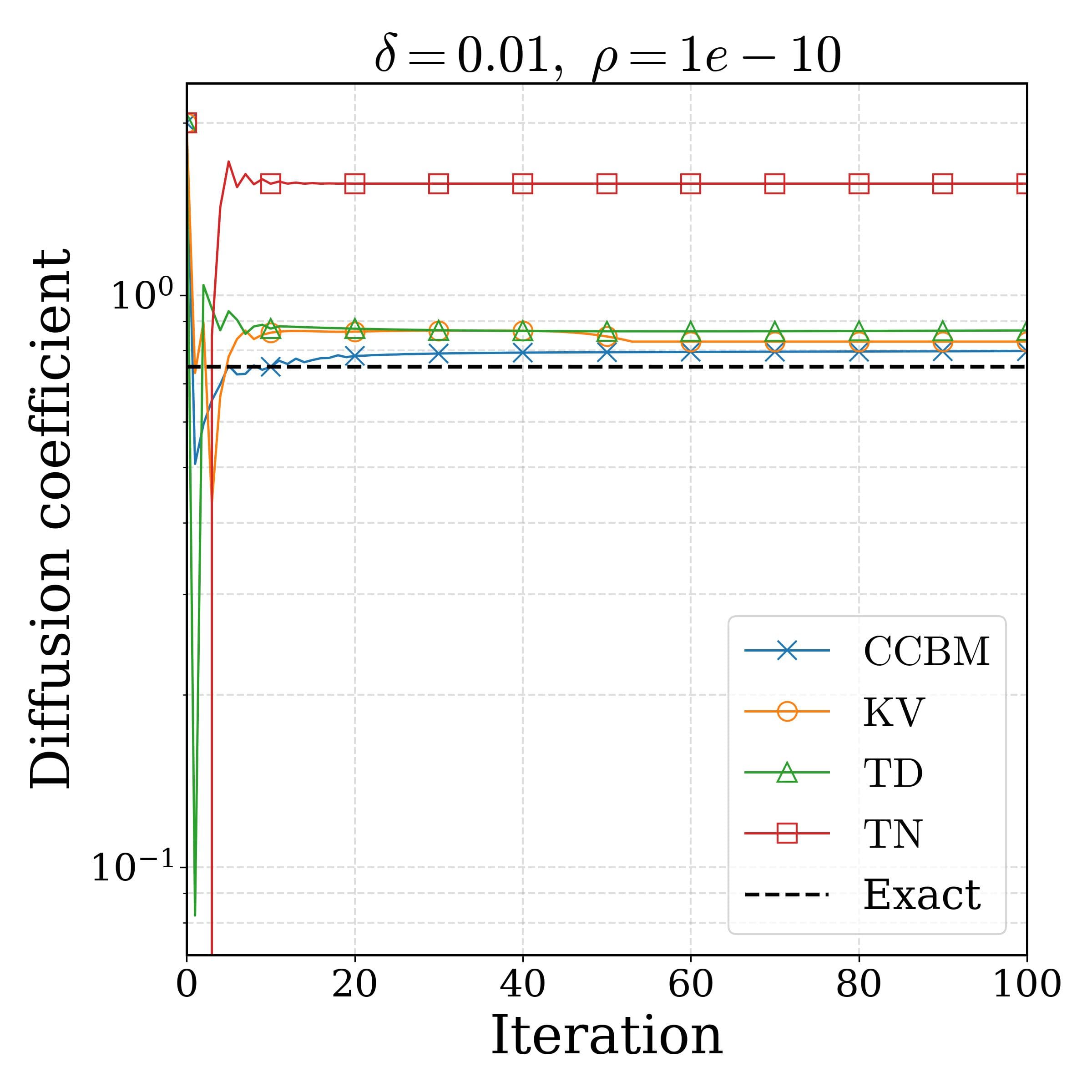}} \  
\resizebox{0.225\textwidth}{!}{\includegraphics{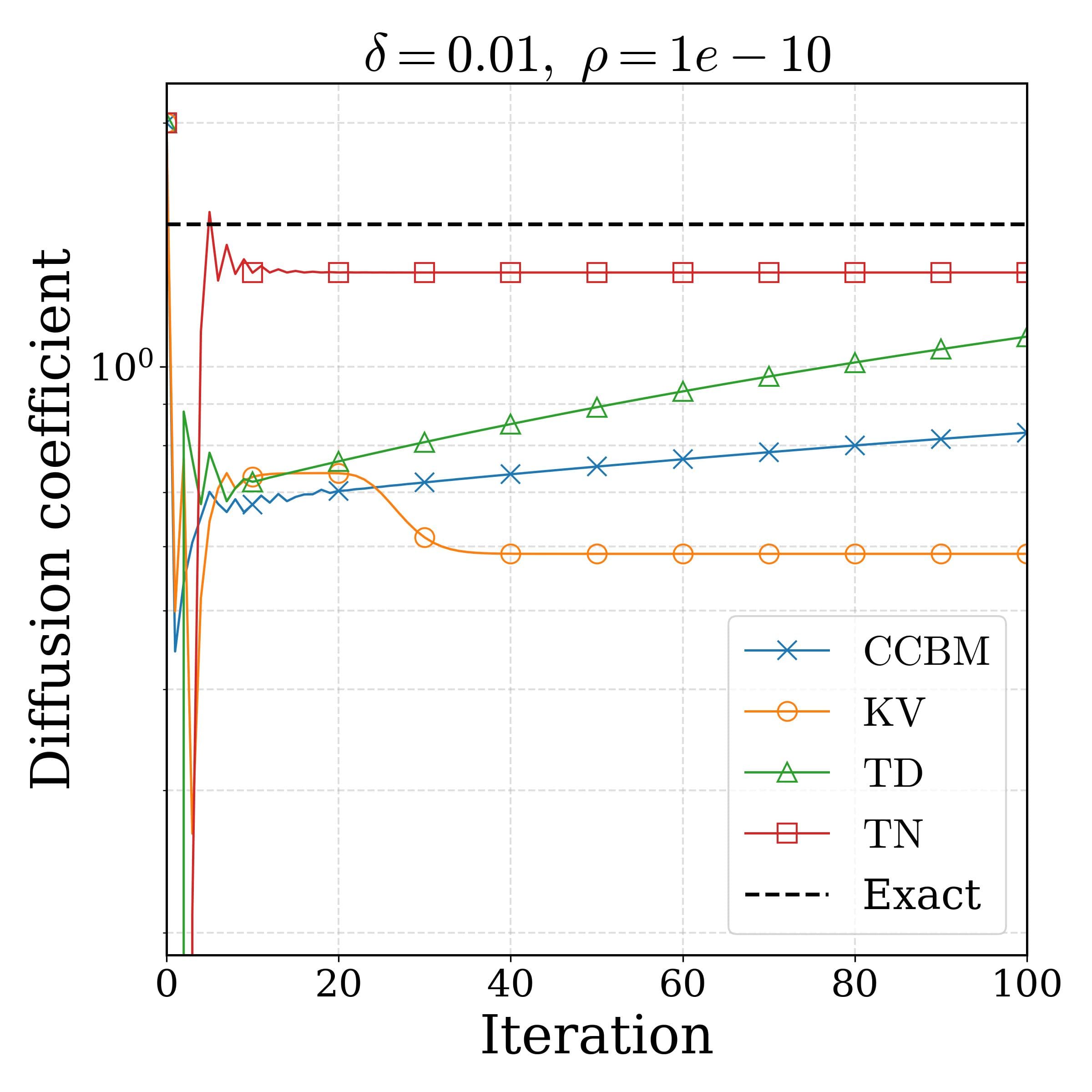}} \  
\resizebox{0.225\textwidth}{!}{\includegraphics{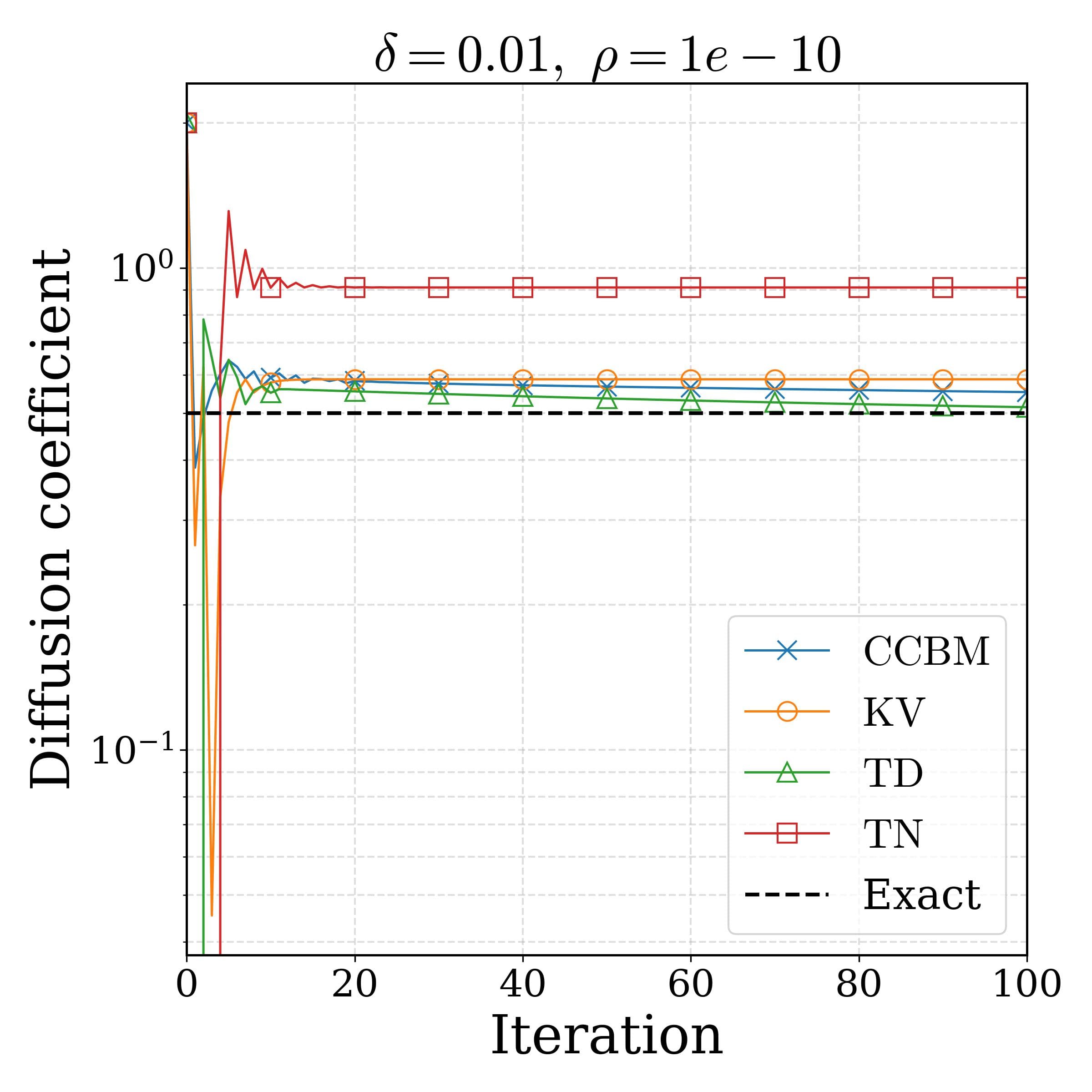}} \  
\resizebox{0.225\textwidth}{!}{\includegraphics{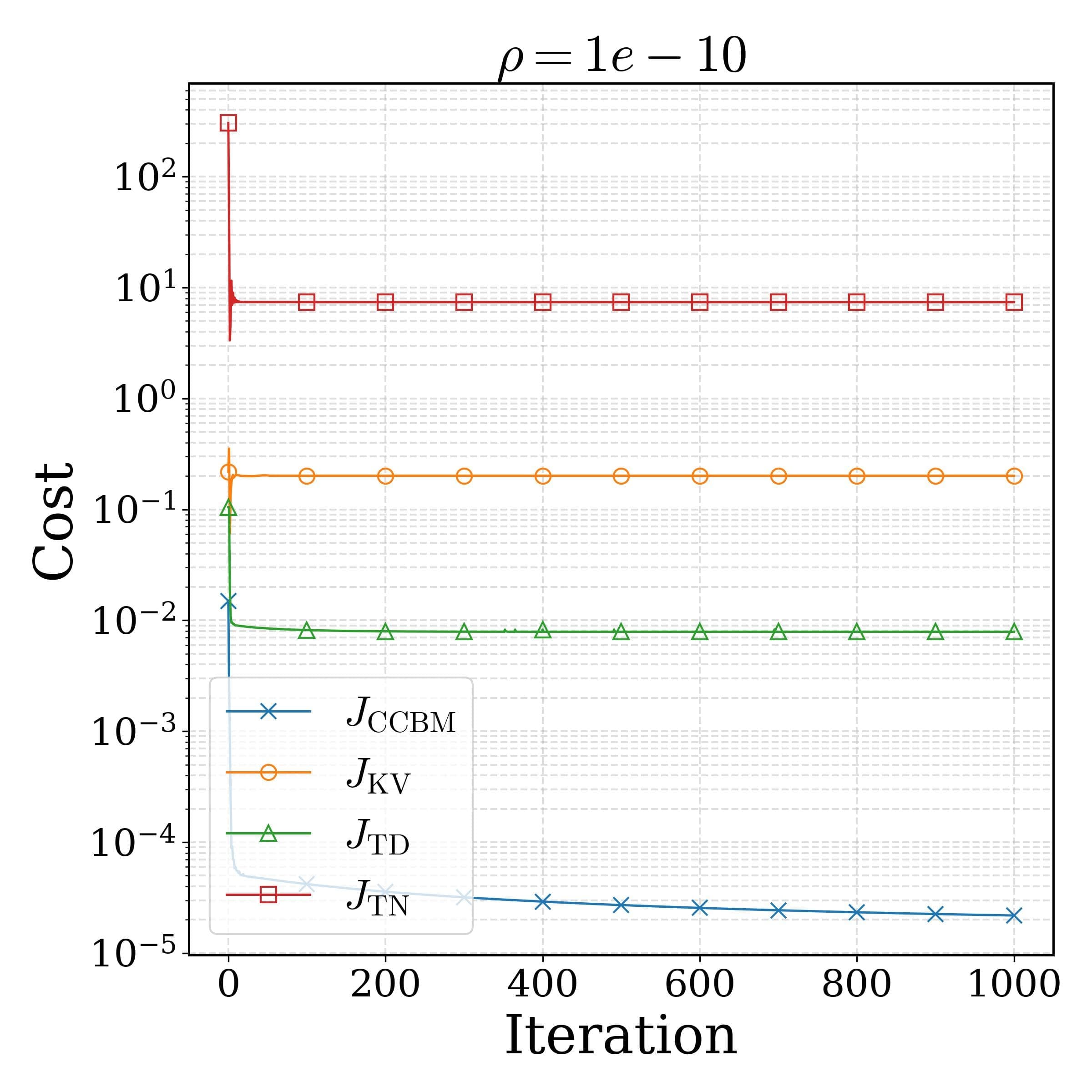}}
\caption{Reconstruction histories and cost histories}
\label{fig:three_subregions_histories_of_values}
\end{figure}

\begin{table}[htbp]
\centering
\resizebox{\textwidth}{!}{%
\begin{tabular}{c l c c c c}
\toprule
Noise & Parameter & CCBM & KV & TD & TN \\\midrule
\multirow{4}{*}{0}
 & $\alpha_{\varOmega_L}$ & 0.740086 (0.009914) & 0.746712 (0.003288) & 0.744950 (0.005050) & 0.804400 (0.054400) \\
 & $\alpha_{\varOmega_C}$ & 1.128213 (0.371787) & 1.432428 (0.067572) & 1.253930 (0.246070) & 0.454844 (1.045156) \\
 & $\alpha_{\varOmega_R}$ & 0.528290 (0.028290) & 0.503495 (0.003495) & 0.516044 (0.016044) & 0.407875 (0.092125) \\
\rowcolor{grey1} & Avg. Err. & 0.136664 & \cellcolor{yellow1}0.024785 & 0.089055 & 0.397227\\\hline
\multirow{4}{*}{0.001}
 & $\alpha_{\varOmega_L}$ & 0.747266 (0.002734) & 0.751368 (0.001368) & 0.755269 (0.005269) & -0.568317 (1.318317) \\
 & $\alpha_{\varOmega_C}$ & 1.156297 (0.343703) & 1.499315 (0.000685) & 1.334599 (0.165401) & -1.177029 (2.677029) \\
 & $\alpha_{\varOmega_R}$ & 0.524042 (0.024042) & 0.496646 (0.003354) & 0.508652 (0.008652) & -1.421624 (1.921624) \\
\rowcolor{grey1} & Avg. Err. & 0.123493 & \cellcolor{yellow1}0.001802 & 0.059774 & 1.972323\\\hline
\multirow{4}{*}{0.002}
 & $\alpha_{\varOmega_L}$ & 0.755081 (0.005081) & 0.740432 (0.009568) & 0.768172 (0.018172) & 0.982248 (0.232248) \\
 & $\alpha_{\varOmega_C}$ & 1.201932 (0.298068) & 1.503213 (0.003213) & 1.552304 (0.052304) & 0.715645 (0.784355) \\
 & $\alpha_{\varOmega_R}$ & 0.518350 (0.018350) & 0.488640 (0.011360) & 0.494292 (0.005708) & 0.410353 (0.089647) \\
\rowcolor{grey1} & Avg. Err. & 0.107166 & \cellcolor{yellow1}0.008047 & 0.025395 & 0.368750\\\hline
\multirow{4}{*}{0.0025}
 & $\alpha_{\varOmega_L}$ & 0.759160 (0.009160) & 0.862642 (0.112642) & 0.774324 (0.024324) & 1.034820 (0.284820) \\
 & $\alpha_{\varOmega_C}$ & 1.228561 (0.271439) & 0.566966 (0.933034) & 1.614653 (0.114653) & 0.756538 (0.743462) \\
 & $\alpha_{\varOmega_R}$ & 0.515296 (0.015296) & 0.566966 (0.066966) & 0.490531 (0.009469) & 0.409096 (0.090904) \\
\rowcolor{grey1} & Avg. Err. & 0.098632 & 0.370881 & \cellcolor{yellow1}0.049482 & 0.373062\\\hline
\multirow{4}{*}{0.005}
 & $\alpha_{\varOmega_L}$ & 0.780975 (0.030975) & 0.948906 (0.198906) & 0.802153 (0.052153) & 1.673080 (0.923080) \\
 & $\alpha_{\varOmega_C}$ & 1.387665 (0.112335) & 0.546452 (0.953548) & 1.520016 (0.020016) & 1.348167 (0.151833) \\
 & $\alpha_{\varOmega_R}$ & 0.499592 (0.000408) & 0.546452 (0.046452) & 0.491849 (0.008151) & 0.436586 (0.063414) \\
\rowcolor{grey1} & Avg. Err. & 0.047906 & 0.399635 & \cellcolor{yellow1}0.026773 & 0.379442\\\hline
\multirow{4}{*}{0.01}
 & $\alpha_{\varOmega_L}$ & 0.829310 (0.079310) & 0.829252 (0.079252) & 0.891907 (0.141907) & 1.566150 (0.816150) \\
 & $\alpha_{\varOmega_C}$ & 1.726544 (0.226544) & 0.587548 (0.912452) & 1.946164 (0.446164) & 1.307170 (0.192830) \\
 & $\alpha_{\varOmega_R}$ & 0.473853 (0.026147) & 0.587548 (0.087548) & 0.468012 (0.031988) & 0.911091 (0.411091) \\
\rowcolor{grey1} & Avg. Err. & \cellcolor{yellow1}0.110667 & 0.359751 & 0.206686 & 0.473357\\\hline
\multirow{4}{*}{0.03}
 & $\alpha_{\varOmega_L}$ & 0.949163 (0.199163) & 0.234017 (0.515983) & 0.963771 (0.213771) & 2.000000 (1.250000) \\
 & $\alpha_{\varOmega_C}$ & 0.809424 (0.690576) & 0.075030 (1.424970) & 0.538981 (0.961019) & 2.000000 (0.500000) \\
 & $\alpha_{\varOmega_R}$ & 0.532718 (0.032718) & -0.111561 (0.611561) & 0.538980 (0.038980) & 2.000000 (1.500000) \\
\rowcolor{grey1} & Avg. Err. & \cellcolor{yellow1}0.307486 & 0.850838 & 0.404590 & 1.083333\\\hline
\multirow{4}{*}{0.05}
 & $\alpha_{\varOmega_L}$ & 1.108570 (0.358570) & 0.268648 (0.481352) & 1.063890 (0.313890) & 27.772400 (27.022400) \\
 & $\alpha_{\varOmega_C}$ & 0.887906 (0.612094) & 0.114273 (1.385727) & 0.537079 (0.962921) & 23.340143 (21.840143) \\
 & $\alpha_{\varOmega_R}$ & 0.511133 (0.011133) & -0.061770 (0.561770) & 0.537079 (0.037079) & -0.096416 (0.596416) \\
\rowcolor{grey1} & Avg. Err. & \cellcolor{yellow1}0.327266 & 0.809616 & 0.437963 & 16.486320\\\hline
\bottomrule
\end{tabular}%
}
\caption{Reconstructed values with absolute errors, and average absolute errors per method and noise level. Exact parameter values are denoted as $\alpha_{\varOmega_L} = 0.75$, $\alpha_{\varOmega_C} = 1.5$, and $\alpha_{\varOmega_R} = 0.5$.}
\label{tab:three_subregions}
\end{table}

\subsubsection{Four-subregions case}
For the following example, the boundary data are given by $g(x_1,x_2) = \exp\!\bigl(\sin(\pi x_1)\sin(\pi x_2)\bigr)$.

\begin{example}[Four-subregions case]
Consider the domain $\varOmega = (-1,1)^2$ partitioned into four quadrants:
\[
\begin{aligned}
\varOmega_{Q1} &= \{ (x_1,x_2) : x_1 \ge 0,\ x_2 \ge 0 \},\\
\varOmega_{Q2} &= \{ (x_1,x_2) : x_1 < 0,\ x_2 \ge 0 \},\\
\varOmega_{Q3} &= \{ (x_1,x_2) : x_1 < 0,\ x_2 < 0 \},\\
\varOmega_{Q4} &= \{ (x_1,x_2) : x_1 \ge 0,\ x_2 < 0 \}.
\end{aligned}
\]
The diffusion coefficient is piecewise constant,
\[
\alpha|_{\varOmega_{Q1}}=0.25,\quad
\alpha|_{\varOmega_{Q2}}=0.50,\quad
\alpha|_{\varOmega_{Q3}}=0.75,\quad
\alpha|_{\varOmega_{Q4}}=1.00.
\]
A pick-a-point restriction is applied: after each update, $\alpha$ is set uniformly on each quadrant using the values at
$(\xi,\xi)$,
$(-\xi,\xi)$,
$(-\xi,-\xi)$,
and $(\xi,-\xi)$, respectively, where $\xi \in (0,1)$.
\end{example}

Tables~\ref{tab:four_quadrants_complete_xi0.9}--\ref{tab:four_quadrants_complete_xi0.5} summarize the reconstructed diffusion coefficients for the four-subregions example across different noise levels $\delta$ and pick-a-point parameters $\xi$, with relative errors reported in parentheses. The average relative error over the four quadrants is shown for each method, and the smallest average error is highlighted.

Overall, CCBM provides the most stable and accurate reconstructions, with average relative errors typically below $0.2$ for moderate noise. TD performs moderately well for intermediate noise ($\delta \approx 0.002$--$0.006$), KV shows mixed results, and TN is generally unreliable, exhibiting large deviations under higher noise. Variations in the pick-a-point parameter $\xi$ have only minor impact on reconstruction quality, and the relative ranking of methods remains consistent.

Figure~\ref{fig:four_subregions} shows the iteration histories of the reconstructed diffusion coefficients for $\delta = 0.01$, with dashed lines denoting the exact values. The cost functional and gradient norm histories (Figure~\ref{fig:four_subregions_cost_and_gradient_norms}) reveal rapid decreases during the first 100 iterations, followed by slower convergence over 1000 iterations. The zoomed-in plots highlight early-stage behavior, confirming smooth and stable progress of the optimization.

\begin{table}[htp!]
\centering
\resizebox{!}{0.45\textheight}{%
\begin{tabular}{c l c c c c}
\toprule
$\delta$ & Region & CCBM & KV & TD & TN \\
\midrule
\multirow{5}{*}{0.001}
& $\varOmega_{Q1}$ & 0.2046 (0.1817) & 0.1652 (0.3393) & 0.0876 (0.6495) & 0.1413 (0.4347) \\
& $\varOmega_{Q2}$ & 0.2459 (0.5082) & 0.2068 (0.5864) & 0.3304 (0.3392) & 0.2501 (0.4998) \\
& $\varOmega_{Q3}$ & 0.7180 (0.0426) & 0.6948 (0.0736) & 0.6904 (0.0795) & 0.2712 (0.6385) \\
& $\varOmega_{Q4}$ & 1.0057 (0.0057) & 0.9920 (0.0080) & 0.9643 (0.0357) & 0.4476 (0.5524) \\
\rowcolor{grey1}
& Avg. Err. & \cellcolor{yellow1}0.1846 & 0.2518 & 0.2760 & 0.5314 \\
\midrule
\multirow{5}{*}{0.002}
& $\varOmega_{Q1}$ & 0.2016 (0.1937) & 0.1661 (0.3357) & 0.1815 (0.2741) & 13.9312 (54.7250) \\
& $\varOmega_{Q2}$ & 0.2969 (0.4062) & 0.1996 (0.6009) & 0.2135 (0.5730) & 12.6856 (24.3712) \\
& $\varOmega_{Q3}$ & 0.7272 (0.0304) & 0.6672 (0.1104) & 0.7099 (0.0534) & 0.8175 (0.0900) \\
& $\varOmega_{Q4}$ & 0.9991 (0.0009) & 0.9399 (0.0601) & 0.9965 (0.0035) & 0.8588 (0.1412) \\
\rowcolor{grey1}
& Avg. Err. & \cellcolor{yellow1}0.1578 & 0.2768 & 0.2260 & 19.8319 \\
\midrule
\multirow{5}{*}{0.003}
& $\varOmega_{Q1}$ & 0.2031 (0.1876) & 0.0642 (0.7433) & 0.1821 (0.2715) & 41.6961 (165.7843) \\
& $\varOmega_{Q2}$ & 0.2788 (0.4423) & 0.2053 (0.5894) & 0.2130 (0.5740) & 35.8780 (70.7559) \\
& $\varOmega_{Q3}$ & 0.7300 (0.0267) & 0.6307 (0.1591) & 0.7129 (0.0494) & 0.9643 (0.2858) \\
& $\varOmega_{Q4}$ & 0.9979 (0.0021) & 0.8752 (0.1248) & 0.9992 (0.0008) & 0.7625 (0.2375) \\
\rowcolor{grey1}
& Avg. Err. & \cellcolor{yellow1}0.1647 & 0.4042 & 0.2239 & 59.2659 \\
\midrule
\multirow{5}{*}{0.004}
& $\varOmega_{Q1}$ & 0.2077 (0.1694) & 0.1490 (0.4041) & 0.1821 (0.2717) & 58.9242 (234.6967) \\
& $\varOmega_{Q2}$ & 0.3773 (0.2453) & 0.2066 (0.5868) & 0.2115 (0.5771) & 49.2954 (97.5909) \\
& $\varOmega_{Q3}$ & 0.7460 (0.0053) & 0.7186 (0.0418) & 0.7194 (0.0408) & 1.1755 (0.5674) \\
& $\varOmega_{Q4}$ & 0.9904 (0.0096) & 0.9708 (0.0292) & 0.9815 (0.0185) & 0.6848 (0.3152) \\
\rowcolor{grey1}
& Avg. Err. & \cellcolor{yellow1}0.1074 & 0.2655 & 0.2270 & 83.2926 \\
\midrule
\multirow{5}{*}{0.005}
& $\varOmega_{Q1}$ & 0.2051 (0.1797) & 0.1684 (0.3265) & 0.1818 (0.2728) & 0.3386 (0.3544) \\
& $\varOmega_{Q2}$ & 0.2723 (0.4553) & 0.1857 (0.6285) & 0.2158 (0.5685) & 0.5171 (0.0343) \\
& $\varOmega_{Q3}$ & 0.7386 (0.0152) & 0.6151 (0.1799) & 0.7165 (0.0447) & 1.7723 (1.3631) \\
& $\varOmega_{Q4}$ & 0.9933 (0.0067) & 0.7953 (0.2047) & 0.9874 (0.0126) & 0.5391 (0.4609) \\
\rowcolor{grey1}
& Avg. Err. & \cellcolor{yellow1}0.1642 & 0.3349 & 0.2247 & 0.5532 \\
\midrule
\multirow{5}{*}{0.006}
& $\varOmega_{Q1}$ & 0.2154 (0.1385) & 0.1656 (0.3376) & 0.1822 (0.2711) & 0.2127 (0.1491) \\
& $\varOmega_{Q2}$ & 0.3439 (0.3122) & 0.2011 (0.5977) & 0.2125 (0.5749) & 0.4402 (0.1196) \\
& $\varOmega_{Q3}$ & 0.7534 (0.0046) & 0.7119 (0.0508) & 0.7296 (0.0272) & 2.0657 (1.7543) \\
& $\varOmega_{Q4}$ & 0.9881 (0.0119) & 0.9169 (0.0831) & 0.9706 (0.0294) & 0.4898 (0.5102) \\
\rowcolor{grey1}
& Avg. Err. & \cellcolor{yellow1}0.1168 & 0.2673 & 0.2257 & 0.6333 \\
\midrule
\multirow{5}{*}{0.007}
& $\varOmega_{Q1}$ & 0.2058 (0.1767) & 0.0989 (0.6044) & 0.1833 (0.2667) & 0.2891 (0.1564) \\
& $\varOmega_{Q2}$ & 0.2586 (0.4829) & 0.1952 (0.6095) & 0.2125 (0.5750) & 0.7830 (0.5661) \\
& $\varOmega_{Q3}$ & 0.7463 (0.0050) & 0.6837 (0.0884) & 0.7383 (0.0156) & 2.5762 (2.4350) \\
& $\varOmega_{Q4}$ & 0.9890 (0.0110) & 0.7986 (0.2014) & 0.9751 (0.0249) & 0.5731 (0.4269) \\
\rowcolor{grey1}
& Avg. Err. & \cellcolor{yellow1}0.1689 & 0.3759 & 0.2206 & 0.8961 \\
\midrule
\multirow{5}{*}{0.008}
& $\varOmega_{Q1}$ & 0.2899 (0.1597) & 0.1650 (0.3401) & 0.1826 (0.2695) & 1.3825 (4.5301) \\
& $\varOmega_{Q2}$ & 0.5128 (0.0257) & 0.1736 (0.6529) & 0.2129 (0.5743) & 0.5350 (0.0701) \\
& $\varOmega_{Q3}$ & 0.7964 (0.0618) & 0.5495 (0.2673) & 0.7458 (0.0057) & 2.5601 (2.4134) \\
& $\varOmega_{Q4}$ & 0.9826 (0.0174) & 0.6686 (0.3314) & 0.9719 (0.0281) & 0.3958 (0.6042) \\
\rowcolor{grey1}
& Avg. Err. & \cellcolor{yellow1}0.0662 & 0.3979 & 0.2194 & 1.9045 \\
\midrule
\multirow{5}{*}{0.009}
& $\varOmega_{Q1}$ & 0.2000 (0.1998) & 0.1697 (0.3210) & 0.1823 (0.2706) & 0.0908 (0.6370) \\
& $\varOmega_{Q2}$ & 0.2214 (0.5572) & 0.1840 (0.6321) & 0.2203 (0.5595) & 0.3418 (0.3164) \\
& $\varOmega_{Q3}$ & 0.7534 (0.0046) & 0.6181 (0.1758) & 0.7391 (0.0145) & 1.8331 (1.4442) \\
& $\varOmega_{Q4}$ & 0.9820 (0.0180) & 0.7430 (0.2570) & 0.9880 (0.0120) & 1.3235 (0.3235) \\
\rowcolor{grey1}
& Avg. Err. & \cellcolor{yellow1}0.1949 & 0.3465 & 0.2142 & 0.6803 \\
\midrule
\multirow{5}{*}{0.01}
& $\varOmega_{Q1}$ & 0.1996 (0.2015) & 0.1671 (0.3316) & 0.1840 (0.2639) & 0.2660 (0.0642) \\
& $\varOmega_{Q2}$ & 0.2216 (0.5567) & 0.1954 (0.6091) & 0.2114 (0.5773) & 0.3844 (0.2312) \\
& $\varOmega_{Q3}$ & 0.7568 (0.0090) & 0.6998 (0.0669) & 0.7471 (0.0038) & 2.9767 (2.9689) \\
& $\varOmega_{Q4}$ & 0.9820 (0.0180) & 0.8068 (0.1932) & 0.9484 (0.0516) & 0.3512 (0.6488) \\
\rowcolor{grey1}
& Avg. Err. & \cellcolor{yellow1}0.1963 & 0.3002 & 0.2242 & 0.9783 \\
\bottomrule
\end{tabular}%
}
\caption{Reconstructed diffusion coefficients and relative errors (parentheses)  $\xi = 0.9$.}
\label{tab:four_quadrants_complete_xi0.9}
\end{table}

\begin{table}[htp!]
\centering
\resizebox{!}{0.45\textheight}{%
\begin{tabular}{c l c c c c}
\toprule
$\delta$ & Region & CCBM & KV & TD & TN \\
\midrule
\multirow{5}{*}{0.001}
 & $\varOmega_{Q1}$ & 0.2457 (0.0170) & 0.1980 (0.2080) & 0.2220 (0.1119) & 0.1918 (0.2329) \\
 & $\varOmega_{Q2}$ & 0.2929 (0.4141) & 0.2336 (0.5327) & 0.2827 (0.4346) & 0.2424 (0.5152) \\
 & $\varOmega_{Q3}$ & 0.7319 (0.0241) & 0.7087 (0.0551) & 0.7269 (0.0309) & 0.2424 (0.6768) \\
 & $\varOmega_{Q4}$ & 1.0093 (0.0093) & 0.9862 (0.0138) & 1.0020 (0.0020) & 0.1918 (0.8082) \\
\rowcolor{grey1}
 & Avg. Err. & \cellcolor{yellow1}0.1161 & 0.2024 & 0.1456 & 0.5589 \\
\midrule
\multirow{5}{*}{0.002}
 & $\varOmega_{Q1}$ & 0.2437 (0.0252) & 0.1351 (0.4595) & 0.2190 (0.1240) & 11.6115 (45.4460) \\
 & $\varOmega_{Q2}$ & 0.2892 (0.4216) & 0.2290 (0.5420) & 0.2433 (0.5134) & 11.0857 (21.1714) \\
 & $\varOmega_{Q3}$ & 0.7356 (0.0192) & 0.6746 (0.1005) & 0.7242 (0.0344) & 0.8408 (0.1210) \\
 & $\varOmega_{Q4}$ & 1.0066 (0.0066) & 0.9363 (0.0637) & 1.0024 (0.0024) & 0.9152 (0.0848) \\
\rowcolor{grey1}
 & Avg. Err. & \cellcolor{yellow1}0.1183 & 0.2929 & 0.1686 & 14.4043 \\
\midrule
\multirow{5}{*}{0.003}
 & $\varOmega_{Q1}$ & 0.2408 (0.0368) & 0.0937 (0.6253) & 0.2191 (0.1235) & 35.4089 (140.6357) \\
 & $\varOmega_{Q2}$ & 0.2798 (0.4403) & 0.2166 (0.5668) & 0.2448 (0.5104) & 32.2431 (63.4862) \\
 & $\varOmega_{Q3}$ & 0.7385 (0.0153) & 0.6271 (0.1639) & 0.7265 (0.0314) & 1.0013 (0.3351) \\
 & $\varOmega_{Q4}$ & 1.0042 (0.0042) & 0.9314 (0.0686) & 1.0028 (0.0028) & 0.8507 (0.1493) \\
\rowcolor{grey1}
 & Avg. Err. & \cellcolor{yellow1}0.1242 & 0.2412 & 0.1670 & 17.9356 \\
\midrule
\multirow{5}{*}{0.004}
 & $\varOmega_{Q1}$ & 0.2580 (0.0319) & 0.1470 (0.4118) & 0.2184 (0.1263) & 25.5883 (101.3531) \\
 & $\varOmega_{Q2}$ & 0.3298 (0.3403) & 0.2203 (0.5593) & 0.2420 (0.5161) & 22.9964 (44.9929) \\
 & $\varOmega_{Q3}$ & 0.7520 (0.0026) & 0.6850 (0.0867) & 0.7310 (0.0253) & 1.2384 (0.6512) \\
 & $\varOmega_{Q4}$ & 1.0006 (0.0006) & 0.9536 (0.0464) & 0.9925 (0.0076) & 0.7908 (0.2092) \\
\rowcolor{grey1}
 & Avg. Err. & \cellcolor{yellow1}0.0939 & 0.1985 & 0.1953 & 32.7071 \\
\midrule
\multirow{5}{*}{0.005}
 & $\varOmega_{Q1}$ & 0.2690 (0.0760) & 0.1918 (0.2327) & 0.2174 (0.1304) & 0.2328 (0.0689) \\
 & $\varOmega_{Q2}$ & 0.3524 (0.2952) & 0.2186 (0.5628) & 0.2445 (0.5111) & 0.4568 (0.0865) \\
 & $\varOmega_{Q3}$ & 0.7613 (0.0150) & 0.6920 (0.0773) & 0.7420 (0.0106) & 1.7817 (1.3756) \\
 & $\varOmega_{Q4}$ & 0.9978 (0.0022) & 0.8959 (0.1041) & 0.9861 (0.0139) & 0.6572 (0.3428) \\
\rowcolor{grey1}
 & Avg. Err. & \cellcolor{yellow1}0.1021 & 0.2727 & 0.2260 & 0.6440 \\
\midrule
\multirow{5}{*}{0.006}
 & $\varOmega_{Q1}$ & 0.2545 (0.0181) & 0.1252 (0.4993) & 0.2192 (0.1232) & 2.3093 (8.2371) \\
 & $\varOmega_{Q2}$ & 0.3118 (0.3764) & 0.2143 (0.5714) & 0.2511 (0.4978) & 2.3302 (3.6604) \\
 & $\varOmega_{Q3}$ & 0.7587 (0.0116) & 0.7109 (0.0521) & 0.7492 (0.0011) & 2.3302 (2.1069) \\
 & $\varOmega_{Q4}$ & 0.9961 (0.0039) & 0.9675 (0.0325) & 0.9821 (0.0179) & 2.3093 (1.3093) \\
\rowcolor{grey1}
 & Avg. Err. & \cellcolor{yellow1}0.1028 & 0.2891 & 0.2125 & 2.0662 \\
\midrule
\multirow{5}{*}{0.007}
 & $\varOmega_{Q1}$ & 0.2305 (0.0781) & 0.1885 (0.2462) & 0.2187 (0.1251) & 0.1095 (0.5619) \\
 & $\varOmega_{Q2}$ & 0.2520 (0.4961) & 0.2014 (0.5973) & 0.2456 (0.5088) & 0.8284 (0.6568) \\
 & $\varOmega_{Q3}$ & 0.7529 (0.0039) & 0.6098 (0.1870) & 0.7457 (0.0057) & 2.2004 (1.9339) \\
 & $\varOmega_{Q4}$ & 0.9934 (0.0066) & 0.7551 (0.2449) & 0.9931 (0.0069) & 0.7230 (0.2770) \\
\rowcolor{grey1}
 & Avg. Err. & \cellcolor{yellow1}0.1462 & 0.3188 & 0.1916 & 0.8649 \\
\midrule
\multirow{5}{*}{0.008}
 & $\varOmega_{Q1}$ & 0.2299 (0.0803) & 0.1892 (0.2433) & 0.2167 (0.1331) & -0.0729 (1.2917) \\
 & $\varOmega_{Q2}$ & 0.2519 (0.4962) & 0.2192 (0.5615) & 0.2439 (0.5122) & 0.2037 (0.5927) \\
 & $\varOmega_{Q3}$ & 0.7577 (0.0103) & 0.7206 (0.0392) & 0.7573 (0.0097) & 1.8169 (1.4225) \\
 & $\varOmega_{Q4}$ & 0.9902 (0.0098) & 0.9088 (0.0912) & 0.9684 (0.0316) & 0.2843 (0.7157) \\
\rowcolor{grey1}
 & Avg. Err. & \cellcolor{yellow1}0.1491 & 0.2335 & 0.1782 & 0.8288 \\
\midrule
\multirow{5}{*}{0.009}
 & $\varOmega_{Q1}$ & 0.2349 (0.0604) & 0.1268 (0.4930) & 0.2191 (0.1235) & -1.1447e11 (4.5788e11) \\
 & $\varOmega_{Q2}$ & 0.2627 (0.4746) & 0.2016 (0.5968) & 0.2441 (0.5118) & -1.1447e11 (2.2894e11) \\
 & $\varOmega_{Q3}$ & 0.7633 (0.0177) & 0.6939 (0.0748) & 0.7635 (0.0180) & -1.1447e11 (1.5263e11) \\
 & $\varOmega_{Q4}$ & 0.9890 (0.0110) & 0.8875 (0.1125) & 0.9718 (0.0282) & -1.1447e11 (1.1447e11) \\
\rowcolor{grey1}
 & Avg. Err. & \cellcolor{yellow1}0.0152 & 0.3193 & 0.1709 & 1.1447e11 \\
\midrule
\multirow{5}{*}{0.010}
 & $\varOmega_{Q1}$ & 0.3571 (0.4283) & 0.1767 (0.2932) & 0.2173 (0.1307) & -0.5609 (3.2437) \\
 & $\varOmega_{Q2}$ & 0.4667 (0.0667) & 0.1812 (0.6376) & 0.2439 (0.5121) & 0.0470 (0.9060) \\
 & $\varOmega_{Q3}$ & 0.8141 (0.0854) & 0.5072 (0.3237) & 0.7634 (0.0179) & 1.7888 (1.3851) \\
 & $\varOmega_{Q4}$ & 0.9849 (0.0151) & 0.5995 (0.4005) & 0.9628 (0.0372) & 0.0659 (0.9341) \\
\rowcolor{grey1}
 & Avg. Err. & \cellcolor{yellow1}0.1489 & 0.3313 & 0.2428 & 0.9112 \\
\bottomrule
\end{tabular}%
}
\caption{Reconstructed diffusion coefficients and relative errors (parentheses)  $\xi = 0.7$.}
\label{tab:four_quadrants_complete_xi0.7}
\end{table}

\begin{table}[htp!]
\centering
\resizebox{!}{0.45\textheight}{%
\begin{tabular}{c l c c c c}
\toprule
$\delta$ & Region & CCBM & KV & TD & TN \\
\midrule
\multirow{5}{*}{0.001}
 & $\varOmega_{Q1}$ & 0.3392 (0.3568) & 0.2836 (0.1344) & 0.3129 (0.2516) & 0.3029 (0.2117) \\
 & $\varOmega_{Q2}$ & 0.3517 (0.2966) & 0.3028 (0.3943) & 0.3161 (0.3678) & 0.3029 (0.3942) \\
 & $\varOmega_{Q3}$ & 0.7538 (0.0051) & 0.7316 (0.0246) & 0.7461 (0.0052) & 0.3029 (0.5961) \\
 & $\varOmega_{Q4}$ & 1.0152 (0.0152) & 1.0014 (0.0014) & 1.0143 (0.0143) & 0.3029 (0.6971) \\
\rowcolor{grey1}
 & Avg. Err. & 0.1684 & \cellcolor{yellow1}0.1387 & 0.1597 & 0.4750 \\
\midrule
\multirow{5}{*}{$0.002$}
 & $\varOmega_{Q1}$ & 0.3372 (0.3486) & 0.2445 (0.0219) & 0.2275 (0.0902) & 0.3348 (0.3393) \\
 & $\varOmega_{Q2}$ & 0.3510 (0.2980) & 0.2929 (0.4141) & 0.4807 (0.0387) & 0.3348 (0.3303) \\
 & $\varOmega_{Q3}$ & 0.7582 (0.0110) & 0.6987 (0.0684) & 0.7521 (0.0028) & 0.3348 (0.5536) \\
 & $\varOmega_{Q4}$ & 1.0124 (0.0124) & 0.9911 (0.0089) & 0.9914 (0.0086) & 0.3348 (0.6652) \\
\rowcolor{grey1}
 & Avg. Err. & 0.1675 & 0.1283 & \cellcolor{yellow1}0.0351 & 0.4710 \\
\midrule
\multirow{5}{*}{$0.003$}
 & $\varOmega_{Q1}$ & 0.3639 (0.4558) & 0.2727 (0.0908) & 0.3085 (0.2340) & 0.2170 (0.1319) \\
 & $\varOmega_{Q2}$ & 0.4066 (0.1868) & 0.2940 (0.4119) & 0.3166 (0.3668) & 0.3691 (0.2617) \\
 & $\varOmega_{Q3}$ & 0.7729 (0.0305) & 0.7563 (0.0085) & 0.7561 (0.0082) & 1.3034 (0.7379) \\
 & $\varOmega_{Q4}$ & 1.0078 (0.0078) & 0.9771 (0.0229) & 1.0045 (0.0045) & 0.8744 (0.1256) \\
\rowcolor{grey1}
 & Avg. Err. & 0.1702 & 0.1358 & \cellcolor{yellow1}0.1534 & 0.3143 \\
\midrule
\multirow{5}{*}{$0.004$}
 & $\varOmega_{Q1}$ & 0.3368 (0.3471) & 0.2648 (0.0593) & 0.1933 (0.2269) & 0.2077 (0.1693) \\
 & $\varOmega_{Q2}$ & 0.3525 (0.2951) & 0.2836 (0.4328) & 0.4597 (0.0806) & 0.4114 (0.1771) \\
 & $\varOmega_{Q3}$ & 0.7681 (0.0241) & 0.7364 (0.0181) & 0.7518 (0.0024) & 1.4999 (0.9998) \\
 & $\varOmega_{Q4}$ & 1.0071 (0.0071) & 0.9284 (0.0716) & 0.9797 (0.0203) & 0.8833 (0.1167) \\
\rowcolor{grey1}
 & Avg. Err. & 0.1683 & 0.1457 & \cellcolor{yellow1}0.0825 & 0.3657 \\
\midrule
\multirow{5}{*}{$0.005$}
 & $\varOmega_{Q1}$ & 0.3318 (0.3271) & 0.1654 (0.3383) & 0.2793 (0.1171) & 0.1862 (0.2553) \\
 & $\varOmega_{Q2}$ & 0.3437 (0.3125) & 0.2367 (0.5266) & 0.3691 (0.2618) & 0.4574 (0.0852) \\
 & $\varOmega_{Q3}$ & 0.7713 (0.0284) & 0.5917 (0.2110) & 0.7669 (0.0225) & 1.7421 (1.3228) \\
 & $\varOmega_{Q4}$ & 1.0047 (0.0047) & 0.8548 (0.1452) & 0.9912 (0.0088) & 0.9008 (0.0992) \\
\rowcolor{grey1}
 & Avg. Err. & 0.1682 & 0.3053 & \cellcolor{yellow1}0.1026 & 0.4406 \\
\midrule
\multirow{5}{*}{$0.006$}
 & $\varOmega_{Q1}$ & 0.3682 (0.4729) & 0.2375 (0.0502) & 0.2255 (0.0979) & 0.1745 (0.3021) \\
 & $\varOmega_{Q2}$ & 0.4070 (0.1860) & 0.2758 (0.4484) & 0.4253 (0.1495) & 0.5513 (0.1026) \\
 & $\varOmega_{Q3}$ & 0.7887 (0.0516) & 0.7440 (0.0080) & 0.7665 (0.0220) & 1.8744 (1.4992) \\
 & $\varOmega_{Q4}$ & 1.0002 (0.0002) & 0.9627 (0.0373) & 0.9786 (0.0214) & 0.8167 (0.1833) \\
\rowcolor{grey1}
 & Avg. Err. & 0.1777 & 0.1355 & \cellcolor{yellow1}0.0727 & 0.5218 \\
\midrule
\multirow{5}{*}{$0.007$}
 & $\varOmega_{Q1}$ & 0.3499 (0.3996) & 0.2356 (0.0577) & 0.2650 (0.0601) & -0.3008 (2.2032) \\
 & $\varOmega_{Q2}$ & 0.3735 (0.2530) & 0.2432 (0.5136) & 0.3864 (0.2273) & 0.2104 (0.5792) \\
 & $\varOmega_{Q3}$ & 0.7875 (0.0500) & 0.6090 (0.1880) & 0.7763 (0.0350) & 1.6207 (1.1610) \\
 & $\varOmega_{Q4}$ & 0.9987 (0.0013) & 0.7272 (0.2728) & 0.9814 (0.0186) & 0.5099 (0.4901) \\
\rowcolor{grey1}
 & Avg. Err. & 0.1760 & 0.2580 & \cellcolor{yellow1}0.0852 & 0.6061 \\
\midrule
\multirow{5}{*}{$0.008$}
 & $\varOmega_{Q1}$ & 0.3551 (0.4204) & 0.2474 (0.0106) & -0.0017 (1.0069) & 1.1168e11 (4.4672e11) \\
 & $\varOmega_{Q2}$ & 0.3801 (0.2399) & 0.2714 (0.4572) & 0.4792 (0.0416) & 1.1168e11 (2.2336e11) \\
 & $\varOmega_{Q3}$ & 0.7942 (0.0589) & 0.7358 (0.0189) & 0.6637 (0.1151) & 1.1168e11 (1.4891e11) \\
 & $\varOmega_{Q4}$ & 0.9960 (0.0040) & 0.9086 (0.0914) & 0.8751 (0.1249) & 1.1168e11 (1.1168e11) \\
\rowcolor{grey1}
 & Avg. Err. & 0.1808 & \cellcolor{yellow1}0.1445 & 0.3221 & 1.1168e11 \\
\midrule
\multirow{5}{*}{$0.009$}
 & $\varOmega_{Q1}$ & 0.3603 (0.4414) & 0.1570 (0.3718) & 0.3040 (0.2159) & -1.0390 (5.1561) \\
 & $\varOmega_{Q2}$ & 0.3860 (0.2280) & 0.2154 (0.5693) & 0.3164 (0.3672) & -0.2156 (1.4312) \\
 & $\varOmega_{Q3}$ & 0.8008 (0.0678) & 0.5953 (0.2063) & 0.7729 (0.0305) & 1.5407 (1.0543) \\
 & $\varOmega_{Q4}$ & 0.9934 (0.0066) & 0.7814 (0.2186) & 0.9876 (0.0124) & 0.0332 (0.9668) \\
\rowcolor{grey1}
 & Avg. Err. & 0.1859 & 0.3415 & \cellcolor{yellow1}0.1565 & 2.1521 \\
\midrule
\multirow{5}{*}{$0.010$}
 & $\varOmega_{Q1}$ & 0.3421 (0.3685) & 0.2267 (0.0933) & 0.2986 (0.1944) & -10.8965 (44.5859) \\
 & $\varOmega_{Q2}$ & 0.3587 (0.2826) & 0.2318 (0.5365) & 0.3210 (0.3581) & -6.4513 (13.9025) \\
 & $\varOmega_{Q3}$ & 0.8001 (0.0668) & 0.5737 (0.2351) & 0.7906 (0.0541) & -4.4919 (6.9892) \\
 & $\varOmega_{Q4}$ & 0.9916 (0.0084) & 0.6694 (0.3306) & 0.9762 (0.0238) & -8.4691 (9.4691) \\
\rowcolor{grey1}
 & Avg. Err. & 0.1766 & 0.2989 & \cellcolor{yellow1}0.1576 & 10.7367 \\
\bottomrule
\end{tabular}
}
\caption{Reconstructed diffusion coefficients and relative errors (parentheses)  $\xi = 0.5$.}
\label{tab:four_quadrants_complete_xi0.5}
\end{table}

\begin{table}[htp!]
\centering
\resizebox{!}{0.45\textheight}{%
\begin{tabular}{c l c c c c}
\toprule
$\delta$ & Region & CCBM & KV & TD & TN \\
\midrule
\multirow{5}{*}{0.001}
 & $\varOmega_{Q1}$ & 0.2457 (0.0170) & 0.1980 (0.2080) & 0.2220 (0.1119) & 0.1918 (0.2329) \\
 & $\varOmega_{Q2}$ & 0.2929 (0.4141) & 0.2336 (0.5327) & 0.2827 (0.4346) & 0.2424 (0.5152) \\
 & $\varOmega_{Q3}$ & 0.7319 (0.0241) & 0.7087 (0.0551) & 0.7269 (0.0309) & 0.2424 (0.6768) \\
 & $\varOmega_{Q4}$ & 1.0093 (0.0093) & 0.9862 (0.0138) & 1.0020 (0.0020) & 0.1918 (0.8082) \\
\rowcolor{grey1}
 & Avg. Err. & \cellcolor{yellow1}0.1161 & 0.2024 & 0.1449 & 0.4060 \\
\midrule
\multirow{5}{*}{$0.002$}
 & $\varOmega_{Q1}$ & 0.2437 (0.0252) & 0.1351 (0.4595) & 0.2190 (0.1240) & 11.6115 (45.4460) \\
 & $\varOmega_{Q2}$ & 0.2892 (0.4216) & 0.2290 (0.5420) & 0.2433 (0.5134) & 11.0857 (21.1714) \\
 & $\varOmega_{Q3}$ & 0.7356 (0.0192) & 0.6746 (0.1005) & 0.7242 (0.0344) & 0.8408 (0.1210) \\
 & $\varOmega_{Q4}$ & 1.0066 (0.0066) & 0.9363 (0.0637) & 1.0024 (0.0024) & 0.9152 (0.0848) \\
\rowcolor{grey1}
 & Avg. Err. & \cellcolor{yellow1}0.2181 & 0.2113 & 0.1686 & 6.5883 \\
\midrule
\multirow{5}{*}{$0.003$}
 & $\varOmega_{Q1}$ & 0.2408 (0.0368) & 0.0937 (0.6253) & 0.2191 (0.1235) & 35.4089 (140.6357) \\
 & $\varOmega_{Q2}$ & 0.2798 (0.4403) & 0.2166 (0.5668) & 0.2448 (0.5104) & 32.2431 (63.4862) \\
 & $\varOmega_{Q3}$ & 0.7385 (0.0153) & 0.6271 (0.1639) & 0.7265 (0.0314) & 1.0013 (0.3351) \\
 & $\varOmega_{Q4}$ & 1.0042 (0.0042) & 0.9314 (0.0686) & 1.0028 (0.0028) & 0.8507 (0.1493) \\
\rowcolor{grey1}
 & Avg. Err. & \cellcolor{yellow1}0.1241 & 0.3666 & 0.1670 & 17.4108 \\
\midrule
\multirow{5}{*}{$0.004$}
 & $\varOmega_{Q1}$ & 0.2580 (0.0319) & 0.1470 (0.4118) & 0.2184 (0.1263) & 25.5883 (101.3531) \\
 & $\varOmega_{Q2}$ & 0.3298 (0.3403) & 0.2203 (0.5593) & 0.2420 (0.5161) & 22.9964 (44.9929) \\
 & $\varOmega_{Q3}$ & 0.7520 (0.0026) & 0.6850 (0.0867) & 0.7310 (0.0253) & 1.2384 (0.6512) \\
 & $\varOmega_{Q4}$ & 1.0006 (0.0006) & 0.9536 (0.0464) & 0.9925 (0.0076) & 0.7908 (0.2092) \\
\rowcolor{grey1}
 & Avg. Err. & \cellcolor{yellow1}0.0939 & 0.2761 & 0.1688 & 12.4709 \\
\midrule
\multirow{5}{*}{$0.005$}
 & $\varOmega_{Q1}$ & 0.2690 (0.0760) & 0.1918 (0.2327) & 0.2174 (0.1304) & 0.2328 (0.0689) \\
 & $\varOmega_{Q2}$ & 0.3524 (0.2952) & 0.2186 (0.5628) & 0.2445 (0.5111) & 0.4568 (0.0865) \\
 & $\varOmega_{Q3}$ & 0.7613 (0.0150) & 0.6920 (0.0773) & 0.7420 (0.0106) & 1.7817 (1.3756) \\
 & $\varOmega_{Q4}$ & 0.9978 (0.0022) & 0.8959 (0.1041) & 0.9861 (0.0139) & 0.6572 (0.3428) \\
\rowcolor{grey1}
 & Avg. Err. & \cellcolor{yellow1}0.0971 & 0.2442 & 0.1665 & 0.6362 \\
\midrule
\multirow{5}{*}{$0.006$}
 & $\varOmega_{Q1}$ & 0.2545 (0.0181) & 0.1252 (0.4993) & 0.2192 (0.1232) & 2.3093 (8.2371) \\
 & $\varOmega_{Q2}$ & 0.3118 (0.3764) & 0.2143 (0.5714) & 0.2511 (0.4978) & 2.3302 (3.6604) \\
 & $\varOmega_{Q3}$ & 0.7587 (0.0116) & 0.7109 (0.0521) & 0.7492 (0.0011) & 2.3302 (2.1069) \\
 & $\varOmega_{Q4}$ & 0.9961 (0.0039) & 0.9675 (0.0325) & 0.9821 (0.0179) & 2.3093 (1.3093) \\
\rowcolor{grey1}
 & Avg. Err. & \cellcolor{yellow1}0.1025 & 0.2188 & 0.1605 & 2.5664 \\
\midrule
\multirow{5}{*}{$0.007$}
 & $\varOmega_{Q1}$ & 0.2305 (0.0781) & 0.1885 (0.2462) & 0.2187 (0.1251) & 0.1095 (0.5619) \\
 & $\varOmega_{Q2}$ & 0.2520 (0.4961) & 0.2014 (0.5973) & 0.2456 (0.5088) & 0.8284 (0.6568) \\
 & $\varOmega_{Q3}$ & 0.7529 (0.0039) & 0.6098 (0.1870) & 0.7457 (0.0057) & 2.2004 (1.9339) \\
 & $\varOmega_{Q4}$ & 0.9934 (0.0066) & 0.7551 (0.2449) & 0.9931 (0.0069) & 0.7230 (0.2770) \\
\rowcolor{grey1}
 & Avg. Err. & \cellcolor{yellow1}0.1446 & 0.3189 & 0.1616 & 1.1699 \\
\midrule
\multirow{5}{*}{$0.008$}
 & $\varOmega_{Q1}$ & 0.2299 (0.0803) & 0.1892 (0.2433) & 0.2167 (0.1331) & -0.0729 (1.2917) \\
 & $\varOmega_{Q2}$ & 0.2519 (0.4962) & 0.2192 (0.5615) & 0.2439 (0.5122) & 0.2037 (0.5927) \\
 & $\varOmega_{Q3}$ & 0.7577 (0.0103) & 0.7206 (0.0392) & 0.7573 (0.0097) & 1.8169 (1.4225) \\
 & $\varOmega_{Q4}$ & 0.9902 (0.0098) & 0.9088 (0.0912) & 0.9684 (0.0316) & 0.2843 (0.7157) \\
\rowcolor{grey1}
 & Avg. Err. & \cellcolor{yellow1}0.0501 & 0.2333 & 0.1717 & 0.7567 \\
\midrule
\multirow{5}{*}{$0.009$}
 & $\varOmega_{Q1}$ & 0.2349 (0.0604) & 0.1268 (0.4930) & 0.2191 (0.1235) & -1.1447e11 (4.5788e11) \\
 & $\varOmega_{Q2}$ & 0.2627 (0.4746) & 0.2016 (0.5968) & 0.2441 (0.5118) & -1.1447e11 (2.2894e11) \\
 & $\varOmega_{Q3}$ & 0.7633 (0.0177) & 0.6939 (0.0748) & 0.7635 (0.0180) & -1.1447e11 (1.5263e11) \\
 & $\varOmega_{Q4}$ & 0.9890 (0.0110) & 0.8875 (0.1125) & 0.9718 (0.0282) & -1.1447e11 (1.1447e11) \\
\rowcolor{grey1}
 & Avg. Err. & \cellcolor{yellow1}0.1331 & 0.3195 & 0.1709 & 1.1447e11 \\
\midrule
\multirow{5}{*}{$0.010$}
 & $\varOmega_{Q1}$ & 0.3571 (0.4283) & 0.1767 (0.2932) & 0.2173 (0.1307) & -0.5609 (3.2437) \\
 & $\varOmega_{Q2}$ & 0.4667 (0.0667) & 0.1812 (0.6376) & 0.2439 (0.5121) & 0.0470 (0.9060) \\
 & $\varOmega_{Q3}$ & 0.8141 (0.0854) & 0.5072 (0.3237) & 0.7634 (0.0179) & 1.7888 (1.3851) \\
 & $\varOmega_{Q4}$ & 0.9849 (0.0151) & 0.5995 (0.4005) & 0.9628 (0.0372) & 0.0659 (0.9341) \\
\rowcolor{grey1}
 & Avg. Err. & \cellcolor{yellow1}0.2373 & 0.3548 & 0.1709 & 0.9337 \\
\bottomrule
\end{tabular}
}
\caption{Reconstructed diffusion coefficients and relative errors (parentheses)  $\xi = 0.3$.}
\label{tab:four_quadrants_complete_xi0.3}
\end{table}

\begin{table}[htp!]
\centering
\resizebox{!}{0.45\textheight}{%
\begin{tabular}{c l c c c c}
\toprule
$\delta$ & Region & CCBM & KV & TD & TN \\
\midrule
\multirow{5}{*}{0.001}
 & $\varOmega_{Q1}$ & 0.2457 (0.0170) & 0.1980 (0.2080) & 0.2220 (0.1119) & 0.1918 (0.2329) \\
 & $\varOmega_{Q2}$ & 0.2929 (0.4141) & 0.2336 (0.5327) & 0.2827 (0.4346) & 0.2424 (0.5152) \\
 & $\varOmega_{Q3}$ & 0.7319 (0.0241) & 0.7087 (0.0551) & 0.7269 (0.0309) & 0.2424 (0.6768) \\
 & $\varOmega_{Q4}$ & 1.0093 (0.0093) & 0.9862 (0.0138) & 1.0020 (0.0020) & 0.1918 (0.8082) \\
\rowcolor{grey1}
 & Avg. Err. & \cellcolor{yellow1}0.1161 & 0.2024 & 0.1449 & 0.4060 \\
\midrule
\multirow{5}{*}{$0.002$}
 & $\varOmega_{Q1}$ & 0.2437 (0.0252) & 0.1351 (0.4595) & 0.2190 (0.1240) & 11.6115 (45.4460) \\
 & $\varOmega_{Q2}$ & 0.2892 (0.4216) & 0.2290 (0.5420) & 0.2433 (0.5134) & 11.0857 (21.1714) \\
 & $\varOmega_{Q3}$ & 0.7356 (0.0192) & 0.6746 (0.1005) & 0.7242 (0.0344) & 0.8408 (0.1210) \\
 & $\varOmega_{Q4}$ & 1.0066 (0.0066) & 0.9363 (0.0637) & 1.0024 (0.0024) & 0.9152 (0.0848) \\
\rowcolor{grey1}
 & Avg. Err. & \cellcolor{yellow1}0.2181 & 0.2113 & 0.1686 & 6.5883 \\
\midrule
\multirow{5}{*}{$0.003$}
 & $\varOmega_{Q1}$ & 0.2408 (0.0368) & 0.0937 (0.6253) & 0.2191 (0.1235) & 35.4089 (140.6357) \\
 & $\varOmega_{Q2}$ & 0.2798 (0.4403) & 0.2166 (0.5668) & 0.2448 (0.5104) & 32.2431 (63.4862) \\
 & $\varOmega_{Q3}$ & 0.7385 (0.0153) & 0.6271 (0.1639) & 0.7265 (0.0314) & 1.0013 (0.3351) \\
 & $\varOmega_{Q4}$ & 1.0042 (0.0042) & 0.9314 (0.0686) & 1.0028 (0.0028) & 0.8507 (0.1493) \\
\rowcolor{grey1}
 & Avg. Err. & \cellcolor{yellow1}0.1241 & 0.3666 & 0.1670 & 17.4108 \\
\midrule
\multirow{5}{*}{$0.004$}
 & $\varOmega_{Q1}$ & 0.2580 (0.0319) & 0.1470 (0.4118) & 0.2184 (0.1263) & 25.5883 (101.3531) \\
 & $\varOmega_{Q2}$ & 0.3298 (0.3403) & 0.2203 (0.5593) & 0.2420 (0.5161) & 22.9964 (44.9929) \\
 & $\varOmega_{Q3}$ & 0.7520 (0.0026) & 0.6850 (0.0867) & 0.7310 (0.0253) & 1.2384 (0.6512) \\
 & $\varOmega_{Q4}$ & 1.0006 (0.0006) & 0.9536 (0.0464) & 0.9925 (0.0076) & 0.7908 (0.2092) \\
\rowcolor{grey1}
 & Avg. Err. & \cellcolor{yellow1}0.0939 & 0.2761 & 0.1688 & 12.4709 \\
\midrule
\multirow{5}{*}{$0.005$}
 & $\varOmega_{Q1}$ & 0.2690 (0.0760) & 0.1918 (0.2327) & 0.2174 (0.1304) & 0.2328 (0.0689) \\
 & $\varOmega_{Q2}$ & 0.3524 (0.2952) & 0.2186 (0.5628) & 0.2445 (0.5111) & 0.4568 (0.0865) \\
 & $\varOmega_{Q3}$ & 0.7613 (0.0150) & 0.6920 (0.0773) & 0.7420 (0.0106) & 1.7817 (1.3756) \\
 & $\varOmega_{Q4}$ & 0.9978 (0.0022) & 0.8959 (0.1041) & 0.9861 (0.0139) & 0.6572 (0.3428) \\
\rowcolor{grey1}
 & Avg. Err. & \cellcolor{yellow1}0.0971 & 0.2442 & 0.1665 & 0.6362 \\
\midrule
\multirow{5}{*}{$0.006$}
 & $\varOmega_{Q1}$ & 0.2545 (0.0181) & 0.1252 (0.4993) & 0.2192 (0.1232) & 2.3093 (8.2371) \\
 & $\varOmega_{Q2}$ & 0.3118 (0.3764) & 0.2143 (0.5714) & 0.2511 (0.4978) & 2.3302 (3.6604) \\
 & $\varOmega_{Q3}$ & 0.7587 (0.0116) & 0.7109 (0.0521) & 0.7492 (0.0011) & 2.3302 (2.1069) \\
 & $\varOmega_{Q4}$ & 0.9961 (0.0039) & 0.9675 (0.0325) & 0.9821 (0.0179) & 2.3093 (1.3093) \\
\rowcolor{grey1}
 & Avg. Err. & \cellcolor{yellow1}0.1025 & 0.2188 & 0.1605 & 2.5664 \\
\midrule
\multirow{5}{*}{$0.007$}
 & $\varOmega_{Q1}$ & 0.2305 (0.0781) & 0.1885 (0.2462) & 0.2187 (0.1251) & 0.1095 (0.5619) \\
 & $\varOmega_{Q2}$ & 0.2520 (0.4961) & 0.2014 (0.5973) & 0.2456 (0.5088) & 0.8284 (0.6568) \\
 & $\varOmega_{Q3}$ & 0.7529 (0.0039) & 0.6098 (0.1870) & 0.7457 (0.0057) & 2.2004 (1.9339) \\
 & $\varOmega_{Q4}$ & 0.9934 (0.0066) & 0.7551 (0.2449) & 0.9931 (0.0069) & 0.7230 (0.2770) \\
\rowcolor{grey1}
 & Avg. Err. & \cellcolor{yellow1}0.1446 & 0.3189 & 0.1616 & 1.1699 \\
\midrule
\multirow{5}{*}{$0.008$}
 & $\varOmega_{Q1}$ & 0.2299 (0.0803) & 0.1892 (0.2433) & 0.2167 (0.1331) & -0.0729 (1.2917) \\
 & $\varOmega_{Q2}$ & 0.2519 (0.4962) & 0.2192 (0.5615) & 0.2439 (0.5122) & 0.2037 (0.5927) \\
 & $\varOmega_{Q3}$ & 0.7577 (0.0103) & 0.7206 (0.0392) & 0.7573 (0.0097) & 1.8169 (1.4225) \\
 & $\varOmega_{Q4}$ & 0.9902 (0.0098) & 0.9088 (0.0912) & 0.9684 (0.0316) & 0.2843 (0.7157) \\
\rowcolor{grey1}
 & Avg. Err. & \cellcolor{yellow1}0.0501 & 0.2333 & 0.1717 & 0.7567 \\
\midrule
\multirow{5}{*}{$0.009$}
 & $\varOmega_{Q1}$ & 0.2349 (0.0604) & 0.1268 (0.4930) & 0.2191 (0.1235) & -1.1447e11 (4.5788e11) \\
 & $\varOmega_{Q2}$ & 0.2627 (0.4746) & 0.2016 (0.5968) & 0.2441 (0.5118) & -1.1447e11 (2.2894e11) \\
 & $\varOmega_{Q3}$ & 0.7633 (0.0177) & 0.6939 (0.0748) & 0.7635 (0.0180) & -1.1447e11 (1.5263e11) \\
 & $\varOmega_{Q4}$ & 0.9890 (0.0110) & 0.8875 (0.1125) & 0.9718 (0.0282) & -1.1447e11 (1.1447e11) \\
\rowcolor{grey1}
 & Avg. Err. & \cellcolor{yellow1}0.1331 & 0.3195 & 0.1709 & 1.1447e11 \\
\midrule
\multirow{5}{*}{$0.010$}
 & $\varOmega_{Q1}$ & 0.3571 (0.4283) & 0.1767 (0.2932) & 0.2173 (0.1307) & -0.5609 (3.2437) \\
 & $\varOmega_{Q2}$ & 0.4667 (0.0667) & 0.1812 (0.6376) & 0.2439 (0.5121) & 0.0470 (0.9060) \\
 & $\varOmega_{Q3}$ & 0.8141 (0.0854) & 0.5072 (0.3237) & 0.7634 (0.0179) & 1.7888 (1.3851) \\
 & $\varOmega_{Q4}$ & 0.9849 (0.0151) & 0.5995 (0.4005) & 0.9628 (0.0372) & 0.0659 (0.9341) \\
\rowcolor{grey1}
 & Avg. Err. & \cellcolor{yellow1}0.2373 & 0.3548 & 0.1709 & 0.9337 \\
\bottomrule
\end{tabular}
}
\caption{Reconstructed diffusion coefficients and relative errors (parentheses)  $\xi = 0.2$.}
\label{tab:four_quadrants_complete_xi0.2}
\end{table}

%
%
\begin{figure}[htp!]
\resizebox{0.225\textwidth}{!}{\includegraphics{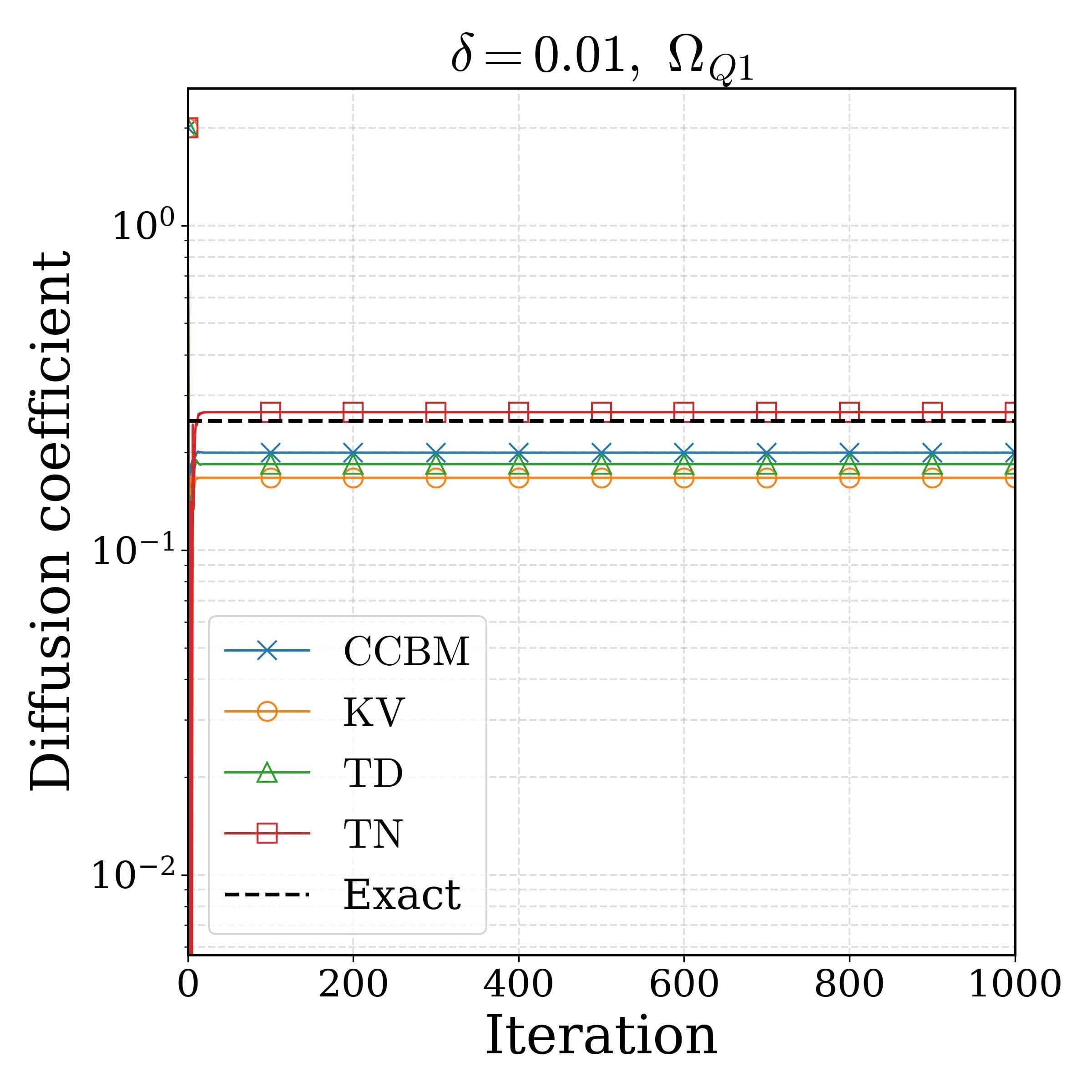}} \  
\resizebox{0.225\textwidth}{!}{\includegraphics{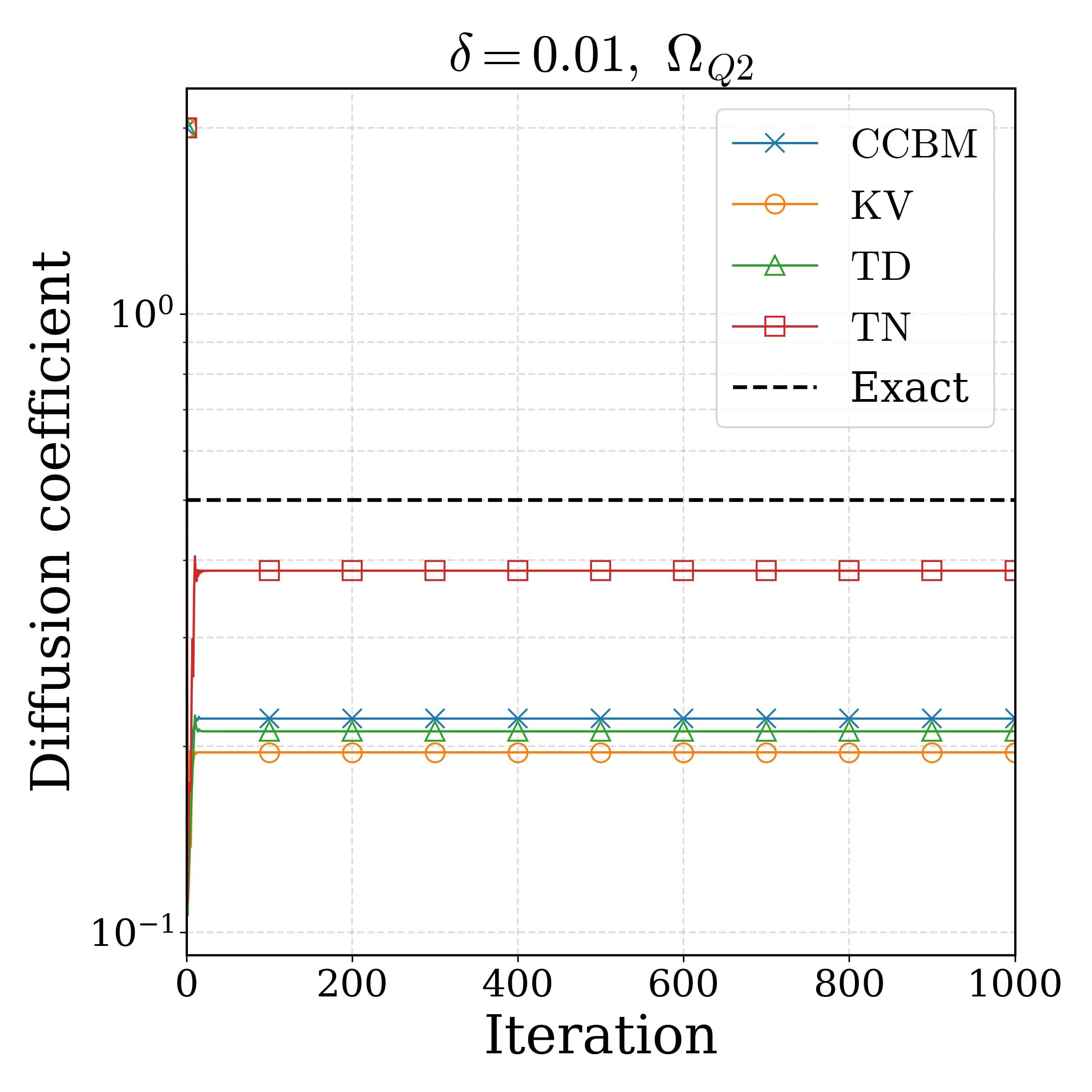}} \  
\resizebox{0.225\textwidth}{!}{\includegraphics{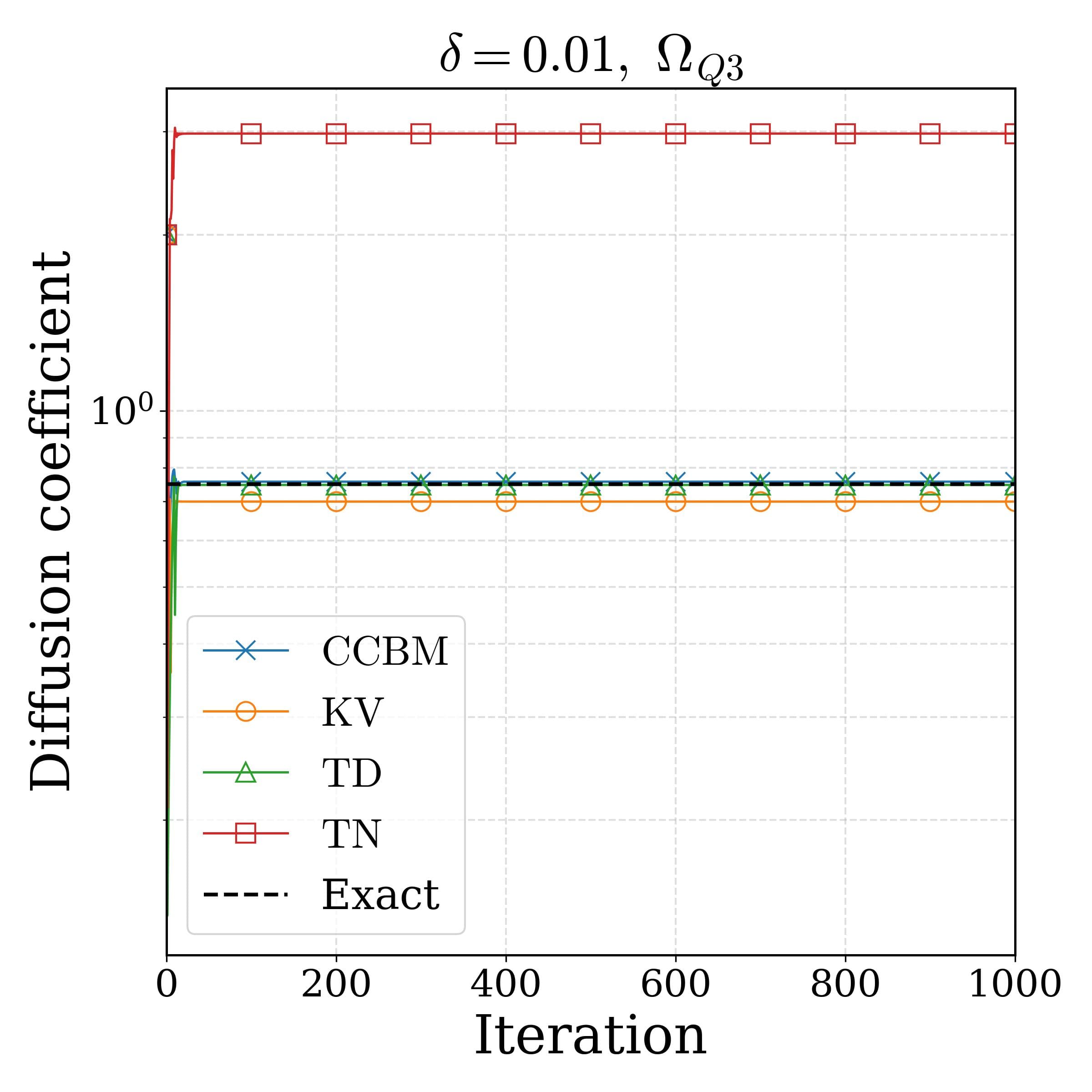}} \  
\resizebox{0.225\textwidth}{!}{\includegraphics{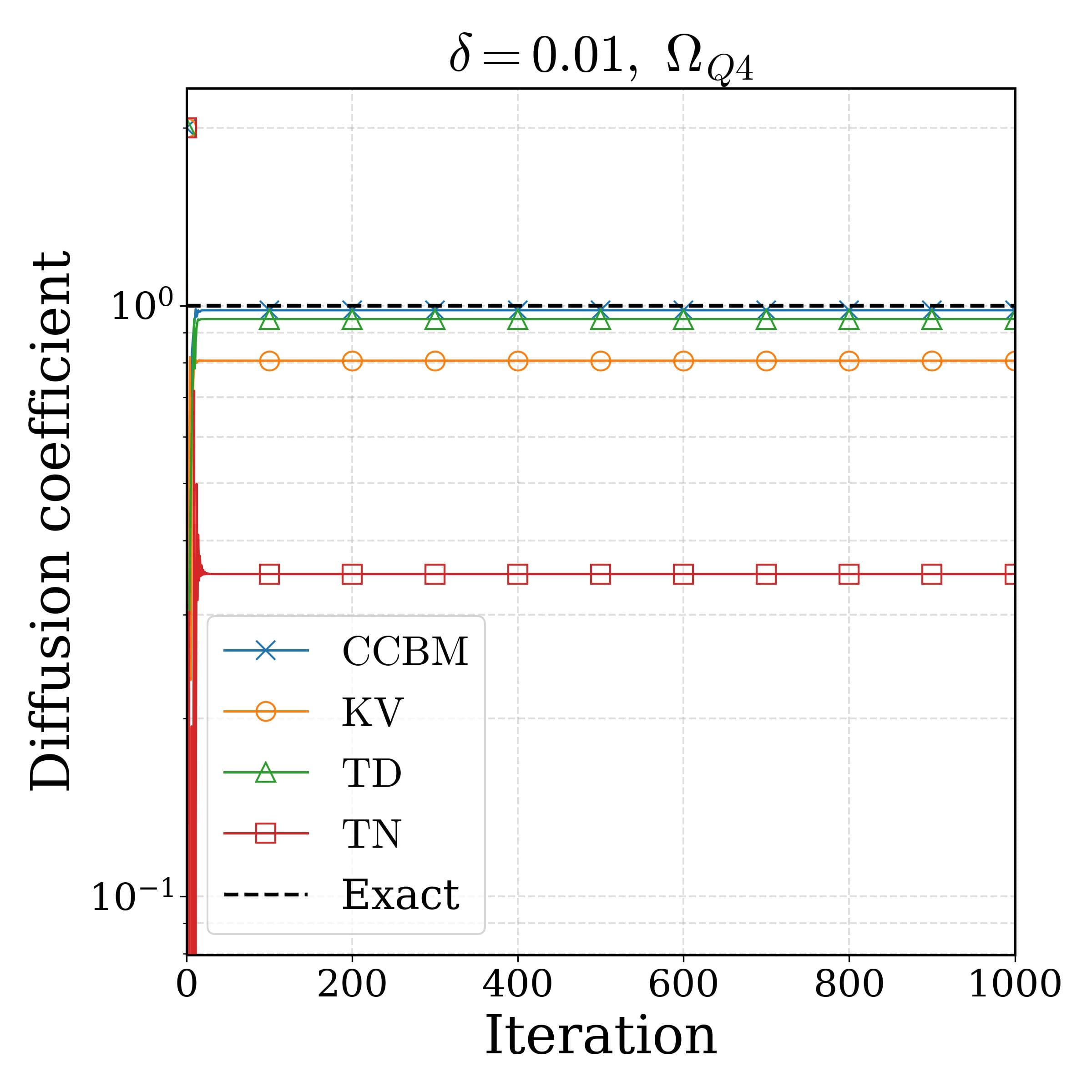}}
\caption{Iteration histories of the reconstructed diffusion coefficients in the four-subregions case for noise level $\delta = 0.01$.  Dashed lines denote the exact values.}
\label{fig:four_subregions}
\end{figure}
%

%
%
\begin{figure}[htp!]
\resizebox{0.225\textwidth}{!}{\includegraphics{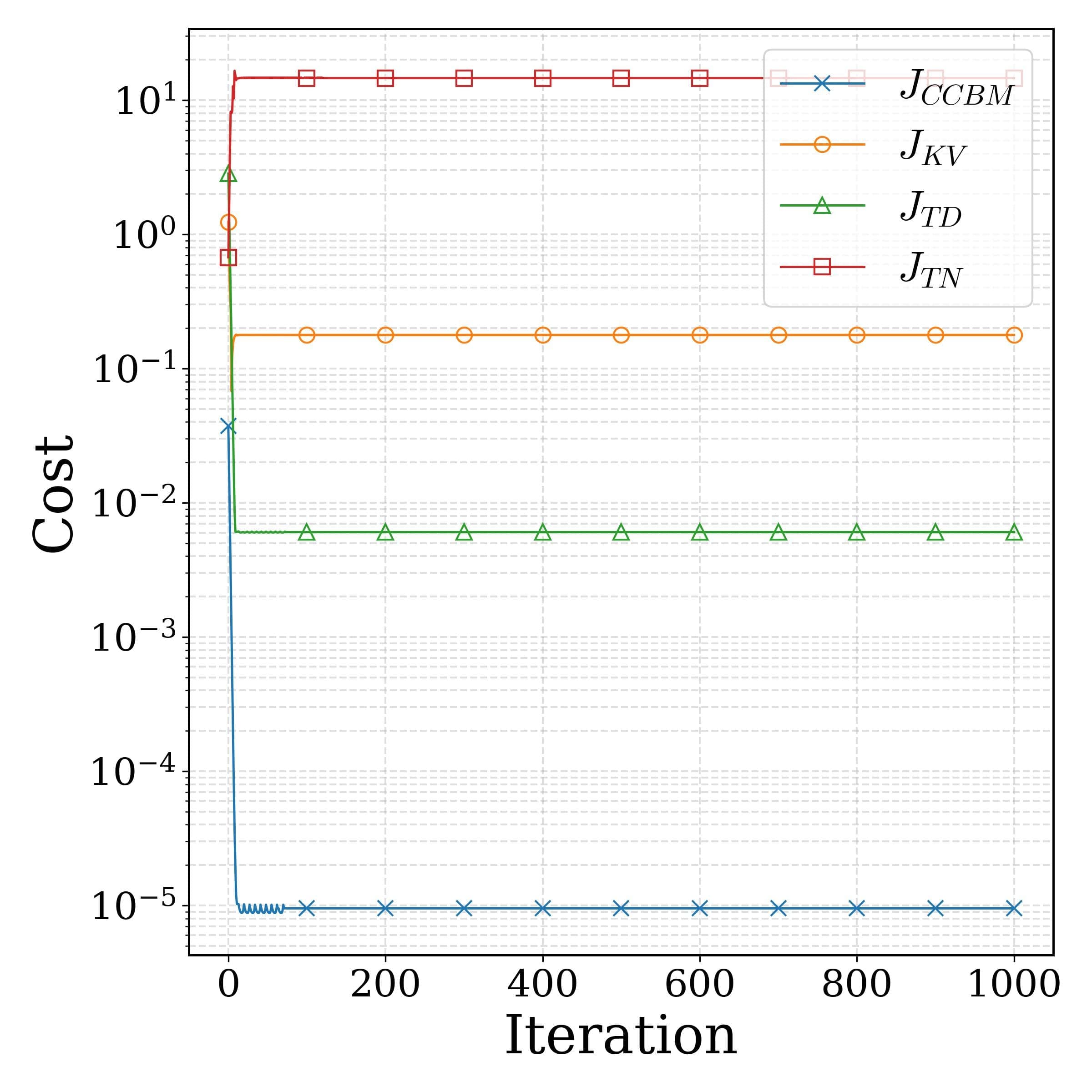}} \  
\resizebox{0.225\textwidth}{!}{\includegraphics{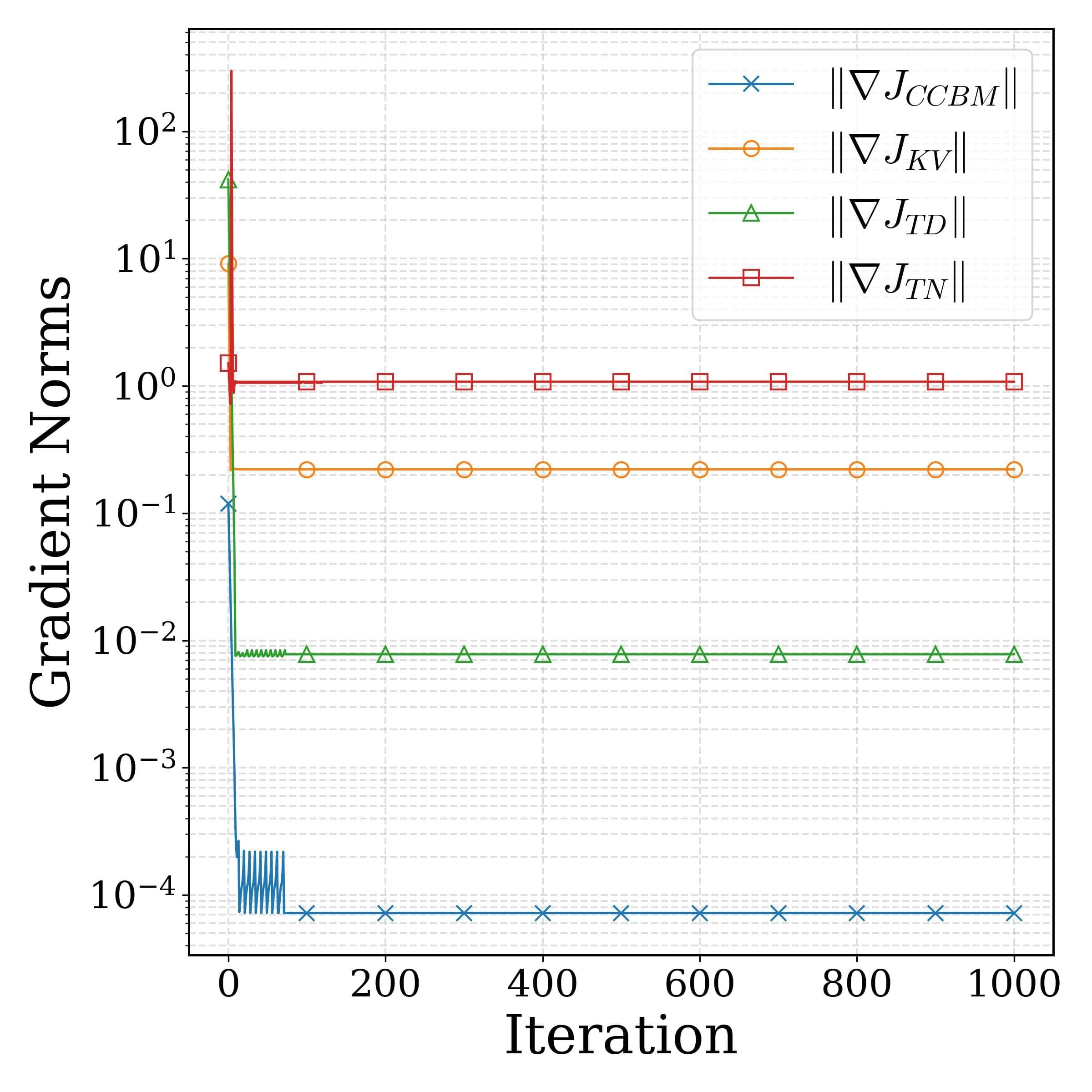}} \  
\resizebox{0.225\textwidth}{!}{\includegraphics{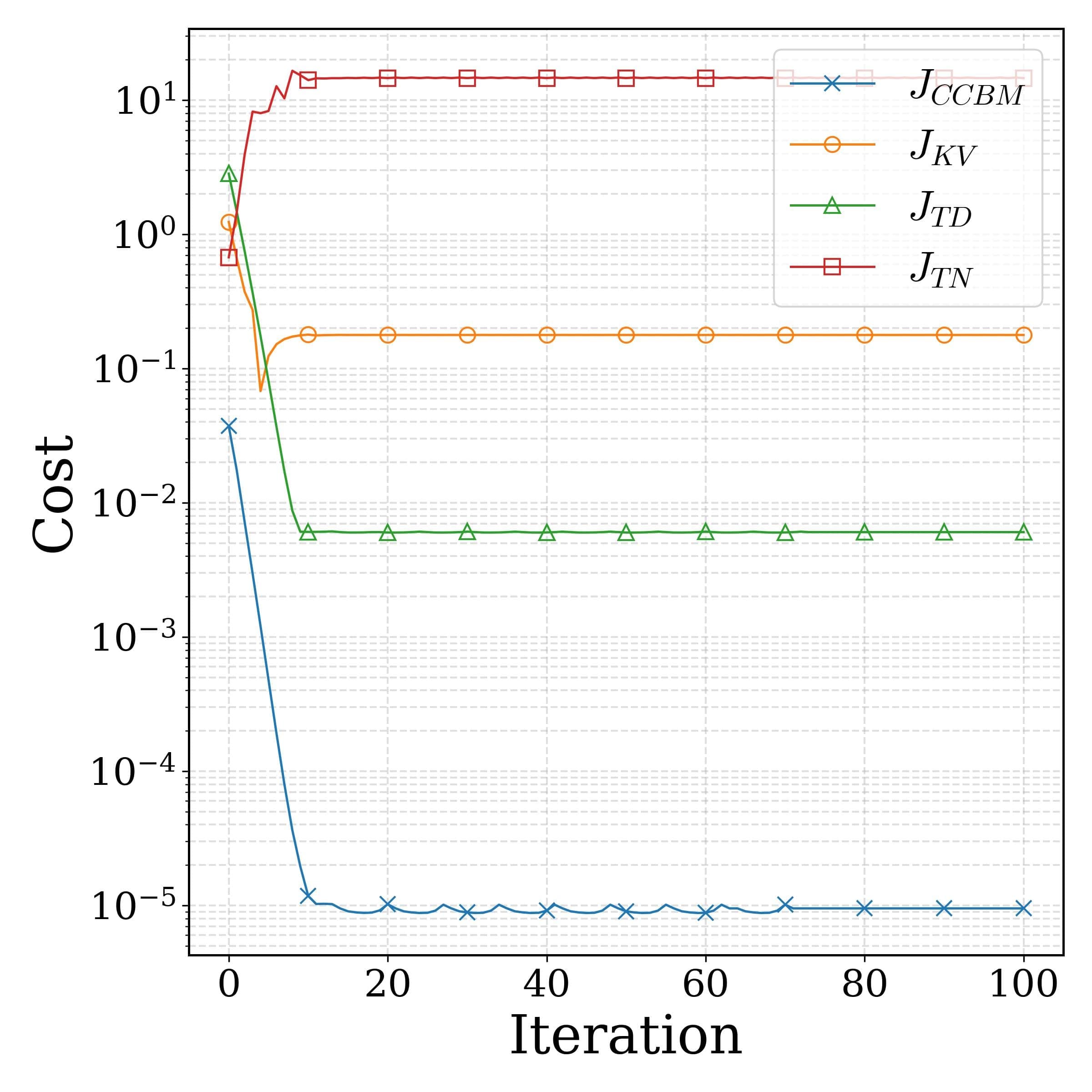}} \  
\resizebox{0.225\textwidth}{!}{\includegraphics{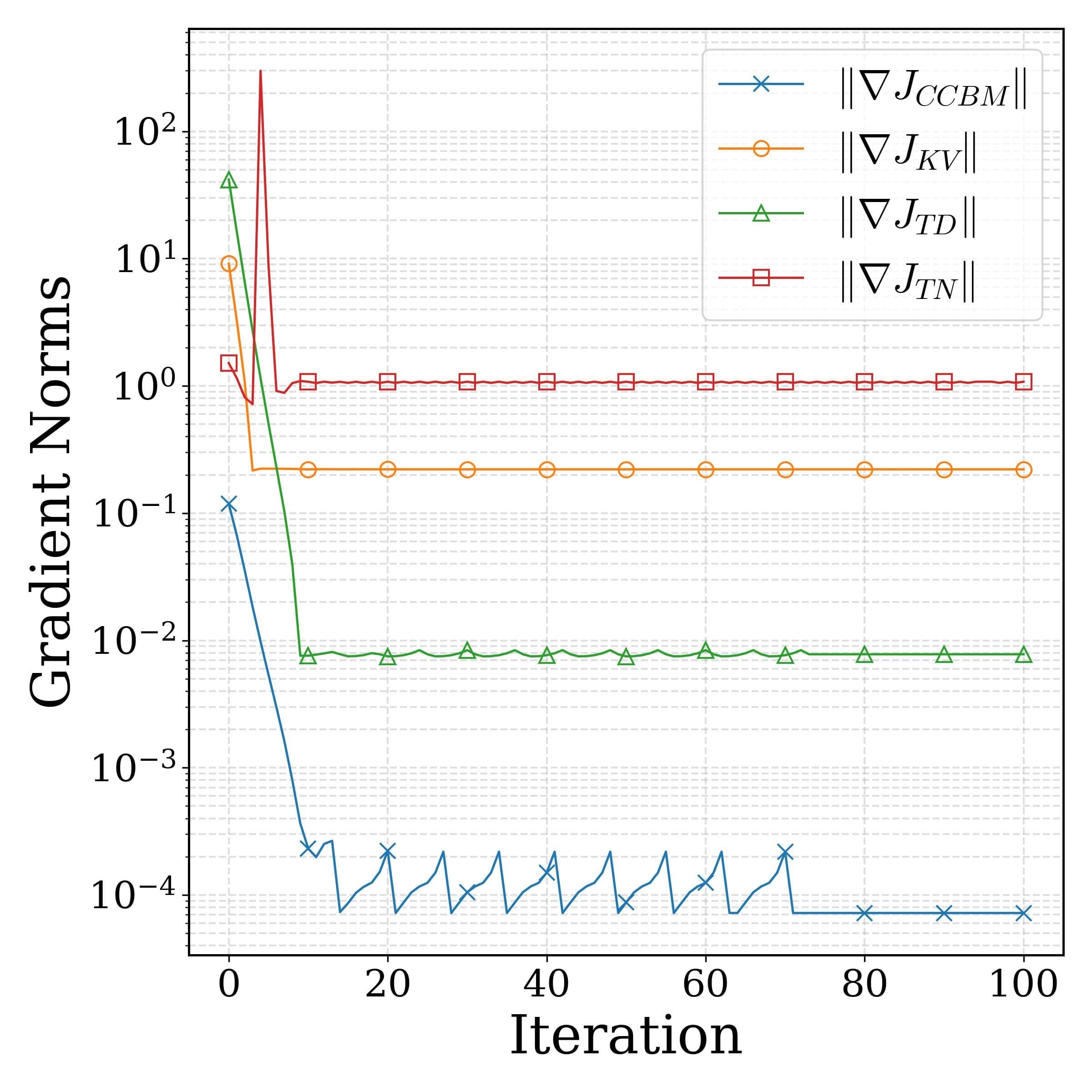}}
\caption{Left: Histories of the cost functional and gradient norm for the four-subregions example with noise level $\delta = 0.01$ over 1000 iterations; Right: Zoomed-in histories for the first 100 iterations.}
\label{fig:four_subregions_cost_and_gradient_norms}
\end{figure}
%

\newpage
\section{Conclusion}\label{sec:conclusion}  
We studied the inverse problem of reconstructing spatially varying diffusion coefficients from boundary Cauchy data within a modified coupled complex-boundary method framework. We established theoretical properties of the associated regularized optimization problem, including continuity and differentiability of the forward map, Lipschitz continuity of the cost functional, existence of minimizers, stability with respect to noisy data, and convergence under vanishing noise.

Numerical reconstructions obtained using a Sobolev-gradient descent scheme exhibit stable behavior across the considered test cases. In the reported experiments, incorporating a gradient-weighted misfit term is associated with improved sensitivity to spatial variations and reduced high-frequency artifacts when the weighting parameter is chosen within an appropriate range. A projection-based extension further enables effective recovery of piecewise-constant diffusion coefficients in heterogeneous settings.

Overall, the theoretical analysis and numerical findings support the proposed approach as a viable method for diffusion coefficient identification in both smooth and piecewise-defined media. The numerical framework also suggests potential applicability to related inverse problems involving piecewise-defined parameters, such as absorption coefficient estimation or source identification, which may be explored in future work.

\medskip
\textbf{Acknowledgements} SPN acknowledges the Japanese Ministry of Education, Culture, Sports, Science and Technology (MEXT) for scholarship support during his PhD program. JFTR is supported by the JSPS Postdoctoral Fellowships for Research in Japan and partially by the JSPS Grant-in-Aid for Early-Career Scientists under Japan Grant Number JP23K13012.
HN is partially supported by JSPS Grants-in-Aid for Scientific Research under Grant Numbers JP21H04431, JP24H00188, and JP25K00920.
JFTR and HN are also partially supported by the JST CREST Grant Number JPMJCR2014.

%
%
\bibliographystyle{alpha} 
\bibliography{main}   
\end{document}